\documentclass[a4paper,openright,11pt,english]{book}
\usepackage{amsmath,amsfonts}
\usepackage{amssymb}
\usepackage{amsthm,mathrsfs}
\usepackage{epsfig}
\usepackage{verbatim}
\usepackage[english]{babel}
\usepackage[utf8]{inputenc}  % parol accentate
\usepackage{graphicx}
\usepackage{epstopdf}
\usepackage{bbm}
\usepackage{color}
\usepackage{booktabs}
\usepackage{ctable}
\usepackage{subfigure}
\usepackage{float}
\usepackage{layout}
\usepackage{fancyhdr}
\usepackage{color}
\usepackage{eurosym}
\usepackage{wrapfig}  %per affiancare le figure al testo

\usepackage{amssymb}
\usepackage{graphicx}
\usepackage{amsthm}
\usepackage{enumerate}
\usepackage{hyperref}
\usepackage{emptypage}
\usepackage{amsmath, amsfonts,amssymb}
\usepackage {fancyhdr}

\usepackage{makeidx}

\usepackage{appendix}
\usepackage{chngcntr}
\usepackage{etoolbox}
\usepackage{lipsum}

\font\dsrom=dsrom10 scaled 1200

\makeatletter
\if@twoside % commands below work only for twoside option of \documentclass
\newlength{\textblockoffset}
\fi
\makeatother

\newcommand\ind[1]{\textrm{\dsrom{1}}_{#1}}

\def \m{\mathfrak{m}}
\def \O{\mathcal{O}}
\def \L{\mathcal{L }}
\def \into{\int_{\mathcal{O}}}

\newcommand{\T}{\mathcal{T}}
\def \N{\mathbb{N}}
\def \R{\mathbb{R}}
\def \Z{\mathbb{Z}}
\def \E{\mathbb{E}}
\def \P{\mathbb{P}}

\def \avirg{``}
\DeclareMathOperator{\supp}{supp}

\def\II{\mathcal{I}}
\def \D{\mathcal{D}}
\def \L{\mathcal{L}}
\def \O{\mathcal{O}}

\def \dx{\Delta x}
\def\ii{\mathbf{i}}
\def\Id{\mathrm{Id}}
\def\tr{\mathrm{Tr}}
\def \pol{{\mathbf{pol}}}
\def \Ex{\mathcal{E}}
\def \etildet{\tilde{\Ex}_t}

\def \into{\int_{\mathcal{O}}}

\def\<{\langle}
\def\>{\rangle}
\def\cvd{$\quad\Box$\medskip}

\def\T{{\mathcal{T}}}

\def\<{\langle}
\def\>{\rangle}

\def\Z{{\mathbb{Z}}}

\def\bu{\overline{u}}

\def \Sum{\displaystyle\sum}

\def \Frac{\displaystyle\frac}

\newcommand{\facciatabianca}{\newpage\shipout\null\stepcounter{page}}
\newtheorem{theorem}{Theorem}[section]

\newtheorem{rema}[theorem]{Remark}

\newtheorem{definition}[theorem]{Definition}

\newtheorem{corollary}[theorem]{Corollary}

\newtheorem{lemma}[theorem]{Lemma}

\newtheorem{proposition}[theorem]{Proposition}

\newtheorem{remark}[theorem]{Remark}

\begin{document}
	\titlepage
	\thispagestyle{empty} \pagenumbering{Roman}

\begin{center}
		\begin{center}
			
				\begin{figure} [!htb]
			\begin{center}
				\begin{tabular}{c}
					\includegraphics[scale=0.5]{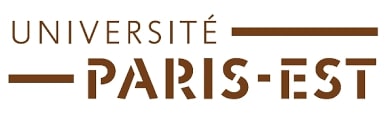} $\qquad\qquad\qquad
					\qquad\qquad$
					\includegraphics[scale=1.2]{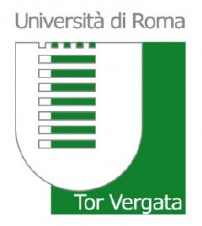}
				\end{tabular}
			\end{center}
			\end{figure}
			
								{\textsc {\LARGE Universit\`a di Roma Tor Vergata}\\
									\Large	Dipartimento di Matematica}	

							\bigskip
							
											{\LARGE \textsc {Universit\'e  Paris-Est Marne-La-Vallée}\\
												\Large		\'{E}cole Doctorale Math\'{e}matiques et STIC. Discipline: Math\'{e}matiques}	
				%			\vspace{2cm}
								
%								\begin{figure}[h]
%									\centering
%								%	\includegraphics[scale=.5]{Logo1}
%								%	\hspace{1cm}
%									\includegraphics[scale=.35]{logoparisest}
%								\end{figure}
								
\end{center}
%\text{\sc \LARGE Universit\`a di  Roma Tor Vergata}

%	
%	\text{\LARGE and }
%	
%\text{\sc \LARGE 	Universit\'e Paris-Est Marne-la-Vallée}

\bigskip\bigskip
%\text{\sc Dipartimento di Matematica}

\vspace{1 cm}
\text{\sc \LARGE Ph.D. Thesis}
\vspace{10pt}

\vspace{10pt}

%\text{\LARGE a.a. 2014/2015}

\vspace{0.5 cm}

	\begin{minipage}{\textwidth}
								\hrule
								\vspace{0.5cm}
								\centering
								{\textbf {\LARGE { 
									%		\textcolor{red!70!black}
											% {The role of first order superlinear terms \\ [-1mm] 
											% in local and global time behaviour\\  [2.5mm]
											%  of nonlinear parabolic equations}}}}
											%\vspace{0.5cm}
											{Option prices in \\[3mm] 
												stochastic volatility models}}}
								}
								\vspace{0.5cm}
								\hrule
							\end{minipage}
%\vspace{0.5 cm}
%%\text{\LARGE Matematica Pura ed Applicata}
%
%\
%
%\
%
%\LARGE{\bf Option prices in stochastic volatility models}
\end{center}

\begin{center}
\bigskip%\bigskip%\bigskip
\textbf{\large }
\end{center}
\begin{center}

\text{\Large Giulia Terenzi}
\bigskip\bigskip\bigskip\bigskip

{\large  Defence date: 17/12/2018}

\bigskip

{\large Dissertation defence committee:\\

	Fabio ANTONELLI (Examinator)\\
	Aurélien ALFONSI (Examinator) \\
	Maya BRIANI (Examinator)\\
	Lucia CARAMELLINO (Advisor)\\
	Damien LAMBERTON (Advisor)}

%	\begin{minipage}{0.4\textwidth}
%			{\sc \Large Advisors:}
%			
%			{\Large  Lucia Caramellino}
%		
%			{\Large Damien Lamberton}
%					\end{minipage}
%\hfill
%	\begin{minipage}{0.4\textwidth}
%			\begin{flushright}
%				{	\sc\Large Coordinators:} 
%					
%					{\Large A. Braides}
%					
%					{\Large B. Jourdain}
%				\end{flushright}		
%			\end{minipage}
%\vspace{2cm}

%{\Large  Academic year 2017/2018}
\end{center}

	\facciatabianca

	\thispagestyle{empty}
	\null
	\vspace{2cm}
	\begin{flushright}
		\emph{\LARGE Alla mia famiglia}
	\end{flushright}
	\vspace{\stretch{2}}\null

	\pagenumbering{Roman} \linespread{1.5}
	\pagestyle{fancy}
	\textwidth 5.5in \textheight 20cm \topmargin
	0.25in \setcounter{tocdepth}{2}
	\linespread{1.2}
	\renewcommand{\chaptermark}[1]{\markboth{Chap.\ \textbf{\thechapter}\ -\ \textit{#1}}{}}
	\renewcommand{\sectionmark}[1]{\markright{Sec.\ \textbf{\thesection}\ -\ \textit{#1}}}
	\fancyhf{} \fancyfoot[CE,CO]{\thepage} \fancyhead[CO]{\rightmark}
	\fancyhead[CE]{\leftmark}
	\renewcommand{\headrulewidth}{0.5pt}
	\renewcommand{\footrulewidth}{0.0pt}
	\addtolength{\headheight}{0.5pt}
	\fancypagestyle{plain}{\fancyhead{}\renewcommand{\headrulewidth}{0pt}}
	
	\chapter*{Abstract}{\small
		We study option pricing problems in stochastic volatility models.
		In the first part of this thesis  we focus on American options in the Heston model.  We first give an analytical characterization of the value function of an American option as the unique solution of the associated (degenerate) parabolic obstacle problem. Our approach is based on variational inequalities in suitable weighted Sobolev spaces and  extends recent results of Daskalopoulos and Feehan (2011, 2016) and Feehan and Pop (2015).
		We also investigate the  properties of the American value function. In particular, we prove that, under suitable assumptions on the payoff, the value function is nondecreasing with respect to the volatility variable. Then, we focus on an  American put option and we extend some results which are well known in the Black and Scholes world.  In particular, we prove the strict convexity of the value function in the continuation region, some properties of the free boundary function, the Early Exercise Price formula and a weak form of the smooth fit principle. 
		This is done  mostly by using probabilistic techniques.
		
		In the second part we deal with the numerical computation of European and American option prices in jump-diffusion stochastic volatility models. We first focus on the Bates-Hull-White model, i.e. the Bates model 
		with a stochastic interest rate. We consider a backward hybrid  algorithm which uses  a Markov chain approximation (in particular, a  \avirg multiple jumps" tree) in the direction of the volatility and the interest rate and a (deterministic) finite-difference approach in order to handle the underlying asset price process. Moreover, we provide a simulation scheme to be used for Monte Carlo evaluations. Numerical results show the reliability and the efficiency of the proposed methods.
		% We also propose hybrid simulations for the model, following a binomial tree in the direction of both the volatility and the interest rate, and a space-continuous approximation for the underlying asset price process coming from a Euler-Maruyama type scheme.   Several numerical results are provided.
		
		Finally, we analyse the rate of convergence of the hybrid algorithm applied to  general jump-diffusion models. 
		We study first order weak convergence of Markov chains to diffusions under quite general assumptions. Then, we prove the convergence of the algorithm, by studying the stability and the consistency of the hybrid scheme, in a sense that allows us to exploit the probabilistic features of the Markov chain approximation.
		
		\medskip
		
		\textbf{Keywords}: stochastic volatility; European options; American options; degenerate parabolic problems; optimal stopping;   tree methods; finite-difference.
		
		\chapter*{R\'esum\'e}\small
		L'objet de cette th\`ese est l'\'etude de probl\`emes d'\'evaluation d'options dans les mod\`eles \`a volatilit\'e stochastique.
		La premi\`ere partie est centr\'ee sur  les options am\'ericaines dans  le mod\`ele de Heston.
		Nous donnons d'abord une caract\'erisation analytique de la fonction de valeur d'une option am\'ericaine comme l'unique solution
		du probl\`eme d'obstacle parabolique d\'eg\'en\'er\'e associ\'e. Notre approche est bas\'ee sur des in\'equations
		variationelles dans des espaces de Sobolev avec poids  \'etendant les r\'esultats r\'ecents de
		Daskalopoulos et Feehan (2011, 2016) et Feehan et Pop (2015).
		On \'etudie aussi les propri\'et\'es de la fonction de
		valeur d'une option am\'ericaine. En particulier, nous prouvons que, sous des hypoth\`eses convenables sur le payoff,
		la fonction de valeur est d\'ecroissante par rapport \`a la volatilit\'e. Ensuite nous nous concentrons sur le
		put am\`ericain et nous \'etendons  quelques r\'esultats qui sont bien connus dans le monde Black-Scholes.
		En particulier nous prouvons la convexit\'e stricte de la fonction de valeur dans la r\'egion de continuation, quelques
		propri\'et\'es de la  fronti\`ere libre, la formule de Prime d'Exercice Anticipée et une forme  faible de la propriété
		du smooth fit. Les techniques utilis\'ees sont de type probabiliste.
		
		Dans la deuxi\`eme partie nous abordons  le probl\`eme du calcul num\'erique du prix des options europ\'eenne et am\'ericaines
		dans des
		mod\`eles à volatilit\'e stochastique et avec sauts. Nous \'etudions d'abord le mod\`ele
		de Bates-Hull-White, c'est-\`a-dire le  mod\`ele de Bates avec un taux d'int\'er\^et stochastique. On consid\`ere un
		algorithme hybride r\'etrograde qui utilise une approximation par cha\^\i ne de Markov (notamment un arbre \avirg avec sauts multiples")
		dans la direction de la volatilit\'e et du taux d'int\'er\^et et une approche (d\'eterministe) par diff\'erence finie pour
		traiter le processus de prix d'actif. De plus, nous fournissons une proc\'edure de simulation pour des
		\'evaluations  Monte Carlo. Les r\'esultats num\'eriques montrent la fiabilit\'e et l'efficacit\'e de ces m\'ethodes.
		Finalement, nous analysons le taux de convergence de l'algorithme hybride appliqu\'e \`a  des mod\`eles g\'en\'eraux  de diffusion
		avec sauts. Nous \'etudions d'abord la convergence faible au premier ordre de cha\^\i nes de Markov vers la diffusion sous
		des hypoth\`eses assez g\'en\'erales. Ensuite nous prouvons la convergence de l'algorithme: nous \'etudions la stabilit\'e et
		la consistance de la m\'ethode hybride par une technique qui exploite les caract\'eristiques probabilistes de
		l'approximation par cha\^\i ne de Markov.
		
		\medskip
		
		\textbf{Mots cl\'es} : volatilit\'e stochastique ;  options am\'ericaines ; options europ\'eennes ; probl\`emes paraboliques d\'eg\'en\'er\'es ; arr\^{e}t optimal ;   approximation par arbres ; diff\'erences finies.
	}

	\tableofcontents \clearpage
	
	\fancyhf{} \fancyfoot[CE,CO]{\thepage}
	\fancyhead[CO]{\textit{Introduction}}
	\fancyhead[CE]{\textit{Introduction}}
	\renewcommand{\headrulewidth}{0.5pt}
	\renewcommand{\footrulewidth}{0.0pt}
	\addtolength{\headheight}{0.5pt}
	\fancypagestyle{plain}{\fancyhead{}\renewcommand{\headrulewidth}{0pt}}
	
	\chapter*{Introduction}
	\addcontentsline{toc}{chapter}{Introduction}
	%The fact that volatility in a market model should vary randomly is nowadays completely  recognized.

	The seminal work by Black and Scholes (\cite{BS}, 1973) was the starting point of equity dynamics modelling and it is still widely used as a useful approximation.  
	It owns its great success to its high intuition, simplicity and  parsimonious description of the 	market derivative prices. Nevertheless, it is a well known fact that it disagrees with reality in a number of significant ways. Even   F. Black, 15 years after the publication of the original  paper, wrote about  the flaws of the model \cite{Black}. 
	Indeed, empirical studies show that in the real market the log-return process is not normally distributed and its distribution is often affected by heavy tail, jumps and high peaks.
	Moreover,   the assumption of a constant volatility turns out to be too rigid to model the real world financial market. 	
	%	It is enough to analyse  traded call options to recognize  the well known  \textit{smile/skew effect}. Let us plot the  implied volatility (that is the value of the volatility parameter that, replaced in the Black and Scholes formula, gives the real market price) against the strike price. Then, we can observe that the resulting shape is not a horizontal line, as it should derive from assuming a constant volatility, but it is usually convex and can present higher values for high and low values of the strike price (a \textit{smile}) or asymmetries (from which the term \textit{skew}). 
	It is enough to analyse the so-called implied volatility (that is the value of the volatility parameter that, replaced in the Black and Scholes formula, gives the real market price) in  a set of  traded call options to recognize  the well known  \textit{smile/skew effect}. In fact, if we plot the implied volatility against the strike price, we can observe that the resulting shape is not a horizontal line, as it should derive from assuming a constant volatility, but it is usually convex and can present higher values for high and low values of the strike price (a \textit{smile}) or asymmetries (from which the term \textit{skew}). 
	Furthermore, the assumption of a constant volatility does not allow to properly price and hedge options which strongly depend on the volatility itself, such as the options on the realized variance or the \textit{cliquet options}.

	These results have called for more sophisticated models which can better reflect the reality. % In particular, the fact that volatility  should vary randomly is nowadays completely  recognized.
	%(remaining, at the same time, analytically tractable).
	% Of course,  a model derives its legitimacy and usefulness  not only from the accuracy with which it captures the real market  prices  but also from its analytical tractability.  
	Various approaches to model volatility have been introduced over time, paving the way for a huge body of literature devoted to this subject.  Let us  briefly recall some of the most famous ones. %referring to the existing literature for   complete overviews (see, for instance,  \cite{Berg}).   

	Roughly speaking, we can recognize two different classes of models. The first class is given by models in which the volatility is assumed to  depend on the same noise source as the underlying asset. Here, we can find the so-called  \textit{local volatility models}, where the volatility is assumed to be a function of   time and of the current underlying asset price. Therefore, the asset price $S$ is modeled by a diffusion process of the type 
	$$
	dS_t=\mu(t,S_t)S_tdt+\sigma(t,S_t)S_tdB_t.
	$$ 
	Under classical assumptions these models preserve the completeness of the market and all the Black-Sholes pricing and hedging theory can be adapted (see, for example, \cite[Chapter 2]{Berg}).  The choice of a suitable local volatility function $\sigma=\sigma(t,S)$, is a delicate problem. 
	%Also the delicate problem of finding a volatility function  $\sigma=\sigma(t,S_t)$ consistent with the market was (at least theoretically) solved by Dupuire \cite{D}.
	Bruno Dupire proved   in \cite{D} that it is  possible to find a  function   $\sigma=\sigma(t,S)$ which gives theoretical prices matching   a given configuration of vanilla options' prices.
	%
	% The delicate problem of finding a volatility function  $\sigma=\sigma(t,S_t)$  which is consistent with the  such that, by solving equation (2.2) they are recovered? What is the condition for the existence of a σ (t, S)?
	%The expression of σ (t, S),which we derive below, was found by Bruno Dupire
	%where C (K, T ) is the price of a call option of strike K and maturity T . This equation expresses the fact that the local volatility for spot S and time t is re ected in the di erences of option prices with strikes straddling S and maturities straddling t.
	%
	%
	Typically, the local volatility function is calibrated at $t=0$ on the market smile and kept frozen afterwards.  Therefore,  it does not take into account the daily changes in the volatility smile observed in the market. 
	For this reason, local volatility models seem to be an analytically tractable simplification of the reality rather than a representation of how volatility really evolves.  Other different models  presented in the literature belong to this first class, for instance \textit{path dependent volatility models}, in which volatility depends on the whole past trajectory of the
	asset price (see \cite{FPa,HR}).

	The second class of models consists of the so-called \textit{stochastic volatility models}. Here, the volatility  is modelled by an autonomous stochastic process $Y$ driven by some additional random noise. Typically, a stochastic volatility model is a Markovian model of the form
	\begin{align*}
	&dS_t=\mu_S(t,S_t)S_tdt+\sigma_S(Y_t)S_tdB_t,\\
	&dY_t=\mu_Y(t,Y_t)dt+\sigma_Y(t,Y_t)dW_t,
	\end{align*}
	where $B$ and $W$ are  possibly correlated Brownian motions. Moreover, often jumps are added  to the dynamics of the assets prices  and/or their volatilities. The literature on stochastic volatility models  is huge. The most successful model is the one introduced by S. Heston \cite{H}, which will be 
	extensively studied later on in this thesis. Among the others   we cite, for example, the models by Hull and White \cite{HW}, Bates \cite{bates} and Stein and Stein \cite{SS}. Moreover, there are also examples of \textit{local-stochastic volatility models} (such as the famous SABR model \cite{HKLW}) in which the volatility coefficient $\sigma_S(Y_t)$ of the underlying asset price is more general and has the form $\sigma_S(S_t,Y_t)$, that is it depends also on the current asset price.
	
	These models are, in general, not complete: the  derivative securities are usually not replicable by trading in the underlying. However, this does not affect the practice since the market can be completed with well known  procedures of market completion (for example by trading  a finite number of vanilla options).
	
	We point out that the research is still fervent in this area. For example,  empirical studies have questioned the smoothness of the volatility dynamics. As a consequence, new  models called \textit{rough volatility models} have recently  been introduced. They are non-Markovian models in which   the volatility is driven by a  Fractional Brownian motion, see the reference paper \cite{GJR} and the comprehensive website \cite{rough}, which gathers all the  developments on this subject.

	In this thesis we  consider Markovian stochastic volatility models and we collect some  results on the  problem of pricing European and American options.  It is divided into  two strongly correlated parts.
	%\begin{itemize}
	%	\item \textit{Part I: American option prices in Heston-type models}, in which we study some theoretical properties of the American option value function in the stochastic volatility Heston model.  This part is extracted from the following works:
	%	\begin{itemize}
	%		\item D. Lamberton, G. Terenzi \textit{Variational formulation of American option prices in Heston-type models}, submitted to publication, 2017.
	%	\item D. Lamberton, G. Terenzi \textit{American option price properties in Heston-type models}.
	%	\end{itemize}
	%\item \textit{Part II: Hybrid schemes for pricing options in jump-diffusion stochastic volatility models}, in which we deal with the problem of numerically computing the prices, describing and theoretically studying hybrid schemes for pricing European and American options  in jump-diffusion stochastic volatility models which extend the Heston model.  %Almost all the results of the second part can be found in  \cite{bctz,bct}. 
	%This part is mostly  based on the following related papers:
	%\begin{itemize}
	%	\item M. Briani, L. Caramellino, G. Terenzi, A. Zanette,  \textit{Hybrid  Monte Carlo and tree-finite differences algorithm for pricing options in the Bates-Hull-White model}, submitted to publication, 2017.
	%	\item M. Briani, L. Caramellino, G. Terenzi, \textit{Weak convergence of Markov chains and numerical schemes to jump diffusion processes}, submitted to publication, 2018.
	%\end{itemize}
	%\end{itemize} 
	In the first one we study some theoretical properties of the American option prices in  Heston-type models. In the second part, we deal with the problem of the numerical computation of the prices,  describing and theoretically studying hybrid schemes for pricing European and American options  in jump-diffusion stochastic volatility models.  
	More precisely, the thesis is organized as follows:
	
	\begin{itemize}
		\item Part I:
		American option prices in Heston-type models
		\begin{itemize}
			\item 
			Chapter 1. Variational formulation of American option prices in Heston-type models;
			\item
			Chapter 2. American option price properties in Heston-type models.
		\end{itemize}
		
		\item Part II:
		Hybrid schemes for pricing options in jump-diffusion stochastic volatility models
		\begin{itemize}
			\item 
			Chapter 3. Hybrid  Monte Carlo and tree-finite differences algorithm for pricing options in the Bates-Hull-White model;
			\item
			Chapter 4. 
			Weak convergence of Markov chains and numerical schemes for jump diffusion processes.
		\end{itemize}
		
	\end{itemize}
	The above chapters are extracted, sometimes verbatim, from the papers \cite{LT,LT2,bct,bctz} respectively. We now give a brief outline of the main results collected in this thesis. 
	
	% The dissertation  is extracted, sometimes verbatim, from the
	% following works:
	%	\begin{itemize}
	%		\item \textsc{Chapter 1}: D. Lamberton, G. Terenzi, \textit{Variational formulation of American option prices in Heston-type models}, submitted to publication, 2017.
	%	\item \textsc{Chapter 2}: D. Lamberton, G. Terenzi \textit{American option price properties in Heston-type models}.
	%\item \textsc{Chapter 3}: M. Briani, L. Caramellino, G. Terenzi, A. Zanette,  \textit{Hybrid  Monte Carlo and tree-finite differences algorithm for pricing options in the Bates-Hull-White model}, submitted to publication, 2017.
	%	\item \textsc{Chapter 4}: M. Briani, L. Caramellino, G. Terenzi, \textit{Weak convergence of Markov chains and numerical schemes to jump diffusion processes}, submitted to publication, 2018.
	% \end{itemize}
	%
	%We now give a brief outline of the main results collected in this thesis. 
	\section*{Part I: American option prices in Heston-type models}
	
	The model introduced by S. Heston in 1993 \cite{H} is 
	one of the most widely used stochastic volatility models in the financial world and it was the starting point for several generalizations. 
	In this model, the dynamics under the pricing measure of the asset price $S$ and the volatility process $Y$ are governed by the stochastic differential equation system
	\begin{equation}\label{Hintro}
	\begin{cases}
	dS_t=(r-\delta)S_tdt + \sqrt{Y_t}S_tdB_t,\qquad&S_0=s>0,\\ dY_t=\kappa(\theta-Y_t)dt+\sigma\sqrt{Y_t}dW_t, &Y_0=y\geq 0,
	\end{cases}
	\end{equation}
	where $B$ and $W$ denote two correlated Brownian motions with 
	$$
	d\langle B,W \rangle_t=\rho dt, \qquad \rho \in (-1,1).
	$$
	Here $r\geq0$ and $\delta\geq0$ are the risk free rate of interest and the continuous dividend rate respectively. The dynamics of the volatility follows a square-root diffusion process, which was originally introduced by E. Feller in 1951 \cite{feller} and then rediscovered by Cox, Ingersoll and Ross as an interest rate model in \cite{cir}. For this reason this process is known in the financial literature as the CIR process.
	The parameters  $\kappa\geq 0$  and $\theta> 0$ are known respectively as the  mean-reversion rate  and the long run state, while the parameter $\sigma>0$ is called the vol-vol (volatility of the volatility). One can observe that the volatility $(Y_t)_t$ tends to fluctuate around the value $\theta$ and that $\kappa$ indicates the velocity of this fluctuation and determines its frequency. This is the mean reversion feature of the CIR process and justifies the names of the constants $\kappa$ and $\theta$.

	It is well known (see, for example, \cite[Section 1.2.4]{Abook}) that under the so called Feller condition $2\kappa\theta \geq \sigma^2$, the process $Y$ with starting condition $Y_0 = y > 0 $ remains always positive. On the other hand,  if the Feller condition is not satisfied, as happens in many cases of practical importance (see e.g. the calibration results in \cite{bm,dps}), $Y$ reaches zero with probability one for any $Y_0 = y \geq 0$. %The latter case  happens  in many cases of practical importance: in equity markets, one often requires large values for the	  	vol-vol $\sigma$  whereas in interest rates context, $\sigma$ is markedly lower (see e.g. the calibration results in \cite{dps} and in \cite{bm} p. 115,  respectively).  So,results in the full parameter regime are actually essential. 
	
	The great success of the 	Heston model is due to the fact that the dynamics of the underlying asset price can take into account the non-lognormal distribution of the asset returns and the observed
	mean-reverting property of the volatility. Moreover, it remains analytically tractable and provides a closed-form valuation formula for vanilla European options using Fourier transform.

	In this framework, the price at time $t\in [0,T]  $ of an American option with payoff function $\varphi$ and maturity $T$ is given by $P(t,S_t,Y_t)$, where
	\begin{equation*}
	P(t,s,y)=\sup_{\tau \in \mathcal{T }_{t,T}}\E\left[e^{-r(\tau-t)}\varphi(S^{t,s,y}_\tau)\right],
	\end{equation*}
	$\mathcal{T}_{t,T}$ being the set of all the stopping times with values in $[t,T]$ and $S^{t,s,y}$ denoting the solution to \eqref{Hintro} with starting condition $S_t=s$, $Y_t=y$.
	
	If we consider, as  usual, the log-price process $X_t=\log S_t$,  the 2-dimensional diffusion $(X,Y)$ has infinitesimal generator given by
	\begin{equation*}
	\mathcal{L}= \frac y 2 \left( \frac{\partial^2}{\partial x^2 } +2 \rho \sigma  \frac{\partial^2}{\partial y \partial x }+  \sigma^2  \frac{\partial^2}{\partial y^2 }  \right) + \left(r-\delta-\frac y 2\right)\frac{\partial}{\partial x}+\kappa(\theta - y)\frac{\partial}{\partial y}
	\end{equation*}
	and defined on the set $\mathcal O=\R\times (0,\infty)$. 
	Note that the differential operator $\L$ has unbounded coefficients and it is not uniformly elliptic: it degenerates on the boundary of $\O$, that is, when the volatility vanishes. This degenerate property gives rise to some technical difficulties when dealing with the theoretical properties of the model, in particular when  the problem of pricing American options is considered. In the first part of this thesis we address some of these issues.
	\subsection*{Chapter \ref{chapter-art1}: Variational formulation of American option prices in Heston type models}
	Chapter \ref{chapter-art1} is devoted to the identification of the American option value function as the unique solution of the associated obstacle problem. Indeed, despite the great success of the Heston model,  as far as we know, an exhaustive analysis  of the  analytic characterization of the value function for American options in Heston-type models is missing in the literature, at least for a large class of  payoff functions which include the standard call and put options. 
	
	Our approach is based on variational inequalities  and extends recent results of Daskalopoulos and Feehan \cite{DF,DF2} and Feehan and Pop \cite{FP} (see also \cite{CMR}). 
	More precisely, we first study  the existence and uniqueness of a weak solution of the associated degenerate parabolic obstacle problem in suitable weighted Sobolev spaces  introduced in \cite{DF} (Section \ref{sect-existenceanduniqueness}). Moreover, we also get a comparison principle. The proof essentially relies on the classical penalization technique (see \cite{BL}), with some technical devices due to the degenerate nature of the problem. 
	
	Once we have  the  existence and uniqueness of an analytical weak solution, in Section \ref{sect-identification} we  identify it with the solution to the optimal stopping problem, that is the American option value function. In order to do this,  we use suitable estimates on the joint distribution of the log-price process and the volatility process. Moreover,  we rely on semi-group techniques and on the affine property of the model.
	
	\subsection*{Chapter \ref{chapter-art2}: American option price properties in Heston type models}
	In Chapter \ref{chapter-art2} we study some qualitative properties of an American option value function in the  Heston model. 
	We first prove in Section \ref{sect-monotony} that, if the payoff function is convex and satisfies some regularity assumptions, then the option value function is increasing with respect to the volatility variable. Then, in Section \ref{sect-put}, we focus on the standard put option, that is we fix the payoff function $\varphi(s)=(K-s)_+$, and we extend to the Heston model some  results which are well known in the Black and Scholes world, mostly by using probabilistic techniques. 
	In particular, in Section \ref{sect-exerciseboundary} we introduce the so called \textit{exercise boundary} or \textit{critical price}, that is the map
	$$
	b(t,y)=\inf\{ s>0\mid P(t,s,y)>(K-s)_+\},\qquad (t,y)\in [0,T)\times [0,\infty),
	$$ 
	and we study some features of this function such as continuity properties.
	Then,  in Section \ref{sect-CIR} we prove that the American put value function is strictly convex with respect to the stock price in  the continuation region, and we do it by using purely probabilistic arguments. In Section \ref{sect-EEP} we extend to the stochastic volatility Heston model the \textit{early exercise premium formula}, that is, we prove that 
	\begin{equation*}
	P(0,S_0,Y_0)=P_e(0,S_0,Y_0)-\int_0^T e^{-rs}\E[(\delta S_s  -rK)\ind{\{S_s\leq b(s,Y_s)\}}]ds,
	\end{equation*} where $P_e(0,S_0,Y_0)$ is the  price at time $0$ of a European put with the same maturity $T$ and strike price $K$ of the original American put with price $P$. Finally, in Section \ref{sect-sf} we prove a weak form  of  the \textit{smooth fit principle}, a well known concept in optimal stopping theory. %hence we prove that
	% 		$$
	% 	\frac{\partial P}{\partial s}(t,b(t,y),y)=\varphi'(b(t,y))=1,
	% 		$$
	% 		and, if the Feller condition is satisfied,
	% 		$$
	% 			\frac{\partial P}{\partial y}(t,b(t,y),y)=0.
	% 		$$
	\section*{Part II: Hybrid schemes for pricing options in jump-diffusion  stochastic volatility models}
	In the second part of this thesis we face up with the problem of the numerical computation of  European and American options prices in jump-diffusion stochastic volatility models. In particular, we consider the Heston model and some  generalizations of it which  have other random sources such as  jumps and a stochastic interest rate (see \cite{bates,HW}).

	From a computational point of view,  the most delicate point  is the treatment of the CIR dynamics for the volatility process in the full parameter regime - it is well known that the standard techniques fail when the square root process is considered.   %It is well known that the standard Euler-Maruyama discretization does not work in this framework. As far as we know, the most accurate simulation schemes in this framework have been introduced by Alfonsi \cite{A}. 
	Moreover, one has to be careful in choosing the approximation method according to the European or American option case.
	In fact, when dealing with European options, i.e. solutions to Partial (Integro) Differential Equation (hereafter P(I)DE) problems, numerical approaches involve tree methods \cite{ads, nv}, Monte Carlo procedures \cite{A-MC,A,AN, and, z}, finite-difference numerical schemes  \cite{ckmz,it, t} or quantization algorithms  \cite{PP}. 
	When American options are considered, that is, solutions to  specific optimal stopping problems or  P(I)DEs with obstacle, it is very useful to consider numerical methods which are able to easily handle dynamic programming principles, for example trees or finite-difference. 
	
	In this thesis we consider a  backward ``hybrid''    algorithm  which  combines:
	\begin{itemize}
		\item finite difference schemes	to handle the jump-diffusion  price process;
		\item Markov chains (in particular, multiple jumps trees) to approximate the other random sources, such as the stochastic volatility and the stochastic interest rate.
	\end{itemize}
	%		In particular, in Chapter \ref{chapter-art3}, we focus on the so called Bates-Hull-White model and, following the previous work in \cite{bcz,bcz-hhw},  we further develop and study the hybrid tree/finite-difference approach and the hybrid Monte Carlo technique in order to numerically evaluate European and American option prices. Then, in Chapter \ref{chapter-art4}, we deal with the theoretical analysis of the hybrid algorithm. We set a reference jump-diffusion model and we study the rate of convergence of the generalized  hybrid procedure.
	\subsection*{Chapter \ref{chapter-art3}: Hybrid  Monte Carlo and tree-finite differences algorithm for pricing options in the Bates-Hull-White model}
	In Chapter \ref{chapter-art3} we focus on the Bates-Hull-White model, where  the volatility $Y$ is a CIR  process and the underlying asset price process $S$ contains a further noise from a jump as introduced by Merton \cite{mer}. Moreover,  the interest rate $r$ is stochastic and evolves according to a generalized Ornstein-Uhlenbeck (hereafter OU) process.
	More precisely,  under the pricing measure, we consider  the following jump-diffusion model:
	\begin{equation*}
	\begin{array}{l}
	\displaystyle\frac{dS_t}{S_{t^-}}= (r_t-\delta)dt+\sqrt{Y_t}\, dZ^S_t+d H_t,
	\smallskip\\
	dY_t= \kappa_Y(\theta_Y-Y_t)dt+\sigma_Y\sqrt{Y_t}\,dZ^Y_t,
	\smallskip\\
	dr_t= \kappa_r(\theta_r(t)-r_t)dt+\sigma_r dZ^r_t,
	\end{array}
	\end{equation*}
	where, as usual, $\delta$ denotes the continuous dividend rate, $S_0,r_0>0$, $Y_0\geq 0$, $Z^S$, $Z^Y$ and $Z^r$ are correlated Brownian motions and $H$ is a compound Poisson process with
	intensity $\lambda$ and i.i.d. jumps $\{J_k\}_k$, that is,
	\begin{equation*}
	H_t=\sum_{k=1}^{K_t} J_k,
	\end{equation*}
	$K$ denoting a Poisson process with intensity $\lambda$. We assume that  the random sources , given by  the Poisson process $K$, the jump amplitudes $\{J_k\}_k$ and the $3$-dimensional correlated Brownian motion $(Z^S,Z^Y,Z^r)$, are independent.
	
	We refer to the introduction of  Chapter \ref{chapter-art3} for an overview on the existing numerical schemes for pricing options in this model.	%We just point out that, from a computational point of view, the most delicate point is the treatment of the CIR dynamics for the volatility process in the full parameter regime.  %It is well known that the standard Euler-Maruyama discretization does not work in this framework.
	%As far as we know, the most accurate simulation schemes in this framework have been introduced by Alfonsi \cite{A}.
	%We assume that the Poisson process $K$, the jump amplitudes $\{J_k\}_k$ and the $3$-dimensional correlated Brownian motion $(Z^S,Z^Y,Z^r)$ are independent. Moreover, we assume that the couple $(Z^Y,Z^r)$ is a standard Brownian motion in $\R^2$ and $Z^S$ is a Brownian motion in $\R$ which is correlated both with $Z^Y$ and $Z^r$:
	%$$
	%d\<Z^S,Z^Y\>_t=\rho_1dt \ \mbox{ and }\ d\<Z^S,Z^r\>_t=\rho_2dt.
	%$$
	
	Our pricing  procedures work as follows. 
	We first  approximate  both the stochastic volatility  and the interest rate processes with a binomial  \avirg multiple jumps" tree approach which is based on the techniques originally introduced in  \cite{nr}. Such a  multiple jumps tree approximation for the CIR process was first introduced and analysed in \cite{acz}, where it is shown   to be reliable and accurate without imposing restrictions on the coefficients. 
	
	Then, we develop two different pricing procedures. 
	In Section \ref{sect-MC} we propose a (forward) Monte Carlo method, based on simulations for the model following the binomial tree in the direction of both the volatility and the interest rate, and a space-continuous approximation for the underlying asset price process coming from a Euler-Maruyama type scheme. 
	
	In Section \ref{sect:htfd}, we describe  a hybrid backward procedure which works following the tree method in the direction of the volatility and the interest rate and a finite-difference approach in order to handle the underlying asset price process. We also give a first theoretical result on this algorithm, studying some stability properties of the procedure.
	
	Finally, Section \ref{sect-numerics} is entirely  devoted to numerical results. Several experiments are provided, both  for European and American options,  with different values of the parameters of the model. In particular, we  also consider   cases in which the Feller condition for the volatility process is not satisfied.  All numerical  results show the reliability, the accuracy and the efficiency of both the Monte Carlo and the hybrid algorithm.
	
	\subsection*{Chapter \ref{chapter-art4}: Weak convergence rate of Markov chains and hybrid numerical schemes for jump-diffusion processes}
	We devote Chapter \ref{chapter-art4} to the  study of  the theoretical convergence of a generalization of the hybrid numerical procedure described in Chapter \ref{chapter-art3}.  Here we just briefly describe our main results,   referring to Section \ref{sect-introC4} for an overview on the existing literature on the rate of convergence of numerical methods for pricing options in Heston-type models. 
	
	Recall that the hybrid algorithm uses  tree approximations and that, in their turn,
	tree methods  rely on Markov chains. So, we first consider in Section \ref{sect-markovapprox} a $d$-dimensional diffusion process
	$(Y_t)_{t\in[0,T]}$ which evolves according to the SDE
	$$
	dY_t=\mu_Y(Y_t)dt+\sigma_Y(Y_t)dW_t.
	$$
	Fix a natural number $N\geq 1$, $h=T/N$ and  assume that $(Y_{nh})_{n=0,\dots,N}$ is approximated by a Markov chain $(Y^h_n)_{n=0,\dots,N}$. It is well known  that the weak convergence  of Markov chains to diffusions relies on assumptions on the local moments of the approximating process up to order 3 or 4. We prove that,  stressing these assumptions, we can study the  rate of  the weak convergence. This analysis is  independent of the financial framework but, as an example, we apply our results to  the multiple jumps tree approximation of the CIR process introduced in \cite{acz} and used in  \cite{bcz,bcz-hhw,bctz}. 
	Let us mention that our general convergence result (Theorem \ref{conv_T})  may in principle be applied to more general trees constructed through the multiple jumps approach by Nelson and  Ramaswamy \cite{nr}, on which the tree in \cite{acz} is based  -- to our knowledge, a  theoretical study of the rate of convergence for such trees is missing in the literature. And it could also be used in other cases, e.g. the recent tree method for the Heston model developed in \cite{ads}.
	
	Then, in Section \ref{sect-hybrid} we  combine the Markov chain approach with other numerical techniques in order to handle the different components in jump-diffusion coupled models. In particular, we link $(Y_t)_{t\in[0,T]}$ with a  jump-diffusion process $(X_t)_{t\in[0,T]}$ which evolves according to a stochastic differential  whose coefficients only depend on the process. In mathematical terms, we consider the stochastic differential equation system 
	$$
	\begin{cases}
	dX_t=\mu_X(Y_t)dt+\sigma_X(Y_t)dB_t+\gamma_X(Y_t)dH_t,\\
	dY_t=\mu_Y(Y_t)dt+\sigma_Y(Y_t)dW_t,
	\end{cases}
	$$
	where $H$ is a compound Poisson process independent of the 2-dimensional Brownian motion $(W,B)$.
	We generalize the hybrid procedure developed in  \cite{bcz,bcz-hhw,bctz} which works backwardly by
	approximating the process $Y$ with a Markov chain and by using a different numerical scheme for solving a (local) PIDE allowing us to work  in the direction of the process $X$.
	We study the speed of convergence of this hybrid approach. The  main difficulty comes from the fact that, in general, the hybrid procedure cannot be directly written on a Markov 	chain, so we cannot apply the convergence results obtained in Section \ref{sect-markovapprox}. Therefore, the idea is to follow the hybrid nature of the procedure: we use classical numerical techniques, that is an analysis of the stability and of the consistency of the method, but in a sense that allows us to exploit the probabilistic properties of the Markov chain approximating the process $Y$. Again, we provide examples from the financial framework, applying our convergence results to the tree-finite difference algorithm in   the Heston or Bates model. 
	\clearpage
	\pagenumbering{arabic}
	\renewcommand{\chaptermark}[1]{\markboth{Chap..\ \textbf{\thechapter}\ -\ \textit{#1}}{}}
	\renewcommand{\sectionmark}[1]{\markright{Sec.\ \textbf{\thesection}\ -\ \textit{#1}}}
	\fancyhf{} \fancyfoot[CE,CO]{\thepage} \fancyhead[CO]{\rightmark}
	\fancyhead[CE]{\leftmark}
	\renewcommand{\headrulewidth}{0.5pt}
	\renewcommand{\footrulewidth}{0.0pt}
	\addtolength{\headheight}{0.5pt}
	\fancypagestyle{plain}{\fancyhead{}\renewcommand{\headrulewidth}{0pt}}
	\linespread{1.2}
	\part{American option prices in Heston-type models}

	\chapter{Variational formulation of American option prices}\label{chapter-art1}

	\section{Introduction}
	The Heston model is the most celebrated stochastic volatility model in the financial world.
	%	The great success of	  	Heston model is due to the fact that the dynamics of the underlying asset can take into account the non-lognormal distribution of the asset returns and the observed
	%		  	mean-reverting property of the volatility. Moreover, it remains analytically tractable and provides a closed-form valuation formula for European options using Fourier transform. 
	%These features have  called for 
	As a consequence, there is
	an extensive literature on numerical methods to price derivatives in Heston-type models. In this framework, besides purely probabilistic methods such as standard Monte Carlo and tree approximations,  there is a large class of algorithms which  exploit numerical analysis techniques in order to solve the standard PDE (resp. the obstacle problem) formally associated with the European (resp. American) option price  function. 
	However, these algorithms  have, in general, little mathematical support and in particular, as far as we know, a rigorous and complete study of the  analytic characterization of the American price function is not present in the literature.
	
	The main difficulties in this sense come from the degenerate nature of the model. In fact, the infinitesimal generator associated with the two dimensional diffusion given by the log-price process and the volatility process is not uniformly elliptic: it degenerates on the boundary of the domain, that is when the volatility  variable vanishes. Moreover, it has unbounded coefficients with linear growth. Therefore, the existence and the uniqueness of the solution to the pricing PDE and obstacle problem do not follow from the classical theory, at least in the case in which the boundary of the state space is reached with positive probability, as happens in many cases of practical importance (see \cite{andersen}). Moreover, the probabilistic representation of the solution, that is the identification with the price function, is far from trivial in the case of non regular payoffs. 
	
	It should be emphasized that a clear analytic characterization of the price function allows not only  to formally justify  the   theoretical convergence of some classical pricing algorithms but also to investigate the regularity properties of the price function (see  \cite{JLL} for the case of the Black and Scholes models).

	Concerning the existing literature, E. Ekstrom and J. Tysk in  \cite{ET} give a rigorous and complete analysis of these issues in the case of European options, proving that, under some regularity assumptions on the payoff functions, the price function is the unique classical solution of the associated PDE with a certain boundary behaviour for vanishing values of the volatility.  However, the payoff functions they consider do not include the case of standard put and call options. 
	
	Recently,  P. Daskalopoulos and P. Feehan  in \cite{DF,DF2} studied the existence, the uniqueness, and some regularity properties of the solution of this kind of degenerate PDE and obstacle problems in the elliptic case, introducing suitable weighted Sobolev spaces which clarify the behaviour of the solution near the degenerate boundary (see also \cite{CMR}). In another paper (\cite{FP}) P. Feehan and C. Pop addressed the issue of the probabilistic representation of the solution,  but we do not know if their assumptions on the solution of the parabolic obstacle problem are satisfied in the case of standard American options. 
	Note that Feehan and Pop did prove regularity results in the elliptic case, see \cite{FP2}. They also announce results for the parabolic case in \cite{FP}.
	
	The aim of this chapter is to give a precise analytical characterization of the American option price function in the Heston model for a large class of payoffs which includes the standard put and call options. In particular, we give a variational formulation of the American pricing problem using the weighted Sobolev spaces and the bilinear form introduced in \cite{DF}. 
	
	The chapter is organized as follows. In Section 2, we introduce our notations and we state our main results. Then, in Section 3, we study the existence and uniqueness of the solution of the associated variational inequality, extending the results obtained in \cite{DF} in the elliptic case. The proof  relies, as in \cite{DF}, on the classical penalization technique introduced by Bensoussan and Lions \cite{BL} with some technical devices due to the degenerate nature of the problem. We also establish a Comparison Theorem. Finally, in section 4, we prove that the solution of the variational inequality  with obstacle function $\psi$ is actually the American option price function with payoff $\psi$, with  conditions on $\psi$ which are satisfied, for example, by the standard call and put options. In order to do this, we use  the affine property  of the underlying diffusion given by the log price process $X$ and the volatility process $Y$. Thanks to this property, we first identify the analytic semigroup associated with the bilinear form with a correction term  and the transition semigroup of the pair $(X,Y)$ with a killing term. Then, we prove regularity results on the solution of the variational inequality and suitable estimates on the joint law  of the process $(X,Y)$ and we deduce from them the analytical characterization of the solution of the optimal stopping problem, that is the  American option price.
	
	\section{Notations and main results}
	\subsection{The Heston model}
	We recall that in the Heston model the dynamics under the pricing measure of the asset price $S$ and the volatility process $Y$ are governed by the stochastic differential equation system
	\begin{equation*}
	\begin{cases}
	\frac{dS_t}{S_t}=(r-\delta)dt + \sqrt{Y_t}dB_t,\qquad&S_0=s>0,\\ dY_t=\kappa(\theta-Y_t)dt+\sigma\sqrt{Y_t}dW_t, &Y_0=y\geq 0,
	\end{cases}
	\end{equation*}
	where $B$ and $W$ denote two correlated Brownian motions with 
	$$
	d\langle B,W \rangle_t=\rho dt, \qquad \rho \in (-1,1).
	$$
	We exclude  the  degenerate case $\rho=\pm 1$, that is the case in which the same Brownian motion drives the dynamics of $X$ and $Y$. Actually, it can be easily seen that, in this case, $S_t$ reduces to  a function of the pair $\left(Y_t,\int_0^tY_s ds\right)$ and the resulting degenerate model  cannot be treated with the techniques we develop in this chapter. Moreover, this particular situation is not very interesting from a financial point of view.
	
	Moreover, we recall that $r\geq0$ and $\delta\geq0$ are respectively  the risk free rate of interest and the continuous dividend rate. The dynamics of $Y$ follows a CIR process with mean reversion rate $\kappa>0$,  long run state $\theta>0$ and  volatility of the volatility $\theta>0$. We stress that  we do not require the Feller condition $2\kappa\theta \geq \sigma^2$: the volatility process $Y$  can hit $0$ (see, for example, \cite[Section 1.2.4]{Abook}). 
	
	We are interested in studying the price of an American option with payoff function $\psi$. For technical reasons which will be clarified later on, hereafter we consider the process
	\begin{equation}\label{traslation}
	X_t=\log S_t-\bar{c}t,\qquad 	\mbox{with }	\bar{c}=r-\delta- \frac{\rho\kappa\theta}{\sigma},
	\end{equation}
	which satisfies
	\begin{equation}\\
	\begin{cases}\label{hest}
	dX_t=\big(\frac{\rho\kappa\theta}{\sigma}-\frac {Y_t} 2 \big)dt + \sqrt{Y_t}dB_t,\\dY_t=\kappa(\theta-Y_t)dt+\sigma\sqrt{Y_t}dW_t .
	\end{cases}
	\end{equation}

	Note that, in this framework,  we have to consider payoff functions  $\psi$ which depend on both the time and the space variables. For example, in the case of a standard put option (resp. a call option) with strike price $K$ we have $ \psi(t,x)=(K-e^{x+\bar c t})_+ $ (resp. $ \psi(t,x)=(e^{x+\bar c t}-K)_+ $). So, the natural price at time $t$ of an American option with a nice enough payoff $(\psi(t,X_t,Y_t))_{0\leq t\leq T}$ is given by $P(t,X_t,Y_t)$, with
	\begin{equation*}
	P(t,x,y)=\sup_{\theta \in\mathcal{T}_{t,T}}\E[e^{-r(\theta-t )}\psi(\theta,{X}^{t,x,y}_\theta, Y^{t,y}_\theta)],
	\end{equation*}  
	where $\mathcal{T}_{t,T}$ is the set of all stopping times with values in $[t,T]$ and $(X^{t,x,y}_s, Y^{t,y}_s)_{t\leq s\leq T}$ denotes the solution to \eqref{hest} with the starting condition $(X_t,Y_t) = (x,y)$.
	
	Our aim is to give an analytical characterization of the price function $P$.
	%	Without loss of generality hereafter we assume $r=0$. 
	In this chapter we denote by $\L$ the infinitesimal generator of the two dimensional diffusion $(X,Y)$, given by
	\begin{equation*}
	\mathcal{L}= \frac y 2 \left( \frac{\partial^2}{\partial x^2 } +2 \rho \sigma  \frac{\partial^2}{\partial y \partial x }+  \sigma^2  \frac{\partial^2}{\partial y^2 }  \right) + \left(\frac{\rho\kappa\theta}{\sigma}-\frac y 2\right)\frac{\partial}{\partial x}+\kappa(\theta - y)\frac{\partial}{\partial y} ,
	\end{equation*}
	which is	defined on the open set $\mathcal{O}:= \R \times (0,\infty) $. Note that  $\mathcal{L}$ has unbounded coefficients and  is not uniformly elliptic: it degenerates on the boundary $\partial\mathcal{O}=\R\times\{0\}$.
	\subsection{American options and variational inequalities}
	\subsubsection{Heuristics}
	From the optimal stopping theory, we know that the discounted price process $\tilde{P}(t,X_t,Y_t)=e^{-rt}P(t,X_t,Y_t)$ is a supermartingale  and that its 
	finite variation part only decreases on the set $P=\psi$ with respect to the time variable $t$. We want to have an analytical interpretation of these features on the function $P(t,x,y)$.
	%		We introduce the partial differential operator 
	%	\begin{equation}\label{Hestonoperator}
	%	\mathcal{L}u:=(\mathcal{L}-r)u.
	%	\end{equation}
	So, assume that $P \in\mathcal{C}^{1,2} ((0,T)\times \mathcal{O})$. Then, by applying It\^{o}'s formula, the finite variation part of $\tilde{P}(t,X_t,Y_t)$ is 
	$$
	\left( \frac{\partial \tilde P }{\partial t } + \mathcal{L}\tilde P\right)(t,X_t,Y_t) .
	$$
	Since $\tilde{P}$ is a supermartingale, we can deduce the inequality 
	$$
	\frac{\partial \tilde P }{\partial t } + \mathcal{L}\tilde P\leq 0
	$$
	and, since its  finite variation part decreases only on the set $P(t,X_t,Y_t)=\psi(t,X_t,Y_t)$, we can write
	$$
	\left(  \frac{\partial \tilde P }{\partial t } + \mathcal{L}\tilde P\  \right) (\psi- P)=0.
	$$
	This relation has to be satisfied $dt-a.e.$ along the trajectories of $(t,X_t,Y_t)$. Moreover, we have the two trivial conditions $P(T,x,y)=	\psi(T,x,y)$ and $P\geq \psi$.
	
	The previous discussion is only heuristic, since the price function $P$ is not regular enough to apply  It\^{o}'s formula. However, it suggests the following strategy:
	\begin{enumerate}
		\item Study the obstacle problem
		\begin{equation} \label{system1}
		\begin{cases}
		\frac{\partial u }{\partial t } + \mathcal{L}u\leq 0, \qquad u\geq \psi, \qquad &in \  [0,T]\times \mathcal{O},\\
		\left(  \frac{\partial u }{\partial t } + \mathcal{L}u  \right) (\psi-u)=0, \qquad &in \  [0,T]\times \mathcal{O},\\
		u(T,x,y)=\psi(T,x,y).
		\end{cases}
		\end{equation}
		\item Show that the discounted price function $\tilde P$ is equal to the solution of \eqref{system1} where $\psi$ is replaced by $\tilde \psi(t,x,y)=e^{-rt}\psi(t,x,y)$.
	\end{enumerate}
	We will follow this program providing a variational formulation of system \eqref{system1}.
	\subsubsection{Weighted Sobolev spaces and  bilinear form associated with the Heston operator} \label{sect-sobolev}
	We consider the measure first introduced in \cite{DF}:
	\begin{equation*}\label{m}
	\mathfrak{m}_{\gamma, \mu}(dx,dy)=y^{\beta - 1}e^{-\gamma |x|- \mu y } dx dy, 
	\end{equation*}
	with $\gamma>0, \ \mu >0 \mbox{ and } \beta:= \frac{2 \kappa \theta}{\sigma^2}$.

	It is worth noting  that in \cite{DF} the authors fix  $\mu=\frac{2\kappa}{\sigma^2}$ in the definition of the measure $\m_{\gamma,\mu}$. This specification will not be necessary in this chapter, but it is useful to mention it in order to better understand how this measure arises. In fact, recall that the density of the speed measure of the CIR process is given by $y^{\beta-1}e^{-\frac{2\kappa}{\sigma^2}y}$. Then,  the term $y^{\beta-1}e^{ -\frac{2\kappa}{\sigma^2}y}$ in the definition of $\m_{\gamma,	\mu}$ has a clear probabilistic interpretation, while the exponential term $e^{-\gamma |x|}$ is classically introduced  just to deal with the unbounded domain in the $x-$component.

	%It will be clear later on that this measure  in some sense describes the qualitative behaviour of the process $(X,Y)$ (see Section \ref{sect-estimjointlaw}). 

	For $u\in \R^n$ we denote by $|u|$ the standard Euclidean norm of $u$ in $\R^n$. Then, we recall the weighted  Sobolev spaces introduced in \cite{DF}. The choice of these particular Sobolev spaces will  allow us to formulate the obstacle problem \eqref{system1} in a variational framework with respect to the measure $\m_{\gamma,\mu}$. 
	\begin{definition}\label{def-sob}
		For every $p\geq 1$, let $L^p(\mathcal{O},\mathfrak{m}_{\gamma, \mu})$ be the space of all Borel measurable functions $u: \mathcal{O} \rightarrow \R $ for which 
		$$
		\Vert u \Vert^p_{L^p(\mathcal{O},\mathfrak{m}_{\gamma, \mu})} := \int_{\mathcal{O}} |u|^p d\mathfrak{m}_{\gamma, \mu}  < \infty, 
		$$
		%	$$
		%	\Vert u \Vert^\infty_{L^\infty(\mathcal{O},\mathfrak{m}_{\gamma, \mu})} := \int_{\mathcal{O}} \sup_{(x,y)\in\O} |u| d\mathfrak{m}_{\gamma, \mu}  < \infty, \qquad \mbox{ if }p= \infty,
		%	$$
		and denote $H^0(\mathcal{O},\mathfrak{m}_{\gamma, \mu}):=L^2(\mathcal{O},\mathfrak{m}_{\gamma, \mu}).$
		
		\begin{enumerate}
			\item If $\nabla u:=(u_x,u_y)$ and $u_x$, $u_y$ are defined in the sense of distributions, we set 
			$$
			H^1(\mathcal{O},\mathfrak{m}_{\gamma, \mu}):=\{  u \in  L^2(\mathcal{O},\mathfrak{m}_{\gamma, \mu}): \sqrt{1+y}  u \mbox{ and }  \sqrt{y}|\nabla  u| \in L^2(\mathcal{O},\mathfrak{m}_{\gamma, \mu}) \},
			$$
			and 
			$$
			\Vert u \Vert^2_{H^1(\mathcal{O},\mathfrak{m}_{\gamma, \mu})} := \int_{\mathcal{O}} \left(  y|\nabla  u|^2 + (1+y)u^2 \right) d\mathfrak{m}_{\gamma, \mu}.
			$$
			\item If $D^2u:=(u_{xx},u_{xy},u_{yx}, u_{yy})$ and all derivatives of $u$ are defined in the sense of distributions, we set 
			$$
			H^2(\mathcal{O},\mathfrak{m}_{\gamma, \mu}):=\{  u \in  L^2(\mathcal{O},\mathfrak{m}_{\gamma, \mu}): \sqrt{1+y}  u, \ (1+y)|\nabla  u|, \   y|D^2u| \in L^2(\mathcal{O},\mathfrak{m}_{\gamma, \mu}) \}
			$$
			and 
			$$
			\Vert u \Vert^2_{H^2(\mathcal{O},\mathfrak{m}_{\gamma, \mu})} := \int_{\mathcal{O}} \left(  y^2|D^2u|^2 + (1+y)^2|\nabla  u|^2+(1+y)u^2 \right) d\mathfrak{m}_{\gamma, \mu} .
			$$
		\end{enumerate}
	\end{definition}
	%\begin{remark}
	%Note that the choice of the measure $\m_{\gamma, \mu}$ ensures that the constant functions and the polynomials are integrable in $\O$.
	%		 
	%
	%\end{remark}
	
	For brevity and when the context is clear, we shall often denote
	$$
	H:= H^0(\mathcal{O}, \mathfrak{m}_{\gamma, \mu}), \qquad V:= H^1(\mathcal{O}, \mathfrak{m}_{\gamma, \mu}) 
	$$
	and
	$$\Vert
	u\Vert_H :=\Vert u \Vert_{L^2(\mathcal{O}, \mathfrak{m}_{\gamma, \mu}) }, \qquad \Vert u \Vert_V :=\Vert u \Vert_{H^1(\mathcal{O}, \mathfrak{m}_{\gamma, \mu}) }.
	$$
	Note that we have the inclusion 
	$$
	H^2(\mathcal{O}, \mathfrak{m}_{\gamma, \mu}) \subset H^1(\mathcal{O}, \mathfrak{m}_{\gamma, \mu})
	$$
	and that the spaces $H^k(\mathcal{O}, \mathfrak{m}_{\gamma, \mu})$, for $k=0,1,2$ are Hilbert spaces with the inner products
	$$
	(u,v)_H=	(u,v)_{L^2(\mathcal{O}, \mathfrak{m}_{\gamma, \mu}) }=\int_{\mathcal{O}}^{}uv d \m_{\gamma, \mu} ,
	$$ 
	$$
	(u,v)_V=(u,v)_{H^1(\mathcal{O}, \mathfrak{m}_{\gamma, \mu})} =\int_{\mathcal{O}}^{}\left(y \left( \nabla  u,\nabla  v\right) + (1+y)uv \right)d\m_{\gamma, \mu} 
	$$
	and
	$$
	(u,v)_{H^2(\mathcal{O}, \mathfrak{m}_{\gamma, \mu})} :=\int_{\mathcal{O}}^{}\left(y^2\left( D^2u,D^2v\right) + (1+y)^2 \left( \nabla  u,\nabla  v\right)+ (1+y)uv \right) d \m_{\gamma, \mu},
	$$
	where $(\cdot,\cdot)$ denotes the standard scalar product in $\R^n$.
	
	Moreover, for every  $T>0, \,p\in [1,+\infty)$ and $i=0,1,2$, we set
	\begin{align*}
	L^p([0,T];H^i(\O,\m_{\gamma,\mu}))=\bigg\{& u:[0,T]\times \O\rightarrow \R \mbox{ Borel measurable}: u(t,\cdot,\cdot) \in H^i(\O,\m_{\gamma,\mu})\\&\mbox{ for a.e. } t\in [0,T] \mbox{ and }	\int_0^T \Vert u(t,\cdot.\cdot)\Vert^p_{H^i(\O,\m_{\gamma,\mu})}dt<\infty\bigg\}
	\end{align*}
	%\begin{align*}
	%L^\infty([0,T];H^i(\O,\m_{\gamma,\mu}))=\bigg\{& u:[0,T]\times \O\rightarrow \R : u(t,\cdot,\cdot) \in H^i(\O,\m_{\gamma,\mu})\mbox{ for a.e. } t\in [0,T] \\&\mbox{ and } \sup_{t\in[0,T]}\Vert u(t,\cdot.\cdot)\Vert_{H^i(\O,\m_{\gamma,\mu})}<\infty\bigg\},\qquad\mbox{ if }p= \infty,
	%\end{align*}
	and
	$$
	\Vert u\Vert^p_{L^p([0,T];H^i(\O,\m_{\gamma,\mu}))}= 	\int_0^T \Vert u(t,\cdot.\cdot)\Vert^p_{H^i(\O,\m_{\gamma,\mu})}dt.
	$$
	%		$$
	%		\Vert u\Vert_{L^\infty([0,T];H^i(\O,\m_{\gamma,\mu}))}= \sup_{t\in[0,T]}\Vert u(t,\cdot.\cdot)\Vert_{H^i(\O,\m_{\gamma,\mu})}, \qquad\mbox{ if }p= \infty.
	%		$$
	We also define $L^\infty([0,T];H^i)$ with the usual essential sup norm.
	
	We can now introduce the following bilinear form.
	\begin{definition}
		For any $u,v \in H^1(\mathcal{O}, \mathfrak{m}_{\gamma, \mu})$ we define the bilinear form
		\begin{align*}
		a_{\gamma, \mu}(u,v) 
		=&  \frac 1 2 \int_{\mathcal{O}} y\left(  u_xv_x(x,y)  +\rho\sigma u_xv_y(x,y) + \rho\sigma u_yv_x(x,y) +\sigma^2 u_y v_y(x,y ) \right)d\m_{\gamma, \mu} \\
		&+\into y \left(  j_{\gamma,\mu}(x)u_x(x,y) + k_{\gamma,\mu}  (x)u_y(x,y) \right) v(x,y) d\m_{\gamma, \mu},
		\end{align*}
		where
		\begin{equation}\label{JandGamma}
		j_{\gamma,\mu}(x)=\frac 1 2 \left( 1  -\gamma \mbox{sgn}(x)-\mu\rho\sigma \right), \qquad k_{\gamma,\mu}(x)=\kappa-\frac{\gamma\rho\sigma}{2} \mbox{sgn}(x) -\frac{\mu\sigma^2}{2}  . 
		\end{equation}
	\end{definition}
	We will prove that $a_{\gamma,\mu}$ is the bilinear form associated with the operator $\L$, in the sense that for every $ u\in H^2(\mathcal{O},\m_{\gamma,\mu})$ and for every $ v\in H^1(\mathcal{O},\m_{\gamma,\mu})$, we have
	\begin{equation*}
	(\mathcal{L}u,v)_H=-a_{\gamma, \mu}(u,v).
	\end{equation*}
	In order to simplify the notation,  for the rest of this chapter we will write   $	\m$ and $a(\cdot,\cdot)$ instead of $\m_{\gamma, \mu}$ and $a_{\gamma, \mu}(\cdot,\cdot)$ every time the dependence on $\gamma$ and $\mu$ does not play a role in the analysis and computations.
	\subsection{Variational formulation of the American price}
	Fix $T>0$. We consider an assumption on the payoff function $\psi$ which will be crucial in the discussion of the penalized problem.
	
	\medskip
	\noindent	
	\textbf{Assumption $\mathcal{H}^1$.}
	We say that a function $\psi$ satisfies Assumption $\mathcal{H}^1$ if	$\psi\in \mathcal{C}([0,T];H) $, $\sqrt{1+y}\psi \in L^2([0,T];V)$, $\psi(T)\in V$ and there exists  $\Psi\in L^2([0,T];V)$ such that $\left| \frac{\partial \psi}{\partial t} \right|\leq \Psi$. 
	
	\medskip
	
	We will also need a domination condition on $\psi$ by a function $\Phi$ which satisfies the following assumption. 
	
	\medskip
	\noindent	
	\textbf{Assumption $\mathcal{H}^2$.}
	We say that a function $\Phi\in L^2([0,T];H^2(\mathcal{O},\m))$ satisfies Assumption $\mathcal{H}^2$   if  
	$(1+y)^{\frac 32}\Phi\in L^2([0,T];H)$, $\frac{\partial \Phi}{\partial t}+ \mathcal{L}\Phi \leq 0$ and $\sqrt{1+y}\Phi \in L^\infty([0,T]; L^2(\mathcal{O}, 	\m_{\gamma,\mu'}))$ for some $0<\mu'<\mu$.
	
	\medskip
	The domination condition is needed to deal with the lack of coercivity of the bilinear form associated with our problem. Similar conditions are also used in \cite{DF}.
	
	The first step in the variational formulation of the problem is to introduce the associated variational inequality and 
	to prove the following existence and uniqueness result.
	\begin{theorem}\label{variationalinequality}
		Assume that $	\psi$ satisfies Assumption $\mathcal{H}^1$ together with $0\leq \psi\leq \Phi$, where $\Phi$ satisfies Assumption $\mathcal{H}^2$.
		%Let $\psi:[0T]\rightarrow H $ be a continuous function such that $\sqrt{1+y}\psi \in L^2([0,T];V)$ and  $\left| \frac{\partial \psi}{\partial t} \right|\leq \Psi$ with $\Psi\in L^2([0,T];V)$. 
		%	Moreover, assume that $0\leq \psi \leq \Phi$ where $\Phi:[0,T]\rightarrow H^2(\mathcal{O},\m)$ is such that $(1+y)^{\frac 32}\Phi\in L^2([0,T];H)$, $\frac{\partial \Phi}{\partial t}+ \mathcal{L}\Phi \leq 0$ and $\sqrt{1+y}\Phi \in L^\infty([0,T]; L^2(\mathcal{O}, 	\m_{\gamma,\mu'}))$ for some $0<\mu'<\mu$ .
		%	Assume that there exists a function $\Phi:[0,T]\rightarrow H^2(\mathcal{O},\m)$ such that $(1+y)^{\frac 32}\Phi\in L^2([0,T];H)$, $\frac{\partial \Phi}{\partial t}+ \mathcal{L}\Phi \leq 0$ and $0\leq \psi \leq \Phi$. Moreover, let us assume that there is $0<\mu'<\mu$ such that $\sqrt{1+y}\Phi \in L^2(\mathcal{O}, 	\m_{\gamma,\mu'})$. 
		Then, there exists a unique function $u$ such that $	u\in\mathcal{C}([0,T];H)\cap L^2([0,T];V),\,\frac{\partial u}{\partial t } \in L^2([0,T];  H)$ and 
		\begin{equation}\label{VI}
		\begin{cases}
		-\left( \frac{\partial u}{\partial t },v -u  \right)_H + a(u,v-u)\geq 0, \quad \mbox{a.e. in } [0,T] \quad v\in L^2([0,T];V), \ v\geq \psi,\\
		u\geq \psi \mbox{ a.e. in } [0,T]\times \R \times (0,\infty),\\
		u(T)=\psi(T),\\
		0\leq u \leq \Phi.
		\end{cases}
		\end{equation}
	\end{theorem}
	The proof is presented in Section 3 and essentially relies on the penalization technique introduced by Bensoussan and Lions  (see also \cite{F})  with some technical devices due to the degenerate nature of the problem. We extend in the parabolic framework the  results obtained in \cite{DF} for the elliptic case.  
	
	The second step is to identify the unique solution of the variational inequality \eqref{VI} as the solution of the optimal stopping problem, that is the (discounted) American option price. In order to do this, we consider the following  assumption on the payoff function.
	
	\medskip
	\noindent	
	\textbf{Assumption $\mathcal{H}^*$.}
	We say that a function $\psi:[0,T]\times \R\times [0,\infty)\rightarrow \R$ satisfies Assumption $\mathcal{H}^*$   if  $\psi$ is continuous and there exist constants $C>0$ and $L\in \left[0,\frac{2\kappa}{\sigma^2}\right)$ such that, for all $(t,x,y)\in[0,T]\times \R\times [0,\infty)$,
	\begin{equation}\label{boundonpsi}
	0\leq \psi(t,x,y)\leq C(e^{x}+e^{Ly}),
	\end{equation}
	and
	\begin{equation}\label{boundonderivatives}
	\left|\frac{\partial \psi}{\partial t}(t,x,y)\right|+\left|\frac{\partial \psi}{\partial x}(t,x,y)\right|+ \left|\frac{\partial \psi}{\partial y}(t,x,y)\right|\leq C(e^{a|x|+by}),
	\end{equation}
	for some $a,b\in \R$.
	\medskip

	Note that  the payoff functions of a standard  call  and put option with strike price $K$ (that is, respectively, $\psi=\psi(t,x)=(K-e^{x+\bar c t})_+ $ and $\psi=\psi(t,x)=(e^{x+\bar c t}-K)_+$) satisfy Assumption $\mathcal H^*$.
	Moreover, it is easy to see that, if $\psi$ satisfies Assumption $\mathcal{H}^*$, then it is possible to choose $\gamma$ and $\mu$ in the definition of the measure $\m_{\gamma,\mu}$  (see \eqref{m})  such that $\psi$ satisfies  the assumptions of Theorem \ref{variationalinequality}. Then, for such $\gamma$ and $\mu$, we get the following identification result.
	\begin{theorem}	\label{theorem2}
		Assume that $\psi$ satisfies Assumption $\mathcal{H}^*$.
		Then, the solution $u$ of the variational inequality \eqref{VI} associated with $\psi$ is continuous and coincides with  the function $u^*$ defined by 
		\begin{equation*}\label{americanprice}
		u^*(t,x,y)= \sup_{\tau \in \mathcal{T}_{t,T}}\E \left[   \psi(\tau,X_\tau^{t,x,y}, Y_\tau^{t,x,y})     \right].
		\end{equation*}

	\end{theorem}

	\section{Existence and uniqueness of solutions to the variational inequality}\label{sect-existenceanduniqueness}
	
	\subsection{Integration by parts and energy estimates}
	The following result justifies the definition of the bilinear form $a$.
	\begin{proposition}\label{Prop_integrationbyparts}
		If $u\in H^2(\mathcal{O},\m)$ and $v\in H^1(\mathcal{O},\m)$, we have
		\begin{equation}\label{A=a}
		(\mathcal{L}u,v)_H=-a(u,v).
		\end{equation}
	\end{proposition}
	This result is proved with the same arguments of  \cite[Lemma 2.23]{DF} or \cite[Lemma A.3]{DF2} but  we prefer to repeat here  the proof  since it clarifies why we have considered the process $X_t=\log S_t- \bar c t$ instead of the standard log-price process $\log S_t$. 
	
	Before proving Proposition \ref{Prop_integrationbyparts}, we show some preliminary results. The first one is  about the standard regularization of a function by convolution.
	
	\begin{lemma}\label{lemmaIPP}
		Let $\varphi:\R\times\R \rightarrow \R^+ $ be a $C^\infty$ function  with compact support in $[-1,+1]\times [-1,0]$ and such that $\int\int \varphi(x,y)dxdy=1$. For $j\in\N$ we set $\varphi_j(x,y)=j^2\varphi(jx,jy  )$. Then, for every function u locally square-integrable on $\R\times(0,\infty)$ and for every compact set $K$, we have
		$$
		\lim_{j\rightarrow \infty } \int\!\!\! \int_K(\varphi_j\ast u-u)^2(x,y)dxdy=0.
		$$	\end{lemma}
	\begin{proof}
		We first observe that, by using Jensen's inequality with respect to the measure $ \varphi_j(\xi,\zeta)d\xi d\zeta$, we get
		\begin{align*}
		\int \!\!\!\int_K(\varphi_j\ast u )^2(x,y)dxdy &\leq \int\!\!\!\int_Kdxdy\int\!\!\!\int \varphi_j(\xi,\zeta)u^2(x-\xi,y-\zeta)d\xi d\zeta\\
		&=\int\!\!\!\int \varphi_j(\xi,\zeta)d\xi d\zeta\int\!\!\!\int \mathbbm{1}_K(x+\xi,y+\zeta )u^2(x,y)dxdy.
		\end{align*}
		We deduce, for $j$ large enough,
		$$
		\int\!\!\! \int_K(\varphi_j\ast u )^2(x,y)dxdy \leq 	\int\!\!\! \int_{\bar{K}} u^2(x,y)dxdy,
		$$
		where $\bar{K}=\{ (x,y) \in \O | d_\infty\big((x,y),K) \leq \frac 1 j\}$. Let $\epsilon$ be a positive constant and $v$ be a continuous function such that $\int \!\!\! \int_{\bar{K}} (u(x,y)-v(x,y))^2dxdy \leq \epsilon $. By using the well known inequality $(x_1+\dots+x_l)^2\leq l(x_1^2+\dots+x_l^2)$, we have
		\begin{align*}
		&	\int\!\!\int_K(\varphi_j\ast u-u)^2(x,y)dxdy\\&\leq 3 \int\!\!\int_K(\varphi_j\ast u-\varphi_j \ast v )^2(x,y)dxdy +3\int\!\!\int_K(\varphi_j\ast v-v)^2(x,y)dxdy \\&\qquad+ 3 \int \!\!\int_K(v-u)^2(x,y)dxdy
		\\&\leq 3\left(\int\!\!\int_{\bar{K}}(v-u)^2(x,y)dxdy
		+  \int\!\!\int_K(\varphi_j\ast v-v)^2(x,y)dxdy + \int\!\!\int_{\bar{K}}(v-u)^2(x,y)dxdy\right)
		\\& \leq 6\epsilon +3 \int\!\!\int_K (\varphi_j\ast v-v)^2(x,y)dxdy .
		\end{align*}
		Since $v$ is continuous, we have $|\varphi_j\ast v|\leq \sup_{x,y\in \bar{K}}|v(x,y)|$ and $\lim_{j\rightarrow \infty }\varphi_j\ast v(x,y)=v(x,y)$ on $K$. Therefore, by Lebesgue Theorem,  we can pass to the limit in the above inequality and we get
		$$
		\limsup_{j\rightarrow\infty}\int\!\!\!\int_K(\varphi_j\ast u-u)^2(x,y)dxdy\leq 6\epsilon,
		$$
		which completes the proof.
	\end{proof}
	Then, the following two propositions justify the integration by parts formulas with respect to the measure $\m$.
	%	Before proving Proposition \ref{Prop_integrationbyparts}, we state two preliminary results which justify the integration by parts formulas with respect to the measure $\m$. We 
	\begin{proposition}\label{IPP1}
		Let us consider $u,v: \mathcal{O} \rightarrow \R$ locally square-integrable on $\mathcal{O}$, with derivatives $u_x$ and $v_x$ locally square-integrable on $\mathcal{O}$ as well.  Moreover, assume that
		$$
		\into \big(| u_x(x,y)v(x,y)|+ | u(x,y)v_x(x,y)|+ |u(x,y)v(x,y)| \big) d\m <\infty.
		$$ Then, we have
		\begin{equation}\label{IPPx}
		\into u_x(x,y) v(x,y)d\m=-\int_\mathcal{O} u(x,y) \left(v_x(x,y)-\gamma sgn(x)v\right)d\m.
		\end{equation}
	\end{proposition}
	\begin{proof}
		First we assume that $v$ has compact support in $\R \times (0,\infty)$. For any $j\in \N$ we consider the $C^\infty$ functions $u_j=\varphi_j\ast u$ and $v_j=\varphi_j\ast v$, with $\varphi_j$ as in Lemma \ref{lemmaIPP}. Note that $\mbox{supp }  v_j \subset \mbox{supp}  \ v + \mbox{supp }  \varphi_j$ and so, for $j$ large enough, $\mbox{supp }  v_j \subset \R \times (0,\infty)$. For any $\epsilon >0$, integrating  by parts, we have
		$$
		\int_{-\infty }^\infty \!\!(u_j)_x(x,y)v_j (x,y)e^{-\gamma \sqrt{x^2+\epsilon}}dx=-\int_{-\infty}^{\infty } \!\!u_j\left( (v_j)_x(x,y)-\gamma \frac{ x}{\sqrt{x^2+\epsilon}}v_j(x,y)   \right) e^{-\gamma \sqrt{x^2+\epsilon}}dx,
		$$
		and, letting $\epsilon \rightarrow 0 $,
		$$
		\int_{-\infty }^\infty (u_j)_x(x,y)v_j (x,y)e^{-\gamma |x|}dx=-\int_{-\infty}^{\infty } u_j \big( (v_j)_x(x,y)-\gamma sgn(x) v_j(x,y) \big) e^{-\gamma |x|}dx.
		$$
		Multiplying by $y^{\beta-1}e^{-\mu y}$ and integrating in $y$ we obtain 
		$$
		\into (u_j)_x(x,y)v_j (x,y)d\m=- \into u_j(x,y) \big( (v_j)_x(x,y)-\gamma sgn(x) v_j (x,y)\big) d\m.
		$$
		Recall that, for $j$ large enough, $v_j$ has compact support in $\R \times (0,\infty)$ and $\m$ is bounded on this compact. By using Lemma \ref{lemmaIPP}, letting $j\rightarrow \infty$ we get
		$$
		\into u_x(x,y)v (x,y)d\m=- \into u\big( v_x(x,y)-\gamma sgn(x) v(x,y\big) d\m.
		$$
		Now let us consider the general case of a function $v$ without compact support. We introduce a $C^\infty-$function $\alpha$ with values in $[0,1]$,  $\alpha(x,y)=0$ for all $(x,y) \notin [-2,+2]\times [-2,+2]$, $\alpha(x,y)=1$ for all $(x,y) \in [-1,+1]\times [-1,+1]$ and a $C^\infty-$function $\chi $ with values in $[0,1]$,  $\chi(y)=0$ for all $	 y\in [0,\frac 1 2 ] $, $\chi(y)=1$ for all $ y \in [+1,\infty).$ We set
		\begin{equation*}\label{Aj}
		A_j(x,y)= \alpha\left(\frac x j,\frac y j\right)\chi(jy),\qquad j\in \N.
		\end{equation*}
		For every $j\in \N$, $A_j$ has compact support in $ \O $ and we have
		$$
		\into u_x(x,y)A_j(x,y)v(x,y ) d\m $$ $$=- \into u(x,y) 	\big( v_x(x,y)-\gamma sgn(x)v(x,y)         \big) A_j(x,y)d\m-\into u(x,y)v(x,y) (A_j)_x(x,y)d\m.
		$$
		The function $A_j$ is bounded by $\Vert \alpha\Vert_\infty \Vert \chi\Vert_\infty$ and $\lim_{j\rightarrow +\infty } A_j(x,y)=1$ for every $(x,y) \in \O$. Moreover $(A_j)_x(x,y)=\frac 1 j \alpha_x \left(\frac x j,\frac y j\right)\chi(jy)$, so that
		$$
		\left| \into u(x,y)v(x,y) (A_j)_x(x,y)d\m\right| \leq \frac C j \into \ind{\{|x|\geq j\}} |u(x,y)v(x,y) | d\m,
		$$
		where $C=\Vert \alpha_x\Vert_\infty \Vert \chi\Vert_\infty$. Therefore, we obtain \eqref{IPPx} letting $j\rightarrow \infty$.
	\end{proof}
	\begin{proposition}\label{IPP2}
		Let us consider $u,v: \mathcal{O} \rightarrow \R$ locally square-integrable on $\mathcal{O}$, with derivatives $u_y$ and $v_y$ locally square-integrable on $\mathcal{O}$ as well.  Moreover, assume that
		$$
		\into y\big(| u_y(x,y)v(x,y)|+ | u(x,y)v_xy(x,y)|\big)+ |u(x,y)v(x,y)|  d\m <\infty.
		$$ Then, we have
		\begin{equation}\label{IPPy}
		\int_\mathcal{O} y u_y(x,y) v(x,y)d\m=-\int_\mathcal{O}y u(x,y) v_y(x,y)d\m- \int_{\mathcal{O}} (\beta-\mu y)u(x,y) v(x,y) d\m.
		\end{equation}
	\end{proposition}
	\begin{proof}
		If $v$ has compact support in $\O$, we obtain \eqref{IPPy} as in the proof of Proposition \ref{IPP1}. On the other hand, if $v$  does not have compact support,
		\begin{align*}
		\into y u_y(x,y)&v(x,y)A_j(x,y)d\m =- \into yu(x,y)v_y(x,y) A_j(x,y)d\m \\&- \into (\beta-\mu y)u(x,y)v(x,y)A_j(x,y)d\m -\into y u(x,y) v(x,y)(A_j)_y(x,y) d\m,
		\end{align*}
		where $A_j(x,y)=\alpha(\frac x j,\frac y j)\chi(jy)$, as in the proof of Proposition \ref{IPP1} but choosing $\chi$ such that, moreover, $\|y\chi'(y)\|_\infty<\infty$. We have $ (A_j)_y(x,y) =\frac 1 j  \alpha_y (\frac x j,\frac y j)\chi(jy) +j\alpha (\frac x j,\frac y j)\chi'(jy)$.
		Note that 
		$$
		\left| \into yu(x,y) v(x,y) j\alpha \left(\frac x j,\frac y j\right)\chi'(jy) d\m \right| \leq \into \ind{ \left \{y\leq \frac 1 j  \right \}}|u(x,y)v(x,y)| \Vert \alpha\Vert_\infty \sup_{\zeta>0}|\zeta \chi' (\zeta)| d\m.
		$$
		The last expression goes to 0 as $j \rightarrow \infty$ since $\into |u(x,y) v(x,y) | d\m <\infty$. The assertion follows by passing to the limit $j\rightarrow \infty$.
	\end{proof}
	We can now prove Proposition  \ref{Prop_integrationbyparts}.	
	\begin{proof}[Proof of Proposition \ref{Prop_integrationbyparts}]
		By using Lemma \ref{IPP1}  we have
		\[
		\into y \frac{\partial^2 u}{\partial x^2}vd\m=-\into  y\frac{\partial u}{\partial x}\left(\frac{\partial v}{\partial x}
		-\gamma sgn(x) v\right)d\m,
		\]
		\[
		\into  y \frac{\partial^2 u}{\partial y^2}vd\m=-\into y \frac{\partial u}{\partial y}\frac{\partial v}{\partial y}d\m
		+\into  (\mu y -\beta)\frac{\partial u}{\partial y} vd\m,
		\]
		\[
		\into y \frac{\partial^2 u}{\partial x\partial y}vd\m=-\into y\frac{\partial u}{\partial y}\left(\frac{\partial v}{\partial x}
		-\gamma sgn(x) v\right)d\m
		\]
		and
		\[
		\into  y \frac{\partial^2 u}{\partial x\partial y}vd\m=-\into y \frac{\partial u}{\partial x}\frac{\partial v}{\partial y}d\m
		+\into  (\mu y -\beta)\frac{\partial u}{\partial x} vd\m.
		\]
		Recalling  that 
		\[
		\mathcal{L} =\frac{y}{2}\left(\frac{\partial^2 }{\partial x^2}+2\rho\sigma\frac{\partial^2 }{\partial x\partial y}+\sigma^2\frac{\partial^2 }{\partial y^2}\right)
		+\left(\frac{\rho\kappa\theta}{\sigma}-\frac{y}{2}\right)\frac{\partial}{\partial x}
		+\kappa(\theta-y) \frac{\partial}{\partial y}
		\]
		and using the equality $\beta=2\kappa\theta/\sigma^2$, we get
		\begin{align*}
		&	(\mathcal{L} u,v)_H=-\into \frac{y}{2}\left(\frac{\partial u}{\partial x}\frac{\partial v}{\partial x}+
		\sigma^2 \frac{\partial u}{\partial y}\frac{\partial v}{\partial y}+
		\rho\sigma \frac{\partial u}{\partial x}\frac{\partial v}{\partial y}+
		\rho\sigma \frac{\partial u}{\partial y}\frac{\partial v}{\partial x}\right) d\m \\&+\into \frac{1}{2}\frac{\partial u}{\partial x}\left(y\gamma sgn(x) +{\rho\sigma}(\mu y -\beta)\right)vd\m 
		+\into\frac{1}{2}\frac{\partial u}{\partial y}\left(\mu \sigma^2y -\beta\sigma^2+ {\rho\sigma}y\gamma sgn(x)\right) vd\m\\&+\into \left[\left(\frac{\rho\kappa\theta}{\sigma}-\frac{y}{2}\right)\frac{\partial u}{\partial x}
		+\kappa(\theta-y) \frac{\partial u}{\partial y}\right]
		vd\m=-a(u,v).
		\end{align*}
	\end{proof}
	
	\begin{remark}\label{rem_trasl}
		By a closer look at the proof of Proposition \ref{Prop_integrationbyparts}	it is clear that the choice of $\bar c$  in \eqref{traslation} allows  to avoid terms of the type $\int  (u_x+u_y) v d\m$ in the associated bilinear form $a$. This trick will be crucial in order to obtain suitable energy estimates.
	\end{remark}
	Recall the well-known inequality 
	\begin{equation}\label{inequality}
	bc=(\sqrt{\zeta}b) \left(\frac c {\sqrt{\zeta}}\right)  \leq \frac \zeta 2  b^2+\frac 1 {2\zeta} c^2, \qquad \ b,c \in \R, \ \zeta >0.
	\end{equation}
	Hereafter we will often apply \eqref{inequality} in the proofs even if it is not explicitly recalled each time.
	
	We have the following energy estimates.
	\begin{proposition}\label{energyestimates}
		For every $u,v\in V$, the bilinear form $a(\cdot,\cdot)$ satisfies 
		\begin{equation}	\label{lb}
		|a(u,v)|\leq C_1 \Vert u \Vert_V\Vert v \Vert_V,
		\end{equation}
		\begin{equation}\label{ub}
		a(u,u) \geq  C_2 \Vert u \Vert^2_V -C_3\Vert (1 + y)^{\frac 1 2 }u\Vert^2_H,
		\end{equation}
		where $$C_1=\delta_0+K_1,\quad C_2=\frac{\delta_1} 2, \quad C_3= \frac{\delta_1}{2}+\frac {K_1^2} {2\delta_1}, $$
		with
		\begin{equation}\label{delta0}
		\delta_0=\sup_{s_1^2+t_1^2>0, \;s_2^2+t_2^2>0}\frac{|s_1s_2+\rho\sigma s_1t_2+\rho\sigma s_2 t_1+\sigma^2 t_1t_2|}{2\sqrt{(s_1^2+t_1^2)(s_2^2+t_2^2)}},
		\end{equation}
		\begin{equation}\label{delta1}
		\delta_1=\inf_{s^2+t^2>0}\frac{s^2+2\rho\sigma st+\sigma^2 t^2}{2(s^2+t^2)},
		\end{equation}
		and
		\begin{equation}\label{K1}
		K_1=\sup_{x\in \R}\sqrt{j^2_{\gamma,\mu}(x)+k^2_{\gamma,\mu}(x)}.
		\end{equation}
	\end{proposition}
	It is easy to see  that the constants $\delta_0,\delta_1$ and $K_1$ defined in \eqref{delta0} and \eqref{K1} are positive and finite (recall that
	the functions 	 $j_{\gamma,\mu}=j_{\gamma,\mu}(x)$  and $k_{\gamma,\mu}=\kappa_{\gamma,\mu}(x)$ defined in \eqref{JandGamma} are bounded). 
	
	These energy estimates  were already proved in  \cite[Lemma 2.40]{DF} with a very similar statement. Here we repeat the proof for the sake of completeness, since we will refer to it later on.
	\begin{proof}[Proof of Proposition \ref{energyestimates}]
		%		 	Recall that
		%		 	\begin{align*}
		%		 	a(u,v) =& \frac 1 2 \int_{\mathcal{O}} y\left(  u_xv_x(x,y) +\rho\sigma u_xv_y (x,y)+\rho\sigma u_yv_x (x,y)+\sigma^2 u_yv_y (x,y) \right)d\m  \\
		%		 	&+\into y \left(  j_{\gamma,\mu}(x)u_x(x,y) + k_{\gamma,\mu}  (x)u_y(x,y) \right) v (x,y) d\m.
		%		 	\end{align*}
		%		 	We can easily see that
		%		 	\begin{align*}
		%		 	&	\left|	 \frac 1 2 \int_{\mathcal{O}} y\left(  u_xv_x+\rho\sigma u_xv_y + \rho\sigma u_yv_x +\sigma^2 u_y v_y \right)d\m \right| \leq  \delta_0 \into y|\nabla u||\nabla v| d\m \leq\delta_0 \|u\|_V\|v\|_V
		%		 	\end{align*}
		%		 	and
		%		 	\begin{align*}
		%		 	&\left|  \into y \left(  j_{\gamma,\mu}(x)u_x(x,y) + k_{\gamma,\mu}  (x)u_y(x,y) \right) v(x,y) d \m   \right| \leq K_1  \into y|\nabla u||v| d\m\leq  K_1\|u\|_V\|v\|_V.
		%		 	\end{align*}
		%	Then \eqref{lb} immediately follows.
		In order to prove \eqref{ub}, we note that
		\begin{align*}
		&	 \frac 1 2 \int_{\mathcal{O}} y\left(  u_xv_x+\rho\sigma u_xv_y + \rho\sigma u_yv_x +\sigma^2 u_y v_y \right)d\m  \geq \delta_1 \into y|\nabla u|^2 d\m.
		\end{align*}
		Therefore
		\begin{align*}
		a(u,u)&\geq\delta_1 \into y|\nabla u|^2 d\m
		-		  K_1  \into y|\nabla u| |u| d\m  \\&\geq\delta_1 \into y|\nabla u|^2 d\m
		- \frac{K_1\zeta} 2 \int_{\mathcal{O}}y|\nabla u|^2   
		d\m  -  \frac {K_1} {2\zeta} \int_{\mathcal{O}} (1+y) u^2d\m  \\
		&= \left(\delta_1-\frac{K_1\zeta} 2 		\right)\int_{\mathcal{O}} \left(y |\nabla   u |^2 + (1+y)u^2\right) d\m -\left( \delta_1-\frac{K_1\zeta} 2+\frac {K_1} {2\zeta}  \right)\into (1+y)u^2 d\m.
		\end{align*}	
		The assertion then follows by choosing $\zeta=\delta_1/K_1$.  \eqref{lb} can be proved in a similar way.
		%		 	\begin{align*}
		%		 	a(u,u) \geq \frac{\delta_1}{2}\|u\|_V^2-\left(    \frac{\delta_1}{2}+\frac {K_1^2} {2\delta_1}  \right)\|\sqrt{1+y}u\|^2_H
		%		 	\end{align*}
		%	 	and the assertion is proved.
	\end{proof}
	\subsection{Proof of Theorem \ref{variationalinequality} }
	Among the standard assumptions required in \cite{BL}  for the penalization procedure, there are the  coercivity and  the boundedness of the coefficients.  In the Heston-type models these assumptions are no longer satisfied and this leads to some technical difficulties. 
	In order to overcome them, we introduce some auxiliary operators.

	From now on, we set 
	$$
	a(u,v)=\bar{a}(u,v)+	\tilde{a}(u,v),
	$$
	where
	\begin{eqnarray*}
		\bar{a}(u,v)&=&\into\frac{y}{2}\left(\frac{\partial u}{\partial x}\frac{\partial v}{\partial x}+
		\rho\sigma \frac{\partial u}{\partial x}\frac{\partial v}{\partial y}+
		\rho\sigma \frac{\partial u}{\partial y}\frac{\partial v}{\partial x}+
		\sigma^2 \frac{\partial u}{\partial y}\frac{\partial v}{\partial y}\right) d\m,\\
		\tilde{a}(u,v)&=&\into y\frac{\partial u}{\partial x}j_{\gamma,\mu}vd\m
		+  \into y\frac{\partial u}{\partial y}k_{\gamma,\mu} vd\m.
	\end{eqnarray*}
	Note that $\bar a$ is symmetric.
	As in the proof of Proposition \eqref{energyestimates} we have, for every $u, v\in V$,
	\begin{eqnarray*}
		|\bar{a}(u,v)|&\leq& \delta_0 \into y|\nabla u||\nabla v| d\m ,
	\end{eqnarray*}
	%	with
	%		 	\[
	%		 	\delta_0=\sup_{s_1^2+t_1^2>0, \;s_2^2+t_2^2>0}\frac{|s_1s_2+\rho\sigma s_1t_2+\rho\sigma s_2 t_1+\sigma^2 t_1t_2|}{2\sqrt{(s_1^2+t_1^2)(s_2^2+t_2^2)}}.
	%		 	\]
	%		 and
	\begin{eqnarray*}
		\bar{a}(u,u)&\geq& \delta_1 \into y|\nabla u|^2 d\m,
	\end{eqnarray*}
	%		 	with
	%		 	\[
	%		 	\delta_1=\inf_{s^2+t^2>0}\frac{s^2+2\rho\sigma st+\sigma^2 t^2}{2(s^2+t^2)}.
	%		 	\]
	%	 On the other hand, for every $u,v\in V$,
	and
	\[
	|\tilde{a}(u,v)|\leq K_1  \into y|\nabla u||v| d\m,
	\]
	%		 	with\
	%		 	\[
	%		 	K_1=\sup_{(x,y)\in \R\times ]0,+\infty[}\sqrt{j^2_{\gamma,\mu}(x,y)+k^2_{\gamma,\mu}(x,y)}.
	%		 	\]
	with $\delta_0,\,\delta_1$ and $K_1$ defined in Proposition \ref{energyestimates}.
	Moreover,  for $\lambda\geq 0$ and  $M>0$ we consider the bilinear forms
	\begin{eqnarray*}
		a_\lambda(u,v)&=&{a}(u,v)+\lambda \into (1+y)uv d\m,\\
		\bar{a}_\lambda(u,v)&=&\bar{a}(u,v)+\lambda \into (1+y)uv d\m,\\
		\tilde{a}^{(M)}(u,v)&=&\into (y\wedge M)\left(\frac{\partial u}{\partial x}j_{\gamma,\mu}
		+ \frac{\partial u}{\partial y}k_{\gamma,\mu}\right) vd\m
	\end{eqnarray*}
	and
	\begin{eqnarray*}
		a^{(M)}_\lambda(u,v)&=& \bar{a}_\lambda(u,v)+\tilde{a}^{(M)}(u,v).
	\end{eqnarray*}
	The operator $a_\lambda$  was introduced in \cite{DF} to deal with the lack of coercivity of the bilinear form $a$, while the introduction of the truncated operator $a^{(M)}_\lambda$ with $M>0$ will be useful in order to overcome the technical difficulty related to the unboundedness of the coefficients.
	\begin{lemma}\label{coercivity}
		Let $\delta_0,\, \delta_1, \, K_1$ be defined as in \eqref{delta0}, \eqref{delta1} and \eqref{K1} respectively.
		For any fixed $\lambda\geq \frac{\delta_1}{2}+\frac{K_1^2}{2\delta_1}$ the bilinear forms $a_{\lambda}$ and $a_{\lambda}^{(M)}$ are continuous and coercive. More precisely, we have
		\begin{equation}
		|a_{\lambda}(u,v)|\leq C\Vert u \Vert_V \Vert v \Vert_V, \qquad  u,v \in V,
		\end{equation}
		%	There exist $\lambda_0>0$ such that $\forall \lambda\geq \lambda_0$ the bilinear forms $a_\lambda$ and $a_\lambda^{(M)}$ are coercive. In particular, we have
		%				 		\begin{equation}
		%				 		|a_{\lambda}(u,v)|\leq C\Vert u \Vert_V \Vert v \Vert_V, \qquad \forall u,v \in V,
		%				 		\end{equation}
		\begin{equation}
		a_{\lambda}(u,u) \geq \frac{\delta_1} 2\Vert u \Vert_V^2, \qquad  u\in V,
		\end{equation}
		and
		\begin{equation}
		|a_{\lambda}^{(M)}(u,v)|\leq C\Vert u \Vert_V \Vert v \Vert_V, \qquad u,v \in V,
		\end{equation}
		\begin{equation}
		a_{\lambda}^{(M)}(u,u) \geq \frac{\delta_1} 2 \Vert u \Vert_V^2, \qquad  u\in V.
		\end{equation}
		where $C=\delta_0+K_1+\lambda$.
	\end{lemma}
	\begin{proof}
		The proof for the bilinear form  $a_\lambda$ follows as in \cite[Lemma 3.2]{DF}. We give the details for  $a_\lambda^{(M)}$ to check that the constants do not depend on $M$. Note that, for every $u,v\in V$,
		\begin{equation*}
		|\tilde{a}^{(M)}(u,v)|\leq K_1  \into y |\nabla u||v| d\m,
		\end{equation*}
		so that by straightforward computations we get
		\begin{align*}
		|a^{(M)}_\lambda(u,v)|
		%&\leq   |\bar{a}(u,v)| +\lambda\into(1+y) |u||v| d\m+ 	K_1  \into y |\nabla u||v| d\m\\
		\leq (\delta_0+\lambda+K_1)\|u\|_V\|v\|_V.
		\end{align*}
		On the other hand, for every $\zeta>0$,
		\begin{eqnarray*}
			a^{(M)}_\lambda(u,u)&\geq& \delta_1 \into y|\nabla u|^2 d\m+\lambda
			\into (1+y)u^2 d\m-K_1 \into y|\nabla u||u| d\m\\
			&\geq&\left(\delta_1-\frac{K_1\zeta}{2}\right)\into y|\nabla u|^2 d\m
			+\left(\lambda -\frac{K_1}{2\zeta}\right)
			\into (1+y)u^2 d\m.
		\end{eqnarray*}
		By choosing $\zeta=\delta_1/K_1$, we get
		\begin{eqnarray*}
			a^{(M)}_\lambda(u,u)&\geq& \frac{\delta_1}{2}\into y|\nabla u|^2 dm
			+\left(\lambda -\frac{K_1^2}{2\delta_1}\right)
			\into (1+y)u^2 d\m\geq  \frac{\delta_1}{2}\|u\|^2_V,
		\end{eqnarray*}
		for every $\lambda\geq \frac{\delta_1}{2}+\frac{K_1^2}{2\delta_1}$.
	\end{proof}
	From now on in the rest of this chapter we assume  $\lambda\geq \frac{\delta_1}{2}+\frac{K_1^2}{2\delta_1}$  as in Lemma \ref{coercivity}. Moreover, we will denote by $\|b\|=\sup_{u,v\in V,u,v\neq0}\frac{|b(u,v)|}{\|u\|_V\|v\|_V}$ the norm of a bilinear form $b:V\times V\rightarrow \R$.
	\begin{remark}
		%	The introduction of the truncated operator $a^{(M)}_\lambda$ with $M>0$ will be useful in order to overcome the technical difficulty given by the unboundedness of the coefficients of the bilinear form $a$. So,  the idea is to deal in a first step with the bilinear form  $a^{(M)}_\lambda$ which has bounded coefficients. 	
		%We denote by $\|a\|=\sup_{u,v\in V,u,v\neq0}\frac{|a(u,v)|}{\|u\|_V\|v\|_V}$ the norm of a bilinear form $a:V\times V\rightarrow \R$.
		We stress that  Lemma \ref{coercivity} gives us
		\begin{equation}\label{estimate_aM}
		\sup_{M>0}\|a_\lambda^{(M)}\|\leq C ,
		\end{equation}
		where $C=\delta_0+K_1+\lambda$. This will be crucial in the penalization technique we are going to describe in  Section \ref{sect-penalization}.  Roughly speaking, 
		in order to prove the existence of a solution of the penalized coercive problem we will introduce in Theorem \ref{penalizedcoerciveproblem}, we proceed  as follows. First, we replace the bilinear form  $a_\lambda$ with the operator $a^{(M)}_\lambda$, which has bounded coefficients, and we solve the associated penalized truncated coercive problem (see Proposition \ref{penalizedcoercivetruncatedproblem}). 
		Then,  thanks to \eqref{estimate_aM},  we can  deduce estimates on the solution which are uniform in $M$ (see Lemma \ref{lemma_estim})  and which will allow us to pass to the limit as $M$ goes to infinity and to find a solution of the original penalized coercive problem. %Without this technical device, we  should require further regularity assumptions on the payoff function $\psi$.
		%Therefore, in a second step, we can then pass to the limit as $M$ goes to infinity and come back to the operator $a_\lambda$. This is the scheme we will follow in the proof of Theorem \ref{penalizedcoerciveproblem} in next section.
	\end{remark}

	%		 	
	%		 	First of all we write
	%		 	$$
	%		 	a(f,g)=a_0(f,g)+a_1(f,g),
	%		 	$$
	%		 	where
	%		 	$$
	%		 	a_1(f,g)=\into y \left(  j_{\gamma,\mu}(x) \frac{\partial f}{\partial x} + k_{\gamma,\mu}  (x)\frac{\partial f}{\partial y}  \right) g(x,y) d\m
	%		 	$$
	%		 	is the non symmetric part of $a$, and 
	%		 	$$
	%		 	a_0(f,g)=a(f,g)-a_1(f,g)
	%		 	$$
	%		 	is its symmetric part. We also need to introduce the truncated operator
	%		 	$$
	%		 	\tilde{a}^{(M)}(f,g)= \into (y \wedge M) \left(  j_{\gamma,\mu}(x) \frac{\partial f}{\partial x} + k_{\gamma,\mu}  (x)\frac{\partial f}{\partial y}  \right) g(x,y)d \m.
	%		 	$$
	%		 	Note that, $\forall \epsilon >0$,
	%		 	\begin{align*}
	%		 	|\tilde{a}^{(M)}(f,g)|&\leq C \into( y\wedge M) \vert \nabla f(x,y)\vert \vert g(x,y)\vert d\m \\
	%		 	& \leq \frac {C\epsilon} 2 \into y | \nabla f(x,y)|^2 d\m + \frac  C {2\epsilon}  M \into  g^2(x,y)d\m,
	%		 	\end{align*}
	%		 	with $C= \Vert j\Vert_\infty + \Vert k \Vert_\infty$, but also
	%		 	\begin{equation}\label{estimaM}
	%		 	|\tilde{a}^{(M)}(f,g)| \leq \frac {C\epsilon} 2  \into y | \nabla f(x,y)|^2 \m(x,y)dxdy + \frac  C {2\epsilon}  \into  y g^2(x,y)d\m.
	%		 	\end{equation}
	%		 	Observe that the last inequality is independent of $M$.
	%		 	
	%		 	Moreover, as usual in the case of non coercive forms, for $\lambda >0$ we consider the bilinear form
	%		 	$$
	%		 	a_{\lambda}(f,g)= a(f,g)+ \lambda \into (1+y) fg d\m
	%		 	$$
	%		 	and its truncated version
	%		 	\begin{align*}
	%		 	a_{\lambda}^{(M)}(f,g)&= a_0(f,g)+\tilde{a}^{(M)}(f,g)+ \lambda \into (1+y) fg d\m\\
	%		 	&=a_0^{\lambda}(f,g) +\tilde{a}^{(M)} (f,g).
	%		 	\end{align*}
	Finally, we define
	$$
	\mathcal{L}^\lambda  := 	\mathcal{L}-\lambda(1+y) 
	$$
	the differential operator associated with the bilinear form $a_\lambda$, that is 
	$$
	(	\mathcal{L}^\lambda u,v)_H=- 	a_\lambda(u,v), \qquad u\in H^2(\O,\m),\, v\in V.
	$$
	\subsubsection{Penalized problem }\label{sect-penalization}
	\begin{comment}
	for any $v \in V$ we have 
	\begin{align*}
	a_{\lambda}^{(M)} (v,v) &\geq \delta \into y |\nabla v |^2 d\m + a_1(v,v) + \lambda \into (1+y) v^2 d\m\\
	& \geq  (\delta-\varepsilon) \into y |\nabla v |^2 d\m  + (\lambda-C_\varepsilon) \into (1+y) v^2 d\m.
	\end{align*}
	Again, there exists a $\lambda_0>0$ such that for any $\lambda \geq \lambda_0$ 
	$$
	a_{\lambda}^{(M)}(v,v)\geq \frac{\delta}{2}\Vert v\Vert_V^2.
	$$
	where $\delta$ is a positive constant independent of $M$.
	In the usual way we also have 
	$$
	a_{\lambda}^{(M)}(f,g)\leq C\Vert f\Vert_V \Vert g \Vert_V
	$$
	\end{comment}
	For any fixed $\varepsilon >0$  we define the penalizing operator 
	\begin{equation}\label{penalizedoperator}
	\zeta_\varepsilon(t,u)= -\frac 1 \varepsilon (\psi(t)-u)_+= \frac 1 \varepsilon \zeta(t,u),\qquad t\in[0,T],  u\in V.
	\end{equation}
	Since for every fixed $t\in[0,T]$ the function $x\mapsto -(\psi(t)-x)_+$ is nondecreasing, we have the following well known monotonicity result (see \cite{BL}).
	\begin{lemma}	\label{monotonicity}
		For any fixed $t\in [0,T]$ the penalizing operator \eqref{penalizedoperator} is monotone, in the sense that
		$$
		(\zeta_\varepsilon(t,u)-\zeta_\varepsilon(t,v),u-v)_H \geq 0, \qquad   u,v \in V.
		$$
	\end{lemma}
	
	We now introduce the intermediate penalized coercive problem with  a source  term $g$. We consider the following assumption:
	
	\medskip
	\noindent	
	\textbf{Assumption $\mathcal{H}^0$.}
	We say that a function $g$ satisfies Assumption $\mathcal{H}^0$ if	$ \sqrt{1+y}g\in L^2([0,T];H)$.
	
	\medskip

	\begin{theorem}\label{penalizedcoerciveproblem} 
		Assume that  $\psi$ satisfies Assumption $\mathcal{H}^1$ and  $g$ satisfies Assumption $\mathcal{H}^0$. Then, for every fixed $\varepsilon>0$,  there exists a unique function $u_{\varepsilon,\lambda}$ such that 
		$	u_{\varepsilon,\lambda}\in L^2([0,T];V)$,  $\frac{\partial u_{\varepsilon,\lambda}}{\partial t } \in L^2([0,T]; H)$ and, for all $v\in L^2([0,T];V)$,
		\begin{equation}\label{PCP}
		\begin{cases}
		-\left( \frac{\partial u_{\varepsilon,\lambda}}{\partial t },v  \right)_H + a_{\lambda}(u_{\varepsilon,\lambda},v)+ (\zeta_\varepsilon(t,u_{\varepsilon,\lambda}),v)_H= (g,v)_H,\qquad \mbox{a.e. in } [0,T],\\
		u_{\varepsilon,\lambda}(T)=\psi(T).
		\end{cases}
		\end{equation}
		Moreover, the following estimates hold:
		\begin{equation}\label{sp1}
		\Vert u_{\varepsilon,\lambda} \Vert_{L^\infty([0,T],V)}\leq K,
		\end{equation}
		\begin{equation}\label{sp2}
		\left\Vert \frac{\partial u_{\varepsilon,\lambda}}{\partial t }\right \Vert_{L^2([0,T];H)}\leq K,
		\end{equation}
		\begin{equation}\label{sp3}
		\frac {1} {\sqrt{\varepsilon }}\left\Vert (\psi-u_{\varepsilon,\lambda})^+ \right \Vert_{L^\infty([0,T],H)}\leq K,
		\end{equation}
		where $K=C \left(    \Vert \Psi \Vert_{L^2([0,T];V)} + \Vert \sqrt{1+y}g\Vert_{L^2([0,T];H)}  + \Vert\sqrt{1+y}\psi \Vert_{L^2([0,T];V)}+\Vert\psi(T)\Vert_V^2   \right)$,  with $C>0$ independent of $\varepsilon$, and $\Psi$ is given in Assumption $\mathcal{H}^1$.
	\end{theorem}
	The proof of uniqueness of the solution of the penalized coercive problem follows a standard monotonicity argument as in \cite{BL}, so we omit the proof.
	%		 	We first prove uniqueness of the penalized coercive problem.
	%		 	
	%		 	\begin{proof}[Proof of Theorem \ref{penalizedcoerciveproblem}: uniqueness]
	%		 		Assume that there exist two functions $u_1$ and $u_2$ satisfying \eqref{PCP} and set $w=u_1-u_2$.
	%		 		If we choose $v=u_1-u_2$ in the equation satisfied by $u_1$ and $v=u_2-u_1$ in the one satisfied by $u_2$ and then we add the resulting equations, we get
	%%		 		$$
	%%		 		-\left( \frac{\partial u_1}{\partial t }(t),u_1(t)-u_ 2(t) \right)_H + a_{\lambda}(u_1(t),u_1(t)-u_2(t))+ (\zeta_\varepsilon(u_1(t)),u_1(t)-u_ 2(t))_H= (g(t),u_1(t)-u_ 2(t))_H
	%%		 		$$
	%%		 		and
	%%		 		$$
	%%		 		-\left( \frac{\partial u_2}{\partial t }(t),u_2(t)-u_ 1(t) \right)_H + a_{\lambda}(u_2(t),u_2(t)-u_1(t))+ (\zeta_\varepsilon(u_2(t)),u_2(t)-u_ 1(t))_H= (g(t),u_2(t)-u_ 1(t))_H.
	%%		 		$$
	%%		 		Adding the second equation from the first one and defining $w=u_1-u_2$, we get
	%		 		$$
	%		 		-\left( \frac{\partial w}{\partial t }(t),w(t) \right)_H +a_\lambda(w(t),w(t))+ (\zeta_\varepsilon(u_1(t))-\zeta_\varepsilon(u_2(t)),w(t))_H= 0.
	%		 		$$
	%		 		By the coercivity of $a_\lambda$ and the monotonicity of the penalized operator we deduce that
	%		 		$$
	%		 		-\left( \frac{\partial w}{\partial t }(t),w (t)\right)_H\leq 0 \Rightarrow  \left( \frac{\partial w}{\partial t }(t),w(t) \right)_H=\frac 1 2 \frac{\partial  }{ \partial t}\Vert w(t)\Vert^2_H \geq 0 .
	%		 		$$
	%		 		But $w(T)= \psi(T)-\psi(T) =0$ , so $w(t)=0 $ a.e. in $[0,T]$, which means $u_1=u_2$.
	%		 	\end{proof}
	
	The  proof of existence in Theorem \ref{penalizedcoerciveproblem} is quite long and technical, so we split it into two propositions. We first consider the truncated penalized problem, which requires less stringent conditions on $\psi$ and $g$.
	\begin{proposition}\label{penalizedcoercivetruncatedproblem}
		Let $\psi\in \mathcal{C}([0,T];H)\cap     L^2([0,T];V)$ and $ g\in L^2([0,T];H)$. Moreover, assume that $ \psi(T) \in H^2(\O,\m)$, $ (1+y)\psi(T) \in H$, $\frac{\partial \psi}{\partial t}\in L^2([0,T];V) $  and $\frac{\partial g}{\partial t}\in L^2([0,T];H)$. Then, there exists a unique function $u_{\varepsilon,\lambda,M }$ such that $	u_{\varepsilon,\lambda,M }\in  L^2([0,T]; V),  \,
		\frac{\partial u_{\varepsilon,\lambda,M}}{\partial t } \in L^2([0,T]; V)$ and for all $v\in L^2([0,T]; V)$
		%	$\in L^2([0,T];V)$ such that %$\frac{\partial u_{\varepsilon,\lambda}}{\partial t } \in L^2([0,T], H)$ and
		\begin{equation}\label{PTCP}
		\begin{cases}
		-\left( \frac{\partial u_{\varepsilon,\lambda,M}}{\partial t },v \right)_H + a_{\lambda}^{(M)}(u_{\varepsilon,\lambda,M},v)+ (\zeta_\varepsilon(t,u_{\varepsilon,\lambda,M}),v)_H=  (g,v)_H,  \qquad \mbox{a.e. in     } [0,T) ,\\		u_{\varepsilon,\lambda,M}(T)=\psi(T). 
		\end{cases}
		\end{equation}
	\end{proposition}
	\begin{proof}
		%We start by considering the case with the further assumptions that $ \mathcal{L}\psi(T) \in V$, $\frac{\partial \psi}{\partial t}\in L^2([0,T];V) $  and $\frac{\partial g}{\partial t}\in L^2([0,T];H)$. We will pass to the general case at the end of the proof.
		\begin{enumerate}
			\item \textbf{ Finite dimensional problem}
			We use the classical Galerkin method of approximation, which consists in introducing a nondecreasing sequence $(V_j)_j$ of subspaces of $V$ such that $dimV_j<\infty$ and, for every $ v\in V ,$ there exists  a sequence $  \ (v_j)_{j\in \N}$  such that $v_j \in V_j $ for any $ j\in\N $ and $\Vert v-v_j\Vert_V\rightarrow0$ as $j\rightarrow \infty$. Moreover, we assume that $\psi(T)\in V_j,$ for all $ j\in\N$. Let $P_j$ be the projection of $V$ onto $V_j$ and $\psi_j(t)=P_j\psi(t)$. We have  $\psi_j (t) \rightarrow \psi(t)$ strongly in $V$ and $\psi_j(T)=\psi(T)$ for any $ j\in \N$. 
			The finite dimensional problem is, therefore, to find $u_j:[0,T]\rightarrow V_j$ such that 
			\begin{equation}\label{ATPCP}
			\begin{cases}
			
			-\left( \frac{\partial u_j}{\partial t }(t),v   \right)_H + a_{\lambda}^{(M)}(u_j(t),v)-\frac 1 \varepsilon ((\psi_j(t)-u_j(t))_+,v)_H= (g(t),v)_H, \qquad 	 v\in V_j,\\
			u_j(T)=\psi(T).
			\end{cases}
			\end{equation}
			
			This problem can be interpreted as an ordinary differential equation in $V_j$ (dim $V_j<\infty$),
			that is 
			$$
			\begin{cases}
			-\frac{\partial u_j}{\partial t }(t) + A_{\lambda,j}^{(M)}u_j(t)-\frac 1 \varepsilon Q_j((\psi_j(t)-u_j(t))_+)= Q_jg(t)\\
			u_j(T)=\psi(T),
			\end{cases}
			$$
			where $  A_{\lambda,j}^{(M)}:V_j\rightarrow V_j$ is a finite dimensional linear operator and $Q_j$ is the projection of $H$ onto $V_j$. 
			Note that the function $u\rightarrow Q_j((\psi_j(t)-u)_+)$ is Lipschitz continuous, since 
			\begin{align*}
			&	\| Q_j((\psi_j(t)-u)_+)- Q_j((\psi_j(t)-v)_+)\|_{V_j}\\&\qquad\leq C_j\| Q_j((\psi_j(t)-u)_+)- Q_j((\psi_j(t)-v)_+)\|_{H}\leq C_j \|u-v\|_H.
			\end{align*}
			On the other hand,   the function $(t,u)\rightarrow Q_j((\psi_j(t)-u(t)_+)$ is continuous with values in $V_j$. In fact,  we can easily prove that it is weakly continuous, that is, for  $v\in V_j$, the application $(t,u)\rightarrow (Q_j((\psi_j(t)-u)_+),v)$ is continuous. In fact
			\begin{equation}\label{continuityode}
			\begin{split}
			&	\left|\big(    Q_j((\psi_j(t)-u)_+)-  Q_j((\psi_j(s)-w)_+),v\big) \right|\leq 	\left| \big(  Q_j(  (\psi_j(t)-u)_+)-  Q_j((\psi_j(s)-u)_+),v\big) \right|\\&\qquad+ 	\left| \big(    Q_j((\psi_j(s)-u)_+)-  Q_j((\psi_j(s)-w)_+),v\big) \right|.
			\end{split}
			\end{equation}
			The second term in the right hand side of \eqref{continuityode} goes to 0 by using the Lipschitz continuity proved above. On the other hand, it is easy to prove that for any $u\in V,v\in H^2(\O,\m)$, one has $|(u,v)_V|\leq C\|u\|_H\|v\|_{H^2(\O(\m))} $. Since $v\in V_j$ we can assume without loss of generality that $v\in H^2(\O,\m)$, so that  for   the first term in the right hand side of \eqref{continuityode}, we easily get 
			$$
			\left| \big(  Q_j(  (\psi_j(t)-u)_+)- Q_j( (\psi_j(s)-u)_+),v\big) \right|\leq \|\psi_j(t)-\psi_j(s)\|_H\|v\|_{H^2(\O,\m)},
			$$
			which goes to 0.	Finally, it is easy to see that  the term $Q_jg$ belongs to $L^2([0,T];V_j)$. 
			
			Therefore, we can use the Cauchy-Lipschitz Theorem and we  deduce the existence and the uniqueness of a solution $u_j$ of \eqref{ATPCP}, continuous from $[0,T]$ into $V_j$, a.e. differentiable and with integrable derivative.
			
			\item \textbf{ Estimates on the finite dimensional problem}
			First, we take $v=u_j(t)-\psi_j(t)$ in  \eqref{ATPCP}. We get
			\begin{align*}
			&	-\left( \frac{\partial u_j}{\partial t }(t),u_j(t)-\psi_j (t)  \right)_H + a_{\lambda}^{(M)}(u_j(t),u_j(t)-\psi_j(t))\\&\qquad-\frac 1 \varepsilon ((\psi_j(t)-u_j(t))_+,u_j(t)-\psi_j(t))_H
			= (g(t),u_j(t)-\psi_j(t))_H,
			\end{align*}
			which can be rewritten as
			\begin{align*}
			&	-\frac 1 2 \frac{d }{d t }\Vert u_j(t)-\psi_j(t) \Vert_H^2 - \left( \frac{\partial \psi_j}{\partial t }(t),u_j(t)-\psi_j  (t) \right)_H \\&+ a_{\lambda}^{(M)}(u_j(t)-\psi_j(t),u_j(t)-\psi_j(t))_H + \frac 1 \varepsilon ((\psi_j(t)-u_j(t))_+,\psi_j(t)-u_j(t))_H\\&+ a_{\lambda}^{(M)}(\psi_j(t),u_j(t)-\psi_j(t))  = (g(t),u_j(t)-\psi_j(t))_H .
			\end{align*}
			%		 			Note that
			%		 			$$
			%		 			\frac 1 \varepsilon ((\psi_j(t)-u_j(t))_+,\psi_j(t)-u_j(t))_H\geq 0.
			%		 			$$ 
			
			We integrate between $t$ and $T$ and we use coercivity and $u_j(T)=\psi_j(T)$ to obtain  
			\begin{align*}
			&	\frac 1 2 \Vert u_j(t)-\psi_j(t) \Vert_H^2  + \frac{\delta_1}2\int_t^T \Vert u_j(s)-\psi_j(s)\Vert_V^2 ds +\frac 1 \varepsilon\int_t^T \Vert(\psi_j(s)-u_j(s))_+\Vert^2_Hds\\
			%		 		&	\leq \int_t^T \left \Vert \frac{\partial \psi_j(s)}{\partial t } \right \Vert_H \!\Vert u_j(s)-\psi_j(s) \Vert_H ds + \int_t^T  \Vert g(s) \Vert_H \Vert u_j(s)-\psi_j(s) \Vert_H ds\\&\qquad
			%		 		 +\|a^{(M)}_\lambda\| \int_t^T  \Vert \psi_j(s) \Vert_V \Vert u_j(s)-\psi_j(s) \Vert_V ds\\
			&
			\leq    \frac 1 {2\zeta}      \int_t^T \left \Vert \frac{\partial \psi_j(s)}{\partial t } \right \Vert_H^2 ds+\frac \zeta 2 \int_t^T  \Vert u_j(s)-\psi_j(s) \Vert_H^2 ds +    \frac 1 {2\zeta}   \int_t^T  \Vert g(s) \Vert^2_H ds\\& + \frac \zeta 2  \int_t^T\!\! \Vert u_j(s)-\psi_j(s) \Vert_H^2 ds
			+ \frac {\|a^{(M)}_\lambda\|\zeta} 2 \! \int_t^T \!  \!\Vert u_j(s)-\psi_j(s) \Vert^2_V ds +  \frac {\|a^{(M)}_\lambda\|}{2\zeta}  \int_t^T\!\!  \Vert \psi_j(s) \Vert_V^2ds,
			\end{align*}
			for any $\zeta>0$.
			%Recall that our function $\psi(t,x,y)=\phi(x-\bar{c}t,y)$, so that $\frac{\partial \psi}{\partial t }=-k \frac{\partial \phi}{\partial x }$ and $\frac{\partial }{\partial t } P_j \psi=P_j \frac{\partial }{\partial t } \psi $. From this we deduce that $  \Vert \frac{\partial \psi_j(s)}{\partial t } \Vert_H \leq \Vert \phi_j\Vert_V$.
			Recall  that $\psi_j=P_j\psi$,  and so $ \Vert\psi_j (t)\Vert^2_V\leq  \Vert\psi  (t)\Vert^2_V$. In the same way $\Vert \frac{\partial \psi_j(t)}{\partial t } \Vert^2_H\leq \Vert \frac{\partial \psi_j(t)}{\partial t } \Vert^2_V \leq \Vert \frac{\partial \psi(t)}{\partial t } \Vert^2_V    $ .
			Choosing $\zeta=\frac{\delta_1}{4+2\|a^{(M)}_\lambda\|}$ 
			after simple calculations we  deduce that there exists $C>0$ independent of $M$, $\varepsilon$ and $j$ such that 
			\begin{equation}\label{estim1}
			\begin{array}{c}
			\frac 1 4 	\Vert u_j(t)\Vert_H^2  +\frac{\delta_1}{8}\int_t^T \Vert u_j(s)
			\Vert_V^2 ds +\frac 1 \varepsilon\int_t^T \Vert(\psi_j(s)-u_j(s))_+\Vert^2_Hds \\	\leq C \left(  \left  \Vert \frac{\partial \psi}{\partial t } \right\Vert^2_{L^2([t,T];V)} + \Vert g\Vert^2_{L^2([t,T];H)}  + \Vert\psi\Vert^2_{L^2([t,T];V)}  +\Vert\psi(T)\Vert^2_H \right).
			\end{array}
			\end{equation}
			
			\medskip
			
			We now go back to \eqref{ATPCP} and we take $v=\frac{\partial u_j}{\partial t}(t)$ so we get
			%		 			$$
			%		 			-\left( \frac{\partial u_j}{\partial t }(t),\frac{\partial u_j}{\partial t} (t) \right)_H + a_{\lambda}^{(M)}\left(u_j(t),\frac{\partial u_j}{\partial t}(t)\right)-\frac 1 \varepsilon \left(\left(\psi_j(t)-u_j(t)\right)_+,\frac{\partial u_j}{\partial t}(t)\right)_H= \left(g(t),\frac{\partial u_j}{\partial t}(t)\right)_H,
			%		 			$$
			%		 			so that
			\begin{align*}
			&	-\left \Vert \frac{\partial u_j}{\partial t} (t)\right \Vert^2_H+ \bar{a}_\lambda\left(u_j(t),\frac{\partial u_j}{\partial t}(t)\right) + \tilde{a}^{(M)}\left(u_j(t),\frac{\partial u_j}{\partial t}(t)\right)\\&\qquad-\frac 1 \varepsilon \left(\left(\psi_j(t)-u_j(t)\right)_+,\frac{\partial u_j}{\partial t}(t)\right)_H= \left(g(t),\frac{\partial u_j}{\partial t}(t)\right)_H.
			\end{align*}
			Note that
			\begin{align*}
			&	-\frac 1 \varepsilon \left((\psi_j(t)-u_j(t))_+,\frac{\partial u_j}{\partial t}(t)\right)_H\\&\quad= \frac 1 \varepsilon \left((\psi_j-u_j)_+,\frac{\partial (\psi_j-u_j)}{\partial t}(t)\right)_H - \frac 1 \varepsilon \left((\psi_j(t)-u_j(t))_+,\frac{\partial \psi_j}{\partial t}(t)\right)_H\\
			&\quad= \frac 1 {2\varepsilon} \frac d  {d t}\Vert (\psi_j-u_j)_+(t)\Vert^2_H - \frac 1 \varepsilon \left((\psi_j(t)-u_j(t))_+,\frac{\partial \psi_j}{\partial t}(t)\right)_H.
			\end{align*}
			Therefore, using the symmetry of $\bar a_\lambda$, we have
			$$-\left \Vert \frac{\partial u_j}{\partial t} (t)\right \Vert^2_H+ \frac 1 2 \frac{d }{d t}     \bar{a}_\lambda(u_j(t), u_j(t)) + \tilde{a}^{(M)} \left(u_j(t),\frac{\partial u_j}{\partial t}(t)\right)+\frac 1 {2\varepsilon} \frac \partial  {\partial t}\Vert (\psi_j(t)-u_j(t))_+\Vert^2_H $$ 
			$$- \frac 1 \varepsilon \left((\psi_j(t)-u_j(t))_+,\frac{\partial \psi_j}{\partial t}(t)\right)_H= \left(g(t),\frac{\partial u_j}{\partial t}(t)\right)_H.
			$$
			Integrating between $t$ and $T$, we obtain 
			\begin{align*}
			&\int_t^T\left \Vert \frac{\partial u_j}{\partial t}(s) \right \Vert^2_Hds+ \frac 1 2 \bar{a}_\lambda(u_j(t), u_j(t)) +\frac 1 {2\varepsilon} \Vert (\psi_j(t)-u_j(t))_+\Vert^2_H\\
			&= \int_t^T \tilde{a}^{(M)}\left(u_j(s),\frac{\partial u_j}{\partial s}(s)\right)ds +\frac 1 2 \bar{a}_\lambda(\psi_j(T),\psi_j(T))\\& -\int_t^T\frac 1 \varepsilon \left((\psi_j(s)-u_j(s)_+,\frac{\partial \psi_j}{\partial s}(s)\right)_Hds-\int_t^T \!\left(g(s),\frac{\partial u_j}{\partial s}(s)\right)_Hds.
			\end{align*}
			Recall that $
			\bar{a}_\lambda(u_j(t), u_j(t)) \geq \frac{\delta_1} 2 \Vert u_j(t)\Vert_V^2$, $
			|\tilde{a}^{(M)}(u,v)|\leq K_1  \into y\wedge M |\nabla u||v| d\m		  			$ and \linebreak
			$\bar a _\lambda(\psi_j(T),\psi_j(T))=\bar a _\lambda(\psi(T),\psi(T))\leq \|\bar a_\lambda\|\|\psi(T)\|_V^2$,
			so that, for every  $\zeta>0$,
			\begin{align*}
			&	\int_t^T \left \Vert \frac{\partial u_j}{\partial s}(s) \right \Vert^2_Hds+ \frac {\delta_1} 4  \Vert u_j(t)\Vert_V^2 +\frac 1 {2\varepsilon} \Vert (\psi_j(t)-u_j(t))_+\Vert^2_H\\
			&	\leq K_1 \int_t^Tds \into y\wedge M |\nabla u_j(s,.)|\left|\frac{\partial u_j}{\partial t}(s, .)\right|d\m+\frac{\|\bar a_\lambda\|} 2 \Vert\psi
			(T)\Vert_V^2  \\&\quad +\frac 1 \varepsilon \int_t^T \Vert (\psi_j(s)-u_j(s))_+\Vert_H  \left \Vert \frac{\partial \psi_j}{\partial s}(s)\right \Vert_Hds
			+\int_t^T \Vert g(s)\Vert_H  \left \Vert \frac{\partial u_j}{\partial s}(s)\right \Vert_Hds\\
			&	\leq \frac{K_1}{2\zeta} \int_t^T \Vert u_j(s)\Vert^2_Vds + \frac{K_1M}{2}\zeta \int_t^T \left \Vert \frac{\partial u_j}{\partial s}(s) \right \Vert_H^2ds +\frac{\|\bar a_\lambda\|} 2 \Vert\psi(T)\Vert_V^2 \\& \quad+  \frac {\zeta} { 2\varepsilon}
			\int_t^T \Vert (\psi_j(s)-u_j(s))_+\Vert^2_H ds+  \frac 1 {2\zeta\varepsilon}\int_t^T \left \Vert \frac{\partial \psi_j}{\partial t}(s)\right \Vert_H^2ds+\frac 1 {2 \zeta} \int_t^T \Vert g(s)\Vert_H^2 ds\\&\qquad +\frac \zeta {2} \int_t^T  \left \Vert \frac{\partial u_j}{\partial s}(s)\right \Vert^2_Hds.
			\end{align*}
			
			From \eqref{estim1}, we already know that 
			\begin{align*}\\
			&	\int_t^T \Vert u_j(s)
			\Vert_V^2 ds+\frac 1 \varepsilon\int_t^T                 \Vert(\psi_j(s)-u_j(s))_+\Vert^2_Hds \\&\qquad\leq C \left(  \left  \Vert \frac{\partial \psi}{\partial t } \right \Vert^2_{L^2([t,T];V)}+ \Vert g\Vert^2_{L^2([t,T];H)}  + \Vert\psi\Vert^2_{L^2([t,T];V)} +\Vert\psi(T)\Vert^2_H  \right),\end{align*}
			%		 			and 
			%		 			\begin{equation*}
			%		 			\int_t^T \Vert u_j(s)
			%		 			) \Vert_V^2 ds 	\leq C \left(    \Vert \frac{\partial \psi}{\partial t } \Vert^2_{L^2([0,T];V)} + \Vert g\Vert^2_{L^2([0,T];H)}  + \Vert\psi\Vert^2_{L^2([0,T];V)}   \right).	\end{equation*}
			then we can finally deduce
			\begin{equation}\label{estim2}
			\begin{split}
			\int_t^T& \left \Vert \frac{\partial u_j}{\partial t} (s)\right \Vert^2_Hds+   \Vert u_j(t)\Vert_V^2 +\frac 1 {2\varepsilon} \Vert (\psi_j(t)-u_j(t))_+\Vert^2_H\\
			&	\leq C_{\varepsilon,M} \left(  \left  \Vert \frac{\partial \psi}{\partial t } \right\Vert^2_{L^2([t,T];V)} + \Vert g\Vert^2_{L^2([t,T];H)}  + \Vert\psi \Vert^2_{L^2([t,T];V)}  +\Vert\psi(T)\Vert^2_V \right),
			\end{split}
			\end{equation}   	
			where $C_{\varepsilon,M} $  is a constant which depends on $\varepsilon$ and $M$ but not on $j$.
			
			\medskip
			We will also need a further estimation.	If we denote $\bar{u}_j=\frac{\partial u_j}{\partial t}$ and we differentiate the equation \eqref{ATPCP} with respect to $t$ for a fixed $v$ independent of $t$, we obtain that $\bar{u}_j$ satisfies
			\begin{equation}\label{eqforubar}
			-\left(  \frac{\partial \bar{u}_j}{\partial t }(t), v\right)_H\! \!+ a_{\lambda}^{(M)}(\bar{u}_j(t),v)- \frac 1 \varepsilon \left( \! \bigg(\frac{\partial \psi_j}{\partial t }(t)-\bar{u}_j(t)\bigg)\mathbbm{1}_{\{ \psi_j(t) \geq u_j(t)\}},v   \right)_H\!\!\!=\left(\frac{\partial g}{\partial t}(t),v\right)_H\!\!\!,
			\end{equation}
			for any $v\in V_j$.	As regards the initial condition, from \eqref{ATPCP} computed in $t=T$,   for every $ v\in V_j$ we have
			\begin{align*}
			\left(  \frac{\partial u_j(T)}{\partial t},v     \right)_H &= a_{\lambda}^{(M)}(\psi(T),v)-(g(T),v)_H.
			\\	&=-\left( \mathcal{L}\psi(T),v\right)_H+\lambda\left( (1+y)\psi(T),v\right)_H\\&+\left( (y\wedge M-y) ( j_{\gamma,\mu}u_x+k_{\gamma,\mu}u_y) ,v\right)_H+\left(g(T),v\right)_H.
			\end{align*}
			Choosing $v= \frac{\partial u_j(T)}{\partial t}$, we deduce that 
			\begin{align*}
			\left \Vert  \frac{\partial u_j(T)}{\partial t}\right  \Vert_H& \leq C \left(     \Vert \mathcal L \psi(T) \Vert_H + \Vert(1+y)\psi(T)\Vert_H +  \Vert(y-M)_+\nabla   \psi(T)\Vert_H +\Vert g(T)\Vert_H \right)
			\\& \leq C\left(     \Vert \psi(T) \Vert_{H^2(\O,\m)}  + \Vert(1+y)\psi(T)\Vert_H +\Vert g(T)\Vert_H \right),
			\end{align*}
			that is, 
			$\left \Vert  \frac{\partial u_j(T)}{\partial t}\right  \Vert_H\leq C  \left(  \Vert \psi(T) \Vert_{H^2(\O,\m)}+\Vert(1+y)\psi(T)\Vert_H +\Vert g(T)\Vert_H \right)$.
			%\left(     \Vert \mathcal L \psi(T) \Vert_H + \Vert(1+y)\psi(T)\Vert_H +  \Vert y\nabla   \psi(T)\Vert_H +\Vert g(T)\Vert_H \right)$.
			%where $C$ is a constant which does not depend of $j$ 

			We can take $v=\bar{u}_j(t)$ in \eqref{eqforubar} and we obtain
			\begin{align*}
			&	-\left(  \frac{\partial \bar{u}_j}{\partial t }(t), \bar{u}_j(t)\right)_H + a_{\lambda}^{(M)}(\bar{u}_j(t),\bar{u}_j(t))- \frac 1 \varepsilon \left(  \bigg(\frac{\partial \psi_j}{\partial t }(t)-\bar{u}_j(t)\bigg)\mathbbm{1}_{\{ \psi_j (t)\geq u_j(t)\}},\bar{u}_j  (t) \right)_H
			\\&\quad	=\left(\frac{\partial g}{\partial t}(t),\bar{u}_j(t)\right)_H,
			\end{align*}
			so that
			\begin{align*}
			&	-\frac 1 2\frac{d}{d t} \left\Vert   \bar{u}_j(t)\right\Vert_H^2 + \frac{\delta_1}{2}\Vert \bar{u}_j(t)\Vert_V^2\\&\leq \frac 1 \varepsilon \left(  \bigg(\frac{\partial \psi_j}{\partial t }(t)-\bar{u}_j(t)\bigg)\mathbbm{1}_{\{ \psi_j(t) \geq u_j\}},\bar{u}_j (t)  \right)_H\!\!\!+\left(\frac{\partial g}{\partial t}(t),\bar{u}_j(t)\right)_H\\
			&\leq \frac 1 \varepsilon \left(  \frac{\partial \psi_j}{\partial t }(t)\mathbbm{1}_{\{ \psi_j(t) \geq u_j\}},\bar{u}_j (t)  \right)_H+\left(\frac{\partial g}{\partial t}(t),\bar{u}_j(t)\right)_H.
			\end{align*}
			Integrating between $t$ and $T$, with the usual calculations, we obtain, in particular, that
			\begin{equation}\label{estim3}
			\begin{split}
			&	\Vert \bar{u}_j(t)\Vert_H^2+\frac{\delta_1}{2}	 \int_t^T\Vert \bar{u}_j(s)\Vert_V^2 ds \\&  \leq C_\varepsilon \bigg( \! \Vert \psi(T) \Vert^2_{H^2(\O,\m)} \!+\Vert (1+y) \psi(T) \Vert^2_{H} +\Vert g(T)\Vert^2_H   \!+\! \left \Vert \frac{\partial \psi}{\partial t }\right \Vert^2_{L^2([t,T];H)}\!\!\!\!+ \left \Vert \frac{\partial g}{\partial t }\right \Vert^2_{L^2([t,T];H)} \! \bigg),
			\end{split}
			\end{equation}
			where $C_\varepsilon$ is a constant which depends on $\varepsilon$, but not on $j$.

			\item \textbf{Passage to the limit}
			
			Let $\varepsilon$ and $M$ be fixed. By passing to a subsequence, from \eqref{estim2} we can assume that $\frac{\partial u_j}{\partial t}$ weakly converges to a function $u_{\varepsilon,\lambda,M}'$ in $L^2([0,T];H)$. We deduce that, for any fixed $t\in [0,T]$, $u_j(t)$  weakly  converges in $H$ to 
			$$
			u_{\varepsilon,\lambda,M}(t)= \psi(T)-\int_t^T u_{\varepsilon,\lambda,M}'(s) ds. 
			$$ 
			Indeed, $u_{j}(t)$ is bounded in $V$, so the convergence is weakly in $V$. Passing to the limit in \eqref{estim3} we deduce that $\frac{\partial u_{\varepsilon,\lambda,M}}{\partial t } \in L^2([0,T];V)$.
			Moreover, from $\eqref{estim2}$, we have that $(\psi_j-u_j(t))^+ $ weakly converges in $H$ to a certain function $\chi(t) \in H$. 
			Now, for any $v\in V$ we know that there exists a sequence $(v_j)_{j\in \N}$  such that $v_j \in V_j $ for all $j \in\N $ and $\Vert v-v_j\Vert_V\rightarrow0$. 
			We have
			$$
			-\left( \frac{\partial u_j}{\partial t }(t),v_j   \right)_H + a_{\lambda}^{(M)}(u_j(t),v_j)_H -\frac 1 \varepsilon ((\psi_j(t)-u_j(t))_+,v_j)_H= (g(t),v_j)_H
			$$
			so, passing to the limit as $j\rightarrow \infty$, 
			$$
			-\left( \frac{\partial u_{\varepsilon,\lambda,M}}{\partial t }(t),v   \right)_H + a_\lambda(u_{\varepsilon,\lambda,M}(t),v)_H-\frac 1 \varepsilon (\chi(t),v)_H= (g(t),v)_H.
			$$
			We only have to note that $\chi(t)= (\psi(t)-u_{\varepsilon,\lambda,M}(t))_+$. In fact, $\psi_j(t)\rightarrow \psi(t)$ in $V$ and, up to a subsequence, $\mathbbm{1}_\mathcal{U}u_j(t)\rightarrow \mathbbm{1}_\mathcal{U}u_{\varepsilon,\lambda,M}(t)$ in $L^2(\mathcal{U},\m)$ for every open $\mathcal{U} $ relatively compact in $\mathcal{O}$. Therefore,  there exists a subsequence which converges a.e.  and this allows to conclude the proof. 
		\end{enumerate}
	\end{proof}
	We  now want to get rid of the truncated operator, that is to pass to the limit for $M\rightarrow  \infty$. In order to do this we need some  estimates on the function $u_{\varepsilon,\lambda,M}$ which are uniform in $M$.
	\begin{lemma}\label{lemma_estim}
		Assume that, in addition to the assumptions of Proposition \ref{penalizedcoercivetruncatedproblem},  
		$  \sqrt{1+y}\psi \in L^2([0,T];V),$    $\left| \frac{\partial \psi}{\partial t} \right|\leq \Psi$ with $\Psi\in L^2([0,T];V)$ and $g$ satisfies Assumption $\mathcal{H}^0$.
		Let $u_{\varepsilon,\lambda,M}$ be the solution of \eqref{PTCP}.
		Then,
		\begin{equation}\label{estim4}
		\begin{array}{c}
		\int_t^T\left \Vert \frac{\partial u_{\varepsilon,\lambda,M}}{\partial s}(s) \right \Vert^2_Hds+   \Vert u_{\varepsilon,\lambda,M}(t)\Vert_V^2 +\frac 1 {\varepsilon} \Vert (\psi(t)-u_{\varepsilon,\lambda,M}(t))_+\Vert^2_H\\
		\leq C  \left( \Vert 	\Psi \Vert_{L^2([0,T];V)}+ \|\sqrt{1+y}g\|_{L^2([0,T];H)} +\|\sqrt{1+y}\psi\|^2_{L^2([0,T];V)}   +\Vert\psi(T)\Vert_V^2\right),
		\end{array}
		\end{equation}   	
		where $C$ is a positive constant independent of $M$ and $\varepsilon$.
	\end{lemma}
	\begin{proof}
		To simplify the notation we denote $u_{\varepsilon,\lambda,M}$ by $u$ and $u_{\varepsilon,\lambda,M}-\psi=u-\psi$ by $w$.
		For $n\geq 0$, 
		define  $\varphi_n(x,y)=1+y\wedge n$. Since  $\varphi_n$ and its derivatives are bounded, if $v\in V$,  we have $v\varphi_n\in V$. 
		Choosing  $v= (u-\psi)\varphi_n=w\varphi_n$ in \eqref{PTCP}, with simple passages we get
		%		 $$
		%		 		-\left( \frac{\partial u}{\partial t },(u-\psi)\varphi_n   \right)_H + a^{(M)}_{\lambda}(u,(u-\psi)\varphi_n)
		%				+ (\zeta_\varepsilon(u),(u-\psi)\varphi_n)_H= (g,(u-\psi)\varphi_n)_H,
		%$$
		%		 	so that
		\begin{align*}
		&	-\left( \frac{\partial w}{\partial t }(t),w(t)\varphi_n   \right)_H  + a^{(M)}_{\lambda}(w(t),w(t)\varphi_n)
		+ (\zeta_\varepsilon(t,u(t)),w(t)\varphi_n)_H \\&\qquad    = \left(\frac{\partial \psi}{\partial t }(t)+g(t),w(t)\varphi_n\right)_H- a^{(M)}_{\lambda}(\psi(t),w(t)\varphi_n).
		\end{align*}
		With the notation  $\varphi'_n=\frac{\partial \varphi_n}{\partial y}= \ind{\{y\leq n\}}$, we have
		\begin{align*}
		& a^{(M)}_{\lambda}(w(t),w(t)\varphi_n)= \\&
		\into\frac{y}{2}\left[\left(\frac{\partial w}{\partial x}(t)\right)^2+
		2\rho\sigma \frac{\partial w}{\partial x}(t)\frac{\partial w}{\partial y}(t)+\sigma^2 \left(\frac{\partial w}{\partial y}(t)\right)^2
		\right ]\varphi_n d\m	+      \lambda \into (1+y)w^2(t)\varphi_nd\m\\
		&+\into \frac{y}{2}\left(\rho\sigma \frac{\partial w}{\partial x}(t)+
		\sigma^2\frac{\partial w}{\partial y}(t)\right)w(t)\varphi'_nd\m
		+\into y\wedge M\left(\frac{\partial w}{\partial x}(t)j_{\gamma,\mu}+
		\frac{\partial w}{\partial y}(t)k_{\gamma,\mu}\right)w(t)\varphi_nd\m\\
		& \geq\delta_1 \into y\left|\nabla w(t)
		\right|^2 \varphi_n d\m+\lambda \into (1+y)w^2(t)\varphi_nd\m
		-K_1\into y\left|\nabla w(t)\right|| w(t)|\varphi_n d\m\\
		&\qquad
		-K_2\into y\left|\nabla w(t)\right|| w(t)|\ind{\{y\leq n\}}d\m,
		\end{align*}
		where $K_2=\frac{\sqrt{\rho^2\sigma^2+\sigma^4}}{2}$. Note that, if $n=0$, the last term vanishes, and that, for all $n>0$,
		\[
		\into y\left|\nabla w(t)\right| |w(t)|\ind{\{y\leq n\}}d\m\leq \Vert w(t)\Vert _V^2.
		\]
		Therefore, for all $\zeta>0$,
		\begin{align*}
		&  a^{(M)}_{\lambda}(w(t),w(t)\varphi_n)
		\geq\delta_1 \into y\left|\nabla w(t)
		\right|^2 \varphi_n d\m+\lambda \into (1+y)w^2(t)\varphi_nd\m\\
		&\qquad
		-K_1\into y\left(\frac{\zeta}{2}\left|\nabla w(t)\right|^2
		+\frac{1}{2\zeta}|w(t)|^2\right)\varphi_n d\m				
		-K_2\Vert w(t)\Vert _V^2\\
		&\qquad \geq\left( \delta_1 -\frac{K_1\zeta}{2}\right)\into y\left|\nabla w(t)
		\right|^2 \varphi_n d\m+\left(\lambda -\frac{K_1}{2\zeta}\right)\into (1+y)w^2(t)\varphi_nd\m			
		-K_2\Vert w(t)\Vert _V^2\\
		&\qquad\geq 
		\frac{\delta_1}{2}\into\left( y\left|\nabla w(t)
		\right|^2+(1+y)w^2(t)\right)\varphi_nd\m-K_2\Vert w(t)\Vert _V^2,
		\end{align*}
		where, for the last inequality, we have chosen $\zeta=\delta_1/K_1$ and used the inequality 
		$\lambda\geq \frac{\delta_1}{2}+\frac{K_1^2}{2\delta_1}$. Again, in the case $n=0$ the last term 
		on the righthand side can be omitted.

		Hence, we have, with the notation $\Vert v\Vert ^2_{V,n}=\into\left( y\left|\nabla v
		\right|^2+(1+y)v^2\right)\varphi_nd\m$,		 		
		\begin{align*}
		&\-\frac{1}{2}\frac{d}{dt}\into w^2(t)\varphi_n d\m+
		\frac{\delta_1}{2}\Vert w(t)\Vert ^2_{V,n}
		+\frac{1}\varepsilon \into (-w(t))_+^2\varphi_nd\m\leq  	\\&\qquad	
		\left(g(t)+ \frac{\partial \psi}{\partial t }(t),w(t)\varphi_n   \right)_H
		-a^{(M)}_{\lambda}(\psi(t),w(t)\varphi_n)
		+K_2\Vert w(t)\Vert _V^2.
		\end{align*}
		In the case $n=0$, the inequality reduces to
		$$
		-\frac{1}{2}\frac{d}{dt}\into w^2(t)d\m+
		\frac{\delta_1}{2}\Vert w(t)\Vert ^2_{V}
		+\frac{1}\varepsilon \into (\psi-u)_+^2d\m
		\leq  		
		\left(g(t)+ \frac{\partial \psi}{\partial t }(t),w(t)   \right)_H
		-a^{(M)}_{\lambda}(\psi(t),w(t)).
		$$
		%	Note that it is easy to check that, for all $v$, $w\in V$,
		%	\[
		%	|a^{(M)}_{\lambda}(w,v\varphi_n)|\leq K_3 \Vert w\Vert _{V,n}\Vert v\Vert _{V,n}, \mbox{ with } K_3=\delta_0+K_1+K_2+\lambda.
		%	\]
		Now, integrate from  $t$ to $T$ and use $u(T)=\psi(T)$ to derive
		\begin{equation}\label{calcul}
		\begin{split}
		&	\frac{1}{2}\into w(t)^2\varphi_n d\m+
		\frac 	{\delta_1} 2\int_t^T ds \Vert w(s)\Vert ^2_{V,n}
		+\frac{1}\varepsilon\int_t^T ds \into (-w(s))_+^2\varphi_n d\m\\	&\leq\int_t^T\left(g(s)+ \frac{\partial \psi}{\partial t }(s),w(s)\varphi_n   \right)_Hds+\left|\int_t^Ta^{(M)}_{\lambda}(\psi(s),w(s)\varphi_n) ds\right|  +K_2\int_t^T\Vert  w(s) \Vert _V^2ds,
		\end{split}
		\end{equation}
		and, in the case $n=0$,
		\begin{equation}\label{calcul0}
		\begin{split}
		&		\frac{1}{2}\Vert w(t)\Vert _H^2+
		\frac 	{\delta_1} 2\int_t^T\!  \Vert w(s)\Vert ^2_{V}ds
		+\frac{1}\varepsilon\int_t^T ds \into (-w(s))_+^2 d\m\\&\quad	\leq\int_t^T\!\left(g(s)+ \frac{\partial \psi}{\partial t }(s),w(s)
		\right)_H\!\|ds+\int_t^T\left|a^{(M)}_{\lambda}(\psi(s),w(s))\right|  ds.
		\end{split}
		\end{equation}
		
		We have, for all $\zeta_1>0$,
		\begin{align*}
		&	\int_t^T\bigg(g(s)+ \frac{\partial \psi}{\partial t }(s),w(s)\varphi_n   \bigg)_Hds\\&\quad \leq \frac {\zeta_1}2 \int_t^Tds\into|w(s)|^2\varphi_nd\m+\frac 1 {2\zeta_1}\int_t^Tds\into\left|g(s)+ \frac{\partial \psi}{\partial t }(s)\right|^2\varphi_n d\m\\
		&\quad \leq \frac {\zeta_1}2 \int_t^Tds\into|w(s)|^2\varphi_nd\m+\frac 1 {\zeta_1}\|\sqrt{1+y}g\|_{L^2([t,T];H)}^2+\frac 1 {\zeta_1}
		\left	\|\sqrt{1+y}\frac{\partial\psi}{\partial t}\right\|_{L^2([t,T];H)}^2.
		\end{align*}
		Moreover,	 it is easy to check that, for all $v_1$, $v_2\in V$,
		\[
		|a^{(M)}_{\lambda}(v_1,v_2\varphi_n)|\leq K_3 \Vert v_1\Vert _{V,n}\Vert v_2\Vert _{V,n}, \mbox{\qquad with } K_3=\delta_0+K_1+K_2+\lambda,
		\]
		so that, for any $\zeta_2>0$,
		\begin{align*}
		&	\int_t^T |a^{(M)}_{\lambda}(\psi(s),w(s)\varphi_n)| ds\\&\quad\leq K_3\int_t^Tds\Vert \psi(s)\Vert _{V,n}\Vert w(s)\Vert _{V,n} \leq  \frac {K_3\zeta_2}2\int_t^Tds\Vert w(s)\Vert ^2_{V,n}
		+\frac {K_3}{2\zeta_2}\int_t^Tds\Vert \psi(s)\Vert ^2_{V,n}.
		\end{align*}
		Now, if we chose  $\zeta_1=K_3\zeta_2=\delta_1/4$
		and we go back to \eqref{calcul} and \eqref{calcul0}, using $\left|\frac{\partial \psi}{\partial t}\right|\leq \Psi$ we get
		\begin{equation}\label{calculbis}
		\begin{split}
		&	\frac{1}{2}\into w^2(t)\varphi_n d\m+
		\frac 	{\delta_1} 4\int_t^T  \Vert w(s)\Vert ^2_{V,n}ds
		+\frac{1}\varepsilon\int_t^T ds \into (-w(s))_+^2\varphi_n d\m\\	&\leq
		\frac 4 {\delta_1}\left(\|\sqrt{1+y}g\|_{L^2([t,T];H)}^2+\|\sqrt{1+y}\Psi\|_{L^2([t,T];H)}^2\right)
		+\frac {2K_3^2}{\delta_1}\int_t^T\Vert \psi(s)\Vert ^2_{V,n}ds\\&\qquad+K_2\Vert  w \Vert _{L^2([t,T];H)}^2,\\
		&\leq
		\frac 4 {\delta_1}\left(\|\sqrt{1+y}g\|_{L^2([t,T];H)}^2+\|\sqrt{1+y}\Psi\|_{L^2([t,T];H)}^2\right)
		+\frac {4K_3^2}{\delta_1}\left\|\sqrt{1+y}\psi\right\|^2_{L^2([t,T];V)}\\&\qquad+K_2\Vert  w \Vert _{L^2([t,T];H)}^2,
		\end{split}
		\end{equation}
		where the last inequality follows from the estimate $\Vert v\Vert ^2_{V,n}\leq 2\Vert \sqrt{1+y}v\Vert _V^2$, and, in the case $n=0$,
		\begin{equation}\label{calcul0bis}
		\begin{split}
		&	\frac{1}{2}\Vert w(t)\Vert _H^2+
		\frac 	{\delta_1} 4\int_t^T \!\! \Vert w(s)\Vert ^2_{V}
		ds	+\frac{1}\varepsilon\int_t^T \!ds\! \into \!(-w(s))_+^2 d\m\\&\quad	\leq
		\frac 4 {\delta_1}\left(\|g\|_{L^2([t,T];H)}^2+\|\Psi\|_{L^2([t,T];H)}^2\right)
		+\frac {2K_3^2}{\delta_1}\| \psi\|^2_{L^2([t,T];V)}.
		\end{split}
		\end{equation}
		From \eqref{calcul0bis} recalling that $w=u-\psi$ we deduce 
		\begin{equation}\label{calcul0tris}
		\begin{split}
		\int_t^T\!\!\|u(s)\|_V^2ds&\leq \int_t^T\!\!2( \|w(s)\|_V^2+\|\psi(s)\|_V^2)ds\\& \leq \frac{32}{\delta_1^2}\left(\|g\|_{L^2([t,T];H)}^2+\|\Psi\|_{L^2([t,T];H)}^2\right)+\left(\frac {16K_3^2}{\delta_1^2}+2\right)\| \psi\|^2_{L^2([t,T];V)}.
		\end{split}
		\end{equation}
		Moreover, combining \eqref{calculbis} and \eqref{calcul0bis}, we have
		\begin{equation*}
		\begin{split}
		&	\frac{1}{2}\into w^2(t)\varphi_n d\m+
		\frac 	{\delta_1} 4\int_t^T  \Vert w(s)\Vert ^2_{V,n}ds
		+\frac{1}\varepsilon\int_t^T ds \into (-w(s))_+^2\varphi_n d\m\\	&\leq
		\left(\frac 4 {\delta_1}+ \frac{16K_2}{\delta_1^2}\right)\left(\|\sqrt{1+y}g\|_{L^2([t,T];H)}^2+\|\sqrt{1+y}\Psi\|_{L^2([t,T];H)}^2\right)
		\\&\quad+\frac {4K_3^2}{\delta_1}\left(1+\frac{2K_2}{\delta_1}\right)\|\sqrt{1+y}\psi\|^2_{L^2([t,T];V)}.
		\end{split}
		\end{equation*}
		In particular,
		\begin{equation*}
		\begin{split}
		&\int_t^Tds\into y|\nabla u(s)|^2\varphi_n d\m \leq \int_t^T\Vert u(s)\Vert ^2_{V,n }ds\leq 2 \int_t^T  \Vert w(s)\Vert ^2_{V,n}ds+ 2\int_t^T ds\Vert \psi(s)\Vert ^2_{V,n}ds\\
		&\leq \frac 8 {\delta_1}  \left(\frac 4 {\delta_1}+ \frac{16K_2}{\delta_1^2}\right)\left(\|\sqrt{1+y}g\|_{L^2([t,T];H)}^2+\|\sqrt{1+y}\Psi\|_{L^2([t,T];H)}^2\right)
		\\&\quad+\left(\frac {32K_3^2}{\delta_1^2}\left(1+\frac{2K_2}{\delta_1}\right)+4\right)\|\sqrt{1+y}\psi\|^2_{L^2([t,T];V)} 
		\end{split}
		\end{equation*}
		and, 
		%		recalling that $\|\sqrt{y}|\nabla u| \|_{H}\leq \infty$ and 
		by	using the Monotone convergence theorem, we deduce 
		\begin{equation}\label{uchapeau}
		\begin{split}
		& 		\int_t^T|y|\nabla u(s)| \|_{H}^2 ds\\&\qquad \leq K_4\left(\|\sqrt{1+y}g\|_{L^2([t,T];H)}^2\!+\|\sqrt{1+y}\Psi\|_{L^2([t,T];H)}^2\!+\|\sqrt{1+y}\psi\|^2_{L^2([t,T];V)} \right),
		\end{split}
		\end{equation}
		where $K_4=\frac 8 {\delta_1}  \left(\frac 4 {\delta_1}+ \frac{16K_2}{\delta_1^2}\right) \vee\left(\frac {32K_3^2}{\delta_1^2}\left(1+\frac{2K_2}{\delta_1}\right)+4  \right)		$.

		We are now in a position to prove \eqref{estim4}.  Taking $v=\frac{\partial u}{\partial t }$ in \eqref{PTCP}, we have
		$$
		-\left \Vert \frac{\partial u}{\partial t}\right \Vert^2_H+ \bar{a}_\lambda\left(u,\frac{\partial u}{\partial t}\right) +
		\tilde{a}^{(M)}\left(u,\frac{\partial u}{\partial t}\right)-\frac{1}{\varepsilon}\left((\psi-u)_+,\frac{\partial u}{\partial t}\right)_H
		= \left(g(t),\frac{\partial u}{\partial t}(t)\right)_H.
		$$
		Note that, since $\bar a_\lambda$ is symmetric, $\frac{d}{dt}\bar{a}_\lambda\left(u(t),u(t)\right)=2\bar{a}_\lambda\left(u(t),\frac{\partial u}{\partial t}(t)\right)$.
		On the other hand, 
		\begin{align*}
		\left((\psi(t)-u(t))_+,\frac{\partial u}{\partial t}\right)_H &=-\frac{1}{2}\frac{d}{dt}\Vert (\psi(t)-u(t))_+\Vert _H^2+\left((\psi(t)-u(t))_+,\frac{\partial \psi}{\partial t}(t)\right)_H,
		\end{align*}
		so that
		\begin{align*}
		&	\left\Vert \frac{\partial u}{\partial t} (t)\right \Vert^2_H-\frac{1}{2}\frac{d}{dt}\bar{a}_\lambda\left(u(t),u(t)\right) 
		-\frac{1}{2\varepsilon}\frac{d}{dt}\Vert (\psi(t)-u(t))_+\Vert _H^2\\&\qquad
		= \tilde{a}^{(M)}\left(u(t),\frac{\partial u}{\partial t}(t)\right) -\left(g(t),\frac{\partial u}{\partial t}(t)\right)_H
		- \frac 1{\varepsilon}\left((\psi(t)-u(t))_+,\frac{\partial \psi}{\partial t}(t)\right)_H\\
		& \qquad\leq \left|\tilde{a}^{(M)}\left(u(t),\frac{\partial u}{\partial t}(t)\right)\right|+
		\Vert g(t)\Vert _H\left\Vert \frac{\partial u}{\partial t}(t) \right\Vert_H+\frac 1{\varepsilon}\left((\psi(t)-u(t)_+,\Psi (t)\right)_H\\
		& \qquad\leq \left(K_1 \left\Vert y|\nabla u(t)|\right\Vert_H+\Vert g(t)\Vert _H \right) \left\Vert \frac{\partial u}{\partial t}(t)\right\Vert_H
		+\frac 1{\varepsilon}\left((\psi(t)-u(t))_+,\Psi(t) \right)_H.
		\end{align*}		
		Moreover, if we take $v=\Psi(t)$ in \eqref{PTCP}, we get
		$$
		-\left( \frac{\partial u}{\partial t }(t),\Psi(t) \right)_H + a_{\lambda}^{(M)}(u(t),\Psi(t))-
		\frac 1{\epsilon} \left((\psi(t)-u(t))_+,\Psi(t)\right)_H= \left(g(t),\Psi(t) \right)_H,
		$$
		so that
		\begin{equation}\label{estavecv1}
		\begin{split}
		&	\frac{1}{\varepsilon}\left( (\psi(t)-u(t))_+,\Psi(t)\right)_H \leq  \left \Vert    \frac{\partial u}{\partial t } (t)\right \Vert_H \Vert \Psi(t)\Vert_H+ \Vert a^{(M)}_\lambda\Vert 
		\Vert u(t) \Vert_V\Vert \Psi(t)\Vert_V+\Vert g(t)\Vert_H\Vert \Psi(t)\Vert_H.	\end{split}
		\end{equation}
		Therefore,
		\begin{align*}
		&	\left\Vert \frac{\partial u}{\partial t} (t)\right \Vert^2_H-\frac{1}{2}\frac{d}{dt}\bar{a}_\lambda\left(u(t),u(t)\right) 
		-\frac{1}{2\varepsilon}\frac{d}{dt}\Vert (\psi(t)-u(t))_+\Vert _H^2
		\\&	 \leq	 \left(K_1 \left\Vert y|\nabla u(t)|\right\Vert_H+\Vert g(t)\Vert _H +\Vert \Psi(t)\Vert_H\right) 
		\left\Vert \frac{\partial u}{\partial t} (t)\right\Vert_H
		+ \Vert a^{(M)}_\lambda\Vert 
		\Vert u (t)\Vert_V\Vert \Psi(t)\Vert_V\\&	
		\qquad+ \Vert g(t)\Vert_H\Vert \Psi(t)\Vert_H,
		\end{align*}
		hence
		\begin{align*}
		&	\frac{1}{2}\left\Vert \frac{\partial u}{\partial t}(t) \right \Vert^2_H-\frac{1}{2}\frac{d}{dt}\bar{a}_\lambda\left(u(t),u(t)\right) 
		-\frac{1}{2\varepsilon}\frac{d}{dt}\Vert (\psi(t)-u(t))_+\Vert _H^2
		\\&		 \leq \frac{1}{2}\left(K_1 \left\Vert y|\nabla u(t)|\right\Vert_H+\|g(t)\Vert _H +\Vert \Psi(t)\Vert_H\right)^2 		
		+\Vert a^{(M)}_\lambda\Vert 
		\Vert u(t) \Vert_V^2 \Vert \Psi(t)\Vert_V^2+\Vert g(t)\Vert_H\Vert \Psi(t)\Vert_H.
		\end{align*}
		Integrating between $t$ and $T$, we get, 
		\begin{align*}
		&	\frac{1}{2}\left \Vert \frac{\partial u}{\partial s} \right \Vert^2_{L^2([t,T];H)}+ 
		\frac 1 2 \bar{a}_\lambda\left(u(t), u(t)\right) +\frac 1 {2\varepsilon} \Vert (\psi(t)-u(t))_+\Vert^2_H\\&\qquad
		\leq \frac 1 2 \bar{a}_\lambda(\psi(T),\psi(T))+2\|g\Vert ^2_{L^2([t,T];H)}  +
		2\Vert \Psi\Vert_{L^2([t,T];H)}^2 	+\frac{3K_1^2}{2} \left\Vert y|\nabla u|\right\Vert^2_{L^2([t,T];H)}
		\\&\qquad	+\frac{ \|a^{(M)}_\lambda\|}2\Vert u \Vert_{L^2([t,T];V)}+\frac{ \|a^{(M)}_\lambda\|}{2}\Vert \Psi\Vert_{L^2([t,T];V)},
		\end{align*}
		so, recalling that $\bar{a}_\lambda(u(t),u(t)\geq\delta_1\into y|\nabla u(t)|^2d\m+\lambda\into(1+y)u^2d\m\geq (\delta_1\wedge\lambda) \|u(t)\|_V^2$,
		\begin{align*}
		\frac{1}{2}&\left \Vert \frac{\partial u}{\partial s} \right \Vert^2_{L^2([t,T];H)}+ 
		\frac {\delta_1\wedge\lambda} 2 \|u(t)\|_V^2 +\frac 1 {2\varepsilon} \Vert (\psi(t)-u(t))_+\Vert^2_H\\&
		\leq \frac {\|\bar{a}_\lambda\|} 2 \|\psi(T)\|_V^2+2\|g\Vert ^2_{L^2([t,T];H)} +
		2\Vert \Psi\Vert_{L^2([t,T];H)}^2 \\&\quad 	+\frac{3K_1^2}{2} \left\Vert y|\nabla u|\right\Vert^2_{L^2([t,T];H)}
		+\frac{ \Vert a^{(M)}_\lambda\Vert }2\Vert u \Vert^2_{L^2([t,T];V)}+\frac{ \Vert a^{(M)}_\lambda\Vert }{2}\Vert \Psi\Vert^2_{L^2([t,T];V)}\\&\leq  \frac {\|\bar{a}_\lambda\|} 2 \|\psi(T)\|_V^2+2\|g\Vert ^2_{L^2([t,T];H)} +
		2\Vert \Psi\Vert_{L^2([t,T];H)}^2 \\&\quad +
		\frac{3K_1^2}{2}	K_4\left(\|\sqrt{1+y}g\|_{L^2([t,T];H)}^2+\|\sqrt{1+y}\Psi\|_{L^2([t,T];H)}^2+\|\sqrt{1+y}\psi\|^2_{L^2([t,T];V)}\right )\\
		&\quad +\frac{ \Vert a^{(M)}_\lambda\Vert }2\left(\frac{32}{\delta_1^2}\left(\|g\|_{L^2([t,T];H)}^2+\|\Psi\|_{L^2([t,T];H)}^2\right)+\left(\frac {16K_3^2}{\delta_1^2}+2\right)\| \psi\|^2_{L^2([t,T];V)}\right)\\&\quad +\frac{ \Vert a^{(M)}_\lambda\Vert }{2}\Vert \Psi\Vert^2_{L^2([t,T];V)},
		\end{align*}
		where the last inequality follows from \eqref{calcul0tris} and \eqref{uchapeau}.
		Rearranging the terms, 
		we deduce that there exists a constant $C>0$ independent of $M$ and $\varepsilon$ such that 
		\begin{align*}
		&	\frac{1}{2}\left \Vert \frac{\partial u}{\partial s} \right \Vert^2_{L^2([t,T];H)}+ 
		\frac {\delta_1\wedge\lambda} 4 \|u(t)\|_V^2 +\frac 1 {2\varepsilon} \Vert (\psi(t)-u(t))_+\Vert^2_H
		\\&\qquad \leq C\left(\|\sqrt{1+y}g\|_{L^2([t,T];H)}^2+\|\Psi\|_{L^2([t,T];V)}^2+\left\|\sqrt{1+y}\psi\right\|^2_{L^2([t,T];V)}+ \|\psi(T)\|_V^2 \right),
		\end{align*}
		which concludes the proof.
	\end{proof}
	\begin{proof}[Proof of Theorem \ref{penalizedcoerciveproblem}: existence]
		Assume for a first moment that we have the further assumptions $ \psi(T) \in H^2(\O,\m)$, $ (1+y)\psi(T) \in H$, $\frac{\partial \psi}{\partial t}\in L^2([0,T];V) $  and $\frac{\partial g}{\partial t}\in L^2([0,T];H)$.
		Thanks to $\eqref{estim4}$ we can repeat the same arguments as in the proof of Proposition \ref{penalizedcoercivetruncatedproblem} in order to pass to the limit in $j$, but this time as $M\rightarrow \infty $.  
		%		 		
		%		 		In fact, up to pass to a subsequence, from \eqref{estim4} we can suppose that $\frac{\partial u_{\varepsilon,\lambda,M}}{\partial t}$ weakly converges to a function $u_{\varepsilon,\lambda}'$ in $L^2([0,T];H)$. We deduce that, for any fixed $t\in [0,T]$, $u_j(t)$ converges weakly in $H$ to 
		%		 		$$
		%		 		u_{\varepsilon,\lambda}(t)= \psi(T)-\int_t^T u_{\varepsilon,\lambda}'(s) ds. 
		%		 		$$ 
		%		 		Indeed, $u_{\varepsilon,\lambda,M}(t)$ is bounded in $V$, so the convergence is weakly in $V$.
		%		 		Moreover, again from $\eqref{estim4}$ and from the fact that there is a subsequence of $u_{\varepsilon,\lambda,M}(t)$ which converges a.e. to $u_{\lambda,M}(t)$, we get that $(\psi(t)-u_{\varepsilon,\lambda,M}(t))_+ $ weakly converges in $H$ to $(\psi(t)-u_{\varepsilon,\lambda}(t))_+ $. 
		%		 		We have
		%		 		$$
		%		 		-\left( \frac{\partial u_{\varepsilon,\lambda,M}}{\partial t }(t),v   \right)_H + a_{\lambda}^{(M)}(u_{\varepsilon,\lambda,M}(t),v)+ (\zeta_\varepsilon(u_{\varepsilon,\lambda,M})(t),v)_H= (g(t),v)_H
		%		 		$$
		%		 		and, passing to the limit as $M\rightarrow \infty$, we get 
		Therefore, we deduce the existence of a function $u_{\varepsilon,\lambda}\in L^2([0,T];V)$ with $\frac{\partial u_{\varepsilon,\lambda}}{\partial t}\in L^2([0,T];H)$ and such that
		$$
		-\left( \frac{\partial u_{\varepsilon,\lambda}}{\partial t }(t),v   \right)_H + a_\lambda(u_{\varepsilon,\lambda}(t),v)_H-\frac 1 \varepsilon ((\psi(t)-u_{\varepsilon,\lambda}(t))_+,v)_H= (g(t),v)_H.
		$$
		The estimates \eqref{sp1}, \eqref{sp2} and \eqref{sp3} directly follow from \eqref{estim4} as $M\rightarrow \infty$.

		We have now to weaken the assumptions on $g$ and $\psi$. We can do this by a regularization procedure. In fact, let us assume that $\psi$ satisfies Assumption $\mathcal H^1$ (so, in particular, $\left|\frac{\partial \psi}{\partial t}\right|\leq \Psi$ for a certain $\Psi\in L^2([0,T];V)$ and $g$ satisfies Assumption $\mathcal H^0$. Then, by  standard regularization techniques (see  for example \cite[Corollary A.12]{DF}), we can find  sequences of functions $(g_n)_n$, $(\psi_n)_n$ and $(\Psi_n)_n$ of class $C^\infty$ with compact support such that, for any $n\in\N$, $n\in\N$, $|\frac{\partial\psi_n}{\partial t}|\leq \Psi_n$ and all the regularity assumptions required in the first part of the proof are satisfied. Moreover,  it is easy to see that $\Vert \sqrt{1+y}g_n -\sqrt{1+y}g\Vert_{L^2([0,T];H)}\rightarrow0$, $ \Vert \sqrt{1+y}\psi_n -\sqrt{1+y}\psi\Vert_{L^2([0,T];V)} \rightarrow 0$, $ \Vert\Psi_n-\Psi\Vert_{L^2([0,T];V)} \rightarrow 0$, $\Vert \psi_n(T)-\psi(T)\Vert_V\rightarrow0$ as $n\rightarrow \infty$.  Therefore, the solution $u_{\varepsilon,\lambda,M}^n$ of the equation \eqref{PCP} with source function $g_n$ and obstacle function $\psi_n$ satisfies  
		\begin{equation}\label{lessass}
		\begin{array}{c}
		\int_t^T \left \Vert \frac{\partial u^n_{\varepsilon,\lambda,M}}{\partial s} (s)\right \Vert^2_H \,ds+   \Vert u^n_{\varepsilon,\lambda,M}(t)\Vert_V^2 +\frac 1 {\varepsilon} \Vert (\psi_n(t)-u^n_{\varepsilon,\lambda,M}(t))_+\Vert^2_H\\
		\leq C  \left(   \|\sqrt{1+y}g_n\|_{L^2([0,T];H)} +\|\sqrt{1+y}\psi_n\|^2_{L^2([0,T];V)}   +\|\Psi_n\|_{L^2([0,T];V)}^2 +\Vert\psi_n(T)\Vert_V^2 \right).
		\end{array}
		\end{equation} Then, we can take the limit for $n\rightarrow \infty$ in \eqref{lessass} and the assertion follows as in the  first part of the proof.
	\end{proof}

	Moreover, we have the following Comparison principle for the coercive penalized problem.
	\begin{proposition} \label{CompPrinc1}
		
		\begin{enumerate}
			\begin{comment}
			\item If there exists a constant $M>0$ such that $ 0\leq \psi \leq M$ and $0\leq g \leq \lambda M(1+y)$ , then the solution $u_{\varepsilon,\lambda}$ of the penalized coercive problem \eqref{PCP} satisfies 
			$$
			0 \leq u_{\varepsilon,\lambda} \leq M.
			$$
			\end{comment}
			\item Assume that $\psi_i$  satisfies Assumption $\mathcal{H}^1$ for $i=1,2$ and $g$ satisfies Assumption $\mathcal{H}^0$.  Let $u^i_{\varepsilon,\lambda}$ be the unique solution of \eqref{PCP}   with  obstacle function $\psi_i$ and source function $g$.  If $\psi_1\leq \psi_2$, then $u^1_{\varepsilon,\lambda}\leq u^2_{\varepsilon,\lambda}$.
			\item Assume that  $\psi$ satisfies Assumption $\mathcal{H}^1$  and  $g_i$ satisfy Assumption $\mathcal{H}^0$ for $i=1,2$. Let $u^i_{\varepsilon,\lambda}$ 
			be the unique solution of \eqref{PCP} with obstacle function $\psi$ and source function $g_i$.  If $g_1\leq g_2$, then $u^1_{\varepsilon,\lambda}\leq u^2_{\varepsilon,\lambda}$.
			\item Assume that $\psi_i$  satisfies Assumption $\mathcal{H}^1$ for $i=1,2$ and $g$ satisfies Assumption $\mathcal{H}^0$.  Let $u^i_{\varepsilon,\lambda}$ be the unique solution of \eqref{PCP}   with  obstacle function $\psi_i$ and source function $g$.  If $\psi_1- \psi_2\in L^\infty$, then $u^1_{\varepsilon,\lambda}- u^2_{\varepsilon,\lambda}\in L^\infty $ and $\Vert  u^1_{\varepsilon,\lambda}- u^2_{\varepsilon,\lambda}\Vert_\infty \leq   \Vert \psi_1- \psi_2\Vert_\infty$.
		\end{enumerate}
	\end{proposition}
	Proposition \ref{CompPrinc1} can be proved with standard techniques introduced in \cite[Chapter 3]{BL} so we omit the proof.
	\subsubsection{Coercive variational inequality}
	\begin{proposition}\label{coercive variational inequality}
		Assume that $\psi$ satisfies Assumption $\mathcal{H}^1$  and $g$ satisfies Assumption $\mathcal{H}^0$.
		Moreover, assume that $ 0\leq \psi \leq \Phi$ with $\Phi\in L^2([0,T]; H^2(\mathcal{O},\m))$ such  that $\frac{\partial \Phi}{\partial t}+\mathcal{L}\Phi \leq 0$ and $0\leq g \leq -\frac{\partial \Phi}{\partial t} -\mathcal{L}^\lambda \Phi$.
		Then, there exists a unique function $u_\lambda $ such that $ u_\lambda\in L^2([0,T];V), \,\frac{\partial u_\lambda}{\partial t } \in L^2([0,T]; H)$ and 
		\begin{equation} \label{CVI}
		\begin{cases}
		-\left( \frac{\partial u_\lambda}{\partial t },v -u_\lambda  \right)_H + a_\lambda(u_\lambda,v-u_\lambda)\geq (g,v-u_\lambda)_H, \quad \mbox{a.e. in } [0,T],\\\qquad \qquad\qquad\qquad\qquad\qquad\qquad\qquad\qquad\qquad\qquad\qquad v\in  L^2([0,T];V), \ v\geq \psi, \\u_\lambda(T)=\psi(T),\\u_\lambda \geq \psi \mbox{ a.e. in } [0,T]\times \R \times (0,\infty).
		\end{cases}
		\end{equation}
		Moreover, $0 \leq u_\lambda \leq  \Phi$.
	\end{proposition}
	%		 	\begin{proof}[Proof of uniqueness in Proposition \ref{coercive variational inequality}]
	%		 		Suppose that there are two functions $u_1 $ and $u_2$ which satisfy \eqref{CVI}. We can take $v=u_2$ in the equation satisfied by $u_1$ and $v=u_1$ in the one satisfied by $u_2$ and we get
	%		 		$$
	%		 		-\left( \frac{\partial u_1}{\partial t },u_2-u_1  \right)_H + a_\lambda(u_1, u_2-u_1)\geq (g,u_2-u_1)_H,
	%		 		$$
	%		 		$$
	%		 		-\left( \frac{\partial u_2}{\partial t },u_1-u_2 \right)_H + a_\lambda(u_2, u_1-u_2)\geq (g,u_1-u_2)_H.
	%		 		$$
	%		 		Setting $w:= u_2-u_1$ and adding the second equation from the first one we obtain
	%		 		$$
	%		 		\left( \frac{\partial w}{\partial t },w \right)_H - a_\lambda(w,w)\geq 0,
	%		 		$$
	%so that
	%		 		$$
	%		 		\left( \frac{\partial w}{\partial t },w \right)_H =\frac 1 2\frac{d}{d t} \Vert w \Vert_H \geq 0 .
	%		 		$$
	%		 		But $w(T)=u_1(T)-u_2(T)=\psi(T)-\psi(T)=0$ and, therefore, $w\equiv0$, that is $u_1=u_2 $.
	%		 	\end{proof}
	\begin{proof}
		The uniqueness of the solution of \eqref{CVI} follows by a standard monotonicity argument introduced in \cite[Chapter 3]{BL} (see \cite{T}). 
		As regards the existence of a solution, we follow  the lines of the proof of \cite[Theorem 2.1]{BL} but we repeat here the details since we use a compactness argument which is not present in the classical theory.
		
		For each fixed $\varepsilon >0$ we have the estimates \eqref{sp1} and \eqref{sp2}, so, for every $t\in[0,T]$, we can extract a subsequence $u_{\varepsilon,\lambda}$ such that $u_{\varepsilon,\lambda}(t) \rightharpoonup u_\lambda(t)$ in $V$ as $\varepsilon \rightarrow 0 $  and $u'_\varepsilon(t) \rightharpoonup u'_\lambda(t)$ in $H$ for some function $u_\lambda \in V$.
		
		Note that $u=0 $ is the unique solution of \eqref{PCP} when $\psi=g=0$, while $u=\Phi $ is the unique solution of \eqref{PCP} when $\psi=\Phi$ and $g=-\frac{\partial \Phi}{\partial t}-\mathcal{L}^\lambda \Phi=-\frac{\partial \Phi}{\partial t}-\mathcal{L} \Phi +\lambda(1+y)\Phi  $. Therefore, Proposition \ref{CompPrinc1} implies that $0\leq u_{\varepsilon,\lambda} \leq \Phi$. Recall that $u_{\varepsilon,\lambda}(t) \rightarrow u_\lambda(t)$ in $L^2(\mathcal{U}, \m ) $ for every relatively compact open $\mathcal{U} \subset \mathcal{O}$. This, together with the fact that $d\m$ is a finite measure, allows to conclude that we have strong convergence of $u_{\varepsilon,\lambda}$ to $u_\lambda$ in $H$. In fact, if $\delta >0$ and  $\mathcal{O}_{\delta}:=(-\frac 1 \delta, \frac 1 \delta )\times (\delta, \frac 1 \delta)$, 
		\begin{align*}
		&	\int_0^Tds	\int_{\mathcal{O}} |u_{\varepsilon,\lambda}(s)- u_\lambda(s)|^2d\m\\&\qquad
		\leq  	\int_0^T\!ds\int_{\mathcal{O}_{\delta}}\! |u_{\varepsilon,\lambda}(s)- u_\lambda(s)|^2d\m+	\int_0^T\!ds\int_{\mathcal{O}^c_{\delta}} \!|u_{\varepsilon,\lambda}(s)- u_\lambda(s)|^2d\m\\&\qquad\leq 	\int_0^T\!ds\int_{\mathcal{O}_{\delta}}\! |u_{\varepsilon,\lambda}(s)- u_\lambda(s)|^2d\m + 	\int_0^T\!ds\int_{\mathcal{O}^c_{\delta}} \!4\Phi^2(s) d\m
		\end{align*} 
		%	\begin{align*}
		%	\int_{\mathcal{O}} |u_{\varepsilon,\lambda}- u_\lambda|^2d\m&\leq  \int_{\mathcal{O}_{\delta}} |u_{\varepsilon,\lambda}- u_\lambda|^2d\m+\int_{\mathcal{O}^c_{\delta}} |u_{\varepsilon,\lambda}- u_\lambda|^2d\m\\
		%	&\leq \int_{\mathcal{O}_{\delta}} |u_{\varepsilon,\lambda}- u_\lambda|^2d\m + 4M^2 \m(\mathcal{O}^c_{\delta}),
		%	\end{align*} 
		and it is enough to let $\delta$ goes to 0.
		
		From \eqref{sp3} we also have that $(\psi(t)-u_{\varepsilon,\lambda}(t))^+ \rightarrow 0 $ strongly in $H$ as $\varepsilon \rightarrow 0$ . On the other hand $(\psi(t)-u_{\varepsilon,\lambda}(t))_+\rightharpoonup \chi(t)$ weakly in $H$ and $\chi =(\psi-u_\lambda)_+$ since there exists a subsequence of $u_{\varepsilon,\lambda} (t)
		$ which converges pointwise to $u_\lambda(t)$. Therefore, $(\psi(t)-u_\lambda(t))^+=0$, which means $u_\lambda(t) \geq \psi(t)$.
		
		Then we consider the penalized coercive equation in \eqref{PCP} replacing $v$ by $v-u_{\varepsilon,\lambda}(t)$, with $v \geq \psi(t)$. Since $ \zeta_\varepsilon(t,v)=0$ and $ (\zeta_\varepsilon(t,v)- \zeta_\varepsilon(t,u_{\varepsilon,\lambda}(t)),v-u_{\varepsilon,\lambda}(t))_H\geq 0$ we easily deduce that
		%		 		If we consider the penalized coercive equation in \eqref{PCP} replacing $v$ by $v-u_{\varepsilon,\lambda}$, with $v \geq \psi$, we have 
		%		 		$$
		%		 		-\left( \frac{\partial u_{\varepsilon,\lambda}}{\partial t },v-u_{\varepsilon,\lambda}   \right)_H + a_\lambda(u_{\varepsilon,\lambda},v-u_{\varepsilon,\lambda})_H+ (\zeta_\varepsilon(u_{\varepsilon,\lambda}),v-u_{\varepsilon,\lambda})_H= (g,v-u_{\varepsilon,\lambda})_H.
		%		 		$$
		%		 		Since $ \zeta_\varepsilon(v)=0$, we can write
		%		 		$$
		%		 		-\left( \frac{\partial u_{\varepsilon,\lambda}}{\partial t },v-u_{\varepsilon,\lambda}   \right)_H + a_\lambda(u_{\varepsilon,\lambda},v-u_{\varepsilon,\lambda})_H- \underbrace{(\zeta_\varepsilon(v)- \zeta_\varepsilon(u_{\varepsilon,\lambda}),v-u_{\varepsilon,\lambda})_H}_{\geq 0 }= (g,v-u_{\varepsilon,\lambda})_H.
		%		 		$$
		%		Therefore 
		$$
		-\left( \frac{\partial u_{\varepsilon,\lambda}}{\partial t }(t),v-u_{\varepsilon,\lambda}(t)   \right)_H + a_\lambda(u_{\varepsilon,\lambda}(t),v-u_{\varepsilon,\lambda}(t))\geq (g(t),v-u_{\varepsilon,\lambda}(t))_H
		$$ 
		so that, letting $\varepsilon $ goes to 0, we have
		\begin{align*}
		-\left( \frac{\partial u_\lambda}{\partial t }(t),v-u_\lambda (t) \right)_H + a_\lambda(u_\lambda(t),v)&\geq (g(t),v-u_\lambda(t))_H+ \liminf_{\varepsilon\rightarrow 0} a_\lambda(u_{\varepsilon,\lambda}(t),u_{\varepsilon,\lambda}(t))\\& \geq (g(t),v-u_\lambda(t))_H+  a_\lambda(u_\lambda(t),u_\lambda(t)).
		\end{align*}
		
		Moreover, since $0\leq u_{\varepsilon,\lambda}\leq \Phi$ for every $\varepsilon>0$ and $u_\lambda=\lim_{\varepsilon  \rightarrow 0} u_{\varepsilon,\lambda}$, we have $0\leq u_\lambda \leq \Phi$ and the assertion follows.
	\end{proof}
	The following Comparison Principle is a direct consequence of Proposition \ref{CompPrinc1},.
	\begin{proposition}\label{CompPrinc2}
		
		\begin{enumerate}
			\item 
			For $i=1,\,2$, assume that  $\psi_i$ satisfies Assumption $\mathcal{H}^1$, $g$ satisfies Assumption $\mathcal{H}^0$ and $ 0\leq \psi_i \leq \Phi$ with $\Phi\in L^2([0,T];H^2(\mathcal{O},\m))$ such  that $\frac{\partial \Phi}{\partial t}+\mathcal{L}\Phi \leq 0$ and $0\leq g \leq -\frac{\partial \Phi}{\partial t} -\mathcal{L}^\lambda \Phi$. Let  $u^i_{\lambda}$  be the unique solution of \eqref{CVI}  with obstacle function $\psi_i$ and source function $g$.  If $\psi_1\leq \psi_2$, then $u^1_{\lambda}\leq u^2_{\lambda}$.
			\item For $i=1,\,2$, assume that  $\psi$ satisfies Assumption $\mathcal{H}^1$, $g_i$ satisfy Assumption $\mathcal{H}^0$ and $ 0\leq \psi \leq \Phi$ with $\Phi\in L^2([0,T];H^2(\mathcal{O},\m))$ such  that $\frac{\partial \Phi}{\partial t}+\mathcal{L}\Phi \leq 0$ and $0\leq g_i \leq -\frac{\partial \Phi}{\partial t} -\mathcal{L}^\lambda \Phi$. Let  $u^i_{\lambda}$  be the unique solution of \eqref{CVI}  with obstacle function $\psi$ and source function $g_i$. If $g_1\leq g_2$, then $u^1_{\lambda}\leq u^2_{\lambda}$.
			\item 	For $i=1,\,2$, assume that  $\psi_i$ satisfies Assumption $\mathcal{H}^1$, $g$ satisfies Assumption $\mathcal{H}^0$ and $ 0\leq \psi_i \leq \Phi$ with $\Phi\in L^2([0,T];H^2(\mathcal{O},\m))$ such  that $\frac{\partial \Phi}{\partial t}+\mathcal{L}\Phi \leq 0$ and $0\leq g \leq -\frac{\partial \Phi}{\partial t} -\mathcal{L}^\lambda \Phi$. Let  $u^i_{\lambda}$  be the unique solution of \eqref{CVI}  with obstacle function $\psi_i$ and source function $g$.  If $\psi_1- \psi_2\in L^\infty$, then $u^1_{\lambda}- u^2_{\lambda}\in L^\infty $ and $\Vert  u^1_{\lambda}- u^2_{\lambda}\Vert_\infty \leq   \Vert \psi_1- \psi_2\Vert_\infty$.
		\end{enumerate}
	\end{proposition}
	\subsubsection{Non-coercive variational inequality}
	We can finally prove Theorem \ref{variationalinequality}. Again, we first study the uniqueness of the solution and then we deal with the existence.
	%		 	\begin{theorem} \label{ncvi}
	%		 		%	Assume that $\psi$ is such that $\psi, (1+y)\psi \in L^2([0,T];V)$ and there exists a constant $M>0$ such that $0\leq \psi \leq M$. 
	%		 	Let $\psi$ be a function such that $\psi, (1+y)\psi \in L^2([0,T];V)$ and $\left| \frac{\partial \psi}{\partial t} \right|\leq C$ for a certain constant $C>0$. Moreover, assume that there exists a function $M\in H^2(\mathcal{O},\m)$ and such that $\mathcal{L}M \leq 0$, $0\leq \psi \leq M$, $(1+y)^2M\in H$ and $M\in L^2(\mathcal{O},\m_{\gamma,\mu'})$ for a certain $\mu'<\mu$. 
	%		 		Then, there exists a unique function $u$ such that
	%		 		\begin{equation} 
	%		 		\begin{cases}
	%		 		u\in L^2([0,T];V)\mbox{ and }\frac{\partial u}{\partial t } \in L^2([0,T], H), \\u(T)=\psi(T),\\u\geq \psi \mbox{ a.e. in } [0,T]\times \R \times (0,\infty),\\
	%		 		\forall v\in V, \ v\geq \psi  \\
	%		 		\qquad-\left( \frac{\partial u}{\partial t },v -u  \right)_H + a(u,v-u)\geq 0, \qquad \mbox{ ae in } [0,T],\\0\leq u \leq M.
	%		 		\end{cases}
	%		 		\end{equation}\end{theorem}
	\begin{proof}[Proof of uniqueness in Theorem \ref{variationalinequality}]
		Suppose that there are two functions $u_1 $ and $u_2$ which satisfy \eqref{VI}. As usual, we take $v=u_2$ in the equation satisfied by $u_1$ and $v=u_1$ in the one satisfied by $u_2$ and we add the resulting equations. 
		%		 		We get
		%		 		$$
		%		 		-\left( \frac{\partial u_1}{\partial t },u_2 -u_1  \right)_H + a(u_1,u_2-u_1)\geq 0
		%		 		$$
		%		 		and
		%		 		$$
		%		 		-\left( \frac{\partial u_2}{\partial t },u_1 -u_2  \right)_H + a(u_2,u_1-u_2)\geq 0.
		%		 		$$
		Setting $w:= u_2-u_1$,		  		we get that,  a.e. in $[0,T]$,
		$$
		\left( \frac{\partial w}{\partial t }(t),w(t) \right)_H - a(w(t),w(t))\geq 0. 
		$$
		From the energy estimate \eqref{ub}, we know that
		$$
		a(u(t),u(t)) \geq  C_1 \Vert u(t) \Vert^2_V -C_2\Vert (1 + y)^{\frac 1 2 }u(t)\Vert^2_H,
		$$
		so that
		$$
		\frac 1 2 \frac{d}{d t } \Vert w(t)\Vert_H^2+ C_2\Vert (1 + y)^{\frac 1 2 }w(t)\Vert^2_H \geq 0.
		$$
		By integrating from $t$ to $T$, since $w(T)=0$, we have
		\begin{align*}
		\Vert &w(t)\Vert_H^2\leq C_2\int_t^T \Vert (1 + y)^{\frac 1 2 }w(s)\Vert^2_Hds\\
		& \leq C_2 \bigg( \int_t^T ds \int_{\mathcal{O}} \mathbbm{1}_{\{y\leq \lambda\}}  (1 + y)w^2(s)d\m+ \int_t^T ds \int_{\mathcal{O}} \mathbbm{1}_{\{y> \lambda\}}(1 + y)w^2(s)d\m \bigg)\\
		&\leq C \bigg( \int_t^T ds \int_{\mathcal{O}}(1 + \lambda)w^2(s)y^{\beta-1} e^{-\gamma |x| } e^{-\mu y}dxdy \bigg)\\&\qquad+C \bigg( + \int_t^T ds \int_{\mathcal{O}} \mathbbm{1}_{\{y> \lambda\}}(1 + y)w^2(s)y^{\beta-1} e^{-\gamma |x| } e^{-(\mu-\mu') y }e^{-\mu'y}dxdy \bigg)\\
		&\leq C \bigg( \int_t^T ds \int_{\mathcal{O}}dxdy (1 + \lambda)w^2(s)y^{\beta-1} e^{-\gamma |x| } e^{-\mu y}\bigg)\\&\qquad+C \bigg( e^{-(\mu-\mu') \lambda} \int_t^T ds \int_{\mathcal{O}}dxdy(1 + y)\Phi^2(s)y^{\beta-1} e^{-\gamma |x| } e^{-\mu'y}\bigg),
		\end{align*}
		where $\mu'<\mu$ and $\lambda>0$.
		%and we have used the fact that $w=u_2-u_1$ is bounded. 
		Since $ C_2= \int_{\mathcal{O}}dxdy (1 + y)\Phi^2(s)y^{\beta-1} e^{-\gamma |x| } e^{-\mu' y }<\infty$, we have
		\begin{align*}
		\Vert w(t)\Vert_H^2&\leq C  (1+\lambda )\int_t^T  \Vert w(s)\Vert_H^2ds +C_2(T-t)e^{-(\mu-\mu') \lambda},
		\end{align*}
		so, by using the Gronwall Lemma, 
		$$
		\Vert w(t)\Vert_H^2\leq C_2Te^{-(\mu-\mu') \lambda+C(T-t)(1+\lambda)}.
		$$
		Sending  $\lambda\rightarrow \infty $, we deduce that $w(t)=0$ in $[T,t]$ for $t$ such that $T-t< \frac{\mu-\mu'}{C}$. Then, we iterate the same argument: we integrate between $t'$ and $t$ with $t-t'<\frac{\mu-\mu'}{C}$ and we have $w(t)=0$ in $[T,t']$  and so on. We deduce that $w(t)=0$ for all $t\in [0,T]$ so the assertion follows.
	\end{proof}

	\begin{proof}[Proof of existence in Theorem \ref{variationalinequality}]
		Given $u_0=\Phi$, we can construct a sequence $(u_n)_n\subset V$ such that
		
		\begin{equation}\label{claim1}
		u_n\geq \psi \mbox{ a.e. in } [0,T] \times \mathcal{O}, \qquad n\geq 1,
		\end{equation}
		\begin{equation}\label{claim2}
		\begin{split}
		-\left( \frac{\partial u_n}{\partial t },v -u_n  \right)_H +a(u_n,v-u_n) + \lambda ((1+y)u_n,v-u_n)_H \geq  \lambda ((1+y)u_{n-1},v-u_n)_H,  \\ v\in V, \quad v\geq \psi , \quad
		\mbox{ a.e. on } [0,T]\times \mathcal{O}, \qquad n \geq 1,
		\end{split}
		\end{equation}
		\begin{equation}
		u_n(T)=\psi(T), \qquad \mbox{ in } \mathcal{O},
		\end{equation}
		\begin{equation}\label{claim3}
		\Phi\geq u_1\geq u_2\geq \dots \geq u_{n-1}\geq u_n \geq \dots \geq 0 ,\qquad  \mbox{ a.e. on } [0,T]\times \mathcal{O}.
		\end{equation}
		In fact, if   we have $0\leq u_{n-1} \leq \Phi$ for all $ n\in \N$, then the assumptions of Proposition \ref{coercive variational inequality} are satisfied with 
		$$
		g_n= \lambda(1+y)u_{n-1}.
		$$
		Indeed, since $(1+y)^{\frac 3 2 }\Phi\in L^2([0,T];H)$, we have that  $g_n$ and $\sqrt{1+y}g_n$ belong to $ L^2([0,T];H) $  and, moreover, $0\leq g_n \leq \lambda(1+y)\Phi \leq - \frac{\partial \Phi}{\partial t} -\L_\lambda \Phi$. Therefore, step by step, we can deduce the existence and the uniqueness of a solution $u_n$ to \eqref{claim2} such that $0\leq u_n \leq \Phi$.  \eqref{claim3} is a simple consequence of Proposition \ref{CompPrinc2}. In fact, proceeding by induction, at each step we have 
		$$
		g_n=	\lambda(1+y)u_{n-1} \leq 	\lambda(1+y)u_{n-2}=g_{n-1}
		$$
		so that $u_{n} \leq 	u_{n-1}$.
		Now, recall that 
		\begin{equation*}
		\Vert u_{n} \Vert_{L^\infty([0,T],V)}\leq K,
		\end{equation*}
		\begin{equation*}
		\left\Vert \frac{\partial u_{n}}{\partial t }\right \Vert_{L^2([0,T];H)}\leq K,
		\end{equation*}
		where $K=C \left(    \Vert \Psi \Vert_{L^2([0,T];V)} + \Vert \sqrt{1+y}g_n\Vert_{L^2([0,T];H)}  + \Vert\sqrt{1+y}\psi \Vert_{L^2([0,T];V)}+\Vert\psi(T)\Vert_V \right)$.
		Note that the constant $K$ is independent of $n$ since $|g_n|=|\lambda(1+y)u_{n-1},|\leq \lambda (1+y)\Phi,$ for every $n \in \N.$
		Therefore, by passing to a subsequence, we can assume that there exists a function $u$ such that $u\in L^2([0,T];V) $, $ \frac{\partial u }{\partial t} \in L^2([0,T];H) $ and for every $t\in [0,T]$, $u'_n(t) \rightharpoonup u'(t) $ in $H$ and $u_n (t) \rightharpoonup u(t) $ in $V$. Indeed, again thanks to the fact that $0\leq u_n \leq \Phi$, we can deduce that $u_n(t)\rightarrow u(t)$ in $H$.  Therefore we can pass to the limit in 
		\begin{align*}
		-\left( \frac{\partial u_n}{\partial t },u_n-v \right)_H +a(u_n,v-u_n) + \lambda ((1+y)u_n,v-u_n)_H  \geq  \lambda((1+y)u_{n-1},v-u_n)_H
		\end{align*}
		and the assertion follows.
	\end{proof}
	\begin{remark}
		Keeping in mind our purpose of identifying  the solution of the variational inequality \eqref{VI} with the American option price we have considered the case without source term ($g=0$) in the variational inequality \eqref{VI}. However, under the same assumptions of Theorem \ref{variationalinequality}, we can prove in the same way the existence and the uniqueness of a solution of 
		\begin{equation*}
		\begin{cases}
		-\left( \frac{\partial u}{\partial t },v -u  \right)_H + a(u,v-u)\geq (g,v-u)_H, \quad \mbox{a.e. in } [0,T] \quad v\in  L^2([0,T];V), \ v\geq \psi,\\
		u\geq \psi \mbox{ a.e. in } [0,T]\times \R \times (0,\infty),\\
		u(T)=\psi(T),\\
		0\leq u \leq \Phi,
		\end{cases}
		\end{equation*}
		where $g$ satisfies Assumption $\mathcal{H}^0$  and $0\leq g\leq -\frac{\partial \Phi}{\partial t }-\L\Phi$.
	\end{remark}
	We conclude stating the following Comparison Principle, whose proof  is a direct consequence of  Proposition \ref{CompPrinc2} and the proof of Proposition \ref{variationalinequality}.
	\begin{proposition}\label{CompPrinc3}
		For $i=1,2$, assume that $\psi_i$ satisfies Assumption $\mathcal{H}^1$ and $0\leq \psi_i\leq \Phi$ with  $\Phi$ satisfying Assumption $\mathcal{H}^2$. Let $u^i_{\lambda}$  be the unique solution of \eqref{CVI} with obstacle function $\psi_i$. Then:
		\begin{enumerate}
			\item   If $\psi_1\leq \psi_2$, then $u^1_{\lambda}\leq u^2_{\lambda}$.
			\item  If $\psi_1- \psi_2\in L^\infty$, then $u^1_{\lambda}- u^2_{\lambda}\in L^\infty $ and $\Vert  u^1_{\lambda}- u^2_{\lambda}\Vert_\infty \leq   \Vert \psi_1- \psi_2\Vert_\infty$.
		\end{enumerate}
	\end{proposition}
	
	\section{Connection with the optimal stopping problem}\label{sect-identification}
	Once we have the existence and the uniqueness of a solution $u$ of the variational inequality \eqref{variationalinequality}, our aim is to prove that  it matches the solution of the optimal stopping problem, that is
	\begin{equation*}
	u(t,x,y)=u^*(t,x,y), \qquad \mbox{ on } [0,T] \times \bar{\mathcal{O}},
	\end{equation*}
	where $u^*$ is defined by
	$$
	u^*(t,x,y)= \sup_{\tau \in \mathcal{T}_{t,T}}\E \left[  \psi(\tau,X_\tau^{t,x,y}, Y_\tau^{t,x,y})     \right],
	$$ 
	$	\mathcal{T}_{t,T}$ being the set of the stopping times with values in $[t,T]$.
	Since the function $u$ is not regular enough to apply It\^{o}'s Lemma, we use another strategy in order to prove the above identification.
	So, we first show, by using the affine character of the underlying diffusion, that the semigroup associated with the bilinear form $a_\lambda$ coincides with the transition semigroup of the two dimensional diffusion $(X,Y)$ with a killing term. Then, we prove suitable estimates on the joint law of $(X,Y)$ and $L^p$-regularity results on the solution of the variational inequality and we deduce from them the probabilistic interpretation.

	\subsection{Semigroup associated with the bilinear form }

	We introduce now the semigroup associated with the coercive  bilinear form $a_\lambda$. With a natural notation, we define the following spaces
	$$
	L^2_{loc}(\R^+;H)=\left\{f:\R^+\rightarrow H : \forall t\geq 0 \int_0^t \|f(s)\|_H^2ds<\infty  \right \},
	$$
	$$
	L^2_{loc}(\R^+;V)=\left \{f:\R^+\rightarrow V :  \forall t\geq 0 \int_0^t \|f(s)\|_V^2ds<\infty \right  \}.
	$$
	First of all, we  state the following result:
	\begin{proposition}\label{propsg2}
		For every $\psi\in V$,
		$f\in L^2_{loc}(\R^+;H)$ with $\sqrt{y}f\in L^2_{loc}(\R^+;H)$,   there exists a unique function $u\in L^2_{loc}(\R^+;V)$ such that
		$\frac{\partial u}{\partial t}\in L^2_{loc}(\R^+;H)$, $u(0)=\psi$ and
		\begin{equation}\label{var_eq}
		\left(\frac{\partial u}{\partial t},v\right)_H+a_\lambda(u,v)=(f,v)_H, \quad v\in V.
		\end{equation}
		Moreover we have, for every $t\geq 0$,
		\begin{equation}
		\Vert u(t)\Vert _H^2+\frac{\delta_1}{2}\int_0^t\Vert u(s)\Vert ^2_Vds\leq \Vert \psi \Vert ^2_H
		+\frac{2}{\delta_1} \int_0^t \Vert f(s)\Vert _H^2ds \label{estimL2bis}
		\end{equation}
		and
		$$
		||u(t)||^2_V+	\int_0^t||u_t(s)||^2_Hds \leq C\left(||\psi||_V^2+\frac{1}{2}\int_0^t||\sqrt{1+y}f(s)||^2_H
		ds \right),
		$$
		with $C>0$.
	\end{proposition}
	The proof can be found in the appendix of this chapter. Moreover, we can prove a Comparison Principle for the equation \eqref{var_eq} as we have done for the variational inequality.  
	
	We denote $u(t)=\bar{P}^\lambda_t\psi$ the solution of \eqref{var_eq} corresponding to $u(0)=\psi$ and $f=0$. From \eqref{estimL2bis} we deduce that the operator $\bar{P}^\lambda_t$ is a linear contraction on $H$ and, from uniqueness, we have the semigroup property. 
	\begin{proposition}\label{propsg3}
		Let us consider $f:\R^+\to H$ such that
		$\sqrt{1+y}f\in L^2_{loc}(\R^+,H)$.  Then, the solution of
		\[
		\begin{cases}
		\left(\frac{\partial u}{\partial t},v\right)_H+a_\lambda(u,v)=(f,v)_H,\quad v\in V,\\
		u(0)=0,
		\end{cases}
		\]
		is given by $ u(t)=\int_0^t  \bar P^\lambda_sf(t-s)ds=\int_0^t\bar P^\lambda_{t-s}f(s)ds$.
	\end{proposition}
	\begin{proof}
		Note that $V$ is dense in $H$ and recall the estimate \eqref{estimL2bis}, so  it is enough to prove the assertion for  $f=\ind{(t_1,t_2]}\psi$, with $0\leq t_1<t_2$ and $\psi\in V$. If we set $u(t)=\int_0^t\bar P^\lambda_{t-s}f(s)ds$, we have
		\begin{align*}
		u(t)&=\ind{\{t\geq t_1\}}\int_{t_1}^{t\wedge t_2}\bar P^\lambda_{t-s}\psi ds	=\begin{cases}
		\int_{t_1}^{t_2}\bar P^\lambda_{t-s}\psi ds=\int_{t-t_2}^{t-t_1}\bar P^\lambda_{s}\psi ds\quad&\mbox{ if } t\geq t_2\\
		\displaystyle \int_{t_1}^{t}\bar P^\lambda_{t-s}\psi ds=\int_{0}^{t-t_1}\bar P^\lambda_{s}\psi ds \quad &\mbox{ if } t\in[t_1, t_2)
		\end{cases}.
		\end{align*}
		Therefore, for every $v\in V$, we have $(u_t,v)_H+a_\lambda(u,v)=0$ if $t\leq t_1$ and, if $t\geq t_1$,
		$$
		\left(\frac{\partial u}{\partial t},v\right)_H+a_\lambda(u(t),v)=
		\begin{cases}
		\left(\bar P^\lambda_{t-t_1}\psi-\bar P^\lambda_{t-t_2}\psi, v\right)_H+
		a_\lambda\left(\int_{t-t_2}^{t-t_1}\bar P^\lambda_{s}\psi ds, v\right)\quad &\mbox{ if } t\geq t_2\\
		\left(\bar P^\lambda_{t-t_1}\psi, v\right)_H+
		a_\lambda\left(\int_{0}^{t-t_1}\bar P^\lambda_{s}\psi ds, v\right)  \quad&\mbox{ if } t\in[t_1, t_2)
		\end{cases}.
		$$
		The assertion follows from $(\bar P^\lambda_t\psi,v)_H+\int_0^t a_\lambda(\bar P_s\psi,v)ds  =(\psi,v)_H$.
	\end{proof} 
	\begin{remark}
		It is not difficult to prove that  $\bar{P}^\lambda_t:L^p(\mathcal{O},\m) \rightarrow L^p(\mathcal{O},\m) $  is a contraction for every $ p\geq 2$, and it is an analytic semigroup. This is not useful to our purposes so we omit the proof.
	\end{remark}
	\subsection{Transition semigroup}
	We define $\E_{x_0,y_0}(\quad)= \E(\quad| X_0=x_0, Y_0=y_0)$. 
	Fix $\lambda >0$.  For every measurable positive function $f$ defined on $\R\times [0,+\infty)$, we define
	\[
	P^\lambda_tf(x_0,y_0)=\E_{x_0,y_0}\left(e^{-\lambda\int_0^t(1+Y_s)ds} f(X_t,Y_t)\right).
	\]
	The operator	$	P^\lambda_t$ is the transition semigroup of the two dimensional diffusion $(X,Y)$ with the killing term $e^{-\lambda\int_0^t(1+Y_s)ds}$.
	
	%	Recall that the CIR process satisfies the SDE  
	%	$$
	%	dY_t=\kappa(\theta-Y_t)dt+\sigma\sqrt{Y_t}dW_t,\qquad Y_0=y_0\geq0,
	%	$$
	%	with $\sigma>0,$ $k \geq 0 $, $\theta\geq0$. 
	Set $\E_{y_0}(\quad)= \E(\quad| Y_0=y_0)$. 
	We first prove some useful results about the Laplace transform of the pair $(Y_t,\int_0^tY_sds)$. These results rely on the affine structure of the model and have already appeared in slightly different forms in the literature (see, for example, \cite[Section 4.2.1]{Abook}). We include a proof for convenience.
	\begin{proposition}\label{Laplace}
		Let $z$  and  $w$ be two complex numbers with nonpositive real parts. The equation
		\begin{equation}\label{*}
		\psi'(t)=\frac{\sigma^2}{2}\psi^2(t)-\kappa \psi(t)+w
		\end{equation}
		has a unique solution $ \psi_{z,w}$ defined on $[0,+\infty)$, such that $\psi_{z,w}(0)=z$. Moreover,
		for every $t\geq 0$,
		\[
		\E_{y_0}\left( e^{z Y_t+w \int_0^t Y_s ds}\right)=e^{y_0\psi_{z,w}(t) +\theta\kappa\phi_{z,w}(t)},
		\]
		with $\phi_{z,w}(t)=\int_0^t \psi_{z,w}(s)ds$.
	\end{proposition}
	\begin{proof}
		Let $\psi$ be the solution of \eqref{*}. We define $\psi_1$ (resp. $w_1$) and $\psi_2$
		(resp. $w_2$) the real and the imaginary part of
		$\psi$ (resp. $w$). We have
		\[
		\left\{
		\begin{array}{l}
		\psi'_1(t)=\frac{\sigma^2}{2}\left(\psi_1^2(t)-\psi_2^2(t)\right)-\kappa \psi_1(t)+w_1,\\
		\psi'_2(t)=\sigma^2\psi_1(t)\psi_2(t)-\kappa \psi_2(t)+w_2.
		\end{array}
		\right.
		\]
		From the first equation we deduce that $ \psi'_1(t)\leq \frac{\sigma^2}{2}\left(\psi_1(t)-\frac{2\kappa}{\sigma^2}\right) \psi_1(t)+w_1$
		and, since $w_1\leq 0$, the function $t\mapsto \psi_1(t)e^{-\frac{\sigma^2}2\int_0^t(\psi_1(s)-\frac{2\kappa}{\sigma^2})ds}$ is nonincreasing. Therefore $\psi_1(t)\leq 0$
		if $\psi_1(0)\leq 0$. Multiplying the first equation by $\psi_1(t)$ and the second one by $\psi_2(t)$ and adding we get
		\begin{eqnarray*}
			\frac 1 2 	\frac{d}{dt}\left(|\psi(t)|^2\right)&=&\left(\frac{\sigma^2}{2}\psi_1(t)-\kappa\right)|\psi(t)|^2+w_1\psi_1(t)+w_2\psi_2(t)\\
			&\leq &\left(\frac{\sigma^2}{2}\psi_1(t)-\kappa\right)|\psi(t)|^2+|w||\psi(t)|\\
			&\leq &\left(\frac{\sigma^2}{2}\psi_1(t)-\kappa\right)|\psi(t)|^2+\epsilon |\psi(t)|^2+\frac{|w|^2}{4\epsilon}.
		\end{eqnarray*}
		We deduce that $|\psi(t)|$ cannot explode in finite time and, therefore,  $\psi_{z,w}$
		actually exists on $[0, +\infty)$. 
		
		Now, let us define the function  $F_{z,w}(t,y)=e^{y\psi_{z,w}(t) +\theta\kappa\phi_{z,w}(t)}$.
		$F_{z,w}$ is  $C^{1,2}$ on $[0,+\infty)\times \R$ and it satisfies by construction the following equation
		\[
		\frac{\partial F_{z,w}}{\partial t}=\frac{\sigma^2}{2}y\frac{\partial^2 F_{z,w}}{\partial y^2}+
		\kappa(\theta-y)\frac{\partial F_{z,w}}{\partial y}+w y F_{z,w}.
		\]
		Therefore, for every $T>0$, the process $(M_t)_{0\leq t\leq T}$ defined by \begin{equation}\label{Mt}
		M_t=e^{w\int_0^t Y_sds}F_{z,w}(T-t, Y_t)
		\end{equation} 
		is a local martingale. On the other hand, note that 
		$$
		|M_t|= \left|  e^{w\int_0^t Y_sds} \right|\left|  e^{Y_t\psi_{z,w}(T-t) +\theta\kappa\phi_{z,w}(T-t)} \right|\leq 1
		$$ 
		since $w$, $\psi_{z,w}(t)$ and $\phi_{z,w}(t)=\int_0^t\psi_{z,w}(s)ds$ all have nonpositive real parts. Therefore the process $(M_t)_t$ is a  true martingale indeed. We deduce that 
		$F_{z,w}(T,y_0)=\E_{y_0}\left(e^{w\int_0^T Y_sds}e^{zY_T}\right)$ and the assertion follows.
	\end{proof}
	We also have the following result which specifies the behaviour of the Laplace transform  of $(Y_t,\int_0^tY_sds)$ when evaluated in two real numbers, not necessarily nonpositive.
	\begin{proposition}\label{Laplace2}
		Let $\lambda_1$  and $\lambda_2$ be two real numbers such that 
		\[
		\frac{\sigma^2}{2}\lambda_1^2-\kappa\lambda_1+\lambda_2\leq 0.
		\]
		Then, the equation
		\begin{equation}\label{**}
		\psi'(t)=\frac{\sigma^2}{2}\psi^2(t)-\kappa \psi(t)+\lambda_2
		\end{equation}
		has a  unique solution $ \psi_{\lambda_1,\lambda_2}$ defined on $[0,+\infty)$ such that $\psi_{\lambda_1,\lambda_2}(0)=\lambda_1$. Moreover, for every  $t\geq 0$, we have
		\[
		\E_{y_0}\left( e^{\lambda_1 Y_t+\lambda_2 \int_0^t Y_s ds}\right)\leq 
		e^{y_0\psi_{\lambda_1,\lambda_2}(t) +\theta\kappa\phi_{\lambda_1,\lambda_2}(t)},
		\]
		with $\phi_{\lambda_1,\lambda_2}(t)=\int_0^t \psi_{\lambda_1,\lambda_2}(s)ds$.
	\end{proposition} 
	\begin{proof}
		Let $\psi$ be the solution of \eqref{**} with $\psi(0)=\lambda_1$. We have
		\[
		\psi''(t)=(\sigma^2 \psi(t)-\kappa)\psi'(t).
		\]
		Therefore, the function $t\mapsto \psi'(t)e^{-\int_0^t(\sigma^2\psi(s)-\kappa)ds}$is a constant, hence $\psi'(t)$ has constant sign.  Moreover, the assumption on $\lambda_1$
		and $\lambda_2$ ensures that  $\psi'(0)\leq 0$. We deduce that $\psi'(t)\leq 0$
		and $\psi(t)$ remains between the solutions of the equation
		\[
		\frac{\sigma^2}{2}\lambda^2-\kappa\lambda+\lambda_2=0.
		\]
		This proves that the solution is defined on the whole interval
		$[0,+\infty)$. Now the assertion follows as in the proof of Proposition \ref{Laplace}: just note that  the process $(M_t)_t$ defined as in \eqref{Mt} is no more uniformly bounded, so we cannot directly deduce that it is a martingale. However,  it remains a positive local martingale, hence a supermartingale.
	\end{proof}
	
	\begin{remark}\label{remarksemigroup}
		Let us now consider	two real numbers $\lambda_1$  and $\lambda_2$  such that 
		\[
		\frac{\sigma^2}{2}\lambda_1^2-\kappa\lambda_1+\lambda_2<0.
		\]	 From the proof of Proposition \ref{Laplace2}, by using the optional sampling theorem we have
		\begin{eqnarray*}
			\sup_{\tau\in\T_{0,T}}\E_y\left(e^{\lambda_2\int_0^\tau Y_sds}e^{\psi_{\lambda_1,\lambda_2}(T-\tau)Y_\tau+\theta\kappa\phi_{\lambda_1,\lambda_2}(T-\tau)}\right)
			&\leq&e^{y\psi_{\lambda_1,\lambda_2}(T)+\theta\kappa\phi_{\lambda_1,\lambda_2}(T)}.
		\end{eqnarray*}
		
		Consider now $\epsilon>0$ and let $\lambda_1^\epsilon=(1+\epsilon)\lambda_1$ and $\lambda_2^\epsilon = (1+\epsilon)\lambda_2$.
		For $\epsilon$ small enough, we have $ \frac{\sigma^2}{2}(\lambda_1^\epsilon)^2-\kappa\lambda_1^\epsilon+\lambda_2^\epsilon< 0$.
		Therefore
		\begin{eqnarray*}
			\sup_{\tau\in\T_{0,T}}\E_y\left(e^{\lambda_2^\epsilon\int_0^\tau Y_sds}e^{\psi_{\lambda_1^\epsilon,\lambda_2^\epsilon}(T-\tau)Y_\tau+
				\theta\kappa\phi_{\lambda_1^\epsilon,\lambda_2^\epsilon}(T-\tau)}\right)
			&\leq&e^{y\psi_{\lambda_1^\epsilon,\lambda_2^\epsilon}(T)+\theta\kappa\phi_{\lambda_1^\epsilon,\lambda_2^\epsilon}(T)}.
		\end{eqnarray*}
		If we have $\psi_{\lambda_1^\epsilon,\lambda_2^\epsilon}\geq (1+\epsilon)\psi_{\lambda_1,\lambda_2}$,
		we can deduce that
		\begin{eqnarray*}
			\sup_{\tau\in\T_{0,T}}\E_y\left(e^{\lambda_2(1+\epsilon)\int_0^\tau Y_sds}
			e^{(1+\epsilon)\left(\psi_{\lambda_1,\lambda_2}(T-\tau)Y_\tau+\theta\kappa\phi_{\lambda_1,\lambda_2}(T-\tau)\right)}\right)
			&\leq&e^{y\psi_{\lambda_1^\epsilon,\lambda_2^\epsilon}(T)+\theta\kappa\phi_{\lambda_1^\epsilon,\lambda_2^\epsilon}(T)},
		\end{eqnarray*}
		and, therefore, that the family  $\left(e^{\lambda_2\int_0^\tau Y_sds}e^{\psi_{\lambda_1,\lambda_2}(T-\tau)Y_\tau+\theta\kappa\phi_{\lambda_1,\lambda_2}(T-\tau)}\right)_{\tau\in\T_{0,T}}$
		is uniformly integrable. As a consequence, the process $(M_t)_t$ is a true  martingale and we have
		\[
		\E_{y}\left( e^{\lambda_1 Y_t+\lambda_2 \int_0^t Y_s ds}\right) = 
		e^{y\psi_{\lambda_1,\lambda_2}(t) +\theta\kappa\phi_{\lambda_1,\lambda_2}(t)}.
		\]
		So, it remains to show  that $\psi_{\lambda_1^\epsilon,\lambda_2^\epsilon}\geq (1+\epsilon)\psi_{\lambda_1,\lambda_2}$. In order to do this
		we set $g_\epsilon(t)=\psi_{\lambda_1^\epsilon,\lambda_2^\epsilon}(t)-(1+\epsilon)\psi_{\lambda_1,\lambda_2}(t)$.
		From the equations satisfied by $\psi_{\lambda_1^\epsilon,\lambda_2^\epsilon}$ and
		$\psi_{\lambda_1,\lambda_2}$ we deduce that
		\begin{eqnarray*}
			g_\epsilon'(t)&=&\frac{\sigma^2}{2}\left(\psi^2_{\lambda_1^\epsilon,\lambda_2^\epsilon}(t)-(1+\epsilon)
			\psi^2_{\lambda_1,\lambda_2}(t)\right)-\kappa \left(\psi_{\lambda_1^\epsilon,\lambda_2^\epsilon}(t)-
			(1+\epsilon) \psi_{\lambda_1,\lambda_2}(t)\right)\\
			&=&
			\frac{\sigma^2}{2}\left(\psi^2_{\lambda_1^\epsilon,\lambda_2^\epsilon}(t)-(1+\epsilon)^2
			\psi^2_{\lambda_1,\lambda_2}(t)\right)
			-\kappa g_\epsilon(t)+\frac{\sigma^2}{2}\left((1+\epsilon)^2-(1+\epsilon)\right)\psi^2_{\lambda_1,\lambda_2}(t)\\
			&=&\frac{\sigma^2}{2}\left(\psi_{\lambda_1^\epsilon,\lambda_2^\epsilon}(t)+(1+\epsilon)
			\psi_{\lambda_1,\lambda_2}(t)\right)g_\epsilon(t)
			-\kappa g_\epsilon(t)+\frac{\sigma^2}{2}\epsilon(1+\epsilon)\psi^2_{\lambda_1,\lambda_2}(t)\\
			&=&f_\epsilon(t) g_\epsilon(t)+\frac{\sigma^2}{2}\epsilon(1+\epsilon)\psi^2_{\lambda_1,\lambda_2}(t),
		\end{eqnarray*}
		where
		\[
		f_\epsilon(t)=\frac{\sigma^2}{2}\left(\psi_{\lambda_1^\epsilon,\lambda_2^\epsilon}(t)+(1+\epsilon)
		\psi_{\lambda_1,\lambda_2}(t)\right)
		-\kappa.
		\]
		Therefore, the function $g_\epsilon(t)e^{-\int_0^t f_\epsilon(s)ds}$ is nondecreasing  and, since $g_\epsilon(0)=0$,
		we have $g_\epsilon(t)\geq 0$.
	\end{remark}    
	
	We can now prove the following Lemma, which will be useful in  Section \ref{sect-estimjointlaw} to  prove suitable estimates on the joint law of the process $(X,Y)$.
	\begin{lemma}\label{RemarqueMomentsNeg}
		For every $q>0$  there exists $C>0$ such that for all $y_0\geq 0$,
		\begin{equation}
		\E_{y_0}\left(\int_0^t Y_v dv\right)^{-q}\leq \frac{C}{t^{2q}}.
		\end{equation}
	\end{lemma}
	\begin{proof}
		If we take $\lambda_1=0$ and $\lambda_2=-s$ with $s>0$  in Proposition \ref{Laplace2},
		we get
		\[
		\E_{y_0}\left( e^{-s \int_0^t Y_v dv}\right)=e^{y_0\psi_{0,-s}(t) +\theta\kappa\phi_{0,-s}(t)}.
		\]
		Since $\psi'_{0,-s}(0)=-s<0$, we can deduce by the  proof of Proposition \ref{Laplace2} that $\psi'_{0,-s}(t)=-se^{\int_0^t(\sigma^2\psi(u)-\kappa)du}$. Therefore, since $\psi_{0,-s}=0$, we have  
		\begin{equation}\label{psi0-s}
		\psi_{0,-s}(t)=-s\int_0^te^{\int_0^u(\sigma^2\psi(v)-\kappa)dv}du.
		\end{equation}
		Again from the proof of Proposition \ref{Laplace2},
		\[
		\psi_{0,-s}(t)\geq 
		\frac{\kappa}{\sigma^2}-\sqrt{\left(\frac{\kappa}{\sigma^2}\right)^2+2\frac{s}{\sigma^2}}
		\geq -\sqrt{2s/\sigma^2},
		\]
		so, by using \eqref{psi0-s}, we deduce that
		$$
		\psi_{0,-s}(t)\leq 
		-s\int_0^te^{\int_0^u-(\sigma \sqrt{2s}+\kappa)dv}du=-s\int_0^te^{-\lambda_su}du=-\frac{s}{\lambda_s}(1-e^{-t\lambda_s}).
		$$	
		where $\lambda_s=\sigma\sqrt{2s}+\kappa$. Since $\phi_{0,-s}(t)=\int_0^t \psi_{0,-s}(u)du$, we have
		$$
		\phi_{0,-s}(t)\leq-\frac{s}{\lambda^2_s} \left( t\lambda_s   -1+e^{-t\lambda_s}\right).
		$$
		Therefore, since $\psi_{0,-s}(t)\leq 0$,  for any $y_0\geq 0$ we get
		\begin{align*}
		\E_{y_0}\left(e^{-s\int_0^tY_vdv}\right)&\leq e^{\kappa\theta\phi_{0,-s}(t)}\leq e^{ -\frac{\kappa\theta s}{\lambda_s^2} (t\lambda_s-1+e^{-t\lambda_s})    } .
		\end{align*}
		Now, recall that for every $q>0$ we can write
		$$
		\frac{1}{y^q}=\frac{1}{\Gamma(q)}\int_0^\infty s^{q-1}e^{-sy}ds.
		$$
		Therefore
		\begin{align*}
		\E_{y_0}\left(\int_0^t Y_v dv\right)^{-q}&=\E_{y_0}\left(\frac{1}{\Gamma(q)}\int_0^\infty s^{q-1}
		e^{-s\int_0^t Y_v dv}ds\right)\\
		&\leq\frac{1}{\Gamma(q)}\int_0^1s^{q-1}
		e^{  -\frac{\kappa\theta s}{\lambda_s^2} (t\lambda_s-1+e^{-t\lambda_s})} ds+ \frac{1}{\Gamma(q)}\int_1^\infty s^{q-1}
		e^{  -\frac{\kappa\theta s}{\lambda_s^2} (t\lambda_s-1+e^{-t\lambda_s})} ds.
		\end{align*}
		Recall that $\lambda_s=\sigma\sqrt{2s}+\kappa$, so the first terms in the right hand side is finite. Moreover, for $s>1$, we have $\frac{\kappa\theta s}{\lambda_s^2}\leq C$.  Then, by noting that the function $u\mapsto tu-1+e^{-tu}$ is nondecreasing, we have
		\begin{align*}
		\E_{y_0}\left(\int_0^t Y_v dv\right)^{-q}&\leq C+ \frac{1}{\Gamma(q)}\int_1^\infty s^{q-1}
		e^{  -C (t\sigma\sqrt{2s}-1+e^{-t\sigma\sqrt{2s}})} ds\\
		&\leq C+ \frac{1}{t^{2q}\Gamma(q)}\int_0^\infty v^{q-1}
		e^{  -C (\sigma\sqrt{2v}-1+e^{-\sigma\sqrt{2v}})} dv\\
		&\leq \frac{C}{t^{2q}},
		\end{align*}
		which concludes  the proof.
	\end{proof}
	
	Now recall that the diffusion $(X,Y)$ evolves according to the following stochastic differential system
	
	\begin{equation*}
	\begin{cases}
	dX_t=\left( \frac{\rho\kappa\theta}{\sigma}-\frac{Y_t}{2}\right)dt + \sqrt{Y_t}dB_t,\\
	dY_t=\kappa(\theta-Y_t)dt+\sigma\sqrt{Y_t}dW_t.
	\end{cases}
	\end{equation*}	 	
	If we set $\tilde{X}_t=X_t-\frac{\rho}{\sigma}Y_t$, we have 
	\begin{equation}
	\label{model2}
	\begin{cases}
	d\tilde{X}_t=\left( \frac{\rho\kappa}{\sigma}-\frac{1}{2}\right)Y_tdt+\sqrt{1-\rho^2}\sqrt{Y_t} d\tilde{B}_t,\\
	dY_t=\kappa(\theta-Y_t)dt+\sigma\sqrt{Y_t}dW_t.
	\end{cases}
	\end{equation}
	where $\tilde{B}_t=(1-\rho^2)^{-1/2}\left(B_t-\rho W_t\right)$. Note that $\tilde B $ is a standard Brownian motion with $\langle \tilde{B},W\rangle_t=0$.
	%	We can now prove the following result.
	\begin{proposition}\label{prop-TF}
		For all  $u, \, v \in \R$, for all $\lambda\geq 0$
		and for all $(x_0,y_0)\in \R\times [0,+\infty)$ we have
		\[
		\E_{x_0,y_0}\left( e^{iu X_t+ivY_t}e^{-\lambda\int_0^t Y_s  ds}\right)=
		e^{iu x_0+y_0(\psi_{\lambda_1,\mu}(t)-iu\frac{\rho}{\sigma})+\theta\kappa\phi_{\lambda_1,\mu}(t)},
		\]
		where $\lambda_1= i(u\frac{\rho}{\sigma}+v)$, $\mu=iu\left(\frac{\rho\kappa}{\sigma}-\frac{1}{2}\right)-\frac{u^2}{2}(1-\rho^2)-\lambda$
		and the function $\psi_{\lambda_1,\mu}$
		and $\phi_{\lambda_1,\mu}$ are defined in Proposition \ref{Laplace}.
	\end{proposition}
	\begin{proof}
		We have
		\[
		\E_{x_0,y_0}\left( e^{iu X_t+ivY_t-\lambda\int_0^t Y_s  ds}\right)=\E_{x_0,y_0}\left( e^{iu (\tilde{X_t}+\frac{\rho}{\sigma}Y_t)+ivY_t-\lambda\int_0^t Y_s  ds}\right)
		\]
		and 
		\[
		\tilde{X_t}=x_0-\frac{\rho}{\sigma}y_0+\int_0^t \left( \frac{\rho\kappa}{\sigma}-\frac{1}{2}\right)Y_sds+\int_0^t\sqrt{(1-\rho^2)Y_s}d\tilde{B}_s.
		\]
		Since $\tilde{B}$ and $W$ are independent,
		\[
		\E\left(e^{iu\tilde{X}_t} \;|\; W\right)=e^{iu\left(x_0-\frac{\rho}{\sigma}y_0+\int_0^t \left(\frac{\rho\kappa}{\sigma}-\frac{1}{2}\right)Y_sds\right)
			-\frac{u^2}{2}(1-\rho^2)\int_0^t Y_s ds}
		\]
		and
		\begin{eqnarray*}
			\E_{x_0,y_0}\left( e^{iu X_t+ivY_t-\lambda\int_0^t Y_s  ds}\right)=e^{iu\left(x_0-\frac{\rho}{\sigma}y_0\right)}
			\E_{y_0}\left(e^{i\left(u\frac{\rho}{\sigma}+v\right)Y_t+
				\left(iu(\frac{\rho\kappa}{\sigma}-\frac{1}{2})-\frac{u^2}{2}(1-\rho^2)-\lambda\right)\int_0^tY_sds}\right).
		\end{eqnarray*}
		Then the assertion follows by using  Proposition~\ref{Laplace}.
	\end{proof}
	
	\subsection{Identification of the semigroups}
	We now show that the semigroup $\bar P^\lambda_t$ associated with the coercive bilinear form  can be actually identified with the transition semigroup $ P^\lambda_t$.  Recall  the Sobolev  spaces $L^p(\O,\m_{\gamma,\mu})$ introduced in Definition \ref{def-sob} for $p\geq 1$. In order to prove the identification of the semigroups, we need the following property of the transition semigroup.
	\begin{theorem}\label{theotransition}
		For all $p>1$, $\gamma>0$ and $\mu>0$ there exists $\lambda >0$ such that,
		for every compact $K\subseteq\R\times[0,+\infty)$ and for every  $T>0$, there is $C_{p,K, T}>0$ such that
		\[
		P^\lambda_t f(x_0,y_0)\leq \frac{C_{p,K, T}}{t^{\frac{\beta}{p}+\frac{3}{2p}}}||f||_{L^p(\O,\m_{\gamma,\mu})},\qquad (x_0,y_0)\in K.
		\]
		for every measurable positive function  $f$ on $\R\times[0,+\infty)$ and for every $t\in(0,T]$.
	\end{theorem}
	Theorem \ref{theotransition} will  also play a crucial role in order to prove Theorem \ref{theorem2}.   Its proof relies on suitable estimates on the joint law of the diffusion $(X,Y)$ and we postpone it to the following section. Then, we can prove the following result.
	\begin{proposition}\label{identificationsemigroup}
		There exists $\lambda>0$ such that, for every function $f\in H$ and for every $t\geq 0$,
		\[
		\bar{P}^\lambda_t f(x,y)=P^\lambda_tf(x,y),\quad dxdy \mbox{ a.e.}
		\]
	\end{proposition}
	\begin{proof}
		We can easily deduce from Theorem \ref{theotransition} with $p=2$ that, for $\lambda$ large enough, if $(f_n)_n$ is a sequence of functions which converges to $f$ in $H$, then the sequence  $(P^\lambda_tf_n)_n$ converges uniformly to $P^\lambda_tf$ on the compact sets. 	On the other hand, recall  that $\bar P^\lambda_t$ is a contraction  semigroup on $H$ so that  the function $f\mapsto \bar P^\lambda_tf$ is continuous and we have $\bar P^\lambda_t f_n \rightarrow \bar P^\lambda_t f$ in $H$.

		Therefore, by density arguments, it is enough to prove the equality for $
		f(x,y)=e^{iux+ivy}$
		with $u$, $v\in \R$. We  have, by using  Proposition \ref{prop-TF},
		\begin{eqnarray*}
			P^\lambda_tf(x,y)&=&\E_{x,y}\left(e^{-\lambda\int_0^t(1+Y_s)ds} e^{iuX_t+ivY_t}\right)\\
			&=& e^{-\lambda t}e^{iu x+y\left(\psi_{\lambda_1,\mu}(t)-iu\frac{\rho}{\sigma}\right)+\theta\kappa\phi_{\lambda_1,\mu}(t)},
		\end{eqnarray*}
		with $\lambda_1= i(u\frac{\rho}{\sigma}+v)$, $\mu=iu\left(\frac{\rho\kappa}{\sigma}-\frac{1}{2}\right)-\frac{u^2}{2}(1-\rho^2)-\lambda$.
		The function $F(t,x,y)$ defined by
		$F(t,x,y)=e^{-\lambda t}e^{iu x+y\left(\psi_{\lambda_1,\mu}(t)-iu\frac{\rho}{\sigma}\right)+\theta\kappa\phi_{\lambda_1,\mu}(t)}$
		satisfies $F(0,x,y)=e^{iux+ivy}$ and
		\[
		\frac{\partial F}{\partial t}=\left(\mathcal{L}-\lambda (1+y)\right)F.
		\]
		Moreover, since the real parts of $\lambda_1$ and $\mu$ are nonnegative, we can deduce from the proof of Proposition \ref{Laplace} that the real part of the function $t\rightarrow \psi(t)$ is nonnegative. Then, it is straightforward to see  that, for every $t\geq 0$, we have $F(t,\cdot,\cdot)\in H^2(\O,\m)$ and $t\mapsto F(t,\cdot,\cdot)$ is continuous,
		so that,
		for every $v\in V$,  $(\mathcal{L} F(t,.,.), v)_H=-a(F(t,.,.),v)$. Therefore
		\[
		\left(\frac{\partial F}{\partial t}, v\right)_H +a_\lambda(F(t,.,.),v)=0 \quad v\in V,
		\]
		and $F(t,.,.)=\bar{P}^\lambda_tf$.
	\end{proof}
	
	\subsection{Estimates on the joint law}\label{sect-estimjointlaw}
	In this section we prove Theorem \ref{theotransition}. We first recall some results about the density of the process $Y$.
	
	With the notations
	\[
	\nu=\beta-1=\frac{2\kappa\theta}{\sigma^2}-1,\quad y_t=y_0e^{-\kappa t}, \quad L_t=\frac{\sigma^2}{4\kappa}\left(1-e^{-\kappa t}\right),
	\]
	it is well known (see, for example, \cite[Section 6.2.2]{LL}) that the transition density of the process $Y$ is given by
	\[
	p_t(y_0,y)=\frac{e^{-\frac{y_t}{2L_t}}}{2y_t^{\nu/2}L_t}e^{-\frac{y}{2L_t}}y^{\nu/2}I_\nu\left(\frac{\sqrt{yy_t}}{L_t}\right),
	\]
	where $I_\nu$ is the first-order modified Bessel function with index $\nu$, defined by
	\[
	I_\nu(y)=\left(\frac{y}{2}\right)^\nu\sum_{n=0}^\infty \frac{(y/2)^{2n}}{n!\Gamma(n+\nu+1)}.
	\]
	It is clear that near $y=0$ we have $I_\nu(y)\sim \frac{1}{\Gamma(\nu+1)}\left(\frac{y}{2}\right)^\nu$ while, for $y\rightarrow \infty$, we have the asymptotic behaviour  $I_\nu(y)\sim e^y/\sqrt{2\pi y}$ (see \cite[page 377]{AS}).
	\begin{proposition}\label{density_CIR}
		There exists a constant $C_\beta>0$ (which depends only on $\beta$) such that, for every $t>0$,
		\[
		p_t(y_0,y)\leq
		\frac{C_\beta}{L_t^{\beta+\frac{1}{2}}}
		e^{-\frac{(\sqrt{y}-\sqrt{y_t})^2}{2L_t}}
		y^{\beta -1}
		\left(L_t^{1/2}
		+
		(yy_t)^{1/4}\right), \qquad (y_0,y)\in [0,+\infty)\times]0,+\infty).
		\]
	\end{proposition}
	\begin{proof}
		From the asymptotic behaviour of  $I_\nu$ near $0$ and $\infty$ we deduce the existence of a constant  $C_\nu>0$ such that
		\[
		I_\nu(x)\leq C_\nu\left(x^\nu\ind{\{x\leq 1\}}+ \frac{e^x}{\sqrt{x}}\ind{\{x>1\}}\right).
		\]
		Therefore 
		\begin{eqnarray*}
			p_t(y_0,y)&=&\frac{e^{-\frac{y_t+y}{2L_t}}}{2y_t^{\nu/2}L_t}y^{\nu/2}I_\nu\left(\frac{\sqrt{yy_t}}{L_t}\right)\\
			&\leq &\frac{e^{-\frac{y_t+y}{2L_t}}}{2y_t^{\nu/2}L_t}y^{\nu/2}
			C_\nu\left(\frac{(yy_t)^{\nu/2}}{L_t^\nu}
			\ind{\{yy_t\leq L_t^2\}}+ \frac{e^{\frac{\sqrt{yy_t}}{L_t}}}{(yy_t)^{1/4}/L_t^{1/2}}
			\ind{\{yy_t> L_t^2\}}\right)\\
			&=&
			\frac{C_\nu}{2}e^{-\frac{y_t+y}{2L_t}}\left(\frac{y^\nu}{L_t^{\nu+1}}\ind{\{yy_t\leq L_t^2\}}
			+
			\frac{y^{\frac{\nu}{2}-\frac{1}{4}}e^{\frac{\sqrt{yy_t}}{L_t}}}{(y_t)^{\frac{\nu}{2}+\frac{1}{4}}L_t^{1/2}}
			\ind{\{yy_t> L_t^2\}}\right).
		\end{eqnarray*}
		On $\{yy_t> L_t^2\}$, we have $y_t^{-1}\leq y/L_t^2$ and, since $\nu+1>0$,
		\begin{eqnarray*}
			\frac{y^{\frac{\nu}{2}-\frac{1}{4}}}{(y_t)^{\frac{\nu}{2}+\frac{1}{4}}}=
			y_t^{1/4} \frac{y^{\frac{\nu}{2}-\frac{1}{4}}}{(y_t)^{\frac{\nu}{2}+\frac{1}{2}}}
			\leq 
			y_t^{1/4} \frac{y^{\nu+\frac{1}{4}}}{L_t^{\nu+1}}.
		\end{eqnarray*}
		So
		\begin{eqnarray*}
			p_t(y_0,y)&\leq&
			\frac{C_\nu}{2}e^{-\frac{y_t+y}{2L_t}}\left(\frac{y^\nu}{L_t^{\nu+1}}\ind{\{yy_t\leq L_t^2\}}
			+
			\frac{(yy_t)^{1/4}y^{\nu}
				e^{\frac{\sqrt{yy_t}}{L_t}}}{L_t^{\nu+\frac{3}{2}}}\ind{\{yy_t> L_t^2\}}\right)\\
			&\leq&
			\frac{C_\nu}{2L_t^{\nu+\frac{3}{2}}}
			e^{-\frac{y_t+y}{2L_t}}
			y^{\nu}e^{\frac{\sqrt{yy_t}}{L_t}}
			\left(L_t^{1/2}\ind{\{yy_t\leq L_t^2\}}
			+
			(yy_t)^{1/4}\ind{\{yy_t> L_t^2\}}\right)  \\
			&=&
			\frac{C_\nu}{2L_t^{\nu+\frac{3}{2}}}
			e^{-\frac{(\sqrt{y}-\sqrt{y_t})^2}{2L_t}}
			y^{\nu}
			\left(L_t^{1/2}\ind{\{yy_t\leq L_t^2\}}
			+
			(yy_t)^{1/4}\ind{\{yy_t> L_t^2\}}\right),
		\end{eqnarray*}
		and the assertion follows.
	\end{proof}
	We are now ready to prove Theorem \ref{theotransition}, which we have used in order to prove the identification of the semigroups in Proposition \ref{identificationsemigroup} and which we will use again later on in this chapter. 
	%	Recall now the Sobolev  spaces $L^p(\O,\m_{\gamma,\mu})$ introduced in Definition \ref{def-sob} for $p\geq 1$. We can prove the following property of the transition semigroup.
	%\begin{theorem}\label{theotransition}
	%For all $p>1$, $\gamma>0$ and $\mu>0$ there exists $\lambda >0$ such that,
	%for every compact $K\subseteq\R\times[0,+\infty)$ and for every  $T>0$, there is $C_{p,K, T}>0$ such that
	%	\[
	%P^\lambda_t f(x_0,y_0)\leq \frac{C_{p,K, T}}{t^{\frac{\beta}{p}+\frac{3}{2p}}}||f||_{L^p(\O,\m_{\gamma,\mu})},\qquad (x_0,y_0)\in K.
	%	\]
	%	for every measurable positive function  $f$ on $\R\times[0,+\infty)$ and for every $t\in(0,T]$.
	%\end{theorem}
	\begin{proof}[Proof of Theorem \ref{theotransition}]
		Note that
		\begin{eqnarray*}
			P^\lambda_tf(x_0,y_0)&=&\E_{x_0,y_0}\left(e^{-\lambda\int_0^t(1+Y_s)ds} \tilde{f}(\tilde X_t,Y_t)\right),
		\end{eqnarray*}
		where
		\[
		\tilde{f}(x,y)=f\left(x+\frac{\rho}{\sigma}y,y\right) \quad\mbox{and}\quad \tilde X_t=X_t-\frac{\rho}{\sigma} Y_t.
		\]
		Recall that the dynamics of $\tilde X$ is given by \eqref{model2} so we have
		\[
		\tilde X_t=\tilde{x}_0+\bar\kappa\int_0^tY_sds+\bar\rho\int_0^t\sqrt{Y_s}d\tilde{B}_s,
		\]
		with
		\[
		\tilde{x}_0=x_0-\frac{\rho}{\sigma}y_0,\quad \bar\kappa=\frac{\rho\kappa}{\sigma}-\frac{1}{2},\quad
		\bar\rho=\sqrt{1-\rho^2}.
		\]
		Recall that the Brownian motion $\tilde{B}$ is independent of the process $Y$. We set
		$
		\Sigma_t=\sqrt{\int_0^t Y_sds}
		$
		% $\E_{y_0}(\,)=\E\left(\;\;\;|\,Y_0=y_0\right)$
		and $n(x)= \frac{1}{\sqrt{2\pi}}e^{-x^2/2}$. Therefore
		\begin{eqnarray*}
			P^\lambda_t f(x_0,y_0)&=&\E_{y_0}\left( e^{-\lambda t-\lambda\Sigma_t^2}
			\int \tilde{f}\left(\tilde{x}_0+\bar\kappa\Sigma_t^2+\bar\rho\Sigma_t z,Y_t\right)n(z)dz\right)\\
			&\leq&\E_{y_0}\left( e^{-\lambda\Sigma_t^2}
			\int \tilde{f}\left(\tilde{x}_0+\bar\kappa\Sigma_t^2+\bar\rho\Sigma_t z,Y_t\right)n(z)dz\right)\\
			&=&\E_{y_0}\left( e^{-\lambda\Sigma_t^2}
			\int \tilde{f}\left(
			\tilde{x}_0+z,Y_t\right)n\left(\frac{z-\bar\kappa\Sigma_t^2}{\bar\rho\Sigma_t} \right)
			\frac{dz}{\bar\rho\Sigma_t}\right).
		\end{eqnarray*}
		H\"{o}lder's inequality with respect to the measure $e^{-\gamma|z|-\bar{\mu}Y_t}dzd\P_{y_0}$,
		where $\gamma >0$ and $\bar \mu>0$ will be chosen later on, gives, for every $p>1$
		\begin{eqnarray}
		P^\lambda_t f(x_0,y_0)&\leq &
		\left[\E_{y_0}\left(\int e^{-\gamma|z|-\bar\mu Y_t} \tilde{f}^p\left(
		\tilde{x}_0+z,Y_t\right)dz\right)\right]^{1/p}J_q,\label{eq-P1}
		\end{eqnarray}
		with $q=p/(p-1)$ and
		\[
		(J_q)^q=\E_{y_0}\left(\int e^{(q-1)\gamma|z|+(q-1)\bar\mu Y_t-q\lambda \Sigma_t^2}
		n^q\left(\frac{z-\bar\kappa\Sigma_t^2}{\bar\rho\Sigma_t} \right)
		\frac{dz}{(\bar\rho\Sigma_t)^q}\right).
		\]
		Using Proposition \ref{density_CIR} we can write, for every $z\in\R$,
		\begin{align*}
		&	\E_{y_0} \left(e^{-\bar\mu Y_t} \tilde{f}^p\left(
		\tilde{x}_0+z,Y_t\right)\right)=\int_0^\infty dy p_t(y_0,y) e^{-\bar\mu y}
		\tilde{f}^p\left(
		\tilde{x}_0+z,y \right)\\
		&\quad \leq \frac{C_\beta\left(\sqrt{\frac{\sigma^2}{4\kappa}}+y_0^{1/4}\right)}{L_t^{\beta+\frac{1}{2}}}\int_0^\infty \!\!\!dy
		e^{-\frac{(\sqrt{y}-\sqrt{y_t})^2}{2L_t}-\bar\mu y}
		y^{\beta -1}
		\left(1
		+
		y^{1/4}\right) \tilde{f}^p\left(
		\tilde{x}_0+z,y \right).
		\end{align*}
		If we set $L_\infty=\sigma^2/(4\kappa)$,  for every $\epsilon\in(0,1)$  we have
		\begin{align*}
		e^{-\frac{(\sqrt{y}-\sqrt{y_t})^2}{2L_t}}&\leq e^{-\frac{(\sqrt{y}-\sqrt{y_t})^2}{2L_\infty}}\\
		&= e^{-\frac{y}{2L_\infty}}e^{\frac{\sqrt{yy_t}}{L_\infty}-\frac{y_t}{2L_\infty}}\\
		&\leq e^{-\frac{y}{2L_\infty}}e^{\epsilon\frac{y}{2L_\infty}}e^{\frac{y_t}{2\epsilon L_\infty}}
		e^{-\frac{y_t}{2L_\infty}}\\
		&=e^{-(1-\epsilon)\frac{y}{2L_\infty}}e^{\frac{y_t}{2\epsilon L_\infty}(1-\epsilon)}\\
		&\leq e^{-(1-\epsilon)\frac{y}{2L_\infty}}e^{\frac{y_0}{2\epsilon L_\infty}(1-\epsilon)}.
		\end{align*}
		It is easy to see that $e^{-y\left( \bar \mu +\frac{1-\epsilon}{2L_\infty}   \right)}(1+y^{1/4})\leq C_{\epsilon,\sigma,\kappa}e^{-y\left( \bar \mu +\frac{1-2\epsilon}{2L_\infty}   \right)}$. Therefore,  we can write
		\begin{align*}
		&	\E_{y_0} \left(e^{-\bar\mu Y_t} \tilde{f}^p\left(
		\tilde{x}_0+z,Y_t\right)\right)\\&\quad\leq
		\frac{C_\beta e^{\frac{y_0(1-\epsilon)}{2\epsilon L_\infty}}
			\left(\sqrt{\frac{\sigma^2}{4\kappa}}+y_0^{1/4}\right)}{L_t^{\beta+\frac{1}{2}}}\int_0^\infty \!\!\!dy
		e^{-y\left(\bar\mu + \frac{1-\epsilon}{2L_\infty}\right)}
		y^{\beta -1}
		\left(1
		+
		y^{1/4}\right) \tilde{f}^p\left(
		\tilde{x}_0+z,y \right)\\&\quad
		\leq
		\frac{C_{\beta, \sigma,\kappa, \epsilon} e^{\frac{y_0(1-\epsilon)}{\epsilon L_\infty}}
		}{L_t^{\beta+\frac{1}{2}}}\int_0^\infty \!\!\!dy
		e^{-y\left(\bar\mu + \frac{1-2\epsilon}{2L_\infty}\right)}
		y^{\beta -1}
		\tilde{f}^p\left(
		\tilde{x}_0+z,y \right).
		\end{align*}
		As regards $J_q$, setting $z'=\frac{z-\bar\kappa\Sigma_t^2}{\bar\rho\Sigma_t}$, we have
		\begin{eqnarray*}
			(J_q)^q&=&\E_{y_0}\left(\int e^{(q-1)\gamma|z'\bar\rho\Sigma_t +\bar\kappa\Sigma_t^2|+(q-1)\bar\mu Y_t-q\lambda \Sigma_t^2}
			n^q\left(z' \right)
			\frac{dz'}{(\bar\rho\Sigma_t)^{q-1}}\right)\\
			&\leq &
			\E_{y_0}\left(
			\int e^{(q-1)\gamma\bar\rho\Sigma_t |z|
				+(q-1)\bar\mu Y_t+ ((q-1)| \bar\kappa|\gamma-q\lambda )\Sigma_t^2}
			n^q\left(z \right)
			\frac{dz}{(\bar\rho\Sigma_t)^{q-1}}\right).
		\end{eqnarray*}
		Note that
		\begin{eqnarray*}
			\int e^{(q-1)\gamma\bar\rho\Sigma_t |z|
			}
			n^q\left(z \right)
			dz
			&=&\frac{1}{(\sqrt{2\pi})^{q}}\int e^{(q-1)\gamma\bar\rho\Sigma_t |z|
			}
			e^{-qz^2/2}
			dz\\
			&\leq &\frac{2}{\sqrt{2\pi}} \int e^{(q-1)\gamma\bar\rho\Sigma_t z
			}
			e^{-qz^2/2}
			dz\\
			&=&\frac{2}{\sqrt{2\pi}}e^{\frac{(q-1)^2}{2q}\gamma^2\bar\rho^2 \Sigma_t^2}
			\int e^{-\frac 1 2\left( \sqrt{q }z-\frac{(q-1)\gamma\bar\rho\Sigma_t }{\sqrt{ q}}\right)^2}	
			dz\\
			&=&\frac 2 {\sqrt q}e^{\frac{(q-1)^2}{2q}\gamma^2\bar\rho^2 \Sigma_t^2},
		\end{eqnarray*}
		so that
		\begin{eqnarray*}
			(J_q)^q&\leq&\frac 2 {\sqrt q}\E_{y_0}\left(
			e^{
				(q-1)\bar\mu Y_t+ \bar\lambda_q\Sigma_t^2}
			\frac{1}{(\bar\rho\Sigma_t)^{q-1}}
			\right),
		\end{eqnarray*}
		with
		\[
		\bar\lambda_q= (q-1)| \bar\kappa|\gamma +\frac{(q-1)^2}{2q}\gamma^2\bar\rho^2 -q\lambda
		=\frac{1}{p-1}\left(| \bar\kappa|\gamma+\frac{1}{2p}\gamma^2\bar\rho^2-p\lambda\right).
		\]
		Using  H\"older's  inequality again
		we get, for every $p_1>1$ and $q_1=p_1/(p_1-1)$,
		\begin{eqnarray*}
			(J_q)^q&\leq&\sqrt{\frac 2 q}\left(
			\E_{y_0}\left(
			e^{
				p_1(q-1)\bar\mu Y_t+ p_1\bar\lambda_q\Sigma_t^2}\right)\right)^{1/p_1}
			\left(
			\E_{y_0}\left(
			\frac{1}{(\bar\rho\Sigma_t)^{q_1(q-1)}}\right)\right)^{1/q_1}\\
			&\leq&
			\frac{C_{q,q_1}}{t^{q-1}} \left(\E_{y_0}\left(
			e^{
				p_1(q-1)\bar\mu Y_t+ p_1\bar\lambda_q\Sigma_t^2}\right)\right)^{1/p_1},
		\end{eqnarray*}
		where the last inequality follows from Lemma \ref{RemarqueMomentsNeg}.
		
		We now apply Proposition~\ref{Laplace2} with $\lambda_1=p_1(q-1)\bar\mu$
		and $\lambda_2=p_1\bar\lambda_q$. The assumption on $\lambda_1$ and $\lambda_2$ becomes
		\[ \frac{\sigma^2}{2}p_1(q-1)\bar\mu^2-\kappa \bar\mu+| \bar\kappa|\gamma+\frac{1}{2p}\gamma^2\bar\rho^2-p\lambda\leq 0
		%p_1(q-1)\bar\mu\leq \frac{\kappa}{\sigma^2}+\sqrt{\left(\frac{\kappa}{\sigma^2}\right)^2-
		%     2\frac{p_1\bar\lambda_q}{\sigma^2}}
		\]
		or, equivalently, 
		\[
		\lambda\geq \frac{\sigma^2}{2p(p-1)}p_1\bar\mu^2-\kappa \frac{\bar\mu}{p}+| \bar\kappa|\frac{\gamma}{p}+\frac{1}{2p^2}\gamma^2\bar\rho^2.
		\]
		Note that the last inequality is satisfied for at least a  $p_1>1$ if and only if
		\begin{equation}\label{assumption_lambda}
		\lambda> \frac{\sigma^2}{2p(p-1)}\bar\mu^2-\kappa \frac{\bar\mu}{p}+| \bar\kappa|\frac{\gamma}{p}+\frac{1}{2p^2}\gamma^2\bar\rho^2.
		\end{equation}
		%et on peut trouver un $p_1>1$ qui la v\'erifie si et seulement si
		%\[
		%\frac{\sigma^2}{2p(p-1)}\bar\mu^2<\frac{\kappa\bar\mu}{p}+\lambda-
		%     \frac{|\bar\kappa|\gamma}{p}-\frac{1}{p^2}\gamma^2\bar\rho^2.
		%\]
		%Donc, pour $p$ fix\'e, si $\bar\mu< \frac{\kappa}{\sigma^2}$  la seule condition
		% que $\lambda$ doit v\'erifier est
		%\[
		%\lambda\geq  \frac{|\bar\kappa|\gamma}{p}+\frac{1}{p^2}\gamma^2\bar\rho^2
		%\]
		%Si $\bar\mu\geq  \frac{\kappa}{\sigma^2}$, on doit avoir en outre
		%\[
		%\lambda+\frac{\kappa\bar\mu}{p}>\frac{\sigma^2}{2p(p-1)}\bar\mu^2+
		% \frac{|\bar\kappa|\gamma}{p}+\frac{1}{p^2}\gamma^2\bar\rho^2.
		%\]
		Going back to \eqref{eq-P1} under the condition \eqref{assumption_lambda}, we have
		\begin{eqnarray*}
			P^\lambda_t f(x_0,y_0)&\leq &
			\frac{C_{p,\epsilon}}{L_t^{\frac{\beta}{p}+\frac{1}{2p}}t^{1/p}}e^{A_{p, \epsilon}y_0}
			\left(\int dz e^{-\gamma|z|}\int_0^\infty \!\!\!dy
			e^{-y\left(\bar\mu + \frac{1-2\epsilon}{2L_\infty}\right)}
			y^{\beta -1}
			\tilde{f}^p\left(
			\tilde{x}_0+z,y \right)\right)^{1/p}\\
			&\leq&
			\frac{C_{p,\epsilon}e^{A_{p, \epsilon}y_0}}{t^{\frac{\beta}{p}+\frac{3}{2p}}}
			\left(\int dz e^{-\gamma|z|}\int_0^\infty \!\!\!dy
			e^{-y\left(\bar\mu + \frac{1-2\epsilon}{2L_\infty}\right)}
			y^{\beta -1}
			f^p\left(
			\tilde{x}_0+z+\frac{\rho}{\sigma}y,y \right)\right)^{1/p}\\
			&=&
			\frac{C_{p,\epsilon}e^{A_{p, \epsilon}y_0}}{t^{\frac{\beta}{p}+\frac{3}{2p}}}
			\left(\int dz e^{-\gamma|z-\tilde{x}_0-\frac{\rho}{\sigma}y|}\int_0^\infty \!\!\!dy
			e^{-y\left(\bar\mu + \frac{1-2\epsilon}{2L_\infty}\right)}
			y^{\beta -1}
			f^p\left(
			z,y \right)\right)^{1/p}\\
			&\leq&
			\frac{C_{p,\epsilon}e^{A_{p, \epsilon}y_0+\gamma|\tilde{x}_0|}}{t^{\frac{\beta}{p}+\frac{3}{2p}}}
			\left(\int dz e^{-\gamma|z|}\int_0^\infty \!\!\!dy
			e^{-y\left(\bar\mu -\gamma\frac{|\rho|}{\sigma}+ \frac{1-2\epsilon}{2L_\infty}\right)}
			y^{\beta -1}
			f^p\left(
			z,y \right)\right)^{1/p}.
		\end{eqnarray*}
		If we choose $\epsilon=1/2$ and $\bar\mu=\mu+\gamma\frac{|\rho|}{\sigma}$,  the assertion follows provided $\lambda$ satisfies
		\[
		\lambda>\frac{\sigma^2}{2p(p-1)}\left(\mu+\gamma\frac{|\rho|}{\sigma}\right)^2-
		\kappa \frac{\mu+
			\gamma\frac{|\rho|}{\sigma}}{p}+| \bar\kappa|\frac{\gamma}{p}+\frac{1}{p^2}\gamma^2\bar\rho^2.
		\]
	\end{proof}
	
	\subsection{Proof of Theorem  \ref{theorem2}}
	We are finally ready to prove the identification Theorem \ref{theorem2}. We first prove the result under further regularity assumptions on the payoff function $\psi$, then we deduce the general statement by an approximation technique. 
	
	\subsubsection{Case with a regular function $\psi$ }
	The following regularity result paves the way for the identification theorem in the case of a  regular payoff function.
	\begin{proposition}\label{reg2}
		Assume that $\psi$ satisfies Assumption $\mathcal H^1$ and $0\leq \psi\leq \Phi$ with $\Phi$ satisfying Assumption $\mathcal{H}^2$. If moreover we assume  $\psi \in L^2([0,T];H^2(\O,\m))$ and $ \frac{\partial \psi}{\partial t}+\L \psi , \,(1+y)\Phi\in L^p([0,T];L^p(\mathcal{O,\m}))$ for some $p\geq 2$, then there exist $\lambda_0>0$ and $F\in L^p([0,T];L^p(\mathcal{O,\m}))$ such that for all $ \lambda \geq \lambda_0$ the solution $u$ of \eqref{VI} satisfies
		\begin{equation}
		-\left(\frac{\partial u}{\partial t}	 ,v\right)_H + a_\lambda(u,v)=(F,v)_H,\qquad \mbox{a.e. in } [0,T], \quad v\in V.
		\end{equation}
	\end{proposition}
	\begin{proof}
		Note that, for $\lambda$ large enough, $u$ can be seen as the solution $u_\lambda$ of an equivalent coercive variational inequality, that is 
		\begin{equation*}
		-\left(\frac{\partial u_{\lambda}}{\partial t}	 ,v -u_\lambda  \right)_H + a_\lambda(u_\lambda,v-u_\lambda)\geq (g,v-u_\lambda)_H,
		\end{equation*}
		where $g= \lambda(1+y)u $ satisfies the assumptions of Proposition \ref{coercive variational inequality}. Therefore,  there exists a sequence $(u_{\varepsilon, \lambda})_\varepsilon$ of non  negative functions such that $\lim_{\varepsilon\rightarrow 0 }u_{\varepsilon, \lambda}=u_\lambda$ and
		$$
		-\left( \frac{\partial u_{\varepsilon,\lambda}}{\partial t },v   \right)_H + a_\lambda(u_{\varepsilon,\lambda},v)-\left(\frac 1 \varepsilon (\psi-u_{\varepsilon,\lambda})_+,v\right)_H= (g,v)_H,\qquad v\in V.
		$$
		Since  both $u_{\varepsilon,\lambda}$ and $\psi$ are positive and $\psi$ belongs to $L^p([0,T]; L^p(\mathcal{O},\m))$, we have $(\psi-u_{\varepsilon,\lambda})_+\in L^p([0,T]; L^p(\mathcal{O},\m))$. In order to simplify the notation, we set $w=(\psi-u_{\varepsilon,\lambda})_+$. Taking $v=w^{p-1}$ and assuming that $\psi$ is bounded we observe that $v\in L^2([0,T]; V  )$ and we can write
		$$
		-\left(	\frac{\partial u_{\varepsilon,\lambda}}{\partial t}, w^{p-1}\right)_H
		+a_\lambda(u_{\varepsilon,\lambda},w^{p-1})-\frac 1 \varepsilon \Vert w\Vert^p_{L^p(\mathcal{O},\m)} =\left( g, w^{p-1} \right)_H,
		$$
		so that
		$$
		\frac 1 p \frac{d }{dt }\Vert w\Vert_{ L^p(\mathcal{O},\m)}^p  -a_\lambda(\psi-u_{\varepsilon,\lambda},w^{p-1})-\frac 1 \varepsilon \Vert w\Vert^p_{L^p(\mathcal{O},\m)} 
		=\left( g, w^{p-1} \right)_H-\left( \frac{\partial \psi }{\partial t },w^{p-1} \right)_H+ a_\lambda(\psi,w^{p-1}).
		$$
		Integrating from $0$ to $T$ we get
		\begin{equation}\label{regpaux}
		\begin{split}
		&	-	\frac 1 p\Vert w(0)\Vert_{ L^p(\mathcal{O},\m)}^p -\int_0^T a_\lambda((\psi-u_{\varepsilon,\lambda})(t),w^{p-1}(t))dt  -\frac 1 \varepsilon\int_0^T \Vert w(t)\Vert_{ L^p(\mathcal{O},\m)}^p dt \\&\quad
		=\int_0^T\left( g(t), w^{p-1}(t) \right)_Hdt-\int_0^T \left( \frac{\partial \psi }{\partial t }(t),w_+^{p-1}(t) \right)_Hdt+ \int_0^Ta_\lambda(\psi(t),w^{p-1}(t))dt.
		\end{split}
		\end{equation}

		Now, with the usual integration by parts,
		\begin{align*}
		&a_{\lambda}(w,w^{p-1})
		= \into \frac y 2 (p-1 ) w^{p-2 }\left[ \left( \frac{\partial w}{\partial x }\right)^2 +2\rho\sigma      \frac{\partial w}{\partial x }\frac{\partial w}{\partial y }+ \sigma^2 \left( \frac{\partial w}{\partial y }\right)^2 \right]d\m
		\\&		 \qquad		+\into y\left( j_{\gamma,\mu }(x)  \frac{\partial w}{\partial x }+ k_{\gamma,\mu}(x)   \frac{\partial w}{\partial y }  \right)  w^{p-1} d\m + \lambda \into  ( 1+y ) w^p d\m 
		\\& 	\geq \delta_1(p-1) \into y w^{p-2 } \left[ \left( \frac{\partial w}{\partial x }\right)^2 + \left( \frac{\partial w}{\partial y }\right)^2 \right]d\m
		+\into y\left( j_{\gamma,\mu }(x)  \frac{\partial w}{\partial x }+ k_{\gamma,\mu}(x)   \frac{\partial w}{\partial y }  \right)  w^{p-1} d\m \\&\qquad+ \lambda \into  y w^p d\m 
		\\& 		= \into y  w^{p-2 }   \bigg[    \delta_1(p-1)   \left( \frac{\partial w}{\partial x }\right)^2  
		+  j_{\gamma,\mu }(x)  \frac{\partial w}{\partial x } w + \frac{\lambda}{2 }  w^2 \bigg] d\m\\&\qquad
		+ \into y  w^{p-2 }   \bigg[    \delta_1(p-1)   \left( \frac{\partial w}{\partial y }\right)^2 
		+  k_{\gamma,\mu }(x)  \frac{\partial w}{\partial y } w + \frac{\lambda}{2 }  w^2 \bigg] d\m
		\geq 0,
		\end{align*}
		since, for $\lambda$ large enough, the quadratic forms $(a,b)\rightarrow  \delta_1(p-1)a^2 + j_{\gamma.\mu}  ab + \frac \lambda 2 b^2 $ and $(a,b)\rightarrow  \delta_1(p-1)a^2 + k_{\gamma.\mu}  ab + \frac \lambda 2 b^2 $ are both positive definite.
		
		Recall that $\psi\in L^2([0,T];H^2(\O,\m)), $  $\frac{\partial \psi }{\partial_t}+\mathcal{L }\psi\in L^p([0,T],L^p(\O,\m)), \, (1+y)\psi\leq(1+y)\Phi\in L^p([0,T],L^p(\O,\m))$  and $g=(1+y)u\leq(1+y)\Phi\in L^p([0,T];L^p(\mathcal{O,\m}))$.
		Therefore, going back to \eqref{regpaux} and using H\"{o}lder's inequality,
		\begin{align*}
		&	\frac 1 \varepsilon\int_0^T \Vert w(t)\Vert_{ L^p(\mathcal{O},\m)}^pdt\\&\quad \leq \left[\left(\int_0^T \Vert g(t)\Vert_{ L^p(\mathcal{O},\m)}^pdt  \right)^{\frac 1 p}\!\!+\!\! \left(\int_0^T\!\left \Vert   \frac{\partial \psi }{\partial t} (t)+ \mathcal{L}^\lambda\psi (t) \right \Vert_{ L^p(\mathcal{O},\m)}^p\!\!\!\!\!\!
		dt  \right)^{\frac 1 p} 
		\right]  \left( \int_0^T \Vert w\Vert_{ L^p(\mathcal{O},\m)}^pdt  \right)^{\frac {p-1} p}\!\!.
		\end{align*}
		Recalling that $w=(\psi-u_{\varepsilon,\lambda})_+$, we deduce that
		\begin{equation}
		\left \Vert \frac 1 \varepsilon (\psi-u_{\varepsilon,\lambda})_+ \right \Vert_{L^p([0,T]; L^p(\mathcal{O},\m))} \leq C,
		\end{equation}
		for a positive constant $C$ independent of $\varepsilon$. 	Note that the estimate does not involve the $L^\infty$-norm of $\psi$ (which we assumed to be bounded for the payoff) so that by a standard approximation argument, it remains valid for unbounded $\psi$. 
		The assertion then follows passing to the limit for $\varepsilon\rightarrow0$ in 
		$$
		-\left( \frac{\partial u_{\varepsilon,\lambda}}{\partial t },v   \right)_H + a_\lambda(u_{\varepsilon,\lambda},v)= \left(\frac 1 \varepsilon (\psi-u_{\varepsilon,\lambda})_+,v\right)_H +(g,v)_H,\qquad  v\in V.
		$$
	\end{proof}
	Now, note that we can easily prove the continuous dependence of the process $X$ with respect to the initial state. 
	\begin{lemma}\label{propflo}
		Fix $(x,y)\in \R\times [0,+\infty)$. Denote by $(X^{x,y}_t, Y^y_t)_{t\geq 0}$ the solution of the system
		\[
		\left\{
		\begin{array}{l}
		dX_t=\left( \frac{\rho\kappa\theta}{\sigma}-\frac{Y_t}{2}\right)dt+\sqrt{Y_t} dB_t,\\
		dY_t=\kappa(\theta-Y_t)dt+\sigma\sqrt{Y_t}dW_t,
		\end{array}
		\right.
		\]
		with $X_0=x$, $Y_0=y$ and   $\langle B,W\rangle_t=\rho t$.
		We have, for every  $ t\geq 0$ and for every $(x,y), \, (x',y')\in \R\times [0,+\infty)$,
		$\E\left|Y^{y'}_t-Y^y_t\right| \leq |y'-y|$ and
		\[
		\E\left|X^{x',y'}_t-X^{x,y}_t\right| \leq |x'-x|+\frac{t}{2}|y'-y|+\sqrt{t|y'-y|}.
		\]    
	\end{lemma}
	The proof of Lemma \ref{propflo} is straightforward so we omit the details: the inequality $\E\left|Y^{y'}_t-Y^y_t\right| \leq  |y'-y|$ can be proved by using standard techniques introduced in \cite{IW} (see the proof of Theorem 3.2 and its Corollary in Section IV.3) and the other inequality easily follows.

	Then, we can prove the following result.
	\begin{proposition}\label{propconti1}
		Let $\psi:\R\times[0,\infty)\rightarrow\R$ be  continuous 
		and such that there exist $C>0$ and $ a,\, b \geq0$ with  $|\psi(x,y)|\leq Ce^{a|x|+by}$ for every $(x,y)\in \R\times [0,+\infty)$. Then, if
		\[
		\lambda> ab|\rho|\sigma+\frac{b^2\sigma^2}{2}-\kappa b+\frac{a^2-a}{2},
		\]
		we have $P^\lambda_t|\psi|(x,y)<\infty$ for every $t\geq 0$, $(x,y)\in \R\times [0,+\infty)$ and the function
		$(t,x,y)\mapsto P^\lambda_t\psi(x,y)$ is continuous on $[0,\infty)\times\R\times[0,\infty)$.
	\end{proposition}
	\begin{proof}
		We can prove, as in the proof of  Proposition \ref{prop-TF},  that
		\begin{eqnarray*}
			\E_{x,y}\left(e^{a X_t+bY_t-\lambda\int_0^t Y_s ds}\right)&=&
			e^{a\left(x-\frac{\rho}{\sigma}y\right)}
			\E_{y}\left(e^{\left(a\frac{\rho}{\sigma}+b\right)Y_t+
				\left(a(\frac{\rho\kappa}{\sigma}-\frac{1}{2})+\frac{a^2}{2}(1-\rho^2)-\lambda\right)\int_0^tY_sds}\right).
		\end{eqnarray*}
		Thanks to Proposition~\ref{Laplace2}, if
		\begin{equation}\label{6*}
		\frac{\sigma^2}{2}\left(a\frac{\rho}{\sigma}+b\right)^2-\kappa\left(a\frac{\rho}{\sigma}+b\right)+
		\left(a(\frac{\rho\kappa}{\sigma}-\frac{1}{2})+\frac{a^2}{2}(1-\rho^2)-\lambda\right)<0,
		\end{equation}
		we have, for any $T>0$ and for any compact $K\subseteq \R\times [0,+\infty[$,
		\[
		\sup_{(t,x,y)\in[0,T]\times K} \E_{x,y}\left(e^{a X_t+bY_t-\lambda\int_0^t Y_s ds}\right)<\infty.
		\]
		Note that \eqref{6*} is equivalent to
		\[
		\lambda>ab\rho\sigma+\frac{b^2\sigma^2}{2}-\kappa b+\frac{a^2-a}{2}.
		\]
		Therefore, under the assumptions of the Proposition, we have,
		for any $T>0$ and for any compact set $K\subseteq \R\times [0,+\infty[$,
		\[
		\sup_{(t,x,y)\in[0,T]\times K} \E_{x,y}\left(e^{a |X_t|+bY_t-\lambda\int_0^t Y_s ds}\right)<\infty.
		\]
		Moreover, for $\epsilon$ small enough,
		\begin{equation}\label{6**}
		\sup_{(t,x,y)\in[0,T]\times K} \E_{x,y}\left(e^{a(1+\epsilon) |X_t|+b(1+\epsilon)Y_t-\lambda(1+\epsilon)\int_0^t Y_s ds}\right)<\infty.
		\end{equation}
		Then, let  $\psi$ be a continuous function on $\R\times[0,+\infty[$ such that $|\psi(x,y)|\leq Ce^{a|x|+by}$. It is evident that $P^\lambda_t |\psi|(x,y)<\infty$
		and we have
		\[
		P^\lambda_t\psi(x,y)=\E\left(e^{-\lambda\int_0^t(1+Y^y_s)ds}\psi(X^{x,y}_t,Y^{y}_t)\right).
		\]
		If $((t_n,x_n,y_n))_n$ converges to $(t,x,y)$,  we deduce from Lemma~\ref{propflo} that $X^{x_n,y_n}_{t_n}\rightarrow X^{x,y}_t$, $Y^{y_n}_{t_n}\rightarrow Y^{y}_t$ and  $\int_0^{t_n} Y^{y_n}_sds \rightarrow\int_0^t Y^y_sds$ in probability. 
		Therefore $e^{-\lambda\int_0^{t_n}(1+Y_s)ds}\psi(X^{x_n,y_n}_{t_n},Y^{y_n}_{t_n})$
		converges to $e^{-\lambda\int_0^{t}(1+Y_s)ds}\psi(X^{x,y}_{t},Y^{y}_{t})$ in probability. The estimate~\eqref{6**}
		ensures the uniformly integrability of $e^{-\lambda\int_0^{t_n}(1+Y_s)ds}\psi(X^{x_n,y_n}_{t_n},Y^{y_n}_{t_n})$
		so that  $\lim_{n\to \infty}P^\lambda_{t_n}\psi(x_n,y_n)=P^\lambda_{t}\psi(x,y)$ which concludes the proof.		
		%		 			If $((t_n,x_n,y_n))_n$ converges to $(t,x,y)$,  we deduce from Lemma~\ref{propflo} that $X^{x_n,y_n}_{t_n}\rightarrow X^{x,y}_t$, $Y^{y_n}_{t_n}\rightarrow Y^{y}_t$ and  $\int_0^{t_n} Y^{y_n}_sds \rightarrow\int_0^t Y^y_sds$ in probability. 
		%		 			Therefore, the assertion is trivial if $\varphi$ is bounded continuous. If $\varphi\in\L^p(\O,\m)$, $\varphi$ is the limit in $L^p$ of a sequence of bounded continuous functions  $(\varphi_n)_n$. Moreover, thanks to Theorem \ref{theotransition}, for every compact $K\subseteq\R\times[0,+\infty)$, there is $C_{p,K, T}>0$ such that
		%		 			$$
		%		 			P^\lambda_t| \varphi_n-\varphi|(x,y) \leq \frac{C_{p,K, T}}{t^{\frac{\beta}{p}+\frac{3}{2p}}}||\varphi_n-\varphi||_{L^p(\O,\m)}, \qquad(x,y)\in K.
		%		 			$$
		%		 			Therefore $P^\lambda_t \varphi_n(x,y)$ converges locally uniformly to $P^\lambda_t \varphi(x,y)$ and the assertion follows.
	\end{proof} 
	
	\begin{proposition}\label{martingale}
		Fix  $p>\beta+\frac{5}{2}$ and $\lambda$ as in Theorem \ref{theotransition}. Let us consider
		$u\in C([0,T];H)\cap L^2([0,T];V)$, with $\frac{\partial u}{\partial t}\in L^2([0,T];H)$ such that
		\[
		\begin{cases}
		\left	(\frac{\partial u}{\partial t},v\right)_H+a_\lambda(u(t),v)=(f(t),v)_H,\qquad v\in V,\\
		u(0)=\psi,
		\end{cases}
		\]
		with $\psi$ continuous, $\psi\in V$, $\sqrt{1+y}f\in L^2([0,T];H)$  and $f\in L^p([0,T];L^p(\mathcal{O}, \m))$. Then, if $\psi$
		and $\lambda$ satisfy the assumptions of Proposition ~\ref{propconti1},  we have 
		\begin{enumerate}
			\item For every $t\in [0,T]$, $u(t)=P^\lambda_t\psi+\int_0^tP^\lambda_sf(t-s)ds$.
			\item  The function  $(t,x,y)\mapsto u(t,x,y)$ is continuous on $[0,T]\times \R\times [0,+\infty)$.
			\item If $\Lambda_t=\lambda\int_0^t(1+Y_s)ds$, 
			the process $(M_t)_{0\leq t\leq T}$, defined by
			\[
			M_t=e^{-\Lambda_t}u(T-t,X_t,Y_t)+\int_0^te^{-\Lambda_s}f(T-s,X_s,Y_s)ds,
			\]
			with $X_0=x,\, Y_0=y$	is a martingale
			for every $(x,y)\in \R\times [0,+\infty)$.
		\end{enumerate}
	\end{proposition}
	\begin{proof}
		The first assertion follows from Proposition~\ref{propsg3}.

		The continuity of $(t,x,y)\mapsto P^\lambda_t\psi(x,y)$ is given by Proposition~\ref{propconti1}. The continuity of  $(t,x,y)\mapsto \int_0^t P^\lambda_sf(t-s,.)(x,y)ds$  is trivial if $(t,x,y)\mapsto f(t,x,y)$
		is bounded continuous.  If $f\in L^p([0,T];L^p(\mathcal{O}, \m))$, $f$ is the limit in  $L^p$ of a sequence of  bounded continuous functions and we have $\int_0^t P^\lambda_s f_n(t-s,\cdot)ds \rightarrow \int_0^t P^\lambda_s f(t-s,\cdot)ds$ uniformly in $[0,T]\times K$ for every compact
		$K$ of $\R\times [0,+\infty)$). In fact, thanks to Theorem~\ref{theotransition}, we can write for 
		$t\in[0,T]$ and $(x,y)\in K$ 
		\begin{equation}
		\begin{split}
		&	\int_0^tP^\lambda_s|f_n-f|(t-s,\cdot,\cdot)(x,y)ds\leq \int_0^t\frac{C_{p,K,T}}{s^{\frac{2\beta+3}{2p}}}ds
		||(f_n-f)(t-s,\cdot,\cdot)||_{L^p(\O,\m)}\\
		&\quad\leq C_{p,K,T}\left(\int_0^t||(f_n-f)(t-s,\cdot,\cdot)||^p_{L^p(\O,\m)}ds\right)^{1/p}\left(\int_0^t\frac{ds}{s^{\frac{2\beta+3}{2(p-1)}}}\right)^{1-\frac{1}{p}}\\
		&\quad \leq C_{p,K,T}\left(\int_0^T||(f_n-f)(s,\cdot,\cdot)||^p_{L^p(\O,\m)}ds\right)^{1/p}\left(\int_0^T\frac{ds}{s^{\frac{2\beta+3}{2(p-1)}}}\right)^{1-\frac{1}{p}}.
		\end{split}
		\end{equation}
		The assumption $p>\beta+\frac{5}{2}$ ensures the convergence of the integral in the right hand side.

		For the last assertion, note that $M_T=e^{-\Lambda_T}\psi(X_T,Y_T)+\int_0^Te^{-\Lambda_s}f(T-s,X_s,Y_s)ds$.
		Then, we can prove that $M_t$ is integrable with the same arguments that we used to show the continuity of $(t,x,y)\mapsto u(t,x,y)$.
		%		 			The integrability of the first term follows from Proposition~\ref{propconti1} while for the second term in the right hand side we use Theorem ~\ref{theotransition}
		%		 			which gives, for $(x,y)$ in a compact $K$,
		%		 			\begin{eqnarray*}
		%		 				\E_{x,y}\left(\int_0^T e^{-\Lambda_s}|f(T-s,X_s,Y_s)|ds\right)&=&
		%		 				\int_0^TP^\lambda_s |f(T-s,.,.)|(x,y)ds\\
		%		 				&\leq & C_{p,K,T}\int_0^T\frac{1}{s^{\frac{2\beta+3}{2p}}}ds
		%		 				||f(T-s,.,.)||_{L^p(\O,\m)}\\
		%		 				&\leq&C_{p,K,T}\left(\int_0^T||f(s,.,.)||^p_{L^p(\O,\m)}ds\right)^{1/p}\left(\int_0^T\frac{ds}{s^{\frac{2\beta+3}{2(p-1)}}}\right)^{1-\frac{1}{p}}.
		%		 			\end{eqnarray*}
		%		 		The assumption $p>\beta+\frac{5}{2}$ ensures again that the right hand side is finite.
		Moreover, by using the Markov property,
		\begin{align*}
		&	\E_{x,y}\left( M_T \,|\, \mathcal{F}_t\right)\\&=e^{-\Lambda_t}P^{\lambda}_{T-t}\psi(X_t,Y_t)
		+\int_0^t\!e^{-\Lambda_s}f(T-s,X_s,Y_s)ds
		+e^{-\Lambda_t}\!\int_t^TP^\lambda_{s-t}f(T-s,.,.)(X_t,Y_t)ds\\
		&
		=e^{-\Lambda_t}\left(P^{\lambda}_{T-t}\psi(X_t,Y_t)+\int_0^{T-t}\!\!P^\lambda_{s}f(T-t-s,.,.)(X_t,Y_t)ds\right)\!+\!\int_0^te^{-\Lambda_s}f(T-s,X_s,Y_s)ds\\
		&=e^{-\Lambda_t}u(T-t, X_t, Y_t)+\int_0^te^{-\Lambda_s}f(T-s,X_s,Y_s)ds=M_t.
		\end{align*}
		
	\end{proof}
	
	We are now ready to prove the following proposition.
	
	\begin{proposition}\label{identification}
		Assume that $\psi$ satisfies Assumption $\mathcal{H}^*$.
		Moreover, fix $p>\beta+\frac{5}{2}$ and assume that $\psi\in L^2([0,T];H^2(\O,\m))$ and $\frac{\partial \psi}{\partial t}+\L\psi\in L^p([0,T];L^p(\O,\m))$.
		Then, the solution u of the variational inequality \eqref{VI} satisfies
		\begin{equation} \label{u=u*}
		u(t,x,y)=u^*(t,x,y), \qquad \mbox{on }[0,T] \times \bar{\mathcal{O}},
		\end{equation}
		where $u^*$ is defined by
		$$
		u^*(t,x,y)= \sup_{\tau \in \mathcal{T}_{t,T}}\E \left[   \psi(\tau,X_\tau^{t,x,y},
		Y_\tau^{t,y})     \right].
		$$
	\end{proposition}
	\begin{proof}
		We first check that $\psi$ satisfies the assumptions of Proposition \ref{reg2}. Note that, thanks to the growth condition \eqref{boundonpsi}, it is possible to write  $0\leq \psi(t,x,y)\leq \Phi(t,x,y)$ with $\Phi(t,x,y)=C_T(e^{x-\frac{\rho \kappa\theta}{\sigma}t}+e^{Ly-\kappa\theta L t})$, where $L\in\left[0,\frac{2\kappa}{\sigma^2}\right)$ and $C_T$ is a positive constant which depends on $T$. Moreover,  recall the growth condition on the derivatives  \eqref{boundonderivatives}. Then, it is easy to see that we can choose $\gamma $ and $\mu$ in the definition of the measure $\m$ (see \eqref{m}) such that $\psi$ satisfies Assumption $\mathcal{H}^1$, $ \Phi$ satisfies Assumption $\mathcal{H}^2$ (note that  $\frac{\partial \Phi}{\partial t}+\L\Phi\leq 0$) and
		$(1+y)\Phi$, $\frac{\partial \psi}{\partial t}+\L\psi\in L^p([0,T];L^p(\O,\m))$.  Therefore we can apply  Proposition \ref{reg2} and we get  that, for $\lambda$ large enough,  there exists $F\in L^p([0,T];L^p(\mathcal{O,\m}))$ such that  $u$ satisfies
		\begin{equation*}
		-\left(\frac{\partial u}{\partial t}	 ,v\right)_H + a_\lambda(u,v)=(F,v)_H,\qquad v\in V, 
		\end{equation*}
		that is
		\begin{equation*}
		-\left(\frac{\partial u}{\partial t}	 ,v\right)_H + a(u,v)=(F-\lambda(1+y)u,v)_H,\qquad v\in V.
		\end{equation*}
		%	Then the continuity of $u$ follows from Proposition \ref{martingale}.
		On the other hand we know that
		\begin{equation*}
		\begin{cases}
		-\left( \frac{\partial u}{\partial t },v -u  \right)_H + a(u,v-u)\geq 0 , \qquad \mbox{ a.e. in } [0,T]\qquad v\in V, \ v\geq \psi, \\u(T)=\psi(T),\\u \geq \psi \mbox{ a.e. in } [0,T]\times \R \times (0,\infty).
		\end{cases}
		\end{equation*}
		From the previous relations we easily   deduce that $F-\lambda(1+y)u\geq 0$ a.e. and, taking $v=\psi$, that $(F-\lambda(1+y)u,\psi-u)_H=0.$ 
		Moreover, note that  the assumptions of Proposition \ref{martingale} are satisfied, so
		the process $(M_t)_{0\leq t\leq T}$ defined by
		\begin{equation}\label{M}
		M_t=e^{-\Lambda_t}u(t,X_t,Y_t)+\int_0^te^{-\Lambda_s}F(s,X_s,Y_s)ds,
		\end{equation}
		with $X_0=x, \, Y_0=y$	is a martingale
		for every $(x,y)\in \R\times [0,+\infty)$. Then, we deduce that
		the process
		\[
		\tilde M_t= u(t,X_t,Y_t)+\int_0^{t} \left(F(s,X_s,Y_s)   -\lambda(1+Y_s)u(s,X_s,Y_s)\right)ds
		\]
		is a local martingale. In fact, from \eqref{M} we can write
		\begin{align*}
		d	\tilde M_t&=d\left[e^{\Lambda_t}M_t-e^{\Lambda_t}\int_0^te^{-\Lambda_s}F(s,X_s,Y_s)ds
		\right]+ F(t,X_t,Y_t)dt-\lambda(1+Y_t)u(t,X_t,Y_t)dt\\&
		=e^{\Lambda_t}dM_t+\Big[ \lambda(1+Y_t)e^{\Lambda_t}M_t-\lambda(1+Y_t)e^{\Lambda_t}\int_0^te^{-\Lambda_s}F(s,X_s,Y_s)ds\\&\quad-e^{\Lambda_t}e^{-\Lambda_t} F(t,X_t,Y_t)   + F(t,X_t,Y_t)-\lambda(1+Y_t)u(t,X_t,Y_t)\Big]dt\\
		&=e^{\Lambda_t}dM_t.
		\end{align*}

		So, for any stopping time $\tau$ there exists an increasing sequence of stopping times $(\tau_n)_n$ such that $\lim_n\tau_n=\infty$ and 
		\begin{equation}\label{relationmgloc}
		\E_{x,y}[u(\tau\wedge \tau_n,X_{\tau\wedge \tau_n}, Y_{\tau\wedge \tau_n})]=u(0,x,y)-\E_{x,y}\left[\int_0^{\tau\wedge\tau_n}\!\!\!(  F(s,X_s,Y_s)-\lambda(1+Y_s)u(s,X_s,Y_s))ds\right]\!.
		\end{equation}
		Since $F-\lambda(1+y)u\geq0$ we can pass to the limit in the right hand side of \eqref{relationmgloc} thanks to the monotone convergence theorem.  Recall now that an adapted right continuous process $(Z_{t})_{t \geq 0}$ is said to be 
		of class $\mathcal{D}$ if the family $(Z_{\tau})_{\tau \in \mathcal{T}_{0,\infty} }$, where $\mathcal{T}_{0,\infty}$ is the set of all stopping times with values in $[0,\infty)$,  is uniformly integrable. Moreover, recall that $0\leq u(t,x,y)\leq \Phi(x,y)=C_T(e^{x-\frac{\rho \kappa\theta}{\sigma}t}+e^{Ly-\kappa\theta L t})$.  The discounted and dividend adjusted price process $(e^{-(r-\delta)t}S_t)_t=(e^{X_t-\frac{\rho\kappa\theta}{\sigma}t})_t$  is a martingale (we refer to \cite{Kr} for an analysis of the martingale property in general affine stochastic volatility models), so we deduce that it is of class $\mathcal{D}$. On the other hand, we can prove that the process $(e^{LY_t-\kappa\theta t})_t$ is of class $\mathcal D$ following the same arguments used in Remark \ref{remarksemigroup}. Therefore, the process $(\Phi(t+s,X^{t,x,y}_s))_{s\in[t,T]}$ is of class $ \mathcal{D}$ for every $(t,x,y)\in[0,T]\times\R\times [0,\infty)$.
		So we can pass to the limit in the left hand side of \eqref{relationmgloc} and  we get that $\lim_{n\rightarrow \infty }\E_{x,y}[u(\tau\wedge \tau_n,X_{\tau\wedge \tau_n}, Y_{\tau\wedge \tau_n})]= \E_{x,y}[u(\tau,X_{\tau}, Y_{\tau})]$. Therefore, passing to the limit as $n\rightarrow \infty$, we get
		\begin{equation*}
		\E_{x,y}[u(\tau,X_{\tau}, Y_{\tau})]=u(0,x,y)-\E_{x,y}\left[\int_0^{\tau} (  F(s,X_s,Y_s)-\lambda(1+Y_s)u(s,X_s,Y_s))ds\right], 
		\end{equation*}
		for every $\tau\in\mathcal{T}_{0,T}$.
		Recall that $F -\lambda(1+y)u\geq0$, so the process $ u(t,X_t,Y_t)$ is actually a supermartingale. Since $u\geq \psi$, we deduce directly from the definition of Snell envelope that $u(t,X_t,Y_t) \geq u^*(t,X_t,Y_t)  $ a.e. for $t\in [0,T]$. 
		
		In order to show the opposite inequality, we consider the so called continuation region 
		$$
		\mathcal{C}=\{ (t,x,y)\in [0,T)\times \R\times [0,\infty): u(t,x,y)>\psi(t,x,y)   \},
		$$ 
		its $t$-sections
		$$
		\mathcal{ C}_t=\{(x,y)\in 	\R\times [0,\infty) : (t,x,y)\in \mathcal C  \}, \qquad t\in [0,T),
		$$
		and the stopping time 
		\begin{align*}
		\tau_t=\inf\{ s \geq t : (s,X_s,Y_s)\notin \mathcal{C} \}=\inf\{ s \geq t : u(s,X_s,Y_s)=\psi(s,X_s,Y_s) \}.
		\end{align*}
		Note that $u(x,X_s,Y_s)>\psi(s,X_s,Y_s)$ for $t\leq s < \tau_t$. Moreover, recall  that $(F-\lambda(1+y)u,\psi-u)=0$ a.e., so $Leb\{ (x,y)\in \mathcal{C}_t :  F-\lambda(1+y)u \neq 0 \}= 0 \, dt$ a.e.. Since the two dimensional diffusion $(X,Y)$ has a density, we deduce that  $\E\left[ F(s,X_s,Y_s)-\lambda(1+Y_s)u(s,X_s,Y_s)\ind{\{(X_s,Y_s)\in \mathcal C_s \}}  \right]\\=0$, and so $F(s,X_s,Y_s)-\lambda(1+Y_s)u(s,X_s,Y_s) = 0 $ $ds,\ d\P- a.e. $ on $\{s<\tau_t\}$. Therefore,
		$$
		\E\left[    u(\tau_t,X_{\tau_t},Y_{\tau_t}) \right]= 	\E\left[    u(t,X_t,Y_t) \right],
		$$
		and, since $u(\tau_t,X_{\tau_t},Y_{\tau_t}) = \psi(\tau_t,X_{\tau_t},Y_{\tau_t})$ thanks to the continuity of $u$ and $\psi$,
		$$
		E\left[  u(t,X_t,Y_t) \right]= 	\E\left[    \psi(\tau_t,X_{\tau_t},Y_{\tau_t}) \right]\leq 	 	\E\left[    u^*(t,X_{t},Y_{t}) \right],
		$$
		so that $u(t,X_t,Y_t) = u^*(t,X_t,Y_t)  $  a.e..  With the same arguments  we can prove that $u(t,x,y)=u^*(t,x,y)$ and this concludes the proof.
	\end{proof}
	
	\subsubsection {Weaker assumptions on $\psi$}
	The last step is to establish the equality $u=u^*$ under weaker assumptions on $\psi$, so proving Theorem \ref{theorem2}. 
	\begin{proof}[Proof of Theorem \ref{theorem2}]
		First assume that there exists a sequence $(\psi_n)_ {n	\in \N}$ of continuous functions on $[0,T]	\times \R\times[0,\infty)$ which converges uniformly to $\psi$ and such that, for each $n\in \N$,  $\psi_n$ satisfies the assumptions of Proposition \ref{identification}.
		% satisfies Assumption $\mathcal{H}^*$ with the same bounds \eqref{boundonpsi} and \eqref{boundonderivatives} of $\psi$, $\psi_n\in L^2([0,T];H^2(\O,\m))$ and $\frac{\partial \psi}{\partial t}+\L\psi\in L^p([0,T];L^p(\O,\m))$ for some  fixed $p>\beta+\frac{5}{2}$. 
		For every $n\in \N$,  we set  $u_n=u_n(t,x,y)$     the  unique solution of the variational inequality \eqref{variationalinequality} with final condition $u_n(T,x,y)=\psi_n(T,x,y)$ and $u_n^*(t,x,y)= \sup_{\tau\in\T_{t,T}}\E[\psi_n(\tau,X_\tau^{t,x,y},Y^{t,y}_\tau)]$. Then, thanks to Proposition \ref{identification}, for every $n\in\N$ we have
		$$
		u_n(t,x,y)=u_n^*(t,x,y) \qquad \mbox{ on } [0,T]\times \bar{\mathcal{O}}.
		$$
		Now, the left hand side  converges to $u(t,x,y)$ thanks to the Comparison Principle. As regards the right hand side, 
		$$
		\sup_{\tau \in \mathcal{T}_{t,T}}\E \left[   \psi_n(\tau,X_\tau^{t,x,y}, Y_\tau^{t,x,y})     \right]\rightarrow \sup_{\tau \in \mathcal{T}_{t,T}} \E \left[ e^{-r(\tau-t)}  \psi(\tau,X_\tau^{t,x,y}, Y_\tau^{t,x,y})     \right]
		$$
		thanks to the uniform convergence of $\psi_n$ to $\psi$.

		Therefore, it is enough to prove that, if $\psi$ satisfies Assumption  $\mathcal{H}^*$, then it is the uniform limit of  a sequence of functions $\psi_n$ which satisfy the assumptions of Proposition \ref{identification}. 
		This can be done following the very same arguments of \cite[Lemma 3.3]{JLL} so we omit  the technical details (see \cite{T}).
	\end{proof}
	
	%		  	\section{The European options case}
	%		  		So far we have  studied the analytical characterization of the value function of an American option in the Heston model. Recall that the price at time $t$ of a European option with payoff function $\psi$	is given by $P_e(t,X_t,Y_t)$, with 
	%	$$
	%P_e(t,x,y)=\E[e^{-r(T-t)}\psi(T,X^{t,x,y}_T,Y_T^{t,y})].
	%	$$  
	%The analytical characterization of the function $P_e$ as  the unique solution of the associated variational equation  follows along the lines of 
	%the proof of Theorem \ref{variationalinequality} and Theorem \ref{theorem2}. Therefore, we can easily prove the following result.
	%\begin{theorem}
	%	Assume that $\psi\in C([0,T];H)$, $\psi(T)\in V$. Moreover,  $0\leq \psi\leq \Phi$, where $\Phi$ satisfies Assumption $\mathcal H_2$. Then, there exists a unique function $u$ such that $	u\in  L^2([0,T];V),\,\frac{\partial u}{\partial t } \in L^2([0,T];  H)$ and 
	%	\begin{equation}\label{VE}
	%	\begin{cases}
	%	-\left( \frac{\partial u}{\partial t },v   \right)_H + a(u,v)= 0, \quad \mbox{a.e. in } [0,T] \quad v\in L^2([0,T];V),\\
	%	u(T)=\psi(T),\\
	%	0\leq u \leq \Phi.
	%	\end{cases}
	%	\end{equation}
	%	Moreover, assume that $\psi$ satisfies Assumption $\mathcal{H}^*$.
	%	Then, the solution $u$ of the variational inequality \eqref{VE} associated with $\psi$ is continuous and coincides with  the function $u^*$ defined by 
	%	\begin{equation*}\label{americanprice}
	%	u^*(t,x,y)= \E \left[   \psi(T,X_T^{t,x,y}, Y_T^{t,y})     \right].
	%	\end{equation*}
	%\end{theorem}
	%		
	%		
	
	\section{Appendix: Proof of Proposition \ref{propsg2}}\label{sect-appendix-var}
	The proof of  Proposition  \ref{propsg2}  can be carried out following the very same lines of the proof of Proposition \ref{coercive variational inequality}.  For this reason, we retrace here only the main steps of the proof.
	So, the first step is to solve the following truncated coercive problem.
	\begin{proposition}\label{propsg1}
		Assume $\lambda\geq \frac{\delta_1}{2}+\frac{K_1^2}{2\delta_1}$. For every $\psi\in V$,
		$f\in L^2_{loc}(\R^+,H)$
		and $M>0$, there exists a unique function $u^{(M)}\in L^2_{loc}(\R^+,V)$, such that
		$u^{(M)}_t\in L^2_{loc}(\R^+,H)$, $u^{(M)}(0)=\psi$ and
		\[
		(u^{(M)}_t,v)_H+a^{(M)}_\lambda(u^{(M)},v)=(f,v)_H,\qquad v\in V.
		\]
		Moreover, for every $t\geq 0$,
		\begin{equation}
		\|u^{(M)}(t)\|_H^2+\frac{\delta_1}{2}\int_0^t\|u^{(M)}(s)\|^2_Vds\leq \|\psi \|^2_H
		+\frac{2}{\delta_1} \int_0^t \|f(s)\|_H^2ds \label{estimL2}
		\end{equation}
		and
		\begin{equation}\label{estimL2bisM}
		\begin{split}
		\frac{1}{2}& \int_0^t\|u^{(M)}_t(s)\|^2_Hds+\frac{\delta_1}{4}\|u^{(M)}(t)\|^2_V \\&\leq \frac{1}{2}\bar{a}_\lambda(\psi,\psi)+\frac{1}{2}\int_0^t\|f(s)\|^2_H
		ds
		+K_1\int_0^t ds 
		\int\!\!\int y\wedge M |\nabla u^{(M)}(s)\|u^{(M)}_t(s)| d\m.
		\end{split}
		\end{equation}
	\end{proposition}
	\begin{proof}
		Fix $\psi\in V$ and $f\in L^2_{loc}(\R^+,H)$. 	Let $(V_j)_{j}$ be an increasing sequence of subspaces of $V$ with finite dimension such that $\bigcup_j V_j$ is dense in $V$ and $\psi\in V_0$ . For every $j$,  denote by $u_j$ the unique solution of the differential equation
		\[
		\left(\frac{\partial u_j}{\partial t},v\right)_H+a^{(M)}_\lambda(u_j,v)=(f,v)_M,\qquad v\in V_j,
		\]
		with $u_j(0)=\psi$. 
		
		Taking $v=u_j$ and using the inequality $a^{(M)}_\lambda(u,u)\geq \frac{\delta_1}{2}\|u\|_V$, we get
		\begin{eqnarray*}
			\left(\frac{\partial u_j}{\partial t},u_j\right)_H+a^{(M)}_\lambda(u_j,u_j)&=&(f,u_j)_H\\
			\frac{1}{2}\frac{d}{dt}\|u_j(t)\|_H^2+a^{(M)}_\lambda(u_j(t),u_j(t))&=&(f(t),u_j(t))_H\\
			\frac{1}{2}\frac{d}{dt}\|u_j(t)\|_H^2+\frac{\delta_1}{2}\|u_j(t)\|_V^2&\leq &(f(t),u_j(t))_H.
		\end{eqnarray*}
		Integrating between  $0$ and $t$, we get
		\begin{eqnarray*}
			\frac{1}{2}\|u_j(t)\|_H^2+\frac{\delta_1}{2}\int_0^t \|u_j(s)\|_V^2ds&\leq &\frac{1}{2}\|\psi\|_H^2+ \int_0^t \|f(s)\|_H\|u_j(s)\|_Hds.
		\end{eqnarray*}
		So, if $f=0$,
		\begin{eqnarray*}
			\|u_j(t)\|_H^2+\delta_1\int_0^t \|u_j(s)\|_V^2ds&\leq &\|\psi\|_H^2,
		\end{eqnarray*}
		and, for $f\neq 0$,
		\begin{eqnarray*}
			\frac{1}{2}\|u_j(t)\|_H^2+\frac{\delta_1}{2}\int_0^t \|u_j(s)\|_V^2ds&\leq &\frac{1}{2}\|\psi\|_H^2+
			\frac{\delta_1}{4}\int_0^t\|u_j(s)\|_H^2ds + \frac{1}{\delta_1}\int_0^t \|f(s)\|^2_Hds.
		\end{eqnarray*}
		Therefore,
		\begin{eqnarray*}
			\frac{1}{2}\|u_j(t)\|_H^2+\frac{\delta_1}{4}\int_0^t \|u_j(s)\|_V^2ds&\leq &\frac{1}{2}\|\psi\|_H^2+
			\frac{1}{\delta_1}\int_0^t \|f(s)\|^2_Hds.
		\end{eqnarray*}
		By taking $v=\partial u_j/\partial t$, we get, using the symmetry of $\bar{a}_\lambda$,
		\begin{eqnarray*}
			\left\|\frac{\partial u_j}{\partial t}\right\|^2_H+a^{(M)}_\lambda\left(u_j,\frac{\partial u_j}{\partial t}\right)&=&\left(f,\frac{\partial u_j}{\partial t}\right)_H\\
			\left\|\frac{\partial u_j}{\partial t}\right\|^2_H+\bar{a}_\lambda \left(u_j,\frac{\partial u_j}{\partial t}\right)
			+\tilde a^{(M)}\left(u_j,\frac{\partial u_j}{\partial t}\right)&=&\left(f,\frac{\partial u_j}{\partial t}\right)_H\\
			\left\|\frac{\partial u_j}{\partial t}\right\|^2_H+\frac{1}{2}\frac{d}{dt}\bar{a}_\lambda \left(u_j, u_j\right)
			+\tilde a^{(M)}\left(u_j,\frac{\partial u_j}{\partial t}\right)&=&\left(f,\frac{\partial u_j}{\partial t}\right)_H,
		\end{eqnarray*}
		and, integreting from $0$ to $t$,
		\begin{align*}
		&	\int_0^t  \left\|\frac{\partial u_j}{\partial t}(s)\right\|^2_Hds+
		\frac{1}{2}\bar{a}_\lambda \left(u_j(t), u_j(t)\right)= \frac{1}{2}\bar{a}_\lambda \left(\psi,\psi\right)
		+ \int_0^t \left(f(s),\frac{\partial u_j}{\partial t}(s)\right)_Hds
		\\&\qquad	- \int_0^t\tilde a^{(M)}\left(u_j(s), \frac{\partial u_j}{\partial t}(s)\right)_Hds.
		\end{align*}
		Therefore,
		\begin{align*}
		&	 	\int_0^t  \left\|\frac{\partial u_j}{\partial t}(s)\right\|^2_Hds+
		\frac{\delta_1}{4}\left\|u_j(t)\right\|_V^2\\&\leq \frac{1}{2}\bar{a}_\lambda \left(\psi,\psi\right)
		+ \int_0^t \left(f(s),\frac{\partial u_j}{\partial t}(s)\right)_Hds
		+K_1\int_0^t ds \into y\wedge M |\nabla u_j(s,.)|\left|\frac{\partial u_j}{\partial t}(s, .)\right|d\m\\
		&\leq \frac{1}{2}\bar{a}_\lambda \left(\psi,\psi\right)
		+ \int_0^t \|f(s)\|_H\left\|\frac{\partial u_j}{\partial t}(s)\right\|_Hds\\&\qquad
		+\int_0^t ds\into\!\left( \frac{K_1y}{2\zeta} |\nabla u_j(s,.)|^2+
		\frac{K_1M\zeta}{2} \left|\frac{\partial u_j}{\partial t}(s, .)\right|^2\right)d\m\\
		&\leq \frac{1}{2}\bar{a}_\lambda \left(\psi,\psi\right)
		+ \int_0^t \|f(s)\|_H\left\|\frac{\partial u_j}{\partial t}(s)\right\|_Hds
		+\frac{K_1}{2\zeta}\int_0^t  \|u_j(s)\|_V^2ds+
		\frac{K_1M}{2}\zeta \int_0^t \!  \left\|\frac{\partial u_j}{\partial t}(s)\right\|^2_H\!ds.         
		\end{align*}
		Then the assertion follows by passing to the limit as $j$ tends to infinity and by using the estimates above.
	\end{proof}
	Then, we have the following Lemma.
	\begin{lemma} If, in addiction to the assumptions of Proposition \ref{propsg1} we also assume $\sqrt{1+y}f\in L^2_{loc}(\R^+,H)$, we have 
		\begin{eqnarray*}
			\frac{1}{4} \int_0^t\|u^{(M)}_t(s)\|^2_Hds+\frac{\delta_1}{4}\|u^{(M)}(t)\|^2_V& \leq& \frac{1}{2}\bar{a}_\lambda(\psi,\psi)+\frac{1}{2}\int_0^t\|f(s)\|^2_H
			ds\\
			&&+
			\frac{4K_1^2K_3}{\delta_1}\left( \|\sqrt{1+y}\psi\|_H^2+\int_0^t ds \|\sqrt{1+y}f(s)\|_H^2\right).
		\end{eqnarray*}
	\end{lemma}
	
	\begin{proof}
		%	By deriving the equation with respect to the time, we can easily see that,	 if	 $\partial f/\partial t\in L^2_{loc}(\R^+;H)$ (in particular if $f=0$), the solution$u$  of the equation satisfies   $\partial u/\partial t\in L^2_{loc}(\R^+;V)$.
		
		Let us denote $\phi_M(x,y)=y\wedge M$. Since $\phi_M$ and its derivatives are bounded, if $u^{(M)}\in V$, $u^{(M)}\phi_M\in V$. Then, taking  $v=u^{(M)}\phi_M$, we get
		\begin{eqnarray*}
			\left(\frac{\partial u^{(M)}}{\partial t},u^{(M)}\phi_M\right)_H+a^{(M)}_\lambda(u^{(M)},u^{(M)}\phi_M)&=&\left(f,u^{(M)}\phi_M\right)_H,
		\end{eqnarray*}
		which,  setting $\phi'_M=\partial \phi_M/\partial y$, can be rewritten as
		\begin{align*}
		& 	\into \frac{\partial u^{(M)}}{\partial t}u^{(M)}\phi_M d\m+
		\into \frac{y}{2}\left(\frac{\partial u^{(M)}}{\partial x}\frac{\partial u^{(M)}}{\partial x}+
		\sigma^2 \frac{\partial u^{(M)}}{\partial y}\frac{\partial u^{(M)}}{\partial y}+
		2\rho\sigma \frac{\partial u^{(M)}}{\partial x}\frac{\partial u^{(M)}}{\partial y}
		\right) \phi_M d\m\\&\quad
		+ \into \frac{y}{2}\left(\rho\sigma \frac{\partial u^{(M)}}{\partial x}+
		\sigma^2\frac{\partial u^{(M)}}{\partial y}\right)u^{(M)}\phi'_Md	\m+\into y\left(\frac{\partial u^{(M)}}{\partial x}j_{\gamma,\mu}+
		\frac{\partial u^{(M)}}{\partial y}k_{\gamma,\mu}\right)u^{(M)}\phi_Md\m\\&\quad
		+      \lambda \into (1+y)(u^{(M)})^2\phi_Md\m=(f,u^{(M)}\phi_M)_H.
		\end{align*}
		Then, by using $0\leq \phi'_M\leq \ind{\{y\leq M\}}$, 
		\begin{align*}
		&	\frac{1}{2}\frac{d}{dt}\into (u^{(M)})^2\phi_M d\m+
		\delta_1 \into y\left|\nabla u^{(M)}
		\right|^2 \phi_M d\m+\lambda \into (1+y)(u^{(M)})^2\phi_Md\m 
		\\
		&\quad\leq 
		(f,u^{(M)}\phi_M)_H+K_1\into y\left|\nabla u^{(M)}\right\|u^{(M)}|\phi_Md\m+
		\into \frac{y}{2}\left|\rho\sigma \frac{\partial u^{(M)}}{\partial x}+
		\sigma^2\frac{\partial u^{(M)}}{\partial y}\right\|u^{(M)}|\phi'_Md\m\\
		&\quad \leq
		(f,u^{(M)}\phi_M)_H+K_1\into y\left|\nabla u^{(M)}\right\|u^{(M)}|\phi_Md\m+
		\frac{\sqrt{\rho^2\sigma^2+\sigma^4}}{2}\into y\wedge M\left|\nabla u^{(M)}\right| |u^{(M)} |d\m\\
		&\quad \leq 
		(f,u^{(M)}\phi_M)_H+\frac{K_1\zeta}{2}\into y\left|\nabla u^{(M)}\right|^2\phi_Md\m
		+\frac{K_1}{2\zeta}\into y\left| u^{(M)}\right|^2\phi_Md\m\\
		&\qquad+
		\frac{\sqrt{\rho^2\sigma^2+\sigma^4}}{2}\into y\wedge M\left|\nabla u^{(M)}\right| |u^{(M)} |d\m.
		\end{align*}
		By taking $\zeta=\delta_1/K_1$ and noting that  $\into y\wedge M\left|\nabla u^{(M)}\right| |u^{(M)} |d\m\leq \|u^{(M)}\|_V^2$, we get
		\begin{align*}
		&\frac{1}{2}\frac{d}{dt}\into (u^{(M)})^2\phi_M d\m+
		\frac{\delta_1}{2} \into y\left|\nabla u^{(M)}
		\right|^2 \phi_M d\m+\left(\lambda -\frac{K_1^2}{2\delta_1}\right)\into (1+y)(u^{(M)})^2\phi_Md\m\\ &\quad \leq
		(f,u^{(M)}\phi_M)_H       +K_2\|u^{(M)}\|_V^2
		\end{align*}
		with $K_2=\frac{\sqrt{\rho^2\sigma^2+\sigma^4}}{2}$
		and, by using  $\lambda\geq \frac{\delta_1}{2}+\frac{K_1^2}{2\delta_1}$ and integrating from $0$ to $t$,
		\begin{align*}
		&	\frac{1}{2}\into (u^{(M)})^2(t,.)\phi_M d\m+
		\frac{\delta_1}{2}\int_0^t ds
		\into \left(y\left|\nabla u^{(M)}(s)
		\right|^2 +(1+y)(u^{(M)})^2(s)\right)\phi_Md\m\\&\quad \leq
		\int_0^t  (f(s),u^{(M)}(s)\phi_M)_Hds
		+\frac{1}{2}\into \psi^2\phi_M d\m
		+K_2\int_0^t ds\|u^{(M)}(s) \|_V^2d\m.
		\end{align*}
		We have, for every $\zeta>0$,
		\begin{equation*}
		\int_0^t (f(s),u^{(M)}(s)\phi_M)_Hds\leq
		\frac{\zeta}{2}\int_0^t ds\into \phi_M\left|u^{(M)}(s)\right|^2d\m
		+\frac{1}{2\zeta}\int_0^t ds\into \phi_M\left|f(s)\right|^2d\m
		\end{equation*}
		and, taking $\zeta=\delta_1/2$,
		\begin{align*}
		&	\frac{1}{2}\into (u^{(M)})^2(t,.)\phi_M d\m+
		\frac{\delta_1}{4}\int_0^t ds
		\into \left(y\left|\nabla u^{(M)}(s)
		\right|^2 +(1+y)(u^{(M)})^2(s)\right)\phi_Md\m\\ & \quad \leq
		\frac{1}{\delta_1}\int_0^t ds\into \phi_M\left|f(s)\right|^2d\m
		+\frac{1}{2}\into \psi^2\phi_M d\m
		+K_2\int_0^t \|u^{(M)}(s) \|_V^2ds.
		\end{align*}
		Then, by using \eqref{estimL2},
		\begin{align*}
		&	\frac{1}{2}\into (u^{(M)})^2(t,.)\phi_M d\m+
		\frac{\delta_1}{4}\int_0^t ds
		\into \left(y\left|\nabla u^{(M)}(s)
		\right|^2 +(1+y)(u^{(M)})^2(s)\right)\phi_Md\m\\&\quad\leq
		\frac{1}{\delta_1}\int_0^t ds\into \phi_M\left|f(s)\right|^2d\m
		+\frac{1}{2}\into \psi^2\phi_M d\m
		+\frac{2K_2}{\delta_1}\|\psi \|_H^2+\frac{4K_2}{\delta_1^2}\int_0^t\|f(s)\|_H^2  ds\\
		\quad	&\quad \leq K_3\left( \|\sqrt{1+y}\psi\|_H^2+\int_0^t ds \|\sqrt{1+y}f(s)\|_H^2\right),
		\end{align*}
		where $K_3=\max\left(\frac{1}{\delta_1}, \frac{1}{2}, \frac{2K_2}{\delta_1},\frac{4K_2}{\delta_1^2}\right)$.
		Note that $K_3$ does not depend on $M$. We deduce from the last inequality that
		\begin{align*}
		\int_0^t ds
		\into \left|\nabla u^{(M)}(s)
		\right|^2 \phi_M^2d\m\leq \frac{4K_3}{\delta_1}\left( \|\sqrt{1+y}\psi\|_H^2+\int_0^t ds \|\sqrt{1+y}f(s)\|_H^2\right)
		\end{align*}
		and, by using \eqref{estimL2bisM},
		\begin{align*}
		&	\frac{1}{2} \int_0^t\|u^{(M)}_t(s)\|^2_Hds+\frac{\delta_1}{4}\|u^{(M)}(t)\|^2_V\\& \leq \frac{1}{2}\bar{a}_\lambda(\psi,\psi)+\frac{1}{2}\int_0^t\|f(s)\|^2_H
		ds     
		+K_1\int_0^t ds 
		\into y\wedge M |\nabla u^{(M)}(s)\|u^{(M)}_t(s)| d\m\\
		&\leq
		\frac{1}{2}\bar{a}_\lambda(\psi,\psi)+\frac{1}{2}\int_0^t\!\|f(s)\|^2_H
		ds     
		+\frac{K_1\zeta}{2}\int_0^t \!ds \!
		\into  \!\!|u^{(M)}_t(s)|^2 d\m+
		\frac{K_1}{2\zeta}\int_0^t \!ds 
		\!\!	\into \phi_M^2 |\nabla u^{(M)}(s)|^2 d\m
		\end{align*}
		By taking $\zeta=1/(2K_1)$, we get
		\begin{align*}
		&\qquad	\frac{1}{4} \int_0^t\|u^{(M)}_t(s)\|^2_Hds+\frac{\delta_1}{4}\|u^{(M)}(t)\|^2_V \\& \leq \frac{1}{2}\bar{a}_\lambda(\psi,\psi)+\frac{1}{2}\int_0^t\|f(s)\|^2_H
		ds+	\frac{4K_1^2K_3}{\delta_1}\left( \|\sqrt{1+y}\psi\|_H^2+\int_0^t ds \|\sqrt{1+y}f(s)\|_H^2\right).
		\end{align*}
	\end{proof}
	Now,  in order to prove Proposition \ref{propsg2}, it is enough to let $M$ go to infinity.
	
	%			\begin{remark}
	%				Note that, if in the statement of Proposition  \ref{propsg2} we chose $f=0$, we replace     the interval $(0,\infty)$ with the finite interval $(0,T)$ and we do the usual change of variables in time $t\mapsto T-t$, then we get the existence and uniqueness of solutions of the (coercive) variational equality associated with the European option price function with maturity $T$ and payoff function $\psi$. Therefore, we can get the existence and uniqueness of the solution of the non-coercive variational equality by using the same technique used in the proof of Proposition \ref{variationalinequality}. Then, we can identify  the European price with this solution by using similar arguments to the one used for the American case.  
	%			\end{remark}
	%	\end{subappendices}
	\chapter{American option price properties in Heston type models}\label{chapter-art2}

	\section{Introduction}

	One of the strengths of the Black and Scholes type models relies in their analytical tractability. 
	A large number of papers have been devoted to the pricing of European and American options and to the study of the regularity properties of the price in this framework.  
	
	Things become more complicated in the case of stochastic volatility models. Some properties of European options were studied, for example, in \cite{Oa} but if we consider American options, as far as we know, the existing literature is rather poor. One of the main reference  is a paper by Touzi \cite{T}, in which  the author studies some  properties of a standard American put option in a class of stochastic volatility models under  classical assumptions, such as the  uniform ellipticity of the model. 
	
	However, the assumptions in \cite{T}  are not satisfied by the well known  Heston  model because of its degenerate nature and some of the analytical techniques used in \cite{T} cannot be directly applied. 
	
	This chapter, which is extracted from \cite{LT2}, is devoted to the study of some properties of the American option price in the Heston model. Our main aim is to extend some well known results in the Black and Scholes world to the Heston type stochastic volatility models. We do it  mostly by using probabilistic techniques.
	
	In more details, the chapter is organized as follows. In Section \ref{sect-cap2model} we set up our new notation. In Section \ref{sect-monotony}, we prove that, if the payoff function is convex and satisfies some regularity assumptions, the American option value function is increasing with respect to the volatility variable. This topic was already addressed in \cite{AOJ} with an elegant probabilistic approach, under the assumption that the coefficients of the model satisfy the well known Feller condition. Here, we prove it without imposing  conditions on the coefficients.
	
	Then, in Section \ref{sect-put} we focus on the standard American put option. We first generalise to the Heston model the well known notion of  critical price or exercise boundary  and we study some properties of this function. Then we prove that the American option price is strictly convex in the continuation region with respect to the stock price. This result was already proved in \cite{T} for uniformly elliptic stochastic volatility  by using PDE techniques.  Here, we extend the result to the degenerate Heston model by using a probabilistic approach.  We also give an explicit formulation of the early exercise premium, that  is the difference in price
	between an American option and an otherwise identical European option, and we do it by using  results first  introduced in \cite{J}. Finally,  we  provide  a weak formulation of the so called smooth fit property. 
	The chapter ends with an appendix, which is devoted to the proofs of some technical results.
	\section{Notation}\label{sect-cap2model}
	Recall that in the Heston model we have
	\begin{equation}\label{hest2volta}
	\begin{cases}
	\frac{dS_t}{S_t}=(r-\delta)dt + \sqrt{Y_t}dB_t,\qquad&S_0=s>0,\\ dY_t=\kappa(\theta-Y_t)dt+\sigma\sqrt{Y_t}dW_t, &Y_0=y\geq 0,
	\end{cases}
	\end{equation}
	where $B$ and $W$ denote two correlated Brownian motions with correlation coefficient  $ \rho \in (-1,1).
	$	
	Through this chapter we denote by $\L$ the infinitesimal generator of the pair $(S,Y)$, that is the differential operator given by
	\begin{equation}\label{L}
	\L= \frac y 2 \left(s^2 \frac{\partial^2}{\partial s^2 } +2s\rho \sigma  \frac{\partial^2}{\partial s \partial y }+ \sigma^2  \frac{\partial^2}{\partial y^2 }\right)+ \left(r-\delta \right)s\frac{\partial}{\partial s} +\kappa(\theta - y)\frac{\partial}{\partial y}.
	\end{equation}
	Let  $(S^{t,s,y}_u,Y^{t,y}_u)_{u\in [t,T]}$ be  the solution of \eqref{hest2volta} which starts  at time $t$ from the position  $(s,y)$. When the initial time is $t=0$ and there is no ambiguity, we will often write $(S^{s,y}_u,Y^{y}_u)$ or directly $(S_u,Y_u)$  instead of $(S^{0,s,y}_u,Y^{0,y}_u)$.
	We recall that the price of an American option with a nice enough payoff $(\varphi(S_t))_{t\in [0,T]} $ and maturity $T$ is given by $P_t=P(t,S_t,Y_t)$, where 
	\begin{equation*}
	P(t,s,y)=\sup_{\tau \in\mathcal{T}_{t,T}}\E[e^{-r(\tau-t )}\varphi(S^{t,s,y}_\tau)],
	\end{equation*}
	$\mathcal{T}_{t,T}$ being  the set of the stopping times with values in $[t,T]$.
	
	It  will be useful in this chapter to consider  the log-price process, so we set  $X_t=\log S_t$. In this case, recall that the pair $(X,Y)$  evolves according to
	\begin{equation}
	\begin{cases}\label{hest2}
	dX_t=\left(r-\delta-\frac{1}{2}Y_t\right)dt + \sqrt{Y_t}dB_t, \qquad&X_0=x=\log s\in\R,\\ dY_t=\kappa(\theta-Y_t)dt+\sigma\sqrt{Y_t}dW_t, &Y_0=y\geq 0,
	\end{cases}
	\end{equation}
	%with $X_0=x=\log s,Y_0=y >0$. 
	%fixed and $W$ and $B$ denote two correlated Brownian motions with 
	%$$
	%d\left\langle W,B \right\rangle_t=\rho dt, \qquad \rho \in (-1,1).
	%$$
	%We assume that $\kappa, \ \theta $ and $\sigma$ are positive constants, so that there exist a unique continuous solution of the SDE \eqref{hest2volta}.  
	and has infinitesimal generator given by
	\begin{equation}\label{Llog}
	\tilde {\mathcal{L}}= \frac y 2 \left( \frac{\partial^2}{\partial x^2 } +2\rho \sigma  \frac{\partial^2}{\partial x \partial y }+ \sigma^2  \frac{\partial^2}{\partial y^2 }\right)+ \left(r-\delta-\frac y 2 \right)\frac{\partial}{\partial x} +\kappa(\theta - y)\frac{\partial}{\partial y}.
	\end{equation}
	%	 Note that $\tilde \L$ has unbounded coefficients and it is not uniformly elliptic: it degenerates on the boundary of the definition set $\O=\R\times (0,\infty)$, that is when $y=0$.
	
	With this change of variables, the American option price function is given by $u(t,x,y)=P(t,e^x,y)$, which can be rewritten as
	\begin{equation*}
	u(t,x,y)=\sup_{\tau \in \mathcal{T }_{t,T}}\E[e^{-r(\tau-t)}\psi(X^{t,x,y}_\tau)],
	\end{equation*}
	where $\psi(x)=\varphi(e^x)$.

	\section{Monotonicity with respect to the volatility}\label{sect-monotony}
	In this section we  prove the increasing feature of the option price with respect to the volatility variable under the assumption that the payoff function $\varphi$ is convex and satisfies some regularity properties. The same topic was addressed by Touzi in \cite{T}  for uniformly elliptic stochastic volatility models  and by Assing \textit{et al.} \cite{AOJ} for a class of models which includes the Heston model when the Feller condition is satisfied.
	
	For convenience we pass to the logarithm in the $s-$variable and we study the monotonicity of the function $u$. Note that the convexity assumption on the payoff function $\varphi\in  C^2(\R)$ corresponds to the condition $ \psi''-\psi'\geq 0$ for the function $\psi(x)=\varphi(e^x)$.
	
	Let us recall some standard notation. For $\gamma >0$ we introduce the following weighted Sobolev spaces 
	$$
	L^2(\R,e^{-\gamma|x|})=\left \{ u:\R\rightarrow\R:\|u\|_2^2=\int u^2(x)e^{-\gamma|x|}dx<\infty \right   \},
	$$
	$$
	W^{1,2}(\R,e^{-\gamma|x|})=\left \{ u\in L^2(\R,e^{-\gamma|x|}): \frac{\partial u }{\partial x}\in L^2(\R,e^{-\gamma|x|}) \right   \},
	$$
	$$
	W^{2,2}(\R,e^{-\gamma|x|})= \left  \{ u\in L^2(\R,e^{-\gamma|x|}): \frac{\partial u }{\partial x},\frac{\partial^2 u }{\partial x^2}\in L^2(\R,e^{-\gamma|x|})   \right  \}.
	$$
	\begin{theorem}\label{monotonie}
		Let  $\psi$ be a bounded function such that $\psi\in W^{2,2}(\R,e^{-\gamma |x|})\cap  C^2(\R)$ and   $\psi''-\psi'\geq 0$.
		Then the value function $u$ is nondecreasing with respect to the volatility variable.
		%, that is $\frac{\partial u}{\partial y}\geq 0$.
	\end{theorem}
	In order to prove Theorem \ref{monotonie}, let us consider   a smooth approximation $f_n\in C^\infty(\R)$ of the function $f(y)=\sqrt{y^+}$, such that $f_n$ has  bounded derivatives, $1/n\leq f_n\leq n$, $f_n(y)$ is increasing in $y$, $f_n^2$ is Lipschitz continuous uniformly in $n$ and $f_n\rightarrow f$ locally uniformly  as $n\rightarrow \infty$. 
	%Moreover, assume that there exists a constant $A>0$ such that $f_n(x)\leq A(1+|x|)$.
	
	Then, we consider  the sequence of SDEs
	\begin{equation}
	\begin{cases}\label{hestapprox}
	dX^n_t=\left(r-\delta-\frac{f_n^2(Y^n_t)} 2 \right) dt + f_n(Y^n_t)dB_t,\qquad &X^n_0=x,\\ dY^n_t=\kappa\left(\theta-f_n^2(Y^n_t)\right)dt+\sigma f_n(Y^n_t)dW_t, &Y_0^n=y.
	\end{cases}
	\end{equation}
	Note that, for every $n\in\N$, the diffusion matrix $a_n(y)=\frac 1 2 \Sigma_n(y) \Sigma_n(y)^t$, where 
	$$
	\Sigma_n(y) = \left(
	\begin{array}{cc}
	\sqrt{1-\rho^2}	f_n(y) &  \rho f_n(y) \\ 
	0 &  \sigma f_n(y)
	\end{array} 	\right),
	$$
	is uniformly elliptic.
	For any fixed $n\in\N$ the infinitesimal generator of the diffusion $(X^n,Y^n)$ is given by
	$$
	\tilde{ \mathcal{L}}^n= \frac {f_n^2(y)} 2\left( \frac{\partial^2}{\partial x^2} +2\rho\sigma \frac{\partial^2u}{\partial x \partial y}+\sigma^2 \frac{\partial^2}{\partial y^2}\right) +\left(r-\delta-\frac{f_n^2(y)} 2 \right) \frac{\partial }{\partial x} + \kappa\left(\theta-f_n^2(y)\right)\frac{\partial }{\partial y}
	$$
	and it is uniformly elliptic with  bounded coefficients.
	
	We will need the following result.
	\begin{lemma}\label{lemmasup_diff}
		For any $\lambda>0$, we have\begin{equation}
		\label{1} \lim_{n \rightarrow \infty} \P\left( \sup_{t \in [0,T]}|X^n_t-X_t|\geq \lambda \right)=0 \end{equation}  and\begin{equation}
		\label{2}
		\lim_{n \rightarrow \infty} \P\left(\sup_{t \in [0,T]}|Y^n_t-Y_t|\geq \lambda \right)=0 .\end{equation}  
	\end{lemma}
	The proof is inspired by the proof of uniqueness of the solution for the CIR process (see \cite[Section IV.3]{IW}). We postpone it to the Appendix.
	
	From now on, let us set $\E_{x,y}[\cdot]=\E[\cdot|(X_0,Y_0)=(x,y)]$.  For every $n \in \N$,	we consider the American value function with payoff $\psi$ and underlying diffusion $(X^n,Y^n)$,  that is
	$$
	u^n(t,x,
	y)=   \sup_{\tau \in \mathcal{T}_{0,T-t}} \E_{x,y} \left[   e^{-r\tau} \psi(X_\tau^{n})       \right],\qquad  (t,x,y) \in [0,T]\times \R \times [0,\infty).
	$$
	We prove that $u^n$ is actually an approximation of the function $u$, at least for bounded continuous payoff functions.
	\begin{proposition}\label{convP}
		Let $\psi$ be a bounded continuous function.
		Then, 
		$$
		\lim_{n \rightarrow \infty }|	u^n(t,x,y) -u(t,x,y) |=0,\qquad  (t,x,y) \in [0,T]\times \R \times [0,\infty).
		$$ 
	\end{proposition}
	\begin{proof}
		For any $\lambda>0$,
		\begin{align*}
		\bigg|      \sup_{\tau \in \mathcal{T}_{0,T-t}}& \E_{x,y} \left[   e^{-r\tau} \psi(X_\tau^{n})       \right]-  \sup_{\tau \in \mathcal{T}_{0,T-t}} \E_{x,y} \left[   e^{-r\tau}  \psi(X_\tau)      \right] \bigg|\\
		&\leq  \sup_{\tau \in \mathcal{T}_{0,T-t}}   	\bigg|  \E_{x,y} \left[   e^{-r\tau}  (\psi(X_\tau^{n})    -\psi(X_\tau)  ) \right] \bigg|\\
		&\leq \E_{x,y} \left[ \sup_{t\in[0,T]}  | \psi(X_t^{n})    -\psi(X_t)   |   \right]\\
		& \leq   \E_{x,y} \left[ \sup_{t\in[0,T]}  | \psi(X_t^{n})    -\psi(X_t)    |  \ind{\{|X^n_t-X_t|\leq \lambda\}  }\right]+ 2 \|\psi\|_\infty \P\left(  \sup_{t\in[0,T]}  |X^n_t-X_t |>\lambda \right).
		\end{align*}
		Then the assertion easily follows using \eqref{1} and the arbitrariness of $\lambda$.
	\end{proof}
	
	We can now prove that, for every $n\in\N$,  the approximated price function $u^n$ is nondecreasing with respect to the volatility variable. 
	\begin{proposition}\label{mon2}
		Assume that $\psi \in W^{2,2}(\R,e^{-\gamma|x|}dx)\cap C^2(\R)$ and $\psi''-\psi'\geq 0$. Then  $\frac{\partial u^n}{\partial y }\geq 0$ for every $n\in\N$.
	\end{proposition}
	\begin{proof}
		Fix $n\in\N$.  We know from the classical theory of variational inequalities that
		%, under the assumption $\psi\in C(\R)\cap W^{1,2}(\R,e^{-\gamma|x|}dx)$, 
		$u^n$ is the unique solution of the associated variational inequality (see, for example, \cite{JLL}). Moreover, 
		$u^n$ is the limit of the solutions of a sequence of  penalized problems. In particular, consider a family of penalty functions $\zeta_\varepsilon:\R\rightarrow\R$  such that, for each $\varepsilon>0$, 
		$\zeta_\varepsilon$ is a  $C^2 $, nondecreasing and concave function with bounded derivatives, satisfying $\zeta_\varepsilon(u)=0$, for $u\geq 	\varepsilon$ and $\zeta_\varepsilon(0)=b$, where $b$ is such that $\tilde{\mathcal{A}}^n\psi\geq b$ with the notation $\tilde{\mathcal{A}}^n=\tilde \L^n-r$ (see the proof of Theorem 3 in \cite{damien}). Then,  there exists a sequence  $(u^n_\varepsilon)_{\varepsilon>0}$ such that $\lim_{\varepsilon\rightarrow 0 }u^n_\varepsilon=u^n$ in the sense of distributions and, for every $\varepsilon >0$, 
		\begin{equation*}
		\begin{cases}
		-\frac{\partial u^n_\varepsilon}{\partial t}- \mathcal{A}^n u^n_\varepsilon + \zeta_\varepsilon(u^n_\varepsilon-\psi)=0,\\
		u^n_\varepsilon(T)=	\psi(T).
		\end{cases}
		\end{equation*}
		In order to simplify the notation, hereafter in this proof we denote  by $u$ the function $u^n_\varepsilon$. 
		
		Recall that, from the classical theory of parabolic semilinear equations, since $\psi\in C^2(\R)$ we have that $u\in  C^{2,4}([0,T), \R\times (0,\infty))$ (here we refer, for example, to \cite{LSU}). 
		Set now $\bar u= \frac{\partial u }{\partial y }$.  Differentiating the equation satisfied by $u^n$, we get that $\bar u$ satisfies 
		\begin{equation*}
		\begin{cases}
		-\frac{\partial \bar u}{\partial t}-\bar{ \mathcal{A}}^n \bar u=  f_n(y)f_n'(y) \left(\frac{\partial^2 u}{\partial x^2}-\frac{\partial u}{\partial x}\right),\\
		\bar u(T)=0,
		\end{cases}
		\end{equation*}
		where 
		\begin{align*}
		\bar{\mathcal{A}}^n&= \frac {f_n^2(y)} 2\left( \frac{\partial^2}{\partial x^2} +2\rho\sigma \frac{\partial^2u}{\partial x \partial y}+\sigma^2 \frac{\partial^2}{\partial y^2}\right)
		+\left(r-\delta-\frac{f_n^2(y)} 2 +2\rho\sigma f_n(y)f_n'(y)\right) \frac{\partial }{\partial x} \\&\qquad +\left( \kappa\left(\theta-f_n^2(y)\right)+\sigma^2f_n(y)f'_n(y)\right)\frac{\partial }{\partial y}-2\kappa f_n(y)f_n'(y) + \zeta_\varepsilon'(u^n_\varepsilon-\psi)       -(r-\delta).
		\end{align*}
		
		By using the  Comparison principle, we deduce that, if $ f_n(y)f_n'(y) \left(\frac{\partial^2 u}{\partial x^2}-\frac{\partial u}{\partial x}\right)\geq 0 $, then $\bar u\geq 0$ and the assertion follows letting $\varepsilon$ tend to 0.
		
		Since  $f_n$ is positive and  nondecreasing, it is enough to prove that $\frac{\partial^2 u}{\partial x^2}-\frac{  \partial u}{\partial x}\geq 0$. 
		We write the equations satisfied by $u'=\frac{  \partial u}{\partial x}$ and $u''=\frac{  \partial^2 u}{\partial x^2}$. We have
		\begin{equation}\label{u'}
		\begin{cases}
		-\frac{\partial u'}{\partial t}-\tilde{\mathcal{A}}^n u'+\zeta'_\varepsilon(u-\psi)(u'-\psi')=0,\\
		u(T)=\psi,
		\end{cases}
		\end{equation} 
		and
		\begin{equation}\label{u''}
		\begin{cases}
		-\frac{\partial u''}{\partial t}-\tilde{\mathcal{A}}^nu''+\zeta''_\varepsilon(u-\psi)(u'-\psi')^2+\zeta'_\varepsilon(u-\psi)(u''-\psi'')=0,\\
		u''(T)=\psi''.
		\end{cases} 
		\end{equation}
		Using \eqref{u'} and \eqref{u''}, we get that $u''-u'$ satisfies
		\begin{equation}
		\begin{cases}
		-\frac{\partial (u''-u')}{\partial t
		}-\mathcal{A}^n(u''-u')+\zeta'_\varepsilon(u-\psi)(u''-u')=\zeta'_\varepsilon(u-\psi)(\psi''-\psi')-\zeta''_\varepsilon(u-\psi)(u'-\psi')^2,\\
		u''(T)-u'(T)=\psi''-\psi'.
		\end{cases} 
		\end{equation}
		Recall that $\psi''-\psi'\geq0$ by assumption and
		that $\zeta_\varepsilon$ is increasing and concave. Then, 
		$$
		\zeta'_\varepsilon(u-\psi)(\psi''-\psi')-\zeta''_\varepsilon(u-\psi)(u'-\psi')^2\geq 0, \quad u''(T)-u'(T)=\psi''-\psi'\geq0,
		$$
		hence, by using again the Comparison principle, we deduce that $ u''-u'\geq0$ which concludes the proof.
	\end{proof}
	The proof of Theorem \ref{monotonie} is now almost immediate.
	\begin{proof}[Proof of Theorem \ref{monotonie}]
		Thanks to Proposition \ref{mon2}, the function $u^n$ is increasing in the $y$ variable for all $n\in\N$. 	Then, the assertion follows by using Proposition \ref{convP}.
		%Therefore, under the assumptions of Proposition \ref{mon2}, we have $\frac{\partial u }{\partial y}\geq 0$. Finally, it is enough to approximate the payoff function $\psi$ by a sequence $(\psi_m)_m $ of   
	\end{proof}
	
	\section{The American put price}\label{sect-put}
	From now on we focus our attention on the standard put option with strike price $K$ and maturity $T$, that is we fix $\varphi(s)=(K-s)_+$ and we study the properties of the function
	\begin{equation}\label{put-price}
	P(t,s,y)=\sup_{\tau \in\mathcal T_{t,T}}\E[e^{-r(\tau-t )}(K-S^{t,s,y}_\tau)_+].
	\end{equation} 
	The following result easily follows from \eqref{put-price}.
	\begin{proposition}\label{properties_P}
		The price function $P$ satisfies:
		\begin{enumerate}
			\item $(t,s,y)\mapsto P(t,s,y)$ is continuous and positive;
			\item $t\mapsto P(t,s,y)$ is nonincreasing;
			\item $y\mapsto P(t,s,y)$ is nondecreasing;
			\item $s\mapsto P(t,s,y)$ is nonincreasing and convex.
		\end{enumerate}
	\end{proposition}
	\begin{proof}
		The proofs of $1.$ and $2.$ are classical and straightforward.  As regards $3.$, we
		note that $\varphi$ is convex and  the function $\psi(x)=(K-e^x)_+$ belongs to the space $W^{1,2}(\R,e^{-\gamma|x|})$ for a $\gamma>1$ but it is not regular enough to  apply  Proposition \ref{monotonie}. However, we can use an approximation procedure. Indeed, thanks to density results and \cite[Lemma 3.3]{JLL}, we can  approximate the function $\psi$ with a sequence of functions $\psi_n\in W^{2,2}(\R,e^{-\gamma|x|})\cap C^2(\R)$ such that $\psi_n''-\psi_n'\geq0$, so the assertion easily follows passing to the limit.   $4.$  follows from the fact that $\varphi(s)=(K-s)_+$ is nonincreasing and convex.
	\end{proof}

	Moreover, thanks to  the Lipschitz continuity of the payoff function, we have the following result.
	\begin{proposition}	\label{prop-lipschitz}
		The function $x \mapsto u(t,x,y)$ is Lipschitz continuous while the function $y \mapsto u(t,x,y)$ is H\"{o}lder continuous. If $2\kappa\theta \geq \sigma^2$ the function $y\mapsto u(t,x,y)$ is locally Lipschitz continuous on $(0,\infty)$. 
	\end{proposition}
	\begin{proof}
		It is easy to  prove that, for every fixed $t\geq 0$ and $y,y'\geq 0$ with $y\geq y'$, 
		\begin{equation}\label{flow-y}
		\E\left[Y^{y}_t-Y^{y'}_t\right] \leq y-y'. 
		\end{equation}
		Then, for $(x,y), (x',y')\in\R\times [0,\infty)$ we have
		\begin{align*}
		&	|u(t,x,y)-u(t,x',y')| = \left|    \sup_{\theta \in\mathcal{T}_{t,T}}\E[e^{-r(\theta-t )}(K-e^{X^{t,x,y}_\theta})_+]-  \sup_{\theta \in \mathcal{T}_{t,T}}\E[e^{-r(\theta-t )}(K-e^{X^{t,x',y'}_\theta})_+] \right|\\&\quad
		\leq     \sup_{\theta \in\mathcal{T}_{t,T}}\left|  \E\Big[e^{-r(\theta-t )}(K-e^{X^{t,x,y}_\theta})_+-e^{-r(\theta-t )}(K-e^{X^{t,x',y'}_\theta})_+  \Big]  \right| \\
		&\quad\leq C  \E \left[  \sup_{u \in [t,T]}|X_u^{t,x,y}      -X_u^{t,x',y'}| \right] \\
		&\quad \leq C\left( |x-x'|+\int_t^T\E [|Y^{t,y}_u-Y^{t,y'}_y|]du + \E\left[ \sup_{s\in[t,T]} \left|  \int_t^s (\sqrt{Y^{t,y}_u}-\sqrt{Y^{t,y'}_u})dW_u   \right| \right]\right)\\
		& \quad \leq C\left( |x-x'|+\int_t^T\E[ |Y^{t,y}_u-Y^{t,y'}_y|]du +\left(\E\left[ \sup_{s\in[t,T]} \left|  \int_t^s  (\sqrt{Y^{t,y}_u}-\sqrt{Y^{t,y'}_u})dW_u   \right| \right]^2\right)^{\frac 1 2} \right)\\
		&\quad \leq  C\left( |x-x'|+\int_t^T\E[|Y^{t,y}_u-Y^{t,y'}_u| ] du +\left(\E\left[  \int_t^T (Y^{t,y}_u-Y^{t,y'}_u)du    \right]\right)^{\frac 1 2} \right)\\
		&\quad \leq C_T(|x-x'|+\sqrt{|y-y'|}).
		\end{align*}
		Now, recall that, if $2\kappa \theta \geq \sigma^2$, the  volatility process $Y$ is strictly positive so we can apply It\^o's Lemma to the square root  function   and the process $Y_t$ in the open set $(0,\infty)$. We get
		\begin{align*}
		\sqrt{Y_t^y}&= \sqrt{y}+ \int_0^t \frac{1}{2\sqrt{Y_u^y}}dY^y_u - \frac 1 2 \int_0^t
		\frac{1}{4(Y^y_u)^{\frac 3 2 }}\sigma^2 Y_u^y du\\
		&=\sqrt{y}+ \left(  \frac{\kappa \theta}{2} - \frac{\sigma^2}{8} \right)  \int_0^t \frac{1}{\sqrt{Y_u^y}}du  - \frac \kappa 2 \int_0^t \sqrt{Y^y_u}du + \frac \sigma 2 W_t.
		\end{align*}
		Differentiating with respect to $y$ (see also \cite{Oa}) we deduce that
		\begin{equation}\label{stimaperholder}		\begin{split}
		\frac{\dot{Y_t^y}}{2\sqrt{Y_t^y}}  & = \frac{1}{2\sqrt{y}} + \left(  \frac{\kappa \theta}{2} - \frac{\sigma^2}{8} \right)  \int_0^t -\frac{\dot{Y_u^y}}{2(Y_u^y)^{\frac 3 2 }}du - \frac \kappa 2 \int_0^t \frac{\dot{Y_u^y}}{2\sqrt{Y_u^y}} du \leq \frac{1}{2\sqrt{y}}, \qquad a.s. 
		\end{split}
		\end{equation}
		since $\kappa\theta \geq \sigma^2/2\geq \sigma^2/4$ and $Y^y_t>0, \ \dot{Y}^y_t \geq0$ (see  \cite[Theorem 3.7, Chapter 9]{RY}). 
		
		Therefore, let us consider $y,y' \geq a$. Repeating the same calculations as before
		\begin{align*}
		&	|u(t,x,y)-u(t,x,y')| \\\quad
		%		&= \left|    \sup_{\theta \in[t,T]}\E[e^{-r(\theta-t )}(K-S^{t,s,y}_\theta)_+]-  \sup_{\theta \in[t,T]}\E[e^{-r(\theta-t )}(K-S^{t,s,y'}_\theta)_+] \right|\\&
		%		\leq \left|      \sup_{\theta \in[t,T]}\E\Big[e^{-r(\theta-t )}(K-S^{t,s,y}_\theta)_+-e^{-r(\theta-t )}(K-S^{t,s,y'}_\theta)_+  \Big]  \right| \\
		%		&\leq   \E \left[  \sup_{u \in [t,T]}|S_u^{0,s,y}      -S_u^{t,s,y'}| \right]  \\
		%		& \leq C\left(\int_0^T\E|Y^y_u-Y^{y'}_u| du +\E\left[ \sup_{t\in[0,T]} \left|  \int_0^t  (\sqrt{Y^y_u}-\sqrt{Y^{y'}_u}) dW_u   \right| \right]\right)\\
		&\quad\leq C\left(\int_t^T\E[|Y^{t,y}_u-Y^{t,y'}_u|]du +\left(\E\left[ \sup_{s\in[t,T]} \left|  \int_t^s  (\sqrt{Y^{t,y}_u}-\sqrt{Y^{t,y'}_u})dW_u   \right| \right]^2\right)^{\frac 1 2} \right)\\
		&	\quad\leq  C\left(\int_t^T\E[|Y^{t,y}_u-Y^{t,y'}_u|  ]du +\left(\E\left[  \int_t^T (\sqrt{Y^{t,y}_u}-\sqrt{Y^{t,y'}_u})^2du    \right]\right)^{\frac 1 2} \right)\\
		&\quad =  C\left(\int_t^T\E[|Y^{t,y}_s-Y^{t,y'}_s |]du +\left(\E\left[  \int_t^Tdu \left(\int_{y}^{y'}   \frac{\dot{Y}_u^{t,w}}{2\sqrt{Y_u^{t,w}}} dw  \right)^2  \right]\right)^{\frac 1 2} \right)\\
		&\quad \leq C_T\left(   |y-y'| + \left( \E\left[  \int_t^T \left(   \frac{1}{2\sqrt{a}} |y-y'|  \right)^2du  \right]\right)^{\frac 1 2}    \right)\\
		&\quad \leq C_T|y-y'|,
		\end{align*}
		which completes the proof. 
	\end{proof}
	\begin{remark}
		Studying the properties of the put price also clarifies the behaviour of the call price since it is straightforward to extend to the Heston model the  symmetry relation between call and put prices. In fact, let us   highlight the dependence of the prices with respect to the  parameters $K,r,\delta,\rho$, that is let us write
		$$
		P(t,x,y;K,r,\delta,\rho)=\sup_{\tau\in\mathcal{T}_{t,T}}\E[e^{-r(\tau-t)}(K-S^{t,s,y}_\tau)_+],
		$$
		for the put option price and
		$$ C(t,s,y;K,r,\delta,\rho)=\sup_{\tau\in\mathcal{T}_{t,T}}\E[e^{-r(\tau-t)}(S^{t,s,y}_\tau-K)_+],
		$$
		for the call option. Then, we have $C(t,s,y;K,r,\delta,\rho)=P(t,K,y;x,\delta,r,-\rho)$.
		
		In fact, for every $\tau \in \mathcal{T}_{t,T}$, we have
		\begin{align*}
		&\E e^{-r(\tau-t)}  \bigg( se^{\int_t^\tau\left( r-\delta -\frac{Y^{t,y}_s}2\right)ds+\int_t^\tau \sqrt{Y^{t,y}_s}dB_s}-K\bigg)_+  \\&\qquad=\E e^{-\delta(\tau-t)}e^{\int_t^\tau \sqrt{Y^{t,y}_s}dB_s-\int_t^\tau\frac{Y^{t,y}_s}2ds } \bigg(  x-Ke^{\int_t^\tau\left(\delta-r +\frac{Y^{t,y}_s}2\right)ds -\int_t^\tau dB_s  } \bigg)_+  \\&\qquad=\E e^{-\delta(\tau-t)}e^{\int_t^T \sqrt{Y^{t,y}_s}dB_s-\int_t^T\frac{Y^{t,y}_s}2ds } \bigg(  x-Ke^{\int_t^\tau\left(\delta-r +\frac{Y^{t,y}_s}2\right)ds -\int_t^\tau  \sqrt{Y^{t,y}_s}dB_s  } \bigg)_+ , 
		\end{align*}
		where the last equality follows from the fact that $(e^{\int_t^s \sqrt{Y^{t,y}_s}dB_s-\int_t^s\frac{Y^{t,y}_s}2ds })_{s\in[t,T]}$ is a martingale. Then, note that the process  $\hat B_t= B_t-\sqrt{Y^{t,y}_t }t$ is a Brownian motion under the  probability measure $\hat P$ which has density $d\hat{\P}/d\P=e^{\int_t^T \sqrt{Y^{t,y}_s}dB_s-\int_t^T\frac{Y^{t,y}_s}2ds }$. Therefore
		$$
		\E e^{-r(\tau-t)}  \bigg( se^{\int_t^\tau\left( r-\delta -\frac{Y^{t,y}_s}2\right)ds+\int_t^\tau \sqrt{Y^{t,y}_s}dB_s}-K\bigg)_+ =\hat \E e^{-\delta (\tau-t)} \bigg(  x-Ke^{\int_t^\tau\left(\delta-r -\frac{Y^{t,y}_s}2\right)ds -\int_t^\tau \sqrt{Y^{t,y}_s}dB_s  } \bigg)_+.
		$$
		Under the probability $\hat{\P}$, the process $(-\hat B,W )$ is a Brownian motion with correlation coefficient $-\rho$ so that the assertion follows.
	\end{remark}
	\subsection{The  exercise boundary}\label{sect-exerciseboundary}
	Let us introduce the so called continuation region
	$$
	\mathcal C=\{ (t,s,y) \in [0,T)\times (0,\infty)\times[0,\infty) :P(t,s,y)>\varphi(s)  \}
	$$
	and its complement, the exercise region
	$$
	\mathcal E= \mathcal C^c=\{ (t,s,y) \in [0,T)\times (0,\infty)\times[0,\infty) :P(t,s,y)=\varphi(s)\}.
	$$
	Note that, since $P$ and $\varphi$ are both continuous,  $\mathcal C$ is an (relative) open set while  $\mathcal E$ is a closed set.

	Generalizing the standard definition given in the Black and Scholes  type models,  we consider the \textit{critical exercise price} or \textit{free exercise boundary}, defined as
	$$
	b(t,y)=\inf\{ s>0| P(t,s,y)>(K-s)_+\},\qquad (t,y)\in [0,T)\times [0,\infty).
	$$
	We have $P(t,s,y)=\varphi(s)$  for $s\in [0,b(t,y))$ and also for $s= b(t,y)$, due to the continuity of $P$ and $\varphi$.	%Note also that, since $t\mapsto P(t,s,y)$ is nonincreasing and $y\mapsto P(t,s,y)$ is nondecreasing, we have that $t\mapsto b(t,y)$ is nondecreasing and $y\mapsto b(t,y)$ is nonincreasing. 
	Note also that, since $P>0$, we have $b(t,y) \in [0,K)$.
	Moreover, since $P$ is convex, we can write
	$$
	\mathcal C=\{(t,s,y) \in [0,T)\times (0,\infty)\times[0,\infty) :  s>b(t,y)    \}
	$$
	and
	$$
	\mathcal E=\{(t,s,y) \in [0,T)\times (0,\infty)\times[0,\infty) :  s\leq b(t,y)    \}.
	$$

	% \subsection{Positivity of the critical price}
	We now study some properties of the free boundary  $b:[0,T)\times [0,\infty)\rightarrow [0,K)$.  First of all, we have the following simple result.
	\begin{proposition}
		We have:
		\begin{enumerate}
			\item 	for every fixed $y\in[0,\infty)$, the function $t\mapsto b(t,y)$ is nondecreasing and right continuous;
			\item  for every fixed $t\in[0,T)$, the function $y\mapsto b(t,y)$ is nonincreasing and  left continuous.
		\end{enumerate}
	\end{proposition}
	\begin{proof}
		$1.$ Recalling that the map $t\mapsto P(t,s,y)$ is nonincreasing, we directly deduce that $t\mapsto b(t,y)$  is nondecreasing.
		Then, fix $t\in [0,T)$ and let $(t_n)_{n \geq 1}$ be a decreasing sequence such that $\lim_{n \rightarrow \infty} t_n = t$. The sequence $(b(t_n,y))_n$ is nondecreasing so that $ \lim_{n \rightarrow \infty}b(t_n,y)$ exists and we have $  \lim_{n \rightarrow \infty}b(t_n,y) \geq b(t,y)$.
		On the other hand, we have		
		\begin{equation*}
		P(t_n,b(t_n,y),y)= \varphi(b(t_n,y)) \qquad  n \geq 1,
		\end{equation*}
		and, by the continuity of $P$ and $\varphi$,
		\begin{equation*}
		P(t,\lim_{n \rightarrow \infty}b(t_n,y),y)= \varphi(\lim_{n \rightarrow \infty}b(t_n,y)).
		\end{equation*}
		We deduce by the definition of $b$ that $  \lim_{n \rightarrow \infty}b(t_n,y) \leq b(t,y)$ which concludes the proof.
		
		$2.$  The second assertion can be proved with the same arguments, this time recalling that $y\mapsto P(t,s,y)$ is a nondecreasing function.
	\end{proof}
	Recall that  $b(t,y) \in [0,K)$.   Indeed, we can prove  the positivity of the function.
	\begin{proposition}\label{positivity}
		We have $b(t,y)>0$ for every $(t,y) \in [0,T)\times [0,\infty)$.
	\end{proposition}
	\begin{proof}
		Without loss of generality we can assume that $0<t<T$, since $T$ is arbitrary and the put price is a function of $T-t$.  Suppose that $b(t^*,y^*)=0$ for some $(t^*,y^*) \in (0,T)\times [0,	\infty) $. Since $b(t,y)\geq 0$, $t \mapsto b(t,y)$ is nondecreasing and $y \mapsto b(t,y)$ is nonincreasing, we have $ b(t,y)=0$ for $(t,y) \in (0,t^*) \times (y^*,\infty)$, so that
		\begin{equation*}
		P(t,s,y)> \varphi(s),\qquad (t,s,y) \in (0,t^*) \times(0,\infty) \times (y^*,\infty).
		\end{equation*}
		To simplify the calculations, we pass  to the logarithm in the space variable and we consider the functions $u(t,x,y)= P(t,e^x,y)$ and $\psi(x)=\varphi(e^x) $. We  have $u(t,x,y)>\psi(x)$ and $$(\partial_t+\tilde{\mathcal{L}}-r)u=0 \qquad \mbox{ on } (0,t^*) \times \R \times (y^*,\infty),$$ where $\tilde \L$ was defined  in \eqref{Llog}. Since $t \mapsto u(t,x,y)$ is nondecreasing, we deduce that, for $t\in (0,t^*)$, $(\tilde \L-r)u=-\partial_t u\geq 0$ in the sense of distributions. Therefore, for any nonnegative and $C^\infty$ test functions $\theta$, $\phi$ and $\zeta$ which have support respectively in $(0,t^*)$, $(-\infty,\infty)$ and $(y^*,\infty)$, we have 
		\begin{equation*}
		\int_0^{t^*}\theta(t)dt \int_{-\infty}^\infty dx \int_{y^*}^\infty dy \tilde{ \mathcal{L} }u(t,x,y) \phi(x)\zeta(y) \geq r 	\int_0^{t^*}\theta(t)dt \int_{-\infty}^\infty dx \int_{y^*}^\infty dy (K-e^x)\phi(x)\zeta(y),
		\end{equation*}
		or equivalently, by the continuity of the integrands in $t$, 			
		\begin{equation} \label{distr}
		\int_{-\infty}^\infty dx \int_{y^*}^\infty dy\tilde {\mathcal{L}} u(t,x,y) \phi(x)\zeta(y) \geq r 	 \int_{-\infty}^\infty dx \int_{y^*}^\infty dy (K-e^x)\phi(x)\zeta(y).
		\end{equation}
		Let $\chi_1$ and $\chi_2 $ be two nonnegative $C^\infty$ functions  such that $\supp\chi_1\subseteq [-1,0]$, $\supp\chi_2\subseteq [0,1]$ and $\int \chi_1(x)dx=\int \chi_2(x)dx=1$. Let us apply \eqref{distr} with $\phi(x)=\lambda \chi_1(\lambda x)$ and $\zeta(y)=\sqrt{\lambda} \chi_2(\sqrt{\lambda} (y-y^*))$, with $\lambda>0$. For the right hand side of \eqref{distr}, we have
		\begin{align*}
		r  \int_{-\infty}^\infty dx \int_{y^*}^{\infty}dy (K-e^x)\phi(x)\zeta(y)= rK -r\int_{-\infty}^\infty e^{\frac{x}{\lambda}} \chi_1(x)dx .
		\end{align*}
		Since $ \supp\chi_1 \subset [-1,0]$, $\lim_{\lambda\rightarrow 0}\int e^{\frac{x}{\lambda}} \chi_1(x)dx =0$, so that
		\begin{equation}\label{leftside}
		\lim_{\lambda\rightarrow 0} r  \int_\R dx \int_{-\infty}^{y^*}dy (K-e^x)\phi(x)\zeta(y)= rK >0.
		\end{equation}
		
		As regards the left hand side of \eqref{distr}, we have
		\begin{align*}
		&\int_{-\infty}^{+\infty} dx \int_{y^*}^{\infty}\tilde{\mathcal{L}   } u(t,x,y) \phi(x)\zeta(y)dy\\&=	\int_{-\infty}^{+\infty} dx \int_{y^*}^{\infty}\! \frac y 2 \left( \frac{\partial^2u}{\partial x^2}(t,x,y)+2\rho\sigma \frac{\partial^2u}{\partial x \partial y}(t,x,y)+ \sigma^2\frac{\partial^2u}{\partial y^2}(t,x,y) \right)\lambda \chi_1(\lambda x )\sqrt{\lambda}\chi_2(\sqrt{\lambda}(y-y^*))dy\\&+\int_{-\infty}^{+\infty} dx \int_{y^*}^{\infty} \left(\left( r-\delta-\frac y 2 \right)\frac{\partial u}{\partial x}(t,x,y)+\kappa(\theta-y)\frac{\partial u}{\partial y}(t,x,y)\right)\lambda \chi_1(\lambda x )\sqrt{\lambda}\chi_2(\sqrt{\lambda}(y-y^*))dy.
		\end{align*}
		%			\begin{align}
		%			& \int_{-\infty}^{+\infty} dx \int_{y^*}^{\infty}dy\mathcal{L   } u(t,x,y) \phi(x)\zeta(y)= - \int_{-\infty}^{+\infty} dx \int_{y^*}^{\infty} dy\frac{\partial}{\partial x }u(t,x,y) \frac{\sigma^2(y)}{2}\lambda \chi_2(\lambda (y-y^*)) \lambda^2 \chi_1'(\lambda x)  \label{1}\\
		%			&-\int_{-\infty}^{+\infty} dx \int_{-\infty}^{y^*}dy \frac{\partial}{\partial y }u(t,x,y)\lambda \chi_1(\lambda x) \frac{\partial}{\partial y }\left( \frac{\gamma^2(t,y)}{2}\lambda \chi_2(\lambda (y-y^*)) \right)\label{2}\\
		%			&+\int_{-\infty}^{+\infty} dx \int_{-\infty}^{y^*}dy u(t,x,y)\frac{\partial^2}{\partial x \partial y }\left(\rho \sigma\gamma(t,y)\lambda \chi_1(\lambda x)\lambda \chi_2(\lambda (y-y^*)) \right)\label{3}\\
		%			&-\int_{-\infty}^{+\infty} dx \int_{-\infty}^{y^*}dy u(t,x,y)) \left(r- \frac{\sigma^2(y)}{2} \right)\lambda \chi_2(\lambda (y-y^*)) \lambda^2 \chi_1'(\lambda x)\label{4}\\
		%			&+\int_{-\infty}^{+\infty} dx \int_{-\infty}^{y^*}dy \frac{\partial}{\partial y }u(t,x,y) \beta(t,y) \lambda \chi_1(\lambda x)\lambda \chi_2(\lambda (y-y^*)).\label{5}
		%			\end{align}
		%			We prove that each term in the right hand side goes to $0$ as $\lambda$ goes to $0$.
		We first study the second order derivatives term. Integrating by parts two times  we have
		\begin{align*}
		\int_{-\infty}^{+\infty} &dx \int_{y^*}^\infty   \frac y 2 \frac{\partial^2}{\partial x^2 }u(t,x,y) \lambda \chi_1(\lambda x)\sqrt{\lambda}\chi_2(\sqrt{\lambda}(y-y^*))dy\\
		%&= 	-\int_{-\infty}^{+\infty} dx \int_{y^*}^\infty  dy \frac y 2 \frac{\partial}{\partial x }u(t,x,y) \lambda^2 \chi'_1(\lambda x)\sqrt{\lambda}\chi_2(\sqrt{\lambda}(y-y^*))\\
		&=\int_{-\infty}^{+\infty} dx \int_{y^*}^\infty   \frac y 2 u(t,x,y) \lambda^3 \chi''_1(\lambda x)\sqrt{\lambda}\chi_2(\sqrt{\lambda}(y-y^*))dy\\
		%	&	=\lambda^2 \int_{-\infty}^{+\infty} dx \int_0^{\infty}  dy \frac 1 2\left(\frac y{\sqrt \lambda } +y^*\right) u\left(t,\frac x \lambda,\frac y{\sqrt \lambda } +y^*\right)  \chi''_1( x) \chi_2(y)\\
		&	=\lambda^{\frac 3 2} \int_{-\infty}^{+\infty} dx \int_0^{\infty}   \frac 1 2\left(y +\sqrt \lambda y^*\right) u\left(t,\frac x \lambda,\frac y{\sqrt \lambda } +y^*\right)  \chi''_1( x) \chi_2(y)dy.
		\end{align*}
		Since $u$ is bounded and $\chi_2$ has support in $[0,1]$, the last term goes to 0 as $\lambda$ tends to 0.
		%			\begin{equation*}
		%			\lim_{\lambda\rightarrow 0 } \lambda^2 \int_{-\infty}^{+\infty} dx \int_{-\infty}^0  dy \frac 1 2\left(y^*-\frac y \lambda\right) u\left(t,\frac x \lambda,y^*-\frac y \lambda\right)  \chi''_1( x)\lambda \chi_2(y)=0.
		%			\end{equation*}
		For the mixed derivative term, since $\chi_2(0)=0$,
		\begin{align*}
		&	\int_{-\infty}^{+\infty}  dx \int_{y^*}^\infty \rho\sigma y\frac{\partial^2}{\partial x \partial y} u(t,x,y)    \lambda \chi_1(\lambda x)\sqrt{\lambda}\chi_2(\sqrt{\lambda}(y-y^*))dy\\&\quad=-\rho \sigma
		\int_{-\infty}^{+\infty} dx \int_{y^*}^\infty y\frac{\partial}{\partial y }u(t,x,y) \lambda^2 \chi_1'(\lambda x)  \sqrt{\lambda}\chi_2(\sqrt{\lambda}(y-y^*))dy\\&\quad=\rho \sigma
		\int_{-\infty}^{+\infty} dx \int_{y^*}^\infty  u(t,x,y) \lambda^2 \chi_1'(\lambda x)  \sqrt{\lambda}\chi_2(\sqrt{\lambda}(y-y^*))dy\\&\qquad+\rho \sigma
		\int_{-\infty}^{+\infty} dx \int_{y^*}^\infty u(t,x,y) \lambda^2 \chi_1'(\lambda x)  \lambda\chi'_2(\sqrt\lambda(y^*-y))dy\\
		&\quad=\lambda \rho \sigma
		\int_{-\infty}^{+\infty} dx \int_0^\infty  u\left(t,\frac x \lambda,\frac y{\sqrt \lambda } +y^*\right)  \chi_1'( x)  \chi_2(y) dy\\& \qquad  + \lambda^{\frac  3 2 }\rho \sigma
		\int_{-\infty}^{+\infty} dx \int_0^\infty u\left(t,\frac x \lambda,\frac y{\sqrt \lambda } +y^*\right)  \chi_1'( x)  \chi'_2(y)dy,
		\end{align*}
		which goes to 0 as $\lambda$ tends to 0 with the same arguments as before.
		
		Moreover, integrating by parts two times, we have
		\begin{align*}
		&	\int_{-\infty}^{+\infty} dx \int_{y^*}^\infty   \frac y 2 \sigma^2 \frac{\partial^2}{\partial  y^2} u(t,x,y)    \lambda \chi_1(\lambda x)\sqrt \lambda \chi_2(\sqrt \lambda (y-y^*)) dy\\
		&\quad=-	\int_{-\infty}^{+\infty} dx \int_{y^*}^\infty  \frac  {\sigma^2} 2 \frac{\partial}{\partial  y} u(t,x,y)    \lambda \chi_1(\lambda x)\left( \sqrt \lambda \chi_2(\sqrt \lambda (y-y^*)) + y\lambda \chi'_2(\sqrt \lambda (y-y^*)) \right)dy\\
		&\quad =\int_{-\infty}^{+\infty} dx \int_{y^*}^\infty    \frac  {\sigma^2} 2 u(t,x,y)  \left( 2 \lambda \chi_1(\lambda x)\lambda \chi'_2(\sqrt \lambda (y-y^*))\right)dy \\
		%			&\qquad-	\int_{-\infty}^{+\infty} dx \int_{y^*}^\infty dy  \frac y 2 \sigma^2 \frac{\partial}{\partial  y} u(t,x,y)    \lambda \chi_1(\lambda x)\lambda \chi'_2(\sqrt \lambda (y-y^*)) \\
		%	&\quad =\int_{-\infty}^{+\infty} dx \int_{y^*}^\infty dy   \frac  {\sigma^2} 2 u(t,x,y)    \lambda \chi_1(\lambda x)\lambda \chi'_2(\sqrt \lambda (y-y^*)) \\
		%		&\qquad+	\int_{-\infty}^{+\infty} dx \int_{y^*}^\infty dy  \frac  {\sigma^2} 2  u(t,x,y)    \lambda \chi_1(\lambda x)\lambda \chi'_2(\sqrt \lambda (y-y^*))\\
		%		&\qquad +	\int_{-\infty}^{+\infty} dx \int_{y^*}^\infty dy\, \frac y 2 \sigma^2  u(t,x,y)    \lambda \chi_1(\lambda x)\lambda ^{\frac 3 2} \chi''_2(\sqrt \lambda (y-y^*))\\
		&\quad=\sqrt \lambda  \sigma^2\int_{-\infty}^{+\infty} dx \int_0^{\infty }   u\left(t,\frac x \lambda ,\frac y {\sqrt \lambda}+y^*\right)    \chi_1( x) \left(  \lambda \chi'_2(y)+\frac 1 2\lambda ^{\frac 3 2} \left( y +\sqrt \lambda y^*\right) \chi''_2(y)\right)  dy
		%\\=&\sqrt \lambda \sigma^2 \int_{-\infty}^{+\infty} dx \int_0^{\infty } dy \,  u\left(t,\frac x \lambda ,\frac y {\sqrt \lambda}+y^*\right)    \chi_1( x) \chi'_2(y) \\
		%	&+\sqrt \lambda	\int_{-\infty}^{+\infty} dx \int_0^{\infty } dy \,   \frac {\sigma^2} 2  u\left(t,\frac x \lambda ,\frac y {\sqrt \lambda} +y^*\right)   \chi_1( x) \chi'_2(y)\\
		%	& +	\sqrt \lambda\int_{-\infty}^{+\infty} dx \int_0^{\infty} dy\,  \frac {\sigma^2} 2\left( y +\sqrt \lambda y^*\right)  u\left(t,\frac x \lambda ,\frac y {\sqrt \lambda} +y^*\right)    \chi_1( x) \chi''_2(y),
		\end{align*}
		which again tends to 0 as $\lambda$ goes to 0. We now study the terms in \eqref{distr} which contains the first order derivatives of $u$. First, note that
		\begin{align*}
		&	 	\int_{-\infty}^{+\infty} dx \int_{y^*}^{\infty} \left( r-\delta-\frac y 2 \right)\frac{\partial }{\partial x}u(t,x,y)\lambda \chi_1(\lambda x )\sqrt \lambda\chi_2(\sqrt\lambda(y-y^*))dy \\&\quad=-	\int_{-\infty}^{+\infty} dx \int_{y^*}^{\infty} \left( r-\delta-\frac y 2 \right) u(t,x,y)\lambda^2 \chi'_1(\lambda x )\sqrt \lambda\chi_2(\sqrt\lambda(y-y^*))dy\\
		%&=-\lambda	\int_{-\infty}^{+\infty} dx \int_0^{\infty}dy\, \left( r-\delta-\frac 1 2\left(\frac y {\sqrt \lambda} +y^*\right) \right) u\left(t,\frac x \lambda ,y^*-\frac y \lambda \right)\chi'_1( x )\chi_2(u)\\
		&\quad=-\sqrt \lambda	\int_{-\infty}^{+\infty} dx \int_0^{\infty} \left( \sqrt \lambda r-\sqrt \lambda \delta-\frac 1 2\left(y  +\sqrt \lambda y^*\right) \right) u\left(t,\frac x \lambda ,\frac y {\sqrt \lambda} +y^* \right)\chi'_1( x )\chi_2(y)dy.
		\end{align*}
		Again, passing to the limit, the last term tends to 0. On the other hand, 
		\begin{align*}
		&	\int_{-\infty}^{+\infty} dx \int_{y^*}^{\infty} \kappa(\theta-y)\frac{\partial }{\partial y}u(t,x,y)\lambda \chi_1(\lambda x )\sqrt \lambda\chi_2(\sqrt \lambda(y-y^*))dy\\
		&\quad=	\int_{-\infty}^{+\infty} dx \int_{y^*}^{\infty} \kappa\theta\frac{\partial }{\partial y}u(t,x,y)\lambda \chi_1(\lambda x )\sqrt \lambda\chi_2(\sqrt \lambda(y-y^*))dy \\&\qquad-\int_{-\infty}^{+\infty} dx \int_{y^*}^{\infty}\kappa y\frac{\partial }{\partial y}u(t,x,y)\lambda \chi_1(\lambda x )\sqrt \lambda\chi_2(\sqrt \lambda(y-y^*))dy.
		\end{align*}
		Integrating by parts and doing the usual change of variables we have
		\begin{align*}
		&	\int_{-\infty}^{+\infty} dx \int_{y^*}^{\infty} \kappa\theta\frac{\partial }{\partial y}u(t,x,y)\lambda \chi_1(\lambda x )\sqrt \lambda\chi_2(\sqrt \lambda(y-y^*)) dy\\
		%	&=-	\int_{-\infty}^{+\infty} dx \int_{y^*}^{\infty}\kappa\theta u(t,x,y)\lambda \chi_1(\lambda x ) \lambda\chi'_2(\sqrt \lambda(y-y^*))dy\\
		&\quad=-\sqrt \lambda\int_{-\infty}^{+\infty} dx \int_0^{\infty}\kappa\theta  u\left(t,\frac x \lambda ,\frac y {\sqrt \lambda} +y^*\right)\chi_1( x )\chi'_2(y)	dy,
		\end{align*}
		which tends  to 0 as $\lambda$ tends to 0, 	while 
		$$
		-\int_{-\infty}^{+\infty} dx \int_{y^*}^{\infty}\kappa y\frac{\partial }{\partial y}u(t,x,y)\lambda \chi_1(\lambda x )\sqrt \lambda\chi_2(\sqrt \lambda(y-y^*))dy,
		$$
		which is nonpositive, since $u$ is nondecreasing in $y$. 
		We finally deduce that
		\begin{equation}
		\label{rightside}
		\limsup_{\lambda\rightarrow 0 }		\int_{-\infty}^{+\infty} dx \int_{y^*}^{\infty}dy\mathcal{L   } u(t,x,y) \phi(x)\zeta(y)\leq 0,
		\end{equation}
		which, together with \eqref{leftside}, contradicts \eqref{distr}. Then, the assertion follows.
	\end{proof}
	As regards  the regularity of the free boundary, we can prove the following result.
	
	\begin{proposition}
		For any $t\in[0,T)$ there exists a countable set $\mathcal N\subseteq (0,\infty)$ such that
		$$
		b(t^-,y)= b(t,y), \qquad y\in (0,\infty)\setminus\mathcal N.
		$$
	\end{proposition}
	\begin{proof}
		Without loss of generality we pass to the logarithm in the $s-$variable and we prove the assertion for the function $\tilde b(t,y)=\ln b(t,y)$.
		Fix $t\in [0,T)$ and recall that $y\mapsto \tilde b (t,y)$ is a nonincreasing function, so it has at most a countable set of discontinuity points. Let $y^*\in (0,\infty)$ be a continuity point for the maps $y\mapsto \tilde b(t,y)$ and $y\mapsto \tilde b(t^-,y)$ and assume that
		\begin{equation}\label{assu}
		\tilde b(t^-,y^*)<\tilde b(t,y^*).
		\end{equation}
		Set $	\epsilon=\frac {\tilde b(t,y^*)-\tilde b(t^-,y^*) }2$.  By continuity, there exist $y_0,y_1> 0$ such that for any $y\in (y_0,y_1)$ we have
		$$
		\tilde b(t,y)>	\tilde b(t,y^*)-\frac \epsilon 4, \qquad \mbox{ and }\qquad
		\tilde b(t^-,y)<	\tilde b(t^-,y^*)+\frac \epsilon 4.
		$$
		Therefore, by using \eqref{assu}, we get, for any $y\in (y_0,y_1)$,
		$$
		\tilde b(t,y)>	\tilde b(t,y^*)-\frac \epsilon 4>\tilde b(t^-,y^*)+\frac 3 4 \epsilon >	\tilde b(t^-,y^*)+\frac \epsilon 4>\tilde b(t^-,y).
		$$
		Now, set $b^-=\tilde b(t^-,y^*)+\frac \epsilon 4$ and $b^+=\tilde b(t^-,y^*)+\frac 3 4 $ and let  $(s,x,y)\in (0,t)\times (b^-,b^+)\times (y_0,y_1)$. Since  $t\mapsto		\tilde b(t,\cdot)$ is nondecreasing, we have $x>\tilde b(t^-,y)>\tilde b(s,y)$, so that $u(s,x,y)>\psi(x)$. Therefore, on the set $(0,t)\times (b^-,b^+)\times (y_0,y_1)$ we have 
		$$
		(	\tilde \L-r)u(s,x,y)= -\frac{\partial u}{\partial t}(s,x,y)\geq 0.
		$$
		This means that, for any nonnegative and $C^\infty$ test functions $\theta$, $\psi$ and $\zeta$ which have  support respectively in $(0,t)$, $(b^-,b^+)$ and $(y_0,y_1)$ we can write
		\begin{equation*}
		\int_0^{t}\theta(t)dt \int_{-\infty}^\infty dx \int_{y^*}^\infty dy ({\tilde{\L}} -r)u(t,x,y) \phi(x)\zeta(y) \geq 0.
		\end{equation*}
		By the continuity of the integrands in $t$, 	we deduce that 		$(\tilde \L -r)u(t,\cdot,\cdot)\geq 0$ in the sense of distributions on the set $(b^-,b^+)\times (y_0,y_1)$.
		%				\begin{equation} \label{distr2}
		%				\int_{-\infty}^\infty dx \int_{y^*}^\infty dy\tilde ({\mathcal{L}} -r)u(t,x,y) \phi(x)\zeta(y) \geq 0.
		%				\end{equation}
		
		On the other hand,  for any $(s,x,y)\in (t,T)\times (b^-,b^+)\times (y_0,y_1)$,  we have $x\leq	\tilde b(t,y)\leq 	\tilde b(s,y)$, so that $u(s,x,y)=\psi(x)$. Therefore, it follows from   $\frac{\partial u}{\partial t}+(	\tilde \L-r)u\leq 0$ and the continuity of the integrands that $(	\tilde \L-r)u(t\cdot, \cdot)=(	\tilde \L-r)\psi(\cdot)
		\leq 0$ in the sense of distributions on the set $ (b^-,b^+)\times (y_0,y_1)$. 
		
		We deduce that $(	\tilde \L-r)\psi= 0$ on the set $(b^-,b^+)\times (y_0,y_1)$, but it is easy to see that  $
		(	\tilde \L-r)\psi(x)=(	\tilde \L-r)(K-e^x)=\delta e^x-rK$ and thus cannot be identically zero in a nonempty open set.
	\end{proof}
	
	\begin{remark}
		It is worth observing that the arguments used  in \cite{V} in order to prove the continuity of the exercise price of American options in a multidimensional Black and Scholes model can be easily adapted to our framework.  In particular, if  we  consider the $t$-sections of the exercise region, that is
		\begin{equation}\label{tsection}
		\begin{split}
		\mathcal E_t&= \{ (s,y)\in (0,\infty)\times [0,\infty): P(t,s,y)=\varphi(s)      \},	\\&=\{ (s,y)\in (0,\infty)\times [0,\infty): s\leq b(t,y)\}, \qquad\qquad\qquad t\in [0,T),
		\end{split}
		\end{equation}
		we can easily prove that 
		\begin{equation}\label{tsections}
		\Ex_t=\bigcap_{u>t}\Ex_u,\qquad\qquad
		\Ex_t=\overline{\bigcup_{u<t}\Ex_u}.
		\end{equation}
		However, unlike the case of an American option on several assets, in our case \eqref{tsections} is not sufficient   to  deduce the continuity of the function $t\mapsto b(t,y)$.
	\end{remark}
	\subsection{Strict convexity in the continuation region}
	We know that $P$ is convex in the space variable (see Proposition \ref{properties_P}). In \cite{T} it is also proved that, in the case of non-degenerate stochastic volatility models, $P$ is strictly convex in the continuation region but the proof follows an analytical approach which cannot be applied in our degenerate model.   In this section we extend this result to the Heston model by using  purely probabilistic techniques. 
	
	We will need the following Lemma, whose proof can be found in the Appendix.
	\begin{lemma}\label{support}
		For every continuous function $s: [0,T]\rightarrow \R$ such that $s(0)=S_0$ and for every $\epsilon>0$ we have
		$$
		\P\left(\sup_{t\in [0,T]}|S_t-  s (t)|<\epsilon , \sup_{t\in [0,T]}|Y_t-Y_0|<\epsilon \right)>0.
		$$
	\end{lemma}
	
	\begin{theorem}
		The function $s \mapsto P(t,s,y)$ is strictly convex in the continuation region.
	\end{theorem}
	\begin{proof}
		Without loss of generality we can assume $t=0$. 	 We have to prove that, if $(s_1,y), \, (s_2,y) \in (0,\infty)\times [0,\infty)$ are such that  $(0,s_1,y), (0,s_2,y) \in \mathcal{C}$, then
		\begin{equation}\label{conv}
		P(0,\theta s_1+(1-\theta)s_2,y) < 
		\theta P(0, s_1,y)+(1-\theta) P(0,s_2,y).
		\end{equation}
		Let us   rewrite the price process  as
		$
		S_t^{s,y}=se^{ \int_{0}^{t}\left(r-\delta-\frac {Y_u}{2}\right)du +\int_{0}^{t}\sigma \sqrt{Y_u}dB_u  }:= sM^y_t,
		$
		where $M^y_t=S^{1,y}_t$ and  assume that, for example, $s_1>s_2$.
		We claim that it is enough to prove that, for 
		$\varepsilon>0$ small enough,
		\begin{equation}\label{conv_2}
		\begin{split}
		\P\Big(  &(\theta s_1+(1-\theta)s_2)M^y_t>b(t,Y_t) \, \forall t \in [0,T)\, \& \,(\theta s_1+(1-\theta)s_2)M^y_T \in (K-\varepsilon,K+\varepsilon) \Big) >0.
		\end{split}
		\end{equation}
		In fact, 	let $\tau^*$ be the optimal stopping time for $P(0,\theta s_1+(1-\theta)s_2,y)$. If  $ (\theta s_1+(1-\theta)s_2)M^y_t>b(t,Y_t)$ for every $ t \in [0,T) $, then we are in the continuation region for all $  t \in [0,T) $, hence  $\tau^*=T$. Then, the condition $ (\theta s_1+(1-\theta)s_2)M^y_T \in (K-\varepsilon,K+\varepsilon) $ for  $\varepsilon>0$ small enough ensures on one hand that $s_1M^y_{\tau^*}>K$, since 
		\begin{align*}
		s_1M^y_{\tau^*}	& = (\theta s_1+(1-\theta)s_2 )M^y_{\tau^*}+ (1-\theta)(s_1-s_2)M^y_{\tau^*}\\
		& >K-\varepsilon + \frac{(1-\theta)(s_1-s_2)(K-\varepsilon)}{\theta s_1+(1-\theta)s_2 }>K,
		\end{align*}
		for $\varepsilon$ small enough. 		 	On the other hand, it also ensures that $ s_2M^y_{\tau^*}<K $, which can be proved  with similar arguments. 
		Therefore, we get
		\begin{equation*}\label{condition_convexity}
		\P\left((K-s_1M^y_{\tau^*})_+=0\, \& \, (K-s_2M^y_{\tau^*})_+>0 \right)>0,
		\end{equation*}
		which, from a closer look  at the graph of the function $x	\mapsto (K-x)_+$,  implies that
		\begin{align*}
		\E[e^{-r\tau^*}(K-(\theta s_1+(1-\theta)s_2 ) M^y_{\tau^*})_+] < \theta \E[e^{-r\tau^*}(K-s_1 M^y_{\tau^*})_+] +(1-\theta)\E[e^{-r\tau^*}(K-s_2 M^y_{\tau^*})_+],
		\end{align*}
		and, as a consequence, \eqref{conv}.

		So, the rest of the proof is devoted  to prove that \eqref{conv_2} is actually  satisfied.
		%	Thanks to Lemma \ref{support} we know that, for any $\varepsilon>0$ and any continuous function $m:[0,T]\rightarrow\R$ such that $m(0)=1$, we have
		%	\begin{equation}
		%		\P\left(\sup_{t\in [0,T]}|(	\theta s_1+(1-\theta)s_2 )M^y_t-  (	\theta s_1+(1-\theta)s_2 )m (t)|<\varepsilon , \sup_{t\in [0,T]}|Y_t-Y_0|<\varepsilon \right)>0.
		%	\end{equation}
		
		With this aim, we first consider a suitable continuous function $m:[0,T]\rightarrow\R$ constructed as follows. 
		In order to simplify the notation, we set $s=\theta s_1+(1-\theta)s_2$.
		%We first claim that there exists a continuous function $m:[0,T]\rightarrow\R$ such that, for $\varepsilon>0$ small enough,
		%\begin{equation}\label{conditionsurm}
		%m(0)=1,\qquad sm(t)>b(t,y)+\varepsilon,\qquad sm(T)\in (K-\varepsilon, K+\varepsilon).
		%\end{equation}
		Note that, for $\varepsilon>0$ small enough, we have $s=\theta s_1+(1-\theta)s_2 >b(0,y)+\varepsilon$ since $(0,s_1,y)$ and $ (0,s_2,y)$ are in the continuation region $ \mathcal{C}$, that is $s_1,s_2\in (b(0,y),\infty)$.  By the right continuity of the map $t\mapsto b(t,y)$, we know that there exists $\bar t\in (0,T)$ such that $s>b(t,y)+\frac \varepsilon 2$ for any $t\in [0,\bar t]$. Moreover the function $y\mapsto b(\bar t, y)$ is left continuous and nonincreasing, so there exists $\eta_\varepsilon>0$ such that $s >b(\bar t,z)+\frac \varepsilon 4$ for any $z\geq y-\eta_\varepsilon$. 
		Assume now that $s\leq K+\frac \varepsilon 2$ and
		set 
		$$
		m(t)=\begin{cases}1 + \frac t {\bar t} \left( \frac{ K+\frac \varepsilon 2 }{s}- 1 \right), \qquad &0\leq t\leq \bar t,\\
		\frac{K+\frac \varepsilon 2}{s},\qquad & \bar t\leq t\leq T.
		\end{cases}
		$$
		Note that $m$ is continuous, $m(0)=1$ and, recalling that $t\mapsto b(t,y)$ is nondecreasing and $b(t,y)<K$,
		$$
		sm(t)=\begin{cases}
		s + \frac t {\bar t} \left(K+\frac \varepsilon 2 - s \right)	\geq s >b(\bar t,y-\eta_\varepsilon )+\frac \varepsilon 4, \qquad &0\leq t\leq \bar t,\\
		K+\frac \varepsilon 2\geq b(t,y-\eta_\varepsilon),  &\bar t\leq  t\leq T.
		\end{cases}
		$$
		Moreover, by  Lemma \ref{support}, we know that, for any $\epsilon>0$,
		$$
		\P\left(\sup_{t\in [0,T]}|sM^y_t-  sm (t)|<\epsilon , \sup_{t\in [0,T]}|Y_t-y|<\epsilon \right)>0.
		$$
		Therefore, by applying Lemma \ref{support} with  $\epsilon=\min\left\{\frac \varepsilon 8,\eta_\varepsilon\right\}$, we have that, with positive probability,
		$$
		sM^y_t>sm(t)-\frac \varepsilon 8\geq b(t, y-\eta_\varepsilon)+\frac \varepsilon 8\geq b(t,Y_t).
		$$
		and
		$$
		sM^y_T\leq sm(T)+\frac \varepsilon 8\leq K+\varepsilon,\qquad 	sM^y_T\geq sm(T)-\frac \varepsilon 8\geq K-\varepsilon,
		$$
		%	$$
		%m(0)=1,\quad (	\theta s_1+(1-\theta)s_2 )m(t)>b(t,y), \quad  (\theta s_1+(1-\theta)s_2 )m(T)\in (K-\varepsilon,K+\varepsilon).
		%	$$
		which proves \eqref{conv_2} concluding the proof. If $s>K+\frac \varepsilon 2$, then it is enough to take $m(t)$ as a nonincreasing continuous function such that $m(0)=1$ and $sm(T)=K+\frac \varepsilon 2$. Then, the assertion follows with the same reasoning.

		%	
		%		 	Recall that $sM^y_t>b(t,Y^y_t) $ if and only if $P(t,S^{s,y}_t,Y^y_t) >\varphi(S^{s,y}_t)$ so, since both $P$ and $\varphi$ are continuous, the probability in \eqref{conv_2} can be rewritten as
		%		 	\begin{equation}
		%		 	\P\big((M^y_t,Y^y_t)\in \mathcal{U} \  \forall t \in [0,T]\big)=\P\big((S^{1,y}_t,Y^y_t)\in \mathcal{U} \  \forall t \in [0,T]\big)>0,\label{con_3}
		%		 	\end{equation}
		%		 	for a certain open set $\mathcal{U}.$
		%Therefore,  it is enough to show that the open set $\mathcal{U}$ is not empty.  
		%By using Lemma \ref{support}, it suffices to find a continuous function $m=m(t)$ such that $m(0)=1$,
		%		 	$$
		%		 	(\theta s_1+(1-\theta)s_2)m(t)>b(t,y), \qquad	t \in [0,T) 
		%		 	$$
		%		 	and 
		%		 	$$
		%		 	(\theta s_1+(1-\theta)s_2)m(T) \in (K-\varepsilon,K+\varepsilon).
		%		 	$$
		%
		%Note that $(\theta s_1+(1-\theta)s_2)m(0)=\theta s_1+(1-\theta)s_2 >b(0,y)$ since $(0,s_1,y)$ and $ (0,s_2,y)$ are in the continuation region $ \mathcal{C}$, that is $s_1,s_2\in (b(0,y),\infty)$.  Denote by $\tau^*=\inf \{  t>0:  \theta s_1+(1-\theta)s_2\leq b(t,y) \}$. Then, it is enough to take
		%		 	$$
		%		 	m(t)=\begin{cases}
		%		 	\theta s_1+(1-\theta)s_2+ \frac{K+\frac \varepsilon 2 -(\theta s_1+(1-\theta)s_2)}{\tau^*}t,\qquad& t<\tau^*,\\
		%		 	K+\frac \varepsilon 2, & t\geq \tau^*,
		%		 	\end{cases}
		%		 	$$  
		%		 	and the assertion follows.
	\end{proof}
	
	\subsection{Early exercise premium}\label{sect-EEP}
	We now extend to the stochastic volatility	Heston model a well known result in the Black and Scholes world, the so called \textit{early exercise premium formula}. It is an explicit formulation of the quantity $P-P_e$, where $P_e=P_e(t,s,y)$ is the European put price with the same strike price $K$ and maturity $T$ of the American option with price function $P=P(t,s,y)$.  Therefore, it represents the additional price you have to pay for the possibility of exercising before maturity.
	\begin{proposition}\label{EEP}
		Let $P_e(0,S_0,Y_0)$ be the European put price at time $0$ with maturity $T$ and strike price $K$. Then, one has
		\begin{equation*}
		P(0,S_0,Y_0)=P_e(0,S_0,Y_0)-\int_0^T e^{-rs}\E[(\delta S_s  -rK)\ind{\{S_s\leq b(s,Y_s)\}}]ds.
		\end{equation*}
	\end{proposition}

	The proof of Proposition \ref{EEP} relies on purely probabilistic  techniques and is based on the  results first  introduced in \cite{J}. Let $U_t=e^{-rt}P(t,S_t,Y_t)$ and $Z_t=e^{-rt}\varphi(S_t)$. Since $U_t$ is a supermartingale, we have the Snell decomposition 
	\begin{align}\label{doob_dec}
	& U_t=M_t-A_t,
	\end{align}
	where $M$ is a martingale and $A$ is a nondecreasing predictable process with $A_0=0$, continuous with probability 1 thanks to the continuity of $\varphi$. On the other hand, 
	\begin{align*}
	Z_t=e^{-rt}(K-S_t)_+&=Z_0-r\int_0^t e^{-rs}(K-S_s)_+ds -\int_0^t e^{-rs}\ind{\{S_s\leq K\}}dS_s +\int_0^te^{-rs}dL^K_s(S)\\
	&=m_t+a_t,
	\end{align*}
	where $L^K_t(S)$ is the local time of $S$ in $K$,
	$$
	m_t=Z_0 -\int_0^te^{-rs}\ind{\{S_s\leq K\}}S_s\sqrt{Y_s}dB_s
	$$  is a local martingale, and 
	$$
	a_t=-r\int_0^t e^{-rs}(K-S_s)_+ds -\int_0^t e^{-rs}\ind{\{S_s\leq K\}}S_s(r-\delta)ds +\int_0^te^{-rs}dL^K_s(S)
	$$
	is a predictable process with finite variation and $a_0=0$. Recall that $a_t$ can be written as the sum of an increasing and a decreasing component, that is $a_t=a_t^++a_t^-.$
	Since $(L^K_t)_t$ is increasing, we deduce that the decreasing process $(a_t^-)_t$ is absolutely continuous with respect to the Lebesgue measure, that is 
	$$
	da_t^- \ll dt.
	$$
	We denote by $k_t=k(t,S_t,Y_t)$ the  density of $a_t^-$ w.r.t. $dt$. 
	
	We now define $$\zeta_t=U_t-Z_t \geq 0.$$
	Thanks to  Tanaka's formula,
	\begin{equation*}\label{Tanaka}
	\zeta_t=	\zeta_t^+= \zeta_0 + \int_0^t \ind{\{\zeta_s >0\}}d\zeta_s+ \frac 12L^0_t(\zeta),
	\end{equation*}	
	where $L^0_t(\zeta)$ is the local time of $\zeta$ in $0$.
	%	This means that 
	%	\begin{equation*}
	%	\frac 12 L^0_t(\zeta)= \int_0^t \ind{\{\zeta_s=0\}}d\zeta_s.
	%	\end{equation*}
	%	On the other hand, \eqref{Tanaka} could be rewritten as
	Therefore,
	\begin{align*}
	\zeta_t&= \zeta_0 + \int_0^t \ind{\{\zeta_s >0\}}d(U_s-Z_s)+ \frac 12L^0_t(\zeta)\\
	&=\zeta_0 + \int_0^t \ind{\{\zeta_s >0\}}dM_s- \int_0^t \ind{\{\zeta_s >0\}}dm_s - \int_0^t \ind{\{\zeta_s >0\}}da_s+ \frac 12L^0_t(\zeta),
	\end{align*}	
	where the last equality follows from the fact that the process $A_t$ only increases on the set $\{\zeta_t=0\}$.
	Then, we can write
	\begin{align*}
	U_t&= U_0 + \bar{M}_t - \int_0^t \ind{\{\zeta_s >0\}}da_s+ \frac 12L^0_t(\zeta)+a_t  =U_0 + \bar{M}_t + \int_0^t \ind{\{\zeta_s =0\}}da_s+ \frac 12L^0_t(\zeta),
	\end{align*}
	where $\bar{M}_t=\int_0^t \ind{\{\zeta_s>0\}}d(M_s-m_s)+m_t$ is a local martingale. 
	Thanks to the continuity of $U_t$ we have the uniqueness of the decompositions, so
	\begin{equation}\label{Aa}
	-A_t= \int_0^t \ind{\{\zeta_s =0\}}da_s+ \frac 12L^0_t(\zeta).
	\end{equation}
	This means in particular that $\int_0^t \ind{\{\zeta_s =0\}}da_s+ \frac 12L^0_t(\zeta)$ is decreasing, but $L^0_t(\zeta)$ is increasing  so $-\int_0^t \ind{\{\zeta_s =0\}}da_s$ must be an increasing process and 
	\begin{equation*}
	\frac 12 dL^0_t(\zeta) \ll \ind{\{\zeta_t =0\}}da_t^-\ll dt.
	\end{equation*}
	We define $\mu_t$ the density of $\frac 12 L^0_t(\zeta) $ w.r.t. $dt$ and, by Motoo Theorem (see \cite{DM}), we can write $\mu_t=\mu(S_t,Y_t)$. 
	Moreover, let us consider the $t$-sections of the exercise region defined in \eqref{tsection}. We can easily prove the following Lemma.
	\begin{lemma}\label{lemmat}
		For any $t\in[0,T)$ we have 
		\begin{equation*}
		\Ex_t=\overline{	\mathring{\Ex_t}},
		\end{equation*}
		and
		$
		\mathring{\Ex}_t=\{(s,y)\in (0,\infty)\times[0,\infty): 0<s<b(t,y^+)\}\neq \emptyset,
		$
		where $b(t,y^+)=\lim_{y\rightarrow y^+}b(t,y)$. 
	\end{lemma}
	The proof is given in the Appendix for the sake of completeness.  Now, let us prove the following preliminary result.
	\begin{lemma}\label{local_time}
		The local time $L^0_t(\zeta)$ is indistinguishable from 0. 
	\end{lemma}
	\begin{proof}
		
		In order to simplify the notation, we set $L^0_t=L^0_t(\zeta)$ in this proof. We want to prove that 
		$$
		L^0_t=\int_0^t \ind{\{\zeta_s=0\}}dL^0_s=0.
		$$
		Note that, for $a\neq 0$, we have 
		$$
		\int_0^t\ind{\{\zeta_s=0\}}dL^a_s=0.
		$$
		Therefore, due to the right continuity of the local time with respect to $a$, we have 
		$$
		\int_0^t\ind{\{s\in \mathcal{O} \}}dL^0_s=0,
		$$
		where $\mathcal O$ is the interior of the the set $\{s\mid \zeta_s=0\}$, i.e.
		$$
		\mathcal O=\{  s\in (0,t)\mid \exists\epsilon>0, \forall\tau \in (s-\epsilon,s+\epsilon) \, \zeta_\tau=0   \}.
		$$
		We note that 
		\begin{equation}\label{oprimeino}
		\mathcal O'\subseteq \mathcal O,
		\end{equation}
		where
		$
		\mathcal O'=\{ s\in (0,t)\mid S_s<j(s,Y_s)  \},
		$
		with $
		j(s,y)=\sup_{\tau<s,\zeta>y}b(\tau,s).
		$

		In fact, if $S_s<j(s,Y_s)$, there exists $\tau<s$ and $\zeta>Y_s$ such that $S_s<j(\tau,\zeta)$. By the continuity of the trajectories, there exists $\epsilon >0$ such that 
		$$
		S_\theta <b(\tau,\zeta),\qquad \theta \in (s-\varepsilon,s+\varepsilon).
		$$
		Therefore, for $\theta \in (s-\varepsilon,s+\varepsilon)$ and $\theta$ near enough to $s$, we have $Y_\theta <\zeta$ and $\theta>\tau$, so that $b(\tau,\zeta)\leq b(\theta,Y_\theta)$and so $\zeta_\theta=0$.
		Therefore \eqref{oprimeino} is proved and we have
		$$
		\int_0^t\ind{\{S_s<j(s,Y_s)\}}dL^0_s=0.
		$$
		Now,
		\begin{align*}
		L^0_t&=\int_0^t \ind{\{\zeta_s=0\}}dL^0_s\\&=\int_0^t \ind{\{S_s\leq b(s,Y_s)\}}dL^0_s\\&\leq \int_0^t \ind{\{S_s< j(s,Y_s)\}}dL^0_s+\int_0^t \ind{\{j(s,Y_s)\leq S_s\leq b(s,Y_s)\}}dL^0_s\\&=\int_0^t \ind{\{j(s,Y_s)\leq S_s\leq b(s,Y_s)\}}dL^0_s\\&=\int_0^t \ind{\{j(s,Y_s)\leq S_s\leq b(s,Y_s)\}}\mu(S_s,Y_s)ds\\&=\int_0^tds \int  \ind{\{j(s,y)\leq x\leq b(s,y)\}}\mu(x,y)p(s,x,y)dxdy=0,
		\end{align*}
		if we can prove that $j(s,y)=b(s,y)$ $dsdy$ a.e. 	
		
		In order to prove this, note that $j(s,y)=\sup_{\tau<s}\left(\sup_{\zeta>y}b(\tau,\zeta)\right)$. For any fixed $\tau \geq 0$, we set
		$$
		b_+(\tau,y)=\sup_{\zeta>y}b(\tau,\zeta)=\lim_{n\rightarrow\infty}b\left(\tau,y+\frac 1 n \right),
		$$	
		since the function $y\mapsto b(\tau,y)$ is nonincreasing. On the other hand, $s\mapsto b(s,y)$ is nondecreasing, so 
		$$
		j(s,y)=\sup_{\tau<s}	b_+(\tau,y)=\lim_{n\rightarrow\infty} b_+\left(s-\frac 1 n ,y\right).
		$$
		Therefore, for any $y\geq 0$
		$$j(s,y)=b_+(s,y),\qquad ds\, a.e.
		$$
		and, for any $s>0$
		$$b_+(s,y)=b(s,y),\qquad dy\, a.e.
		$$
		so that
		$$
		j(s,y)=b(s,y),\qquad dsdy\, a.e.
		$$
		which concludes the proof.
		%	 	First of all, note that $L^0_t(\zeta)$ only increases on the set $\{(t,S_t,Y_t)\in \partial\mathcal{E}\}$.
		%	 	In fact, recall that  $L^a_t=\int_0^t\ind{\{U_s-Z_s=a\}}dL^a_s$  for every $a>0$ and $t>0$, so that		 
		%	 	$$
		%	 	\int_0^t\ind{\{(s,S_s,Y_s) \in \mathring{\mathcal{E}}\}}dL^a_s=0.
		%	 	$$
		%	 	Moreover it is well known that, for any $t>0$, $L^0_t=\lim_{a\rightarrow 0}L^a_t$, which implies that the measures $L^a_t$ weakly converge to $L^0_t$ as $a\rightarrow 0$. Then, we can deduce that 
		%	 	$$
		%	 	L^0_t(\{(s,S_s,Y_s) \in \mathring{\mathcal{E}}\})\leq \liminf L^a_t(\{(s,S_s,Y_s) \in \mathring{\mathcal{E}}\})=0.
		%	 	$$
		%	
		%	 	
		%	 	Moreover, thanks to Lemma \ref{lemmat}, we have
		%	 	$$
		%	 	\partial \mathcal E_t=\{(t,b(t,y),y) :(t,y)\in [0,T)\times [0,\infty)\}\cup \bigcup_{y\in\mathcal{D}_t}\{(s,y): b(t,y)< s\leq b(t,y^+)\},
		%	 	$$
		%	 	where $\mathcal{D}_t$ is the (numerable) set of the discontinuity points of $y\mapsto b(t,y)$ and the union is disjoint. Thanks to the continuity of $P$ and $\varphi$, it is easy to show that 		$	Leb\{ (t,b(t,y),y) :(t,y)\in [0,T)\times [0,\infty)\}=0$, so that $Leb \, \partial \mathcal  E_t= 0$ for any $t\in [0,T]$.
		%	 	Therefore,\begin{align*}
		%	 	\E[L^0_t(\zeta)]&= \E[\int_0^t\ind{\{U_s-Z_s=0\}}dL^0_s]= \E[\int_0^t\ind{\{(s,S_s,Y_s)\in \partial\mathcal{E}\}}\mu(s,S_s,Y_s)k(s,S_s,Y_s)ds]\\
		%	 	&=\int_0^tds\int_{\partial\mathcal{E}_s}dxdy \ \mu(s,x,y)k(s,x,y) p(s,x,y)=0.
		%	 	\end{align*}	
	\end{proof}
	We can now prove Proposition \ref{EEP}. 
	\begin{proof}[Proof of Proposition \ref{EEP}]
		Thanks to \eqref{Aa} and Proposition \ref{local_time}  we can rewrite \eqref{doob_dec} as
		\begin{align*}
		U_t&=M_t + \int_0^t \ind{\{U_s=Z_s\}}da_s=M_t+ \int_0^t e^{-rs}(\mathcal{L}-r)\varphi(S_s)\ind{\{S_s\leq b(s,Y_s)\}}ds,
		\end{align*}
		where the last equality derives from the application of the It\^{o} formula to the discounted payoff $Z$. In particular, we have
		\begin{align*}
		U_0= M_0=\E[ M_T]&=\E[U_T]-\E\left[\int_0^T e^{-rs}(\mathcal{L}-r)\varphi(S_s)\ind{\{S_s\leq b(s,Y_s)\}}ds\right]\\&=\E[U_T]-\int_0^T e^{-rs}\E[(\delta S_s  -rK)\ind{\{S_s\leq b(s,Y_s)\}}]ds.
		\end{align*}
		
		The assertion follows  recalling that $U_0=P(0,S_0,Y_0)$ and $\E[U_T]=\E[Z_T]=\E[e^{-rT}(K-S_T)_+]$, which corresponds to the price $P_e(0,S_0,Y_0)$ of an European put with maturity $T$ and strike price $K$. 
	\end{proof}
	\subsection{Smooth fit}\label{sect-sf}
	In this section we analyse the behaviour of the derivatives of the value function with respect to the $s$ and $y$ variables on the boundary of the continuation region.  In other words, we prove a weak formulation of the so called smooth fit principle. 
	
	In order to do this,  we need  two technical lemmas whose proofs can be found in the appendix. The first one is a general result about  the behaviour of the  trajectories of the CIR process.
	\begin{lemma}\label{lemmaY}
		For all $y\geq 0$ we have, with probability one,
		$$
		\limsup_{t\downarrow 0 }\frac{Y^y_t-y}{\sqrt{2t\ln\ln(1/t)}}=-	\liminf_{t\downarrow 0 }\frac{Y^y_t-y}{\sqrt{2t\ln\ln(1/t)}}=\sigma\sqrt{y}.
		$$
	\end{lemma}
	
	The second one is a result about the behaviour of the trajectories of a standard Brownian motion. 
	\begin{lemma}\label{lemmaBM}
		Let $(B_t)_{t\geq 0}$ be a standard Brownian motion and let  $(t_n)_{n\in\N}$ be a deterministic sequence of positive numbers with $\lim_{n\rightarrow\infty}t_n=0$. We have, with probability one, 
		
		\begin{equation}\label{detILL0}
		\liminf_{n\rightarrow \infty }\frac{B_{t_n}}{\sqrt{t_n}}=-\infty
		\end{equation}
	\end{lemma}
	
	%We stress that Lemma \ref{lemmaBM} in other words states that, for any deterministic sequence $t_n$ which tends to infinity,  $B_{t_n}>0$ infinitely often.
	
	We are now in a position to prove the following smooth fit result.
	\begin{proposition}\label{sfx}
		For any $(t,y)\in [0,T)\times [0,\infty)$ we have $\frac{\partial}{\partial s} P(t,b(t,y),y)=\varphi'(b(t,y))$.
	\end{proposition}
	\begin{proof}
		The general idea of the proof goes back to \cite{B} for the Brownian motion (see also \cite[Chapter 4]{PS}).  Without loss of generality we can fix $t=0$. 
		Note that, for $h>0$, since $b(0,y)-h\leq b(0,y)$, we have
		$$
		\frac{	P(0,b(0,y)-h,y)-P(0,b(0,y),y)}{h}= \frac{	\varphi(b(0,y)-h)-\varphi (b(0,y))}{h},		
		$$
		so that, since $\varphi$ is continuously differentiable near $b(0,y)$, $		\frac{\partial ^-}{\partial s}P(0,b(0,y),y)=\varphi'(b(0,y))$. 
		
		On the other hand, for $h>0$	small enough, 	since $P\geq \varphi$ and $P(0,b(0,y),y)=\varphi(b(0,y))$, we get
		\begin{align*}
		\frac{	P(0,b(0,y)+h,y)-P(0,b(0,y),y)}{h}\geq \frac{	\varphi(b(0,y)+h)-\varphi (b(0,y))}{h},
		\end{align*}
		so that 
		\begin{align*}
		\liminf_{h\downarrow0}  \frac{	P(0,b(0,y)+h,y)-P(0,b(0,y),y)}{h} \geq \varphi'(b(0,y)).
		\end{align*}
		Now, for the other inequality, we consider the optimal stopping time related to $P(0,b(0,y)+h,y)$, i.e.
		\begin{align*}
		\tau_h=\inf\{ t\in [0,T) \mid S_t^{0,b(0,y)+h,y}< b(t,Y^y_t)\}\wedge T=\inf\left \{ t\in [0,T) \mid M_t^{y}\leq\frac{b(t,Y^y_t)}{b(0,y)+h}\right\}\wedge T,
		\end{align*}
		where $M^y_t=S^{1,y}_t$.
		Recall that $P(0,b(0,y),y)\geq \E\big(e^{-r\tau_h}\varphi(b(0,y)M^y_{\tau_h})\big)$, so we can write
		\begin{align*}
		\frac{ 	P(0,b(0,y)+h,y)-P(0,b(0,y),y)}{h}&= \frac{\E \left(     e^{-r\tau_h}  \varphi((b(0,y)+h)M^y_{\tau_h}\right)  -   P(0,b(0,y),y)}{h}\\&\leq \E\left(   e^{-r\tau_h}   \frac{        \varphi\left((b(0,y)+h)M^y_{\tau_h}\right)  -   \varphi\left(b(0,y)M^y_{\tau_h}\right)}{h}     \right).
		\end{align*}
		Assume for the moment  that 
		\begin{equation}\label{ot}
		\lim_{h\rightarrow 0 }\tau_h=0,\quad a.s.
		\end{equation}
		so we have
		$$
		\lim_{h\downarrow 0 }   \frac{        \varphi((b(0,y)+h)M^y_{\tau_h})\big)  -   \varphi(b(0,y)M^y_{\tau_h})}{h}=\varphi'(b(0,y)).
		$$
		Moreover, recall that $ M^y_{\tau_h}\leq \frac{b(t,Y^y_t)}{b(0,y)+h}\leq \frac K {b(0,y)}$ if $\tau_h<T$ and $ M^y_{\tau_h}=M^y_T$ if $\tau_h=T$. Therefore, by using the fact that $\varphi$ is Lipschitz continuous and the dominated convergence, we obtain
		$$
		\limsup_{h\downarrow 0 } \frac{ 	P(0,b(0,y)+h,y)-P(0,b(0,y),y)}{h} \leq \varphi'( b(0,y))
		$$
		and the assertion is proved. 
		
		It remains to prove \eqref{ot}. Since $t\mapsto b(t,y) $ is nondecreasing, if $	M^y_t< \frac{b(0,y)}{b(0,y)+h}$ and $Y^y_t= y$,
		%		$$
		%		M^y_t< \frac{b(0,y)}{b(0,y)+h} \quad \mbox{ and } \quad Y^y_t= y,
		%		$$
		we have
		$$
		M^y_t< \frac{b(0,y)}{b(0,y)+h} \leq \frac{b(t,Y^y_t)}{b(0,y)+h},
		$$
		so that
		\begin{equation}
		\tau_h \leq 	\inf \left\{  t\geq 0 \mid  M^y_t< \frac{b(0,y)}{b(0,y)+h} \mbox{ \& } Y^y_t= y \right\}.
		\end{equation}
		We now show that we can find a sequence $t_n \downarrow 0 $ such that  $Y^y_{t_n}=0$ and $ M^y_{t_n}<1$.
		First, recall that with a standard transformation we can write
		\begin{equation}\label{incorrelation}
		\begin{cases}
		\frac{dS_t}{S_t}=(r-\delta)dt +  \sqrt{Y_t}(\sqrt{1-\rho^2}d\bar W_t+\rho dW_t),\qquad&S_0=s>0,\\ dY_t=\kappa(\theta-Y_t)dt+\sigma\sqrt{Y_t}dW_t, &Y_0=y\geq 0,
		\end{cases}
		\end{equation}
		where $\bar W$ is a standard  Brownian motion independent of $W$.	Set $\Lambda_t^y=\ln M^y_t$. We deduce from Lemma \ref{lemmaY} that there exists a sequence  $t_n\downarrow 0$ such that  $Y^y_{t_n} =y$ $\P_y$-a.s. . Therefore,  from \eqref{incorrelation} we can write  $\int_0^{t_n}\sqrt{Y^{y}_s}dW_s=-\frac \kappa \sigma \int_0^{t_{n}}(\theta-Y^{y}_s)ds$ for all $n\in \N$. So, we have
		$$
		\Lambda^y_{t_{n}}=(r-\delta)t_{n}-\int_0^{t_{n}} \frac{Y^y_s}{2}   ds +\sqrt{1-\rho^2}\int_0^{t_{n}}\sqrt{Y^y_s} d\bar W_s -\frac {\rho\kappa} \sigma \int_0^{t_{n}}(\theta-Y^{y}_s)ds .
		$$
		Conditioning with respect to $W$  we have 
		\begin{align*}
		\liminf_{n\rightarrow \infty} \Lambda_{t_n}^y=&\liminf_{n\rightarrow\infty}	\frac{(r-\delta)t_{n}}{\sqrt{\int_0^{t_{n}}Y^{y}_sds}}-\frac{\int_0^{t_{n}} \frac{Y^y_s}{2}   ds}{   \sqrt{\int_0^{t_{n}}Y^{y}_sds}        } +\frac{\sqrt{1-\rho^2}\int_0^{t_{n}}\sqrt{Y^y_s} d\bar W_s  }{\sqrt{\int_0^{t_{n}}Y^{y}_sds} }- \frac{\frac {\rho\kappa} \sigma \int_0^{t_{n}}(\theta-Y^{y}_s)ds }{\sqrt{\int_0^{t_{n}}Y^{y}_sds} }\\
		&=\liminf_{n\rightarrow\infty}	\frac{(r-\delta)t_{n}}{\sqrt{\int_0^{t_{n}}Y^{y}_sds}}-\frac{\int_0^{t_{n}} \frac{Y^y_s}{2}   ds}{   \sqrt{\int_0^{t_{n}}Y^{y}_sds}        } +\frac{\sqrt{1-\rho^2}\tilde  W_{ \int_0^{t_{n}}Y^y_sds }}{\sqrt{\int_0^{t_{n}}Y^{y}_sds} }- \frac{\frac {\rho\kappa} \sigma \int_0^{t_{n}}(\theta-Y^{y}_s)ds }{\sqrt{\int_0^{t_{n}}Y^{y}_sds} }=-\infty,
		\end{align*}
		where  we have used the Dubins-Schwartz Theorem and we have applied Lemma \ref{lemmaBM}  to the standard Brownian motion $\tilde W$ and the sequence  $\sqrt{\int_0^{t_{n}}Y^{y}_sds}$ which can be considered deterministic.
		
		We deduce that, up to extract a subsequence of $t_n$, we have $\Lambda^y_{t_n}<0$ and,  as a consequence, $M^y_{t_n}<1$. Therefore, for any any fixed $n$, there exists  $h$ small enough such that 
		$M^y_{t_{n}} <  \frac{b(0,y)}{b(0,y)+h}$  so that, by definition,  $\tau_h\leq t^n$. We conclude the proof passing to the limit as $n$ goes to infinity.

	\end{proof}
	As regards the derivative with respect to the $y$ variable, we have the following result.
	\begin{proposition}\label{prop _sfy}
		If $2\kappa\theta\geq \sigma^2$, for any $(t,y)\in [0,T)\times (0,\infty)$ we have  $\frac{\partial}{\partial y} P(t,b(t,y),y)=0$.
	\end{proposition}
	\begin{proof}
		Again we fix $t=0$ with no loss of generality. Since $y\rightarrow P(t,s,y)$ in nondecreasing, for any $h>0$ we have $P(0,b(0,y),y-h)\leq P(0,b(0,y),y)=\varphi(b(0,y))$ so that  $P(0,b(0,y),y-h)=\varphi(b(0,y))$. Therefore,
		\begin{align*}
		\frac{	P(0,b(0,y),y-h)-P(0,b(0,y),y)}{h}= 0,
		\end{align*}
		hence $\frac{\partial^-}{\partial y}P(0,b(0,y),y)=0$.
		On the other hand, since $y\mapsto P(t,x,y)$ is nondecreasing, for any $h>0$ we have
		%		\begin{align*}
		%		\frac{	P(0,b(0,y),y+h)-P(0,b(0,y),y)}{h}\geq 0, 
		%		\end{align*}
		%		so that 
		\begin{align*}
		\liminf_{h\downarrow 0}	\frac{	P(0,b(0,y),y+h)-P(0,b(0,y),y)}{h}\geq 0, 
		\end{align*}
		To prove the other inequality, we consider the stopping time related to $P(0,b(0,y),y+h)$, that is
		$$
		\tau_h=\inf\left\{ t\in [0,T) \mid S_t^{0,b(0,y),y+h}<b(t,Y^{y+h}_t)\right \}\wedge T=\inf\left \{ t\in [0,T)  \mid M_t^{y+h}<\frac{b(t,Y^{y+h}_t)}{b(0,y)}\right\}\wedge T
		$$
		and we assume for the moment that
		\begin{equation}\label{limtauh}
		\lim_{h\rightarrow0}\tau_h=0.
		\end{equation}
		We have
		\begin{equation}\label{stimadery}
		\begin{split}
		\frac{ 	P(0,b(0,y),y+h)-P(0,b(0,y),y)}{h}&= \frac{\E \left(     e^{-r\tau_h}  \varphi\left(b(0,y)M^{y+h}_{\tau_h}\right)\right)  -   P(0,b(0,y),y)}{h}\\
		&\leq \E\left[  e^{-r\tau_h}   \frac{        \varphi\left(b(0,y)M^{y+h}_{\tau_h}\right)  -   \varphi(b(0,y)M^y_{\tau_h})}{h}     \right]\\
		&\leq K    \frac{    \E  \left[ \left| M^{y+h}_{\tau_h}  -   M^y_{\tau_h}\right|\right]}{h},
		\end{split}
		\end{equation}
		where the last inequality follows from the fact that $\varphi$ is Lipschitz continuous  and $b(0,y)\leq K$.

		Now,  if the Feller condition $2\kappa\theta\geq\sigma^2$ is satisfied,   we can write
		$$
		M^{y+h}_{t}  -   M^y_{t}=\int_y^{y+h} \left(\int_0^t \frac{\dot{Y}^\zeta_s}{2\sqrt{Y^\zeta_s}}dB_s-\frac 1 2 \int_0^t \dot{Y}^\zeta_sds   \right) e^{(r-\delta)t-\int_0^t \frac {Y^\zeta_s}2 ds+\int_0^t \sqrt{Y^\zeta_s}dB_s   } d\zeta.
		$$
		Fix $\zeta$ and observe that the exponential process $e^{-\int_0^t \frac {Y^\zeta_s}2 ds+\int_0^t \sqrt{Y^\zeta_s}dB_s   } $ satisfies the assumptions of the Girsanov Theorem, namely it is a martingale. Therefore,  we can introduce a new probability measure $\hat \P$ under which the process $\hat W_t=W_t-\int_0^t\sqrt{Y_s}ds$ is a standard Brownian motion. If we denote by $\hat \E$ the expectation under the probability $\hat \P$, substituting in \eqref{stimadery} and using \eqref{stimaperholder} we get
		\begin{align*}
		&\frac{ 	P(0,b(0,y),y+h)-P(0,b(0,y),y)}{h}
		\leq\frac{e^{rT}K}{h}\int_y^{y+h} d\zeta\hat{\E}\left[\left|\int_0^{\tau_h} \frac{\dot{Y}^\zeta_s}{2\sqrt{Y^\zeta_s}}d\hat{W_s}\right|\right]\\ %&\leq\frac{CK}{h}\int_y^{y+h} d\zeta\left(\hat{\E}\left[\left|\int_0^{\tau_h} \frac{\dot{Y^\zeta_s}}{2\sqrt{Y^\zeta_s}}d\hat{W_s}\right|\right]^2\right)^{1/2}\\
		&\quad\leq \frac{e^{rT}K}{h}\int_y^{y+h} d\zeta\left(\hat{\E}\left[\int_0^{\tau_h} \left(\frac{\dot{Y^\zeta_s}}{2\sqrt{Y^\zeta_s}}\right)^2ds\right]\right)^{1/2}
		\leq \frac {e^{rT}K} h \int_y^{y+h}\frac{1}{2\sqrt \zeta}\hat{\E}[\sqrt{\tau_h}]d\zeta
		\end{align*}
		which tends to $0$ as $h$ tends to $0$.
		
		Therefore, as in the proof of Proposition \ref{sfx},  it remains to prove that $\lim_{h\downarrow 0}\tau_h= 0$. In order to do this, we can proceed as follows.
		Again, set
		$$
		\Lambda^y_t=\ln (M^y_t)=(r-\delta)t-\frac 1 2 \int_0^t
		Y^y_sds+\int_0^t\sqrt{Y^y_s}dW_s,
		$$
		so that 
		$$
		\tau_h=\inf\left\{t\in [0,T)\mid \Lambda_t^{y+h}  \leq \ln\left(\frac{ b(t,Y^{y+h}_t)}{b(0,y)}\right) \right \}\wedge T.
		$$
		We deduce from Lemma \eqref{lemmaY} that, almost surely, there exist two sequences $(t_n)_n$ and $(\hat t_n)_n$ which converge to 0 with $0<t_n<\hat t_n$ and such that
		$$
		Y^y_{t_n}=y, \qquad \mbox{ and, for }t\in (t_n,\hat t_n),\quad Y_t<y. 
		$$
		In fact, it is enough to consider a sequence $(\hat t_n)_n$  such that $\lim_{n\rightarrow\infty }\hat t_n=0$ and $Y_{\hat t_n}<y$ and define $t_n=\sup\{t\in [0,\hat t_n) \mid Y^y_t=y   \}$.
		
		Proceeding as in the proof of Proposition \ref{sfx}, up to extract a subsequence we can assume 
		$$
		\Lambda^y_{t_n}<0.
		$$
		
		On the other hand, up to extract a subsequence of $h$ converging to $0$, we can assume that, almost surely,
		\[
		\lim_{h\downarrow 0}\sup_{t\in[0,T]}\left|Y^{y+h}_t-Y^y_t\right|=\lim_{h\downarrow 0}\sup_{t\in[0,T]}\left|\Lambda^{y+h}_t-\Lambda^y_t\right|=0.
		\]

		Now, let us fix $n\in\N$. For $h$ small enough,  there exists $\delta>0$ such that
		\[
		\Lambda^{y+h}_{t}<0, \qquad t\in (t_n-\delta, t_n+\delta).
		\]
		Then, for any $\tilde{t}_n\in (t_n-\delta, t_n+\delta)\cap (t_n,\hat t_n)$, we have at the same time
		$\Lambda^{y+h}_{\tilde{t}_n}<0$ and, since $Y^{y}_{\tilde{t}_n}<y$, $Y^{y+h}_{\tilde{t}_n}<y$ for $h$ small enough.
		Recalling that $t\mapsto b(t,y)$ is nondecreasing and $y\mapsto b(t,y)$ is nonincreasing, we deduce that
		\[
		b(\tilde{t}_n,Y^{y+h}_{\tilde{t}_n})\geq b(0,Y^{y+h}_{\tilde{t}_n})\geq b(0,y).
		\]
		
		Therefore
		\[
		\Lambda^{y+h}_{\tilde{t}_n}\leq \ln\left(\frac{b(\tilde{t}_n,Y^{y+h}_{\tilde{t}_n})}{b(0,y)}\right)
		\]
		and, as a consequence, $\tau_h\leq \tilde{t}_n\leq \hat t_n$ so \eqref{limtauh} follows.

	\end{proof}

	\section{Appendix: some proofs}

	We devote the appendix to the proof of some technical results used in this chapter.
	\subsection{Proofs of Section \ref{sect-monotony}}
	\begin{proof}[Proof of Lemma \ref{lemmasup_diff}]
		Consider $1 > a_1 > a_2 >\dots >a_m >\dots >0$  defined by 
		$$ 	
		\int^1_{a_1} \frac 1 u du =1, \dots, \int_{a_m}^{a_{m-1}} \frac 1 u du = m, \ \dots .
		$$
		We have that $a_m$ tends to $0 $ as $m$ tends to infinity. Let $(\eta_m)_{m\geq 1}$, be a family of continuous functions such that $$\supp \eta_m\subseteq (a_m, a_{m-1}), 
		\quad
		0 \leq \eta_m(u) \leq \frac{2}{um}, \quad  \int_{a_m}^{a_{m-1}}\eta_m(u)du=1.
		$$
		Moreover, we set
		$$
		\phi_m(x) := \int_0^{|x|}dy \int_0^y \eta_m(u)du, \qquad x \in \R .
		$$
		It is easy to see  that $\phi_m \in C^2(\R)$, $|\phi_m^{'} | \leq 1$ and $\phi_m(x)\uparrow |x| $ as $m \rightarrow \infty$.
		Fix $t\in[0,T]$.	Applying It\^{o}'s formula and passing to the expectation we have, for any $m \in \N$,
		\begin{equation}\label{II}
		\begin{split}
		\E[\phi_m(Y^n_t-Y_t)]&= \kappa\int_{0}^{t}  \E\left[\phi_m^{'}(Y^n_s-Y_s)(Y_s- f^2_n(Y^n_s))\right]ds \\&\quad+ \frac {\sigma^2} 2 \int_{0}^{t}   \E \left[       \phi_m^{''}(Y^n_s-Y_s)(f_n(Y^n_s)- \sqrt{Y_s})^2  \right]ds %\\&=\kappa\int_{0}^{t} \E\left[\phi_m^{'}(Y^n_s-Y_s)(f^2(Y_s)- f^2_n(Y^n_s))\right]ds + \frac {\sigma^2} 2 \int_{0}^{t}    \E \left[       \phi_m^{''}(Y^n_s-Y_s)(f_n(Y^n_s)- f(Y_s))^2  \right]ds .
		\end{split}
		\end{equation}
		Let us analyse the  right hand term in \eqref{II}. Since $|\phi_m^{'}|\leq 1$, we have
		\begin{align*}
		&  \left| \kappa \int_{0}^{t} \E\left[\phi_m^{'}(Y^n_s-Y_s)(Y_s- f_n^2(Y^n_s))\right]ds \right|
		\leq \kappa \int_{0}^{t} \E\left[ |f_n^2(Y^n_s)-Y^n_s| \right]ds + \kappa \int_{0}^{t} \E\left[ | Y^n_s- Y_s|\right]ds
		\end{align*}
		On the other hand,
		\begin{align*}
		&\left|\frac {\sigma^2} 2 \int_{0}^{t}    \E \left[       \phi_m^{''}(Y^n_s-Y_s)(f_n(Y^n_s)- \sqrt{Y_s})^2\right]ds\right|   \\
		& \quad \leq  \sigma^2\int_{0}^{t}    \E \left[     |  \phi_m^{''}(Y^n_s-Y_s)|(f_n(Y^n_s)- \sqrt{Y^n_s})^2]ds  \right] + \sigma^2 \int_{0}^{t}    \E \left[  |     \phi_m^{''}(Y^n_s-Y_s)|(\sqrt{Y^n_s}- \sqrt{Y_s})^2\right]ds \\
		&\quad \leq \sigma^2\int_{0}^{t}    \E \left[    \frac{2}{m|Y^n_s-Y_s|}(f_n(Y^n_s)- \sqrt{Y^n_s})^2    \ind{\{a_m\leq Y^n_s-Y_s\leq a_{m-1}\}} ]ds  \right] \\&\qquad+ \sigma^2\int_0^t \E \left[  \frac{2}{m|Y^n_s-Y_s|}  |Y^n_s-Y_s|\right]ds\\
		&\quad\leq\frac{2\sigma^2}{ma_m}  \int_{0}^{t}    \E \left[    (f_n(Y^n_s)- \sqrt{Y^n_s})^2  ]ds  \right]  + \frac{2\sigma^2t}{m}.
		\end{align*}
		Observe that, if $|x| \geq a_{m-1}$,
		$$
		\phi_m(x)\geq \int_{a_{m-1}}^{|x|}dy= |x|-a_{m-1}.
		$$
		Therefore,  for any $m$ large enough,  
		\begin{align*}
		\E[\left| Y^n_t-Y_t\right|] &\leq  \kappa   \int_0^t \E[| Y^n_s-Y_s |]ds+ \kappa \int_{0}^{t} \E\left[ |f_n^2(Y^n_s)-Y^n_s| \right]ds\\&\qquad+\frac{2\sigma^2}{ma_m}  \int_{0}^{t}    \E \left[    (f_n(Y^n_s)- \sqrt{Y^n_s})^2  ]ds  \right]  +  \frac{2\sigma^2t}{m}+ a_{m-1}.
		\end{align*}
		Recall that $f_n(y)\rightarrow f(y)\equiv y$ locally uniformly and that $Y^n$ has continuous paths. Moreover, since $f^2_n(x)\leq A(|x|+1)$ with $A$ independent of $n$, it is easily to see that for any $p>1$ there exists $C>0$ independent of $n$ such that
		\begin{equation}\label{moments}
		\E\left[\sup_{t\in [0,T]} |Y^n_t|^p  \right]\leq C.
		\end{equation}
		Therefore, by using Lebesgue's Theorem and recalling that $\lim_{m\rightarrow\infty}a_m=0$, 
		we deduce that for any $\delta >0 $ it is possible to choose 
		%$\bar{m}$ and, then,  
		$\bar n$ such that for every $ n\geq \bar n$
		$$
		\E[\left| Y^n_t-Y_t\right|] <C \int_{0}^{t} \E[\left| Y^n_s-Y_s\right|]  + \delta. $$ 
		We can now apply  Gronwall's inequality and we deduce  that 
		$
		\E[\left| Y^n_t-Y_t\right|] < \delta e^{Ct} ,$
		so  that \begin{equation}\label{convfixedt}
		\lim_{n\rightarrow \infty} \E[\left| Y^n_t-Y_t\right|] = 0 
		\end{equation} from the arbitrariness of $\delta$.  
		
		Now, note that 
		\begin{equation}\label{relsup}
		\sup_{t \in [0,T]}|Y^n_t-Y_t|\leq \kappa \int_0^T|Y_s-Y^n_s|ds  + \sup_{t \in [0,T]}\left|    \int_0^t  (\sqrt{Y_s}-f_n(Y^n_s))dW_s\right|
		\end{equation}
		The first term in the right hand side of \eqref{relsup} converges to 0 in probability  thanks to \eqref{convfixedt}, so it is enough to prove that the second term converges to 0. We have
		\begin{equation}\label{stimasup}
		\E\left[    \sup_{t \in [0,T]}\left|    \int_0^t  (\sqrt{Y_s}-f_n(Y^n_s))dW_s\right|  \right]\leq\left(\int_0^T \E[ |\sqrt{Y_s}-f_n(Y^n_s)|^2]ds\right)^{\frac 1 2 }
		\end{equation}
		and 
		\begin{align*}
		\E\left[ |\sqrt{Y_s}-f_n(Y^n_s)|^2	\right]&\leq 2\E\left[ |\sqrt{Y_s}-\sqrt{Y^n_s}|^2\right]+2\E\left[ |\sqrt{Y^n_s}-f_n(Y^n_s)|^2\right]\\&\leq 2\E\left[ |Y_s-Y^n_s|\right]+2\E\left[ |\sqrt{Y^n_s}-f_n(Y^n_s)|^2\right].
		\end{align*}
		Therefore, we can conclude that \eqref{stimasup} tends to 0 as $n$ goes to infinity by using \eqref{convfixedt} and the Lebesgue Theorem so that \eqref{supY} is proved.

		As regards \eqref{1}, for every $n\in\N$ we have
		$$
		X^n_t= x +\int_0^t \left( r-\delta-\frac {f_n^2(Y^n_s)} 2 \right)ds + \int_0^t  f_n(Y^n_s)dB_s,
		$$
		% 	so 
		% 	$$
		% 	|X^n_t-X_t|\leq \int_0^t \frac 1 2 |f^2_n(Y^n_s)- f^2(Y_s)| ]ds +\left| \int_0^t f(Y^n_s)- f(Y_s) dW_s  \right|.
		% 	$$
		so that
		\begin{equation}\label{calcoloperX}
		\sup_{t \in [0,T]}|X^n_t-X_t|  \leq \frac 12 \int_0^T |f^2_n(Y^n_s)- Y_s| ds + 	\sup_{t \in [0,T]}  \left|\int_0^t (f_n(Y^n_s)- \sqrt{Y_s} )dB_s \right|    .
		\end{equation}
		It is enough to show that the two terms in the right hand side  of \eqref{calcoloperX} converge to 0 in probability.
		
		Concerning the first term, note that,
		% 	 we know that $\sup_{t\in [0,T]} |Y^n_t-Y_t|\rightarrow 0 $ in probability and that $f^2_n\rightarrow f^2$ locally uniformly, and we want to show that $ \int_0^T|f^2_n(Y^n_s)- f^2(Y_s)| ds  \rightarrow 0 $ in probability.  Up to consider a subsequence, we can assume that  $\sup_{t\in [0,T]} |Y^n_t-Y_t|\rightarrow 0 $ a.s.. 
		since $Y$ has continuous paths, for every $ \omega \in \Omega, \, Y_{[0,T]}(\omega)$ is a compact set and 
		$
		K:=\{    x | d(x,Y_{[0,T]}) \leq 1     \}
		$
		is compact as well.
		For $n$ large enough,  $Y^n$ lies in $K$, so 
		\begin{align*}
		& \int_0^T |f^2_n(Y^n_s)- f^2(Y_s)| ds  \leq  \int_0^T |f^2_n(Y^n_s)- f^2(Y^n_s)| ds + \int_0^T |f^2(Y^n_s)- f^2(Y_s)|ds,
		\end{align*}
		which goes to $0$	as $n$ tends to infinity, since $f^2_n\rightarrow f^2$ locally uniformly and $f^2$ is a continuous function. 
		
		On the other hand, for the second term in the right hand side of \eqref{calcoloperX}, we have
		$$
		\E\left[	\sup_{t \in [0,T]}  \left|\int_0^t f(Y^n_s)- \sqrt{Y_s} dW_s \right|\right]\leq \left(  \int_0^T \E[(f(Y^n_s)- \sqrt{Y_s}  )^2] ds \right)^{\frac 1 2}
		$$
		and we can prove with the usual arguments that the last term goes to 0.

	\end{proof}
	\subsection{Proofs of Section \ref{sect-put}}
	\begin{proof}[Proofs of Lemma \ref{support}]
		To simplify the notation we pass to the logarithm and we prove the assertion for the pair $(X,Y)$. We can get rid of the correlation between the  Brownian motions  with a standard transformation, getting
		\begin{equation*}
		\begin{cases}
		dX_t=(r-\delta-\frac 1 2 Y_t)dt +  \sqrt{Y_t}(\sqrt{1-\rho^2}d\bar W_t+\rho dW_t),\qquad&X_0\in\R,\\ dY_t=\kappa(\theta-Y_t)dt+\sigma\sqrt{Y_t}dW_t, &Y_0\geq 0,
		\end{cases}
		\end{equation*}
		where $\bar W$ is a standard Brownian motion independent of $W$. Moreover, from the SDE satisfied by $Y$ we deduce $\int_0^t \sqrt{Y_s}dW_s=\frac 1 \sigma\left(Y_t-Y_0-\int_0^t\kappa(\theta-Y_s)ds\right)$. Conditioning with respect to $Y$, we reduce to prove that, for every continuous function $m: [0,T]\rightarrow \R$ such that $m(0)=X_0$ and for every $\epsilon>0$ we have
		\begin{equation}
		\label{supX}
		\P\left(\sup_{t\in [0,T]}|X_t-m(t)|<\epsilon\mid Y \right)>0,
		\end{equation}
		and
		\begin{equation}
		\label{supY}
		\P\left(\sup_{t\in [0,T]}|Y_t-Y_0|<\epsilon \right)>0.
		\end{equation}
		As regards  \eqref{supX}, by using the Dubins-Schwartz Theorem, there exists a Brownian motion $\tilde W$ such that 
		\begin{align*}
		&\P\left(\! \sup_{t\in [0,T]}\left |x+\!\! \int_0^t \!\!\left(r-\delta-\frac{Y_s} 2-\frac {\rho\kappa} \sigma (\theta-Y_s) \right) \!ds    +\frac \rho \sigma(Y_t-y) +\sqrt{1-\rho^2} \! \int_0^t \!\!  \sqrt{Y_s}d\bar W_s  -m(t)\right|<\epsilon\mid Y \right)\\
		&=\P\left(\sup_{t\in [0,T]}\left|\sqrt{1-\rho^2}  \!\!\! \int_0^t\sqrt{Y_s}d\bar W_s  -\tilde m(t)\right|<\epsilon\mid Y \right)\\
		&=\P\left(\sup_{t\in [0,T]}\left|\sqrt{1-\rho^2}  \tilde W_{\int_0^tY_sds}  -\tilde m(t)\right|<\epsilon\mid Y \right),
		\end{align*}
		where  $\tilde m(t)=m(t)-x-\int_0^t \left(r-\delta-\frac{Y_s} 2-\frac {\rho\kappa} \sigma (\theta-Y_s) \right) ds    -\frac \rho \sigma(Y_t-y) $ is a continuous function which, conditioning w.r.t. $Y$, can be considered deterministic. Then, \eqref{supX} follows by the support theorem for Brownian motions.
		
		In order to prove \eqref{supY}, we distinguish two cases.  Assume  first that $Y_0=y_0>0$ and, for $a\geq 0$, define the stopping time
		$$
		T_{a}=\inf \left\{ t>0\mid Y_t=a 	\right \}.
		$$ 
		Moreover, let us consider the function 
		$$
		\eta(y)=\begin{cases}
		\sqrt{y},\qquad&\mbox{ if } y> \frac{y_0}{2},\\
		\frac{\sqrt{y_0}} 2 \qquad&\mbox{ if } y\leq \frac{y_0}{2},
		\end{cases}
		$$
		and the process $(\tilde Y_t)_{t\in[0,T]}$, solution to the uniformly elliptic  SDE
		$$
		d\tilde Y_t=\kappa(\theta-\tilde Y_t)dt+\sigma\eta(\tilde Y_t)dW_t,\qquad \tilde Y_0=Y_0.
		$$
		It is clear that $Y_t=\tilde Y_t$ on the set $\left\{t\leq T_{\frac {y_0} 2}\right\}$  so  we have, if $\epsilon<\frac{y_0}2$,
		\begin{align*}
		\P\left(\sup_{t\in [0,T]}|Y_t-Y_0|<\epsilon \right)= \P\left(\sup_{t\in [0,T]}|\tilde Y_t-Y_0|<\epsilon \right),
		\end{align*}
		where the last inequality follows from the classical Support Theorem for  uniformly elliptic diffusions (see, for example,  \cite{SV1}).
		
		On the other hand, if we assume $Y_0=0$, then we can write
		$$
		\P\left(\sup_{t\in [0,T]} Y_t<\epsilon \right)=\P\left(  T_{\frac \epsilon 2}  \geq T \right)+\P\left(  T_{\frac \epsilon 2}  < T , \forall t\in \left[T_{\frac \epsilon 2}  ,T\right] Y_t<\epsilon \right).
		$$
		Now, if $\P\left( T_{\frac \epsilon 2} < T\right)>0$, we can deduce that the second term in the right hand side is positive using the strong Markov property and the same argument we have used before  in the case with $Y_0\neq 0$. Otherwise, $\P\left(  T_{\frac \epsilon 2}  \geq T \right)=1$ which concludes the proof.
	\end{proof}
	
	\begin{proof}[Proof of Lemma \ref{lemmat}]
		Let us define $\tilde{\Ex}_t=\{(s,y)\in (0,\infty)\times[0,\infty): s<b(t,y^+)\}$. Note that $\etildet\neq \emptyset$ since $b>0$.
		We first show that $\overline{ \tilde{\Ex}_t}=\Ex_t$. If $(s,y)\in \etildet$, then $s<b(t,y^+)\leq b(t,y)$, since $y\mapsto b(t,y)$ is nonincreasing. Therefore, $ \tilde{\Ex}_t\subseteq \Ex_t$ so that, since $\mathcal E_t$ is closed, $\overline{ \tilde{\Ex}_t}\subseteq \Ex_t$ .
		
		On the other hand, let $(s,y)\in \Ex_t$ and consider the sequence $((s_n,y_n))_n=((s- 1/ n,y-1 /n))_n$. Then, $(s_n,y_n)\rightarrow (s,y)$ and we  prove that $(s_n,y_n)\in \etildet$, so that $(s,y)\in 	\overline{\etildet}$. In fact, for each $n\in \N$,  we can consider the sequence $((s_{n,k},y_{n,k}))_{k>n}, = \left(\left(s-\frac 1 n+\frac 1 k,y-\frac 1 n+\frac 1 k\right)\right)_{k>n}$. We have
		$$
		s_{n,k}=	s-\frac 1 n+\frac 1 k<s\leq b(t,y)\leq b\left(t,y-\frac 1 n+\frac 1 k\right) =b\left(t,y_{n,k}\right).
		$$
		Letting $k$ tends to infinity, we get
		$$
		s_n<s\leq b(t,y_n^+),
		$$
		hence $(s_n,y_n)\in \etildet$, and the assertion is proved.
		
		Then, we show that $\etildet=\mathring{\Ex}_t$. Note that $\etildet$ is an open set, since the function $(s,y)\mapsto b(t,y^+)-s$ is lower semicontinuous. Therefore $\etildet\subseteq \mathring{\Ex}_t$. Let us now consider an open set $A\subseteq \Ex_t$. Fix $(s,y)\in A$, then $\left(s+\frac 1 n,y+\frac 1 n\right)\in A$ for $n$ large enough. Therefore,
		$$
		s<s+\frac 1 n \leq b\left( t, y+\frac 1 n\right)\leq b(t,y^+),
		$$ 
		hence $(s,y)\in \etildet$.
	\end{proof}
	\begin{proof}[Proof of Lemma \ref{lemmaY}]	
		We have
		\begin{align*}
		Y^y_t-y&=\kappa\int_0^t(\theta-Y^y_s)ds+\sigma\int_0^t\sqrt{Y^y_s}dW_s\\
		&=\sigma\sqrt{y}W_t+\kappa\int_0^t(\theta-Y^y_s)ds+\sigma\int_0^t\left(\sqrt{Y^y_s}-\sqrt{y}\right)dW_s,
		\end{align*}
		so it is enough to prove that, if $(H_t)_{t\geq 0}$ is a predictable process such that 
		$\lim_{t\downarrow 0} H_t=0$ a.s., we have
		\[
		\lim_{t\downarrow 0}\frac{\int_0^tH_s dWs}{\sqrt{2t\ln\ln(1/t)}}=0 \mbox{ p.s.}
		\]
		This follows by using standard arguments, we include a proof for the sake of completeness. By using   Dubins-Schwartz inequality we deduce that, if  $f(t)=\sqrt{2t\ln\ln(1/t)}$, for $t$ near to $0$ we have
		\[
		\left|\int_0^tH_s dWs\right|\leq Cf\left(\int_0^tH_s^2ds\right).
		\]
		Let us consider $\varepsilon>0$. For $t$ small enough, we have $\int_0^tH_s^2ds\leq \varepsilon t$ and, since $f$ increases near $0$,
		\[
		\left|\int_0^tH_s dWs\right|\leq Cf\left(\varepsilon t\right).
		\]
		We have
		\begin{align*}
		\frac{f^2(\varepsilon t)}{f^2(t)}&=\frac{\varepsilon t\ln\ln(1/\varepsilon t)}{t\ln\ln(1/t)}=\varepsilon\frac{\ln\left(\ln(1/t)+\ln(1/\varepsilon)\right)}{\ln\ln(1/t)}\\&\leq \varepsilon\frac{\ln\left(\ln(1/t)\right)+\frac{\ln(1/\varepsilon)}{\ln(1/t)}}{\ln\ln(1/t)}=\varepsilon\left(1+\frac{\ln(1/\varepsilon)}{\ln(1/t)\ln\ln(1/t)}\right),
		\end{align*}
		where we have used the inequality $\ln(x+h)\leq \ln(x)+\frac{h}{x}$ (for $x,h>0$).
		Therefore $\limsup_{t\downarrow 0}\frac{f(\varepsilon t)}{f(t)}\leq \sqrt{\varepsilon}$ and the assertion follows.
	\end{proof}
	\begin{proof}[Proof of Lemma \ref{lemmaBM}]
		%	We  first prove \eqref{detILLinf}.	Fix $(t_n)_n$ such that $t_n\rightarrow \infty$.
		With  standard inversion arguments, it suffices to prove  that,  for a sequence $t_n$ such that $\lim_{n\rightarrow\infty }t_n=\infty$, we have, with probability one,
		\begin{equation}\label{detILLinf}
		\limsup_{n\rightarrow \infty }\frac{B_{t_n}}{\sqrt{t_n}}=+\infty.% \quad a.s.,\qquad\qquad 	\liminf_{n\rightarrow \infty }\frac{B_{t_n}}{\sqrt{t_n}}=-\infty,\quad a.s.,
		\end{equation}
		The assertion is equivalent to
		$$
		\P\left(\limsup_{n\rightarrow\infty}\frac{B_{t_n}}{\sqrt{t_n}}\leq c\right)=0, \qquad c>0,
		$$
		that is
		$$
		\P\left(\bigcup_{m\geq1}\bigcap_{n\geq m} \left\{\frac{B_{t_n}}{\sqrt{t_n} }\leq c \right\}\right)=0, \qquad c>0.
		$$
		Therefore, it is sufficient to prove that $	\P\left(\bigcap_{n\geq m} \left\{\frac{B_{t_n}}{\sqrt{t_n} }\leq c \right\}\right)=0$ for every $m\in \N$ and $c>0$. Take, for example, $m=1$ and consider the random variables $\frac{B_{t_1}}{\sqrt{t_1}}$ and $\frac{B_{t_n}}{\sqrt{t_n}}$, for some $n>1$. Then,
		$$
		\frac{B_{t_1}}{\sqrt{t_1}}, \,\frac{B_{t_n}}{\sqrt{t_n}}\sim \mathcal{N}(0,1),
		$$
		where $\mathcal{N}(0,1)$ is the standard Gaussian law and 
		$$
		\mbox{Cov}\left( \frac{B_{t_1}}{\sqrt{t_1}}, \frac{B_{t_n}}{\sqrt{t_n}}\right)=\frac{t_1\wedge t_n}{\sqrt{t_1t_n}}<
		\sqrt{	\frac{t_1}{t_n}},$$	which tends to $0$ as $n$ tends to infinity.
		We deduce that 
		%$\frac{B_{t_1}}{\sqrt{t_1}}$ and $\frac{B_{t_n}}{\sqrt{t_n}}$ converge to two independent Gaussian random variables and 
		$$
		\P\left(	\frac{B_{t_1}}{\sqrt{t_1}}\leq c, \frac{B_{t_n}}{\sqrt{t_n}}\leq c\right)\rightarrow \P(Z_1\leq c,Z_2\leq c)=\P(Z_1\leq c)^2,
		$$
		where $Z_1$ and $Z_2$ are independent with $Z_1,\,Z_2\sim \mathcal{N}(0,1)$.
		
		Take  now  $m_n\in \N$ such that $t_{m_n}>nt_n$. Then, we have 
		$$
		\frac{B_{t_1}}{\sqrt{t_1}}, \,\frac{B_{t_n}}{\sqrt{t_n}},\frac{B_{t_{m_n}}}{\sqrt{t_{m_n}}}\sim \mathcal{N}(0,1)
		$$
		and
		$$ 
		\mbox{Cov}\left( \frac{B_{t_1}}{\sqrt{t_1}}, \frac{B_{t_{m_n}}}{\sqrt{t_{m_n}}}\right),\,
		\mbox{Cov}\left( \frac{B_{t_n}}{\sqrt{t_n}}, \frac{B_{t_{m_n}}}{\sqrt{t_{m_n}}}\right)\leq
		\sqrt{	\frac{t_n}{t_{m_n}}}.
		$$
		which again tends to $0$ ad $n$ tends to infinity. Therefore, we have 
		$$
		\P\left(	\frac{B_{t_1}}{\sqrt{t_1}}\leq c, \frac{B_{t_n}}{\sqrt{t_n}}\leq c,	\frac{B_{t_{m_n}}}{\sqrt{t_{m_n}}}\leq c\right)\rightarrow \P(Z_1\leq c)^3
		$$
		with $Z_1\sim \mathcal{N}(0,1)$. 
		Iterating this procedure, we can find a subsequence $(t_{n_k})_{k\in \N}$ such that $t_{n_k}\rightarrow \infty $ and 
		$$
		\P\left(\bigcap_{k\geq 1} \left\{\frac{B_{t_{n_k}}}{\sqrt{t_{n_k}} }\leq c \right\}\right)=0
		$$
		which proves that $\limsup_{n\rightarrow \infty }\frac{B_{t_n}}{\sqrt{t_n}}=+\infty$. 
		%	The fact that   $\liminf_{n\rightarrow \infty }\frac{B_{t_n}}{\sqrt{t_n}}=-\infty$ follows now with standard arguments.
		
		%As regards \eqref{detILL0}, recall that $tB_{{1/t}}$ is a Brownian motion, so that \eqref{detILLinf} gives
		%	 $$
		%	 \limsup_{n\rightarrow\infty}\frac{s_nB_{\frac 1{s_n}}}{\sqrt{s_n}}=+\infty.
		%	 $$
		%	 Then the assertion follows from a standard change of variables.
	\end{proof}
	
	\part{Hybrid schemes for  pricing options in jump-diffusion stochastic volatility models}

	\chapter{Hybrid  Monte Carlo and tree-finite differences algorithm for pricing options in the Bates-Hull-White model}\label{chapter-art3}

	\section{Introduction}

	%		In the second part of this thesis we face up with the problem of numerically computing  European and American option prices in complex stochastic volatility models.   
	
	%	The results are collected in a chronological way.
	In this chapter, which is extracted from \cite{bctz}, we focus on the so called Bates-Hull-White model.
	Following the previous work in \cite{bcz,bcz-hhw}, we further develop and study the hybrid tree/finite-difference approach and the hybrid Monte Carlo technique in order to numerically evaluate option prices.

	The Bates model \cite{bates} is a stochastic volatility model
	with price jumps: the dynamics of the underlying asset price is driven
	by both a Heston stochastic volatility \cite{H} and a compound Poisson jump process of the type originally introduced by Merton \cite{mer}. Such a model was introduced by Bates in the  foreign exchange option market in order to tackle the well-known phenomenon of the volatility smile behavior. Here, we assume a possibly stochastic interest rate following the Vasicek model, and we call the full model as Bates-Hull-White. In the case of plain vanilla European options, Fourier inversion methods \cite{CM} lead to closed-form formulas to compute the price under the Bates model. Nevertheless, in the American case the numerical literature is limited. Typically, numerical  methods are based on the use of the dynamic programming principle to which one applies either deterministic schemes  from numerical analysis and/or from tree methods or Monte Carlo techniques.
	
	The option pricing hybrid tree/finite-difference approach %for pricing European and American
	we deal with,  derives from applying an efficient recombining binomial tree method in the direction of the volatility and the interest rate components, whereas the asset price component is locally treated by means of a one-dimensional partial integro-differential equation (PIDE), to which a finite-difference scheme is applied. Here, the numerical treatment of the nonlocal term coming from the jumps involves implicit-explicit techniques, as well as numerical quadratures. 
	%In comparison to what developed in \cite{bcz,bcz-hhw}, the novelty of this paper is in part related to the application to  the Bates-Hull-White process but it mainly consists in the theoretical study of the stability of the resulting numerical scheme for the computation of both European and American options. And we stress that we never require the validity of the Feller condition for the Cox-Ingersol-Ross (CIR) dynamics \cite{cir} of the volatility process.
	
	The existing literature on numerical schemes for the option pricing problem in this framework is quite poor. Tree methods are available only for the Heston model, see \cite{VN},  but they are not really efficient when the Feller condition does not hold. Another approach is given by the dicretization of partial differential problems.
	When the jumps are not considered, namely for the Heston and the Heston-Hull-White models, available references are widely recalled in \cite{bcz, bcz-hhw}. In the standard Bates model, that is, presence of jumps but no randomness in the interest rate, the finite-difference methods for solving the $2$-dimensional PIDE associated with the option pricing problems can be based on implicit, explicit or alternating direction implicit schemes. The implicit scheme requires to solve a dense sparse system at each time step. Toivanen \cite{t} proposes a componentwise splitting method for pricing American options. The linear
	complementarity problem (LCP) linked to the American option problem is
	decomposed into a sequence of five  one-dimensional LCP's problems at
	each time step. The advantage is that LCP's need the use of
	tridiagonal matrices. Chiarella \textit{et al.} \cite{ckmz}
	developed a method of lines algorithm for pricing and hedging American
	options again under the standard Bates dynamics.
	More recently Itkin \cite{it}  proposes a unified approach to handle PIDE's associated with L\'evy's models of interest in Finance, by solving the diffusion equation with standard finite-difference methods and by transforming the jump integral into a pseudo-differential operator. But to our knowledge, no deterministic numerical methods are available in the literature for the Bates-Hull-White model, that is, when the the interest rate is assumed to be stochastic.
	
	From the simulation point of view, the main problem consists in the treatment of the CIR dynamics for the volatility process. It is well known that the standard Euler-Maruyama discretization does not work in this framework. As far as we know, the most accurate simulation schemes for the CIR process have been introduced by Alfonsi \cite{A}. Other methods are available in the literature, see e.g. \cite{andersen}, but in this chapter the Alfonsi technique is the one we compare with. In fact, in our numerical experiments we also apply a hybrid Monte Carlo technique: we couple the simulation of the approximating tree for the volatility and the interest rate components with a standard simulation of the underlying asset price, which uses Brownian increments and a straightforward treatment of the jumps. In the case of American option, this is associated with the Longstaff and Schwartz algorithm \cite{ls}, allowing to treat the dynamic programming principle.
	
	As already observed in \cite{bcz,bcz-hhw}, roughly speaking  our methods consist in the application of the most efficient method whenever this is possible: a recombining binomial tree for the volatility and the interest rate, a standard PIDE approach or a standard simulation technique in the direction of the asset price. The results of the numerical tests again support the accuracy of our hybrid methods and besides, we also justify the good behavior of the methods from the theoretical point of view (see also Chapter \ref{chapter-art4}).

	This chapter is devoted to present in detail  the hybrid  procedures introduced in \cite{bctz} to compute functionals of the Bates jump model with stochastic interest rate. In particular, we consider a  hybrid tree-finite differences procedure which uses a tree method in the direction of the volatility and the interest rate and a finite-difference approach in order to handle the underlying asset price process. We also propose hybrid simulations for the model, following a binomial tree in the direction of both the volatility and the interest rate, and a space-continuous approximation for the underlying asset price process coming from a Euler-Maruyama type scheme. As regards the theoretical analysis of the algorithm, we  study here the stability properties of the procedure and we refer to Chapter \ref{chapter-art4} for an analysis of the rate of convergence of a generalization of this algorithm under quite general assumptions.
	We provide numerical experiments which  show the reliability and the efficiency of the algorithms.
	
	The chapter is organized as follows. In Section \ref{sect-model}, we
	introduce the Bates-Hull-White model. In Section \ref{discretization} we describe the tree procedure for the volatility and the interest rate pair (Section \ref{sect-tree}), we illustrate our discretization of the log-price process (Section \ref{sec:approxY}) and  the hybrid Monte Carlo simulations (Section \ref{sect-MC}). Section \ref{sect:htfd} is devoted to the hybrid tree/finite-difference method: we first set the numerical scheme for  the associated local PIDE problem (Section \ref{sec:approxPDE}), then we apply it to the solution of the whole pricing scheme (Section \ref{sect-alg}) and analyze the numerical stability of the resulting tree/finite-difference method (Section \ref{sect:stability}).
	Section \ref{practice} refers to the practical use of our methods and numerical results and comparisons are widely discussed.

	\section{The Bates-Hull-White model}\label{sect-model}
	
	We recall that in the Bates-Hull-White model  the volatility is assumed to follow
	the CIR  process and the underlying asset price process contains a further noise from a jump as introduced by Merton. Moreover,  the interest rate  follows a stochastic model, which we assume to be described by a generalized Ornstein-Uhlenbeck (hereafter OU) process.
	More precisely, the dynamics under the risk neutral measure of the share price $S$, the volatility process $Y$ and the interest rate $r$, are given by the following jump-diffusion model:
	\begin{equation}\label{BHHmodel}
	\begin{array}{l}
	\displaystyle\frac{dS_t}{S_{t^-}}= (r_t-\delta)dt+\sqrt{Y_t}\, dZ^S_t+d H_t,
	\smallskip\\
	dY_t= \kappa_Y(\theta_Y-Y_t)dt+\sigma_Y\sqrt{Y_t}\,dZ^Y_t,
	\smallskip\\
	dr_t= \kappa_r(\theta_r(t)-r_t)dt+\sigma_r dZ^r_t,
	\end{array}
	\end{equation}
	where $\delta$ denotes the continuous dividend rate, $S_0,Y_0,r_0>0$, $Z^S$, $Z^Y$ and $Z^r$ are correlated Brownian motions and $H$ is a compound Poisson process with
	intensity $\lambda$ and i.i.d. jumps $\{J_k\}_k$, that is
	\begin{equation}\label{H}
	H_t=\sum_{k=1}^{K_t} J_k,
	\end{equation}
	$K$ denoting a Poisson process with intensity $\lambda$.
	We assume that the Poisson process $K$, the jump amplitudes $\{J_k\}_k$ and the $3$-dimensional correlated Brownian motion $(Z^S,Z^Y,Z^r)$ are independent. As suggested by Grzelak and Oosterlee in \cite{GO}, the significant correlations are between the noises governing the pairs $(S,Y)$ and $(S,r)$. So, as done in \cite{bcz-hhw}, we assume that the couple $(Z^Y,Z^r)$ is a standard Brownian motion in $\R^2$ and $Z^S$ is a Brownian motion in $\R$ which is correlated both with $Z^Y$ and $Z^r$:
	$$
	d\<Z^S,Z^Y\>_t=\rho_1dt \ \mbox{ and }\ d\<Z^S,Z^r\>_t=\rho_2dt.
	$$
	We recall that the volatility process $Y$ follows a CIR dynamics with mean reversion rate $\kappa_Y$, long run variance $\theta_Y$ and $\sigma_Y$ denotes the vol-vol (volatility of the volatility). We assume that $\theta_Y,\kappa_Y,\sigma_Y>0$ and we stress that we never require in this chapter that the CIR process satisfies the Feller condition $2\kappa_Y\theta_Y\geq \sigma_Y^2$, ensuring that the process $Y$ never hits $0$. So, we allow the volatility $Y$ to reach $0$. The interest rate $r_t$ is described by a generalized OU process, in particular $\theta_r$ is time-dependent but deterministic and fits the zero-coupon bond market values, for details see  \cite{bm}. We write the process $r$ as follows:
	\begin{equation}\label{rR}
	r_t   = \sigma_r R_t + \varphi_t
	\end{equation}
	where
	\begin{equation}\label{Rphi}
	R_t  = - \kappa_r \int_0^tR_s\,ds + \,Z^r_t \quad\mbox{and}\quad
	\varphi_t=r_0e^{-\kappa_r t}+\kappa_r\int_0^t\theta_r(s)e^{-\kappa_r(t-s)}ds.
	\end{equation}

	From now on we set
	$$
	Z^Y=W^1,\quad Z^r=W^2, \quad Z^S=\rho_1 W^1+\rho_2 W^2+ \rho_3 W^3,
	$$
	where $W=(W^1,W^2,W^3)$ is a standard Brownian motion in $\R^3$ and the correlation parameter $\rho_3$ is given by
	$$
	\rho_3=\sqrt{1-\rho_1^2-\rho_2^2},\quad \rho_1^2+\rho_2^2\leq 1.
	$$
	By passing to the logarithm $X=\ln S$ in the first component, by taking into account the above mentioned correlations and by considering the process $R$ as in \eqref{rR}-\eqref{Rphi}, we reduce to the triple $(X,Y,R)$ given by
	\begin{equation}\label{XYR-dyn}
	\begin{array}{ll}
	&dX_t=\mu_X(Y_t,R_t, t)dt+\sqrt{Y_t}\, \big(\rho_1dW^1_t+\rho_2dW^2_t+\rho_3dW^3_t\big) + dN_t, \quad X_0=\ln S_0\in\R, \smallskip\\
	&dY_t= \mu_Y(Y_t)dt+\sigma_Y\sqrt{Y_t}\,dW^1_t,
	\quad Y_0>0,\smallskip\\
	&dR_t=\mu_R(R_t)dt+dW^2_t,
	\quad R_0=0,
	\end{array}
	\end{equation}
	where
	\begin{align}
	\label{muX}
	&\mu_X(y,r,t)=\sigma_rr+\varphi_t-\delta-\frac 12 \,y,\\
	\label{muY}
	&\mu_Y(y)=\kappa_Y(\theta_Y-y),\\
	\label{muR}
	&\mu_R(r)= -\kappa_r r,
	\end{align}
	and $N_t$ is the compound Poisson process with intensity $\lambda$ and the i.i.d. jumps $\{\log(1+J_k)\}_k$, that is
	$$
	N_t=\sum_{k=1}^{K_t} \log (1+J_k),
	$$
	$K$ being a Poisson process with intensity $\lambda$.
	Recall that $K$, the jump amplitudes $\{\log(1+J_k)\}_k$ and the $3$-dimensional standard Brownian motion $(W^1,W^2, W^3)$ are all independent.
	We also recall that the L\'evy measure associated with $N$ is given by
	$$
	\nu(dx)=\lambda \P(\log(1+J_1)\in dx),
	$$
	and whenever $\log (1+J_1)$ is absolutely continuous then $\nu$ has a density as well:
	\begin{equation}\label{nu}
	\nu(dx)=\nu(x)dx=\lambda p_{\log(1+J_1)}(x)dx,
	\end{equation}
	$p_{\log(1+J_1)}$ denoting the probability density function of $\log(1+J_1)$.
	For example, in the Merton model \cite{mer} it is assumed that $\log(1+J_1)$ has a normal distribution, that is
	\[
	\log(1+J_1) \sim N(\mu,\eta^2).
	\]
	This is the choice we will do in our numerical experiments,
	%. In practice, we shall take  $\mu=\gamma-\frac{1}{2}\delta^2$ for a given suitable $\gamma\in\R$,
	as done in Chiarella \textit{et al.} \cite{ckmz}. But other jump-amplitude measures can be selected. For instance, in the Kou model \cite{kou} the law of $\log(1+J_1)$ is a mixture of exponential  laws:
	$$
	p_{\log(1+J_1)}(x) = p\lambda_+ e^{-\lambda_+ x}\,\ind{\{x>0\}}+(1-p)\lambda_- e^{\lambda_- x}\,\ind{\{x<0\}},
	$$
	$\ind A$ denoting the indicator function of $A$.  Here, the parameters $\lambda_\pm>0$ control the decrease of the distribution tails of negative and positive jumps respectively, and $p$ is the  probability of a positive jump.
	
	Given this framework, our aim is to numerically compute the price of options with maturity $T$ and payoff given by a function of the underlying asset price process $S$. By passing to the transformation $X=\ln S$, we assume that the payoff is a function of the log-price process:
	$$
	\begin{array}{ll}
	\mbox{European payoff: } &
	\displaystyle
	\Psi(X_T),\smallskip\\
	\mbox{American payoff: } &
	\displaystyle
	(\Psi(X_t))_{t\in[0,T]},
	\end{array}
	$$
	where $\Psi\geq 0$. The option price function $P(t,x,y,r)$ is then given by
	\begin{equation}\label{price}
	\begin{array}{ll}
	\mbox{European price: }&
	\displaystyle
	P(t,x,y,r)
	=\E\Big(e^{-\int_t^T (\sigma_rR^{t,r}_s+\varphi_s)ds}\Psi(X^{t,x,y,r}_T)\Big),\smallskip\\
	\mbox{American  price: } &
	\displaystyle
	P(t,x,y,r)=\sup_{\tau\in \mathcal{T}_{t,T}}\E\Big(e^{-\int_t^\tau (\sigma_rR^{t,r}_s+\varphi_s)ds}\Psi(X^{t,x,y,r}_\tau)\Big),
	\end{array}
	\end{equation}
	where $\mathcal{T}_{t,T}$ denotes the set of all stopping times taking values on $[t,T]$. Note that we have used the
	relation between the interest rate $(r_t)_t$ and the process $(R_t)_t$, %: $r_t=\sigma_rX_t+\varphi_t$
	see \eqref{rR} and \eqref{Rphi}.
	Hereafter,  $(X^{t,x,y,r},Y^{t,y},R^{t,r})$ denotes the solution of the jump-diffusion dynamic \eqref{XYR-dyn} starting at time $t$ in the point $(x,y,r)$.

	%%%%%%%%%%%%%%%%%%%%%%%%%%%%%%%%%%%%
	
	\section{The dicretized process}\label{discretization}
	
	We first set up the discretization of the triple $(X,Y,R)$ we will take into account.
	
	\subsection{The 2-dimensional tree for $(Y,R)$}\label{sect-tree}
	
	We consider an approximation for the pair $(Y,R)$ on the time-interval $[0,T]$ by means of a $2$-dimensional computationally simple tree. This means that we construct a Markov chain running over a $2$-dimensional recombining bivariate lattice and, at each time-step, both components of the Markov chain can jump only upwards or downwards. We consider the ``multiple-jumps'' approach by Nelson and  Ramaswamy \cite{nr}.
	A detailed description of this procedure and of the benefits of its use, can be found in \cite{acz,bcz,bcz-hhw}. Here, we limit the reasoning to the essential ideas and to the main steps in order to set-up the whole algorithm.
	We start by considering a discretization of the time-interval $[0,T]$ in $N$ subintervals $[nh,(n+1)h]$, $n=0,1,\ldots,N$, with $h=T/N$.

	For the CIR volatility process $Y$, we consider the binomial tree procedure firstly introduced in \cite{acz}.
	For $n=0,1,\ldots,N$, consider the lattice
	\begin{equation}\label{state-space-V}
	\mathcal{Y}_n =\{y^n_k\}_{k=0,1,\ldots,n}\quad\mbox{with}\quad
	y^n_k=\Big(\sqrt {Y_0}+\frac{\sigma_Y} 2(2k-n)\sqrt{h}\Big)^2\ind{\{\sqrt {Y_0}+\frac{\sigma_Y} 2(2k-n)\sqrt{h}>0\}}.
	\end{equation}
	Note that $y^0_{0}=Y_0$, so that $\mathcal{Y}_0^h=\{Y_0\}$. Moreover, the lattice is binomial recombining and, for $n$ large, the ``small'' points degenerate at $0$. 
	Let us briefly recall how this lattice arises (see \cite{acz} for all the details). 
	The idea  is to reduce to a process with a constant diffusion coefficient. So, let us consider the process $\hat Y_t=\sqrt{Y_t}$.
	If we (heuristically) apply It\^{o} formula, we get that the dynamics of $\hat Y_t$ is given by
	$$
	d\hat{Y}_t=\mu_{\hat Y}(\hat Y_t)dt +\frac \sigma 2 dZ^Y_t,
	$$ for a suitable drift coefficient $\mu_{\hat Y}=\mu_{\hat Y}(y)$.
	The term $ \frac \sigma 2dB_t$ gives the foremost contribution to the local movement of $\hat Y_t$. The standard binomial recombining tree for the Brownian motion lives on the lattice
	$$
	\frac{\sigma} 2(2k-n)\sqrt{h}, \qquad 0\leq k\leq n\leq N .
	$$
	Coming back to $Y$, we get the lattice  in \eqref{state-space-V}. Note that the term $\ind{\{\sqrt {Y_0}+\frac{\sigma_Y} 2(2k-n)\sqrt{h}>0\}}$ is inserted in order to deal with invertible functions.

	We now define the multiple 	 ``up'' and ``down'' jumps: the discretized process can jump just on two nodes which in turn are not necessarily the closest ones to the starting node. 
	In particular, for each fixed $y^n_k\in\mathcal{Y}_n $, we define the ``up'' and ``down'' jump by $y^{n+1}_{ k_u (n,k)}$ and $y^{n+1}_{ k_d (n,k)}$, $k_u (n,k)$ and $k_d (n,k)$ being respectively defined as
	\begin{align}
	\label{ku}
	&k_u (n,k) =\min\{k^*\,:\, k+1\leq k^*\leq n+1\mbox{ and }y^n_k+\mu_Y(y^n_k)h \le y^{n+1}_{k^*}\},\\
	\label{kd}
	&k_d (n,k) =\max\{k^*\,:\, 0\leq k^*\leq k \mbox{ and }y^n_k+\mu_Y(y^n_k)h \ge y^{n+1}_{k^*}\}
	\end{align}
	where  $\mu_Y$ is the drift of $Y$, defined in \eqref{muX}, and
	with the understanding $k_u (n,k)=n+1$, respectively $k_d (n,k)=0$, if the set in the r.h.s. of \eqref{ku}, respectively  \eqref{kd}, is empty.
	The transition probabilities are defined as follows: starting from the node $(n,k)$ the probability that the process jumps to $k_u (n,k)$ and $k_d (n,k)$ at time-step $n+1$ are set as
	\begin{equation}\label{pik}
	p^{Y}_u(n,k)
	=0\vee \frac{\mu_Y(y^n_k)h+ y^n_k-y^{n+1}_{k_d (n,k)} }{y^{n+1}_{k_u (n,k)}-y^{n+1}_{k_d (n,k)}}\wedge 1
	\quad\mbox{and}\quad p^{Y}_d(n,k)=1-p^{Y}_u(n,k)
	\end{equation}
	respectively.
	We recall that the multiple jumps and the  transition probabilities are set in order to best fit the local first moment of the diffusion $Y$. We will see in Chapter \ref{chapter-art4} that this property will be crucial in order to study the theoretical convergence of the procedure.

	\smallskip
	
	We follow the same approach for the binomial tree for the process $R$.
	For $n=0,1,\ldots,N$ consider the lattice
	\begin{equation}\label{state-space-X}
	\mathcal{R}_n =\{r^n_j\}_{j=0,1,\ldots,n}\quad\mbox{with}\quad
	r^n_j=(2j-n)\sqrt{h}.
	\end{equation}
	Notice that $r_{0,0}=0=R_0$. For each fixed $r^n_j\in\mathcal{R}_n $, we define the ``up'' and ``down''  jump by means of $j_u (n,j)$ and $j_d (n,j)$  defined by
	\begin{align}
	\label{ju}
	&j_u (n,j) =\min\{j^*\,:\, j+1\leq j^*\leq n+1\mbox{ and }r^n_j+\mu_R(r^n_j)h \le r^{n+1}_{j^*}\},\\
	\label{jd}
	&j_d (n,j) =\max\{j^*\,:\, 0\leq j^*\leq j \mbox{ and }r^n_j+\mu_R(r^n_j)h \ge r^{n+1}_{j^*}\},
	\end{align}
	$\mu_R$ being the drift of the process $R$, see \eqref{muR}.
	As before, $j_u (n,j)=n+1$, respectively $j_d (n,j)=0$, if the set in the r.h.s. of \eqref{ju}, respectively \eqref{jd}, is empty and
	the transition probabilities are as follows: starting from the node $(n,j)$, the probability that the process jumps to $j_u (n,j)$ and $j_d (n,j)$ at time-step $n+1$ are set as
	\begin{equation}\label{pij}
	p^{R}_u(n,j)
	=0\vee \frac{\mu_R(r^n_j)h+ r^n_j-r^{n+1}_{j_d (n,j)} }{r^{n+1}_{j_u (n,j)}-r^{n+1}_{j_d (n,j)}}\wedge 1
	\quad\mbox{and}\quad p^{R}_d(n,j)=1-p^{R}_u(n,j)
	\end{equation}
	respectively.

	Figure \ref{fig:treeYR} shows a picture of the lattices $\mathcal{Y}_n $ (left) and $\mathcal{R}_n $ (right), together with possible instances of the up and down jumps.
	
	\begin{figure}[htp]
		\centering
		\begin{tabular}{c}
			
			\includegraphics[scale=0.33]{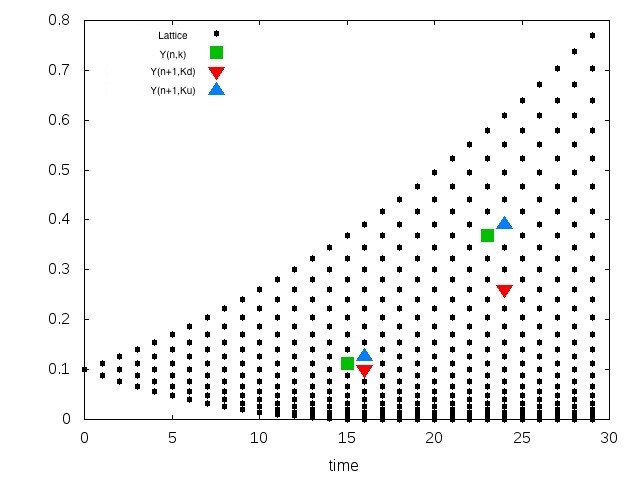}
			\includegraphics[scale=0.33]{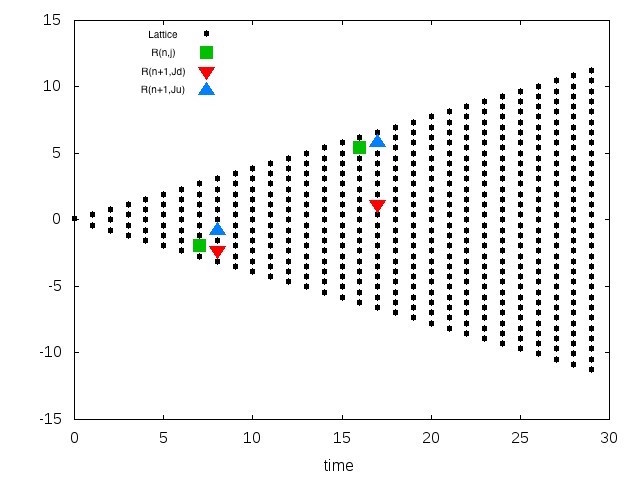}
			
		\end{tabular}
		\caption{The tree for the process $Y$ (left) and for $R$ (right), showing as the trees may be visited.}
		\label{fig:treeYR}
	\end{figure}

	%An example of such procedure is given in Figure \ref{fig:treeX}, which draws the lattice $\mathcal{X}_n $ and possible instances of $r^n_j$, $r^{n+1}_{j_d (n,j)}$ and $r^{n+1}_{j_u (n,j)}$.
	%\begin{figure}[htp]
	%  \centering
	%\begin{tabular}{c}
	%
	%\includegraphics[scale=0.5]{albero_multi_X}
	%
	%\end{tabular}
	%\caption{Example of a tree for the process $X$, showing as the tree may be visited.}
	%\label{fig:treeX}
	%
	%\end{figure}
	%

	\smallskip
	
	The whole tree procedure for the pair $(Y,R)$  is obtained by joining the trees built for $Y$ and for $R$. Namely, for $n=0,1,\ldots,N$, consider the lattice
	\begin{equation}\label{state-space-Yr}
	\mathcal{Y}_n \times \mathcal{R}_n =\{(y^n_k,r^n_j)\}_{k,j=0,1,\ldots,n}.
	\end{equation}
	Starting from the node $(n,k,j)$, which corresponds to the position $(y^n_k,r^n_j)\in\mathcal{Y}_n \times \mathcal{R}_n $, we define the four possible jumps by means of the following four nodes at time $n+1$:
	\begin{equation}\label{treescheme}
	\begin{array}{lcr}
	(n+1,k_u (n,k),j_u (n,j)) & \hbox{with probability} & p_{uu}(n,k,j)=p^{Y}_{u}(n,k)p^{R}_{u}(n,j), \smallskip\\
	(n+1,k_u (n,k),j_d (n,j)) & \hbox{with probability} & p_{ud}(n,k,j)=p^{Y}_{u}(n,k)p^{R}_{d}(n,j), \smallskip\\
	(n+1,k_d (n,k),j_u (n,j)) & \hbox{with probability} & p_{du}(n,k,j)=p^{Y}_{d}(n,k)p^{R}_{u}(n,j), \smallskip\\
	(n+1,k_d (n,k),j_d (n,j)) & \hbox{with probability} & p_{dd}(n,k,j)=p^{Y}_{d}(n,k)p^{R}_{d}(n,j),
	\end{array}
	\end{equation}
	where the above nodes $k_{u}(n,k)$, $k_{d}(n,k)$, $j_{u}(n,j)$, $j_{d}(n,j)$ and
	the above probabilities $p^{Y}_{u}(n,k)$, $p^{Y}_{d}(n,k)$, $p^{R}_{u}(n,j)$, $p^{R}_{d}(n,j)$ are defined in \eqref{ku}-\eqref{kd}, \eqref{ju}-\eqref{jd}, \eqref{pik} and \eqref{pij}. The factorization of the jump probabilities in \eqref{treescheme} follows from the orthogonality property of the noises driving the two processes.
	This procedure gives rise to a Markov chain $(\hat Y^h_{n},\hat R^h_{n})_{n=0,\ldots,N}$ that weakly converges, as $h\to 0$, to the diffusion process $(Y_{t},R_t)_{t\in[0,T]}$ solution to
	\begin{align*}
	&dY_t= \mu_Y(Y_t)dt+\sigma_Y\sqrt{Y_t}\,dW^1_t,
	\quad Y_0>0,\\
	&dR_t  = \mu_R(R_t)\,dt + dW^2_t,\quad R_0=0.
	\end{align*}
	This can be seen by using standard results (see e.g. the techniques in \cite{nr}) and the convergence of the chain approximating the volatility process proved in \cite{acz}. And  this holds independently of the validity of the Feller condition $2\kappa_Y \theta_Y\geq \sigma^2_Y$.
	%Moreover,  $(\hat V^h_{n},\hat X^h_{n})_{n=0,\ldots,N}$ turns out to be a robust tree approximation for $(V_{t},X_t)_{t\in[0,T]}$.
	
	Details and remarks on the extension of this procedure to more general cases can be found in \cite{bcz-hhw}. In particular, if the correlation between the Brownian motions driving $(Y,R)$ was not null, one could define the jump probabilities by matching the local cross-moment (see Remark 3.1 in \cite{bcz-hhw}).
	
	\subsection{The approximation on the $X$-component}\label{sec:approxY}
	
	We describe here how we manage the $X$-component in \eqref{XYR-dyn} by taking into account the tree procedure given for the pair $(Y,R)$.
	We go back to \eqref{XYR-dyn}: by isolating $\sqrt{Y_t}dW^1_t$ in the second line and $dW^2_t$ in the third one, we obtain
	\begin{equation}\label{Y1}
	dX_t=\mu(Y_t,R_t,t)dt+\rho_3\sqrt{Y_t}\,dW^3_t+\frac{\rho_1}{\sigma_Y}dY_t+\rho_2\sqrt{Y_t}dR_t+dN_t \end{equation}
	with
	\begin{equation}\label{mu-fin}
	\begin{array}{rl}
	\mu(y,r,t)
	&=\mu_X(y,r,t)-\frac{\rho_1}{\sigma_Y}\mu_Y(y)-\rho_2\sqrt{y}\,\mu_R(r)\smallskip\\
	&=\sigma_rr+\varphi_t-\delta-\frac 12 \,y
	-\frac{\rho_1}{\sigma_Y}\,\kappa_Y(\theta_Y-y)+\rho_2\kappa_r r\sqrt y
	\end{array}
	\end{equation}
	($\mu_X$, $\mu_Y$ and $\mu_R$ are defined in \eqref{muX}, \eqref{muY} and \eqref{muR} respectively). To numerically solve \eqref{Y1}, we mainly use the fact that the noises $W^3$ and $N$ are independent of the processes $Y$ and $R$.
	So, we first take the approximating tree $(\hat Y^h_n,\hat R_n)_{n=0,1,\ldots,N-1}$ discussed in Section \ref{sect-tree} and we set $(\bar Y^h_t,\bar R^h_t)_{t\in[0,T]}=(\hat Y^h_{\lfloor t/h\rfloor}+1,\hat R^h_{\lfloor t/h\rfloor}+1)_{t\in[0,T]}$ the associated time-continuous c\`adl\`ag approximating process for $(Y,R)$.
	Then, we insert the discretization $(\bar Y^h,\bar R^h)$ for $(Y,R)$ in the coefficients of \eqref{Y1}. Therefore, the final process $\bar X^h$  approximating $X$ is set as follows:
	$\bar X^h_0=X_0$ and for $t\in(nh,(n+1)h]$ with $n=0,1,\ldots,N-1$
	\begin{equation}\label{barXh}
	\begin{array}{rl}
	\bar X^h_t=& \bar X^h_{nh} + \mu(\bar Y^h_{nh},\bar R^h_{nh},nh)(t-nh)
	+\rho_3\sqrt{\bar Y^h_t}(W^3_t-W^3_{nh})
	\medskip\\
	&+\displaystyle\frac{\rho_1}{\sigma_Y}(\bar Y^h_t-\bar Y^h_{nh})+\rho_2\sqrt{\bar Y^h_t}(\bar R^h_t-\bar R^h_{nh})+(N_t-N_{nh}).
	\end{array}
	\end{equation}
	
	%\subsection{The hybrid Monte Carlo algorithm}\label{sect-MC}
	\subsection{The Monte Carlo approach}\label{sect-MC}
	Let us show how one can simulate a single path by using the tree approximation \eqref{state-space-Yr} for the couple $(Y,R)$ and the Euler scheme \eqref{barXh} for the $X$-component.
	
	%Consider the process $(Y,V,R)$ as in \eqref{YVR-dyn}.
	Let $(\hat X_n)_{n=0,1,\ldots,N}$ be the sequence approximating $X$ at times $nh$, $n=0,1,\ldots,N$,  by means of the scheme in \eqref{barXh}: $\hat X^h_0=X_0$ and for $t\in[nh,(n+1)h]$ with $n=0,1,\ldots,N-1$ then
	$$
	\begin{array}{rl}
	\hat X^h_{n+1}=& \hat X^h_{n} + \mu(\hat Y^h_{n},\hat R^h_{n},nh)h
	+\rho_3\sqrt{h\hat Y^h_n}\Delta_{n+1}
	\medskip\\
	&+\displaystyle\frac{\rho_1}{\sigma_Y}(\hat Y^h_{n+1}-\hat Y^h_{n})+\rho_2\sqrt{\hat Y^h_n}(\hat R^h_{n+1}-\hat R^h_{n})+(N_{(n+1)h}-N_{nh}),
	\end{array}
	$$
	where $\mu$ is defined in \eqref{mu-fin} and $\Delta_1,\ldots,\Delta_N$ denote i.i.d. standard normal r.v.'s, independent of the noise driving the chain $(\hat Y,\hat R)$. The simulation of $N_{(n+1)h}-N_{nh}$ is straightforward: one first generates a Poisson r.v. $K_h^{n+1}$ of parameter $\lambda h$ and if $K_h^{n+1}>0$ then also the log-amplitudes $\log(1+J^{n+1}_k)$ for $k=1,\ldots,K^{n+1}_h$ are simulated. Then, the observed jump of the compound Poisson process is written as the sum of the simulated log-amplitudes, so that
	\begin{equation}\label{hmc}
	\begin{array}{rl}
	\hat X^h_{n+1}=& \hat X^h_{n} + \mu(\hat Y^h_{n},\hat R^h_{n},nh)h
	+\rho_3\sqrt{h\hat Y^h_n}\Delta_{n+1}
	\medskip\\
	&+\displaystyle\frac{\rho_1}{\sigma_Y}(\hat Y^h_{n+1}-\hat Y^h_{n})+\rho_2\sqrt{\hat Y^h_n}(\hat R^h_{n+1}-\hat R^h_{n})+\sum_{k=1}^{K_h^{n+1}}\log(1+J^{n+1}_k),
	\end{array}
	\end{equation}
	in which the last sum is set equal to 0 if $K_h^{n+1}=0$.
	
	The above simulation scheme is plain: at each time step $n\geq 1$, one lets the pair $(Y,R)$ evolve on the tree and simulate the process $X$ by using \eqref{hmc}. We will refer to this procedure as \textit{hybrid Monte Carlo algorithm}, the word ``hybrid'' being related to the fact that two different noise sources are considered: we simulate a continuous process in space (the component $X$) starting from a discrete process in space (the tree for $(Y,R)$).
	
	The simulations just described will be used in  Section \ref{practice} in order to set-up a Monte Carlo procedure for the computation of the option price function  \eqref{price}. In the case of American options, the simulations are coupled with the Monte Carlo algorithm by Longstaff and Schwartz in \cite{ls}.
	
	\section{The hybrid tree/finite difference approach}\label{sect:htfd}

	The price-function $P(t,x,y,r)$ in \eqref{price} is typically computed by  means of the standard backward dynamic programming algorithm. So,  consider a discretization of the time interval $[0,T]$ into $N$ subintervals of length $h=T/N$. Then the price $P(0,X_0,Y_0,R_0)$  is numerically approximated through the quantity $P_h(0,X_0,Y_0,R_0)$ backwardly given by
	\begin{equation}\label{backward}
	\begin{cases}
	P_h(T,x,y,r)= \Psi(x)\quad
	\mbox{and as } n=N-1,\ldots,0,\\
	P_h(nh,x,y,r) =  \max \Big\{\widehat \Psi(x), e^{-(\sigma_rr+\varphi_{nh})h}
	\E\Big(P_h\big((n+1)h, X_{(n+1)h}^{nh,x,y,r}, Y_{(n+1)h}^{nh,y}, R_{(n+1)h}^{nh,r}\big)
	\Big)\Big\},
	\end{cases}
	\end{equation}
	for $(x,y,r)\in\R\times\R_+\times \R$, in which
	$$
	\widehat \Psi(x)=\left\{
	\begin{array}{ll}
	0 & \mbox{ in the European case,}\smallskip\\
	\Psi(x) & \mbox{ in the American case.}
	\end{array}
	\right.
	$$
	So, what is needed is a good approximation of the expectations appearing in the above dynamic programming principle. This is what we first deal with, starting from the dicretized process $(\bar Y^h,\bar Y^h,\bar R^h)$ introduced in Section \ref{discretization}.
	
	\subsection{The local 1-dimensional partial integro-differential equation}\label{sec:approxPDE}
	Let $\bar X^h$ denote the process in \eqref{barXh}.
	If we set
	\begin{equation}\label{Z}
	\bar Z^{h}_t=\bar X^h_t
	-\frac{\rho_1}{\sigma_Y}(\bar Y^h_t-\bar Y^h_{nh})
	-\rho_2\sqrt{\bar Y^h_{nh}}(\bar R^h_t-\bar R_{nh}), \quad t\in [nh,(n+1)h]
	\end{equation}
	then we have
	\begin{equation}\label{barZ}
	\begin{array}{l}
	d\bar Z^{h}_t=\mu(\bar Y^h_{nh},\bar R^h_{nh},{nh})dt+\rho_3\sqrt{\bar Y^h_{nh}}\,dW^3_t,+dN_t\quad t\in (nh,(n+1)h],\smallskip\\
	\bar Z^{h}_{nh}=\bar X^h_{nh},
	\end{array}
	\end{equation}
	that is, $\bar Z^h$ solves a jump-diffusion stochastic equation with constant coefficients and at time $nh$ it starts from $\bar Y^h_{nh}$.
	Take now a function $f$: we are interested in computing
	$$
	\E(f(X_{(n+1)h})\mid X_{nh}=x, Y_{nh}=y, R_{nh}=r).
	$$
	We actually need a function $f$ of all variables $(x,y,r)$ but at the present moment the variable $x$ is the most important one, we will see later on that the introduction of $(y,r)$ is straightforward. So, we numerically compute the above expectation by means of the one done on the approximating processes, that is,
	$$
	\begin{array}{l}
	\displaystyle
	\E\big(f(\bar X^h_{(n+1)h})\mid \bar X^h_{nh}=x, \bar Y^h_{nh}=y, \bar R^h_{nh}=r\big)\smallskip\\
	\displaystyle
	=\E\big(f(\bar Z^{h}_{(n+1)h}+\frac{\rho_1}{\sigma_Y}(\bar Y^h_{(n+1)h}-\bar Y^h_{nh})
	+\rho_2\sqrt{\bar Y^h_{nh}}(\bar R^h_{(n+1)h}-\bar R^h_{nh})
	)\mid \bar Z^{h}_{nh}=x, \bar Y^h_{nh}=y, \bar R^h_{nh}=r\big),
	\end{array}
	$$
	in which we have used the process $\bar Z^h$ in \eqref{Z}. Since $(\bar Y^h, \bar R^h)$ is independent of the Brownian noise $W^3$ and on the compound Poisson process $N$ driving $\bar Z^h$ in \eqref{barZ}, we have the following:
	we set
	\begin{equation}\label{E2}
	\Psi_f(\zeta; x,y,r)=\E(f(\bar Z^h_{(n+1)h}+\zeta
	)\mid \bar Z^h_{nh}=x, \bar Y^h_{nh}=y, \bar R^h_{nh}=r)
	\end{equation}
	and we can write
	\begin{equation}\label{E1}
	\begin{array}{l}
	\displaystyle
	\E(f(\bar X^h_{(n+1)h})\mid \bar X^h_{nh}=x, \bar Y^h_{nh}=y, \bar R^h_{nh}=r)\smallskip\\
	\qquad=\E\Big(\Psi_f\Big(\frac{\rho_1}{\sigma_Y}(\bar Y^h_{(n+1)h}-\bar Y^h_{nh})
	+\rho_2\sqrt{y}(\bar R^h_{(n+1)h}-\bar R^h_{nh}); x,y,r
	\Big)\,\Big|\, \bar Y^h_{nh}=y, \bar R^h_{nh}=r\Big).
	\end{array}
	\end{equation}
	Now, in order to compute the quantity $\Psi_f(\zeta)$ in \eqref{E2}, we consider a generic function $g$ and set
	$$
	u(t,x;y,r)=\E(g(\bar Z^h_{(n+1)h})\mid \bar Z^h_{t}=x, \bar Y^h_{t}=y, \bar R^h_{t}=r),\quad t\in[nh,(n+1)h].
	$$
	By \eqref{barZ} and the Feynman-Kac representation formula we can state that, for every fixed $r\in\R$ and $y\geq 0$, the function $(t,x)\mapsto u(t,x;y,r)$ is the solution to
	\begin{equation}\label{PIDE}
	\left\{\begin{array}{ll}
	\displaystyle
	\partial_t u(t,x;y,r)+\L^{(y,r)} u(t,x;y,r)=0 & y\in\R, t\in [nh,(n+1)h),\smallskip\\
	\displaystyle
	u((n+1)h,x;y,r)=g(y) & x\in \R,
	\end{array}\right.
	\end{equation}
	where $\L^{(y,r)}$ is the integro-differential operator
	\begin{equation}\label{PIDE:Loperator}
	\begin{array}{rl}
	\L^{(y,r)}u(t,x;y,r)=&\mu(y,r)\partial_x u(t,x;y,r) +\frac 12 \rho_3^2 y\partial^2_{xx} u(t,x;y,r)
	\medskip\\
	&+ \displaystyle\int_{-\infty}^{+\infty} \left[ u(t,x+\xi;y,r)-u(t,x;y,r)\right]\nu(\xi) d\xi,
	\end{array}
	\end{equation}
	where $\mu$ is given in \eqref{mu-fin} and $\nu$ is the L\'evy measure associated with the compound Poisson process $N$, see \eqref{nu}. We are assuming here that the L\'evy measure is absolutely continuous (in practice, we use a Gaussian density), but it is clear that the procedure we are going to describe can be straightforwardly extended to other cases.

	\subsubsection{Finite-difference and numerical quadrature}
	
	In order to numerically compute the solution to the PIDE \eqref{PIDE} at time $nh$,
	we generalize the approach already developed in \cite{bcz,bcz-hhw}:
	%at each time step $n$,
	we apply a one-step finite-difference algorithm to the differential part of the problem coupled now with a quadrature rule to approximate the integral term.
	
	We start by fixing an infinite grid on the $x$-axis $\mathcal{X}=\{x_i=X_0+i\dx\}_{i\in\Z}$, with $\Delta x=x_i-x_{i-1}$, $i\in\Z$. For fixed $n$ and given $r\in\R$ and $y\geq 0$, we set $u^n_i=u(nh, x_i;y,r)$ the discrete solution of \eqref{PIDE} at time $nh$ on the point $x_i$ of the grid $\mathcal{X}$ -- for simplicity of notations, in the sequel we do not stress in $u^n_i$ the dependence on $(y,r)$.
	
	First of all, to numerically compute the integral term in \eqref{PIDE:Loperator} we need to truncate the infinite integral domain to a bounded interval $\mathcal{I}$, to be taken large enough in order that
	\begin{equation}\label{bound_int}
	\int_{\mathcal{I}} \nu(\xi) d\xi \approx \lambda.
	\end{equation}
	In terms of the process, this corresponds to truncate the large jumps.
	We assume that the tails of $\nu$ rapidly decrease -- this is not really restrictive since applied  models typically require that the tails of $\nu$ decrease exponentially.
	Hence, we take $L\in\N$ large enough, set $\mathcal{I}=[-L\Delta y , +L\Delta y ]$ and  apply to \eqref{bound_int} the trapezoidal rule on the grid $\mathcal{X}$ with the same step $\dx$ previously defined.  Then, for  $\xi_l=l\dx$, $l=-L,\ldots,L$, we have
	\begin{equation}\label{num_int}
	\int_{-L\Delta y}^{+L\Delta y} \left[ u(t,x+\xi)-u(t,x)\right]\nu(\xi) d\xi\approx \dx \sum_{l=-L}^{L} \left(u(t,x+\xi_l)-u(t,x)\right)\nu(\xi_l).
	\end{equation}
	We notice that $x_i+\xi_l=X_0+(i+l)\dx \in\mathcal{X}$, so the values $u(t,x_i+\xi_l)$ are well defined on the numerical grid $\mathcal{X}$ for any $ i,l$. These are technical settings and can be modified and calibrated for different L\'evy measures $\nu$.
	
	But in practice one cannot solve the PIDE problem over the whole real line. So, we have to choose artificial bounds and impose numerical boundary conditions. We take a positive integer $M>0$ and we define a finite grid $\mathcal{X}_M=\{x_i=X_0+i\dx\}_{i\in\mathcal{J}_{M}}$, with $\mathcal{J}_{M}=\{-M,\ldots,M\}$, and we assume that $M>L$.
	%\begin{equation}\label{R}
	%M>R.
	%\end{equation}
	Notice that %At time $t=nh$ and
	for $x=x_i\in \mathcal{X}_M$ then the integral term in \eqref{num_int} splits into two parts: one part concerning nodes falling into the numerical domain $\mathcal{X}_{M}$ and another part concerning nodes falling out of $\mathcal{X}_M$. As an example, at time $t=nh$ we have
	$$
	\sum_{l=-L}^{L}u(nh,x_i+\xi_l)\nu(\xi_l)\approx\sum_{l=-L}^L u^{n}_{i+l} \nu(\xi_l) = \sum_{l\,:\,|l|\leq L,|i+l| \leq M} u^{n}_{i+l}\  \nu(\xi_l) + \sum_{l\,:\,|l|\leq L,|i+l| > M} \tilde u^{n}_{i+l}\  \nu(\xi_l),
	$$
	where $\tilde u^n_{\cdot}$ stands for (unknown) values that fall out of the finite numerical domain $\mathcal{X}_{M}$.
	This implies that we must choose some suitable artificial boundary conditions.
	In a financial context, in \cite{cv} it has been shown that a good choice for the boundary conditions is the payoff function.
	Although this is the choice we will do in our numerical experiments, for the sake of generality we assume here the boundary values outside $\mathcal{X}_{M}$ to be settled as $\tilde u^{n}_{i}=b(nh,x_i)$, where $b=b(t,x)$ is a fixed function defined in
	$[0,T]\times\R$.
	%For example, for a put option, we have $b(y)=(K-S_0e^{y})_+$ and we define the boundary values outside the $\mathcal{X}_{M}$ domain at each time $nh$, $n=0,\ldots,N$, as
	%\begin{equation}\label{f_boundary}
	%u^n_i=b(y_i-r(T-nh)), \ \forall \ i\notin\mathcal{J}_M.
	%\end{equation}
	%\textcolor{blue}{Controllare il punto in cui viene calcolata la f!! Secondo me per ora non ci dovrebbe essere il termine scontato $e^{-r(T-nh)}$ ...}
	%Specifically, the boundary indices are $\mathcal{J}_B=\{-M-R,\ldots,\-M-1;\ M+1,\ldots,M+R\}$.
	
	Going back to the numerical scheme to solve the differential part of the equation \eqref{PIDE}, as already done in \cite{bcz-hhw}, we apply an implicit in time approximation. However, to avoid to solve at each time step a linear system with a dense matrix, the non-local integral term needs anyway an explicit in time approximation. We then obtain an implicit-explicit (hereafter IMER) scheme as proposed in \cite{cv} and \cite{bln}. Notice that more sophisticated IMER methods may be applied, see for instance \cite{bnr,st}. Let us stress that these techniques could be used in our framework, being more accurate but expensive.
	
	As done in \cite{bcz-hhw}, to achieve greater precision we use the centered approximation for both first and second order derivatives in space.
	The discrete solution $u^n$ at time $nh$ is then computed in terms of the known value $u^{n+1}$ at time $(n+1)h$ by solving the following discrete problem: for all $i\in\mathcal{J}_M$,
	\begin{equation}\label{discr_eq}
	\Frac{u^{n+1}_i-u^n_i}{h}+\tilde\mu_X(y,r)\Frac{u^{n}_{i+1}-u^{n}_{i-1}}{2\dx}+\frac{1}{2}\rho_3^2\ y\ \Frac{u^{n}_{i+1}-2u^n_i+u^{n}_{i-1}}{\dx^2}
	+ \dx \Sum_{l=-R}^{R} \left(u^{n+1}_{i+l}-u^{n+1}_{i}\right)\nu(\xi_l) =0.
	\end{equation}
	We then get the solution $u^n=(u^n_{-M},\ldots,u^n_M)^T$ by solving the following linear system
	\begin{equation}\label{lin_sys}
	A\, u^n = B u^{n+1} + d,
	\end{equation}
	where $A=A(y,r)$ and $B$ are $(2M+1)\times (2M+1)$ matrices and $d$ is a $(2M+1)$-dimensional boundary vector defined as follows.
	
	\smallskip
	
	\noindent
	$\blacktriangleright$ \textbf{The matrix $A$.}
	From \eqref{discr_eq}, we set $A$ as the tridiagonal real matrix given by
	\begin{equation}\label{matrixA}
	A = \left(
	\begin{array}{ccccc}
	%%%1+2\beta & -2\beta & & & \\
	1+2\beta & -\alpha-\beta & & & \\
	\alpha-\beta & 1+2\beta & -\alpha-\beta & & \\
	& \ddots & \ddots & \ddots &  \\
	&   & \alpha-\beta & 1+2\beta & -\alpha-\beta \\
	&   &        &   \alpha-\beta & 1+2\beta
	%%% &   &        &   -2\beta & 1+2\beta
	\end{array}
	\right),
	\end{equation}
	with
	\begin{equation}\label{alphabeta_impl}
	\alpha=\frac{h}{2\dx}\,\mu(nh,y,r)\quad\mbox{and}\quad\beta=\frac{h}{2\dx^2}\,\rho_3^2y,
	\end{equation}
	$\mu$ being defined in \eqref{mu-fin}.
	We emphasize that at each time step $n$, the quantities $v$ and $x$ are constant and known values (defined by the tree procedure for $(Y,R)$) and then $\alpha$ and $\beta$ are constant parameters.
	
	\smallskip
	
	\noindent
	$\blacktriangleright$ \textbf{The matrix $B$.}
	Again from \eqref{discr_eq}, $B$ is the $(2M+1)\times(2M+1)$ real matrix given by
	\begin{equation}\label{matrixD_bound}
	B = I + h\dx\left(
	\begin{array}{cccccc}
	\nu(0) - \Lambda &  \nu(\dx)  & \ldots & \nu(L\dx) & 0 & \\
	\nu(-\dx) & \nu(0) - \Lambda & \nu(\dx)  & \ldots &  \nu(L\dx) & \\
	& \ddots & \ddots & \ddots &  &\\
	0 & \nu(-L\dx) & \ldots & \nu(-\dx) &  \nu(0) - \Lambda
	\end{array}
	\right),
	\end{equation}
	where $I$ is the identity matrix and
	$$
	\Lambda=\sum_{l=-L}^L \nu(\xi_l).
	$$
	
	\medskip
	
	\noindent
	$\blacktriangleright$ \textbf{The boundary vector $d$.}
	The vector $d\in\R^{2M+1}$ contains the numerical boundary values:
	\begin{equation}\label{d}
	d = a_b^n+ a_b^{n+1},
	\end{equation}
	with
	$$a_b^n =((\beta-\alpha)b^{n}_{-M-1},0,\ldots,0,(\beta+\alpha)b^n_{M+1})^T\in\R^{2M+1}$$
	and $a_b^{n+1} \in \R^{2M+1}$ is such that
	$$
	%\begin{equation}
	(a^{n+1}_b)_i = \left\{\begin{array}{ll}
	\displaystyle
	h \dx \sum_{l=-L}^{-M-i-1}\nu(x_l)\ b^{n+1}_{i+l}, & \mbox{ for } i=-M,\ldots,-M+L-1,
	\smallskip\\
	0 & \mbox{ for } i=-M+L,\ldots,M-L,
	\smallskip\\
	\displaystyle
	h \dx \sum_{l=M-i+1}^L\nu(x_l)\ b^{n+1}_{i+l}, & \mbox{ for } i=M-L+1,\ldots,M-1,
	\end{array}\right.
	%\end{equation}
	$$
	where we have used the standard notation $b^n_i=b(nh,x_i)$, $i\in\mathcal{J}_M$.
	
	\smallskip
	In practice, we numerically solve the linear system \eqref{lin_sys} with an efficient algorithm (see next Remark \ref{costo}). We notice here that a solution to \eqref{lin_sys} really exists
	because for $\beta\ne|\alpha|$, the matrix $A=A(y,r)$ is invertible (see e.g. Theorem 2.1 in \cite{BT}). Then, at time $nh$, for each fixed $y\geq 0$ and $r\in\R$,
	we approximate the solution $x\mapsto u(nh,x ;y,r)$ of \eqref{PIDE} on the points $x_i$'s of the grid in terms of the discrete solution $u^n=\{u^n_i\}_{i\in \mathcal{J}_{M}}$, which in turn is written in terms of the value $u^{n+1}=\{u^{n+1}_i\}_{i\in \mathcal{J}_{M}}$ at time $(n+1)h$. In other words, we set
	\begin{equation}\label{FD}
	\mbox{$u(nh,x_i;y,r)\approx u^n_i$, $i\in\mathcal{J}_M$, where $u^n=(u^n_i)_{i\in\mathcal{J}_M}$ solves \eqref{lin_sys}}
	\end{equation}

	%\begin{equation}\label{FD}
	%%\begin{array}{l}
	%%u(nh,y_i;v) \approx f^n_i(v), \quad i\in\mathcal{J}_B.
	%%\smallskip\\
	%\begin{array}{l}
	%\displaystyle
	%u(nh,y_i;v,x)\approx
	%\sum_{j\in\mathcal{J}_{M_h}}
	%\Pi_{ij}(v,x)g(y_j)+\pi_i %\tilde d_i
	%(v,x),\quad i\in \mathcal{J}_{M},\quad\mbox{where}\smallskip\\
	%\Pi(v,x) =A^{-1}(v,x)B\quad\mbox{and}\quad \pi %\tilde d
	%(v,x) = A^{-1}(v,x)d.
	%\end{array}
	%\end{equation}

	\subsubsection{The final local finite-difference approximation}
	
	We are now ready to tackle our original problem: the computation of the function $\Psi_f(\zeta;x,y,r)$ in \eqref{E2} allowing one to numerically compute the expectation in \eqref{E1}.
	So, at time step $n$, the pair $(y,r)$ is chosen on the lattice $\mathcal{Y}_n\times\mathcal{R}_n$: $y=y^n_k$, $r=r^n_j$ for $0\leq k,j\leq n$. We call $A^n_{k,j}$ the matrix $A$ in \eqref{matrixA} when evaluated in  $(y^n_k, r^n_j)$ and $d^n$ the boundary vector in \eqref{d} at time-step $n$. Then, \eqref{FD} gives
	%\begin{equation}
	$$
	\begin{array}{c}
	\mbox{$\Psi_f(\zeta; x_i,y^n_k,r^n_j)\simeq u^n_{i,k,j}$, where $u^n_{\cdot,k,j}=(u^n_{i,k,j})_{i\in\mathcal{J}_M}$ solves the linear system}\medskip\\
	\displaystyle
	A^n_{k,j}u^n_{\cdot,k,j}=Bf(x_\cdot+\zeta)+d^n.
	\end{array}
	$$
	%\end{equation}
	%\begin{equation}
	%\Psi_f(\zeta; y_i,y^n_k,r^n_j)\simeq \sum_{l\in\mathcal{J}_M}\Pi_{il}(y^n_k,r^n_j)f(y_l+\zeta) + \pi_i(y^n_k,r^n_j), \quad i\in\mathcal{J}_M, 0\leq k,j\leq n.
	%\end{equation}
	Therefore, by taking the expectation w.r.t. the tree-jumps, the expectation in \eqref{E1} is finally computed on
	$\mathcal{X}_M\times\mathcal{Y}_n\times\mathcal{R}_n$ by means of the above approximation:
	$$
	\E(f(\bar X^h_{(n+1)h})\mid \bar X^h_{nh}=x_i, \bar Y^h_{nh}=y^n_k, \bar R^h_{nh}=r^n_j)\simeq u^n_{i,k,j},
	$$
	where $u^n_{\cdot,k,j}=(u^n_{i,k,j})_{i\in\mathcal{J}_M}$ solves the linear system 
	$$
	A^n_{k,j}u^n_{\cdot,k,j}=\sum_{a,b\in\{u,d\}} p_{ab}(n,k,j) Bf	\Big(x_\cdot+\frac{\rho_1}{\sigma_Y}(y^{n+1}_{k_a (n,k)}-y^n_k)
	+\rho_2\sqrt{y}(r^{n+1}_{j_b (n,j)}-r^n_j)\Big)+
	d^n.
	$$
	Finally, if $f$ is a function on the whole triple $(x,y,r)$, by using standard properties of the conditional expectation one gets
	\begin{equation}\label{E3}
	\begin{array}{l}
	\mbox{$\E(f(\bar X^h_{(n+1)h}, \bar Y^h_{(n+1)h}, \bar R^h_{(n+1)h})\mid \bar X^h_{nh}=x_i, \bar Y^h_{nh}=y^n_k, \bar R^h_{nh}=r^n_j)\simeq u^n_{i,k,j}$,}\smallskip\\
	\mbox{where $u^n_{\cdot,k,j}=(u^n_{i,k,j})_{i\in\mathcal{J}_M}$ solves the linear system}\smallskip\\
	\displaystyle
	A^n_{k,j}u^n_{\cdot,k,j}\smallskip\\
	\displaystyle
	=\sum_{a,b\in\{u,d\}} p_{ab}(n,k,j) Bf
	\Big(x_\cdot+\frac{\rho_1}{\sigma_Y}(y^{n+1}_{k_a (n,k)}-y^n_k)
	+\rho_2\sqrt{y}(r^{n+1}_{j_b (n,j)}-r^n_j), y^{n+1}_{k_a (n,k)}, r^{n+1}_{j_b (n,j)}\Big)+
	d^n.
	\end{array}
	\end{equation}

	\subsection{Pricing European and American options}\label{sect-alg}
	
	We are now ready to approximate the function $P_h$ solution to the dynamic programming principle \eqref{backward}. We consider the discretization scheme $(\bar X^h,\bar Y^h,\bar R^h)$ discussed in Section \ref{sec:approxPDE} and we use the approximation \eqref{E3} for the conditional expectations that have to be computed at each time step $n$. So, for every point $(x_i,y^n_k, r^n_j)\in \mathcal{X}_{M}\times\mathcal{Y}_n\times\mathcal{R}_n$, by \eqref{E3} we have
	$$
	\E\Big(P_h\big((n+1)h, X_{(n+1)h}^{nh,x_{i},y^n_k,r^n_j}, Y_{(n+1)h}^{nh,y^n_k}, R_{(n+1)h}^{nh,r^n_j}\big)\Big)\simeq u^n_{i,k,j}
	$$
	where $u^n_{\cdot,k,j}=(u^n_{i,k,j})_{i\in\mathcal{J}_M}$ solves the linear system
	\begin{equation}\label{E5}
	\begin{array}{l}
	\displaystyle
	A^n_{k,j}u^n_{\cdot,k,j}=
	B\!\!\!\sum_{a,b\in\{u,d\}}\!\!\! p_{ab}(n,k,j)\times \smallskip\\
	\displaystyle
	\times P_h
	\Big((n+1)h,y_\cdot+\frac{\rho_1}{\sigma_Y}(y^{n+1}_{k_a (n,k)}-y^n_k)
	+\rho_2\sqrt{y}(r^{n+1}_{j_b (n,j)}-r^n_j, y^n_k,r^n_j), y^{n+1}_{k_a (n,k)}, r^{n+1}_{j_b (n,j)}\Big)+d^n.
	\end{array}
	\end{equation}
	We then define the approximated price $\tilde P_h(nh,x,y,r)$ for $(x,y,r)\in \mathcal{X}_{M}\times\mathcal{Y}_n\times\mathcal{R}_n$ and $n=0,1,\ldots,N$ as
	\begin{equation}\label{backward-ter0-jumps}
	\begin{cases}
	\tilde P_h(T,x_i,y^N_k,r^N_j)= \Psi(x_i)%\quad \mbox{for $(y_i,y^N_k,r^N_j)\in \mathcal{X}_{M}\times\mathcal{Y}^h_N\times\mathcal{R}^h_N$}\\
	\quad \mbox{and as $n=N-1,\ldots,0$:}\\
	\displaystyle
	\tilde P_h(nh,x_i,y^n_k,r^n_j) =
	\max \Big\{\widehat \Psi(x_i),
	e^{-(\sigma_r r^n_j + \varphi_{nh})h}
	\tilde u^n_{i,k,j}\Big\}
	\end{cases}
	\end{equation}
	in which $\tilde u^n_{\cdot,k,j}=(\tilde u^n_{i,k,j})_{i\in\mathcal{J}_M}$ is the solution to the system in \eqref{E5} with $P_h$ replaced by $\tilde P_h$.
	
	Note that the system in \eqref{E5} requires the knowledge of the function $y\mapsto \tilde P_h((n+1)h,x,y,r)$ in points $x$'s that do not necessarily belong to the grid $\mathcal{X}_M$. Therefore, in practice we compute such a function by means of linear  interpolations, working as follows.
	For fixed $n,k,j,a,b$, we set $I_{n,k,j,a,b}(i)$, $i\in\mathcal{J}_M$, as the index such that
	$$
	x_i+\frac{\rho_1}{\sigma_Y}(y^{n+1}_{k_a (n,k)}-y^n_k)
	+\rho_2\sqrt{y}(r^{n+1}_{j_b (n,j)}-r^n_j)\in[x_{I_{n,k,j,a,b}(i)}, x_{I_{n,k,j,a,b}(i)+1}),
	$$
	with $I_{n,k,j,a,b}(i)=-M$ if $x_i+\frac{\rho_1}{\sigma_Y}(y^{n+1}_{k_a (n,k)}-y^n_k)
	+\rho_2\sqrt{y}(r^{n+1}_{j_b (n,j)}-r^n_j)<-M$ and
	$I_{n,k,j,a,b}(i)+1=M$ if $x_i+\frac{\rho_1}{\sigma_Y}(y^{n+1}_{k_a (n,k)}-y^n_k)
	+\rho_2\sqrt{y}(r^{n+1}_{j_b (n,j)}-r^n_j)>M$. We set
	$$
	q_{n,k,j,a,b}(i)
	=\frac{x_i+\frac{\rho_1}{\sigma_Y}(y^{n+1}_{k_a (n,k)}-y^n_k)
		+\rho_2\sqrt{y}(r^{n+1}_{j_b (n,j)}-r^n_j)-x_{I_{n,k,j,a,b}(i)}}{\Delta x}.
	$$
	Note that $q_{n,k,j,a,b}(i)\in[0,1)$. We  define
	\begin{align*}
	&(\mathfrak{I}_{a,b}\tilde P_h )((n+1)h,x_i, y^{n+1}_{k_a (n,k)}, r^{n+1}_{j_b (n,j)})
	=\tilde P_h ((n+1)h,x_{I_{n,k,j,a,b}(i)}, y^{n+1}_{k_a (n,k)}, r^{n+1}_{j_b (n,j)})\,(1-q_{n,k,j,a,b}(i))\\
	&\quad +\tilde P_h ((n+1)h,x_{I_{n,k,j,a,b}(i)+1}, y^{n+1}_{k_a (n,k)}, r^{n+1}_{j_b (n,j)})\,q_{n,k,j,a,b}(i)
	\end{align*}
	and we set
	$$
	\begin{array}{l}
	\displaystyle
	\tilde P_h\Big((n+1)h,x_i+\frac{\rho_1}{\sigma_Y}(y^{n+1}_{k_a (n,k)}-y^n_k)
	+\rho_2\sqrt{y}(r^{n+1}_{j_b (n,j)}-r^n_j), y^{n+1}_{k_a (n,k)}, r^{n+1}_{j_b (n,j)}\Big)\smallskip\\
	=(\mathfrak{I}_{a,b}\tilde P_h)((n+1)h,x_i, y^{n+1}_{k_a (n,k)}, r^{n+1}_{j_b (n,j)}).
	\end{array}
	$$
	Therefore, starting from \eqref{E5}, in practice the function $\tilde u^n_{\cdot,k,j}=(\tilde u^n_{i,k,j})_{i\in\mathcal{J}_M}$ in \eqref{backward-ter0-jumps} is taken as the solution to the linear system
	\begin{equation}\label{u-interp}
	A^n_{k,j}\tilde u^n_{\cdot,k,j}
	= B\sum_{a,b\in\{u,d\}}\!\!\! p_{ab}(n,k,j)
	(\mathfrak{I}_{a,b}\tilde P_h)((n+1)h,x_\cdot, y^{n+1}_{k_a (n,k)}, r^{n+1}_{j_b (n,j)})+ d^n.
	\end{equation}
	We can then state our final numerical procedure:
	\begin{equation}\label{backward-ter}
	\begin{cases}
	\tilde P_h(T,x_i,y^N_k,r^N_j)= \Psi(x_i)\quad %\mbox{for $(y_i,y^N_k,r^N_j)\in \mathcal{X}_{M}\times\mathcal{Y}^h_N\times\mathcal{R}^h_N$}\\
	\mbox{and as $n=N-1,\ldots,0$:}\\
	\displaystyle
	\tilde P_h(nh,x_i,y^n_k,r^n_j) =
	\max \Big\{\widehat \Psi(x_i),
	e^{-(\sigma_r r^n_j + \varphi_{nh})h}
	\tilde u^n_{i,k,j}\Big\}
	\end{cases}
	\end{equation}
	$\tilde u^n_{\cdot,k,j}=(\tilde u^n_{i,k,j})_{i\in\mathcal{J}_M}$ being the solution to the system \eqref{u-interp}.

	\begin{remark}\label{interp}
		In the case of an infinite grid, that is $M=+\infty$, $i\mapsto I_{n,k,j,a,b}(i)$ is a translation: $I_{n,k,j,a,b}(i)=I_{n,k,j,a,b}(0)+i$. So, $x_i\mapsto (\mathfrak{I}_{a,b}\tilde P_h )((n+1)h,x_i, y^{n+1}_{k_a (n,k)}, r^{n+1}_{j_b (n,j)})$ is just a linear convex combination of  translations of $x_i\mapsto \tilde P_h ((n+1)h,x_i, y^{n+1}_{k_a (n,k)}, r^{n+1}_{j_b (n,j)})$.
	\end{remark}

	\subsection{Stability analysis of the hybrid tree/finite-difference method}\label{sect:stability}
	
	We analyze here the stability of the resulting tree/finite-difference scheme. To this purpose, we consider a norm, defined on functions of the variables $(x,y,r)$, which is the uniform norm with respect to the volatility and the interest rate components $(y,r)$ and coincides with the standard $l_2$ norm with respect to the direction $x$ (see next \eqref{norm}). The choice of the $l_2$ norm allows one to perform a von Neumann analysis  in the component $x$ on the infinite grid $\mathcal{X}=\{x_i=X_0+i\dx\}_{i\in\Z}$, that is, without truncating the domain and without imposing boundary conditions.  Therefore, our stability analysis does not take into account boundary effects. This approach is extensively used in the literature, see e.g. \cite{duffy}, and yields good criteria on the robustness of the algorithm independently of the boundary conditions.
	
	Let us first write down explicitly the scheme \eqref{backward-ter} on the infinite grid $\mathcal{X}=\{x_i\}_{i\in\Z}$.
	For a fixed  function $f=f(t,x,y,r)$, we set $g=f$ (in the case of American options) or $g=0$ (in the case of European options) and we consider the numerical scheme given by
	\begin{equation} \label{backward-ter-inf-noL}
	\begin{cases}
	F_h(T,x_i,y^N_k,r^N_j)= f(T,x_i,y^N_k,r^N_j)\quad
	%\mbox{for $(y_i,y^N_k,r^N_j)\in \mathcal{X} \times\mathcal{Y}^h_N\times\mathcal{R}^h_N$}		\\
	\mbox{and as $n=N-1,\ldots,0$:}\\
	\displaystyle
	F_h(nh,x_i,y^n_k,r^n_j) =\max \Big\{g(nh,x_i,y^n_k,r^n_j),
	e^{-(\sigma_r r^n_j + \varphi_{nh})h} u^n_{i,k,j}\Big\}
	\end{cases}
	\end{equation}
	where $u^n_{\cdot,k,j}=(u^n_{i,k,j})_{i\in\Z}$ is the solution to
	\begin{equation}\label{scheme_bates}
	\begin{array}{l} (\alpha_{n,k,j}-\beta_{n,k})u^n_{i-1,k,j}+(1+2\beta_{n,k})u^n_{i,k,j}-(\alpha_{n,k,j}+\beta_{n,k})u^n_{i+1,k,j}
	\smallskip\\
	\displaystyle
	=\sum_{a,b\in\{d,u\}}p_{ab}(n,k,j)\times\Big[ (\mathfrak{I}_{a,b}F_h)((n+1)h,x_i, y^{n+1}_{k_a (n,k)}, r^{n+1}_{j_b (n,j)})+\smallskip\\
	\displaystyle \quad
	+h\dx\sum_l\nu(\xi_l)\big((\mathfrak{I}_{a,b}F_h)((n+1)h, x_{i+l}, y^{n+1}_{k_a (n,k)}, r^{n+1}_{j_b (n,j)})\smallskip\\
	\displaystyle \quad
	-(\mathfrak{I}_{a,b}F_h)((n+1)h, x_{i}, y^{n+1}_{k_a (n,k)}, r^{n+1}_{j_b (n,j)})\big)\Big],
	\end{array}
	\end{equation}
	in which $\alpha_{n,k,j}$ and $\beta_{n,k,j}$ are the coefficients $\alpha$ and $\beta$ defined in \eqref{alphabeta_impl} when evaluated in the pair $(y^n_k, r^n_j)$. Note that \eqref{scheme_bates} is simply the  linear system \eqref{u-interp} on the infinite grid, with $d^n\equiv 0$ (no boundary conditions are needed). Let us stress that in next Remark \ref{antitrasf} we will see that, since $\beta_{n,k}\geq 0$, a solution to \eqref{scheme_bates} does exist, at least for ``nice'' functions $f$. It is clear that the case $g=f$ is linked to the American algorithm whereas the case $g=0$ is connected to the European one: \eqref{backward-ter-inf-noL} gives our numerical approximation of the function
	\begin{equation}\label{F}
	F(t,x,y,r)=\left\{
	\begin{array}{ll}
	\displaystyle
	\E\Big(e^{-(\sigma_r\int_t^TR^{t,r}_sds+\int_t^T\varphi_sds)}f(T,X_T^{t,x,y,r},Y^{t,y}_T, R^{t,r}_T)\Big)
	&\mbox{ if } g=0,\smallskip\\
	\displaystyle
	\sup_{\tau\in\mathcal{T}_{t,T}}\E\Big(e^{-(\sigma_r\int_t^{\tau}R_s^{t,r}ds+\int_t^\tau \varphi_sds)}f(\tau,X_\tau^{t,x,y,r},Y^{t,y}_\tau, R^{t,r}_\tau)\Big)
	&\mbox{ if } g=f,
	\end{array}
	\right.
	\end{equation}
	at times $nh$ and in the points of the grid $\mathcal{X}\times \mathcal{Y}_n\times \mathcal{R}_n$.
	
	\subsubsection{The ``discount truncated scheme'' and its stability}
	
	In our stability analysis, we consider a numerical scheme which is a slight modification of \eqref{backward-ter-inf-noL}: we fix a (possibly large) threshold $\vartheta>0$ and we consider the scheme
	\begin{equation} \label{backward-ter-inf}
	\begin{cases}
	F^\vartheta_h(T,x_i,y^N_k,r^N_j)= f(T,x_i,y^N_{k},r^N_{j})\quad %\mbox{for $(y_i,y^N_k,r^N_j)\in \mathcal{X} \times\mathcal{Y}^h_N\times\mathcal{R}^h_N$}		\\
	\mbox{and as $n=N-1,\ldots,0$:}\\
	\displaystyle
	F^\vartheta_h(nh,x_i,y^n_k,r^n_j)
	=\max \Big\{g(nh,x_i,y^n_k,r^n_j),
	e^{-(\sigma_r r^n_j\,\ind{\{r^n_j>-\vartheta\}} + \varphi_{nh})h}
	u^n_{i,k,j}\Big\}
	\end{cases}
	\end{equation}
	with $g=f$ or $g=0$, where $u^n_{\cdot,k,j}=(u^n_{i,k,j})_{i\in\Z}$ is the solution to \eqref{scheme_bates}, with $(\mathfrak{I}_{a,b}F_h)$ replaced by
	$(\mathfrak{I}_{a,b}F^\vartheta_h)$.
	Let us stress that the above scheme \eqref{backward-ter-inf-noL} really differs from \eqref{backward-ter-inf} only when $\sigma_r>0$ (stochastic interest rate). And in this case, in the discounting factor of \eqref{backward-ter-inf} we do not allow  $r^n_j$ to run everywhere on its grid: in the original scheme \eqref{backward-ter-inf-noL}, the exponential contains the term $r^n_j$ whereas in the present scheme \eqref{backward-ter-inf} we put $r^n_j\ind{\{r^n_j>-\vartheta\}}$, so we kill the points of the grid $\mathcal{R}_n$ below the threshold $-\vartheta$. And in fact, \eqref{backward-ter-inf} aims to numerically compute the function
	\begin{equation}\label{FL}
	F^\vartheta(t,x,y,r)=\left\{
	\begin{array}{ll}
	\displaystyle
	\E\Big(e^{-(\sigma_r\int_t^TR^{t,r}_s\,\ind{\{R^{t,r}_s>-\vartheta\}}ds
		+\int_t^T\varphi_sds)}f(T,X_T^{t,x,y,r},Y^{t,y}_T, R^{t,r}_T)\Big)
	&\mbox{ if } g=0,\smallskip\\
	\displaystyle
	\sup_{\tau\in\mathcal{T}_{t,T}}
	\E\Big(e^{-(\sigma_r\int_t^{\tau}R_s^{t,r}\,\ind{\{R^{t,r}_s>-\vartheta\}}ds+\int_t^\tau \varphi_sds)}f(\tau,X_\tau^{t,x,y,r},Y^{t,y}_\tau, R^{t,r}_\tau)\Big)
	&\mbox{ if } g=f,
	\end{array}
	\right.
	\end{equation}
	at times $nh$ and in the points of the grid $\mathcal{X}\times \mathcal{Y}_n\times \mathcal{R}_n$.
	Recall that in practice $h$ is small but fixed, so that the implemented scheme incorporates a threshold (see for instance the tree given in Figure \ref{fig:treeYR}). And actually, in our numerical experiments we observe a real stability. However, we will discuss later on how much one can lose with respect to the solution of \eqref{backward-ter-inf-noL}.

	For $n=N,\ldots,0$, the scheme \eqref{backward-ter-inf} returns a function in the variables $(x,y,r)\in \mathcal{X}\times\mathcal{Y}_n\times\mathcal{R}_n$. Note that
	$\mathcal{Y}_n\times\mathcal{R}_n\subset I^Y_n\times I^R_n$, where
	$$
	I^Y_n=[y^n_0,y^n_n]\quad\mbox{and}\quad I^R_n=[r^n_0, r^n_n],
	$$
	that is, the intervals between the smallest and the biggest node at time-step $n$:
	$$
	y^n_0=\Big(\sqrt {Y_0}-\frac{\sigma_Y} 2\,n\sqrt{h}\Big)^2\ind{\{\sqrt {Y_0}-\frac{\sigma_Y} 2\,n\sqrt{h}>0\}},\qquad
	y^n_n=\Big(\sqrt {Y_0}+\frac{\sigma_Y} 2\,n\sqrt{h}\Big)^2,
	$$
	$$
	r^n_0=-n\sqrt{h},\qquad
	r^n_n=n\sqrt{h}.
	$$
	As $n$ decreases to 0, the intervals $I^Y_n$ and $I^R_n$ are becoming smaller and smaller
	and at time 0 they collapse to the single point $y^0_0=Y_0$
	and $r^0_0=R_0=0$ respectively. So, the norm we are going to define  takes into account these facts: at time $nh$ we consider for $\phi=\phi(t,x,y,r)$ the norm
	\begin{equation}\label{norm}
	\|\phi(nh,\cdot)\|_n = \sup_{(y,r)\in I^Y_n\times I^R_n}\|\phi(nh,\cdot,y,r)\|_{l_2(\mathcal{X})} = \sup_{(y,r)\in I^Y_n\times I^R_n}\Big(\sum_{i\in\Z}|\phi(nh,x_i,y,r)|^2\Delta y\Big)^{\frac 12}.
	\end{equation}
	In particular,
	\begin{align*}
	&\|\phi(0,\cdot)\|_0 =\|\phi(0,\cdot,Y_0,R_0)\|_{l_2(\mathcal{X})} = \Big(\sum_{i\in\Z}|\phi(x_i,Y_0,R_0)|^2\Delta y\Big)^{1/2}\quad\mbox{and}\\
	&\|\phi(T,\cdot)\|_N
	\leq \sup_{(y,r)\in\R_+\times \R} \|\phi(x_i,y,r)\|_{l_2(\mathcal{X})}
	= \sup_{(y,r)\in\R_+\times \R} \Big(\sum_{i\in\Z}|\phi(x_i,y,r)|^2\Delta y\Big)^{1/2}.
	\end{align*}
	
	We are now ready to give our stability result.
	
	\begin{theorem}\label{prop-stability}
		Let $f\geq 0$ and, in the case $g=f$, suppose that
		$$
		\sup_{t\in[0,T]}|f(t,x,y,r)|\leq \gamma_T|f(T,x,y,r)|,
		$$
		for some $\gamma_T>0$. Then, for every $\vartheta>0$ the numerical scheme \eqref{backward-ter-inf} is stable with respect to the norm \eqref{norm}:
		$$
		\|F^\vartheta_h(0,\cdot)\|_0\leq C^{N,\vartheta}_T
		\|F^\vartheta_h(T,\cdot)\|_N
		=C^{N,\vartheta}_T
		\|f(T,\cdot)\|_N,\quad \forall h,\Delta y,
		$$
		where
		$$
		C^{N,\vartheta}_T=
		\left\{
		\begin{array}{ll}
		e^{2\lambda cT+\sigma_r\vartheta T-\sum_{n=1}^N\varphi_{nh}h}\stackrel{N\to\infty}{\longrightarrow}
		C^\vartheta_T=e^{2\lambda cT+\sigma_r\vartheta T-\int_0^T\varphi_{t}dt}&\mbox{ if } g=0,\smallskip\\
		\max\Big\{ \gamma_T, e^{2\lambda cT+\sigma_r\vartheta T-\sum_{n=1}^N\varphi_{nh}h}\Big\}\stackrel{N\to\infty}{\longrightarrow}
		C^\vartheta_T=\max\Big\{\gamma_T, e^{2\lambda cT+\sigma_r\vartheta T-\int_0^T\varphi_{t}dt}\Big\}&\mbox{ if } g=f,
		\end{array}
		\right.
		$$
		in which $c>0$ is such that $\sum_l \nu(\xi_l)\dx\leq \lambda c$.
		In the standard Bates model, that is $\sigma_r=0$ and deterministic interest rate $r_t=\varphi_t$, the discount truncated scheme \eqref{backward-ter-inf} coincides with the standard scheme \eqref{backward-ter} and the stability follows for \eqref{backward-ter}.
		
	\end{theorem}
	
	\textbf{Proof.}
	In order to simplify the notation, we set $g^n_{i,k,j}=g(nh,x_i,y^n_k,r^n_j)$ and, similarly, $F^n_{i,k,j}=F^\vartheta_h(nh,x_i,y^n_k,r^n_j)$, $(\mathfrak{I}_{a,b}F_h^{n+1})_{i,k_a,j_b}
	=(\mathfrak{I}_{a,b}F^\vartheta_h)((n+1)h,x_i,y^{n+1}_{k_a (n,k)},r^{n+1}_{j_b (n,j)})$
	(we have also dropped the dependence on $\vartheta$).
	The scheme \eqref{backward-ter-inf} says that, at each time step $n<N$ and for each fixed $0\leq k,j\leq n$,
	\begin{equation}\label{max}
	F^n_{i,k,j} =
	\max \Big\{g^n_{i,k,j},
	e^{-(\sigma_r r^n_j\ind{\{r^n_j>-\vartheta\}} + \varphi_{nh})h}u^{n}_{i,k,j} \Big\},
	\end{equation}
	where, according to \eqref{scheme_bates}, $u^{n}_{i,k,j}$ solves
	\begin{equation}\label{scheme_bates1}
	\begin{array}{l} (\alpha_{n,k,j}-\beta_{n,k})u^n_{i-1,k,j}+(1+2\beta_{n,k})u^n_{i,k,j}-(\alpha_{n,k,j}+\beta_{n,k})u^n_{i+1,k,j}
	\smallskip\\
	\displaystyle
	=\sum_{a,b\in\{d,u\}}p_{ab}(n,k,j) \Big((\mathfrak{I}_{a,b}F^{n+1})_{i,k_a,j_b}+h\dx\sum_l\nu(\xi_l)
	\big[(\mathfrak{I}_{a,b}F^{n+1})_{i+l,k_a,j_b}-(\mathfrak{I}_{a,b}F^{n+1})_{i,k_a,j_b}\big]\Big).
	\end{array}
	\end{equation}
	Let $\mathfrak{F}\varphi$ denote the Fourier transform of $\varphi\in l_2(\mathcal{X})$, that is,
	\def\i{\mathbf{i}\,}
	\begin{equation*}
	\mathfrak{F}\varphi(\theta) = \frac{\dx}{\sqrt{2\pi}}\sum_{
		s\in\Z} \varphi_s e^{-\i s\Delta y\theta},\quad \theta\in\R,
	\end{equation*}
	$\i$ denoting the imaginary unit. We get from \eqref{scheme_bates1}
	\begin{equation}\label{pp}
	\begin{array}{l}
	\Big((\alpha_{n,k,j}-\beta_{n,k})e^{-\i\theta\dx}+1+2\beta_{n,k}
	-(\alpha_{n,k,j}+\beta_{n,k})e^{\i\theta\dx}\Big) \mathfrak{F} u^n_{k,j}(\theta)
	\medskip\\
	=\Big(1+h\dx\sum_l\nu(\xi_l)(e^{\i l\theta\dx}-1)\Big)\sum_{a,b\in\{d,u\}}p_{ab}(n,k,j)
	\mathfrak{F}(\mathfrak{I}_{a,b}F^{n+1})_{k_a,j_b}(\theta).
	\end{array}
	\end{equation}
	Note that
	\begin{align*}
	|(\alpha_{n,k,j}-\beta_{n,k})&e^{-\i\theta\dx}+1+2\beta_{n,k}
	-(\alpha_{n,k,j}+\beta_{n,k})e^{\i\theta\dx}|\\
	&\geq\big|\mathfrak{Re}\big[(\alpha_{n,k,j}-\beta_{n,k})e^{-\i\theta\dx}
	+1+2\beta_{n,k}-(\alpha_{n,k,j}+\beta_{n,k})e^{\i\theta\dx}\big]\big|\\
	&=1+2\beta_{n,k}(1-\cos(\theta\dx))
	\geq 1,
	\end{align*}
	for every $\theta \in [0,2\pi)$ (recall that $\beta_{n,k}\geq 0$).
	And since $\sum_l \nu(\xi_l)\dx\leq \lambda c$, we obtain
	\begin{align*}
	| \mathfrak{F} u^n_{k,j}(\theta)|
	%&\leq  |(\alpha_{n,k,j}-\beta_{n,k})e^{-\i\theta\dx}+1+2\beta_{n,k}-(\alpha_{n,k,j}+\beta_{n,k})e^{\i\theta\dx}|\times | \mathfrak{F} u^n_{k,j}(\theta)|\\
	&\leq \Big(1+h\dx\sum_{l\in\Z} | e^{\i l\theta\dx}-1|\nu(\xi_l)\Big)
	\sum_{a,b\in\{d,u\}}p_{ab}(n,k,j)|  \mathfrak{F}(\mathfrak{I}_{a,b}F^{n+1})_{k_a,j_b}(\theta)|\\
	& \leq (1+2\lambda ch)\sum_{a,b\in\{d,u\}}p_{ab}(n,k,j)|  \mathfrak{F}(\mathfrak{I}_{a,b}F^{n+1})_{k_a,j_b}(\theta)|.
	\end{align*}
	Therefore,
	\begin{align*}
	\| \mathfrak{F} u^n_{k,j}\|_{L^2([0,2\pi),\mathrm{Leb})}
	&\leq (1+2\lambda ch)\sum_{a,b\in\{d,u\}}p_{ab}(n,k,j)
	\|\mathfrak{F}(\mathfrak{I}_{a,b}F^{n+1})_{k_a,j_b}\|_{L^2([0,2\pi),\mathrm{Leb})}.
	\end{align*}
	We use now the Parseval identity $\|\mathfrak{F} \varphi\|_{L^2([0,2\pi),\mathrm{Leb})}
	=\|\varphi\|_{l_2(\mathcal{X})}$ and we get
	\begin{align*}
	\| u^n_{\cdot, k,j}\|_{l^2(\mathcal{X})}
	&\leq (1+2\lambda ch)\sum_{a,b\in\{d,u\}}p_{ab}(n,k,j)
	\|(\mathfrak{I}_{a,b}F^{n+1})_{\cdot, k_a,j_b}\|_{l^2(\mathcal{X})}\\
	&= (1+2\lambda ch)\sum_{a,b\in\{d,u\}}p_{ab}(n,k,j)
	\|F^{n+1}_{\cdot, k_a,j_b}\|_{l^2(\mathcal{X})},
	\end{align*}
	the first equality following from the fact that $i\mapsto (\mathfrak{I}_{a,b}F^{n+1})_{i, k_a,j_b}$ is a linear convex combination of translations of $i\mapsto F^{n+1}_{i, k_a,j_b}$ (see Remark \ref{interp}).
	This gives
	\begin{align*}
	\sup_{0\leq k,j\leq n}\|e^{-(\sigma_r r^n_j\,\ind{\{r^n_j>-\vartheta\}} + \varphi_{nh})h} u^n_{\cdot,k,j}\|_{l_2(\mathcal{X})}
	\leq (1+2\lambda c h)e^{\sigma_r\vartheta h-\varphi_{nh}h}
	\sup_{0\leq k,j\leq n+1}\|F^{n+1}_{\cdot,k,j}\|_{l_2(\mathcal{X})}
	\end{align*}
	and from \eqref{max}, we obtain
	\begin{align*}
	\sup_{0\leq k,j\leq n}\| F^n_{\cdot,k,j}\|_{l_2(\mathcal{X})}
	&\leq \max\Big(\sup_{0\leq k,j\leq n}\|g^n_{\cdot,k,j}\|_{l_2(\mathcal{X})}, (1+2\lambda c h)e^{\sigma_r\vartheta h-\varphi_{nh}h}
	\sup_{0\leq k,j\leq n+1}\|F^{n+1}_{\cdot,k,j}\|_{l_2(\mathcal{X})}\Big).
	\end{align*}
	We now continue assuming that $g=f$, the case $g=0$ following in a similar way. So,
	\begin{align*}
	\sup_{0\leq k,j\leq n}\| F^n_{\cdot,k,j}\|_{l_2(\mathcal{X})}
	&\leq \max\Big(\gamma_T\|f(T,\cdot)\|_N, (1+2\lambda c h)e^{\sigma_r\vartheta h-\varphi_{nh}h}
	\sup_{0\leq k,j\leq n+1}\|F^{n+1}_{\cdot,k,j}\|_{l_2(\mathcal{X})}\Big).
	\end{align*}
	For $n=N-1$ we then obtain
	\begin{align*}
	\sup_{0\leq k,j\leq n}\| F^{N-1}_{\cdot,k,j}\|_{l_2(\mathcal{X})}
	&\leq \max\Big(\gamma_T\|f(T,\cdot)\|_N, (1+2\lambda c h)e^{\sigma_r\vartheta h-\varphi_{(N-1)h}h}
	\|f(T,\cdot)\|_{N}\Big)
	\end{align*}
	and by iterating the above inequalities, we finally get
	\begin{align*}
	\| F^0\|_0
	=\|F^0_{\cdot,0,0}\|_{l_2(\mathcal{X})}
	&\leq \max\Big(\gamma_T\|f(T,\cdot)\|_N, (1+2\lambda c h)^Ne^{N\sigma_rL h-\sum_{n=1}^N\varphi_{nh}h}
	\|f(T,\cdot)\|_N\Big).
	\end{align*}
	\cvd
	
	\begin{remark}\label{antitrasf}
		We have incidentally proved that, as $n$ varies, the solution $u^n_{\cdot,k,j}$ to the infinite linear system \eqref{scheme_bates} actually exists and is unique if $\|f(T,\cdot)\|_N<\infty$. In fact, starting from equality \eqref{pp}, we define the function $\psi_{k,j}(\theta)$, $\theta\in[0,2\pi)$, by
		$$
		\begin{array}{l}
		\Big((\alpha_{n,k,j}-\beta_{n,k})e^{-\i\theta\dx}+1+2\beta_{n,k}
		-(\alpha_{n,k,j}+\beta_{n,k})e^{\i\theta\dx}\Big) \psi_{k,j}(\theta)
		\smallskip\\
		=\Big(1+h\dx\sum_l\nu(\xi_l)(e^{\i l\theta\dx}-1)\Big)\sum_{a,b\in\{d,u\}}p_{ab}(n,k,j)
		\mathfrak{F}(\mathfrak{I}_{a,b}F^{n+1})_{k_a,j_b}(\theta).
		\end{array}
		$$
		As noticed in the proof of Proposition \ref{prop-stability}, the factor multiplying $\psi_{k,j}(\theta)$ is different from zero because $\beta_{n,k}\geq 0$. So, the definition of $\psi_{k,j}$ is well posed and moreover, $\psi_{k,j}\in L^2([0,2\pi,),\mathrm{Leb})$. We now set $u^n_{\cdot,k,j}$ as the inverse Fourier transform of $\psi_{k,j}$, that is,
		$$
		u^n_{l,k,j}=\frac 1{\Delta y\sqrt{2\pi}}\int_0^{2\pi}\psi_{k,j}(\theta)e^{\mathbf{i}\,l\theta\Delta y}d\theta,\quad l\in\Z.
		$$
		Straightforward computations give that $u^n_{\cdot,k,j}$ fulfils the equation system \eqref{scheme_bates}.
	\end{remark}
	
	Of course, Theorem \ref{prop-stability} gives a stability property for the scheme introduced in \cite{bcz-hhw} for the Heston-Hull-White model: just take $\lambda=0$ (no jumps are considered).
	
	\subsubsection{Back to the original scheme \eqref{backward-ter-inf-noL}}\label{ritornoalloschema}
	
	Let us now discuss what may happen when one introduces the threshold $\vartheta $. We recall that the original scheme \eqref{backward-ter-inf-noL} gives the numerical approximation of the function $F$ in \eqref{F} whereas
	the discount truncated scheme \eqref{backward-ter-inf} aims to numerically compute the function $F^\vartheta$ in \eqref{FL}. Proposition \ref{prop-FL} below shows that, under standard hypotheses, $F^\vartheta$ tends to $F$ as $\vartheta\to\infty$ very fast. This means that, in practice, we lose very few in using \eqref{backward-ter-inf} in place of \eqref{backward-ter-inf-noL}.
	\begin{proposition}\label{prop-FL}
		Suppose that $f=f(t,x,y,r)$ has a polynomial growth in the variables $(x,y,r)$, uniformly in $t\in[0,T]$. Let $F$ and $F^\vartheta$, with $\vartheta>0$, be defined in \eqref{F} and \eqref{FL} respectively. Then there exist positive constants $c_T$ and $C_T(x,y,r)$ (depending on $(x,y)$ in a polynomial way and on $r$ in an exponential way) such that for every $\vartheta>0$
		$$
		|F(t,x,y,r)-F^\vartheta(t,x,y,r)|
		\leq \sigma_r C_T(x,y,r)
		e^{-c_T|\vartheta+xe^{-\kappa_r(T-t)}|^2},
		$$
		for every $t\in[0,T]$ and $(x,y,r)\in\R\times\R_+\times \R$.
	\end{proposition}
	
	\begin{proof}
		In the following, $C$ denotes a positive constant, possibly changing from line to line, which depends on $(x,y,r)$  polynomially in $(x,y)$ and exponentially in $r$.  We have
		\begin{equation}\label{itsthefinalerror}
		\begin{split}
		&|F(t,x,y,r)-F^\vartheta(t,x,y,r)|\\
		&\leq C\E\left(\sup_{t\leq u\leq T}|f(u,X_u^{t,x,y,r}, Y_u^{t,y}, R_u^{t,r})|
		\times e^{-\sigma_r\int_t^uR_s^{t,r}\ind{\{R_s^{t,r}>-\vartheta\}}ds}\times
		\left(e^{-\sigma_r\int_t^uR_s^{t,r}\ind{\{R_s^{t,r}<-\vartheta\}}ds}-1\right)
		\right).
		\end{split}
		\end{equation}
		Set now
		$$
		\tau_{-\vartheta}^{t,r}=\inf\{s\geq t\,:\, R_s^{t,r}\leq -\vartheta\}.
		$$
		Notice that $\{  R_s<-\theta\}\subseteq \{\tau_{-\theta} < s\}\subseteq \{ \tau_{-\theta} < T   \}$. Therefore, one has  $\ind{\{R_s^{t,r}<-\vartheta\}}\leq \ind{\{\tau_{-\vartheta}^{t,r}< T\}}$ and
		$$
		-\sigma_r\int_t^uR_s^{t,r}\ind{\{R_s^{t,r}<-\vartheta\}}ds= \int_t^u| \sigma_r R_s^{t,r}| \ind{\{R_s^{t,r}<-\vartheta\}}ds\leq   \sigma_r \ind{\{\tau_{-\vartheta}^{t,r}< T\}}\int_t^u|  R_s^{t,r}| ds.
		$$
		So we can write
		$$
		0\leq e^{-\sigma_r\int_t^uR_s^{t,r}\ind{\{R_s^{t,r}<-\vartheta\}}ds}-1\leq e^{\sigma_r \ind{\{\tau_{-\vartheta}^{t,r}< T\}}\int_t^u|  R_s^{t,r}| ds}-1= \left(e^{\sigma_r \int_t^u|  R_s^{t,r}| ds}-1\right) \ind{\{\tau_{-\vartheta}^{t,r}< T\}}
		$$
		Substituting in \eqref{itsthefinalerror} and applying H\"{o}lder inequality, we get
		\begin{align}
		\nonumber&|F(t,x,y,r)-F^\vartheta(t,x,y,r)|\\
		&\nonumber\leq C\E\left(\sup_{t\leq u\leq T}|f(u,X_u^{t,x,y,r}, Y_u^{t,y}, R_u^{t,r})|
		e^{-\sigma_r\int_t^uR_s^{t,r}\ind{\{R_s^{t,r}>-\vartheta\}}ds}
		\left(e^{\sigma_r \int_t^u|  R_s^{t,r}| ds}-1\right)  \ind{\{\tau_{-\vartheta}^{t,r}< T\}}
		\right)\\
		\nonumber&\leq C\E\left(\sup_{t\leq u\leq T}|f(u,X_u^{t,x,y,r}, Y_u^{t,y}, R_u^{t,r})|^2
		e^{2\sigma_r\int_t^u|R_s^{t,r}|ds}
		\left(e^{\sigma_r \int_t^u|  R_s^{t,r}| ds}-1\right)^2\right)^{1/2} \!\!\!\times \\\nonumber&\qquad \P\left(\ind{\{\tau_{-\vartheta}^{t,r}< T\}}
		\right)^{1/2}\\
		\nonumber&\leq C\E\left(\sup_{t\leq u\leq T}|f(u,X_u^{t,x,y,r}, Y_u^{t,y}, R_u^{t,r})|^2
		\times e^{4\sigma_r \int_t^T|  R_s^{t,r}| ds}\right)^{1/2} \!\!\!\!\times \P\left(\ind{\{\tau_{-\vartheta}^{t,r}< T\}}
		\right)^{1/2}\\
		&\leq C\E\left(\sup_{t\leq u\leq T}|f(u,X_u^{t,x,y,r}, Y_u^{t,y}, R_u^{t,r})|^4
		\right)^{1/4}\!\!\!\! \times \E\left( e^{8\sigma_r \int_t^T|  R_s^{t,r}| ds} \right)^{1/4}\!\!\!\!\times\P\left(\ind{\{\tau_{-\vartheta}^{t,r}< T\}}
		\right)^{1/2}.\label{boh} 
		%\\&\qquad \nonumber.
		\end{align}	
		The first term in the left hand side of \eqref{boh} is finite since $f$ has polynomial growth in the space variables, uniformly in the time variable, and by using standard estimates.
		Also the second term in \eqref{boh} is finite. This is because, for every $c>0$, 
		\begin{equation}\label{boh2}
		\E\left( e^{c\sup_{t\leq s\leq T}|  R_s^{t,r}|}\right)<\infty.
		\end{equation}
		In fact, recalling that that $R_s^{t,r}=re^{-\kappa_r(s-t)}+\int_t^se^{-\kappa_r(s-u)}dW^2_u$, \eqref{boh2}  follows from the fact that,
		for a  Brownian motion $W$,  $\sup_{0\leq s\leq T}|W_s|$ has finite exponential moments of any order, for every $T>0$. This is true since $\sup_{0\leq s\leq T}|W_s|\leq \sup_{0\leq s\leq T}W_s+ \sup_{0\leq s\leq T}(-W_s)$ and $\E(e^{p\sup_{0\leq s\leq T}W_s})$ $<\infty$ for every $p>0$. As regards the third term  in \eqref{boh}, note that
		\begin{align*}
		&\P(\tau_{-\vartheta}^{t,r}\leq T)
		=\P(\inf_{s\in[t,T]}R_s^{t,r}<-\vartheta)
		=\P\Big(\inf_{s\in[t,T]}\Big(re^{-\kappa_r(s-t)}+\int_t^se^{-\kappa_r(s-u)}dW^2_u\Big)<-\vartheta\Big)\\
		&\leq \P\Big(\sup_{s\in[t,T]}\Big|\int_t^se^{\kappa_ru}dW^2_u\Big|>\vartheta+re^{-\kappa_r(T-t)}\Big)
		\leq 2\exp\Big(-\frac{|\vartheta+re^{-\kappa_r(T-t)}|^2}{2\int_t^Te^{2\kappa_ru}du} \Big).
		\end{align*}
		By inserting the above estimates in \eqref{boh}, we get the result.
	\end{proof}
	
	\subsubsection{Further remarks}
	
	As already stressed, the introduction of the threshold $-\vartheta$ allows one to handle the discount term. In order to get rid of the discount, a possible approach consists in the use of a transformed function, as  developed by several authors (see e.g. Haentjens and in't Hout \cite{Hint} and references therein). This is a nice fact for European options (PIDE problem), being on the contrary a non definitive tool when dealing with American options (obstacle PIDE problem). Let us see why.
	
	First of all, let us come back to the model for the triple $(X,Y,R)$, see \eqref{XYR-dyn}. The infinitesimal generator is
	\begin{equation}\label{Lt}
	\begin{array}{ll}
	\L_tu
	=&
	\displaystyle
	\Big(\sigma_rr+\varphi_t-\delta-\frac 12y\Big)\partial_xu
	+\kappa_Y(\theta_Y-y)\partial_yu
	-\kappa_rr\partial_ru\smallskip\\
	&\displaystyle
	+\frac 12\Big(y\partial^2_{xx}u+\sigma_Y^2y\partial^2_{yy}u+\partial^2_{rr}u
	+2\rho_1\sigma_Yy\partial^2_{xy}u+2\rho_2\sqrt  y\,\partial^2_{xr}u\Big)\smallskip\\
	&\displaystyle
	+ \displaystyle\int_{-\infty}^{+\infty} \left[ u(t,x+\xi;y,r)-u(t,x;y,r)\right]\nu(\xi) d\xi.
	\end{array}
	\end{equation}
	We set
	$$
	G(t,r)=\E\Big(e^{-\sigma_r\int_t^T R^{t,r}_s ds}\Big)
	$$
	and we recall several known facts: one has (see e.g. \cite{LL})
	\begin{equation}\label{G}
	G(t,r)=e^{-r\sigma_r\Lambda(t,T)-\frac {\sigma_r^2}{2\kappa_r^2}(\Lambda(t,T)-T+t)-\frac {\sigma_r^2}{4\kappa_r}\Lambda^2(t,T)},\quad \Lambda(t,T)=\frac{1-e^{-\kappa_r (T-t)}}{\kappa_r}
	\end{equation}
	and moreover, $G$ solves the PDE
	\begin{equation}\label{PDE-G}
	\begin{array}{l}
	\displaystyle
	\partial_tG
	-\kappa_r x\partial_x G+\frac 12\partial_{rr}^2 G-\sigma_rrG=0,\quad t\in[0,T), r\in\R,\smallskip\\
	G(T,r)=1.
	\end{array}
	\end{equation}

	\begin{lemma}\label{G-lemma}
		Let $\L_t$ denote the infinitesimal generator in \eqref{Lt}. Set $\overline{u}=u\cdot G^{-1}$. Then
		$$
		\partial_tu +\L_tu -ru= G\big(\partial_t\bu +\overline \L_t \bu \big),
		$$
		where
		$$
		\overline \L_t=\L_t-\sigma_r\frac{1-e^{-\kappa_r (T-t)}}{\kappa_r}\big[\rho_{2}\sqrt y \partial_{x}\bu+\partial_r\bu\big].
		$$
	\end{lemma}
	
	\textbf{Proof.}
	Since $G$ depends on $t$ and $r$ only, straightforward computations give
	\begin{align*}
	\partial_tu +\L_tu -xu=&
	G\big[\partial_t \bu +\L_t\bu\big]
	+\partial_rG(t,r)
	\big[\rho_{2}\sqrt y \partial_{x}\bu+\partial_r\bu\big]
	+\bu\big[\partial_tG-\kappa_rr\partial_r G+\frac 12\partial_{rr}^2 G-\sigma_rrG\big].
	\end{align*}
	By \eqref{PDE-G}, the last term is null.  The statement now follows by observing that $\partial_r\ln G(t,r)=-\sigma_r\frac{1-e^{-\kappa_r (T-t)}}{\kappa_r}$. \cvd
	
	\medskip
	
	We notice that the operator $\overline{L}_t$ in Lemma \ref{G-lemma} is the infinitesimal generator of the jump-diffusion process $(\overline{X},\overline{Y},\overline{R})$ which solves the stochastic differential equation as in \eqref{XYR-dyn},  with the same diffusion coefficients and jump-terms but with the new drift coefficients
	$$
	\mu_{\overline{X}}(t,y,r)=\mu_X(y,r)-\sigma_r\frac{1-e^{-\kappa_r (T-t)}}{\kappa_r}
	\rho_{2}\sqrt y,\qquad
	\mu_{\overline{Y}}(y)\equiv \mu_Y(y),
	$$
	$$
	\mu_{\overline{R}}(r)=\mu_R(t,r)-\sigma_r\frac{1-e^{-\kappa_r (T-t)}}{\kappa_r}.
	$$
	Let us first discuss the scheme \eqref{backward-ter-inf-noL} with $g=0$ (European options), which gives the numerical approximation for the function $F$ in \eqref{F}. By passing to the associated PIDE,  Lemma \ref{G-lemma} says that
	$$
	F(t,x,y,r)
	=G(t,r)\overline{F}(t,x,y,r),
	$$
	where
	$$
	\overline{F}(t,x,y,r)
	=\E(e^{-\int_t^T\varphi_sds}f(T,\overline{X}_T^{t,x,y,r},\overline{Y}_T^{t,y},\overline{R}_T^{t,r})).
	$$
	Therefore, in practice one has to numerically evaluate the function $\overline{F}$. By using our hybrid tree/finite-difference approach, this means to consider the scheme in \eqref{backward-ter-inf}, with the new coefficient $\overline{\alpha}_{n,k,j}$ (written starting from the new drift coefficients)  but with a discount depending on the (deterministic) function $\varphi$ only, that is, with $e^{-(\sigma_rr^n_j\ind{\{r^n_j>-L\}}+\varphi_{nh})h}$ replaced by $e^{-\varphi_{nh}h}$. And the proof of the Proposition \ref{prop-stability} shows that one gets
	$$
	\|\overline{F}_h(0,\cdot)\|_0\leq \max\big(\gamma_T,e^{2\lambda cT-\sum_{n=0}^N\varphi_{nh}h}\big)\|f(T,\cdot)\|_N.
	$$
	In other words, by using a suitable transformation, the European scheme is always stable and  no thresholds are needed.
	
	Let us discuss now the American case, that is, the scheme  \eqref{backward-ter-inf-noL} with $g=f$, giving an approximation of the function $F$ in \eqref{F}. One could  think to use the above transformation in order to get rid of the exponential depending on the process $R$. Set again
	$$
	\overline{F}(t,x,y,r)
	=G(t,r)^{-1}F(t,x,y,r).
	$$
	By using the associated obstacle PIDE problem, Lemma \ref{G-lemma} suggests that
	$$
	\overline{F}(t,x,y,r)
	=\sup_{\tau\in\mathcal{T}_{t,T}}\E(e^{-\int_t^\tau \varphi_sds}\overline{f}(\tau,\overline{X}_\tau^{t,x,y,r},\overline{Y}_\tau^{t,y},\overline{R}_\tau^{t,r})),
	$$
	with
	$
	\overline{f}(t,x,y,r)=G^{-1}(t,r)f(t,x,y,r).
	$
	So, in order to numerically compute $\overline{F}$, one needs to set up the scheme \eqref{backward-ter-inf} with the new coefficient $\overline{\alpha}_{n,k,j}$, with $f$ replaced by $\overline f$, $g=\overline{f}$ and with the discounting factor  $e^{-(\sigma_rr^n_j\ind{\{r^n_j>-L\}}+\varphi_{nh})h}$ replaced by $e^{-\varphi_{nh}h}$. So, again one is able to cancel the unbounded part of the discount. Nevertheless, the unpleasant point is that even if $\|f(T,\cdot)\|_N$ has a bound which is uniform in $N$ then $\|\overline{f}(T,\cdot)\|_N$ may not have because $G^{-1}(t,r)$ has an exponential containing $r$, see \eqref{G}. In other words, the unboundedness problem appears now in the obstacle.

	\section{The hybrid Monte Carlo and tree/finite-difference approach algorithms in practice}\label{practice}
	
	The present section is devoted to our numerical experiments. We first summarise the main steps of our algorithms and then we present several numerical tests.

	\subsection{A schematic sketch of the main computational steps in our algorithms}\label{sec:pseudoalg}
	In short, we outline here the main computational steps of the two proposed algorithms.
	
	First, the procedures need the following preprocessing steps, concerning the construction of the bivariate tree:
	\begin{itemize}
		\item[(\textsc{T1})] define a discretization of the time-interval $[0,T]$ in $N$ subintervals $[nh,(n+1)h]$, $n=0,\ldots,N-1$, with $h=T/N$;
		\item[(\textsc{T2})] for the process $Y$, set the binomial tree $y^n_k$, $0\leq k \leq n\leq N$, by using
		\eqref{state-space-X}, then compute the jump nodes $k_a (n,k)$
		and the jump probabilities $p^{Y}_a(n,k)$, $a\in\{u,d\}$,  by using  \eqref{ku}-\eqref{kd} and \eqref{pik};
		\item[(\textsc{T3})] for the process $R$, set the binomial tree $r^n_j$, $0\leq j\leq N$, by using \eqref{state-space-X}, then compute the jump nodes $j_b (n,j)$ and the jump probabilities $p^{R}_b(n,j)$, $b\in\{u,d\}$,   by using \eqref{ju}-\eqref{jd} and \eqref{pij};
		\item[(\textsc{T4})] for the $2$-dimensional process $(Y,R)$, merge the binomial trees in the bivariate tree $(y^n_k,r^n_j)$, $0\leq k,j\leq n\leq N$, by using \eqref{state-space-Yr}, then compute the jump-nodes $(k_a (n,k),j_b (n,j))$ and the transitions probabilities $p_{ab}(n,k,j)$, $(a,b)\in\{d,u\}$, by using \eqref{treescheme}.
	\end{itemize}
	The bivariate tree for $(Y,R)$ is now settled. Our hybrid tree/finite-difference algorithm can be resumed as follows:
	
	\begin{enumerate}
		\item[(\textsc{FD1})] set a mesh grid $x_i$ for the solution of all the PIDE's;
		\item[(\textsc{FD2})] for each node $(y^N_k,
		r^N_j)$, $0\leq k,j\leq N$, compute the option prices at maturity for each $x_i$, $i\in\mathcal{X}_M$, by using the payoff function;
		\item[(\textsc{FD3})] for $n=N-1,\ldots 0$: for each $(y^n_k,r^n_j)$, $0\leq k,j\leq n$, compute the option prices for each $x_i\in \mathcal{X}_M$, by solving the linear system \eqref{u-interp}.
	\end{enumerate}
	
	Notice that, at each time step $n$,  we need only the one-step PIDE solution in the time interval $[nh, (n+1)h]$. Moreover, both the (constant) PIDE coefficients and the Cauchy final condition change according to the position of the volatility and the interest rate components on the bivariate tree at time step $n$.
	
	\begin{remark} \label{costo}
		We observe that in order to compute the option price by the hybrid tree/finite-difference procedure, in step \textsc{(FD3)} we need to solve many times the tridiagonal system \eqref{u-interp}. This is typically solved by the LU-decomposition method in $O(M)$ operations (recall that the total number of the grid values $x_i\in\mathcal{X}_M$ is $2M+1$).
		However, due to the approximation of the integral term \eqref{num_int}, at each time step $n<N$ we have to compute the sum
		\begin{equation}\label{int_sum}
		\sum \tilde u^{n+1}_{i+l} \nu(\xi_l),
		\end{equation}
		which is the most computationally expensive step of this part of the algorithm: when applied directly, it requires $O(M^2)$ operations. Following the Premia software implementation \cite{pr}, in our numerical tests we use the Fast Fourier Transform  to compute the term \eqref{int_sum} and the computational costs of this step reduce to $O(M\log M)$.
	\end{remark}
	
	We conclude by briefly recalling the main steps of the hybrid Monte Carlo method:
	\begin{enumerate}
		\item[(\textsc{MC1})]
		let the chain $(\hat Y^h_n,\hat R^h_n)$ evolve for $n=1,\ldots,N$, following the probability structure in (\textsc{T4});
		\item[(\textsc{MC2})]
		generate $\Delta_1,\ldots,\Delta_N$ i.i.d. standard normal r.v.'s independent of the noise driving the chain $(\hat Y^h,\hat R^h)$;
		\item[(\textsc{MC3})] generate $K_h^1,\ldots,K_h^N$ i.i.d. positive Poisson r.v.'s of parameter $\lambda h$,  independent of both the chain $(\hat Y^h,\hat R^h)$ and the Gaussian r.v.'s $\Delta_1,\ldots,\Delta_N$, and for every $n=1,\ldots,N$, if $K_h^n>0$ simulate the corresponding amplitudes $\log(1+J_1^n),\ldots ,\log(1+J^n_{K_h^n})$;
		\item[(\textsc{MC4})] starting from $\hat X_0^h=X_0$, compute the approximate values $\hat X^h_n$, $1\leq n \leq N$, by using \eqref{hmc};
		\item[(\textsc{MC5})] following the desired Monte Carlo
		method (European or Longstaff-Schwartz algorithm \cite{ls} in the case of American options), repeat the above simulation scheme and compute the option price.
	\end{enumerate}

	\begin{remark}\label{standard-bates}
		In Section \ref{sect-numerics} we develop numerical experiments in order to study the behavior of our hybrid methods. Our tests involve also the standard Bates model, that is without any randomness in the interest rate. Recall that in the standard Bates model the dynamic reduces to
		\begin{equation}\label{stand-bates}
		\begin{array}{l}
		\displaystyle\frac{dS_t}{S_{t^-}}= (r-\delta)dt+\sqrt{Y_t}\, dZ^S_t+d H_t,
		\smallskip\\
		dY_t= \kappa_Y(\theta_Y-Y_t)dt+\sigma_Y\sqrt{Y_t}\,dZ^Y_t,
		\end{array}
		\end{equation}
		with $S_0>0$, $Y_0>0$ and $r\geq 0$ constant parameters. We assume a correlation between the two Brownian noises:
		$$
		d\<Z^S,Z^Y\>_t=\rho dt,\quad |\rho|<1.
		$$
		Finally, $H_t$ is the compound Poisson process already introduced in Section \ref{sect-model}, see \eqref{H}. We can apply our hybrid approach to this case as well: it just suffices  to follow the computational steps listed above except for the construction of the binomial tree for the process $R$. Consequently, we do not need the bivariate tree for $(Y,R)$, specifically we omit steps \textsc{(T3)}-\textsc{(T4)} and we replace step \textsc{(MC1)} with
		\begin{enumerate}
			\item[\textsc{(MC1')}]
			let the chain $\hat Y^h_n$ evolve for $n=1,\ldots,N$, following the probability structure in \textsc{(T2)}.
		\end{enumerate}
		And of course, in all computations we set equal to 0 the parameters involved in the dynamics for $r$, except for the starting value $r_0$. In particular, we have $\sigma_r=0$ and $\varphi_t=r_0$ for every $t$.
	\end{remark}

	\subsection{Numerical results}\label{sect-numerics}
	
	We develop several numerical results in order to assess the
	efficiency and the robustness of the hybrid tree/finite-difference
	method and the hybrid Monte Carlo method in the case of
	plain vanilla options. The Monte Carlo results derive from our hybrid simulations and, for American options, the use of the Monte Carlo algorithm by Longstaff and Schwartz in \cite{ls}.
	
	We first provide results for the standard Bates model (see Remark \ref{standard-bates}) and secondly, for the case in which the interest rate process is assumed to be stochastic, see \eqref{BHHmodel}.
	
	Following Chiarella \emph{et al.} \cite{ckmz}, in our numerical tests we assume that the jumps for the log-returns are normal, that is,
	\begin{equation}\label{nuC}
	\log(1+J_1)\sim N\Big(\gamma-\frac 12\eta^2,\eta^2\Big),
	\end{equation}
	$N$ denoting the Gaussian law (we also notice that the results in \cite{ckmz} correspond to the choice $\gamma=0$).
	In Section \ref{sect-num-bates}, we first compare our results with the ones provided in Chiarella \emph{et al.}
	\cite{ckmz}. Then in Section \ref{andersen} we study options with large maturities and when the Feller condition is not fulfilled. Finally, Section \ref{sect-num-bateshw} is devoted to  test experiments for European and American options in the Bates model with stochastic interest rate. The codes have been written by using the C++ language and the computations have all been performed in double precision on a PC 2,9 GHz Intel Core I5 with 8 Gb of RAM.
	
	\subsubsection{The standard Bates model}\label{sect-num-bates}
	
	We refer here to the standard Bates model as in \eqref{stand-bates}.
	In the European and American option contracts we are dealing with, we
	consider the following set of parameters, already used in the numerical results
	provided in Chiarella \emph{et al.} \cite{ckmz}:
	\begin{itemize}
		\item  initial price $S_0=80,
		90, 100, 110, 120$, strike price $K=100$,
		maturity $T=0.5$;
		\item (constant) interest rate $r=0.03$, dividend rate $\delta=0.05$;
		\item initial
		volatility $Y_0=0.04$, long-mean $\theta_Y=0.04$, speed of
		mean-reversion   $\kappa_Y=2$, vol-vol $\sigma_Y=0.4$,
		correlation $\rho=-0.5,0.5$;
		\item intensity $\lambda=5$, jump parameters $\gamma=0$ and $\eta=0.1$ (recall \eqref{nuC}).
	\end{itemize}
	It is known that the case $\rho>0$ may lead to moment explosion, see. e.g. \cite{ap}. Nevetheless, we report here results for this case as well, for the sake of  comparisons with the study in Chiarella \textit{et al.} \cite{ckmz}.
	
	In order to numerically solve the PIDE  using the finite difference
	scheme, we first localize the variables and the integral term to
	bounded domains. We use for this purpose the estimates for the
	localization domain and the truncation of large jumps given by Yoltchkova and Tankov
	\cite{vota08}. For example, for the previous model parameters the PIDE
	problem is solved in the finite interval
	$[\ln S_0-1.59, \ln S_0+1.93]$.
	
	The numerical study of the hybrid tree/finite-difference method
	\textbf{HTFD} is split into two cases:
	\begin{itemize}
		\item[-]
		\textbf{HTFDa}: time steps $N_t=50$ and varying mesh grid $\Delta
		x=0.01$, $0.005$, $0.0025$, $0.00125$;
		\item[-]
		\textbf{HTFDb}: time steps $N_t=100$ and varying mesh grid $\Delta
		x=0.01$, $0.005$, $0.0025$, $0.00125$.
	\end{itemize}
	
	Concerning the Monte Carlo method, we compare the results by using the hybrid simulation scheme in Section \ref{sect-MC}, that we call \textbf{HMC}. We compare our hybrid simulation scheme with the accurate third-order Alfonsi \cite{A} discretization scheme for the CIR stochastic volatility process and by using an exact scheme for the
	interest rate. In addition, we simulate the jump component in the standard way. The resulting Monte Carlo scheme is here called \textbf{AMC}.
	In both Monte Carlo methods, we consider  varying number of Monte Carlo
	iterations $N_{\mathrm{MC}}$ and two cases for the number of time discretization steps
	iterations:
	\begin{itemize}
		\item[-]
		\textbf{HMCa} and \textbf{AMCa}: $N_t=50$ and $N_{\mathrm{MC}}=10000, 50000, 100000, 200000$;
		\item[-]
		\textbf{HMCb} and \textbf{AMCb}: $N_t=100$ and $N_{\mathrm{MC}}=10000, 50000, 100000, 200000$.
	\end{itemize}
	All Monte Carlo results include the associated $95\% $ confidence
	interval.
	
	Table \ref{tab1} reports European call option
	prices. Comparisons are
	given with a  benchmark value obtained using the Carr-Madan pricing
	formula \textbf{CF} in \cite{CM} that applies Fast Fourier Transform
	methods (see the Premia software implementation \cite{pr}).
	
	In Table \ref{tab2} we provide results for
	American call option prices. In this case we compare with the values obtained by using
	the method of lines in \cite{ckm}, called \textbf{MOL}, with mesh
	parameters $200$ time-steps,
	$250$ volatility lines, $2995$ asset grid points, and the \textbf{PSOR} method with mesh parameters $1000, 3000, 6000$
	that Chiarella \emph{et al.} \cite{ckmz} used as the true solution.
	Moreover, we consider the Longstaff-Schwartz \cite{ls} Monte Carlo
	algorithm both for \textbf{AMC} and \textbf{HMC}. In particular
	\begin{itemize}
		\item[-]
		\textbf{HMCLSa} and \textbf{AMCLSa}:
		$10$ exercise dates, $N_t=50$ and $N_{\mathrm{MC}}=10000, 50000, 100000, $ $200000$;
		\item[-]
		\textbf{HMCLSb} and \textbf{AMCLSb}: $20$ exercise dates, $N_t=100$ and $N_{\mathrm{MC}}=10000, 50000, 100000,$ $ 200000$.
	\end{itemize}
	
	Tables \ref{tab31} and \ref{tab32} refer to the computational time cost (in seconds) of
	the various algorithms for $\rho=-0.5$ in the European and
	American case respectively.

	In order to make some heuristic considerations about the speed of convergence of  our approach \textbf{HTFD},  we consider the convergence ratio proposed in \cite{dfl}, defined as
	\begin{equation}\label{ratio}
	\mathrm{ratio}=\frac{P_{\frac{N}{2}}-P_{\frac{N}{4}}}{P_{N}-P_{\frac{N}{2}}},
	\end{equation}
	where $P_{N}$ denotes here the approximated price obtained with
	$N=N_t$ number of  time steps. Recall that $P_{N}=O(N^{-\alpha})$ means that
	$\mathrm{ratio}=2^{\alpha}$.
	Table \ref{tab4ratio} suggests that the convergence ratio for
	\textbf{HTDFb} is approximatively linear. The analysis of the convergence in Chapter 4 will confirm this heuristic deduction.
	
	We notice that the above argument does not formally allow to state the speed of convergence of a method knowing its ratio. We will come back on this topic in the next chapter of this thesis. However,  we anticipate here that our theoretical analysis of the convergence confirms the first order in time rate of convergence of the procedure.

	The numerical results in Table \ref{tab1}-\ref{tab32}
	show that \textbf{HTFD} is accurate, reliable and efficient for pricing European and American
	options in the Bates model.
	Moreover, our hybrid Monte Carlo algorithm \textbf{HMC} appears to be competitive with \textbf{AMC}, that is the one from the accurate simulations by Alfonsi \cite{A}:
	the numerical results are similar in term of precision
	and variance but \textbf{HMC} is definitely  better from the computational times point of view. Additionally, because of its simplicity,
	\textbf{HMC} represents a real and interesting alternative to
	\textbf{AMC}.
	
	As a further evidence of the accuracy of our hybrid methods,
	in Figure \ref{Fig1} and \ref{Fig2} we study the shapes of implied volatility smiles across
	moneyness $\frac{K}{S_0}$ and maturities $T$ using \textbf{HTFDa} with $N_t=50$ and
	$\Delta y=0.005$, \textbf{HMCa} with $N_t=50$ and
	$N_{\mathrm{MC}}=50000$ and we compare the
	graphs with the results from the benchmark values \textbf{CF}.
	
	\begin{table}[ht]
		\centering
		\subtable[]{
			\tiny
			{\begin{tabular} {@{}c|lrr|c|rrrrrc@{}} %{@{}clcccrccccc@{}}
					\toprule $\rho=-0.5$ & $\Delta x$
					&HTFDa  &HTFDb & CF &$N_{\mathrm{MC}} $&HMCa  &HMCb &AMCa  &AMCb\\
					\hline
					& 0.01& 1.1302& 1.1302 & &10000&1.08$\pm$0.09 &1.11$\pm$0.09
					&1.00$\pm$0.09&1.08$\pm$0.09 &\\
					&0.005& 1.1293& 1.1294 & &50000 &1.12$\pm$0.04 &1.17$\pm$0.04 &1.07$\pm$0.04&1.10$\pm$0.04 &\\
					$S_0 =80$&0.0025& 1.1291& 1.1292& 1.1293 &100000 &1.14$\pm$0.03 &1.14$\pm$0.03
					&1.13$\pm$0.03&1.13$\pm$0.03 &\\
					&0.00125& 1.1291& 1.1292 & &200000 &1.13$\pm$0.02 &1.14$\pm$0.02  &1.11$\pm$0.02&1.12$\pm$0.02&\\
					\hline
					& 0.01& 3.3331& 3.3312 & & 10000&3.27$\pm$0.17 &3.27$\pm$0.17  &3.19$\pm$0.16&3.22$\pm$0.16&\\\
					&0.005& 3.3315& 3.3301 & &50000 &3.32$\pm$0.08 &3.40$\pm$0.08  &3.24$\pm$0.07&3.26$\pm$0.0&\\
					$S_0 =90$&0.0025& 3.3311& 3.3298& 3.3284  &100000 &3.34$\pm$0.05 &3.34$\pm$0.05  &3.32$\pm$0.05&3.33$\pm$0.05&\\
					&0.00125& 3.3310& 3.3297 & &200000 &3.32$\pm$0.04 &3.35$\pm$0.04  &3.28$\pm$0.04&3.31$\pm$0.04&\\
					\hline
					& 0.01& 7.5245& 7.5239 & & 10000&7.46$\pm$0.25 &7.46$\pm$0.25  &7.37$\pm$0.24&7.36$\pm$0.25&\\
					&0.005& 7.5236& 7.5224 & &50000 &7.53$\pm$0.11 &7.62$\pm$0.11  &7.40$\pm$0.11&7.43$\pm$0.11&\\
					$S_0 =100$&0.0025& 7.5231& 7.5221& 7.5210 &100000 &7.54$\pm$0.08 &7.52$\pm$0.08  &7.53$\pm$0.08&7.52$\pm$0.08&\\
					&0.00125& 7.5230& 7.5220 & &200000 &7.50$\pm$0.06 &7.54$\pm$0.06  &7.46$\pm$0.06&7.50$\pm$0.06&\\
					\hline
					& 0.01& 13.6943& 13.6940 & & 10000&13.69$\pm$0.34 &13.69$\pm$0.34  &13.52$\pm$0.33&13.48$\pm$0.33&\\
					&0.005& 13.6923& 13.6924 & &50000 &13.71$\pm$0.15 &13.81$\pm$0.15
					&13.55$\pm$0.15&13.58$\pm$0.15 &\\
					$S_0 =110$&0.0025& 13.6918& 13.6921& 13.6923 &100000 &13.72$\pm$0.11
					&13.69$\pm$0.11  &13.67$\pm$0.11&13.70$\pm$0.11 &\\
					&0.00125& 13.6917& 13.6920 & &200000 &13.64$\pm$0.08 &13.71$\pm$0.08
					&13.63$\pm$0.07&13.69$\pm$0.08 &\\
					\hline
					& 0.01& 21.3173& 21.3185 & & 10000&21.40$\pm$0.41
					&21.40$\pm$0.41  &21.08$\pm$0.40&21.03$\pm$0.41 &\\
					&0.005& 21.3156& 21.3168 & &50000 &21.35$\pm$0.18 &21.46$\pm$0.19
					&21.17$\pm$0.18&21.21$\pm$0.18 &\\
					$S_0 =120$&0.0025& 21.3152& 21.3164& 21.3174 &100000 &21.36$\pm$0.13
					&21.32$\pm$0.13  &21.29$\pm$0.13&21.33$\pm$0.13 &\\
					&0.00125& 21.3152& 21.3163 & &200000 &21.25$\pm$0.09 &21.33$\pm$0.09
					&21.26$\pm$0.09&21.33$\pm$0.09 &\\
					\hline
				\end{tabular}}
			}
			\quad\quad
			\subtable[]{
				\tiny
				{\begin{tabular} {@{}c|lrr|c|rrrrrc@{}} \toprule $\rho=0.5$ & $\Delta x$ &HTFDa
						&HTFDb & CF  &$N_{\mathrm{MC}} $&HMCa  &HMCb &AMCa  &AMCb\\
						\hline
						& 0.01 & 1.4732 & 1.4751 & & 10000&1.42$\pm$0.12
						&1.40$\pm$0.12  &1.37$\pm$0.12&1.35$\pm$0.12 &\\
						&0.005 &1.4724 &1.4744 &  &50000 &1.49$\pm$0.06 &1.47$\pm$0.05
						&1.40$\pm$0.05&1.42$\pm$0.05 &\\
						$S_0 =80$ &0.0025 &1.4723 &1.4742 &1.4760 &100000 &1.48$\pm$0.04 &1.46$\pm$0.04  &1.46$\pm$0.04&1.49$\pm$0.04&\\
						&0.00125 &1.4722 &1.4741 & &200000 &1.47$\pm$0.03 &1.48$\pm$0.03  &1.48$\pm$0.03&1.48$\pm$0.03&\\
						\hline
						& 0.01& 3.6849& 3.6859 & & 10000&3.63$\pm$0.19 &3.63$\pm$0.19
						&3.48$\pm$0.19&3.49$\pm$0.19 &\\
						&0.005 &3.6836 &3.6849 & &50000 &3.70$\pm$0.09 &3.70$\pm$0.09  &3.57$\pm$0.09&3.60$\pm$0.09&\\
						$S_0 =90$&0.0025 &3.6832 &3.6847 &3.6862 &100000 &3.67$\pm$0.06 &3.67$\pm$0.06 &3.66$\pm$0.06&3.71$\pm$0.06 &\\
						&0.00125 &3.6832 &3.6847 & &200000 &3.66$\pm$0.04 &3.70$\pm$0.04  &3.69$\pm$0.04&3.68$\pm$0.04&\\
						\hline
						& 0.01& 7.6247 &7.6245 & & 10000&7.58$\pm$0.28
						&7.58$\pm$0.28  &7.35$\pm$0.28&7.36$\pm$0.27 &\\
						&0.005 &7.6238 &7.6232 & &50000 &7.66$\pm$0.13 &7.65$\pm$0.13
						&7.47$\pm$0.12&7.52$\pm$0.12 &\\
						$S_0 =100$&0.0025& 7.6234& 7.6229 &7.6223 &100000 &7.61$\pm$0.09 &7.59$\pm$0.09  &7.58$\pm$0.09&7.66$\pm$0.09&\\
						&0.00125& 7.6233 &7.6228 & &200000 &7.58$\pm$0.06 &7.64$\pm$0.06  &7.62$\pm$0.06&7.61$\pm$0.06 &\\
						\hline
						& 0.01& 13.4863& 13.4835 & & 10000&13.48$\pm$0.36
						&13.48$\pm$0.36  &13.21$\pm$0.36&13.19$\pm$0.36 &\\
						&0.005& 13.4842& 13.4818 & &50000 &13.55$\pm$0.17 &13.49$\pm$0.16
						&13.27$\pm$0.16&13.35$\pm$0.16 &\\
						$S_0 =110$&0.0025& 13.4837 &13.4814 &13.4791 &100000 &13.47$\pm$0.12
						&13.41$\pm$0.12  &13.44$\pm$0.12&13.54$\pm$0.12 &\\
						&0.00125 &13.4836 &13.4813 & &200000 &13.42$\pm$0.08 &13.49$\pm$0.08
						&13.47$\pm$0.08&13.48$\pm$0.08 &\\
						\hline
						& 0.01 &20.9678 &20.9661 & & 10000&21.04$\pm$0.44
						&21.04$\pm$0.44  &20.67$\pm$0.44&20.64$\pm$0.43 &\\
						&0.005 &20.9659 &20.9642 & &50000 &21.05$\pm$0.20 &20.98$\pm$0.20
						&20.71$\pm$0.20&20.81$\pm$0.20 &\\
						$S_0 =120$&0.0025 &20.9655 &20.9636 &20.9616 &100000 &20.96$\pm$0.14
						&20.87$\pm$0.14  &20.92$\pm$0.14&21.04$\pm$0.14 &\\
						&0.00125& 20.9654 &20.9635 & &200000 &20.88$\pm$0.10 &20.96$\pm$0.10
						&20.97$\pm$0.10&20.98$\pm$0.10 &\\
						\hline
					\end{tabular}}
				}
				\caption{\em \small{Standard Bates model. Prices of European call options. Test parameters: $K=100$, $T=0.5$,
						$r=0.03$, $\delta=0.05$, $Y_0=0.04$, $\theta_Y=0.04$, $\kappa_Y=2$, $\sigma_Y=0.4$,
						$\lambda=5$, $\gamma=0$, $\eta=0.1$, $\rho=-0.5, 0.5$.}}
				\label{tab1}
			\end{table}
			
			\begin{table}[ht] \centering
				\subtable[]{
					\tiny
					{\begin{tabular} {@{}c|lrr|cc|rrrrr@{}} \toprule $\rho=-0.5$ & $\Delta x$ &HTFDa
							&HTFDb & PSOR &MOL &$N_{\mathrm{MC}} $&HMCLSa  &HMCLSb &AMCLSa  &AMCLSb\\
							\hline
							& 0.01& 1.1365& 1.1365 & & & 10000&1.03$\pm$0.08&1.14$\pm$0.09&1.06$\pm$0.09&1.03$\pm$0.09\\
							& 0.005& 1.1356& 1.1358 & & &50000&1.19$\pm$0.04&1.14$\pm$0.04&1.18$\pm$0.04&1.12$\pm$0.04\\
							$S_0=80$& 0.0025& 1.1354& 1.1356& 1.1359& 1.1363 &100000&1.15$\pm$0.03&1.13$\pm$0.03&1.13$\pm$0.03&1.13$\pm$0.03 \\
							& 0.00125& 1.1353& 1.1355 & & &200000&1.14$\pm$0.02&1.14$\pm$0.02&1.14$\pm$0.02&1.14$\pm$0.02\\
							\hline
							& 0.01& 3.3579& 3.3563 & & &10000&3.39$\pm$0.15&3.44$\pm$0.16&3.38$\pm$0.15&3.48$\pm$0.16 \\
							& 0.005& 3.3564& 3.3551 & & &50000&3.46$\pm$0.07&3.33$\pm$0.07&3.46$\pm$0.07&3.32$\pm$0.07 \\
							$S_0=90$& 0.0025& 3.3560& 3.3548& 3.3532& 3.3530 &100000&3.35$\pm$0.05&3.35$\pm$0.05&3.33$\pm$0.05&3.36$\pm$0.05 \\
							& 0.00125& 3.3559& 3.3547 & & &200000&3.35$\pm$0.03&3.33$\pm$0.03&3.35$\pm$0.03&3.34$\pm$0.03 \\
							\hline
							& 0.01& 7.6010& 7.6006 & & &10000&7.68$\pm$0.23&7.88$\pm$0.24&7.63$\pm$0.23&7.80$\pm$0.24 \\
							& 0.005& 7.6001& 7.5992 &  & &50000&7.75$\pm$0.11&7.59$\pm$0.10&7.76$\pm$0.10&7.53$\pm$0.10\\
							$S_0=100$& 0.0025& 7.5997& 7.5989& 7.5970& 7.5959 &100000&7.56$\pm$0.07&7.61$\pm$0.07&7.56$\pm$0.07&7.61$\pm$0.07 \\
							& 0.00125& 7.5996& 7.5989 &  & &200000&7.58$\pm$0.05&7.55$\pm$0.05&7.58$\pm$0.05&7.57$\pm$0.05\\
							\hline
							& 0.01& 13.8853& 13.8854 & & &10000&13.90$\pm$0.29&14.28$\pm$0.30&13.84$\pm$0.29&14.10$\pm$0.29 \\
							& 0.005& 13.8836& 13.8842 & & &50000&14.05$\pm$0.13&13.89$\pm$0.12&14.07$\pm$0.13&13.86$\pm$0.12 \\
							$S_0=110$& 0.0025& 13.8832& 13.8839& 13.8830& 13.8827 &100000&13.80$\pm$0.09&13.91$\pm$0.09&13.84$\pm$0.09&13.89$\pm$0.09\\
							& 0.00125& 13.8831& 13.8838 & & &200000&13.86$\pm$0.06&13.84$\pm$0.06&13.87$\pm$0.06&13.83$\pm$0.06 \\
							\hline
							& 0.01& 21.7180& 21.7199 & & &10000&21.83$\pm$0.34&22.07$\pm$0.33&21.71$\pm$0.30&22.04$\pm$0.34\\
							& 0.005& 21.7168& 21.7187 & & &50000&21.91$\pm$0.15&21.76$\pm$0.13&21.90$\pm$0.15&21.72$\pm$0.13 \\
							$S_0=120$& 0.0025& 21.7166& 21.7184 & 21.7186& 21.7191 &100000&21.59$\pm$0.10&21.78$\pm$0.10&21.64$\pm$0.10&21.72$\pm$0.10\\
							& 0.00125& 21.7165& 21.7183 & & &200000&21.68$\pm$0.07&21.65$\pm$0.07&21.68$\pm$0.07&21.67$\pm$0.07\\
							\hline
						\end{tabular}}
					}\quad\quad
					\subtable[]{
						\tiny
						{\begin{tabular} {@{}c|lrr|cc|rrrrr@{}} \toprule $\rho=0.5$ & $\Delta x$ &HTFDa
								&HTFDb & PSOR &MOL &$N_{\mathrm{MC}} $&HMCLSa  &HMCLSb &AMCLSa  &AMCLSb\\
								\hline
								& 0.01& 1.4817& 1.4837 & & &10000&1.32$\pm$0.11&1.03$\pm$0.09&1.51$\pm$0.13&0.66$\pm$0.08\\
								& 0.005& 1.4809& 1.4830 & & &50000&1.51$\pm$0.05&1.31$\pm$0.05&1.54$\pm$0.05&1.47$\pm$0.05\\
								$S_0=80$& 0.0025& 1.4807& 1.4828& 1.4843& 1.4848 &100000&1.50$\pm$0.04&1.50$\pm$0.04&1.51$\pm$0.04&1.48$\pm$0.04\\
								& 0.00125& 1.4807& 1.4828 & & &200000&1.50$\pm$0.03&1.49$\pm$0.02&1.49$\pm$0.03&1.47$\pm$0.02\\
								\hline
								& 0.01& 3.7134& 3.7148 & & &10000&3.83$\pm$0.19&3.79$\pm$0.17&3.89$\pm$0.19&3.95$\pm$0.19\\
								& 0.005& 3.7121& 3.7139 & & &50000&3.81$\pm$0.08&3.70$\pm$0.08&3.84$\pm$0.08&3.69$\pm$0.08\\
								$S_0=90$& 0.0025& 3.7118& 3.7137& 3.7145& 3.7146 &100000&3.69$\pm$0.06&3.75$\pm$0.06&3.72$\pm$0.06&3.70$\pm$0.06 \\
								& 0.00125& 3.7118& 3.7137 & & &200000&3.70$\pm$0.04&3.71$\pm$0.04&3.72$\pm$0.04&3.70$\pm$0.04\\
								\hline
								& 0.01& 7.7044& 7.7051 & & &10000&7.74$\pm$0.26&7.85$\pm$0.25&7.96$\pm$0.26&7.99$\pm$0.26\\
								& 0.005& 7.7036& 7.7039 & & &50000&7.85$\pm$0.12&7.68$\pm$0.11&7.87$\pm$0.12&7.68$\pm$0.11 \\
								$S_0=100$& 0.0025& 7.7033& 7.7036& 7.7027& 7.7018 &100000&7.66$\pm$0.08&7.75$\pm$0.08&7.65$\pm$0.08&7.73$\pm$0.08\\
								& 0.00125& 7.7032& 7.7036 & & &200000&7.69$\pm$0.06&7.67$\pm$0.05&7.68$\pm$0.06&7.69$\pm$0.05\\
								\hline
								& 0.01& 13.6770& 13.6756 & & &10000&13.57$\pm$0.32&13.98$\pm$0.31&13.88$\pm$0.32&14.12$\pm$0.33\\
								& 0.005& 13.6752& 13.6742 &  & &50000&13.83$\pm$0.14&13.67$\pm$0.13&13.89$\pm$0.14&13.64$\pm$0.13\\
								$S_0=110$& 0.0025& 13.6747& 13.6739& 13.6722& 13.6715 &100000&13.56$\pm$0.09&13.74$\pm$0.10&13.58$\pm$0.10&13.71$\pm$0.10\\
								& 0.00125& 13.6747& 13.6738 & & &200000&13.65$\pm$0.07&13.65$\pm$0.07&13.64$\pm$0.07&13.64$\pm$0.07\\
								\hline
								& 0.01& 21.3668& 21.3671& & &10000&21.45$\pm$0.32&21.60$\pm$0.35&21.39$\pm$0.33&21.84$\pm$0.34 \\
								& 0.005& 21.3655& 21.3658 & & &50000&21.54$\pm$0.15&21.40$\pm$0.14&21.61$\pm$0.16&21.40$\pm$0.13\\
								$S_0=120$& 0.0025& 21.3653& 21.3655& 21.3653& 21.3657 &100000&21.26$\pm$0.10&21.43$\pm$0.10&21.27$\pm$0.10&21.38$\pm$0.10\\
								& 0.00125& 21.3652& 21.3653 & & &200000&21.31$\pm$0.07&21.33$\pm$0.07&21.31$\pm$0.07&21.31$\pm$0.07 \\
								\hline
							\end{tabular}}
						}
						\caption{\em \small{Standard Bates model. Prices of American call options. Test parameters: $K=100$, $T=0.5$,
								$r=0.03$, $\delta=0.05$, $Y_0=0.04$, $\theta_Y=0.04$, $\kappa_Y=2$, $\sigma_Y=0.4$,
								$\lambda=5$, $\gamma=0$, $\eta=0.1$, $\rho=-0.5,0.5$.}}
						\label{tab2}
					\end{table}

					\begin{table} [ht] \centering
						\tiny
						{\begin{tabular} {@{}clcc|rcccc|cc@{}} \toprule & $\Delta x$ &HTFDa &HTDFb
								&$N_{\mathrm{MC}}$&HMCa &HMCb &AMCa &AMCb &CF\\
								\hline
								& 0.01  &0.09 &0.34 &10000 &0.007 &0.16 &0.16 &0.30 & &\\
								& 0.005 &0.18 &0.72  &50000  &0.36 &0.72  &0.79 &1.51 & &\\
								& 0.0025  &0.46 &1.62  &100000 &0.71 &1.44  &1.57 &3.12  &0.001 &\\
								& 0.00125  &0.84 &3.53 &200000 &1.45 &2.95 &3.14 &6.17 & &\\
							\end{tabular}}
							\caption{\em \small{Standard Bates model. Computational times (in seconds) for European call
									options in Table \ref{tab1} for $S_0=100$, $\rho=-0.5.$}}
							\label{tab31}
						\end{table}
						
						\begin{table} [ht] \centering
							\tiny
							{\begin{tabular} {@{}clcc|rccccc@{}} \toprule & $\Delta x$ &HTFDa &HTDFb
									&$N_{\mathrm{MC}}$&HMCLSa &HMCLSb &AMCLSa &AMCLSb\\
									\hline
									& 0.01  &0.10 &0.37 &10000 &0.09 &0.23 &0.20 &0.45 &\\
									& 0.005 &0.19 &0.77  &50000  &0.47 &1.11  &1.01 &2.25 &\\
									& 0.0025  &0.48 &1.77  &100000 &1.07 &2.25  &2.01 &4.57  &\\
									& 0.00125  &0.95 &3.61 &200000 &1.94 &4.55 &4.05 &8.98 &\\
								\end{tabular}}
								\caption{\em \small{Standard Bates model. Computational times (in seconds) for American call
										options in Table \ref{tab2} for $S_0=100$, $\rho=-0.5.$}}
								\label{tab32}
							\end{table}
							
							\begin{table} [ht] \centering
								\tiny
								{\begin{tabular} {@{}cccccc@{}}\toprule $N$ &$S_0 =80$
										& $S_0 =90$ & $S_0 =100$ & $S_0 =110$ & $S_0 =120$ \\
										\hline
										200 &1.919250  &1.961063 &1.894156  &2.299666 &2.109026\\
										400 &2.172836 &2.209762 &2.556021  &1.673541 &1.996332\\
										800 &1.544849 & 1.851932 &1.463712  &2.935697 &2.106880\\
										\hline
									\end{tabular}}
									\caption{\em \small{Standard Bates model. HTFDb-ratio \eqref{ratio} for the price of
											American call options as the starting point $S_0$ varies with
											fixed space step $\Delta x=0.0025$. Test parameters: $T=0.5$,
											$r=0.03$, $\delta=0.05$, $Y_0=0.04$, $\theta=0.04$, $\kappa=2$, $\sigma=0.4$,
											$\lambda=5$, $\gamma=0$, $\eta=0.1$, $\rho=-0.5$.}}
									\label{tab4ratio}
								\end{table}

								\clearpage
								\begin{figure}[ht]
									\begin{center}
										\includegraphics[scale=0.35]{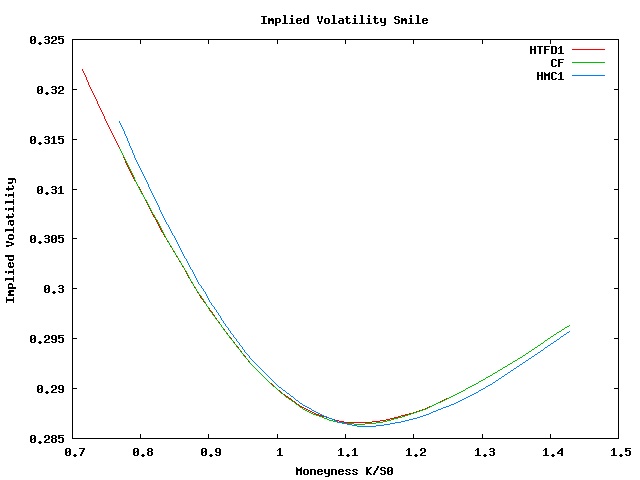}
										\caption{\em \small{ Standard Bates model. Moneyness vs implied volatility for European call options. Test parameters: $T=0.5$,
												$r=0.03$, $\delta=0.05$, $Y_0=0.04$, $\theta_Y=0.04$, $\kappa_Y=2$, $\sigma_Y=0.4$,
												$\lambda=5$, $\gamma=0$, $\eta=0.1$, $\rho=-0.5$.}}
										\label{Fig1}
										%											\end{center}
										%										\end{figure}
										%													\begin{figure}[hb]
										%											\begin{center}
										
										\bigskip
										
										\includegraphics[scale=0.35]{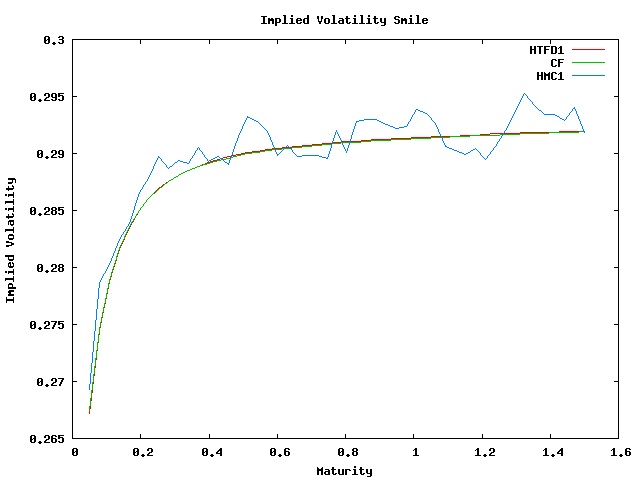}
										\caption{\em \small{Standard Bates model. Maturity vs implied volatility for European call options. Test parameters:
												$S_0=100$, $K=100$,
												$r=0.03$, $\delta=0.05$, $Y_0=0.04$, $\theta_Y=0.04$, $\kappa_Y=2$, $\sigma_Y=0.4$,
												$\lambda=5$, $\gamma=0$, $\eta=0.1$, $\rho=-0.5$.}}
										\label{Fig2}
									\end{center}
								\end{figure}
								\clearpage
								\subsubsection{Options with large maturity in the
									standard Bates model}\label{andersen}
								In order to verify the robustness of the proposed algorithms we consider
								experiments when the Feller condition $2\kappa_Y \theta_Y\geq\sigma_Y^2$  is not fulfilled for the CIR volatility process. We additionally stress our tests by considering  large maturities. For this purpose we consider the parameters from Chiarella \emph{et al.} \cite{ckmz} already used in Section \ref{sect-num-bates} with
								$\rho=-0.5$, except for the maturity and the vol-vol, which are modified as follows: $T=5$ and $\sigma_Y=0.7$ respectively.
								
								Table \ref{tab4} reports European call option
								prices, which are compared with the true values (\textbf{CF}). In Table \ref{tab5} we provide results for
								American call option prices.
								The settings for the experiments \textbf{HTFDa-b}, \textbf{HMCa-b} and
								\textbf{AMCa-b} are the same as described at the beginning of Section \ref{sect-num-bates}.
								The settings for the experiments in the American case
								\textbf{HMCLSa-b} and \textbf{AMCLSa-b} are changed
								\begin{itemize}
									\item[-]
									\textbf{HMCLSa} and \textbf{AMCLSa}: $20$ exercise dates, $N_t=100$ and $N_{\mathrm{MC}}=10000, 50000, 100000, $ $200000$;
									\item[-]
									\textbf{HMCLSb} and \textbf{AMCLSb}: $40$ exercise dates, $N_t=200$ and $N_{\mathrm{MC}}=10000, 50000, 100000$, $200000$.
								\end{itemize}
								In the American case the benchmark values \textbf{B-AMC} are obtained by the Longstaff-Schwartz \cite{ls} Monte Carlo algorithm with $300$
								exercise dates, combined with the accurate
								third-order Alfonsi method with $3000$ discretization time steps and 1
								million iterations.

								The numerical results suggest that large maturities bring to a slight loss of accuracy for \textbf{HTFD} and \textbf{HMC}, even if both methods
								provide a satisfactory approximation of the true option prices, being in turn mostly compatible with the results from the Alfonsi Monte Carlo method. It is
								worth noticing that for long maturity $T=5$ we have developed
								experiments with the same number of steps both in time ($N_t$) and
								space step ($\Delta x$) as for $T=0.5$. So, the numerical experiments are not slower, and it is clear that one could achieve a better accuracy for larger values of $N_t$.
								\begin{table}[ht]
									\centering
									%\subtable[]{
									\tiny
									{\begin{tabular} {@{}c|lrr|c|rrrrrc@{}} \toprule $\rho=-0.5$ & $\Delta x$
											&HTFDa  &HTFDb & CF &$N_{\mathrm{MC}} $&HMCa  &HMCb  &AMCa  &AMCb\\
											\hline
											& 0.01&9.0085 &8.9457  & & 10000 &9.21$\pm$0.55 &9.09$\pm$0.55
											&8.69$\pm$0.53&8.56$\pm$0.51 &\\
											&0.0050 &9.0032 &8.9405  &  &50000 &9.13$\pm$0.25 &8.92$\pm$0.24
											&8.81$\pm$0.24&9.04$\pm$0.24 &\\
											$S_0=80$&0.0025 &9.0020 &8.9392  &8.9262  &100000 &9.01$\pm$0.17
											&8.81$\pm$0.17  &8.92$\pm$0.17&8.88$\pm$0.17  &\\\
											&0.00125 &9.0016 &8.9389  &  &200000 &8.99$\pm$0.12 &8.92$\pm$0.12  &8.95$\pm$0.12&8.90$\pm$0.12 &\\
											\hline
											& 0.01&12.7405 &12.6520   & & 10000 &12.95$\pm$0.67
											&12.95$\pm$0.67  &12.29$\pm$0.65&12.15$\pm$0.6 &\\
											&0.0050 &12.7342 &12.6458  & &50000 &12.87$\pm$0.30 &12.64$\pm$0.29
											&12.49$\pm$0.29&12.76$\pm$0.3 &\\
											$S_0=90 $&0.0025 &12.7327 &12.6442  &12.6257  &100000 &12.72$\pm$0.21
											&12.50$\pm$0.21  &12.63$\pm$0.21&12.58$\pm$0.21 &\\
											&0.00125 &12.7323 &12.6438  &  & 200000 &12.71$\pm$0.15 &12.61$\pm$0.15
											&12.66$\pm$0.15&12.61$\pm$0.15 &\\
											\hline
											& 0.01&17.0324 &16.9176  & &10000  &17.24$\pm$0.80
											&17.24$\pm$0.80  &16.43$\pm$0.77&16.29$\pm$0.75 &\\
											&0.0050 &17.0254 &16.9106  &  &50000 &17.18$\pm$0.36 &16.91$\pm$0.35
											&16.73$\pm$0.35&17.03$\pm$0.35 &\\
											$S_0=100 $&0.0025 &17.0237 &16.9089 &16.8855  &100000  &17.00$\pm$0.25
											&16.74$\pm$0.25  &16.91$\pm$0.25&16.84$\pm$0.25 &\\
											&0.00125 &17.0232 &16.9084  & &200000 &16.99$\pm$0.18 &16.86$\pm$0.18
											&16.94$\pm$0.18&16.88$\pm$0.18 &\\
											\hline
											& 0.01&21.8149 &21.6741   & &10000 &22.04$\pm$0.93
											&22.04$\pm$0.93  &21.06$\pm$0.93&20.91$\pm$0.88 &\\
											&0.0050 &21.8067 &21.6659  & &50000 &21.96$\pm$0.42 &21.67$\pm$0.41 &21.43$\pm$0.41&21.82$\pm$0.41 &\\
											$S_0=110 $&0.0025 &21.8047 &21.6639  &21.6364 &100000 &21.76$\pm$0.29
											&21.47$\pm$0.29  &21.69$\pm$0.29&21.59$\pm$0.29 &\\
											&0.00125 &21.8042 &21.6634  & &200000 &21.76$\pm$0.21 &21.59$\pm$0.20
											&21.70$\pm$0.20&21.63$\pm$0.20 &\\
											\hline
											& 0.01&27.0196 &26.8539  & &10000 &27.26$\pm$1.05
											&27.26$\pm$1.05  &26.12$\pm$1.03&25.94$\pm$1.01 &\\
											&0.0050 &27.0108 &26.8452  & &50000 &27.17$\pm$0.47 &26.86$\pm$0.46
											&26.56$\pm$0.46&27.02$\pm$0.47 &\\
											$S_0=120 $&0.0025 &27.0086 &26.8430  &26.8121 &100000 &26.94$\pm$0.33
											&26.63$\pm$0.33  &26.89$\pm$0.33&26.78$\pm$0.33 &\\
											&0.00125 &27.0081 &26.8425  &  &200000 &26.95$\pm$0.23 &26.75$\pm$0.23
											&26.89$\pm$0.23&26.81$\pm$0.23 &\\
											\hline
										\end{tabular}
									}
									\caption{\em \small{Standard Bates model. Prices of European call options. Test parameters: $K=100$, $T=5$,
											$r=0.03$, $\delta=0.05$, $Y_0=0.04$, $\theta_Y=0.04$, $\kappa_Y=2$, $\sigma_Y=0.7$,
											$\lambda=5$, $\gamma=0$, $\eta=0.1$, $\rho=-0.5$. Case $2\kappa_Y\theta_Y<\sigma^2_Y$.}}
									\label{tab4}
								\end{table}
								
								\begin{table}[ht]\centering
									%\subtable[]{
									\tiny
									{\begin{tabular} {@{}c|lrr|c|rrrrr@{}} \toprule $\rho=-0.5$ & $\Delta y$ &HTFDa
											&HTFDb &B-AMC &$N_{\mathrm{MC}} $&HMCLSa  &HMCLSb &AMCLSa  &AMCLSb\\
											\hline
											& 0.01&9.8335 &9.7978  & &10000&10.15$\pm$0.46&10.20$\pm$0.46&10.47$\pm$0.47&9.80$\pm$0.42\\
											&0.0050 &9.8283 &9.7927  & &50000&9.93$\pm$0.20&9.86$\pm$0.20&9.89$\pm$0.19&9.78$\pm$0.19 \\
											$S_0=80$&0.0025 &9.8271 &9.7914  &9.7907$\pm$ 0.04 &100000&9.76$\pm$0.14&9.69$\pm$0.13&9.74$\pm$0.14&9.76$\pm$0.13\\
											&0.00125 &9.8267 &9.7911  & &200000&9.79$\pm$0.10&9.70$\pm$0.09&9.73$\pm$0.10&9.72$\pm$0.09 \\
											\hline
											& 0.01&14.0801 &14.0318  & &10000&14.58$\pm$0.56&14.46$\pm$0.55&14.94$\pm$0.58&14.08$\pm$0.51 \\
											&0.0050 &14.0741 &14.0258  & &50000&14.13$\pm$0.24&14.14$\pm$0.24&14.19$\pm$0.23&14.12$\pm$0.23\\
											$S_0=90 $&0.0025 &14.0726 &14.0244  &14.0030$\pm$ 0.05 &100000&13.98$\pm$0.16&13.87$\pm$0.16&13.94$\pm$0.16&13.89$\pm$0.16 \\
											&0.00125 &14.0722 &14.0240  & &200000&13.93$\pm$0.12&13.91$\pm$0.11&13.94$\pm$0.12&13.96$\pm$0.11 \\
											\hline
											& 0.01&19.0658 &19.0075  & &10000&19.59$\pm$0.66&19.44$\pm$0.63&19.88$\pm$0.66&19.13$\pm$0.59 \\
											&0.0050 &19.0594 &19.0011  & &50000&19.10$\pm$0.27&19.06$\pm$0.27&19.26$\pm$0.26&19.01$\pm$0.26\\
											$S_0=100 $&0.0025 &19.0578 &18.9995  &18.9632$\pm$ 0.05 &100000&18.92$\pm$0.19&18.88$\pm$0.18&18.85$\pm$0.19&18.90$\pm$0.18 \\
											&0.00125 &19.0574 &18.9991  & &200000&18.80$\pm$0.13&18.84$\pm$0.13&18.85$\pm$0.13&18.92$\pm$0.13 \\
											\hline
											& 0.01&24.7434 &24.6788  & &10000&25.02$\pm$0.74&24.84$\pm$0.72&25.32$\pm$0.72&24.78$\pm$0.67\\
											&0.0050 &24.7364 &24.6719  & &50000&24.79$\pm$0.30&24.57$\pm$0.29&24.94$\pm$0.29&24.72$\pm$0.29\\
											$S_0=110 $&0.0025 &24.7347 &24.6701  &24.6289$\pm$ 0.06 &100000&24.53$\pm$0.21&24.47$\pm$0.20&24.50$\pm$0.21&24.51$\pm$0.20\\
											&0.00125 &24.7343 &24.6697  & &200000&24.42$\pm$0.14&24.45$\pm$0.14&24.50$\pm$0.15&24.53$\pm$0.14\\
											\hline
											& 0.01&31.0646 &30.9983  & &10000&30.88$\pm$0.74&31.15$\pm$0.75&31.18$\pm$0.74&31.04$\pm$0.71\\
											&0.0050 &31.0577 &30.9914  & &50000&31.10$\pm$0.32&30.94$\pm$0.31&31.32$\pm$0.32&30.98$\pm$0.32\\
											$S_0=120 $&0.0025 &31.0559 &30.9896  &30.9052$\pm$0.07 &100000&30.89$\pm$0.23&30.72$\pm$0.22&30.70$\pm$0.22&30.72$\pm$0.22\\
											&0.00125 &31.0555 &30.9892  & &200000&30.72$\pm$0.16&30.73$\pm$0.16&30.77$\pm$0.16&30.89$\pm$0.15\\
											\hline
										\end{tabular}
									}
									\caption{\em \small{Standard Bates model. Prices of American call options. Test parameters: $K=100$, $T=5$,
											$r=0.03$, $\delta=0.05$, $Y_0=0.04$, $\theta_Y=0.04$, $\kappa_Y=2$, $\sigma_Y=0.7$,
											$\lambda=5$, $\gamma=0$, $\delta=0.1$, $\rho=-0.5$. Case $2\kappa_Y\theta_Y<\sigma^2_Y$.}}
									\label{tab5}
								\end{table}
								
								\clearpage

								\subsubsection{Bates model with stochastic interest rate}\label{sect-num-bateshw}
								We consider now the case of Bates model associated with the Vasiceck model for the stochastic interest rate.
								For the Bates model we consider the parameters from Chiarella \emph{et al.}
								\cite{ckmz} already used in Section \ref{sect-num-bateshw}. Moreover, for the interest rate parameter we fix the following parameters:
								\begin{itemize}
									\item
									initial interest rate
									$r_0=0.03$, speed of mean-reversion   $\kappa_r=1$, interest rate
									volatility  $\sigma_r=0.2$;
									\item
									time-varying long-term mean $\theta_r(t)$ fitting the theoretical bond prices to the yield curve observed on the market, here set as
									$P_r(0,T)=e^{-0.03T}$.
								\end{itemize}
								We study the cases
								$$
								\rho_1=\rho_{SY}=-0.5\quad\mbox{and}\quad\rho_2=\rho_{Sr}=-0.5, 0.5.
								$$
								No correlation is assumed to exist between $r$ and $Y$.
								We consider the mesh grid $\Delta
								y=0.02$, $0.01$, $0.005$, $0.0025$,
								the case $\Delta y=0.00125$ being removed because it requires huge computational times.
								The numerical results are labeled \textbf{HTFDa-b},
								\textbf{HMCa-b}, \textbf{AMCa-b}, \textbf{HMCLSa-b}, \textbf{AMCLSa-b}, their settings being given at the beginning of Section \ref{sect-num-bates}.
								
								When the interest rate is assumed to be stochastic, no references are available in the literature. Therefore, we propose benchmark values obtained by using a Monte Carlo method in which the CIR paths are simulated  through the accurate
								third-order Alfonsi \cite{A} discretization scheme and the interest rate paths are generated by an exact scheme. For these benchmark values, called \textbf{B-AMC}, the number of Monte Carlo iterations and of the discretization time steps are set as $N_{\mathrm{MC}}=10^6$  and
								$N_t=300$ respectively. In the American case, \textbf{B-AMC} is evaluated through the Longstaff-Schwartz \cite{ls} algorithm  with $20$
								exercise dates. All Monte Carlo results report the $95\% $ confidence
								intervals.
								
								European and American call option prices are given in tables \ref{tab6} and \ref{tab7} respectively. Tables \ref{tab81} and \ref{tab82} refer to the computational time cost (in seconds) of the different algorithms in the European Call case and American Call case respectively.
								The numerical results confirm the good numerical behavior of \textbf{HTFD}  and
								\textbf{HMC} in the Bates-Hull-White model as well.
								
								\begin{table}[ht]
									\centering
									\subtable[]{
										\tiny
										{\begin{tabular} {@{}c|lrr|c|rrrrrr@{}} %{@{}clcccrccccc@{}}
												\toprule $\rho_{Sr}=-0.5$ & $\Delta x$
												&HTFDa  &HTFDb & B-AMC &$N_{\mathrm{MC}} $&HMCa  &HMCb  &AMCa  &AMCb\\
												\hline
												& 0.02&1.0169 &1.0079  & &  10000 &1.00$\pm$0.09
												&0.96$\pm$0.09  &1.00$\pm$0.09&1.06$\pm$0.10 & \\
												&0.01 &1.0201 &1.0188  & &50000 &1.02$\pm$0.04 &0.97$\pm$0.04
												&0.98$\pm$0.04&1.01$\pm$0.04 &\\
												$S_0=80$&0.0050 &1.0199 &1.0194  &1.0153$\pm0.01$ &100000 &1.00$\pm$0.03
												&1.00$\pm$0.03  &1.01$\pm$0.03&1.03$\pm$0.03 &\\
												&0.0025 &1.0197 &1.0193  & &200000 &1.01$\pm$0.02 &1.01$\pm$0.02
												&1.02$\pm$0.02&1.00$\pm$0.02 &\\
												\hline
												& 0.01&3.1172 &3.1032  & & 10000&3.05$\pm$0.16 &3.05$\pm$0.16
												&3.07$\pm$0.16&3.14$\pm$0.17 &\\
												&0.01 &3.1186 &3.1137  & &50000 &3.10$\pm$0.07 &3.03$\pm$0.07
												&3.02$\pm$0.07&3.09$\pm$0.07 &\\
												$S_0=90 $&0.0050 &3.1174 &3.1135  &3.1008$\pm0.02$
												&100000 &3.07$\pm$0.05 &3.08$\pm$0.05  &3.09$\pm$0.05&3.14$\pm$0.05 &\\
												&0.0025 &3.1174 &3.1136  & &200000 &3.09$\pm$0.04 &3.10$\pm$0.04  &3.11$\pm$0.04&3.08$\pm$0.04 &\\
												\hline
												& 0.02&7.2528 &7.2472  & & 10000&7.17$\pm$0.24
												&7.17$\pm$0.24  &7.20$\pm$0.24&7.24$\pm$0.25 &\\
												&0.01 &7.2528 &7.2479  & &50000 &7.21$\pm$0.11 &7.18$\pm$0.11
												&7.12$\pm$0.11&7.21$\pm$0.11 &\\
												$S_0=100 $&0.0050 &7.2528 &7.2480  &7.2315$\pm 0.02$ &100000 &7.18$\pm$0.08
												&7.24$\pm$0.08  &7.20$\pm$0.08&7.27$\pm$0.08 &\\
												&0.0025 &7.2528 &7.2480  & &200000 &7.22$\pm$0.05 &7.25$\pm$0.05 &7.24$\pm$0.05&7.20$\pm$0.05 &\\
												\hline
												& 0.02&13.4553 &13.4565  & & 10000&13.30$\pm$0.32
												&13.30$\pm$0.32  &13.41$\pm$0.33&13.39$\pm$0.33 &\\
												&0.01 &13.4465 &13.4440  & &50000 &13.37$\pm$0.15 &13.40$\pm$0.15
												&13.27$\pm$0.15&13.38$\pm$0.15 &\\
												$S_0=110 $&0.0050 &13.4435 &13.4407  &13.4256$\pm 0.03$ &100000 &13.35$\pm$0.10
												&13.46$\pm$0.10  &13.38$\pm$0.10&13.48$\pm$0.10 &\\
												&0.0025 &13.4432 &13.4404  & &200000 &13.40$\pm$0.07 &13.47$\pm$0.07
												&13.43$\pm$0.07&13.39$\pm$0.07 &\\
												\hline
												& 0.02&21.1320 &21.1356  & & 10000&20.89$\pm$0.40
												&20.89$\pm$0.40  &21.08$\pm$0.40&20.99$\pm$0.41  &\\
												&0.01 &21.1243 &21.1239  & &50000 &21.03$\pm$0.18 &21.09$\pm$0.18
												&20.92$\pm$0.18&21.03$\pm$0.18 &\\
												$S_0=120 $&0.0050 &21.1222 &21.1214  &21.1070$\pm 0.04$ &100000 &21.01$\pm$0.13
												&21.17$\pm$0.13  &21.04$\pm$0.13&21.17$\pm$0.13 &\\
												&0.0025 &21.1215 &21.1207  & &200000 &21.06$\pm$0.09 &21.16$\pm$0.09 &21.12$\pm$0.09&21.06$\pm$0.09 &\\
												\hline
											\end{tabular}}
										}
										\quad\quad
										\subtable[]{
											\tiny
											{\begin{tabular} {@{}c|lrr|c|rrrrrr@{}} \toprule $\rho_{Sr}=0.5$ & $\Delta x$ &HTFDa
													&HTFDb & B-AMC  &$N_{\mathrm{MC}} $&HMCa  &HMCb  &AMCa  &AMCb\\
													\hline
													& 0.02&1.3459 &1.3379  & & 10000&1.29$\pm$0.11 &1.28$\pm$0.11
													&1.32$\pm$0.10&1.41$\pm$0.11 &\\
													&0.01 &1.3482 &1.3471  & &50000 &1.34$\pm$0.05 &1.30$\pm$0.05
													&1.32$\pm$0.05&1.35$\pm$0.05  &\\
													$S_0=80$&0.0050 &1.3479 &1.3475  &1.3446$\pm$0.01 &100000 &1.32$\pm$0.03
													&1.31$\pm$0.03  &1.34$\pm$0.03&1.34$\pm$0.03 &\\
													&0.0025 &1.3477 &1.3473  & &200000 &1.33$\pm$0.02 &1.34$\pm$0.02 &1.35$\pm$0.02&1.32$\pm$0.02 &\\
													\hline
													& 0.01&3.7320 &3.7233  & & 10000&3.62$\pm$0.18 &3.62$\pm$0.18
													&3.64$\pm$0.18&3.76$\pm$0.19  &\\
													&0.01 &3.7323 &3.7304  & &50000 &3.69$\pm$0.08 &3.65$\pm$0.08 &3.64$\pm$0.18&3.76$\pm$0.19 &\\
													$S_0=90 $&0.0050 &3.7311 &3.7298  &3.7263$\pm 0.02$ &100000 &3.66$\pm$0.06
													&3.68$\pm$0.06  &3.71$\pm$0.06&3.73$\pm$0.06 &\\
													&0.0025 &3.7311 &3.7299  & &200000 &3.69$\pm$0.04 &3.72$\pm$0.04 &3.73$\pm$0.04&3.68$\pm$0.04  &\\
													\hline
													& 0.02&8.0100 &8.0073  & & 10000&7.83$\pm$0.26
													&7.83$\pm$0.26  &7.82$\pm$0.26&8.00$\pm$0.27 &\\
													&0.01 &8.0112 &8.0102  & &50000 &7.92$\pm$0.12 &7.93$\pm$0.12
													&7.93$\pm$0.12&7.97$\pm$0.12 &\\
													$S_0=100 $&0.0050 &8.0114 &8.0107  &8.0069$\pm 0.03$ &100000 &7.91$\pm$0.08 &7.97$\pm$0.08 &7.99$\pm$0.08&8.02$\pm$0.08 &\\
													&0.0025 &8.0114 &8.0107  & &200000 &7.95$\pm$0.06 &8.02$\pm$0.06
													&8.00$\pm$0.06&7.95$\pm$0.06 &\\
													\hline
													& 0.02&14.1482 &14.1505  & & 10000&13.89$\pm$0.35
													&13.89$\pm$0.35  &13.88$\pm$0.35&14.07$\pm$0.36 &\\
													&0.01 &14.1413 &14.1414  & &50000 &14.01$\pm$0.16 &14.05$\pm$0.16 &14.03$\pm$0.16&14.09$\pm$0.16 &\\
													$S_0=110 $&0.0050 &14.1388 &14.1388  &14.1323$\pm 0.03$ &100000 &14.01$\pm$0.11
													&14.10$\pm$0.11  &14.12$\pm$0.11&14.14$\pm$0.11 &\\
													&0.0025 &14.1386 &14.1386  & &200000 &14.06$\pm$0.08 &14.17$\pm$0.08
													&14.13$\pm$0.08&14.07$\pm$0.08 &\\
													\hline
													& 0.02&21.6737 &21.6772  & & 10000&21.37$\pm$0.42
													&21.37$\pm$0.42  &21.35$\pm$0.42&21.51$\pm$0.43 &\\
													&0.01 &21.6670 &21.6674  & &50000 &21.50$\pm$0.19 &21.55$\pm$0.19 &21.52$\pm$0.19&21.60$\pm$0.19 &\\
													$S_0=120 $&0.0050 &21.6651 &21.6653  &21.6501$\pm 0.04$ &100000 &21.52$\pm$0.13
													&21.63$\pm$0.13  &21.64$\pm$0.13&21.68$\pm$0.14 &\\
													&0.0025 &21.6645 &21.6646  & &200000 &21.57$\pm$0.10 &21.71$\pm$0.10  &21.65$\pm$0.10&21.58$\pm$0.09&\\
													\hline
												\end{tabular}}
											}
											\caption{\em \small{Bates-Hull-White model. Prices of European call options. Test parameters: $K=100$, $T=0.5$,
													$\delta=0.05$, , $r_0=0.03$, $\kappa_r=1$, $\sigma_r=0.2$, $Y_0=0.04$, $\theta_Y=0.04$, $\kappa_Y=2$, $\sigma_Y=0.4$,
													$\lambda=5$, $\gamma=0$, $\eta=0.1$,
													$\rho_{SY}=-0.5$,$\rho_{Sr}=-0.5, 0.5$.}}
											\label{tab6}
										\end{table}
										
										\begin{table}[ht] \centering
											\subtable[]{
												\tiny
												{\begin{tabular} {@{}c|lrr|c|rrrrr@{}} \toprule $\rho_{Sr}=-0.5$ & $\Delta x$ &HTFDa
														&HTFDb & B-AMC &$N_{\mathrm{MC}} $&HMCLSa  &HMCLSb &AMCLSa  &AMCLSb\\
														\hline
														& 0.02&1.0561 &1.0470  & &10000&0.76$\pm$0.07&0.56$\pm$0.06&0.95$\pm$0.08&0.82$\pm$0.08 \\
														&0.01 &1.0598 &1.0588  & &50000&1.08$\pm$0.04&0.91$\pm$0.04&1.01$\pm$0.04&0.96$\pm$0.04 \\
														$S_0=80$&0.0050 &1.0597 &1.0596  &1.0544$\pm0.01$ &100000&1.07$\pm$0.03&1.03$\pm$0.03&1.07$\pm$0.03&1.04$\pm$0.03
														\\
														&0.0025 &1.0596 &1.0595  & &200000&1.05$\pm$0.02&1.04$\pm$0.02&1.07$\pm$0.02&1.05$\pm$0.02\\
														\hline
														& 0.01&3.2511 &3.2364  & &10000&3.28$\pm$0.15&3.39$\pm$0.16&3.35$\pm$0.16&3.07$\pm$0.15 \\
														&0.01 &3.2537 &3.2493  & &50000&3.33$\pm$0.07&3.21$\pm$0.07&3.25$\pm$0.07&3.30$\pm$0.07\\
														$S_0=90 $&0.0050 &3.2528 &3.2494  &3.2273$\pm 0.01$ &100000&3.23$\pm$0.05&3.24$\pm$0.05&3.27$\pm$0.05&3.25$\pm$0.05 \\
														&0.0025 &3.2528 &3.2495  & &200000&3.22$\pm$0.03&3.23$\pm$0.03&3.25$\pm$0.03&3.24$\pm$0.03 \\
														\hline
														& 0.02&7.6012 &7.5952  & &10000&7.64$\pm$0.22&7.99$\pm$0.23&7.80$\pm$0.23&7.68$\pm$0.22\\
														&0.01 &7.6020 &7.5976  & &50000&7.72$\pm$0.10&7.58$\pm$0.09&7.61$\pm$0.10&7.65$\pm$0.10\\
														$S_0=100 $&0.0050 &7.6022 &7.5980  &7.5589$\pm 0.02$ &100000&7.54$\pm$0.07&7.62$\pm$0.07&7.61$\pm$0.07&7.54$\pm$0.07 \\
														&0.0025 &7.6022 &7.5980  & &200000&7.54$\pm$0.05&7.54$\pm$0.05&7.56$\pm$0.05&7.60$\pm$0.05\\
														\hline
														& 0.02&14.1510 &14.1524  & &10000&14.22$\pm$0.28&14.61$\pm$0.29&14.35$\pm$0.29&14.07$\pm$0.28\\
														&0.01 &14.1443 &14.1425  & &50000&14.25$\pm$0.13&14.11$\pm$0.12&14.16$\pm$0.12&14.17$\pm$0.13\\
														$S_0=110 $&0.0050 &14.1420 &14.1401  &14.0909$\pm 0.03$ &100000&14.03$\pm$0.09&14.18$\pm$0.09&14.10$\pm$0.09&14.06$\pm$0.09
														\\
														&0.0025 &14.1419 &14.1399  & &200000&14.05$\pm$0.06&14.04$\pm$0.06&14.07$\pm$0.06&14.13$\pm$0.06\\
														\hline
														& 0.02&22.2466 &22.2505  & &10000&22.38$\pm$0.32&22.84$\pm$0.33&22.46$\pm$0.32&22.15$\pm$0.32 \\
														&0.01 &22.2412 &22.2419  & &50000&22.35$\pm$0.15&22.27$\pm$0.14&22.24$\pm$0.14&22.28$\pm$0.14\\
														$S_0=120 $&0.0050 &22.2398 &22.2402  &22.1736$\pm  0.03$ &100000&22.12$\pm$0.10&22.27$\pm$0.10&22.19$\pm$0.10&22.17$\pm$0.10
														\\
														&0.0025 &22.2394 &22.2397  & &100000&22.12$\pm$0.10&22.27$\pm$0.10&22.19$\pm$0.10&22.17$\pm$0.10\\
														\hline
													\end{tabular}}
												}\quad\quad
												\subtable[]{
													\tiny
													{\begin{tabular} {@{}c|lrr|c|rrrrr@{}} \toprule $\rho_{Sr}=0.5$ & $\Delta x$ &HTFDa
															&HTFDb & B-AMC &$N_{\mathrm{MC}} $&HMCLSa  &HMCLSb &AMCLSa  &AMCLSb\\
															\hline
															& 0.02&1.3551 &1.3470  & &10000&1.18$\pm$0.09&1.29$\pm$0.10&1.12$\pm$0.09&0.80$\pm$0.08\\
															&0.01 &1.3576 &1.3566  & &50000&1.35$\pm$0.05&1.17$\pm$0.04&1.33$\pm$0.05&1.25$\pm$0.05\\
															$S_0=80$&0.0050 &1.3573 &1.3570  &1.3559$\pm0.01$ &100000&1.33$\pm$0.03&1.30$\pm$0.03&1.33$\pm$0.03&1.27$\pm$0.03 \\
															&0.0025 &1.3571 &1.3569  & &200000&1.35$\pm$0.02&1.31$\pm$0.02&1.38$\pm$0.02&1.34$\pm$0.02\\
															\hline
															& 0.01&3.7696 &3.7606  & &10000&3.72$\pm$0.17&3.78$\pm$0.17&3.82$\pm$0.18&3.72$\pm$0.17 \\
															&0.01 &3.7705 &3.7688  & &50000&3.86$\pm$0.08&3.71$\pm$0.08&3.80$\pm$0.08&3.81$\pm$0.08\\
															$S_0=90 $&0.0050 &3.7694 &3.7685  &3.7633$\pm0.02$ &100000&3.75$\pm$0.06&3.74$\pm$0.05&3.76$\pm$0.05&3.74$\pm$0.05\\
															&0.0025 &3.7694 &3.7686  & &200000&3.75$\pm$0.04&3.74$\pm$0.04&3.80$\pm$0.04&3.79$\pm$0.04\\
															\hline
															& 0.02&8.1285 &8.1249  & &10000&8.12$\pm$0.24&8.52$\pm$0.26&8.25$\pm$0.26&8.15$\pm$0.25\\
															&0.01 &8.1308 &8.1301  &  &50000&8.25$\pm$0.11&8.08$\pm$0.11&8.15$\pm$0.11&8.18$\pm$0.11\\
															$S_0=100 $&0.0050 &8.1311 &8.1308  &8.1122$\pm0.03$ &100000&8.07$\pm$0.08&8.16$\pm$0.08&8.11$\pm$0.08&8.10$\pm$0.08\\
															&0.0025 &8.1312 &8.1309  &  &200000&8.08$\pm$0.06&8.07$\pm$0.06&8.14$\pm$0.06&8.16$\pm$0.06\\
															\hline
															& 0.02&14.4455 &14.4468  & &10000&14.48$\pm$0.32&14.84$\pm$0.33&14.43$\pm$0.32&14.51$\pm$0.32\\
															&0.01 &14.4409 &14.4414  & &50000&14.60$\pm$0.15&14.40$\pm$0.14&14.45$\pm$0.14&14.47$\pm$0.14\\
															$S_0=110 $&0.0050 &14.4389 &14.4395  &14.3884$\pm0.03$ &100000&14.34$\pm$0.10&14.47$\pm$0.10&14.39$\pm$0.10&14.38$\pm$0.10 \\
															&0.0025 &14.4388 &14.4394  & &200000&14.35$\pm$0.07&14.37$\pm$0.07&14.38$\pm$0.07&14.48$\pm$0.07 \\
															\hline
															& 0.02&22.2859 &22.2893  & &10000&22.23$\pm$0.36&22.87$\pm$0.39&22.45$\pm$0.36&22.29$\pm$0.35\\
															&0.01 &22.2815 &22.2827  & &50000&22.50$\pm$0.17&22.29$\pm$0.16&22.27$\pm$0.16&22.28$\pm$0.16\\
															$S_0=120 $&0.0050 &22.2802 &22.2813  &22.2039$\pm0.04$ &100000&22.17$\pm$0.12&22.31$\pm$0.12&22.24$\pm$0.12&22.22$\pm$0.12\\
															&0.0025 &22.2798 &22.2808 &  &200000&22.17$\pm$0.08&22.17$\pm$0.08&22.17$\pm$0.08&22.32$\pm$0.08\\
															\hline
														\end{tabular}}
													}
													\caption{\em \small{Bates-Hull-White model. Prices of American call options. Test parameters: $K=100$, $T=0.5$,
															$\delta=0.05$, $r_0=0.03$, $\kappa_r=1$, $\sigma_r=0.2$, $Y_0=0.04$, $\theta_Y=0.04$, $\kappa_Y=2$, $\sigma_Y=0.4$,
															$\lambda=5$, $\gamma=0$, $\eta=0.1$, $\rho_{SY}=-0.5$,$\rho_{Sr}=-0.5, 0.5$.}}
													\label{tab7}
												\end{table}
												
												\clearpage
												
												\begin{table} [ht] \centering
													\tiny
													{\begin{tabular} {@{}llrr|rccccc@{}} \toprule & $\Delta x$ &HTFDa &HTDFb &$N_{\mathrm{MC}}$&HMCa &HMCb &AMCa &AMCb\\
															\hline
															& 0.02  &2.77 &22.95 &10000 &0.13  &0.25  &0.36  &0.48  & \\
															& 0.01  &6.15  &48.17 &50000  &0.66  &1.35  &1.11  &2.48  &\\
															& 0.005  &12.12 &99.19  &100000 &1.37  &2.56  &1.82  &4.99  &\\
															& 0.0025  &27.61 &204.88 &200000  &2.56  &5.08  &3.70  &9.96  &\\
														\end{tabular}}
														\caption{\em \small{Bates-Hull-White model. Computational times (in seconds) for European call
																options in Table \ref{tab6} for $S_0=100$, $\rho_{Sr}=-0.5.$}}
														\label{tab81}
													\end{table}
													
													\begin{table} [ht] \centering
														\tiny
														{\begin{tabular} {@{}llrr|rccccc@{}} \toprule & $\Delta x$ &HTFDa &HTDFb &$N_{\mathrm{MC}}$&HMCLSa &HMCLSb &AMCLSa &AMCLSb\\
																\hline
																& 0.02  &2.77 &23.10 &10000 &0.28  &0.43  &0.40  &0.62  & \\
																& 0.01  &6.39  &48.65 &50000  &0.80  &1.79  &1.30  &2.72  &\\
																& 0.005  &12.50 &99.85  &100000 &1.91  &3.89  &3.02  &6.15  &\\
																& 0.0025  &27.92 &205.60 &200000  &4.03  &8.11  &5.20  &10.75  &\\
															\end{tabular}}
															\caption{\em \small{Bates-Hull-White model. Computational times (in seconds) for American call
																	options in Table \ref{tab7} for $S_0=100$, $\rho_{Sr}=-0.5.$}}
															\label{tab82}
														\end{table}
														
														\clearpage

													\chapter{Weak convergence rate of Markov chains and hybrid numerical schemes for jump-diffusion processes}\label{chapter-art4}
													
													\section{Introduction}\label{sect-introC4}
													This chapter is devoted to the study of the weak convergence rate of numerical schemes allowing one to handle specific jump-diffusion processes which include the  Heston  and  Bates models in the full parameters regime.
													%\textcolor{blue}{Since these dynamics involve the square root process for the volatility, a special numerical treatment has to be considered.   When dealing with European options, i.e. solutions to Partial (Integral) Differential Equation (hereafter P(I)DE) problems, numerical approaches involve tree methods \cite{ads, nv}, Monte Carlo procedures \cite{A-MC,A,AN, and, z}, finite-difference numerical schemes  \cite{ckmz,it, t} or quantization algorithms  \cite{PP}. 
													%	When American options are considered, that is, solutions to  specific optimal stopping problems or  P(I)DEs with obstacle, it is very useful to consider numerical methods which are able to easily handle dynamic programming principles, for example trees or finite-difference. }
													We generalize the hybrid tree- finite difference method  %developed and numerically studied in \cite{bcz,bcz-hhw,bctz} 
													described in Chapter \ref{chapter-art3} for the computation of European and American options in the stochastic volatility context
													%As we have seen in Chapter \ref{chapter-art3}, this numerical procedure combines a tree method for the volatility process with a different  numerical approach for the asset price process, for instance finite-difference. %Such a hybrid method has been developed and numerically studied in \cite{bcz,bcz-hhw,bctz} for the computation of European and American options in the stochastic volatility context.
													and we study the rate of convergence.
													% As a result, we can consider the Heston or the Bates model in the full parameter regime, differently from many other approaches. 
													Let us mention that, under these models, the literature is rich in numerical methods but, as far as we know, poor in results on the rate of convergence, with the exception of  the papers \cite{A,AN,bossy,z}, all of them either dealing with schemes written on Brownian increments or requiring restrictions on the Heston diffusion parameters. So, we first study the convergence rate of tree methods and then we tackle the hybrid procedure.
													
													Tree methods rely heavily on Markov chains. So, in the first part (Section \ref{sect-markovapprox}) we study the rate at which a sequence of Markov chains weakly converges to a diffusion process $(Y_t)_{t\in[0,T]}$ solution to
													$$
													dY_t=\mu_Y(Y_t)dt+\sigma_Y(Y_t)dW_t.
													$$
													In this framework, the weak convergence  is well known to be governed by the behaviour of the local moments up to order 3 or 4 (see e.g. %\cite{Bil, EK, SV}
													\cite{SV}). In order to get the speed of convergence, we need to stress such requests, making further but quite general assumptions on the behaviour of the moments, and in Theorem \ref{conv_T} we  prove a  first order weak convergence result. As an application, we give an  example from the financial framework: we theoretically study the convergence rate of the  tree approximation proposed in \cite{acz} for the  CIR process (and described in Section \ref{sect-tree}). 
													%Recall that the  CIR process  \cite{cir} is a square root process, that is,
													%	$$
													%	dY_t=\kappa(\theta-Y_t)dt+\sigma\sqrt{Y_t}dB_t,
													%	$$
													%	with $\kappa,\theta,\sigma>0$. Recall also that this process lives in $[0,+\infty)$ and under the Feller condition $2\kappa\theta\geq\sigma^2$ it never hits 0.
													Several trees are considered in the literature, see e.g. \cite{cgmz, HST,Tian}, but all of them  work poorly from the numerical point of view when the Feller condition fails. Our result  for the tree in \cite{acz} (Theorem \ref{conv_T-CIR}) works in any parameter regime.
													Recall that in equity markets, one often requires large values for the
													vol-vol $\sigma$  whereas in interest rates context, $\sigma$ is markedly lower (see e.g. the calibration results in \cite{dps} and in \cite{bm} p. 115,  respectively).  So, a result in the full parameter regime is actually essential.
													We stress that our  convergence Theorem \ref{conv_T}  is completely general and may in principle be applied to more general trees constructed through the multiple jumps approach by Nelson and  Ramaswamy \cite{nr} or also to  other cases, e.g. the recent tree method developed in \cite{ads}.
													
													In the second part (Section \ref{sect-hybrid}), we link to $(Y_t)_{t\in[0,T]}$ a jump-diffusion process $(X_t)_{t\in[0,T]}$ which evolves according to a stochastic differential equation whose coefficients only depend on the process $(Y_t)_{t\in[0,T]}$:
													$$
													dX_t=\mu_X(Y_t)dt+\sigma_X(Y_t)dB_t+\gamma_X(Y_t)dH_t,
													$$
													where $H$ is a compound Poisson process independent of the 2-dimensional Brownian motion $(B,W)$.
													So, the pair $(X_t,Y_t)_{t\in[0,T]}$ evolves following a Stochastic Differential Equation (hereafter SDE) with jumps. Given a function $f$, we consider the numerical computation of  $\E[f(X_T,Y_T)]$ or $\sup_{\tau\in\mathcal{T}_{0,T}}\E[f(X_\tau,Y_\tau)] $ through a method %generalization of the hybrid method introduced in \cite{bcz, bcz-hhw, bctz} 
													(Section  \ref{sect-hybrid2}), 	which works backwardly by
													approximating the process $Y$ with a Markov chain and by using a different numerical scheme for solving a (local) PIDE allowing us to work  in the direction of the process $X$.
													Then (Section \ref{sect-convergence}), in Theorem \ref{convergencebates} we give a general result on the rate of convergence of the hybrid approach.   We stress that the approximating algorithm is not directly written on a Markov approximation, so one cannot extend the convergence result provided in the first part of the chapter. 
													We then study  the stability and the consistency of the hybrid method,  but  in a sense that allows us to exploit the probabilistic properties of the Markov chain approximating the  process $Y$. 
													
													It is worth mentioning that the test functions on which we study the rate of convergence are smooth. %and growing polynomially.  
													In fact, there is a strict connection between such hybrid schemes and the use of a discrete noise in the approximation procedure. This means that we cannot use regularizing arguments \textit{\`a la Malliavin} in order to relax the smoothness requests, as it can be done when the approximation algorithm is based on the Brownian noise (see the seminal paper \cite{bt} or the recent \cite{AN} for the Heston model) or on a noise having at least a ``good piece of absolutely continuous part''  (Doeblin's condition, see \cite{br}).
													
													We then consider two possible finite-difference schemes (Section \ref{sect-finitedifference}) to handle the (local)  PIDE related to the component $X$:
													%. In particular, we compare two different finite-difference procedures: 
													an implicit in time/centered in space 
													scheme (Section \ref{sect-l2}) and an implicit in time/upwind in space %first order 
													scheme (Section \ref{sect-linf}). In both cases, the numerical treatment of the nonlocal term coming from the jumps involves implicit-explicit techniques, as well as numerical quadratures.  We apply the convergence Theorem \ref{convergencebates} and we obtain that  the hybrid algorithm has a rate of convergence of the first order in time and of a order in space according to the chosen numerical scheme.
													As an application, we give the weak convergence rate of the hybrid procedure   written on the Heston and on the  Bates model for pricing European options (Section \ref{sect-conv-eu}). Finally, in Section \ref{sect-conv-am} we give a theoretical result on the convergence rate  in the case of American options.
													%	 and, as a consequence, we give the weak convergence rate of the procedure in \cite{bctz} for the pricing  in the European case (Section \ref{sect-european}).  It is worth stressing that, in the financial framework, the backward approach applies straightforwardly to the case of American options. This is not only from a numerical point of view: we give here in fact a theoretical result on the convergence rate also in this case (Section \ref{sect-american}).
													%	
													\section{Notation}\label{sect-notation}
													In this section we establish the notation which will be used in this chapter. Let $d\in \N^*=\N\setminus \{0\}$. % and $\R_+=[0,+\infty)$. 
													
													\smallskip
													
													$\bullet$ For a multi-index $l=(l_1,\dots,l_d)\in \N^d$ we define $|l|=\sum_{j=1}^{d}l_j$ and for $y\in \R^d$,
													we define $\partial^l_y=\partial_{y_1}^{l_1}\cdots \partial_{y_d}^{l_d}$ and $y^l=y_1^{l_1}\cdots y_d^{l_d}$.
													Moreover,   we denote by $|y|$ the standard Euclidean norm in  $ \R^d$ and for any linear operator $A: \R^d\rightarrow \R^d$, we denote by $|A|=\sup_{|y|=1}|Ay|$ the induced norm. 
													
													\smallskip
													
													$\bullet$ 
													$L^p(\R^d,d\mathfrak m)$ denotes  the standard $L^p$-space w.r.t. the measure $\mathfrak m$ on $(\R^d,\mathcal{B}_d)$,  $\mathcal{B}_d$ denoting the Borel $\sigma$-algebra on $\R^d$, and we set  $|\cdot|_{L^p(\R^d, d\mathfrak m)}$  the associated norm. The Lebesgue measure is denoted through $dx$.
													
													\smallskip
													
													$\bullet$ Let  $\mathcal{D}\subseteq \R^d$ be a domain (possibly closed)
													and $q\in\N$.  $C^q(\mathcal{D})$ is the set of all functions on $\mathcal{D}$ which are $q$-times continuously differentiable.  We set 
													$C^q_\pol(\mathcal{D} )$ 
													the set of functions $g\in C^q(\mathcal{D})$ such that there exist $C,a>0$ for which 
													\begin{equation*}
													|\partial^{l}_yg(y)|\leq C(1+|y|^a),\qquad  y\in \mathcal{D}, \, |l|\leq q.
													\end{equation*}
													When the above property holds for $q = 0$, in general terms we say that $g$ grows
													polynomially and if $a = 1$, we speak about sublinear growth.																For  $[a,b]\subseteq \R^+$, we set $C^{ q}_{\pol, [a,b]}(\mathcal{D})$ the set of functions $v=v(t,y)$ such that $v\in C^{\lfloor q/2 \rfloor, q}([a,b)\times \mathcal{D} )$ and there exist $C,c>0$ for which 
																													$$
													\sup_{t\in [a,b)}|\partial^{k}_t\partial^{l}_yv(t,y)|\leq C(1+|y|^c),\qquad y\in \mathcal{D} , \, 2k+|l|\leq q.
													$$
													For brevity, we set  $C(\mathcal{D})=C^0(\mathcal{D})$, 
													$C_\pol(\mathcal{D} )=C^0_\pol(\mathcal{D} )$ and $C_{\pol, [a,b]}(\mathcal{D})=C^{0 }_{\pol, [a,b]}(\mathcal{D} )$.
													We also need another functional space, that we call $C^{p,q}_{\pol}(\R^m,\mathcal{D})$, $p\in[1,\infty]$, $q\in\N$, $m\in\N^*$: $g=g(x,y)\in C^{p,q}_{\pol}(\R^m,\mathcal{D})$ if $g\in C_{\pol}^q(\R^m\times \mathcal{D})$ and  there exist $C,c>0$ such that
													$$
													|\partial^{l'}_x\partial^{l}_yg(\cdot, y)|_{L^p(\R^m,dx)}\leq C(1+|y|^c), \quad  |l'|+|l|\leq q.
													$$
													Similarly as above, we set $C^{p,q}_{\pol,[a,b]}(\R^m,\mathcal{D})$ the set of the function $v\in C^q_{\pol,[a,b]}(\R^m\times \mathcal{D} )$ such that 
													$$
													\sup_{t\in [a,b)}|\partial^{k}_t\partial^{l'}_x\partial^{l}_yv(t,\cdot, y)|_{L^p(\R^m,dx)}\leq C(1+|y|^c), \quad  2k+|l'|+|l|\leq q.
													$$
													If $[a,b]=[0,T]$, to simplify the notation, we set $C^{ q}_{\pol, [0,T]}(\mathcal{D})=C^{ q}_{\pol, T}(\mathcal{D})$ and $C^{p, q}_{\pol, [0,T]}(\mathcal{D})=C^{ p,q}_{\pol, T}(\mathcal{D})$.
													
													%Moreover, for $q\in\N$ let $C^q_{\mathbf{b}}(\mathcal D)$  be the set of all functions which are 	$q$-times continuously differentiable with bounded derivatives and set $C^q_{\mathbf{b},[a,b]}(\mathcal D)$ the set of functions $v \in C_{\mathbf b}^{\lfloor q/2 \rfloor, q}([a,b)\times \mathcal{D} )$ such that there exist $C>0$ for which 
													%$$
													%\sup_{t\in [a,b)}|\partial^{k}_t\partial^{l}_yv(t,y)|\leq C,\qquad y\in \mathcal{D} , \, 2k+|l|\leq q.
													%$$. For $p\in[1,\infty]$, $q\in\N$, $m\in\N^*$ we denote by $C^{p,q}_{b}(\R^m,\mathcal{D})$ the set of the functions $g=g(x,y)$ such that  $g\in C_{\textbf{b}}^q(\R^m\times \mathcal{D})$ and  there exist $C>0$ such that
													%	$$
													%	|\partial^{l'}_x\partial^{l}_yg(\cdot, y)|_{L^p(\R^m,dx)}\leq C, \quad  |l'|+|l|\leq q.
													%	$$
													%Finally, we set 	$C^{p,q}_{\mathbf b,[a,b]}(\R^m,\mathcal{D})$ the set of the functions   $v\in C_{\mathbf b, [a,b]}^q(\R^m\times \mathcal{D})$ such  that  there exist $C>0$ such that
													%$$
													%|\partial^{l'}_x\partial^{l}_yg(\cdot, y)|_{L^p(\R^m,dx)}\leq C, \quad  |l'|+|l|\leq q.
													%$$
													%	\smallskip
													
													\smallskip
													
													$\bullet$ 
													For fixed $X_0=(X_{01},\ldots,X_{0d})\in\R^d$ and  $\dx=(\dx_1,\dots,\dx_d)\in(0,+\infty)^d$ (spatial step),  $\mathcal{X}=\{x=(X_{01}+i_1\Delta x_1,\dots, X_{0d}+i_d\Delta x_d)\}_{i\in \Z^d}$ denotes a discrete grid in $\R^d$. For $p\in[1,\infty]$, we  set $l_p(\mathcal{X})$ the discrete $l_p$-space of the functions $\varphi\,:\,\mathcal{X}\to\R$ with the norm $|\varphi|_{p}=(\sum_{x\in\mathcal{X}} |\varphi(x)|^{p}\dx_1\cdots\dx_d)^{ 1 /p}$ if $p\in [1,\infty)$ and $|\varphi|_{\infty}=\sup_{x\in\mathcal{X}}|\varphi(x)|$ if $p=\infty$.  Moreover, for a linear operator  $\Gamma \,:\, l_p(\mathcal{X})\to l_p(\mathcal{X})$, the induced norm is denoted by $|\Gamma|_p=\sup_{|\varphi|_p\leq 1}|\Gamma \varphi|_{p}$.  
													And for a function $g\,:\,\R^d\to\R$, we set $|g|_p$ the $l_p(\mathcal{X})$ norm of the restriction of $g$ on $\mathcal{X}$. When $d=1$, we identify $(\varphi(x))_{x\in\mathcal{X}}$ with $(\varphi_i)_{i\in\Z}$ through $\varphi_i=\varphi(X_0+i\Delta x)$, $i\in\Z$.
													
													\smallskip
													
													$\bullet$ 
													$L^p(\Omega)$ is the short notation for the standard $L^p$-space on the probability space $(\Omega,\mathcal{F},\P)$, on which the expectation is denoted by $\E$. We set $\|\cdot\|_p$  the norm in $L^p(\Omega)$.

													\section{First order weak convergence of Markov chains to diffusions}\label{sect-markovapprox}
													Let $d\in \N^*$ and $\mathcal{D}\subseteq\R^d$ be a convex  domain or a closure of it. %(in our example on the CIR process in Section \ref{sect-CIR}, we take $d=1$ and $\mathcal{D}=[0,\infty)$).
													On a probability space $(\Omega, \mathcal F,\P)$, we consider a $d$-dimensional diffusion process driven by
													\begin{equation}\label{Y}
													dY_t=\mu_Y(Y_t)dt+\sigma_Y(Y_t)dW_t, \qquad Y_0 \in \mathcal{D},
													\end{equation}
													where $W$ is a $\ell$-dimensional standard Brownian motion. From now on, we set $a_Y=\sigma_Y\sigma_Y^\star$, the notation $\star$ denoting transpose. We recall that the associated infinitesimal generator is given by
													\begin{equation}
													\label{A-op}
													\mathcal{A}=\frac 12 \tr (a_Y D_ y^2)+\mu_Y\cdot \nabla_y,
													\end{equation}
													where $\tr$ denotes the matrix trace, $D^2_y$ and $\nabla_y$ are, respectively, the Hessian and the gradient operator w.r.t. the space variable $y$ and the notation  ``$\cdot$'' stands for the scalar product.
													
													Hereafter, we fix $T>0$, $f:\mathcal{D}\rightarrow \R$ and we define
													\begin{equation}\label{u}
													u(t,y)=\E[f(Y^{t,y}_T)],\quad (t,y)\in[0,T]\times\mathcal{D},
													\end{equation}
													where $Y^{t,y}$  denotes the solution to the SDE in \eqref{Y} that starts at $t$ in the position $y$. We do not enter in specific requests for the diffusion coefficients or for $f$, we just ask that the following properties are met:
													\begin{itemize}
														\item [(a)]
														$\mu_Y$ has polynomial growth; % \textcolor{red}{nessuna ipotesi di crescita polinomiale su $\sigma_Y$???}
														\item [(b)]	
														for every $(t,y)\in[0,T]\times \mathcal{D}$ there exists a unique weak solution $(Y^{t,y}_s)_{s\in[t,T]}$ of \eqref{Y} such that
														$\P(\forall s\in[t,T] , \,Y^{t,y}_s\in\mathcal{D})=1;$
														\item[(c)]
														the function  $u$ in \eqref{u} solves the PDE
														\begin{equation}\label{PDE-u-Y}
														\begin{cases}
														\frac{\partial u }{\partial t}+\mathcal{A} u=0,\qquad& \mbox{in } [0,T)\times \D,\\
														u(T,y)=f(y),  &\mbox{in }  \D.
														\end{cases}	
														\end{equation}	
													\end{itemize}
													The above proverties (a), (b) and (c) will be assumed to hold throughout this section.
													
													We are interested in the numerical evaluation of $u(0,Y_0)=\E(f(Y_T))$.
													A widely used and computationally convenient method  is by computing the above expectation on an approximation of the process $Y$. Here, we consider an approximation through a  Markov chain that weakly converges to the diffusion process $Y$, see e.g. the classical references %\cite{Bil, EK, SV}
													\cite{SV}.  We will see in Section \ref{sect-CIR} an application to tree methods, that is, when the process $Y$ is approximated by means of a computationally simple Markov chain. 
													Here, our aim is to study, under suitable but quite general assumptions, the order of weak convergence.

													%        In this framework, Nelson and Ramaswamy in \cite{nr} introduced the so called \avirg multiple  jumps " approach, where the  discrete approximating process can have an up jump and a down jump but not necessarily on the two adjacent nodes of the tree as for the standard tree models. The lattice of the tree, the \avirg up" and \avirg down" jumps   and the transition probabilities are set in order to best fit the first and the second local moments of the approximated process. 
													%        
													%Here, we consider an approximation for $u$ in \eqref{u} through a  Markov chain that weakly approximates the process $Y$. So, our aim is to study, under suitable assumptions, the order of weak convergence. 
													
													So,  let $N\in\N^*$ and set $h=T/ N$. The parameters $N$ and $h$ are fixed once for all. Let $( Y^h_{n})_{n=0,\ldots,N}$ denote a Markov chain, 
													whose state space, at time-step $n$, is given by $\mathcal{Y}_n^h\subset \mathcal{D}$. In our mind, $( Y^h_{n})_{n=0,\ldots,N}$ is a Markov process which is a discrete weak  approximation in time (and possibly in space) of the $d$-dimensional diffusion $Y$, namely, $Y^h_n$ approximates $Y$ at times $nh$, for every $n=0,\ldots,N$. Of course, we assume that $Y^h_0=Y_0$, that is, $\mathcal{Y}^h_0=\{Y_0\}$.
													Without loss of generality, we may assume that $( Y^h_{n})_{n=0,\ldots,N}$ is defined in $(\Omega,\mathcal{F},\P)$.

													%     The weak convergence of the Markov chain to the diffusion process is well known to be governed by the behavior of the local moments up to order 3 or 4 (see \cite{Bil, EK, KD, SV}). In order to get the speed of convergence, we need to stress such requests, as done in the following assumption.
													%Hereafter, the notation $\|\cdot\|_p$ stands for the standard norm in $L^p(\Omega)$ .
													In order to study the rate of the weak convergence of $(Y^h_n)_{n=0,\dots,N}$ to $Y$, we need to stress the requests that   are usually done in order to merely prove the convergence (see e.g. %\cite{Bil, EK, SV}
													\cite{SV}). In particular, we need the following assumption.
													\medskip
													
													\noindent
													\textbf{Assumption $\mathcal{A}_1$.}
													\textit{There exists $	\bar h >0$ such that, for every $h< \bar h$,  the first three local moments satisfy
														\begin{align}
														\E[Y^h_{n+1}-Y^h_n\mid Y^h_n]
														&=\mu_Y(Y^h_n)h+f_h(Y^h_n), \label{momento1-gen}\\
														\E[( Y^h_{n+1}-Y^h_n)( Y^h_{n+1}-Y^h_n)^\star\mid Y^{h}_n]
														&=a_Y(Y^{h}_n) h 
														+g_h(Y^{h}_n),\label{momento2-gen}\\
														\E[( Y^h_{n+1}-Y^h_n)^l\mid Y^{h}_n]
														&=j_{h,l}(Y^{h}_n), \qquad l\in\N^d, \,|l|=3\label{momento3-gen},
														\end{align}
														where  $
														f_h:\mathcal{D}\rightarrow \R^d$, $g_h:\mathcal{D}\rightarrow \R^{d\times d} $ and $j_{h,l}:\mathcal{D}\rightarrow \R$ satisfy the following properties:  there exist $p>1$ and $C>0$ such that
														\begin{align}
														%\exists p> 1, \, C\geq 0 : 
														&\quad \sup_{h\leq \bar h }	\sup_{n=0,\dots, N}	\|f_h(Y^h_n)\|_p\leq Ch^2
														,\label{stima-f}\\
														% \exists p> 1, \, C\geq 0 : 
														&\quad \sup_{h\leq \bar h }	\sup_{n=0,\dots, N}\|g_h(Y^h_n)\|_p\leq Ch^2
														,\label{stima-g}\\
														%\exists p>1, \, C\geq 0 : 
														&\quad \sup_{h\leq \bar h }	\sup_{n=0,\dots, N}\|j_{h,l}(Y^h_n)\|_p\leq Ch^2, \quad |l|=3.\label{stima-j}
														\end{align}
													}

													We also need  the following behavior of the moments.
													
													\medskip
													
													\noindent
													\textbf{Assumption $\mathcal{A}_2$.}
													\textit{There exists $	\bar h >0$ such that for every $p>1$ there exists $C_p>0$ for which
														\begin{align}
														&\label{stima_momento-gen}
														\sup_{h<\bar h}\sup_{0\leq n \leq N} \|Y^h_n\|_p\leq C_{p},\\
														&\label{stima_incremento-gen}
														\sup_{h<\bar h}\sup_{0\leq n\leq N} \frac{1}{\sqrt h} \|Y^h_{n+1}-Y^h_n\|_p\leq C_p.
														\end{align}
													}
													% 	\begin{enumerate}
													% 		\item 
													% 		%$(Y^h_n)_{n=0,\dots,N}$ has uniformly bounded moments:  
													% 		for any $p\geq 1$ there exists $C_{p}>0$ such that
													% 		\begin{equation}\label{stima_momento-gen}
													% 		\sup_{h<\bar h}\sup_{0\leq n \leq N}\E[(Y^h_n)^p]\leq C_{p};
													% 		\end{equation}
													% 		\item 
													% 		for any $p\geq 1$ there exists $C_p>0$ such that for any $h<\bar h$,
													% 		\begin{equation}\label{stima_incremento-gen}
													% 		\sup_{n=0,\dots, N} \E[|Y^h_{n+1}-Y^h_n|^p]\leq C_p\,h^{p/2}.
													% 		\end{equation}
													% 	\end{enumerate}
													% }
													
													\medskip

													We can now state the following first order weak convergence result.
													
													\begin{theorem} \label{conv_T}
														Let assumptions $\mathcal{A}_1$ and $\mathcal{A}_2$ hold and assume that  $u\in C^{4}_{\pol,T}(\D)$, $u$ being defined in \eqref{u}.
														Then  there exist $\bar h>0$ and  $C>0$ such that for every $h<\bar h$ one has
														$$
														|	\E[f(Y^h_N)]-\E[f(Y_T)]|\leq CTh.
														$$
													\end{theorem}
													\begin{proof}
														The proof is quite standard.	Since  $\E[f(Y^h_N)]=\E[u(T,Y^h_T)]$ and $  \E[f(Y_T)] =u(0,Y_0)$, we have
														$$
														\E[f(Y^h_T)]- \E[f(Y_T)] =\E[  u(T,Y^h_T)-u(0,Y_0) ]=\sum_{n=0}^{N-1}\E[  u((n+1)h,Y^h_{n+1})-u(nh,Y^h_n) ].
														$$
														Since $u\in C^{4}_{\pol,T}(\D)$, we can  apply Taylor's formula to $t\mapsto u(t,y)$ around $nh$    up to order 1 and to the functions $y\mapsto u(t,y)$ and $y\mapsto \partial_t u(t,y)$ around $Y^h_n$ up to order 3 and 1 respectively. We  obtain
														\begin{equation}\label{app1}
														u((n+1)h,Y^h_{n+1})%-u(nh,Y^h_n)
														= \sum_{ 
															0\leq |l|+2l'\leq 3} \partial_y^{l}\partial_t^{l'} u(nh,Y^h_n)\frac{h^{l'}(Y^h_{n+1}-Y^h_n)^{l}}{|l|!l'!} +R_1(n,h,Y^h_n,Y^h_{n+1}),
														\end{equation}
														where the remaining term $R_1$ is given by
														\begin{align*}
														R_1(n,h,Y^h_n,Y^h_{n+1})	
														&=h^2\int_0^1 (1-\tau)\partial^2_tu(t+\tau h,Y_{n+1}^h)d\tau\\
														&+h\sum_{|k|=2}(Y^h_{n+1}-Y^h_n)^k\int_0^1 (1-\xi)\partial^k_y\partial_tu(nh,Y^h_n+\xi(Y^h_{n+1}-Y^h_n))d\xi\\
														&+\sum_{|k|=4}\frac{(Y^h_{n+1}-Y^h_n)^k}{3!}\int_0^1 (1-\xi)^3\partial^k_yu(nh,Y^h_n+\xi(Y^h_{n+1}-Y^h_n))d\xi.
														\end{align*}
														We now pass to the conditional expectation w.r.t. $Y^h_n$ in \eqref{app1} and use \eqref{momento1-gen} and \eqref{momento2-gen}. By rearranging the terms we obtain
														\begin{equation}\label{resti}
														\begin{split}
														&\E[u((n+1)h,Y^h_{n+1})-u(nh,Y^h_n)] \\& \qquad=h\E\left[ \partial_tu(nh,Y^h_n)+\mu_Y(Y_n^h)\cdot \nabla_y u(nh,Y^h_n)
														+\frac 12 \mbox{Tr}(a_YD^2_y  u(nh,Y^h_n))\right] +\sum_{i=1}^5 R^i_n(h), %+R^{1}_n(h)+R^{2}_n(h)+R^{3}_n(h)+R^{4}_n(h)+R^{5}_n(h),
														\end{split}	
														\end{equation}
														in which
														$$
														\begin{array}{ll}
														\displaystyle
														R^1_{n}(h)=\E[ R_1(n,h,V^h_n,V^h_{n+1})],
														&\displaystyle
														R^{2}_n(h)=h\E[(\mu_Y(Y^h_n)h+ f_h(Y^h_n))\cdot \nabla_y\partial_tu(nh,Y^h_{n})],\\
														\displaystyle
														R^{3}_n(h)=\E[f_h(Y^h_n)\cdot \nabla_yu(nh,Y^h_{n})],
														&\displaystyle
														R^{4}_n(h)=\frac 12\E[\tr(g_h(Y^h_n)D^{2}_yu(nh,Y^{h}_n))],\\
														\displaystyle
														R^{5}_n(h)=\frac 16\sum_{|k|=3}\E[\partial_y^ku(nh,Y^h_{n})j_{h,k}(Y^h_n)].&
														%\\R^{5}_n(h)=&
														%\E[\partial^{2}_yu(nh,Y^{h}_n)\psi(Y^{h}_n, Y^{h}_{n+1})\ind{\{Y^{h}_n \geq \theta^*/h\}}].
														\end{array}
														$$

														Thanks to \eqref{PDE-u-Y},  the first term in \eqref{resti} is null, so
														$$
														|\E[u((n+1)h,Y^h_{n+1})-u(nh,Y^h_n)]|
														\leq \sum_{i=1}^{5}|R^{i}_n(h)|.
														$$
														We now prove that $|R^{i}_n(h)|\leq C h^2$, for every $i=1,\ldots 5$. Let $\bar h>0$ such that both assumptions $\mathcal{A}_1$ and $\mathcal{A}_2$ hold and let $h<\bar h$.
														Since the derivatives of $u$ have polynomial growth,  one has
														\begin{align*}
														&|R_1(n,h,Y^h_n,Y^h_{n+1})|\leq C\Big(1+|Y^h_n|+|Y^h_{n+1}|\Big)^{a}\Big[h^2+h|Y^h_{n+1}-Y^h_n|^2+|Y^h_{n+1}-Y^h_n|^4\Big],
														\end{align*}
														where $C,a>0$ denote constants that are independent of $h$ and, from now on, may change from a line to another.
														Then, by using the Cauchy-Schwarz inequality, \eqref{stima_momento-gen} and \eqref{stima_incremento-gen}, we get
														\begin{align*}
														|R^{1}_n(h)|&\leq C \big\|(1+|Y^h_{n+1}|+|Y^h_n|)^{a}\big\|_2\, \big\|h^2+h (Y^h_{n+1}-Y^h_n)^2+(Y^h_{n+1}-Y^h_n)^4 \big\|_2\leq C h^2.
														\end{align*}
														As regards $R^2_n(h)$, we use the polynomial growth of $\nabla_y\partial_tu$, the Cauchy-Schwarz inequality and the H\"older inequality, so that
														\begin{align*}
														|R^{2}_n(h)|&
														\leq C \E[\big(1+|Y^h_{n}|^{a}\big)|\mu_Y(Y^h_n)|]\,h^{2}+ C\E[\big(1+|Y^h_{n}|^a\big)|f_h(Y^h_n)|]\,\\&
														\leq C\big\|1+|Y^h_{n}|^{a}\big\|_2
														\,\big\|\mu_Y(Y^h_n)\big\|_2\,h^{2}+ C\big\|1+|Y^h_{n}|^{a}\big\|_q
														\,\big\|f_h(Y_n^h)\big\|_p,
														\end{align*}
														where $p$ is given in \eqref{stima-f} and $q$ is its conjugate exponent.  Since $\mu_Y$ has polynomial growth,  by \eqref{stima-f} and \eqref{stima_momento-gen} we get
														$$
														|R^{2}_n(h)|\leq C h^2.
														$$
														The remaining terms $R^3_n(h)$, $R^4_n(h)$ and $R^5_n(h)$ can be handled similarly, so %we obtain
														%		$$
														%		\big| \E[f(Y^h_T)]- \E[f(Y_T)]\big|\leq \sum_{n=0}^{N-1}\sum_{i=1}^{5}|R^{i}_n(h)|
														%		\leq CT h,
														%		$$
														%		and 
														the statement follows.
													\end{proof}
													
													\subsection{An example: a first order weak convergent binomial tree for the CIR process}\label{sect-CIR}
													We now fix $d=1$ and $\mathcal{D}=\R_+=[0,\infty)$. We consider the CIR process $(Y_t)_{t\in [0,T]}$
													solution to the SDE
													\begin{align*}
													&dY_t= \kappa(\theta- Y_t)dt+\sigma\sqrt{Y_t}\,dW_t,
													\quad Y_0\geq0.
													\end{align*}
													%Here, $\kappa$ is the mean reversion rate and $\theta$ is the long run variance. 
													We assume that $\theta,\kappa,\sigma>0$ and we do not require  the Feller condition. Therefore,  the process  $Y$ can reach $0$.
													
													%	The CIR process is widely used in finance  to model interest rates or the volatility process in stochastic volatility models and there is a large literature on numerical methods to approximate it, see e.g. \cite{A-MC, cgmz, HST, W}.
													
													We consider  here the  \avirg multiple jumps"  tree approximation for the  CIR process described in Section \ref{sect-tree}. %It follows the general  \avirg multiple jumps" approach  by Nelson and Ramaswamy in \cite{nr} and is computationally simple and robust. Moreover,  it turned out to be numerically competitive with the second order Alfonsi's discretization scheme, both in terms of computational time and accuracy, as shown in \cite{bcz, bcz-hhw, bctz}.
													% 	  
													% 	  It is worth stressing that the study we are going to develop may in principle be applied to general trees constructed by means of the multiple jumps approach by Nelson and  Ramaswamy. In fact, to our knowledge, a  theoretical study of the rate of convergence for such trees is missing in the literature.
													We first briefly recall how the tree works and then, as an application of Theorem \ref{conv_T}, we study the rate of convergence.
													
													Recall that, for $n=0,1,\ldots,N$ we have  the lattice
													\begin{equation}\label{vnk}
													\mathcal{Y}_n^h=\{y^n_k\}_{k=0,1,\ldots,n}\quad\mbox{with}\quad
													y^n_k=\Big(\sqrt {Y_0}+\frac{\sigma} 2(2k-n)\sqrt{h}\Big)^2\ind{\{\sqrt {Y_0}+\frac{\sigma} 2(2k-n)\sqrt{h}>0\}}.
													\end{equation}
													Note that $\mathcal{Y}_0^h=\{Y_0\}$. 
													For each fixed node $(n,k)\in\{0,1,\ldots,N-1\}\times\{0,1,\ldots,n\}$, the ``up'' jump $k_u(n,k)$ and the ``down'' jump $k_d(n,k)$ from $y^n_k\in\mathcal{Y}_n^h$ are defined as 
													\begin{align}
													\label{ku2}
													&k_u(n,k) =
													\min \{k^*\,:\, k+1\leq k^*\leq n+1\mbox{ and }y^n_k+\mu_Y(y^n_k)h \le y^{n+1}_{k^*}\},\\
													%\min K_u(n,k),\ \mbox{with}\   K_u(n,k)=\{k^*\,:\, k+1\leq k^*\leq n+1\mbox{ and }y^n_k+\mu_Y(y^n_k)h \le y^{n+1}_{k^*}\},\\
													\label{kd2}
													&k_d(n,k)=
													\max \{k^*\,:\, 0\leq k^*\leq k \mbox{ and }y^n_k+\mu_Y(y^n_k)h \ge y^{n+1}_{ k^*}\},
													%\max K_d(n,k),\ \mbox{with}\   K_d(n,k)=\{k^*\,:\, 0\leq k^*\leq k \mbox{ and }y^n_k+\mu_Y(y^n_k)h \ge y^{n+1}_{ k^*}\},
													\end{align}
													where $\mu_Y(y)=\kappa(\theta-y)$ and
													with the understanding $k_u(n,k)=n+1$, resp. $k_d(n,k)=0$,  if the set in \eqref{ku2}, resp. \eqref{kd2}, is empty. 
													In fact, starting from the node $(n,k)$ the probability that the process jumps to $k_u(n,k)$ and $k_d(n,k)$ at time-step $n+1$ are set as
													\begin{equation*}key
													p_u(n,k)
													=0\vee \frac{\mu_Y(y^n_k)h+ y^n_k-y^{n+1}_{k_d(n,k)} }{y^{n+1}_{k_u(n,k)}-y^{n+1}_{k_d(n,k)}}\wedge 1
													\quad\mbox{and}\quad p_d(n,k)=1-p_u(n,k)
													\end{equation*}
													respectively. We will see in next Proposition \ref{propjumps} that for $h$ small enough the parts ``$0\vee$'' and ``$\wedge 1$'' can be omitted.
													
													We call $(Y^h_n)_{n=0,1,\ldots,N}$ the Markov chain governed by the above jump probabilities. As an application of Theorem \ref{conv_T}, we shall prove the following result.
													
													\begin{theorem}\label{conv_T-CIR}
														Let $f\in C^{4}_\pol(\R_+)$. Then, 
														there exist $\bar h>0$ and $C>0$ such that for every $h<\bar h$,
														$$
														|	\E[f(Y^h_N)]-\E[f(Y_T)]|\leq CTh,
														$$
														that is, the tree approximation $(Y^h_n)_{n=0,\dots,N}$ is first order weak convergent.
													\end{theorem}
													
													In order to discuss the  assumptions $\mathcal{A}_1$ and $\mathcal{A}_2$  of Theorem \ref{conv_T}, we need some preliminary results which pave the way to the analysis of the convergence. 
													
													%We now study suitable up and and down thresholds such that each node belonging to the associated strip gives rise to single jumps, that is, $k_u(n,k)=k+1$ and $k_d(n,k)=k$.
													\begin{proposition}\label{propjumps}
														%		Set
														%		\begin{equation}\label{const}
														%		\theta_*= \Big( \frac{\kappa\theta}{\sigma}\Big)^2, \quad  \theta^*= \frac{\sigma^2}{4\kappa^2}
														%		%\Big(   \frac{\sigma^2}{2}+\sqrt{\kappa \theta}  \Big)
														%		, \quad C_*= (\kappa\theta + 2\sigma\sqrt{ \theta_*+\kappa\theta}+\sigma^2) \vee  \Big(   \frac{\sigma^2}{4}+\sigma \sqrt{\theta_*}  \Big).
														%		\end{equation}
														There exist  $\theta_*,\theta^*, C_*, \bar h>0$ such that for any $h<\bar h$ the following properties hold.
														\begin{enumerate}
															\item If $\theta_*h\leq y^n_k \leq  \theta^*/h$, then $	k_u(n,k)=k+1$, $k_d(n,k)=k$. Moreover,
															\begin{equation*}
															y^{n+1}_{k_u(n,k)}=y^n_k+\frac{\sigma^2}{4}h +\sigma \sqrt{y^n_k h }
															\quad\mbox{and}\quad
															y^{n+1}_{k_d(n,k)}=y^n_k+\frac{\sigma^2}{4}h -\sigma \sqrt{y^n_k h }.
															\end{equation*}
															
															\item If $y^n_k < \theta_*h$, then $k_d(n,k)=k$. Moreover,
															\begin{equation}\label{stimasottosoglia}
															0\leq y^{n+1}_{k_u(n,k)}-y^n_k\leq C_*h.
															\end{equation}
															
															\item  If $y^n_k > \theta^*/h$, then $k_u(n,k)=k+1$.
															
															\item The jump probabilities are
															\begin{equation}\label{proba}
															p_u(n,k)
															= \frac{\mu_Y(y^n_k)h+ y^n_k-y^{n+1}_{k_d(n,k)} }{y^{n+1}_{k_u(n,k)}-y^{n+1}_{k_d(n,k)}},
															\quad \quad p_d(n,k)=    \frac{ y^n_{k_u(n,k)}-y^{n}_{k}-\mu_Y(y^n_k)h }{y^{n+1}_{k_u(n,k)}-y^{n+1}_{k_d(n,k)}}.
															\end{equation}
															
														\end{enumerate}
													\end{proposition}
													
													The proof of Proposition \ref{propjumps} relies on a boring study of the properties of the lattice, so we postpone it in Appendix \ref{app-jumps}.
													This is all we need to prove that $\mathcal{A}_2$ holds:
													
													\begin{proposition}\label{moments-2} 
														The CIR approximating tree $\{Y^h_n\}_{n=0,\ldots,N}$ satisfies Assumption $\mathcal{A}_2$. 
													\end{proposition}
													
													\begin{proof}
														\textbf{Step 1: proof of \eqref{stima_momento-gen}.}
														We use a technique  firstly developed in \cite{A-MC} for a CIR discretization scheme based on Brownian increments. The key point is the proof of a monotonicity property allowing one to control the moments of the tree:  there exist $b,C, \bar h>0$ such that for every $h < \bar h$ and $n=0,\dots,N-1$ one has
														\begin{equation}\label{key}
														0\leq Y^h_{n+1}\leq (1+bh)Y^h_n+Ch+\sigma\sqrt{Y^h_nh} \, W^h_{n+1},
														\end{equation}
														where $W^h_{n+1}$ is a r.v. such that %whose conditional law with respect to $Y^h_n$  is  given by
														\begin{equation}\label{law-W}
														\P(W^h_{n+1}=2p_d(n,k)|Y^h_n=y^n_k )=p_u(n,k)=1-\P(W^h_{n+1}=-2p_u(n,k)|Y^h_n=y^n_k ).
														\end{equation}
														To this purpose,  fix a node $(n,k)$. For the sake of simplicity, we write $k_u$, resp. $k_d$, in place of $k_u(n,k)$, resp. $k_d(n,k)$. We have (see \eqref{tree1}) that
														%	\begin{equation}\label{single}
														$$
														y^{n+1}_{k+1}\leq y^n_k+\frac{\sigma^2}{4}h +\sigma \sqrt{y^n_k h },\qquad
														y^{n+1}_{k}\leq y^n_k+\frac{\sigma^2}{4}h -\sigma \sqrt{y^n_k h }.
														$$
														%	\end{equation}
														By Proposition \ref{propjumps}, for $h<\bar h$,
														if $\theta_*h< y^n_k <\theta^*/h$  the up and down jumps are  both single, hence $y^{n+1}_{k_u}=y^{n+1}_{k+1}$ and $y^{n+1}_{k_d}=y^{n+1}_{k}$
														%		\begin{equation}\label{single}
														%		y^{n+1}_{k_u}=y^{n+1}_{k+1}=y^n_k+\frac{\sigma^2}{4}h +\sigma \sqrt{y^n_k h },\qquad
														%		y^{n+1}_{k_d}=y^{n+1}_{k}=y^n_k+\frac{\sigma^2}{4}h -\sigma \sqrt{y^n_k h }.
														%		\end{equation}
														On the other hand, if $y^n_k \geq \theta^*/h$ the up jump is single, that is $y^{n+1}_{k_u}=y^{n+1}_{k+1}$ , %we have
														%	\begin{equation*}
														%	y^{n+1}_{k_u}=y^{n+1}_{k+1}=y^n_k+\frac{\sigma^2}{4}h +\sigma \sqrt{y^n_k h },
														% 	\end{equation*}
														while the down jump can be multiple but, in every case, is still true that
														\begin{equation*}
														y^{n+1}_{k_d}\leq	y^{n+1}_{k}=  y^n_k+\frac{\sigma^2}{4}h -\sigma \sqrt{y^n_k h }.
														\end{equation*}
														Finally, if $y^n_k \leq \theta_*h$, we have $y^{n+1}_{k_d}=y^{n+1}_{k}$, %, by \eqref{tree1},
														% 	 	\begin{equation*}
														% 	y^{n+1}_{k_d}=y^{n+1}_k\leq y^n_k+\frac{\sigma^2}{4}h -\sigma \sqrt{y^n_k h },
														%	\end{equation*}
														while the up jump can be multiple but we can always write
														$$
														y^{n+1}_{k_u} \leq y^n_k+C_*h \leq y^n_k+C_*h+ \sigma \sqrt{y^n_k h}.
														$$
														Summing up, if we set $\bar C=\max\Big(C_*, \frac{\sigma^2}{4}\Big)$,
														for every $h$ small %and for any $(n,k)$ we get
														%	$$
														%	y^{n+1}_{k_u} \leq y^n_k+\bar Ch+ \sigma \sqrt{y^n_k h }
														%	\quad\mbox{and}\quad
														%	y^{n+1}_{k_d} \leq y^n_k+\bar Ch- \sigma \sqrt{y^n_k h }.
														%	$$
														%	Then, 
														we can write
														$$
														0\leq Y^h_{n+1}\leq Y^h_n+ \bar Ch+\sigma \sqrt{Y^h_n h }\, Z^h_{n+1},
														$$
														where $Z^h_{n+1}$ is a random variable such that $\P( Z^h_{n+1}=+1|Y^h_n=y^n_k )=p_u(n,k)$ and  $\P( Z^h_{n+1}=-1|Y^h_n=y^n_k )=p_d(n,k)$. %whose conditional law w.r.t. $Y^h_n$ is given by
														% 	 	$$
														% 	 Z^h_{n+1}|Y^h_n=y^n_k = \begin{cases}
														% 	 	+1 \qquad \mbox{ with probability } p_u(n,k),\\
														% 	 	-1 \qquad \mbox{ with probability } p_d(n,k).
														% 	 	\end{cases}
														% 	 	$$
														Note that $\E(Z^h_{n+1}|Y^h_n=y^n_k )=p_u(n,k)-p_d(n,k)=2p_u(n,k)-1$. Then,  the random variable
														$$
														W^h_{n+1}= Z^h_{n+1}-\E[Z^h_{n+1}|Y^h_n]
														$$
														has  exactly the law given in \eqref{law-W}.
														%	Moreover, we notice that  $ \E(W^h_{n+1}|Y^h_n)=0$ and $|W^h_{n+1}|\leq 2.$
														We also define  the function $	P_u(y^n_k)= p_u(n,k)$.
														Therefore,
														\begin{align*}
														0\leq &	Y^h_{n+1}\leq %Y^h_n+ \bar Ch+\sigma \sqrt{Y^h_n h } \, \E[Z^h_{n+1}|Y^h_n] +\sigma \sqrt{Y^h_n h }\,W^h_{n+1} \\&=
														Y^h_n+ \bar Ch+\sigma \sqrt{Y^h_n h} \,(2P_u(Y^h_n)-1)  +\sigma \sqrt{Y^h_n h }\,W^h_{n+1}\\
														\leq &
														Y^h_n+ \bar Ch+\sigma \sqrt{\theta^*}\sqrt{\frac{Y^h_n h}{\theta^*} } \,\big|2P_u(Y^h_n)-1\big|\ind{ \{Y^h_n  \geq \frac{\theta^*}{h}\}}
														+\sigma \sqrt{Y^h_n h} \,\big(2P_u(Y^h_n)-1\big)\ind{ \{Y^h_n < \frac{\theta^*}{h}\}}\\
														&+\sigma \sqrt{Y^h_n h }\,W^h_{n+1}.
														\end{align*}
														Now, if $Y^h_n  \geq \frac{\theta^*}{h}$ then $\sqrt{\frac{Y^h_n h}{\theta^*} }
														\leq \frac{Y^h_n h}{\theta^*}$ and, since $P_u\in[0,1]$, we have $|2P_u(Y^h_n)-1|\leq 1$. Then, 
														$$
														0\leq Y^h_{n+1}
														\leq
														(1+bh)Y^h_n+ \bar Ch
														+\sigma \sqrt{Y^h_n h} \,\big(2P_u(Y^h_n)-1\big)\ind{ \{Y^h_n < \frac{\theta^*}{h}\}}
														+\sigma \sqrt{Y^h_n h }\,W^h_{n+1},
														$$
														where $b=\frac{\sigma}{\sqrt{\theta^*}}$.
														Let us study the quantity $\sigma \sqrt{Y^h_n h }\, (2P_u(Y^h_n)-1) \ind{ \{Y^h_n  < \frac{\theta^*}{h}\}}  $. If   $\theta_*h< y^n_k < \theta^*/h$, by using \eqref{proba} and point 1. of Proposition \ref{propjumps}, we can explicitly write
														\begin{align*}
														\sigma \sqrt{y^n_k h }\,&(	2P_u(y^n_k)-1)%=\sigma \sqrt{y^n_k h }\,\Big(2\Big(  \frac{\mu_Y(y^n_k)h+ y^n_k-y^{n+1}_{k_d(n,k)} }{y^{n+1}_{k_u(n,k)}-y^{n+1}_{k_d(n,k)}} \Big)-1\Big)\\
														= 	\sigma \sqrt{y^n_k h }\,\Big(2\Big(  \frac 1 2 + \frac{4\mu_Y(v_k^n)-\sigma^2}{8\sigma \sqrt{y^n_k h }}   \Big)h-1
														\Big)=	\mu_Y(v_k^n)h-\frac{\sigma^2}4h \leq  \kappa\theta h.	
														\end{align*}
														If instead $y^n_k \leq \theta_*h$, then by using 2. in Proposition \ref{propjumps} we have
														\begin{align*}
														\sigma \sqrt{y^n_k h }\,&(	2P_u(y^n_k)-1)=	\sigma \sqrt{y^n_k h }\,  \frac{2\mu_Y(y^n_k)h+ 2y^n_k-y^{n+1}_{k_d(n,k)}-y^{n+1}_{k_u(n,k)} }{y^{n+1}_{k_u(n,k)}-y^{n+1}_{k_d(n,k)}} \\
														&\leq	\sigma \sqrt{y^n_k h }\, \frac{2\mu_Y(y^n_k)h+ 2y^n_k }{y^{n+1}_{k+1}-y^{n+1}_{k}} \leq	\sigma \sqrt{y^n_k h }\,  \frac{2\kappa\theta h+ 2\theta_*h }{2\sigma \sqrt{y^n_k h }} =(\kappa\theta + \theta_*)h.	
														\end{align*}
														So, by inserting, for every $n\leq N-1$ we get
														\begin{align*}
														\nonumber	0\leq Y^h_{n+1}&\leq (1+bh)Y^h_n+\bar Ch+\sigma (\kappa\theta + \theta_*) h
														%\ind{ \{Y^h_n  \leq \frac{\theta^*}{h}\}} 
														+\sigma \sqrt{Y^h_n h }\,W^h_{n+1}
														%\\
														%	&\leq (1+bh)Y^h_n+Ch+\sigma \sqrt{Y^h_n h }\,W^h_{n+1}.	
														\end{align*}
														%where $C=\bar C + \kappa\theta + \theta_*$
														and \eqref{key} is proved.
														
														Now,  we repeat step by step the proof of Lemma 2.6 in \cite{A-MC} in order to get \eqref{stima_momento-gen}.
														We use induction on $p$. For $p=1$, by definition one has  $\E[ Y^h_{n+1}| Y^h_n]=  Y^h_n+\mu_ Y( Y^h_n)h$  and, by  passing to the expectation, $\E[ Y^h_{n+1}]=\E[ Y^h_n]+\E[\mu_ Y( Y^h_n)h]\leq \E[ Y^h_n]+\kappa\theta h$, from which we obtain $ \E[ Y^h_{n+1}]\leq  Y_0+\kappa\theta (n+1)h\leq  Y_0+\kappa\theta T$ and the case $p=1$ is proved. So, assume that \eqref{stima_momento-gen} holds for $p-1$ and let us prove its validity for $p$. 
														Using \eqref{key}, we have
														%	$$
														%	( Y^h_{n+1})^p\leq \sum_{l_1+l_2+l_3=p} \frac{p!}{l_1!l_2!l_3!}( Y^h_n)^{l_1+\frac{l_2}2}C^{l_3}h^{l_3+\frac{l_2}{2}}(C_2\sigma W^h_{n+1})^{l_2},
														%	$$
														%	thus
														\begin{align*}
														\E[	( Y^h_{n+1})^p]\leq \sum_{l_1+l_2+l_3=p} \frac{p!}{l_1!l_2!l_3!}(1+bh)^{l_1}\sigma^{l_2}C^{l_3}\E\left[( Y^h_n)^{l_1+\frac{l_2}2}h^{l_3+\frac{l_2}{2}}(W^h_{n+1})^{l_2}\right].	
														\end{align*}
														So, it is sufficient to control $\mathcal{E}(l_1,l_2,l_3)= \E\left[ ( Y^h_n)^{l_1+\frac{l_2}2} h^{l_3+\frac{l_2}{2}}(W^h_{n+1})^{l_2} \right]$ for $l_1+l_2+l_3=p$.

														Assume first that $l_1+\frac{l_2}2\leq p-\frac 3 2 $, a case giving  $l_3+\frac{l_2}{2}\geq \frac 3 2$. Without loss of generality we can assume  $C_{p-1}\geq 1$. Moreover, recall that $|W^h_{n+1}| \leq 2$. By using the H\"{o}lder's inequality with $\alpha=\frac{p-1}{l_1+\frac{l_2}2}$, we get
														\begin{align*}
														\mathcal{E}(l_1,l_2,l_3)\leq |\mathcal{E}(l_1,l_2,l_3)| \leq \E\left[ ( Y^h_n)^{l_1+\frac{l_2}2}\right] 2^{l_2} h^{l_3+\frac{l_2}{2}} \leq C_{p-1} 2^{l_2} h^{\frac 32}.
														\end{align*}
														Therefore
														\begin{align*}
														\sum_{l_1+l_2+l_3=p \atop l_1+l_2/2\leq p-3/2}   \frac{p!}{l_1!l_2!l_3!}(1+bh)^{l_1}\sigma^{l_2}C^{l_3}\mathcal{E}(l_1,l_2,l_3) 
														&	\leq  C_{p-1}h^{\frac{3}{2}} \sum_{l_1+l_2+l_3=p}\frac{p!}{l_1!l_2!l_3!}(1+bh)^{l_1}(2\sigma)^{l_2}C^{l_3}\\ 
														&%\quad =\eta(p-1)h^{\frac{3}{2}} (1+bh+2\sigma+C)^p
														\leq C_{p-1}h^{\frac{3}{2}} (1+b+2\sigma+C)^p.
														\end{align*}
														The case $l_1+\frac{l_2}2> p-\frac 3 2$ gives 4 further contributions, namely $(l_1,l_2,l_3)=(p,0,0)$, $(p-1,0,1)$, $(p-1,1,0)$ and $(p-2,2,0)$. So, we get
														%and $l_1+l_2+l_3=p$ are: $l_1=p, l_2=l_3=0$; $l_1=p-1, l_2=1,l_3=0$; $l_1=p-1, l_2=0, l_3=1$ and $l_1=p-2, l_2=2, l_3=0$. Therefore, we have:
														\begin{align*}
														\E[	( Y^h_{n+1})^p]&\leq 
														C_{p-1} (1+b+2\sigma+C)^ph^{\frac{3}{2}}
														+ (1+bh)^p\E[	( Y^h_{n})^p]
														+  p(1+bh)^{p-1}Ch\E[	( Y^h_{n})^{p-1}]\\
														&+  p(1+bh)^{p-1}\sigma Ch^{1/2}\E[	( Y^h_{n})^{p-1/2}W^h_{n+1}]
														+  \frac{p(p-1)}{2}(1+bh)^{p-2}\sigma^2h\E[	( Y^h_{n})^{p-1}(W^h_{n+1})^2].
														\end{align*}
														Consider the last two terms above. For the first, we note that
														$$
														\E[	( Y^h_{n})^{p-1/2}W^h_{n+1}]=\E[( Y^h_{n})^{p-1/2}\E[	W^h_{n+1}|^h Y_n]]=0
														$$
														and for the second, we recall that $|W^h_{n+1}|\leq 2$. So, we easily obtain
														\begin{align*}
														\E[	( Y^h_{n+1})^p]
														%	&\leq 
														%	\eta(p-1) h\Big[(1+b+2\sigma+C)^p+p(1+b)^{p-1}C+\frac{p(p-1)}{2}(1+b)^{p-2}(2\sigma)^2\Big]
														%	+ (1+bh)^p\E[	( Y^h_{n})^p]\\
														&\leq 
														C_{p-1} h(1+b+2\sigma+C)^p\Big[1+p+\frac{p(p-1)}{2}\Big]
														+ (1+bh)^p\E[	( Y^h_{n})^p].
														\end{align*}
														By recursion on $n$, we get
														\begin{align*}
														\E[	( Y^h_{n+1})^p]
														&\leq 
														C_{p-1}h(1+b+2\sigma+C)^p\,\frac{p^{2}+p+2}{2}
														\sum_{{j=0}}^{n}(1+bh)^{jp}+ Y_0^{p}(1+bh)^{(n+1)p}
														\end{align*}
														and \ref{stima_momento-gen} now follows.	
														
														\medskip 	 
														
														\textbf{Step 2: proof of \eqref{stima_incremento-gen}}. 	 We can write
														\begin{align*}
														| Y^h_{n+1}- Y^h_n|^p\leq &3^{p-1}\Big| \frac{\sigma^2}{4}h +\sigma\sqrt{ Y^h_nh}Z^h_{n+1}\Big|^p\ind{ \{ \theta_*h<  Y^h_n< \theta^*/h\}}
														+3^{p-1}|  Y^h_{n+1}- Y^h_n|^p\ind{ \{  Y^h_n\leq \theta_*h\}}\\&+3^{p-1}| Y^h_{n+1}- Y^h_n|^p\ind{\{ Y^h_n\geq \theta^*/h\}}=:3^{p-1}(I_1+I_2+I_3),
														\end{align*}
														where we have used that, on the set $\{\theta_*h<  Y^h_n< \theta^*/h\}$, we have
														$
														Y^h_{n+1}=  Y^h_n+ \frac{\sigma^2}{4}h +\sigma\sqrt{ Y^h_nh}Z^h_{n+1},
														$
														with $\P(Z^h_{n+1}=1\mid  Y^h_{n+1})=P_u( Y^h_n)$ and $\P(Z^h_{n+1}=-1\mid  Y^h_{n+1})=P_d( Y^h_n)$.
														Now, by using \eqref{stima_momento-gen}, Proposition \ref{propjumps}, the Cauchy-Swartz and the Markov inequality,
														\begin{align*}
														I_1&\leq \E\Big[\Big(\frac{\sigma^2}{4}h +\sigma\sqrt{ Y^h_nh}\Big)^p\Big]
														\leq
														2^{p-1}\Big(\Big(\frac{\sigma^2}{4}\Big)^p+\sigma^p\E[( Y^h_n)^p]^{1/2} \Big)h^{p/2}
														\leq
														2^{p-1}\Big(\Big(\frac{\sigma^2}{4}\Big)^p+\sigma^p\sqrt{C_p} \Big)h^{p/2},\\
														I_2 &\leq C_*^ph^p,\\
														I_3	&\leq \E[( Y^h_{n+1}- Y^h_n)^{2p}]^{1/2} \P\Big( Y^h_n>\frac{\theta^*}{h}\Big)^{1/2}\leq 2^{p}\sqrt{\frac{C_{2p}C_p}{(\theta^*)^p} }\,h^{p/2},
														\end{align*}
														and \eqref{stima_incremento-gen} follows.
														% 	 	\begin{align*}
														% 	 	\E[| Y^h_{n+1}- Y^h_n|^p ]&\leq 3^{p-1}\Big( 2^{p-1}\Big(\Big( \frac{\sigma^2}{4}\Big)^p+\sigma^p\sqrt{\eta(p)} \Big)+C_*^p+2^{p}\sqrt{\frac{\eta(2p)\eta(p)}{(\theta^*)^p} }\Big) \,h^{p/2}.
														% 	 	\end{align*}
													\end{proof}
													
													\begin{proposition}\label{moments-1} 
														The CIR approximating tree $\{Y^h_n\}_{n=0,\ldots,N}$ satisfies Assumption $\mathcal{A}_1$.
														
													\end{proposition}
													
													\begin{proof}
														Straightforward computations give $	\E[Y^h_{n+1}-Y^h_n\mid Y^h_n]=\mu_Y(Y^h_n)h$, so \eqref{momento1-gen} and \eqref{stima-f} immediately follow. As for	\eqref{momento2-gen}, 
														\begin{align*}
														&\E[( Y^h_{n+1}-Y^h_n)^2\mid Y^h_n=y^n_k]  =\E[( Y^h_{n+1}-Y^h_n)^2\mid Y^h_n=y^n_k]\ind{\{y^n_k \leq \theta_*h\}}\\
														&\quad +\E[( Y^h_{n+1}-Y^h_n)^2\mid Y^h_n=y^n_k]\ind{\{\theta_*h\leq y^n_k \leq \theta^*/h\}} 
														+\E[( Y^h_{n+1}-Y^h_n)^2\mid Y^h_n=y^n_k] \ind{\{y^n_k > \theta^*/h\}} .
														\end{align*}
														We study separately the first two terms of the above r.h.s. If $y^n_k <\theta_*h$, Proposition \ref{propjumps} gives $|y^{n+1}_{k_u}-y^n_k| \leq C_*h $ and $|y^{n+1}_{k_d}-y^n_k| \leq C_*h $
														so that
														$$
														\E[( Y^h_{n+1}-Y^h_n)^2\mid Y^{h}_n=y^n_k]\ind{\{ y^n_k \leq \theta_*h\}}
														=\varphi_1(y^n_k)h^2\ind{\{y^n_k \leq \theta_*h\}},
														$$
														with $\varphi_1$ such that
														$
														|\varphi_1(y)|\leq C_*^{2}.
														$
														If instead  $\theta_*h\leq y^n_k\leq \theta^*/h$, by using \eqref{proba} we get
														\begin{align*}
														(y^{n+1}_{k_u}-y^n_k)^2p_u(n,k)+ (y^{n+1}_{k_d}-y^n_k)^2p_d(n,k)
														%\\
														%=(y^{n+1}_{k_u}+y^{n+1}_{k_d}-2y^n_k)\mu_Y(y^n_k)h+(y^{n+1}_{k_u}-y^n_k)(y^n_k-y^{n+1}_{k_d})\\
														%&\quad
														%= \frac {\sigma^2} 2 \kappa(\theta-y^n_k)h^2+\Big( \sigma \sqrt{y^n_kh}+ \frac{\sigma^2}{4}\Big)\Big( \sigma \sqrt{y^n_kh}-\frac{\sigma^2}{4}\Big)
														=\sigma^2y^n_kh+\frac{\sigma^2}2\Big( \kappa(\theta-y^n_k)-\frac{\sigma^2}8 \Big)h^2.
														\end{align*}
														So,
														$$
														\E[( Y^h_{n+1}-Y^h_n)^2\mid Y^{h}_n=y^n_k]\ind{\{\theta_*h\leq y^n_k \leq \theta^*/h\}}
														=\big(\sigma^2y^n_kh+\varphi_2(y^n_k)h^2\big)\ind{\{\theta_*h\leq y^n_k \leq \theta^*/h\}},
														$$
														with $\varphi_2$ such that
														$
														|\varphi_2(y)|\leq \frac{\sigma^2}2\Big( \kappa(\theta+y)+\frac{\sigma^2}8 \Big).
														$
														By inserting, \eqref{momento2-gen} follows with $g_h$ satisfying
														$$
														|g_h(Y^h_n)|\leq c_1(1+Y^h_n)h^2+
														\E((Y_{n+1}^h-Y_n^h)^2+\sigma h Y_n^h\mid Y^h_n)\ind{\{Y^h_n\geq \theta^*/h\}},
														$$
														$c_1$ denoting a suitable constant. By Proposition \ref{moments-2} and the Markov inequality, \eqref{stima-g}  follows.	
														
														Finally, for \eqref{momento3-gen}, we write
														\begin{align*}
														&\E[( Y^h_{n+1}-Y^h_n)^3\mid Y^h_n=y^n_k]  
														=\E[( Y^h_{n+1}-Y^h_n)^3\mid Y^h_n=y^n_k]\ind{\{y^n_k \leq \theta_*h\}}\\
														&\quad +\E[( Y^h_{n+1}-Y^h_n)^3\mid Y^h_n=y^n_k]\ind{\{\theta_*h<y^n_k < \theta^*/h\}} 
														+\E[( Y^h_{n+1}-Y^h_n)^3\mid Y^h_n=y^n_k] \ind{\{y^n_k \geq \theta^*/h\}} .
														\end{align*}
														Now, if $y^n_k \leq \theta_*h$ then $|Y^h_{n+1}-y^n_k|^3\leq C_*^3 h^3$. If instead $\theta_*h< y^n_k < \theta^*/h$, by \eqref{proba} one obtains
														\begin{align*}
														(y^{n+1}_{k_u}-y^n_k)^3p_u(n,k)+ (y^{n+1}_{k_d}-y^n_k)^3p_d(n,k)
														=\mu_Y(y^n_k)h^{2}\Big(\sigma^{2}y^n_k+\frac{3\sigma^{4}}{16} \,h\Big)+\Big(\frac{\sigma^{4}}{2}\,y^n_k+\frac{\sigma^{4}}{16}\,h\Big)h^{2}.
														\end{align*} 
														Therefore,
														$$
														|j_h(Y^h_n)|\leq c_2 h^2(1+(Y^h_n)^2)+\E(|Y_{n+1}^h-Y_n^h|^3+\sigma h Y_n^h\mid Y^h_n)\ind{\{Y^h_n\geq \theta^*/h\}},
														$$
														$c_2$ denoting a suitable constant, and again by Proposition \ref{moments-2} and the Markov inequality, \eqref{stima-j} follows.	
													\end{proof}

													We are finally ready for the
													\begin{proof}[Proof of Theorem \ref{conv_T-CIR}]
														By Theorem 4.1 in \cite{A-MC} (or Corollary \ref{corollary-cir}), one has that if $f\in C^{4}_\pol(\R_+)$  then  $u\in C^{4}_{\pol,T}(\R_+)$ . Since Assumption $\mathcal{A}_1$ and $\mathcal{A}_2$ both hold,
														the statement  follows as an application of Theorem \ref{conv_T}.
													\end{proof}

								\section{Hybrid schemes for jump-diffusions and convergence rate   } \label{sect-hybrid}
								We now introduce a $m$-dimensional jump-diffusion $(X_t)_{t\in[0,T]}$ whose dynamics is given by coefficients depending on  the process $(Y_t)_{t\in[0,T]}$ discussed in Section \ref{sect-markovapprox}. More precisely, we consider  the stochastic system
								\begin{equation}\label{generalsystem}
								\begin{cases}
								dX_t= \mu_X(Y_t)dt+\sigma_X(Y_t)\, dB_t+\gamma_X(Y_t)dH_t,  \qquad & X_0 \in \R^m,\\
								dY_t=\mu_Y(Y_t)dt+\sigma_Y(Y_t)\,dW_t, \qquad &Y_0\in\D,
								\end{cases}
								\end{equation}
								where  $B$ is a $\ell_1$-dimensional Brownian motion and  $H$ is a $\ell_2$- dimensional compound Poisson process with intensity  $\lambda$ and i.i.d. jumps  $\{J_k\}_k$, that is  
								\begin{equation}\label{H2}
								H_t=\sum_{k=1}^{K_t} J_k,
								\end{equation}
								$K$ denoting a Poisson process with intensity $\lambda$.  We assume that the Poisson process $K$, the jump amplitudes $\{J_k\}_k$ and the Brownian motions $B$ and $W$ are independent. Moreover, we ask that $J_1$ has a density $p_{J_1}$, so that the L\'evy measure associated with $H$  has a density as well:
								\begin{equation*}
								\nu(dx)=\nu(x)dx=\lambda p_{J_1}(x)dx.
								\end{equation*}
								Hereafter, we denote by  $\mathcal{L}$  the infinitesimal generator associated with the diffusion pair $(X,Y)$, i.e.
							\begin{equation}\label{general-L}
							\begin{array}{rl}
							\L g(x,y)=&
							\frac 12 \mbox{Tr}(a(y)D^2_{x,y}g(x,y))+\mu(y) \cdot \nabla_{x,y}g(x,y) \\
							&\displaystyle
							+ \int (g(x+\gamma_X(y)\zeta,y)-g(x,y))\nu(d\zeta),
							\end{array}
							\end{equation}
								where $\mu(y)=(\mu_X(y),\mu_Y(y))^\star$ and $a(y)=\sigma\sigma^\star(y)$, where 
								$$
								\sigma(y)=\begin{pmatrix}
								\sigma_X(y)&0_{m\times d}\\
								0_{d\times m}&\sigma_Y(y)
								\end{pmatrix}.
								$$ Here, $D^2_{x,y}$ and $\nabla_{x,y}$ are respectively the Hessian and the gradient operator w.r.t. the space variables $x$ and $y$.
								We    assume that the coefficients of $X$ do not depend on the time variable just to simplify the notation, but  all the proofs in this chapter are still valid in the time-depending case under  non restrictive classical assumptions.  
								
								Let $(X^{t,x,y}_s,Y^{t,x}_s)_{s\in [t,T]}$ be the solution of \eqref{generalsystem}  with starting condition $(X_t,Y_t)=(x,y)$.	Hereafter, we fix $T>0$   and $f:\R^m\times \mathcal{D}\rightarrow \R$. We are interested in computing the quantity  $u(0,X_0,Y_0)$, where, as specified from time to time, $u$ is given by
								\begin{equation}\label{european}
								u(t,x,y)=\E\Big[f(X^{t,x,y}_T,Y^{t,y}_T)\Big], \qquad(t,x,y)\in[0,T]\times \R^m\times\D,
								\end{equation}
								or
								\begin{equation}\label{american}
								u(t,x,y)=\sup_{\tau\in\mathcal{T}_{t,T}}\E\Big[f(X^{t,x,y}_\tau,Y^{t,y}_\tau)\Big], \qquad(t,x,y)\in[0,T]\times \R^m\times\D,
								\end{equation}
								where $\mathcal{T}_{t,T}$ denotes the set of all stopping times taking values on $[t,T]$.
								
								This can be, in general,  a problem of interest in a large number of applications. Of course, the immediate application in this thesis is in the financial world, where   $X$ can represent the log-price (or a transformation of it) and $Y$ can be interpreted as a random source such as  a stochastic volatility and/or a stochastic  interest rate.  In this framework,  the function defined in \eqref{european}  is the  price value at time $t$ of a European option with maturity $T$ and  (discounted) payoff $f$, while the function $u$ as defined in \eqref{american} is the   value function of the corresponding American option. Therefore, from now on we will refer to the European case when $u$ is defined as in \eqref{european} and to the American case where $u$ is given by \eqref{american}.
													
We do not enter in specific assumptions but from now on, the following requests (1), (2) and (3) will be assumed to hold:
\begin{itemize}
	\item [(1)]
	there exists a unique weak solution of \eqref{generalsystem}  and $\P((X_t,Y_t)\in\R^m\times \mathcal{D}\ \forall t)=1$;
	\item[(2)]
	$\mu=(\mu_X,\mu_Y)^\star$ and $\sigma_X$ have polynomial growth; moreover, either $\gamma_X\equiv 0$ (no jumps) or there exists $\varepsilon>0$ such that  $\inf_{y\in\mathcal{D}}|\gamma_X(y)|\geq \varepsilon$;
	\item[(3)]
	the function $u$ in \eqref{european} solves the PIDE
	\begin{equation}\label{PDEB}
	\left\{
	\begin{array}{ll}
	\partial_tu(t,x,y)+\mathcal{L}u(t,x,y)=0, \quad &(t,x,y)\in [0,T)\times\R^m\times \D, \\
	u(T,x,y)=f(x,y),&\mbox{ in } \R^m\times\D ,
	\end{array}
	\right.
	\end{equation}
	$\mathcal{L}$ being given in \eqref{general-L}.
\end{itemize} 
													
													\subsection{The hybrid procedure}\label{sect-hybrid2}
													
													\subsubsection{The European case}	\label{sect-hybrid-e}
													
													Let $u$ be given in \eqref{european}.
													We study here the  computation of $u(0,X_0,Y_0)$ by a backward hybrid algorithm which  generalizes the procedure developed in \cite{bcz,bcz-hhw,bctz} and described in Chapter 3. Roughly speaking, one uses a Markov  chain in order to approximate the  process $Y$ and a different numerical procedure to handle the jump-diffusion component $X$. Let us briefly recall the main ideas and set up the approximation of $u$.

													We start from the representation of $u(t,x,y)$ at times $nh$, $h=T/N$ and $n=0,\ldots,N$,  by the usual (backward) dynamic programming principle: for $(x,y)\in\R^m\times\D$,
													\begin{equation} \label{backward2volta}
													\begin{cases}
													u(T,x,y)= f(x,y)\quad
													\mbox{and as } n=N-1,\ldots,0,\\
													u(nh,x,y) =  
													\E\Big[u\big((n+1)h, X_{(n+1)h}^{nh,x,y}, Y_{(n+1)h}^{nh,y}\big)
													\Big].
													\end{cases}
													\end{equation}
													So, the central issue is to have a good approximation of the expectations in \eqref{backward2volta}.
													
													As a first step, let $( Y^h_{n})_{n=0,\ldots,N}$  be the Markov chain discussed in Section \ref{sect-convergence} which approximates  $Y$. Of course, we assume that $( Y^h_{n})_{n=0,\ldots,N}$ is  independent of the Brownian motion $B$ and the compound Poisson process $H$ driving $X$ in \eqref{generalsystem}.  Then, at each step  $n=0,1,\ldots,N-1$, for every $y\in \mathcal Y^h_n$ we write
													\begin{align*}
													\E\Big[u\big((n+1)h,  X_{(n+1)h}^{nh,x,y},  Y_{(n+1)h}^{nh,y}\big)\Big] \approx 	\E\Big[u\big((n+1)h,  X_{(n+1)h}^{nh,x,y},  Y^h_{n+1}\big)\big|Y^h_n=y\Big].
													\end{align*}
													Recall that $\mathcal{Y}_n^h\subseteq \D$ is the state space of $Y^h_{n}$  and  that $\mathcal{Y}^h_0=\{Y_0\} $. 
													
													As a second step, we approximate  the component  $X$ on $[nh,(n+1)h]$ by freezing the coefficients in \eqref{generalsystem} at the observed position $Y^h_n=y$, that is, for $t\in[nh,(n+1)h]$,
													$$
													X_{t}^{nh,x,y}\stackrel{\mbox{\tiny law}}{\approx} \widehat{X}^{nh,x}_t(y)=x+\mu_X(y)(t-nh)+\sigma_X(y)\, (B_t-B_{nh})+\gamma_X(y)(H_t-H_{nh}).
													$$
													Therefore,	by using that the Markov chain, $B$ and $H$ are all independent, we write
													\begin{align*}
													\E\Big[u\big((n+1)h,  X_{(n+1)h}^{nh,x,y},  Y_{(n+1)h}^{nh,y}\big)\Big] 
													&\approx 
													\E\Big[u\big((n+1)h,  \widehat{X}^{nh,x}_{(n+1)h}(y),  Y^h_{n+1}\big)\big|Y^h_n=y\Big]\\
													&=\E\big[ \phi(  Y^{h}_{n+1};x,y )\big|Y^h_n=y\big],
													\end{align*}
													where
													\begin{equation}
													\label{phi}
													\phi(\zeta;x,y)= \E\big[u((n+1)h, \widehat{X}^{nh,x}_{(n+1)h}(y),\zeta)\big].
													\end{equation}
													From the Feynman-Kac formula, one gets $\phi(\zeta;x,y)=v(nh,x;y,\zeta)$, where $(t,x)\mapsto v(t,x;y,\zeta)$ is the solution at time $nh$ of the parabolic PIDE Cauchy problem
													\begin{equation}\label{PDE-barug-h}
													\begin{array}{ll}
													\displaystyle
													\partial_t v+\L^{(y)}v=0, \qquad & \mbox{in } [nh,(n+1)h)\times \R^m,\smallskip\\
													\displaystyle
													v((n+1)h,x;y,\zeta)
													=u((n+1)h,x,\zeta), & x\in \R^m,
													\end{array}
													\end{equation}
													where  $\L^{(y)}$ is the integro-differential operator acting on the functions $g=g(x)$ given by
												\begin{equation}
												\label{genh}
												\L^{(y)}g(x)=\mu_X(y)\cdot\nabla_xg(x) +\frac 12\mbox{ Tr}(a_X(y)D^2_{x}g(x))+ \int \big(g(x+\gamma_X(y)\zeta)-g(x)\big)	\nu(\zeta)d\zeta
												\end{equation}
													Here $a_X(y)=\sigma_X(y)\sigma_X^\star(y)$, while
													$\nabla_x$ and $D^2_x$ are  the $m$ dimensional gradient vector and  the Hessian matrix with respect to  the $x$ variable  respectively. Recall that here $y$ is just a parameter and that for each fixed $y\in \D$, $\L^{(y)}$ has constant coefficients.
													
													We consider now a numerical solution of the PIDE \eqref{PDE-barug-h}. Let $\dx=(\dx_1,\dots,\dx_m)$ denote a fixed spatial step and set $\mathcal{X}$ denote a grid on $\R^m$ given by $\mathcal{X}=\{x\,:\,x=((X_0)_1+i_1\Delta x_1,\dots, (X_0)_m+i_m\Delta x_m), (i_1,\ldots,i_m)\in \Z^m\}$. For $y\in\mathcal{D}$,  let $\Pi^h_{\dx}(y)$ be a linear operator (acting on suitable functions on $\mathcal{X}$) which gives 
													the approximating solution to the PIDE \eqref{PDE-barug-h} at time $nh$.  Then we get the numerical approximation
													\begin{align*}
													\E\Big[u\big((n+1)h,  X_{(n+1)h}^{nh,x,y},  Y_{(n+1)h}^{nh,y}\big)\Big] 
													&\approx 
													\E\Big[\Pi^h_{\dx}(y)u\big((n+1)h, \cdot,  Y^h_{n+1}\big)(x)\big|Y^h_n=y\Big],\quad x\in\mathcal{X}.
													\end{align*}
													Therefore, by inserting  in \eqref{backward2volta}, the hybrid numerical procedure works as follows:
													the function $x\mapsto u(0,x,Y_0)$, $x\in\mathcal{X}$,  is approximated by $u^h_0(x,Y_0)$ backwardly defined as
													\begin{equation}\label{backward-ter0}
													\begin{cases}
													u^h_N(x,y)= f(x,y),\quad \mbox{$(x,y)\in \mathcal{X}\times\mathcal{Y}^h_N$},
													\quad \mbox{and as $n=N-1,\ldots,0$:}\\
													\displaystyle    u^h_n(x,y) =  \E[
													\Pi^h_{\dx}(y) u^h_{n+1}( \cdot,  Y^h_{n+1})(x)\mid Y^h_n=y],\quad \mbox{$(x,y)\in \mathcal{X}\times\mathcal{Y}^h_n$.}
													\end{cases}
													\end{equation}
													\subsubsection{The American case}\label{sect-hybrid-a}
													Let us now consider the function $u$ defined in \eqref{american}. Again, we want an approximation of the quantity $u(0,X_0,Y_0)$.
													In practice, at times $nh$, the function $u$ is approximated by the function $\tilde u^h_n$  defined through the backward programming dynamic principle, that is,
													\begin{equation} \label{backward2}
													\begin{cases}
													\tilde u^h_N(x,y)= f(x,y)\quad
													\mbox{and as } n=N-1,\ldots,0\\
													\tilde u^h_n(x,y) =  
													\max\Big\{f(x,y), \E\Big[\tilde u^h_{n+1}\big(X_{(n+1)h}^{nh,x,y}, Y_{(n+1)h}^{nh,y}\big)
													\Big]\Big\}.
													\end{cases}
													\end{equation}
													In financial terms,  $\tilde u^h_0$  corresponds to approximate the original continuous time American option price at $t=0$ by the price of an option  which can be exercised only at the discrete times $nh$, $n=0,\dots,N$ (Bermudean option).
													
													Now, at each step of  \eqref{backward2},  we can use the procedure described in Section \ref{sect-hybrid-e} in order to compute the conditional expectations therein. Therefore, the  hybrid numerical procedure becomes: for $n=0,1,\ldots,N$ and $(x,y)\in \mathcal{X}\times\mathcal{Y}^h_n$,  $\tilde u^h_n(x,y)$ is approximated by $ u^h_n(x,y)$  defined  as
													\begin{equation}\label{backward-ter0-bis}
													\begin{cases}
													u^h_N(x,y)= f(x,y),
													\quad \mbox{and as $n=N-1,\ldots,0$:}\\
													\displaystyle    u^h_n(x,y) =\max\Big\{ f(x,y), \E[
													\Pi^h_{\dx}(y) u^h_{n+1}( \cdot, \bar Y^{nh,y}_{(n+1)h})(x)] \Big\} .
													\end{cases}
													\end{equation}
													
													\subsubsection{The general hybrid procedure}
													As we have done in Chapter 3, it is useful to  put together  in a unique formulation the numerical procedures  described respectively  in Section \ref{sect-hybrid-e} for the European case and in Section \ref{sect-hybrid-a} for the American case.
													In both cases we  have to consider  at time $nh$ the function $\tilde u^h_n$  defined as
													\begin{equation} \label{backward3}
													\begin{cases}
													\tilde u^h_N(x,y)= f(x,y)\quad
													\mbox{and as } n=N-1,\ldots,0\\
													\tilde u^h_n(x,y) =  
													\max\Big\{g(x,y), \E\Big[\tilde u^h_{n+1}\big(X_{(n+1)h}^{nh,x,y}, Y_{(n+1)h}^{nh,y}\big)
													\Big]\Big\},
													\end{cases}
													\end{equation}
													where 
													$$
													g(x,y)=\begin{cases}
													0, \qquad\qquad\qquad&\mbox{ in the European case};\\
													f(x,y),&\mbox{ in the American case}.
													\end{cases}
													$$
													We stress that, in the European case,  the function $\tilde u^h_n$   coincides with  the function $u$ defined in \eqref{european} at time $nh$, while,  in the American case,  it is the Bermudean approximation of the (continuous monitored) American option value given in \eqref{backward2}.

													Then, for $n=0,1,\ldots,N$ and $(x,y)\in \mathcal{X}\times\mathcal{Y}^h_n$,  we approximate the function $\tilde u^h_n$ by the function  $ u^h_n$ defined as
													\begin{equation}\label{backward-ter0-tris}
													\begin{cases}
													u^h_N(x,y)= f(x,y),
													\quad \mbox{and as $n=N-1,\ldots,0$:}\\
													\displaystyle    u^h_n(x,y) =\max\left\{ g(x,y), \E\left[
													\Pi^h_{\dx}(y) u^h_{n+1}( \cdot, \bar Y^{nh,y}_{(n+1)h})(x)\right] \right\} .
													\end{cases}
													\end{equation}

													%	Therefore it is clear that, in the European case, the function $u_h$  is a direct approximation of the function $u$  in \eqref{european}, while, in the American case, it approximates the Bermudean option value.

													Our aim is to study the speed of convergence of the scheme \eqref{backward-ter0-tris}  that is, we give a quantitative estimate for 
													$$
													|\tilde u^h_0(x,y)-u^h_0(x,y)|,\qquad (x,y)\in \mathcal X\times \mathcal{Y}^h_0.
													$$ 
													
													As regards the American case, we recognize two types of error. The first one is the error induced by the approximation of the function $u(0,\cdot)$ in \eqref{american} with the function $\tilde u^h_0(\cdot)$ in the backward programming principle \eqref{backward2}. In the standard hypotheses on the model, that is, for sublinear and Lipschitz continuous diffusion coefficients and standard semiconvex payoff function, this error is known to be of the first order  in $h$ (we refer, for example,  to Theorem 2 in  \cite{BP}). The degenerate models such as the Heston model do not satisfy such requests, so we might just argue a first order error in time. 
													The second type of error is the one related to the approximation of $\tilde u^h_0$ with the function $u^h_0$ defined in \eqref{backward-ter0-bis}. Here,  we focus on studying the latter one.
													
													\subsection{Convergence speed of the hybrid scheme} 
													\label{sect-convergence}
													The idea is to follow the hybrid nature of the procedure by using numerical techniques, that is, an analysis of the stability and of the consistency of the method. This will be done in a sense that allows us to exploit the probabilistic properties of the Markov chain approximating the  process $Y$.
													
													We introduce the following assumption on the linear operator $\Pi^h_{\dx}(y)$ in \eqref{backward-ter0} (recall the notation $l_p(\mathcal{X})$ in Section \ref{sect-notation}).

													\medskip
													
													\noindent
													\textbf{Assumption $\mathcal{B}(p,c,\mathcal{E})$.} 
													\textit{Let $p\in [1,\infty]$, $c=c(y)\geq 0$, $y\in\D$ and $\mathcal{E}= \mathcal{E}(h,\dx)\geq 0$ such that 
														$\lim_{(h,\dx)\rightarrow 0}\mathcal{E}(h,\dx)=0.$
														We say that the 
														linear operator $\Pi^h_{\dx}(y):l_p(\mathcal{X})\to l_p(\mathcal{X})$, $y\in\D$, satisfies Assumption $\mathcal{B}(p,c,\mathcal E)$ if 
														\begin{equation}\label{stab}
														| \Pi^h_{\dx}(y)|_p\leq1+c(y)h
														\end{equation}
														and, $\tilde u^h_n$ being defined in \eqref{backward3}, for every $n=0,\dots, N-1$, one has
														\begin{equation}\label{cons}
														\E\Big[	\Pi^h_{\dx}(Y^h_n) \tilde u^{n+1}_h(\cdot,Y^h_{n+1})(x)\,\big|\, Y^h_n=y\Big]= \E[\tilde u^h_n(X^{nh,x,y},Y^{nh,y}_n)]+ \mathcal{R}_n^h(x,y),
														\end{equation}
														where the remainder $\mathcal{R}_n^h(x ,y)$, $(x,y)\in \mathcal X\times \mathcal Y^h_n$  satisfies the following property: there exist $\bar h<1$ and  $C>0$ such that for every $n\in \N$, $h<\bar h$,  $|\dx|<1$ and $n\leq N=\lfloor T/h\rfloor$ one has		
														\begin{equation}\label{ass_errore}
														\begin{split}
														\Big\|e^{\sum_{l=1}^{n}\, c(Y^h_l)h} |\mathcal{R}_n^h(\cdot,Y^h_n)|_p\Big\|_p&\leq C h\mathcal{E}(h,\dx), \qquad \mbox{if } p\in [1,\infty),\\
														\Big\|e^{\sum_{l=1}^{n}\, c(Y^h_l)h} |\mathcal{R}_n^h(\cdot,Y^h_n)|_\infty\Big\|_1& \leq C h\mathcal{E}(h,\dx), \qquad \mbox{if } p=\infty.
														\end{split}
														\end{equation}
														%\begin{equation}\label{ass_errore}
														%\begin{split}
														%\left(	\E\big[e^{\sum_{l=1}^{n}p\, c(Y^h_l)h} |\mathcal{R}_n^h(\cdot,Y^h_n)|^p_{p}\big]\right)^{ 1 /p} &\leq C h\mathcal{E}(h,\dx), \qquad \mbox{if } p\in [1,\infty),\\
														%\E\big[e^{\sum_{l=1}^{n}\, c(Y^h_l)h} |\mathcal{R}_n^h(\cdot,Y^h_n)|_\infty\big]& \leq C h\mathcal{E}(h,\dx), \qquad \mbox{if } p=\infty.
														%\end{split}
														%\end{equation}
													}
													
													Assumption $\mathcal{B}(p,c,\mathcal{E})$ is inspired by the Lax-Richtmeyer's convergence theorem \cite{Lax}. In fact, recall that
													%		 the  numerical procedure \eqref{backward-ter0} aims to solve the multidimensional equation 
													%		$$
													%		\partial_tu(t,x,y)+\mathcal{L}u(t,x,y)=0 \quad (t,x,y)\in [0,T)\times\R^m\times \D.
													%		$$
													%		Being dependent on $y$, the coefficients of the operator $\L$ (see \eqref{general-L})  are not constant as required by the Lax-Richtmeyer's result. \textcolor{red}{Il teorema di L-R vuole i coefficienti costanti oppure limitati????} But
													at each time step $n$, the hybrid scheme isolates the component $y$ and applies the discrete operator $\Pi^h_{\dx}(y)$ for solving (one step in time) the PIDE 
													$$	
													\partial_t v(t,x)+\L^{(y)}  v(t,x)=0, \qquad (t,x)\in [nh,(n+1)h)\times \R^m.
													$$ 
													Here, $y$ is just a parameter (the current position of the Markov chain), so the coefficients of $\L^{(y)}$	(see \eqref{genh}) are indeed constant. That's  why the Lax-Richtmeyer technique can be adapted, as it follows in the next result.

													\begin{theorem}\label{convergencebates}
														Assume that $\Pi^h_{\dx}(y)$, $y\in\D$, satisfies Assumption $\mathcal{B}(p,c,\mathcal{E})$. Let $\tilde u^h_n$ be the function  defined  in\eqref{backward3}   and $u^h_n$ be the approximation through the  scheme \eqref{backward-ter0-tris}.
														Then, there exist $\bar h\in(0,1)$ and  $C>0$ such that for every $h<\bar h$ and $\dx<1$  one has
														\begin{equation}
														| \tilde u^h_0(\cdot,Y_0)-u^h_0(\cdot,Y_0)|_{p} \leq CT\mathcal{E}(h,\dx). 
														\end{equation}
													\end{theorem}
													\begin{proof}
														Set $\mathrm{err}^{h}_n (\cdot,Y^h_n)=\tilde u^h_n(\cdot,Y^h_n)-u^h_n(\cdot,Y^h_n)$. By using the relation  $|\max\{ (a,b)  \}-\max\{ (a',b')  \}|\leq \max\{|a-a'|,|b-b'|\}$ we get
														\begin{align*}
														|	\mathrm{err}^{h}_n (x,Y^h_n)|
														%& \Big|   \max\Big\{f(\cdot,Y^h_n), \E		 \Big[\tilde u^h_{n+1}(X^{nh,x,y}_{(n+1)h},Y^{nh,y}_{(n+1)h})\Big]\Big|_{y=Y^h_n}\Big\}\\
														%&-\max\Big\{f(\cdot,Y^h_n), \E		 \Big[\Pi^h_{\dx}(Y^h_n)u^h_{n+1}(\cdot,Y^h_{n+1})(x)\big|Y^h_n\Big]\Big\}  \Big|\\
														&\leq \left|   \E		 \left[\tilde u^h_{n+1}(X^{nh,x,y}_{n+1},Y^{nh,y}_{(n+1)h})\right]\Big |_{y=Y^h_n}-  \E		 \left[\Pi^h_{\dx}(Y^h_n)u^h_{n+1}(\cdot,Y^h_{n+1})(x)\big|Y^h_n\right]   \right|
														\\& \leq \left| \E[\Pi^h_{\dx}(Y^h_n)\mathrm{err}^{h}_{n+1}(\cdot,Y^h_{n+1})(x) |Y^h_n ] \right|+|\mathcal{R}_n^h(x ,Y^h_n)|,
														\end{align*}
														in which we have used \eqref{cons}. Since $\mathrm{err}_{n}^h(x_i,Y_N^h)=0$, by iterating one gets
														$$
														|	\mathrm{err}^h_0(\cdot,Y_0)|\leq  \sum_{n=0}^{N-1}\E\left[ \left|\left(\prod_{l=0}^{n-1} \Pi^h_{\dx}(Y^h_l)\right) \mathcal{R}_n^h(\cdot,Y^h_n)\right|\right],
														$$
														%			By using \eqref{stab-am} and \eqref{ass_errore-am}, the assertion follows with the same arguments in the proof of Proposition \ref{convergencebates}. 
														%			
														%			
														%			Let  $\mathrm{err}^{h}_n (\cdot,Y^h_n)$ be the  error at time $nh$, defined by
														%			$$
														%			\mathrm{err}^{h}_n(\cdot,Y^h_n)=\tilde u^h_n(\cdot,Y^h_n)-u^h_n(\cdot,Y^h_n).
														%			$$
														%			Note that $\mathrm{err}_{N}^h(\cdot,Y_N^h)=0$, because the final condition is the same.
														%			
														%			 By \eqref{cons} and \eqref{backward-ter0}, we can write
														%			\begin{align*}
														%			\mathrm{err}^h_n (\cdot,Y^h_n)&=  \E[\Pi^h_{\dx}(Y^h_n)\mathrm{err}^h_{n+1}(\cdot,Y^h_{n+1}) |Y^h_n ] -\mathcal{R}_n^h(\cdot ,Y^h_n)
														%			\end{align*}
														%			and, by iterating,
														%			$$
														%			\mathrm{err}^h_0(\cdot,Y_0)= -\sum_{n=0}^{N-1}\E\Big[ \Big(\prod_{l=0}^{n-1} \Pi^h_{\dx}(Y^h_l))\Big) \mathcal{R}_n^h(\cdot,Y^h_n)\Big],
														%			$$
														in which we use the convention $\displaystyle \prod_{l=0}^{-1}(\cdot)=\Id$. We use now \eqref{ass_errore}. For $p\neq \infty$,
														\begin{align*}
														|\mathrm{err}_h^0(\cdot,Y_0)|_{p}
														&\leq \sum_{n=0}^{N-1}\Big|\E\Big[ \Big(\prod_{l=0}^{n-1} \Pi^h_{\dx}(Y^h_l)\Big) \mathcal{R}_n^h(\cdot,Y^h_n)\Big]\Big|_{p}
														\leq \sum_{n=0}^{N-1}\E\Big[ \Big|\Big(\prod_{l=0}^{n-1} \Pi^h_{\dx}(Y^h_l)\Big) \mathcal{R}_n^h(\cdot,Y^h_n)\Big|^p_{p}\Big]^{1/p}\\
														&\leq  \sum_{n=0}^{N-1}	\left(	\E\big[e^{\sum_{l=1}^{n}pc(Y^h_l)h} |\mathcal{R}_n^h(\cdot,Y^h_n)|^p_{p}\big]\right)^{\frac 1 p} 
														\leq \sum_{n=0}^{N-1}hC\mathcal{E}(h,\dx)
														\leq TC\mathcal{E}(h,\dx).
														\end{align*}
														The case  $p=\infty$ follows the same lines.
													\end{proof}

													\begin{remark}
														In  Assumption $\mathcal{B}(p,c,\mathcal E)$ we have required that the constant $C$ and the function $\mathcal E$ in \eqref{ass_errore}  do  not depend on $h$ and $n$. A closer look at the proof of Theorem \ref{convergencebates} shows that this assumption can be relaxed. In fact,  we can replace $C$ and $\mathcal{E}$ in \eqref{ass_errore} by $C_{h,n}$ and $\mathcal E_{h,n}$  which depend on $h$ and $n$ but such that $\lim_{(h,\dx)\rightarrow (0,0)} 
														\sum_{n=0}^{N-1}hC_{h,n}\mathcal{E}_{h,n}(h,\dx)
														=0$. However, in this case we  do not get  information about the rate of convergence of the method.
													\end{remark}
													
													\subsection{An example: finite difference schemes}\label{sect-finitedifference}	
													
													We specify here some settings ensuring  that the assumptions of Theorem \ref{convergencebates} are satisfied. In particular, we choose the operator $\Pi^h_{\dx}(y)$ in \eqref{backward-ter0} by means of two different finite difference schemes: the first one is a generalization of the procedure described in Chapter 3 and allows us to study the convergence in the $l_2$-norm, while the second one works $l_\infty$.  For the sake of readability, we consider the case $m=d=\ell=\ell_1=\ell_2=1$.

%%													As regards the Markov chain $(Y^h_n)_{n=0,\dots,N}$, in addition to  Assumption $\mathcal{A}_1$ and $\mathcal{A}_2$ (see Section \ref{sect-markovapprox}), we will need also the following:
%%													
%%													\medskip
%%													
%%													\noindent \textbf{Assumption $\mathcal{A}_3(g)$}
%%													\textit{ Let $g=g(y)\geq 0$, $y\in\D$.
%%														$(Y^h_n)_{n=0,\dots,N}$    satisfies Assumption $\mathcal{A}_3(g)$ if 
%%														$$
%%														\E\left[e^{\sum_{l=1}^Ng(Y^h_l)}\right]<\infty.
%%														$$		}
%%													
%%													Moreover, we assume hereafter that the L\'evy measure $\nu$ satisfies the following property: there exists $c_\nu>0$ such that for every $\Delta x<1$ one has
%%													\begin{equation}\label{cnu}
%%													\sum_{l\in\Z}\nu(l\Delta x)\Delta x\leq \lambda c_\nu,
%%													\end{equation}
%%													where  $\lambda$  is the intensity of the Poisson process $K$ in the definition of the coumpound Poisson process $H$ in \eqref{H}.
%%												As regards the Markov chain $(Y^h_n)_{n=0,\dots,N}$, in addition to  Assumption $\mathcal{A}_1$ and $\mathcal{A}_2$ (see Section \ref{sect-markovapprox}), we will need also the following:

The request on $\gamma$ made at the beginning of Section \ref{sect-hybrid}, that is either $\gamma_X\equiv 0$ or $\inf_{y\in\mathcal{D}}|\gamma_X(y)|\geq \varepsilon>0$ now comes on.  Set
%	For $y\in \mathcal{D}$, set
\begin{equation}\label{nuy}
\nu_y(x)=\left\{
\begin{array}{ll}
0 &\mbox{if } \gamma_X\equiv 0,\\
\frac 1{|\gamma_X(y)|}\nu(\frac{x}{\gamma_X(y)}) & \mbox{otherwise},
\end{array}
\right.
\quad y\in\mathcal{D},
\end{equation}
$\nu$ denoting the  density of the L\'evy measure.
	\begin{proposition}\label{prop-nu}
	If $\frac{\nu'}{\nu},\frac{\nu''}{\nu}\in L^1(\R,d\nu)$, there exists $c_\nu\geq 0$ such that
	\begin{equation}\label{cnu}
	\sum_{l\in\Z}\nu_y(l\dx)\dx\leq \lambda c_\nu,\quad \forall y\in\mathcal{D}.
	\end{equation}
	%In particular, $c_\nu=0$ when $\gamma_X\equiv 0$.	
\end{proposition}

\begin{proof}
	The proof follows from the technical \ref{lemma-poisson} below: if $\gamma_X$ is non null, $(i)$ applied to $g(x)=\nu_y(x)$ gives
	$
	\sum_{l\in \Z}\nu_y(l\dx)\dx\leq \int_\R \nu(x)dx+\frac{|\dx|^2}{12|\gamma_X(y)|^2}\int_\R |\nu''(x)|dx.
	$
	Now we use the ``uniformity'' condition $\inf_{y\in\mathcal{D}}|\gamma_X(y)|\geq \varepsilon$. 
\end{proof}

\begin{lemma}\label{lemma-poisson} Let $g\in C^2(\R)$.
	
	\noindent
	$(i)$ 
	If $g,g',g''\in L^1(\R,dx)$ then 
	\begin{equation}\label{poisson1}
	\Big|\sum_{l\in \Z}g(l\Delta x )\Delta x -\int_\R g(x)dx\Big|\leq \frac {\Delta x ^2}{12} \,|g''|_{L^1(\R,dx)}.
	\end{equation}
	
	\noindent
	$(ii)$ 
	If $g,g', g''\in L^2(\R,dx)$ then
	\begin{equation}\label{poisson2}
	%\sum_{l\in \Z}g^2(l\Delta x )\Delta x 
	|g|_2^2
	\leq |g|_{L^2(\R,dx)}^2+
	\frac {\Delta x ^2}{6} \,\big(|g'|_{L^2(\R,dx)}^2+|g|_{L^2(\R,dx)}\times|g''|_{L^2(\R,dx)}\big).
	\end{equation}
\end{lemma}	

\begin{proof}
	We first recall the Poisson summation formula. It is worldwide famous but is usually written on the Schwartz space, we use here the following version (Section \ref{app-poisson} in the appendix contains the detailed proof):  if $\varphi\in C^2(\R)$ with $\varphi,\varphi',\varphi''\in L^1(\R,dx)$ then
	\begin{equation}\label{poisson-formula}
	\sum_{n\in \Z}\varphi(n)=\int_{\R} \varphi(x)dx+\sum_{n\in\Z,n\neq 0}\int_{\R}\varphi(x)e^{-2\pi \ii n x}dx.
	\end{equation}
	$(i)$ 
	We apply \eqref{poisson-formula} to $\varphi(x)=g(x\Delta x )$. So,
	$$
	\begin{array}{l}
	\sum_{n\in \Z}g(n\Delta x )\Delta x -\int_{\R} g(x)dx=
	\sum_{n\in\Z,n\neq 0} %e^{2\pi \ii n x_0/\dx}
	\int_{\R}g(x)e^{-2\pi \ii n x/\Delta x  }dx\\
	\ \ =\sum_{n\in\Z,n\neq 0} %e^{2\pi \ii n x_0/\dx}
	\frac{\Delta x ^2}{(2\pi \ii n)^2}\int_{\R}g''(x)e^{-2\pi \ii n x/\Delta x  } dx,
	\end{array}
	$$
	%		Now, if $g',g''\in L^1(\R,dx)$ then, by integration by parts, 
	%		$$
	%		\int_{\R}g(x)e^{-2\pi \ii n x/\dx }
	%		=\frac{\dx^2}{(2\pi \ii n)^2}\int_{\R}g''(x)e^{-2\pi \ii n x/\dx }.
	%		$$
	%		This gives
	%		$$
	%		\Big|\sum_{n\in \Z}g(x_n)\dx-\int_{\R} g(x)dx\Big|\leq \frac {\dx^2}{2\pi^2}\sum_{n\geq 1}\frac{1}{n^2}|g''|_{L^1(\R,dx)},
	%		$$
	the latter inequality coming from the integration by parts formula. The statement holds by recalling that $\sum_{n\geq 1}\frac{1}{n^2}=\frac{\pi^2}{6}$.
	
	\smallskip
	
	\noindent
	$(ii)$
	By applying \eqref{poisson1} to the function $x\mapsto g^2(X_0+x)$, \eqref{poisson2} immediately follows.
	%Note that if $g,g',g''\in L^2(\R,dx)$ then $g^2$ and its derivatives up to order 2 belong to $L^1(\R,dx)$. Moreover, $|(g^2)''|_{L^1}\leq 2|g'|_{L^2}^2+2|g|_{L^2}|g''|_{L^2}$, and \eqref{poisson2} immediately follows.
\end{proof}

Statement $(ii)$ in \ref{lemma-poisson}  will be used to handle the error in $l_2$-norm coming from suitable Taylor's expansions and from the quadrature approximation.  

	\subsubsection{Convergence in $l_2$-norm}\label{sect-l2}
													We study here a  hybrid procedure which generalizes the one introduced in \cite{bctz} and described in Chapter \ref{chapter-art3} for the Bates model.
													For $y\in\D$, $\Pi^h_{\dx}(y)$ gives the numerical solution on $\mathcal{X}=\{x_i=X_0+i\Delta x\}_{i\in \Z}$ a time $nh$ to the PIDE \eqref{PDE-barug-h}, the operator $\L^{(y)}$ therein being given in \eqref{genh}. It is clear that the solution $v$ of \eqref{PDE-barug-h} depends on $y$ and $\zeta$ as well, but these are just parameters (and not variables of the PIDE), so for simplicity we drop here such dependence. 
													We split the operator $	\L^{(y)}v(t,x)=\L_{\mbox{{\tiny diff}}}^{(y)}v(t,x)+\L_{\mbox{{\tiny int}}}^{(y)}v(t,x)$ in its differential and integral part:
													\begin{equation}
													\label{L_d}
													\L_{\mbox{{\tiny diff}}}^{(y)}v(t,x)=\mu_X(y)\partial_x v(t,x) +\frac 12\sigma_X^2(y)\partial^2_x v(t,x)
													\end{equation}
													and by using the change of variable ($\nu_y$ being defined in \eqref{nuy}),
													\begin{equation}\label{L_i}
													\L_{\mbox{{\tiny int}}}^{(y)}v(t,x)
													=\int \big(v(t,x+\gamma_X(y)z)-v(t,x)\big)	\nu(z)dz\\
													=\int \big(v(t,x+\zeta)-v(t,x)\big)	\nu_y(\zeta)d\zeta.
													\end{equation} 
													We use the central finite difference scheme to solve $\L^{(y)}_{\mbox{{\tiny diff}}}v$ and the trapezoidal rule in order to approximate the integral term $\L^{(y)}_{\mbox{{\tiny int}}}v$.
													Applying an implicit-explicit method in time, we obtain an approximating solution $v^{n}=(v^n_j)_{j\in\Z}$ to the PIDE \eqref{PDE-barug-h} given by the solution of the linear equation
													\begin{equation}\label{equaz}
													A^h_{\dx}(y)v^n=B^h_{\dx}(y)v^{n+1}
													\end{equation}
													(recall that $v^{n+1}$ is known).	Here $A^h_{\dx}(y)$ is the linear operator given by 
													\begin{equation}\label{A}
													(A^h_{\dx})_{ij}(y)=\begin{cases}
													\alpha^h_{\dx}(y)-\beta^h_{\dx}(y),\qquad &\mbox{ if }i=j+1,\\
													1+2\beta^h_{\dx}(y),\qquad &\mbox{ if }i=j,\\
													-\alpha^h_{\dx}(y)-\beta^h_{\dx}(y),\qquad &\mbox{ if }i=j-1,\\
													0, &\mbox{ if }|i-j|>1,
													\end{cases}
													\end{equation} 
													with
													\begin{equation}\label{alpha-beta}
													\alpha^h_{\dx}(y)=\frac{h}{2\dx} \mu_X(y), \qquad \beta^h_{\dx}(y)=\frac{h}{2\dx^2}\sigma_X^2(y),
													\end{equation}
													and $B^h_{\dx}(y)$ is the linear operator defined as
													\begin{equation}\label{B}
													(B^h_{\dx})_{ij}(y)=\begin{cases}
													h\nu_y((j-i)\dx)\dx & \mbox{ if }j\neq i,\\
													1+h\Big(\nu_y(0)\dx-\sum_{l\in\Z} \nu_y(l\dx)\dx\Big)	&\mbox{ if }i=j .
													%	\\
													%	h\frac{\dx}{|\gamma_X(y)|}\nu(\frac{(j-i)\dx}{\gamma_X(y)}),\qquad &\mbox{ if }j\neq i\\
													%	1+h\frac{\dx} {|\gamma_X(y)|}\Big(\nu(0)-\sum_{l\in\Z} \nu(\frac{l\dx}{\gamma_X(y)})\Big)	&\mbox{ if }i=j 
													\end{cases}						\end{equation}
Then we have
\begin{lemma}\label{norma}
	For every  $y\in \D$, the operator $A^h_{\dx}(y):l_2(\mathcal{X})\rightarrow l_2(\mathcal{X})$ is  invertible and $\sup_{y\in\mathcal{D}}| (A^h_{\dx})^{-1}(y)|_2\leq1$. And if $\frac{\nu'}{\nu},\frac{\nu''}{\nu}\in L^1(\R,d\nu)$ then $\sup_{y\in\mathcal{D}}|B^h_{\dx}(y)|_2$ $\leq 1+2\lambda c_\nu h$, $c_\nu$ being defined in \ref{cnu}.
\end{lemma}
\begin{proof}
	Let $ \mathcal{F}\,:\,  l_2(\mathcal{X})\to L^2([0,2\pi), dx)$ denote the Fourier transform: 	$$ \mathcal{F}(\varphi) (\theta) = \frac{\dx}{\sqrt{2\pi}}\sum_{
		j\in \Z} \varphi_j e^{-\ii j\Delta x\theta}, \qquad \qquad \theta\in [0,2\pi), \quad\varphi\in l_2(\mathcal{X}).$$ 
	
	Fix $y\in \D $ and $w \in l_2(\mathcal{X})$. $v\in  l_2(\mathcal{X})$ satisfies  $A^h_{\dx}(y)v=w$ iff 
	$ \mathcal{F}(A^h_{\dx}(y)v)= \mathcal{F}(w)$. Straightforward computations give (see e.g. the proof of Theorem 5.1 in \cite{bctz})
	$ \mathcal{F}(A^h_{\dx}(y)v)=\psi \times   \mathcal{F}(v)$, with $\psi(\theta)=(\alpha^h_{\dx}(y)-\beta^h_{\dx}(y))e^{-\ii\theta\dx}+1+2\beta^h_{\dx}(y)-(\alpha^h_{\dx}(y)+\beta^h_{\dx}(y))e^{\ii\theta\dx}$. 
	It can be easily seen that $|\psi(\theta)|\geq 1+2\beta^h_{\dx}(y)(1-\cos(\theta\dx))$ $
	\geq 1$. Hence $ \mathcal{F}(v) = \mathcal{F} (w)/\psi\in L^2([0,2\pi),dx)$ and its inverse Fourier transform uniquely defines the solution $v\in l_2(\mathcal{X})$ to $A^h_{\Delta x}(y)v=w$.
	Thus $A^h_{\dx}$ is invertible. Moreover, we obtain %$|  \mathcal{F}(v)(\theta)|\leq | \mathcal{F} (w)(\theta)|$, so that
	$	|  \mathcal{F}(v)|_{L^2([0,2\pi),dx)}$ $\leq | \mathcal{F} (w)|_{L^2([0,2\pi),dx)}$.
	By the Parseval identity %$|\hat \varphi|_{L^2([0,2\pi),dx)}=|\varphi|_2$, 
	we get $| (A^h_{\dx})^{-1}(y)w|_2 \leq 	| w|_2$, so $| (A^h_{\dx})^{-1}(y)|_2\leq 1$.
	Finally, for $w\in l_2(\mathcal{X})$ straightforward computations give
	$$
	\mathcal{F}(B^h_{\dx}(y)w)(\theta)=\Big( 1+ h\dx\sum_l  \nu_y(l\dx) (e^{\ii l\theta} -  1)       \Big) \mathcal{F} (w) (\theta).
	$$
	Then, $| \mathcal{F}(B^h_{\dx}(y)w) |_{L^2([0,2\pi),dx)}\leq ( 1+ 2\lambda c_\nu h)| \mathcal{F}(w)|_{L^2([0,2\pi),dx)}$ because  \eqref{cnu} holds. By the Parseval relation, $|B^h_{\dx}(y)w|_{2}\leq ( 1+ 2\lambda c_\nu h)|w|_{2}$, which concludes the proof.
\end{proof}

													In the following we will use functions $v\in C^{p,q}_{\pol,[nh,(n+1)h]}(\R,\D)$ a.e. uniformly in $n$ and $h$. This means that $v\in C^{\lfloor q/2 \rfloor, q}([a,b),\R\times \mathcal{D} )$ a.e. and there exist $C,c>0$ independent of $n$ and $h$  such that
													$$
													\sup_{t\in [nh,(n+1)h)}|\partial^{k}_t\partial^{l'}_x\partial^{l}_yv(t,\cdot, y)|_{L^p(\R^m,dx)}\leq C(1+|y|^c), \quad  2k+|l'|+|l|\leq q.
													$$
													We can now state the convergence result.

													\begin{theorem}
														\label{conv-H}
														Let $\tilde u^h_n$ be defined in \eqref{backward3} and $u^h_n$ be given by \eqref{backward-ter0-tris} with the choice
														$$
														\Pi^h_{\dx}(y)=(A^h_{\dx})^{-1}B^h_{\dx}(y),
														$$
														$A^h_{\dx}(y)$ and $B^h_{\dx}(y)$ being given in \eqref{A} and \eqref{B} respectively.  Moreover, for  $n=0,\dots, N,$ consider the function 
														\begin{equation}\label{v^h_n}
														v^h_n(t,x,y)=\E\left[  \tilde u^h_{n+1}(X^{t,x,y}_{(n+1)h},Y^{t,y}_{(n+1)h})  \right], \qquad \qquad t\in[nh,(n+1)h].
														\end{equation}
														Assume that
														\begin{itemize}
															\item $\frac{\nu'}\nu,\frac{\nu''}\nu\in L^2(\R,d\nu)$;
															\item the Markov chain $(Y^h_n)_{n=0,\dots, N}$ satisfies assumptions $\mathcal{A}_1$ and $ \mathcal{A}_2$;
															\item $v^h_n\in  C^{2,6}_{\pol, [nh,(n+1)h]}( \R, \D)$ a.e. and uniformly in $n$ and $h$.
														\end{itemize}
														Then, there exist $\bar h,C>0$ such that for every $h<\bar h$ and $\dx<1$ one has
														\begin{equation}\label{conv-1}
														| \tilde u^h_0(\cdot,Y_0)-u^h_{0}(\cdot,Y_0)|_{2} \leq CT(h+\dx^2).
														\end{equation}
														
													\end{theorem}
													We stress that, from \eqref{conv-1}, the rate of convergence is of  the second order in space, because of the choice of a second order finite difference scheme, and of first order in time, as it is natural also for the presence of the approximating Markov chain $Y^h$ (see Theorem \ref{conv_T}).

	\begin{proof}%[Proof of   \cref{conv-H}]
		The result follows from   Theorem \ref{convergencebates} once we prove that Assumption $\mathcal{K}(2,2\lambda c_\nu,h+\dx^2)$ holds.
			First, Lemma \ref{norma} gives $|\Pi^h_{\dx}(y)|_2\leq |(A^h_{\dx})^{-1}(y)|_2|B^h_{\dx}(y)|_2\leq 1+2\lambda c_\nu h$, so \eqref{stab} holds with $c(y)=2\lambda c_\nu$. We prove now \eqref{ass_errore} with $p=2$ and $\mathcal{E}(h,\dx)=h+\dx^2$. We first note that \eqref{cons} equals to 
		\begin{equation}
		\label{mmm}
		\begin{array}{l}
		\E\big[B^h_{\dx}(\hat{Y}^h_n)v^n_h((n+1)h, \cdot, \hat{Y}^h_{n+1})(x)\mid \hat{Y}^h_n\big] \\ =A^h_{\dx}(\hat{Y}^h_n)v^n_h(nh,\cdot,\hat{Y}^h_n)(x)+A^h_{\dx}(\hat{Y}^h_n)\mathcal{R}^h_n(\cdot,\hat{Y}^h_n)(x).
		\end{array}
		\end{equation}
		
		\noindent\textbf{Step 1. Taylor expansion of the l.h.s. of \eqref{mmm}.} We set
		\begin{equation}\label{Bapplied}
		\begin{array}{l}
		I_1= B^h_{\dx}(\hat{Y}^h_n)v^n_h((n+1)h, \cdot, \hat{Y}^h_{n+1})(x_i) \\
		%		\sum_j(B^h_{\dx})_{ij}(\hat{Y}^h_n)
		%		v^n_h((n+1)h,x_j,\hat{Y}^h_{n+1})\\
		=v^n_h((n+1)h,x_i,\hat{Y}^h_{n+1}) \\
		+h\sum_l \nu_{\hat{Y}^h_n}(l\dx )\Big(v^n_h((n+1)h,x_i+l\dx,\hat{Y}^h_{n+1})-v^n_h((n+1)h,x_{i},\hat{Y}^h_{n+1})\Big)\dx.
		\end{array}
		\end{equation}
		In the first term of the above r.h.s. we apply several Taylor's expansion: of $t\mapsto v^n_h(t,x_i,\hat{Y}^h_{n+1}) $  around $nh$ up to  order 1,
		of $y\mapsto v^n_h(nh,x_i,y)$ around $\hat{Y}^h_n$ up to order 3 and of $y\mapsto \partial_tv^n_h(nh,x_i,y)$ around $\hat{Y}^h_n$ up to order 1. Rearranging the terms we obtain
		$$
		\begin{array}{l}
		v^n_h((n+1)h,x_i,\hat{Y}^h_{n+1})
		=v^n_h(nh,x_i,\hat{Y}^h_{n}) \\
		+ \partial_tv^n_h(nh,x_i,\hat{Y}^h_{n})h+\partial_yv^n_h(nh,x_i\hat{Y}^h_{n})(\hat{Y}^h_{n+1}-\hat{Y}^h_{n})
		+\frac 12 \partial_y^{2}v^n_h(nh,x_i,\hat{Y}^h_{n})(\hat{Y}^h_{n+1}-\hat{Y}^h_{n})^{2}\\
		+\partial_y\partial_tv^n_h(nh,x_i,\hat{Y}^h_{n})\,h(\hat{Y}^h_{n+1}-\hat{Y}^h_{n})
		+\frac 16\partial_y^3v^n_h(nh,x_i,\hat{Y}^h_{n})(\hat{Y}^h_{n+1}-\hat{Y}^h_{n})^3\\
		+R_{1}(n,h,x_i,\hat{Y}^h_{n},\hat{Y}^h_{n+1}),
		\end{array}
		$$
		where $R_1$  is given by
		\begin{equation}\label{R1}
		\begin{array}{ll}
		&	R_{1}(n,h,x_i,\hat{Y}^h_{n},\hat{Y}^h_{n+1})
		=h^2\int_0^1(1-\tau)\partial^2_tv^n_h(nh+\tau h,x_i,\hat{Y}^h_{n+1})d\tau\\
		&	\quad + \frac{(\hat{Y}^h_{n+1}-\hat{Y}^h_n)^4} 6 \int_0^1(1-\zeta)^3\partial^4_yv^n_h(nh,x_i,\hat{Y}^h_n+\zeta(\hat{Y}^h_{n+1}-\hat{Y}^h_n))d\zeta\\
		&	\quad +h(\hat{Y}^h_{n+1}-\hat{Y}^h_n)^2\int_0^1(1-\zeta)\partial_t\partial^2_yv^n_h(nh,x_i,\hat{Y}^h_n+\zeta(\hat{Y}^h_{n+1}-\hat{Y}^h_n))d\zeta.
		\end{array}
		\end{equation}

		For the second term in the r.h.s. of \eqref{Bapplied},  we stop the Taylor expansion of   $t\mapsto v^n_h((n+1)h,x_i+l\dx,\hat{Y}^h_{n+1})$ around $nh$ at order 0  and of $y\mapsto v^n_h(nh,x_i+l\dx,y)$ around $\hat{Y}^n_h$ at order 1, obtaining
		$$
		\begin{array}{l}
		h\sum_l \nu_{\hat{Y}^h_n}(l\dx)\big[v^n_h((n+1)h,x_i+l\dx,\hat{Y}^h_{n+1})-v^n_h((n+1)h,x_{i},\hat{Y}^h_{n+1})\big]\dx  \\
		=h\sum_l \nu_{\hat{Y}^h_n}(l\Delta x)\big[v^n_h(nh,x_i+l\dx,\hat{Y}^h_{n})-v^n_h(nh,x_{i},\hat{Y}^h_{n})\big]\dx \\
		+h(\hat{Y}^h_{n+1}-\hat{Y}^h_{n})\sum_l \nu_{\hat{Y}^h_n}(\l\dx)\big[\partial_y v^n_h(nh,x_i+l\dx,\hat{Y}^h_{n})-\partial_y v^n_h(nh,x_{i},\hat{Y}^h_{n})\big]\dx \\
		+R_2(n,h,x_i,\hat{Y}^h_{n},\hat{Y}^h_{n+1}),
		\end{array}
		$$	
		where $R_2$ contains the integral terms:
		\begin{equation}\label{R2}
		\begin{array}{l}
		R_2(n,h,x_i,\hat{Y}^h_{n},\hat{Y}^h_{n+1})= 
		h^2\sum_l \nu_{\hat{Y}^h_n}(l\Delta x)\dx\times\\
		\times \int_0^1(1-\tau)\big[\partial_tv^n_h(nh+\tau h,x_i+l\dx,\hat{Y}^h_{n+1})-\partial_tv^n_h(nh+\tau h,x_{i},\hat{Y}^h_{n+1})\big]d\tau \\
		+h (\hat{Y}^h_{n+1}-\hat{Y}^h_{n})^{2} \sum_l \nu_{\hat{Y}^h_n}(\l\dx)\dx\times \\
		\times \!\!\int_0^1\!(1-\zeta)\!\big[\!\partial_yv^n_h(nh ,x_i\!+\!l\dx,\hat{Y}^h_n\!+\!\zeta(\hat{Y}^h_{n+1}\!-\!\hat{Y}^h_n))\!-\!\partial_yv^n_h(nh,x_{i},\hat{Y}^h_n\!+\!\zeta(\hat{Y}^h_{n+1}\!-\!\hat{Y}^h_n))\!\big]\!d\zeta .
		\end{array}
		\end{equation}
		By resuming, we obtain
		\begin{equation}\label{I1}
		\begin{array}{rl}
		I_1=
		&	v^n_h(nh,x_i,\hat{Y}^h_{n})
		+ \partial_tv^n_h(nh,x_i,\hat{Y}^h_{n})h+\partial_yv^n_h(nh,x_i,\hat{Y}^h_{n})(\hat{Y}^h_{n+1}-\hat{Y}^h_{n})\\
		&+\frac 12 \partial_y^{2}v^n_h(nh,x_i,\hat{Y}^h_{n})(\hat{Y}^h_{n+1}-\hat{Y}^h_{n})^{2} \\
		&+h\dx\sum_l \nu_{\hat{Y}^h_n}(l\dx )\big[v^n_h(nh,x_i+l\dx,\hat{Y}^h_{n})-v^n_h(nh,x_{i},\hat{Y}^h_{n})\big]\\
		&+\sum_{i=1}^2 R_{i}(n,h,x_i,\hat{Y}^h_{n},\hat{Y}^h_{n+1}) + S(n,h,x_i,\hat{Y}^h_{n},\hat{Y}^h_{n+1}),
		\end{array}
		\end{equation}
		where 
		\begin{equation}
		\label{S}
		\begin{array}{l}
		S(n,h,x_i,\hat{Y}^h_{n},\hat{Y}^h_{n+1}) \\
		=	\partial_y\partial_tv^n_h(nh,x_i,\hat{Y}^h_{n})\,h(\hat{Y}^h_{n+1}-\hat{Y}^h_{n})
		+\frac 16\partial_y^3v^n_h(nh,x_i,\hat{Y}^h_{n})(\hat{Y}^h_{n+1}-\hat{Y}^h_{n})^3 \\
		+h(\hat{Y}^h_{n+1}-\hat{Y}^h_{n})\sum_l \nu_{\hat{Y}^h_n}(l\dx)\big[\partial_yv^n_h(nh,x_i+l\dx,\hat{Y}^h_{n})-\partial_yv^n_h(nh,x_{i},\hat{Y}^h_{n})\Big]\dx.
		\end{array}
		\end{equation}
		\noindent
		\textbf{Step 2. Taylor expansion  of the first addendum in the r.h.s. of \eqref{mmm}}. We set
		$$
		\begin{array}{rl}
		I_2=&A^h_{\dx}v^n_h(nh,\cdot,\hat{Y}^h_n)(x_i) \\
		%&\sum_j(A^h_{\dx})_{ij}(\hat{Y}^h_n)v^n_h(nh,x_j,\hat{Y}^h_n)\\	
		=&(\alpha^h_{\dx}(\hat{Y}^h_n)-\beta^h_{\dx}(\hat{Y}^h_n))v^n_h(nh,x_{i-1},\hat{Y}^h_n) \\
		&+(1+2\beta^h_{\dx}(\hat{Y}^h_n))v^n_h(nh,x_{i},\hat{Y}^h_n)
		-(\alpha^h_{\dx}(\hat{Y}^h_n)+\beta^h_{\dx}(\hat{Y}^h_n))v^n_h(nh,x_{i+1},\hat{Y}^h_n).
		\end{array}
		$$
		We expand with Taylor $x\mapsto v^n_h(nh,x,\hat{Y}^h_n)$  around $x_i$ up to order 3
		and we insert  the values of $\alpha^h_{\dx}$ and $\beta^h_{\dx}$ in \eqref{alpha-beta}. Rearranging the terms we get 
		\begin{equation}\label{I2}
		\begin{array}{l}
		I_2=
		v^n_h(nh,x_{i},\hat{Y}^h_n)
		-h\mu_X(\hat{Y}^h_n)\partial_xv^n_h(nh,x_{i},\hat{Y}^h_n)
		-\frac 12 \,h \sigma^2_X(\hat{Y}^h_n)\partial^2_xv^n_h(nh,x_{i},\hat{Y}^h_n)\\
		\ \ \ +R_3(n,h,x_i,\hat{Y}^h_{n},\hat{Y}^h_{n+1})
		\end{array}
		\end{equation}
		where
		\begin{equation}\label{R3}
		\begin{array}{l}
		R_3(n,h,x_i,\hat{Y}^h_{n},\hat{Y}^h_{n+1})\\
		=%\displaystyle
		\frac{\dx\mu_X(\hat{Y}^h_n)- \sigma_X^2(\hat{Y}^h_n)}{12}\,h\dx^2\int_0^1(1-\eta)^3\partial^4_xv^n_h(nh,x_i-\eta\dx,\hat{Y}^h_n)d\eta  \\
		%\displaystyle
		-\frac{\dx\mu_X(\hat{Y}^h_n)+ \sigma_X^2(\hat{Y}^h_n)}{12}\,h\dx^2\int_0^1(1-\eta)^3\partial^4_xv^n_h(nh,x_i+\eta\dx,\hat{Y}^h_n)d\eta \\
		%\displaystyle
		-\frac 1 6 \,h  \dx^2\mu_X(\hat{Y}^h_n)\partial^3_xv^n_h(nh,x_{i},\hat{Y}^h_n).
		\end{array}
		\end{equation}

		\noindent
		\textbf{Step 3. Rearranging the terms.} By resuming, from \eqref{I1} and \eqref{I2} we have
		$$
		\begin{array}{l}
		I_1-I_2\\
		=h\partial_tv^n_h(nh,x_i,\hat{Y}^h_{n})+(\hat{Y}^h_{n+1}-\hat{Y}^h_n)\partial_yv^n_h(nh,x_i,\hat{Y}^h_{n})+h\mu_X(\hat{Y}^h_n)\partial_xv^n_h(nh,x_{i},\hat{Y}^h_n)\\
		\quad+\frac 1 2 \big[(\hat{Y}^h_{n+1}-\hat{Y}^h_n)^2\partial^2_yv^n_h(nh,x_i,\hat{Y}^h_{n}) +  h\,\sigma_X^2(\hat{Y}^h_n) \partial^2_xv^n_h(nh,x_{i},\hat{Y}^h_n)\big]\\ 
		\quad +h\int(v^n_h(t,x+\gamma_X(\hat{Y}^n_h)\zeta,\hat{Y}^h_n)-v^n_h(t,x,\hat{Y}^h_n))\nu(\zeta)d\zeta\\
		\quad+ \sum_{i=1}^4R_{i}(n,h,x_i,\hat{Y}^h_{n},\hat{Y}^h_{n+1}) + S(n,h,\hat{Y}^h_n, \hat{Y}^h_{n+1}),
		\end{array}
		$$
		in which we have used the change of variable giving
		$$
		\int(v^n_h(t,x+z,\hat{Y}^h_n)-v^n_h(t,x,\hat{Y}^h_n))\nu_{\hat{Y}^h_n}(z)dz
		=\!\!\int(v^n_h(t,x+\gamma_X(\hat{Y}^n_h)\zeta,\hat{Y}^h_n)-v^n_h(t,x,\hat{Y}^h_n))\nu(\zeta)d\zeta
		$$
		and where
		\begin{equation}\label{R4}
		\begin{array}{l}
		R_4(n,h,x_i,\hat{Y}^h_{n})
		=
		h	\sum_l \big[v^n_h(t,x_i+l\dx,\hat{Y}^h_{n})-v^n_h(t,x_{i},\hat{Y}^h_{n})\big]\nu_{\hat{Y}^h_n}(l\dx)\dx \\
		-h\int\big[v^n_h(t,x_i+z,\hat{Y}^h_{n})-v^n_h(t,x_i,\hat{Y}^h_{n})\big]\nu_{\hat{Y}^h_n}(z)dz.
		\end{array}
		\end{equation} 
		By passing to the conditional expectation and by using formulas \eqref{momento1-gen}, \eqref{momento2-gen} and \eqref{momento3-gen} for the local moments of order 1, 2 and 3, we obtain
		$$
		\begin{array}{l}
		\widetilde{\mathcal{R}}_n^h(x_i,\hat{Y}^h_n):=\E[I_1-I_2\mid \hat{Y}^h_n]
		=h(\partial_tv^n_h(nh,x_i,\hat{Y}^h_{n})+\L v^n_h(nh,x_i,\hat{Y}^h_{n}))\\
		%	&\quad +\E\Big[\partial_y\partial_tv^n_h(nh,x_i,\hat{Y}^h_{n})\,h(\hat{Y}^h_{n+1}-\hat{Y}^h_{n})	+\frac 16\partial_y^3v^n_h(nh,x_i,\hat{Y}^h_{n})(\hat{Y}^h_{n+1}-\hat{Y}^h_{n})^3\mid \hat{Y}_n^h\Big]\\
		\quad +\sum_{i=1}^4\E[R_{i}(n,h,x_i,\hat{Y}^h_{n},\hat{Y}^h_{n+1})\mid \hat{Y}^h_n]+\E(S(n,h,x_i,\hat{Y}^h_n, \hat{Y}^h_{n+1})\mid \hat{Y}^h_n)\\
		\quad 	= \sum_{i=1}^4\E[R_{i}(n,h,x_i,\hat{Y}^h_{n},\hat{Y}^h_{n+1})\mid \hat{Y}^h_n]+\sum_{i=1}^2 S_i(n,h,x_i,\hat{Y}^h_n).
		\end{array}
		$$
		Here we have used the following facts: $u$ solves \eqref{PDEB}; $\E(S(n,h,x_i,\hat{Y}^h_n, \hat{Y}^h_{n+1})\mid \hat{Y}^h_n)=\sum_{i=1}^2 S_i(n,h,x_i,\hat{Y}^h_n)$, with (recall the definition of $S$ in  \ref{S} and of the local moments $f_h$, $g_h$ and $j_h$  in \eqref{momento1-gen}, \eqref{momento2-gen} and  \eqref{momento3-gen})
		\begin{equation}\label{S1}
		\begin{array}{l}
		S_1(n,h,x_i,\hat{Y}^h_{n}) \\
		=f_h(\hat{Y}^h_n)\partial_yv^n_h(nh,x_i,\hat{Y}^h_n)+\frac 12g_h(\hat{Y}^h_n)\partial^2_yv^n_h(nh,x_i,\hat{Y}^h_n)+\frac 16j_h(\hat{Y}^h_n)\partial^3_yv^n_h(nh,x_i,\hat{Y}^h_n) \\ 
		\quad +\partial_y\partial_tv^n_h(nh,x_i,\hat{Y}^h_{n})\,h(\mu_Y(\hat{Y}^h_{n})h+f_h(\hat{Y}^h_n)),
		\end{array}
		\end{equation} 
		\begin{equation}\label{S2}
		\begin{array}{l}
		S_2(n,h,x_i,\hat{Y}^h_{n}) 
		=h(h\mu_Y(\hat{Y}^h_n)+f_h(\hat{Y}^h_n))\times\\
		\times\sum_l \nu_{\hat{Y}^h_n}(l\dx)\big[\partial_yv^n_h(nh,x_i+l\dx,\hat{Y}^h_{n})-\partial_yv^n_h(nh,x_{i},\hat{Y}^h_{n})\Big]\dx.
		\end{array}
		\end{equation} 
		\noindent
		\textbf{Step 4. Estimate of the remainder.} 
		Hereafter, $C$ denotes a positive constant which may vary from a line to another and is independent of $n,h,\dx$.
		
		By \eqref{mmm}, we have to study 
		$\mathcal{R}_n^h(\cdot ,\hat{Y}^h_n)=(A^h_{\dx})^{-1}(\hat{Y}^h_n)\widetilde{\mathcal{R}}_n^h(\cdot ,\hat{Y}^h_n)$. By  Lemma \ref{norma}, \\ $\sup_{y\in\mathcal{D}}|(A^h_{\dx})^{-1}(y)|_2\leq 1$, so
		$$
		\begin{array}{l}
		\E\big[e^{\sum_{l=1}^{n}2\lambda c_\nu h}|\mathcal{R}_n^h(\cdot,\hat{Y}^h_n)|^2_{2}\big]
		\leq e^{2\lambda c_\nu T}\E\big[|\widetilde{\mathcal{R}}_n^h(\cdot,\hat{Y}^h_n)|^2_{2}\big]\\
		\quad \leq C	\sum_{i=1}^4\E\big[|R_{i}(n,h,\cdot,\hat{Y}^h_{n},\hat{Y}^h_{n+1})|^2_{2}\big] +\sum_{i=1}^2\E\big[|S_{i}(n,h,\cdot,\hat{Y}^h_{n})|^2_{2}\big] .
		\end{array}
		$$
		Hence it suffices to prove that the above 6 terms are all upper bounded by $Ch^2(h+\dx^2)^2$. The inequalities studied in $(ii)$ of Lemma \ref{lemma-poisson} now come on.
		
		Consider first $R_1$ in \eqref{R1} and in particular, the first addendum therein. Set
		$$
		g_n(x)= h^2\int_0^1(1-\tau)\partial^2_tv^n_h(nh+\tau h,x,\hat{Y}^h_{n+1})d\tau.
		$$
		Since $u\in C^{2,6}_{\pol,T}(\R,\mathcal{D})$,  $\partial^k_xg_n\in L^2(\R,dx)$ for every $k=0,1,2$ and $|\partial^k_xg_n|_{L^2}$ $\leq Ch^2(1+|\hat{Y}^h_{n+1}|^a)$. So, by using \eqref{poisson2},
		$$
		|g_n|_2^2\leq C h^4(1+|\hat{Y}^h_{n+1}|^a)^2.
		$$
		Similar estimates hold for the other terms in $R_1$, so we can write
		$$
		\begin{array}{rl}
		|R_{1}(n,h,\cdot,\hat{Y}^h_{n},\hat{Y}^h_{n+1})|^2_{2}\leq
		C\big[&\!\!\!\!h^4(1+|\hat{Y}_n^h|^{a})^2+|\hat{Y}_{n+1}-\hat{Y}_n|^{8}(1+|\hat{Y}_n^h|^{a}+|\hat{Y}_{n+1}^h|^{a})^2\\
		&+h^2|\hat{Y}_{n+1}-\hat{Y}_n|^4(1+|\hat{Y}_n^h|^{a})^2\big].
		\end{array}
		$$
		By using the increment estimates \eqref{stima_momento-gen}, the moment estimates \eqref{stima_incremento-gen} and the Cauchy-Schwartz inequality, we obtain
		$$
		\E\big[|R_{1}(n,h,\cdot,\hat{Y}^h_{n},\hat{Y}^h_{n+1})|^2_{2}\big]\leq Ch^4.
		$$
		The same arguments can be developed for $R_3$ in \eqref{R3} and $S_1$ in \eqref{S1}. These give
		%and $S_2$ in \eqref{S2} - for the latter, we use first \eqref{poisson2}. 
		$$
		\E\big[|R_3(n,h,\cdot,\hat{Y}^h_{n},\hat{Y}^h_{n+1})|^2_{2}\big]\leq C h^2\dx^4
		\mbox{ and }\E\big[|S_1(n,h,\cdot,\hat{Y}^h_{n})|^2_{2}\big]	\leq Ch^4.
		$$
		In order to study $R_2$ in \eqref{R2}, consider the first term and set
		$$
		\begin{array}{l}
		g_n(x)
		=h^2	\sum_l \nu_{\hat{Y}^h_n}(l\dx)\dx
		\times\\
		\times	\int_0^1(1-\tau)\big[\partial_tv^n_h(nh+\tau h,x+l\Delta x,\hat{Y}^h_{n+1})-\partial_tv^n_h(nh+\tau h,x,\hat{Y}^h_{n+1})\big]d\tau.
		\end{array}
		$$
		We notice that $g_n\in C^2$. By the Cauchy-Schwarz inequality for the (discrete) finite measure $\nu_{\hat{Y}^h_n}(l\dx)\dx$, $l\in\Z$, we have
		$$
		\begin{array}{l}
		|\partial_x^kg_n(x)|^2
		\leq Ch^4\sum_l \nu_{\hat{Y}^h_n}(l\dx)\dx
		\times\\
		\times	\int_0^1(1-\tau)^2\Big(\big|\partial_x^k\partial_tv^n_h(nh+\tau h,x+l\Delta x,\hat{Y}^h_{n+1})\big|^2+\big|\partial_x^k\partial_tv^n_h(nh+\tau h,x,\hat{Y}^h_{n+1})\big|^2\Big)d\tau.
		\end{array}
		$$
		This gives $|\partial^k_x g_n|_{L^2}\leq C h^2(1+|\hat{Y}^h_{n+1}|^a)$ and, by \eqref{poisson2}, $|g_n|_2^2\leq C h^4 (1+|\hat{Y}^h_{n+1}|^a)^2$. By developing the same arguments to the other terms in $R_2$, we obtain 
		$$
		\begin{array}{l}
		|R_{2}(n,h,\cdot,\hat{Y}^h_{n},\hat{Y}^h_{n+1})|^2_{2}\leq
		C\big[h^4(1+(\hat{Y}_n^h)^{a})^2+
		h^2|\hat{Y}_{n+1}-\hat{Y}_n|^{4}(1+|\hat{Y}_n^h|^{a})\big].
		\end{array}
		$$
		And by passing to the expectation, we get $\E(|R_{2}(n,h,\cdot,\hat{Y}^h_{n},\hat{Y}^h_{n+1})|^2_{2})\leq Ch^4$. A similar approach can be used to handle $R_4$ in \eqref{R4} and in $S_2$ in \eqref{S2}, giving 
		$$
		\E\big[|R_4(n,h,\cdot ,\hat{Y}^h_{n})|_{2}^2\big]
		\leq
		Ch^2\dx^4 \mbox{ and }
		\E\big[|S_2(n,h,\cdot,\hat{Y}^h_{n})|^2_{2}\big]\leq Ch^4.
		$$
	\end{proof}

													\subsubsection{Convergence in $l_\infty$-norm}\label{sect-linf}
													We consider here a different finite difference scheme for equation \eqref{PDE-barug-h}:
													%  
													%We still suppose that the Markov chain $(Y^h_n)_{n=0,\dots,N}$ approximating the process $Y$ satisfies assumptions $\mathcal H_1$ and $\mathcal H_2$, while Assumption $\mathcal H_5(g)$ will be not needed in this case.
													%  
													we  still  approximate (explicit in time) the integral term $\L^{(y)}_{\mbox{\tiny{int}}}v$ in  \eqref{L_i} with a  trapezoidal rule, but  we use an  upwind first order scheme
													to approximate (implicit in time) the differential part $\L^{(y)}_{\mbox{\tiny{diff}}}v$ in  \eqref{L_d}.  As usually done in convection-diffusion problems, we distinguish the cases in which $\mu_X(y)$ is positive or negative in order to take into account the asymmetry given by the convection term and we use one sided difference in the appropriate direction.  Specifically, if $\mu_X(y)\geq 0$, we approximate  $\L^{(y)}_{\mbox{\tiny{diff}}}u$ by using the scheme 
													$$
													\frac{v^{n+1}_i-v^n_i}{h}+ \mu_X(y)\frac{v^{n}_{i+1}-v^n_i}{\dx} + \frac 1 2 \sigma_X^2(y)\frac{v^{n}_{i+1}-2v^n_i+v^n_{i-1}}{\dx^2} ,
													$$
													while, if $\mu_X(y)\leq 0$, we use the approximation
													$$
													\frac{v^{n+1}_i-v^n_i}{h}+ \mu_X(y)\frac{v^{n}_{i}-v^n_{i-1}}{\dx} + \frac 1 2 \sigma_X^2(y)\frac{v^{n}_{i+1}-2v^n_i+v^n_{i-1}}{\dx^2}.
													$$
													The resulting scheme is 
													\begin{equation}\label{equaz2}
													A^h_{\dx}(y)v^n=B^h_{\dx}(y)v^{n+1},
													\end{equation}
													where $A^h_{\dx}(y)$ is the linear operator given by  
													\begin{equation}\label{A2}
													(A^h_{\dx})_{ij}(y)=\begin{cases}
													-\beta^h_{\dx}(y)-|\alpha^h_{\dx}(y)|\ind{\alpha^h_{\dx}(y)<0},\qquad &\mbox{ if }i=j+1,\\
													1+2\beta^h_{\dx}(y)+|\alpha^h_{\dx}(y)|,\qquad &\mbox{ if }i=j,\\
													-\beta^h_{\dx}(y)-|\alpha^h_{\dx}(y)|\ind{\alpha^h_{\dx}(y)>0},\qquad &\mbox{ if }i=j-1,\\
													0, &\mbox{ if }|i-j|>1,
													\end{cases}
													\end{equation}  
													with
													%	\begin{equation}\label{alpha-beta-new}
													$$
													\alpha^h_{\dx}(y)=\frac{h}{\dx}\mu_X(y), \qquad \beta^h_{\dx}(y)=\frac{h}{2\dx^2}\sigma^2_X(y),
													$$
													%\end{equation}
													and $B^h_{\dx}(y)$ is the linear operator defined in \eqref{B}. Then we have:
%													\begin{lemma}\label{lemmamatrici}
%														For every $y\in \D$, the operator $A^h_{\dx}(y):l_\infty(\mathcal{X})\rightarrow l_\infty(\mathcal{X})$ is  invertible and $|( A^h_{\dx})^{-1}(y)|_\infty\leq1$. Moreover, $|B^h_{\dx}(y)|_\infty\leq 1+2\lambda c_\nu|\gamma_X(y)|$. Finally, if $\gamma_X\equiv 1$, $\Pi^h_{\dx}(y)=(A^h_{\dx})^{-1}B^h_{\dx}(y)$ is a stochastic operator, that is,  $$(\Pi^h_{\dx})_{ij}(y)\geq 0, \quad i,j\in\Z,\qquad  \qquad   \sum_{j\in\Z}(\Pi^h_{\dx})_{ij}(y)=1, \quad j\in\Z.$$
%													\end{lemma}
%													\begin{proof}
%														We write $ A^h_{\dx}(y)=\eta(y) I-P(y)$, 
%														where $\eta(y)=1+2\beta^h_{\dx}(y)+|\alpha^h_{\dx}(y)|$, $I$ is the identity operator and $P_{ij}(y)=0$ if $|i-j|\neq 1$ and $P_{ij}=-(A^h_{\dx})_{ij}$ if $|i-j|=1$. So, it is easy to see that the operator $A^h_{\dx}(y):l_\infty(\mathcal{X})\rightarrow l_\infty(\mathcal{X})	$ is invertible with inverse
%														$$
%														(A^h_{\dx})^{-1}(y)=(	\eta(y) I -P)^{-1}=\frac 1 \eta \sum_{k=0}^\infty \frac {P^k} {\eta^k }.
%														$$
%														The assertion for $B^h_{\dx}(y)$ immediately follows from \eqref{B}. Finally, $(A^h_{\dx})^{-1}_{ij}(y)\geq 0$ for all $i,j$ because all entries of $P(y)$ are non negative   and $(B^h_{\dx})_{ij}(y)\geq 0$ if $\mu_X\equiv 1$. Moreover, $\Pi^h_{\dx}(y)1=1$ because, by construction, $A^h_{\dx}(y)1=1$ and $B^h_{\dx}(y)1=1$ when $\mu_X\equiv 1$.
%													\end{proof}
%													
%												
\begin{lemma}\label{lemmamatrici}
	For every $y\in \D$, the operator $A^h_{\dx}(y):l_\infty(\mathcal{X})\rightarrow l_\infty(\mathcal{X})$ is  invertible and $\sup_{y\in\mathcal{D}}|( A^h_{\dx})^{-1}(y)|_\infty\leq1$. And if $\frac{\nu'}{\nu},\frac{\nu''}{\nu}\in L^1(\R,d\nu)$ then  $\sup_{y\in\mathcal{D}}|B^h_{\dx}(y)|_\infty$ $\leq 1+2\lambda c_\nu$, $c_\nu$ being defined in \eqref{cnu}.  Finally, if $\gamma_X\equiv 1$, $\Pi^h_{\dx}(y)=(A^h_{\dx})^{-1}B^h_{\dx}(y)$ is a stochastic operator, that is,  $$(\Pi^h_{\dx})_{ij}(y)\geq 0, \quad i,j\in\Z,\qquad  \qquad   \sum_{j\in\Z}(\Pi^h_{\dx})_{ij}(y)=1, \quad j\in\Z.$$
\end{lemma}
\begin{proof}
	We write $ A^h_{\dx}(y)=(1+\eta(y)) \Id-P(y)$, 
	where $\eta(y)=2\beta^h_{\dx}(y)+|\alpha^h_{\dx}(y)|\geq 0$ and $P_{ij}(y)=0$ if $|i-j|\neq 1$ and $P_{ij}=-(A^h_{\dx})_{ij}$ if $|i-j|=1$. It easily follows that $|P(y)|_\infty\leq \eta(y)$. Moreover, it is easy to see that the operator $A^h_{\dx}(y):l_\infty(\mathcal{X})\rightarrow l_\infty(\mathcal{X})	$ is invertible with inverse
	$$
	(A^h_{\dx})^{-1}(y)=(	(1+\eta(y)) \Id -P(y))^{-1}=\frac 1 {1+\eta(y)} \sum_{k=0}^\infty \frac {P(y)^k} {(1+\eta(y))^k }.
	$$
	It then follows that $|( A^h_{\dx})^{-1}(y)|_\infty\leq1$. The assertion for $B^h_{\dx}(y)$ follows from \eqref{B} and \eqref{cnu}. Finally, $(A^h_{\dx})^{-1}_{ij}(y)\geq 0$ for all $i,j$ because all entries of $P(y)$ are non negative   and $(B^h_{\dx})_{ij}(y)\geq 0$ if $\gamma_X\equiv 1$. Moreover, $\Pi^h_{\dx}(y)1=1$ because, by construction, $A^h_{\dx}(y)1=1$ and $B^h_{\dx}(y)1=1$ when $\gamma_X\equiv 1$.
\end{proof}
	We can now state the convergence result.

													\begin{theorem}
														\label{conv-H2}
														Let $\tilde u^h_n$ be defined in \eqref{backward3} and $u^h_n$ be given by \eqref{backward-ter0-tris} with the choice
														$$
														\Pi^h_{\dx}(y)=(A^h_{\dx})^{-1}B^h_{\dx}(y),
														$$
														$A^h_{\dx}(y)$ and $B^h_{\dx}(y)$ being given in \eqref{A2} and \eqref{B} respectively.  Moreover, for  $n=0,\dots, N,$ consider the function 
														\begin{equation*}
														v^h_n(t,x,y)=\E\left[  \tilde u^h_{n+1}(X^{t,x,y}_{(n+1)h},Y^{t,y}_{(n+1)h})  \right], \qquad \qquad t\in[nh,(n+1)h].
														\end{equation*}
														Assume that
														\begin{itemize}
															\item $\frac{\nu'}\nu,\frac{\nu''}\nu\in L^1(\R,d\nu)$;
															\item the Markov chain $(Y^h_n)_{n=0,\dots, N}$ satisfies assumptions $\mathcal{A}_1,\, \mathcal{A}_2$ and $\mathcal{A}_3(4\lambda c_\nu|\gamma_X|)$;
															\item $v^h_n\in  C^{\infty,4}_{\pol, [nh,(n+1)h]}( \R, \D)$ a.e. and uniformly in $n$ and $h$.
														\end{itemize}
														Then, there exist $\bar h,C>0$ such that for every $h<\bar h$ and $\dx<1$ one has
														\begin{equation*}
														| \tilde u^h_0(\cdot,Y_0)-u^h_{0}(\cdot,Y_0)|_{\infty} \leq CT(h+\dx^2).
														\end{equation*}
														
													\end{theorem}

													\begin{proof}
													The statement follows by applying Theorem  \ref{convergencebates} once it is proved that $\mathcal{K}(\infty, 2\lambda c_\nu, h+\dx)$ holds. This is just a rewriting of the proof of  Theorem  \ref{conv-H} in terms of the norm in $l_\infty(\mathcal{X})$. We only notice that, for handling the remaining terms, we do not need to apply  \eqref{poisson2} for the $l_\infty$-norm, so we do not need more regularity for $u$. That's why the class $C^{\infty,4}_{\pol,T}(\R,\D)$ is enough. 
													%we do not ask here that $u\in C^{\infty,6}_{\pol,T}(\R,\D)$ and 
												\end{proof}

													It is natural to look for conditions on the function  $f$ which ensure that the regularity assumptions on the function $v^h_n$ for $n=0,\dots, N$, which are required  	In Theorem \ref{conv-H2}, are actually satisfied.  Of course, these conditions  depend on the regularity of the model. In Sections \ref{sect-conv-eu} and \ref{sect-conv-am} we will study the case of the degenerate Heston or Bates model.

													\section{The European case in the Heston/Bates model}\label{sect-conv-eu}
													As an application in finance, 	in this section we apply  our convergence results to to a tree-finite difference procedure for  pricing European options  in the Heston (\cite{H}) or Bates (\cite{bates}) model: the asset price process $S$ and the volatility process $Y$  evolve following the stochastic differential system
													\begin{equation}\label{bates}
													\begin{array}{ll}
													&\displaystyle
													\frac{dS_t}{S_{t^-}}= (r-\delta)dt+\mu \sqrt{Y_t}\, dZ^1_t+\gamma d\tilde H_t,  \smallskip\\
													&\displaystyle
													dY_t= \kappa(\theta-Y_t)dt+\sigma\sqrt{Y_t}\,dZ^2_t,
													\end{array}
													\end{equation}
													where  $S_0>0$, $Y_0\geq 0$, $Z=(Z^1,Z^2)$ is a correlated Brownian motions with $d\langle Z^1,Z^2\rangle_t=\rho dt$, $|\rho|<1$, $\tilde H$ is a compound Poisson process with
													intensity $\lambda$ and i.i.d. jumps $\{\tilde  J_k\}_k$ as in \eqref{H2}. Here, $\gamma=1$ (Bates model) or $\gamma= 0$ (Heston model).
													The above quantities  $r$ and $\delta$ are  the interest rate and the dividend interest rate respectively.	We assume, as usual, that the Poisson process $K$, the jump amplitudes $\{\tilde J_k\}_k$ and the correlated Brownian motion $(Z^1,Z^2)$ are independent. 
													
													With a simple transformation, we can reduce the model \eqref{bates} to our reference model \eqref{generalsystem}. 
													To get rid of the correlated Brownian motion,  we set 
													$$
													\bar\rho=\sqrt{1-\rho^2}\quad\mbox{and}\quad Z^2=W,\quad Z^1=\rho Z^2+\bar \rho B,
													$$
													in which $(B,W)$ denotes a standard $2$-dimensional Brownian motion. 
													Moreover,	considering the process $X_t=\log S_t-\frac{\rho}{\sigma}Y_t$, we reduce to the jump-diffusion pair  $(X,Y)$, which evolves according to
													\begin{equation}\label{SDE_bates}
													\begin{split}
													&dX_t= \mu_X(Y_t)dt+\bar\rho\,\sqrt{Y_t}\, dB_t+\gamma dH_t,\\
													&dY_t= \kappa(\theta-Y_t)dt+\sigma\sqrt{Y_t}\,dW_t,
													\end{split}
													\end{equation}
													where 
													$$
													\mu_X(y)=  r-\delta-\frac y2 -\frac \rho{\sigma}\kappa(\theta-y),
													$$
													$H_t$ is the compound Poisson process written through the Poisson process $K$, with intensity $\lambda$, and the i.i.d. jumps $J_k=\log(1+\tilde J_k)$. 
													%Moreover,  $\gamma_X\equiv 1$ in the Bates model and $\gamma_X\equiv 0$ in the Heston model.
													The standard Bates model requires that $J_1$ has a normal law. But it is clear that the  convergence result holds for other laws such that the L\'evy measure $\nu$ satisfies the requests in Theorem \ref{conv-H} or Theorem \ref{conv-H2}. For example, these properties hold for the mixture of exponential laws used by Kou \cite{kou}.
													
													In this section we focus on European options. 
													Recall that, in this case,  the function $\tilde u^h_n(\cdot)$ defined in \eqref{backward3}  is nothing but the European price value at time $nh$, that is  $u(nh,\cdot)$ where $u$ is defined in \eqref{european}.  Moreover, we can easily see that, for any $n=N-1,\dots$, the function $v^h_n$ defined in \eqref{v^h_n} satisfies 
													$$
													v^h_n(t,x,y)=u(t,x,y), \qquad t\in [nh,(n+1)h].
													$$
													We consider  the approximating Markov chain for the CIR process discussed in Section \ref{sect-CIR} and the two possible finite difference operator discussed in Section \ref{sect-l2} and \ref{sect-linf}. As an application, we get the following convergence rate result of the hybrid method.

													\begin{theorem}\label{Bates_conv}
														Let $(X,Y)$ be the solution to \eqref{SDE_bates} and let $(Y^h_n)_{n=0,\dots, N}$ be the Markov chain introduced in Section \ref{sect-CIR} for approximating  the CIR process $Y$. Let $u(t,x,y)=\E(f(X_T^{t,x,y},Y_T^{t,y}))$ be as in \eqref{european} and $(u^h_n)_{n=0,\ldots,N}$ be given by \eqref{backward-ter0} with the choice
														$$
														\Pi^h_{\dx}(y)=(A^h_{\dx})^{-1}B^h_{\dx}(y).
														$$
														
														\begin{itemize}
															\item[$(i)$] 
															$\mathrm{[Convergence\  in\ }l_2(\mathcal{X})]$ Suppose that 
															\begin{itemize}
																\item [$\bullet$]
																$A^h_{\dx}(y)$ and $B^h_{\dx}(y)$ are defined in \eqref{A} and \eqref{B} respectively;
																\item [$\bullet$]
																$\frac{\nu'}{\nu},\frac{\nu''}{\nu}\in L^2(\R,d\nu)$ and $\nu$ has finite moments of any order;
																\item [$\bullet$]
																$\partial^{2j}_xf\in C^{2,6-j}_{\pol}(\R, \R_+)$ for every $j=0,\ldots,6$.
															\end{itemize} 
															Then, there exist $\bar h,C>0$ such that for every $h<\bar h$ and $\dx<1$ one has
															$$
															| u(0,\cdot,Y_0)-u^h_{0}(\cdot,Y_0)|_{2} \leq CT(h+ \dx^2).
															$$
															
															\item[$(ii)$] 
															$\mathrm{[Convergence\  in\ } l_\infty(\mathcal{X})]$ Suppose that 
															\begin{itemize}
																\item [$\bullet$]
																$A^h_{\dx}(y)$ and $B^h_{\dx}(y)$ are defined in \eqref{A2} and \eqref{B} respectively;
																\item [$\bullet$]
																$\frac{\nu'}{\nu},\frac{\nu''}{\nu}\in L^1(\R,d\nu)$  and $\nu$ has finite moments of any order;
																\item [$\bullet$]
																$\partial^{2j}_xf\in C^{\infty,4-j}_{\pol}(\R, \R_+)$ for every $j=0,\ldots,4$.
															\end{itemize} 
															Then, there exist $\bar h,C>0$ such that for every $h<\bar h$ and $\dx<1$ one has
															$$
															| u(0,\cdot,Y_0)-u^h_{0}(\cdot,Y_0)|_{\infty} \leq CT(h+ \dx).
															$$
														\end{itemize} 
														
													\end{theorem}
													
													\begin{proof}
														We apply Theorem \ref{conv-H} for $(i)$ and Theorem \ref{conv-H2} for $(ii)$. The validity of assumptions $\mathcal{A}_1$ and $\mathcal{A}_2$ is proved in Proposition \ref{moments-2}. So, we need only to prove that if $\partial_x^{2j}f\in C^{2,6-j}_{\pol}(\R, \R_+)$ as $j=0,1,\ldots,6$, resp. $\partial_x^{2j}f\in C^{\infty,4-j}_{\pol}(\R, \R_+)$ as $j=0,1,\ldots,4$, then $u\in C^{2,6}_{\pol,T}(\R, \R_+)$, resp. $u\in C^{\infty,4}_{\pol,T}(\R, \R_+)$. This is proved in next Proposition \ref{prop-reg-new} (set $\rho=0$, $\mathfrak{a}=r-\delta-\frac \rho{\sigma}\kappa\theta$ and $\mathfrak{b}=\frac \rho{\sigma}\kappa-\frac 12$ therein),  the whole next Section \ref{appendix-reg} being devoted to.
													\end{proof}
													
													\begin{remark}
														In Chapter \ref{chapter-art3} we have considered the  Bates-Hull-White model \cite{bctz}, which is a Bates model coupled with a stochastic interest rate. Recall that the dynamics follows  \eqref{bates} in which $r$ is not constant but given by the Vasicek model 
														$$
														dr_t=\kappa_r(\theta_r-r_t)dt+\sigma_rdZ^3_t,
														$$
														$Z^3$ being a Brownian motion correlated with $Z^1$ (and possibly $Z^2$).  Here, there is no global transformation allowing one to reduce to our reference model. Nevertheless, a similar convergence result can be proved by means of the local transformation introduced in Section \ref{sec:approxPDE}, acting on each time interval $[nh,(n+1)h]$.	
													\end{remark}
													%
													%		\begin{remark}
													%In Chapter 3 of \cite{bctz} we apply the algorithm in $(i)$ of Theorem \ref{Bates_conv} to  less regular payoff functions such as standard call and put options, which (up to the log-transformation) do not belong to $C^{2,10}_{\pol}(\R,\R_+)$. Nevertheless, the numerical results therein confirm the  rate of convergence in Theorem \ref{Bates_conv}. However, note that the standard call and put payoff functions are piecewise $C^\infty$.  Moreover, in practice one cannot solve  the PIDE problem over the whole real line, so one has to replace the infinite grid $\mathcal X$ with a finite one choosing  suitable boundary conditions. So, for the analysis of the convergence of the algorithm, one might reduce to consider  a function $f$ in the good functional space such that $f$ corresponds to the real payoff function in the bounded domain where the algorithm  actually works.
													%		\end{remark}
													
													\subsection{A regularity result for the Heston PDE/Bates PIDE}\label{appendix-reg}
													We deal here with a slightly more general model: we consider the SDE 
													\begin{equation}\label{SDE}
													\begin{split}
													&dX_t= \left(\mathfrak{a}+\mathfrak{b}Y_t\right)dt+\sqrt{Y_t}\, dW^1_t+\gamma_XdH_t, \\
													&dY_t= \kappa(\theta-Y_t)dt+\sigma\sqrt{Y_t}\,dW^2_t,
													\end{split}
													\end{equation}
													where $W^1,W^2$ are correlated Brownian motions  with $d\langle W^1,W^2\rangle_t=\rho dt$ and $H$ is a compound Poisson process with intensity  $\lambda$ and  L\'evy measure $\nu$, which is assumed hereafter to have finite moments of any order. Here, $\mathfrak{a},\mathfrak{b}\in\R$  and $\gamma_X\in\{0,1\}$ denote constant parameters. Note that 
													when $\mathfrak{a}=r-\delta$ (interest rate minus dividend rate), $\mathfrak{b}=-\frac 12$ and $\gamma_X=0$ (resp. $\gamma_X=1$), then $(X,Y)$ is the standard Heston (resp. Bates) model for the log-price and volatility. When instead $\rho=0$, $\mathfrak{a}=r-\delta-\frac \rho{\sigma}\kappa\theta$ and $\mathfrak{b}=\frac \rho{\sigma}\kappa-\frac 12$, we recover the equation \eqref{SDE_bates} discussed in Theorem \ref{Bates_conv}.
													
													Let  $\L$ denote the infinitesimal generator associated to \eqref{SDE}, that is,
													\begin{equation}	\label{app-L}
													\L u=\frac y 2 \left( \partial^2_x u +2\rho\sigma \partial_x\partial_yu+\sigma^2\partial^2_yu  \right)
													+  \left(\mathfrak{a}+\mathfrak{b}y \right)\partial_xu
													+\kappa(\theta-y)\partial_yu+\mathcal{L}_{\mbox{{\tiny int}}} u,
													\end{equation}
													where, hereafter, we set
													$$
													\mathcal{L}_{\mbox{{\tiny int}}} u(t,x,y)=\gamma_X\int \big[u(t,x+\zeta,y)-u(t,x,y)\big]\nu(\zeta)d\zeta.
													$$
													
													So, the present section is devoted to the proof of the following result.
													
													\begin{proposition} \label{prop-reg-new}
														Let $p\in [1,	\infty]$, $q\in\N$ and  suppose that %$f\in C^{p,q}_{\pol}(\R,\R_+)$ and
														$\partial_x ^{2j}f\in C^{p,q-j}_{\pol}(\R,\R_+)$ for every $j=0,1,\ldots,q$. 
														Set
														$$
														u(t,x,y)=\E\big[f(X^{t,x,y}_T,Y^{t,y}_T)\big].
														$$
														Then $u\in C^{p,q}_{\pol, T}(\R,\R_+)$. %, that is, $\partial^{l}_t\partial^m_x\partial^n_y u\in C^{p,0}_{\pol,T}(\R, \R_+)$ for every $l,m,n$ such that $2l+m+n\leq q$. 
														Moreover, the following stochastic representation holds: for $m+2n\leq 2q$,
														\begin{equation}\label{stoc_repr-new}
														\begin{split}
														&\partial^{m}_x\partial^{n}_yu(t,x,y)=\E\left[e^{-n\kappa (T-t)} \partial^m_x\partial^n_yf(X^{n,t,x,y}_T,Y^{n,t,x,y}_T)\right]\\
														&\quad+
														n\,\E\left[\int_t^T\left[\frac 12 \partial^{m+2}_x\partial^{n-1}_yu+\mathfrak{b}\partial^{m+1}_x\partial^{n-1}_yu\right](s,X^{n,t,x,y}_s,Y^{n,t,x,y}_s)ds  \right],
														\end{split}
														\end{equation}
														where $\partial^{m}_x\partial^{n-1}_y u:=0$ when $n=0$ and $(X^{n,t,x,y},Y^{n,t,x,y})$, $n\geq 0$, denotes the solution starting from $(x,y)$ at time $t$ to the SDE \eqref{SDE}  with parameters
														\begin{equation}\label{parameters-new}
														\rho_n=\rho,\quad \mathfrak{a}_n=\mathfrak{a}+n\rho\sigma,\quad \mathfrak{b}_n=\mathfrak{b},\quad \kappa_n=\kappa,\quad \theta_n=\theta+\frac{n\sigma^2}{2\kappa},\quad \sigma_n=\sigma.
														\end{equation}
														In particular, if $q\geq 2$ then $u\in C^{1,2}([0,T]\times \bar{\mathcal{O}})$, $\bar{\mathcal{O}}=\R\times\R_+$, solves the PIDE
														\begin{equation}\label{PIDE-new}
														\begin{cases}
														\partial_t u(t,x,y)+ \L u(t,x,y)= 0,\qquad& t\in [0,T), \,(x,y)\in \bar{\mathcal{O}},\\
														u(T,x,y)= f(x,y),\qquad& (x,y)\in \bar{\mathcal{O}}.
														\end{cases}	
														\end{equation}
													\end{proposition}

													\begin{remark}
														For our purposes, we need both the polynomial growth condition for $(x,y)\mapsto u(t,x,y)$ and the $L^p$	property for $x\mapsto u(t,x,y)$, and similarly for the derivatives. A closer look to the proof of Proposition \ref{prop-reg-new} shows that the result holds also when one is not interested in the latter $L^p$ condition. In this case, Proposition \ref{prop-reg-new} reads:  for $q\in\N$, if $\partial_x ^{2j}f\in C^{q-j}_{\pol}(\R\times\R_+)$ for every $j=0,1,\ldots,q$ then
														$u\in C^q_{\pol,T}(\R\times \R_+)$. Moreover, the stochastic representation \eqref{stoc_repr-new} holds and, if $q\geq 2$, $u$ solves PIDE \eqref{PIDE-new}.
													\end{remark}	
													
													As an immediate consequence of Proposition \ref{prop-reg-new}, we obtain the already known regularity result for the CIR process which has been already proved in Proposition 4.1 of  \cite{A-MC}.
													
													\begin{corollary}\label{corollary-cir}
														Assume that $f=f(y)$ and set
														$
														u(t,y)=\E\big[f(Y^{t,y}_T)\big].
														$
														If $f\in C^q_\pol(\R_+)$, then $u\in C^q_{\pol, T}(\R_+)$. %, that is,  $\partial^l_t\partial^n_y u\in C_{\pol,T}(\R_+)$  for every $l, n\in \N$ such that $2l+n\leq q$. 
														Moreover, for $n\leq q$,
														$$
														\partial^n_y u(t,y)=\E\left[e^{-n\kappa (T-t)} 	\partial^n_yf(Y^{n,t,y}_T)\right],
														$$
														where $Y^{n,t,y}$ denotes a CIR process starting from $y$ at time $t$ which solves the CIR dynamics with parameters $\kappa_{n}=\kappa$, $\theta_{n}=\theta+\frac {n\sigma^2} {2\kappa}$, $\sigma_n=\sigma$. 
														In particular, if  $q\geq 2$ then $u\in C^2_\pol(\R_+)$ solves the PDE
														$$
														\begin{cases}
														\partial_tu+ \mathcal{A} u= 0,\qquad& (t,y)\in [0,T)\times \R_+,\\
														u_n(T,y)= \partial^n_yf(y),\qquad& y\in \R_+,
														\end{cases}	
														$$	
														where $\mathcal{A}$ is the CIR infinitesimal generator  (see \eqref{A-op}).
														%		and, if $q\geq 2$, $u_n:=\partial^n_y u\in C^2_\pol((0,+\infty))$ solves the PDE
														%	$$
														%	\begin{cases}
														%	\partial_tu_n+ \mathcal{A}_n u= 0,\qquad& (t,y)\in [0,T)\times (0,+\infty),\\
														%	u_n(T,y)= \partial^n_yf(y),\qquad& y\in \R_+,
														%	\end{cases}	
														%	$$	
														%	where $Y^{n,t,y}$ denotes a CIR process starting from $y$ at time $t$ which solves the CIR dynamics with parameters $\kappa_{n}=\kappa$, $\theta_{n}=\theta+\frac {n\sigma^2} {2\kappa}$, $\sigma_n=\sigma$ and $\mathcal{A}_n$ is the associated infinitesimal generator  (see \eqref{A-op}).
													\end{corollary}

													We first need some preliminary results. First of all, recall that $X$ and $Y$ have uniformly bounded moments:
													for every $T>0$ and $a\geq 1$ there exist $A>0$  such that for every $t\in[0,T]$,
													\begin{equation}\label{moments_XY}
													\sup_{s\in[t,T]} \E[|X^{t,x,y}_s|^a]\leq A(1+|x|^{a}+y^{a})\mbox{ and }
													\sup_{s\in[t,T]} \E[|Y^{t,y}_s|^a]\leq A(1+y^a).
													\end{equation}
													For the second property in  \eqref{moments_XY}, we refer, for example, to \cite{A-MC}, whereas the first one follows from standard techniques.
													\begin{lemma}\label{lemma_reg}
														Let $p\in[0,\infty]$, $ g\in C^{p,0}_{\pol}(\R,\R_+)$,
														$ h\in C^{p,0}_{\pol,T}(\R,\R_+)$  and consider the function
														\begin{equation}\label {u-g-f}
														u(t,x,y)=\E\left[e^{\varrho (T-t)} g(X^{t,x,y}_T,Y^{t,y}_T)-\int_t^Te^{\varrho (s-t)} h(s,X^{t,x,y}_s,Y^{t,y}_s)ds  \right],
														\end{equation}
														where $\varrho\in\R$. Then $u\in C^{p,0}_{\pol,T}(\R,\R_+)$. 
													\end{lemma}
													\begin{proof}
														We set
														\begin{align*}
														u_1(t,x,y)=\E\left[e^{\varrho (T-t)} g(X^{t,x,y}_T,Y^{t,y}_T)\right], \qquad
														u_2(t,x,y)=\E\left[\int_t^Te^{\varrho (s-t)} h(s,X^{t,x,y}_s,Y^{t,y}_s)ds  \right]
														\end{align*}
														and we show that, for $i=1,2$, $u_i\in C^{p,0}_{\pol,T}(\R,\R_+)$. We prove it for $i=2$, the case $i=1$ being similar and easier.

														Fix $(t,x,y)\in [0,T]\times \R\times \R_+$ and let $(t_n,x_n,y_n)_n\subset [0,T]\times\R\times\R_+$ be such that $(t_n,x_n,y_n)\to (t,x,y)$ as $n\to\infty$. One can easily prove that, for every fixed $s\geq t_n\vee t$,  $(X_s^{t_n,x_n,y_n},Y_s^{t_n,y_n})\rightarrow(X_s^{t,x,y},Y_s^{t,y})$ in probability. 
														%		Since $ g$ is continuous, $ g(X_T^{t_n,x_n,y_n},Y_T^{t_n,y_n})\rightarrow  g(X_T^{t,x,y},Y_T^{t,y})$ in probability as well. 
														%		Take now $p>1$. By the polynomial growth of $g$ and \ref{moments_XY}, 
														%		\begin{align*}
														%		\sup_{n} \E[	| g(X_T^{t_n,x_n,y_n},Y_T^{t_n,y_n})|^p]&\leq \sup_nC\E[1+|X_T^{t_n,y_n}|^{ap}+(Y_T^{t_n,y_n})^{ap}]<\infty.
														%		\end{align*}
														%		Thus, $( g(X_T^{t_n,x_n,y_n},Y_T^{t_n,y_n}))_n$ is uniformly integrable, so $ g(X_T^{t_n,x_n,y_n},Y_T^{t_n,y_n})\rightarrow  g(X_T^{t,x,y},Y_T^{t,y})$ in $L^1$. Therefore, $u_1$ is continuous in $[0,T]\times\R\times\mathcal{D}$.  Moreover, $u_1\in C_{\pol,T}(\R\times\R_+)$  as a consequence of the fact that $ g$ has polynomial growth  and of \ref{moments_XY}. 
														%		
														%		We study now $u_2$, which we write as
														We write $u_2$ as 
														$$
														u_2(t,x,y)=\int_0^T\ind{s>t}e^{\varrho (s-t)}\E\left[ h(s,X^{t,x,y}_s,Y^{t,y}_s) \right]ds 
														$$
														Since $ h$ is continuous,	for $s>t_n\vee t$ the sequence $( h(s,X^{t_n,x_n,y_n}_s,Y^{t_n,y_n}_s))_n$ converges in probability to $ h(s,X^{t,x,y}_s,Y^{t,y}_s)$.  By the polynomial growth of $h$ and \eqref{moments_XY}, for $p>1$ we have
														\begin{align}\label{a}
														\sup_{n} \E[	| h(X_T^{t_n,x_n,y_n},Y_T^{t_n,y_n})|^p]&\leq \sup_nC\E[1+|X_T^{t_n,y_n}|^{ap}+(Y_T^{t_n,y_n})^{ap}]<\infty.
														\end{align}
														Thus, $( h(X_T^{t_n,x_n,y_n},Y_T^{t_n,y_n}))_n$ is uniformly integrable, so $ h(X_T^{t_n,x_n,y_n},Y_T^{t_n,y_n})\rightarrow  h(X_T^{t,x,y},Y_T^{t,y})$ in $L^1$ and
														$$
														\ind{s>t_n}\E\left[e^{\varrho (s-t_n)} h(s,X^{t_n,x_n,y_n}_s,Y^{t_n,y_n}_s) \right]\to \ind{s>t}\E\left[e^{\varrho (s-t)} h(s,X^{t,x,y}_s,Y^{t,y}_s) \right],
														$$
														a.e. $s\in [0,T] $.
														By \eqref{a}, $u_2(t_n,x_n,y_n)\to u_2(t,x,y)$ thanks to the Lebesgue's dominated convergence and moreover, $u_2$ grows polynomially. So, $u_2\in \mathcal{C}_{\pol,T}(\R\times\R_+)$.
														
														Fix now $p\neq\infty$. We have
														%		\begin{align*}
														%		&\sup_{t\leq T}\|u_1(t,\cdot,y)\|_{L^p(\R,dx)}
														%		= \sup_{t\leq T}\left\|\E\left[e^{\varrho (T-t)} g(X^{t,\cdot,y}_T,Y^{t,y}_T) \right]\right\|_{L^p(\R,dx)}\\
														%		&\qquad\leq C\sup_{t\leq T}\E\left[ \left\| g(X^{t,\cdot,y}_T,Y^{t,y}_T)\right\|^p_{L^p(\R,dx)}\right]^{1/p}
														%		=C\sup_{t\leq T}\E\left[ \left\| g(\cdot+H^{t,y}_T,Y^{t,y}_T)\right\|^p_{L^p(\R,dx)}\right]^{1/p}\\
														%		&\qquad =C\sup_{t\leq T}\E\left[ \left\| g(\cdot,Y^{t,y}_T)\right\|^p_{L^p(\R,dx)}\right]^{1/p}
														%		\leq C\sup_{t\leq T}(1+\E[(Y_T^{t,y})^{pa}])^{1/p}.
														%		\end{align*}
														%		Similarly,
														\begin{align*}
														&\sup_{t\leq T}\|u_2(t,\cdot,y)\|_{L^p(\R,dx)}
														= \sup_{t\leq T}\left\| \E\left[\int_t^Te^{\varrho (s-t)} h(s,X^{t,\cdot,y}_s,Y^{t,y}_s)ds  \right]\right\|_{L^p(\R,dx)}\\
														&\quad\leq C\sup_{t\leq T}\E\left[ \int_t^T\left\| h(s,X^{t,\cdot,y}_s,Y^{t,y}_s)\right\|^p_{L^p(\R,dx)}\right]^{1/p}
														= C\sup_{t\leq T}\E\left[ \int_t^T\left\| h(s,\cdot +H^{t,y}_s,Y^{t,y}_s)\right\|^p_{L^p(\R,dx)}\right]^{1/p}\\
														&\quad = C\sup_{t\leq T}\E\left[ \int_t^T\left\| h(s,\cdot ,Y^{t,y}_s)\right\|^p_{L^p(\R,dx)}\right]^{1/p}
														\leq C T\sup_{t\leq s\leq T}(1+\E[(Y_s^{t,y})^{pa}])^{1/p}
														\end{align*}
														in which we have used twice the Cauchy-Schwarz inequality. Then, by using \eqref{moments_XY}, we have $u_2\in C_{\pol,T}^{p,0}(\R,\R_+)$. The case $p=\infty$ follows the same lines.
													\end{proof}
													To simplify the notation, from now on we set $\E^{t,x,y}[\cdot]=\E[\cdot|X_t=x,Y_t=y]$ and $\mathcal{O}=\R\times(0,\infty)$.. 
													\begin{lemma}\label{u-pde}
														Let $ g\in \mathcal{C}_{\pol}(\bar{\mathcal{O}})$ and $ h\in C_{\pol,T}(\bar{\mathcal{O}})$ be such that $\mathcal{O}\ni z\mapsto  h(t,z)$ is locally H\"older continuous uniformly on the compact sets of $[0,T)$. %in $t\in [0,S]$, for every $S\in[0,T]$. 
														Let $u$ be defined in \eqref {u-g-f}. Then, $u\in\mathcal{C}([0,T]\times\bar{\mathcal{O}})\cap\mathcal{C}^{1,2}([0,T)\times \mathcal{O})$ and solves the PIDE
														\begin{equation}\label{PDE-u}
														\begin{cases}
														\partial_t u+ \L u+\varrho u= h,\qquad&\mbox{ in }[0,T)\times\mathcal{O},\\
														u(T,z)= g(z),\qquad&\mbox{ in }\mathcal{O}.
														\end{cases}	
														\end{equation}
														Moreover, if the Feller condition holds, that is, $2\kappa\theta\geq \sigma^2$, then $u$ is the unique solution to \eqref{PDE-u} in the class $C_{\pol,T}(\bar{\mathcal{O}})$.
														
													\end{lemma}
													%		The proof employs standard techniques, see e.g. Proposition 3.2 in \cite{ET} with the use of  classical results in parabolic PIDEs theory from \cite{GM, MP}. The uniqueness of the solution under the Feller condition follows from the fact that the CIR process never hits $0$. So, we omit this proof.
													
													\begin{proof}
														Let $S\in[0,T)$,  $\mathcal{R}=\R \times (\epsilon,\infty)$, $\epsilon>0$,  $Q=[0,S)\times\mathcal{R}$ and consider the PIDE problem
														\begin{equation*}
														\begin{cases}
														\partial_tv+ \L v+\varrho v= h,\qquad&\mbox{ in }Q,\\
														v=u,\qquad&\mbox{ in }\partial_0Q,
														\end{cases}	
														\end{equation*}
														$\partial_0Q$ denoting the parabolic boundary of $Q$. The coefficients satisfy in $Q$ all the classical assumptions (see e.g.  \cite{GM, MP}), so a unique (bounded) solution $v\in C^{1,2}([0,T)\times\mathcal{R})\cap C([0,T]\times\bar{\mathcal{R}})$ actually exists (and have H\"older continuous derivatives $v_t$, $\nabla_z v$ and $D^2_zv$ in $\bar{Q}$). As a consequence,
														$$
														Z_s:=e^{\varrho s}v(s,X_s,Y_s)-\int_t^s e^{\varrho r} h(r,X_r,Y_r)dr
														$$
														is a martingale over $[t,S\wedge \tau_\mathcal{R}]$, where $\tau_\mathcal{R}$ denotes the exit time of  $(X,Y)$ from $\mathcal{R}$. Then,
														\begin{align*}
														&e^{\varrho t}v(t,x,y)
														=\E^{t,x,y}(Z_t)=\E^{t,x,y}(Z_{S\wedge \tau_\mathcal{R}})\\
														&=\E^{t,x,y}\Big[e^{\varrho {S\wedge \tau_\mathcal{R}}}u(S\wedge \tau_\mathcal{R},X_{S\wedge \tau_\mathcal{R}},Y_{S\wedge \tau_\mathcal{R}})-\int_t^{S\wedge \tau_\mathcal{R}} e^{\varrho r} h(r,X_r,Y_r)dr\Big].
														\end{align*}
														Now, 
														%for $s\geq t$,
														%$$
														%e^{\varrho s}u(s,x,y)
														%=\E^{s,x,y}\Big[e^{\rho T} g(X_T,Y_T)-\int_s^{T} e^{\varrho r} h(r,X_r,Y_r)dr\Big]
														%$$
														%and 
														by the strong Markov property,
														\begin{align*}
														& e^{\varrho {S\wedge \tau_\mathcal{R}}}u(S\wedge \tau_\mathcal{R},X_{S\wedge \tau_\mathcal{R}},Y_{S\wedge \tau_\mathcal{R}})
														=\E\Big[e^{\rho T} g(X_T,Y_T)-\int_{S\wedge \tau_{\mathcal{R}}}^{T} e^{\varrho r} h(r,X_r,Y_r)dr\,\Big|\, \mathcal{F}_{S\wedge \tau_{\mathcal{R}}}\Big].
														\end{align*}
														By replacing above, it follows that $v\equiv u$ in $Q$. Whence, the first assertion is proved.   %\textcolor{red}{Moreover, standard results on the successive derivatives of solutions of PDEs (see e.g.  \cite{GM}) give that $u\in \mathcal{C}^\infty([0,S]\times \mathcal{R})$, and  the first assertion is proved.}
														Suppose now that $2\kappa\theta\geq \sigma^2$ and that $ g$  has polynomial growth. Let $w\in\mathcal{C}([0,T]\times \bar{\mathcal{O}})$ denote a solution to \eqref{PDE-u} with polynomial growth. We prove that $w=u$.
														Let $S_n<T$ and let $\mathcal{R}_n$ denote a sequence rectangles as before such that $Q_n=[0,S_n)\times \mathcal{R}_n\uparrow[0,T)\times \mathcal{O}$. Let $w_n$ the unique solution to
														\begin{equation*}
														\begin{cases}
														\partial_t w_n+ \L w_n+\varrho w_n= h,\qquad&\mbox{ in }Q_n,\\
														w_n=w,\qquad&\mbox{ in }\partial_0Q_n.
														\end{cases}	
														\end{equation*}    
														Since $w$ trivially solves the above PIDE problem, we get $w_n=w$ and
														$$
														e^{\varrho t} w(t,x,y)=\E^{t,x,y}\Big[e^{\varrho {S_n\wedge \tau_{\mathcal{R}_n}}}w(S_n\wedge \tau_{\mathcal{R}_n},X_{S_n\wedge \tau_{\mathcal{R}_n}},Y_{S_n\wedge \tau_{\mathcal{R}_n}})-\int_t^{S_n\wedge \tau_{\mathcal{R}_n}} e^{\varrho r} h(r,X_r,Y_r)dr\Big].
														$$
														%		\begin{align*}
														%	&e^{\varrho t} w(t,x,y)=e^{\varrho t} w_n(t,x,y)\\
														%	&=\E^{t,x,y}\Big[e^{\varrho {S_n\wedge \tau_{\mathcal{R}_n}}}w(S_n\wedge \tau_{\mathcal{R}_n},X_{S_n\wedge \tau_{\mathcal{R}_n}},Y_{S_n\wedge \tau_{\mathcal{R}_n}})-\int_t^{S_n\wedge \tau_{\mathcal{R}_n}} e^{\varrho r} h(r,X_r,Y_r)dr\Big].
														%	\end{align*}
														Now, as $n\to\infty$, one has $\tau_{\mathcal{R}_n}\uparrow \infty$ because, by the Feller condition, $\P^{t,y}(Y_s>0\,\forall s)=1$. Then, we pass to the limit  and since $w$ is continuous  and has polynomial growth, we easily obtain $w\equiv u$.
														%	\begin{align*}
														%	w(t,x,y)
														%	%		&=e^{-\varrho t}\E^{t,x,y}\Big[e^{\varrho {T}}w(T,X_T,Y_T)-\int_t^{T} e^{\varrho r} h(r,X_r,Y_r)dr\Big]\\
														%	&=\E^{t,x,y}\Big[e^{\varrho (T-t)} g(T,X_T,Y_T)-\int_t^{T} e^{\varrho (r-t)} h(r,X_r,Y_r)dr\Big]=u(t,x,y).
														%	\end{align*}
													\end{proof}
													
													\begin{lemma}
														\label{lemma-reg}
														
														Let $u$ be defined in \eqref{u-g-f}, with $g$ and $h$ such that, as $j=0,1$,
														%		, for $i=1,2$, $\partial^i_xg\in C^{2-i}_{\pol}(\bar{\mathcal{O}})$ and 
														%		$\partial^i_xh\in C^{2-i}_{\pol,T}(\bar{\mathcal{O}})$.
														$\partial_x^{2j}g\in C^{1-j}_{\pol}(\bar{\mathcal{O}})$  and  $\partial_x^{2j}h\in \mathcal{C}^{1-j}_{\pol,T}(\bar{\mathcal{O}})$.
														Then 
														$u\in \mathcal{C}^1_{\pol,T}(\bar{\mathcal{O}})$. Moreover, $\partial^2_xu\in \mathcal{C}_{\pol,T}(\bar{\mathcal{O}})$ and one has
														\begin{align}
														&	\partial^m_x u(t,x,y)
														=\E^{t,x,y}\left[e^{\varrho (T-t)} \partial^m_x g(X_T,Y_T)-\int_t^Te^{\varrho (s-t)} \partial^m_x h(s,X_s,Y_s)ds  \right],\quad m=1,2,\label{u-x-xx}\\
														&	\partial_y u(t,x,y)
														=\E^{t,x,y}\left[e^{(\varrho-\kappa) (T-t)} \partial_yg(X^*_T,Y^*_T)\right]\nonumber \\&\qquad\qquad\qquad+\E\left[\int_t^Te^{(\varrho-\kappa) (T-s)} \Big[\partial_yh+\frac 12 \partial^2_xu+\mathfrak{b}\partial_xu\Big](s,X^*_s,Y^*_s)ds  \right],\label{u-y}
														\end{align}
														where
														%		$$
														%		g_*=\partial_yg\quad\mbox{and}\quad h_*=\partial_yh+\frac 12 \partial_xu-\frac 12 \partial^2_xu,
														%		$$
														%		and 
														$(X^*_t,Y^*_t)$ solves  \eqref{SDE} with new parameters 	$\rho_*=\rho$, $\mathfrak{a}_*=\mathfrak{a}+\rho\sigma$, $\mathfrak{b}_*=\mathfrak{b}$,  $\kappa_*=\kappa$, $\theta_*=\theta+\frac{\sigma^2}{2\kappa}$, $\sigma_*=\sigma$.
														%		\begin{align*}
														%		\rho_*=\rho,\quad \sigma_*=\sigma,\quad r_*=r+\rho\sigma,\quad \kappa_*=\kappa, \quad \theta_*=\theta+\frac{\sigma^2}{2\kappa}.
														%		\end{align*}
													\end{lemma}
													
													\begin{proof}
														First, the stochastic flow w.r.t. $x$ is differentiable (here, $(X^*)^{t,x,y}_s=x+Z^{ t,y}_s$ and $Z^{ t,y}_s$ does not depend on $x$). Hence, by using the polynomial growth hypothesis, by  \eqref{u-g-f} one gets  \eqref{u-x-xx}. Let us prove \eqref{u-y}.
														
														By Lemma \ref{u-pde} $u$  solves \eqref{PDE-u}. So, setting $v=\partial_y u$, by derivating \eqref{PDE-u} one has
														$$
														\begin{cases}
														\partial_t v+ \L_* v+\varrho_* v= h_*,\qquad&\mbox{ in }[0,T)\times\mathcal{O},\\
														v(T,z)= g_*(z),\qquad&\mbox{ in }\mathcal{O}.
														\end{cases}	
														$$
														where $\L_*$ is the infinitesimal generator of $(X^*,Y^*)$ and 
														$\varrho_*=\varrho-\kappa$, $h_*=\partial_y h-\mathfrak{b}\partial_xu-\frac 12\partial^2_xu$, $g_*=\partial_yg$.
														By using \eqref{u-x-xx} and  Lemma \ref{lemma_reg}, $h_*\in C_{\pol, T}(\bar{\mathcal{O}})$. Moreover, the Feller condition $2\kappa_*\theta_*\geq \sigma^2_*$ holds, and by Lemma \ref{u-pde} the unique solution with polynomial growth in $(x,y)$ to the above PIDE is
														$$
														\bar v(t,x,y)
														=\E^{t,x,y}\left[e^{\varrho (T-t)} g_*(X^*_T,Y^*_T)-\int_t^Te^{\varrho (s-t)} h_*(s,X^*_s,Y^*_s)ds  \right].
														$$
														In order to identify $\bar v$ with $v=\partial_yu$ we would need to know that $\partial_yu\in C_{\pol, T}(\mathcal{O})$.  If the diffusion coefficient of $Y^*$ was more regular, one could  use arguments from the stochastic flow. But this is not the case, hence we  use a density argument inspired by \cite{ET}. 
														
														For $k\geq 1$, let $\varphi_k$  be a $C^\infty(\R)$ approximation of $\sqrt{|y|}$ such that $\varphi_k(y)\geq 1/k$, $\varphi_k(y)\to \sqrt{|y|}$ uniformly on the compact sets of $[0,+\infty)$ and $\varphi^2_k$ is Lipschitz continuous uniformly in $k$ (which means that $\varphi_k\varphi'_k$ is bounded uniformly in $k$).  Consider the diffusion process $(X^k,Y^k) $ defined by
														\begin{equation}\label{SDE-n}
														\begin{cases}
														dX^k_t=\left( \mathfrak{a}+\mathfrak{b}Y^k_t\right)dt +\varphi_k(Y^k_t)dB_t+dH_t,\\
														dY^k_t=\kappa(\theta-Y^k_t)dt+\sigma\varphi_k(Y^k_t)dW_t, 
														\end{cases}
														\end{equation}
														whose generator is
														$$
														\L_ku= \frac {\varphi^2_k(y)} 2 \left(\partial^2_x u+ 2\rho\sigma\partial_x\partial_y u  
														+ \sigma^2 \partial^2_y u\right)
														+\left(\mathfrak{a}+\mathfrak{b} y \right)\partial_x u+\kappa(\theta-y)\partial_y u  + \II u.
														%	\int_{-\infty}^\infty(u(t,x+\zeta,y)-u(t,x,y) )\nu(\zeta )d\zeta 
														$$
														Set
														$$
														u^k(t,x,y)=\E^{t,x,y}
														\left[e^{\varrho (T-t)} g(X^k_T,Y^k_T)-\int_t^Te^{\varrho (s-t)}h(s, X^k_s, Y^k_s)ds\right].
														$$
														Le us first show that $\partial_yu^k\in C_{\pol,T}(\mathcal{O})$. Since the diffusion coefficients associated to $(X^k,Y^k)$ are good enough, we can consider the first variation process: by calling $Z^{k,t,x,y}_s=(\partial_yX^{k,t,x,y}_s, \partial_y Y^{k,t,x,y}_s)$, we get
														\begin{align*}
														\partial_y  u^k(t,x,y)
														=&\E\left[e^{\varrho (T-t)}\left\langle \nabla_{x,y}g(X^{k,t,x,y}_T,Y^{k,t,x,y}_T),Z^{k,t,x,y}_T\right\rangle\right]\\
														&-\int_t^Te^{\varrho (s-t)}
														\E\left[\left\langle\nabla_{x,y}h(s,X^{k,t,x,y}_s,Y^{k,t,x,y}_s), Z^{k,t,x,y}_s\right\rangle\right]ds.
														\end{align*}
														The functions $g,h$ and their derivatives have polynomial growth, so
														\begin{align*}
														\left|\partial_y u^k(t,x,y)\right|
														\leq &\E\left[C(1+|X^{k,t,x,y}_T|^a+|Y^{k,t,x,y}_T|^a)|Z^{k,t,x,y}_T|\right]\\
														&+\int_t^Te^{\varrho (s-t)}
														\E\left[C(1+|X^{k,t,x,y}_s|^a+|Y^{k,t,x,y}_s|^a)|Z^{k,t,x,y}_s|\right]ds
														\end{align*}
														and the usual $L^p$-estimates give
														$$
														\sup_{t<T}\left|\partial_y  u^k(t,x,y)\right|\leq C_k(1+|x|^{a_k}+y^{a_k}),
														$$
														for suitable constants $C_k,a_k>0$.
														Moreover, from the standard theory of parabolic PIDEs, $u^k$ is a solution to
														$$
														\begin{cases}
														\partial_t  u^k+ \L_k u^k+\varrho u^k= h,\qquad&\mbox{ in }[0,T)\times\mathcal{O},\\
														u^k(T,z)= g(z),\qquad&\mbox{ in }\mathcal{O}.
														\end{cases}	
														$$
														By differentiating, $v^k=\partial_y  u^k$ solves the problem
														$$
														\begin{cases}
														\partial_t  v^k+ \L_{k,*} v^k+\varrho_* v^k= h_{k,*},\qquad&\mbox{ in }[0,T)\times\mathcal{O},\\
														v^k(T,z)= g_{*}(z),\qquad&\mbox{ in }\mathcal{O}.
														\end{cases}	
														$$
														where
														\begin{align*}
														\L_{k,*}v=
														& \frac {\varphi^2_k(y)} 2 \left(\partial^2_xv+ 2\rho\sigma\partial_x\partial_y v
														+ \sigma^2 \partial^2_yv \right)\\
														&+ \left(\mathfrak{a}+\mathfrak{b}y+2\rho\sigma\varphi_k\varphi'_k(y)\right)\partial_xv
														+\left(\kappa(\theta-y)+\sigma^2\varphi_k\varphi'_k(y)\right)\partial_yv+\II v 
														%		&+\int_{-\infty}^\infty(v(t,x+\zeta,y)-v(t,x,y) )\nu(\zeta )d\zeta  
														\end{align*}
														and 
														$
														h_{k,*}=\partial_yh-\mathfrak{b} \partial_x u^k-\varphi_k\varphi'_k(y)\partial^2_x u^k.
														$
														By developing the same arguments as before, we get $h_{k,*}\in C_ {\pol, T}(\bar{\mathcal{O}})$.
														The PIDE for $v^k$ has a unique solution in $C_{\pol,T}(\mathcal{O})$ (recall that, by construction, the second order operator is uniformly elliptic). Thus, the Feynman-Kac formula gives
														$$
														\partial_yu^k(t,x,Y)
														=\E^{t,x,y}\left[e^{\varrho (T-t)} g_*(X^{k,*}_T,Y^{k,*}_T)-\int_t^Te^{\varrho (s-t)} h_{k,*}(s,X^{k,*}_s,Y^{k,*}_s)ds  \right],
														$$
														where $(X^{k,*},Y^{k,*})$ is the diffusion with infinitesimal generator given by $\L_{k,*}$.
														Now, the standard $L^p$ estimates for $(X^k,Y^k)$ and $(X^{k,*},Y^{k,*})$ hold uniformly in $k$ (recall that $\varphi_k$ is sublinear uniformly in $k$ and $\varphi_k\varphi'_k$ is bounded uniformly in $k$): for every $p\geq 1$ there exist $C,a>0$ such that
														$$
														\sup_k\sup_{t\leq T}\E^{t,x,y}\left( |X^k_t|^p+|Y^k_t|^p\right)
														+\sup_k\sup_{t\leq T}\E^{t,x,y}\left( |X^{k,*}_t|^p+|Y^{k,*}_t|^p\right)\leq C(1+|x|^a+|y|^a).
														$$
														This gives that
														$$
														\sup_k\sup_{t<T}|u^k(t,x,y)|+\sup_k\sup_{t<T}\left|\partial_y u^k(t,x,y)\right|\leq C(1+|x|^a+|y|^a),
														$$
														for suitable $C,a>0$ (possibly different from the ones above). Moreover, using the stability results of \cite{BMO} one obtains
														\begin{align*}
														\lim_{n\to\infty}u^k(t,x,y)=u(t,x,y)\quad \mbox{and}\quad
														\lim_{n\to\infty}\partial_y u^k(t,x,y)=v(t,x,y)
														\end{align*}
														for every $(t,x,y)\in [0,T)\times \mathcal{O}$. And thanks to the above uniform polynomial bounds for $u^k$ and $\partial_yu^k$, for every $\phi\in C^\infty (\mathcal{O})$ with compact support we easily get
														\begin{align*}
														&\int v(t,x,y) \phi(x,y)dxdy
														=\int \lim_k \partial_y  u^k(t,x,y)\phi(x,y)dxdy\\
														&=-\int \lim_k u^k(t,x,y)\partial_y\phi(x,y)dxdy
														=-\int u(t,x,y)\partial\phi(x,y)dxdy.
														\end{align*}
														Therefore, $v(t,x,y)=\partial_y u(t,x,y)$ in $[0,T)\times \mathcal{O}$. The statement now follows.
														
													\end{proof}

													We can now prove the result which this section is devoted to.

													\begin{proof}[Proof of Proposition \ref{prop-reg-new}]
														We follow an induction on $q$. If $q=0$, 	Lemma \ref{lemma_reg} gives the result. Suppose the statement is true up to $q-1\geq 1$ and let us prove it for $q$. 
														
														Take $f$ such that	$\partial_x ^{2j}f\in C^{p,q-j}_{\pol}(\R,\R_+)$ for every $j=0,1,\ldots,q$. Then, by induction, $\partial^{l}_t\partial^m_x\partial^n_y u\in C^{p,0}_{\pol,T}(\R,\R_+)$ when $2l+m+n\leq q-1$. So, we just need to prove that $\partial^l_t\partial^m_x\partial^n_y u\in C^{p,0}_{\pol, T}(\R,\R_+)$ for any $l,m,n$ such that $2l+m+n=q$. 
														
														Assume first $l=0$. For $n=0$, we use that $X_T^{t,x,y}=x+Z_T^{t,y}$ and we get
														$
														\partial^{m}_xu(t,x,y)=\E^{t,x,y}\big[\partial^m_xf(X_T,Y_T)\big].
														$
														Since $\partial^m_xf\in C^{p,0}_{\pol}(\R,\R_+)$ for any $m\leq 2q$,  by Lemma \ref{lemma_reg} we obtain  $\partial^{m}_xu\in C^{p,0}_{\pol,T}(\R,\R_+)$	for every $m\leq 2q$.  
														
														Fix now $n>0$ and $m\geq 0$.   Recursively applying Lemma \ref{lemma-reg}, we get  formula \eqref{stoc_repr-new}. Let us stress that,
														because of the presence of the derivatives $\partial^{m+2}_x\partial^{n-1}_yu$ and $\partial^{m+1}_x\partial^{n-1}_yu$ in \eqref{stoc_repr-new}, the recursively application of Lemma \ref{lemma-reg} gives the constraint $m+2n\leq q$. 
														Then, by Lemma \ref{lemma_reg}, it follows that $\partial^{m}_x\partial^{n}_yu \in C^{p,0}_{\pol, T}(\R,\R_+)$	for every $m, n\in \N$ such that $m+2n\leq 2q$, and in particular when $m+n=q$.
														
														Consider now the case $l>0$. By \eqref{stoc_repr-new}, Lemma \ref{u-pde} ensures that if $m+2n\leq 2q$ then $u_{n,m}=\partial^{m}_x\partial^{n}_y u$ solves
														\begin{equation}\label{pdeit-new}
														\begin{cases}
														\partial_t u_{m,n}+ \L_{n}u_{m,n}-n\kappa u_{m,n}=-n \big[\frac 12 u_{m+2,n-1}+ \mathfrak{b}u_{m+1,n-1}\big]\quad \mbox{in }[0,T)\times\mathcal{O},\\
														u_{m,n}(T,x,y)=\partial^m_x\partial^n_y f(x,y)\quad \mbox{in }\mathcal{O},
														\end{cases}
														\end{equation}
														where $\L_n$ is the generator in \eqref{app-L} with the (new) parameters in \eqref{parameters-new}.
														Therefore, the general case concerning  $\partial^l_t\partial^m_x\partial^n_y u$ with $2l+m+n=q$ follows by an iteration on $l$: by \eqref{pdeit-new},
														\begin{align*}
														\partial^{l}_t\partial^{m}_x\partial^n_y u=-\L_n\partial^{l-1}_t\partial^{m}_x\partial^n_y u
														+n\kappa \partial^{l-1}_t\partial^{m}_x\partial^n_y u
														-n\Big[\frac 12\partial^{l-1}_t\partial^{m+2}_x\partial^{n-1}_y u+\mathfrak{b}\partial^{l-1}_t\partial^{m+1}_x\partial^{n-1}_y u\Big].
														\end{align*}
													\end{proof}

													\section{The American case in the Heston/Bates model}\label{sect-conv-am}
													In this section we focus on the American case. 
													%We first apply our convergence results to the Heston or Bates model, as already done for the European case.  We will show  that, in this models, we need stronger regularity assumptions on the function $f$. Therefore,  we then consider other stochastic volatility models which  allow us to  weaken  the  assumptions on the payoff function as long as we strengthen the regularity of the coefficients of the models. 		
													%	\subsection{The Heston or Bates model}	
													We first  prove a simple lemma which better specifies the behaviour of the moments in the Heston and Bates model.
													\begin{lemma}
														For every $p\geq 2$ there exists $C>0$ (depending on $p$ and on the model parameters) such that 
														\begin{align}
														&	\sup_{t\in [nh,(n+1)h]}\E[ |X_{(n+1)h}^{t,x,y}|^p]\leq(1+Ch)(1+|x|^p+y^p),\label{stima_am_X}\\
														&	\sup_{t\in [nh,(n+1)h]}\E[ (Y_{(n+1)h}^{t,y})^p]\leq(1+Ch)(1+y^p).\label{stima_am_Y}
														\end{align}
													\end{lemma}
													\begin{proof}
														It can be easily proved that there exists $C>0$ such that 
														\begin{equation}\label{stime_momenti}
														\sup_{t\in [0,T]}\E[ |X_t|^p]\leq C(1+|x|^p+y^p),\qquad	\sup_{t\in [0,T]}\E[ (Y_t^{t,y})^p]\leq C(1+y^p).
														\end{equation}
														We start by proving \eqref{stima_am_Y}. Let us fix $p\geq1$. By using It\^{o}'s Lemma, for any $t\in [nh,(n+1)h]$ we have
														\begin{align*}
														(Y^{t,y}_{(n+1)h})^p
														&=y^p+p\int_t^{(n+1)h}\left(\Big(\kappa\theta -\frac  {p-1} 2  \sigma^2\Big) (Y^{t,y}_{s})^{p-1}-\kappa  (Y^{t,y}_s)^{p}\right)ds \\&\qquad+p\sigma\! \int_t^{(n+1)h}(Y^{t,y}_{s})^{p-\frac 1 2 }dW_s.
														\end{align*}
														Passing to the expectation and using \eqref{stime_momenti}, we can find $C>0$ (depending on $p$ and on the coefficients of the model) such that 
														\begin{align*}
														\sup_{t\in [nh,(n+1)h]}\E[(Y^{t,y}_{(n+1)h})^p]
														&\leq y^p+hC(1+y^{p-1}+y^p)
														%	\\&= y^p+hC(1+y^{p-1}\left(\ind{y\leq 1}+\ind{y>1})+y^p\right)\\& 
														\leq (1+2Ch)(1+y^p),
														\end{align*}
														from which 	\eqref{stima_am_Y} follows. As regards \eqref{stima_am_X}, again by It\^{o}'s Lemma, for $t\in [nh,(n+1)h]$ we get
														\begin{align*}
														|X_{(n+1)h}^{t,x,y}&|^{2p}=x^p+\int_t^{(n+1)h}\left[ 2p\mu_X(Y_{s}^{t,y})(X_{s^-}^{t,x,y})^{2p-1}+ p(2p-1) \sigma_X^2(Y_{s}^{t,y}) (X_{s^-}^{t,x,y})^{2p-2}\right]ds\\
														& + \int_t^{(n+1)h} (X_{s^-}^{t,x,y}+J_{N_{s}})^{2p}-(X_{s^-}^{t,x,y})^{2p}dK_s+\int_t^{(n+1)h}2p\sigma_X(Y_s^{t,y})(X_{s^-}^{t,x,y})^{2p-1}dB_s,
														\end{align*}
														$K$ denoting the Poisson process driving the compound Poisson process $H$, whose associated L\'evy measure is $\nu$. Passing to the expectation, and using the martingale properties (which hold thanks to \eqref{stime_momenti}) we get
														\begin{align*}
														\E[|X_{(n+1)h}^{t,x,y}|^{2p}]&
														=x^{2p}\!+\!\int_t^{(n+1)h}\!\!\!\left[ \E[2p\mu_X(Y_s^{t,y})(X_s^{t,x,y})^{2p-1}\!\!+ p(2p-1) \sigma_X^2(Y_s^{t,y}) (X_s^{t,x,y})^{2p-2}]\right]ds\\
														&\quad + \int_t^{(n+1)h}ds\int \E[(X_s^{t,x,y}+z)^{2p}-(X_s^{t,x,y})^{2p}]\nu(dz).
														\end{align*}
														\eqref{stima_am_X} now follows by using H\"{o}lder inequality, the estimate \eqref{stime_momenti} and the existence of all moments under $\nu$.
													\end{proof}
													
													Again, we approximate the CIR process with the Markov chain discussed in Section \ref{sect-CIR} and we consider the two  finite difference operators introduced in Section \ref{sect-l2} and \ref{sect-linf}. Therefore, we get the following convergence rate result.

													%	\begin{theorem}\label{Bates_conv_am}
													%		Let $(Y^h_n)_{n=0,\dots, N}$ be the Markov chain introduced in Section \ref{sect-CIR} for the approximation of the CIR process $Y$. Let $\tilde u^h_n$ be defined in \eqref{backward2} and $u^h_h$ be given by \eqref{backward-ter0-bis} with the choice
													%		$$
													%		\Pi^h_{\dx}(y)=(A^h_{\dx})^{-1}B^h_{\dx}(y).
													%		$$
													%		$(i)$ Suppose that $A^h_{\dx}(y)$ and $B^h_{\dx}(y)$ are defined in \eqref{A} and \eqref{B} respectively and $f\in C^\infty_\pol(\R\times\D)$ is such that there exist $C,a>0$ with
													%		$
													%	%	|\partial^{l'}_x\partial^l_yf(x,y)|\leq C(1+|x|^a+y^a), \qquad \qquad  
													%		|\partial^{l'}_x\partial^l_yf(\cdot,y)|_{L^2(\R,dx)}		\leq C(1+y^a).
													%		$ Then, there exist $\bar h<1$ and  $C>0$ such that for every $h<\bar h$ and $\dx<1$ one has
													%		$$
													%		| \tilde u^h_0(\cdot,Y_0)-u^h_{0}(\cdot,Y_0)|_{2} \leq CT(h+ \dx^2).
													%		$$
													%		$(ii)$ Suppose that $A^h_{\dx}(y)$ and $B^h_{\dx}(y)$ are defined in \eqref{A2} and \eqref{B} respectively and $f\in C^\infty_\pol(\R\times\D)$ is such that there exist $C,a>0$ with
													%		$
													%	%	|\partial^{l'}_x\partial^l_yf(x,y)|\leq C(1+|x|^a+y^a), \qquad \qquad  
													%	|\partial^{l'}_x\partial^l_yf(\cdot,y)|_{L^\infty(\R,dx)}
													%		\leq C(1+y^a).
													%		$ Then, there exist $\bar h<1$ and  $C>0$ such that for every $h<\bar h$ and $\dx<1$ one has
													%		$$
													%		| u^h_0(\cdot,Y_0)-u^h_{0}(\cdot,Y_0)|_{\infty} \leq CT(h+ \dx).
													%		$$
													%	\end{theorem}
													
													\begin{theorem}\label{Bates_conv_am}
														Let $(X,Y)$ be the solution to \eqref{SDE_bates} and let $(Y^h_n)_{n=0,\dots, N}$ be the Markov chain introduced in Section \ref{sect-CIR} for the approximation of the CIR process $Y$. Let $\tilde u^h_n$ be defined in \eqref{backward2} and $u^h_h$ be given by \eqref{backward-ter0-bis}  with the choice
														$$
														\Pi^h_{\dx}(y)=(A^h_{\dx})^{-1}B^h_{\dx}(y).
														$$
														
														\begin{itemize}
															\item[$(i)$] 
															$\mathrm{[Convergence\  in\ }l_2(\mathcal{X})]$ Suppose that 
															\begin{itemize}
																\item [$\bullet$]
																$A^h_{\dx}(y)$ and $B^h_{\dx}(y)$ are defined in \eqref{A} and \eqref{B} respectively;
																\item [$\bullet$]
																$\frac{\nu'}{\nu},\frac{\nu''}{\nu}\in L^2(\R,d\nu)$ and $\nu$ has finite moments of any order;
																\item [$\bullet$]
																$f\in C^\infty_\pol(\R\times\D)$ is such that there exist $C,a>0$ with
																$$
																%	|\partial^{l'}_x\partial^l_yf(x,y)|\leq C(1+|x|^a+y^a), \qquad \qquad  
																|\partial^{l'}_x\partial^l_yf(\cdot,y)|_{L^2(\R,dx)}		\leq C(1+y^a), \qquad l',l\in\N.
																$$ 
															\end{itemize} 
															Then, there exist $\bar h,C>0$ such that for every $h<\bar h$ and $\dx<1$ one has
															$$
															| u(0,\cdot,Y_0)-u^h_{0}(\cdot,Y_0)|_{2} \leq CT(h+ \dx^2).
															$$
															
															\item[$(ii)$] 
															$\mathrm{[Convergence\  in\ } l_\infty(\mathcal{X})]$ Suppose that 
															\begin{itemize}
																\item [$\bullet$]
																$A^h_{\dx}(y)$ and $B^h_{\dx}(y)$ are defined in \eqref{A2} and \eqref{B} respectively;
																\item [$\bullet$]
																$\frac{\nu'}{\nu},\frac{\nu''}{\nu}\in L^1(\R,d\nu)$  and $\nu$ has finite moments of any order;
																\item [$\bullet$]
																$f\in C^\infty_\pol(\R\times\D)$ is such that there exist $C,a>0$ with
																$$
																%	|\partial^{l'}_x\partial^l_yf(x,y)|\leq C(1+|x|^a+y^a), \qquad \qquad  
																|\partial^{l'}_x\partial^l_yf(\cdot,y)|_{L^\infty(\R,dx)}		\leq C(1+y^a), \qquad l',l\in\N.
																$$ 
															\end{itemize} 
															Then, there exist $\bar h,C>0$ such that for every $h<\bar h$ and $\dx<1$ one has
															$$
															| u(0,\cdot,Y_0)-u^h_{0}(\cdot,Y_0)|_{\infty} \leq CT(h+ \dx).
															$$
														\end{itemize} 
														
													\end{theorem}

													\begin{proof}	
														We prove $(i)$, $(ii)$ following in the same way.  The validity of assumptions $\mathcal{A}_1$ and $\mathcal{A}_2$ is proved in Proposition \ref{moments-2} and since $\gamma_X\equiv 1$ or $\gamma_X\equiv 0$, $\mathcal{A}_3(4\lambda c_\nu|\gamma_X|)$ trivially holds.
														So, as in the European case,  in order to apply Theorem \ref{conv-H}  it is enough  to prove that the function $v^h_n$  defined in \eqref{v^h_n} belongs to the space $C^{2,6}_{\pol, [nh,(n+1)h]}( \R, \D)$  a.e. and uniformly in $n$ and $h$. 
														
														Let us consider a function $f\in C_\pol^\infty(\R \times \mathcal{D})$ such that for any $l,l'\in \N$  there exist $C_{l',l},\,a_{l,l'}>0$ such that 
														\begin{equation}\label{fintetest}
														%|\partial^{l'}_x\partial^l_yf(x,y)|\leq C_{l',l}(1+|x|^a+y^a), \qquad \qquad 
														|\partial^{l'}_x\partial^l_yf(\cdot,y)|_{L^2(\R,dx)}
														\leq C_{l',l}(1+y^{a_{l,l'}}),\qquad y\in\D.
														\end{equation}
														We point out that in the statement of the theorem we  actually require  that there exist $C,a>0$ such that $C_{l',l}\leq C$ and $a_{l',l}\leq a$ for any $l,l'\in\N$. %In order to make the proof more comprehensive, we start with this weaker assumptions. 
														%We start with these weaker assumptions in order to underline why we are obliged to strengthen them to get the assertion.
														%It will be clear at the end of the proof why we are obliged to require these strong assumptions in order to get the assertion.
														We will use this strong assumption only at the end of the proof, when it will be clear why we need it in order to get the assertion.

														We proceed by a backward iteration. For $n=N-1$ we have $v^{h}_{N-1}(t,x,y)=\E\big[f(X^{t,x,y}_{T},Y^{t,y}_{T})\big]$. By the proof of Proposition \ref{prop-reg-new} and by using \eqref{stima_am_X} and \eqref{stima_am_Y}, we deduce that,
														if $l=0$, by using \eqref{stima_am_X}-\eqref{stima_am_Y} we have
														%\begin{align*}
														%&\sup_{t\in [N-1)h,T)}|\partial^{l'}_xv^h_{N-1}(t,x,y)|=\sup_{t\in [N-1)h,T)}|\E^{t,x,y}\left[ \partial^{l'}_xf(X_T,Y_T)\right]|\\& \qquad \leq C_{l',0}(1+|X_T|^a+Y_T^a)\leq C_{l',0}(1+C_0h)(1+|x|^a+y^a),
														%\end{align*}
														%and, similarly,
														$$
														\sup_{t\in [(N-1)h,T)}|\partial^{l'}_xv^h_{N-1}(t,\cdot,y)|_{L^2(\R,dx)} \leq C_{l',0}(1+C_0h)(1+y^{a_{l',0}}).
														$$
														On the other hand, again from the proof  of Proposition \ref{prop-reg-new}, we have  that, for $t\in [(N-1)h,T)$,
														\begin{equation}\label{stoc_repr}
														\begin{split}
														&\partial^{l'}_x\partial^{l}_yv^h_{N-1}(t,x,y)=\E\left[e^{-l\kappa (T-t)} \partial^{l'}_x\partial^l_yf(X^{l,t,x,y}_T,Y^{l,t,x,y}_T)\right]\\
														&\quad+
														l\,\E\left[\int_t^T\left[\frac 12 \partial^{l'+2}_x\partial^{l-1}_yv^h_{N-1}+b\partial^{l'+1}_x\partial^{l-1}_y v^h_{N-1}\right](s,X^{l,t,x,y}_s,Y^{l,t,x,y}_s)ds  \right],
														\end{split}
														\end{equation}
														where $b=\frac \rho{\sigma}\kappa-\frac 12$ and  $(X^l,Y^l)$ is the solution of the Heston/Bates model with new coefficients
														$r_{l}=r+l\rho\sigma$, $\kappa_{l}=\kappa$, $\theta_{l}=\theta+\frac {l\sigma^2} {2\kappa}$, $\sigma_{l}=\sigma$.
														Denote by $C_l$ the constant such that 
														$$
														%\sup_{t\in [(N-1)h,T)}\E^{t,x,y}[ |X^l_{(n+1)h}|^p]\leq(1+|x|^p+y^p)(1+C_lh),$$$$
														\sup_{t\in [(N-1)h,T)}\E^{t,y}[ (Y^l_{(n+1)h})^p]\leq(1+y^p)(1+C_lh).
														$$
														Then, if $l=1$, by \eqref{stoc_repr}  we get
														\begin{align*}
														&\sup_{ t\in [(N-1)h,T)}   |\partial^{l'}_x\partial_y v^h_{N-1}(t,\cdot,y)|\leq C_{l',1}(1+C_1h)(1+y^{a_{l',1}})\\&\qquad+h \left(  \frac 1 2  C_{l'+2,0}(1+C_1h)(1+y^{a_{l'+2,0}}) + |b| C_{l'+1,0}(1+C_1h)(1+ y^{a_{l'+1,0}})\right).
														\end{align*}
														Without loss of generality we can assume that $\frac 1 2 +|b|\leq C_1$, $C_i\leq C_{i+1}$ and that the constants $C_{l,l'}$ and $a_{l,l'}$  are nondecreasing in both $l$ and $l'$. Then, we easily deduce that
														%$$
														%\sup_{ t\in [N-1)h,T)}   |\partial^{l'}_x\partial_y v^h_{N-1}(t,x,y)|\leq C_{l'+2,1}(1+C_1h)^2(1+|x|^a+|y|^a)
														%$$
														%and
														$$
														\sup_{ t\in[(N-1)h,T)}   |\partial^{l'}_x\partial_y v^h_{N-1}(t,\cdot,y)|_{L^2(\R,dx)}\leq C_{l'+2,1}(1+C_1h)^2(1+y^{a_{l'+2,1}}).
														$$
														%where $C^{N-1}_{l'+2,1}=C_{l'+2,1}$.
														With the same arguments, if $l=2$, we get
														%$$
														%\sup_{ t\in [N-1)h,T)}   |\partial^{l'}_x\partial^2_y v^h_{N-1}(t,x,y)|\leq C_{l'+4,2}(1+C_2h)^3(1+|x|^a+|y|^a), 
														%$$
														$$
														\sup_{ t\in [N-1)h,T)}   |\partial^{l'}_x\partial^2_y v^h_{N-1}(t,\cdot,y)|_{L^2(\R,dx)}\leq C_{l'+4,1}(1+C_2h)^3(1+y^{a_{l'+4,1}}).
														$$
														By iterating, it can be easily seen that
														%$$
														%\sup_{ t\in [N-1)h,T)}   |\partial^{l'}_x\partial^l_y v^h_{N-1}(t,x,y)|\leq C_{l'+2l,l}(1+C_lh)^{l+1}(1+|x|^a+|y|^a).
														%$$
														%and
														$$
														\sup_{ t\in [N-1)h,T)}   |\partial^{l'}_x\partial^l_y v^h_{N-1}(t,\cdot,y)|_{L^2(\R,dx)}\leq C^{(h,N-1)}_{l',l}\left(1+y^{a^{(N-1)}_{l',l}}\right),
														$$
														where 
														$$
														C^{(h,N-1)}_{l',l}=C_{l'+2l,l}(1+C_lh)^{l+1},\qquad a^{(N-1)}_{l',l}=a_{l'+2l,l}. $$
														As regard the derivatives w.r.t. the time variable, again from the proof of Proposition \ref{prop-reg-new}, we have 
														\begin{align*}
														\partial^{l''}_t\partial^{l'}_x\partial^l_y v^h_{N-1}&=-\L_l\partial^{l-1}_t\partial^{l'}_x\partial^l_y v^h_{N-1}
														+l\kappa \partial^{l''-1}_t\partial^{l'}_x\partial^l_y v^h_{N-1}
														\\&\quad	-l\Big[\frac 12\partial^{l-1}_t\partial^{l'+2}_x\partial^{l-1}_y v^h_{N-1}+b\partial^{l''-1}_t\partial^{l'+1}_x\partial^{l-1}_y v^h_{N-1}\Big],
														\end{align*}
														so that
														\begin{equation}\label{dertime}
														\sup_{t\in [nh,(n+1)h)}|\partial_t^{l''}\partial_x^{l'}\partial_y^{l} v^h_{N-1}(t,\cdot, y)|_{L^2(\R,dx)}\leq cl C^{(h,N-1)}_{l'+2,l+2}\left(1+y^{a^{(N-1)}_{l',l}+l''}\right),
														\end{equation}
														where $c$ is a constant which depends on the coefficient of the model. 
														
														Therefore,
														$$
														\tilde u^h_{N-1}(x,y)=\max\{f(x,y), v^h_{N-1}((N-1)h,x,y)\}
														$$
														is a continuous function, whose derivatives, of any order, a.e. continuously exist  and for every $l', l$,
														\begin{equation}\label{oddiononcelafacciopiu}
														|\partial_x^{l'}\partial_y^{l}  \tilde u^h_{N-1}(\cdot ,y)|_{L^2(\R,dx)}\leq C^{(h,N-1)}_{l',l}\left(1+y^{a^{(N-1)}_{l',l}}\right)\quad \mbox{a.e.}.
														\end{equation}
														Note that the estimates \eqref{dertime} on the time derivatives of $v^h_{N-1}$  are not involved in the estimate  \eqref{oddiononcelafacciopiu} and, as a consequence, in the iterative procedure.
														
														At time step $n=N-2$ the function $v^h_{N-2}$ is defined by
														$$
														v^h_{N-2}(t,x,y)=\E\big[\tilde u^h_{N-1}(X^{t,x,y}_{(N-1)h},Y^{t,y}_{(N-1)h})\big], \qquad t\in [(N-2)h, (N-1)h].
														$$
														By developing  arguments already done for $n=N-1$, we get 
														%$$
														%\sup_{ t\in [N-1)h,T)}   |\partial^{l'}_x\partial^l_y v^h_{N-1}(t,x,y)|\leq C^{N-1}_{l'+2l,l}(1+C_lh)^{l+1}(1+|x|^a+|y|^a).
														%$$
														%and
														$$
														\sup_{ t\in [N-1)h,T)}   |\partial^{l'}_x\partial_y v^h_{N-2}(t,\cdot,y)|_{L^2(\R,dx)}\leq C^{(h,N-2)}_{l',l}\left(1+y^{a^{(N-2)}_{l',l}}\right),
														$$
														where 
														$$
														C^{(h,N-2)}_{l',l}=C^{(h,N-1}_{l'+2l,l}(1+C_lh)^{l+1}=C_{l'+4l,l}(1+C_lh)^{2(l+1)},\qquad a^{(N-2)}_{l',l}=a_{l'+4l,l}.$$
														Moreover
														$$
														\sup_{t\in [nh,(n+1)h)}|\partial_t^{l''}\partial_x^{l'}\partial_y^{l} v^h_{N-2}(t,\cdot, y)|_{L^2(\R,dx)}\leq cl C^{h,N-2}_{l'+,l+2}\left(1+y^{a^{N-2}_{l',l}+l''}\right).
														$$
														Therefore, the function 
														$$
														\tilde u^h_{N-2}(x,y)=\max\{f(x,y), v^h_{N-2}((N-2)h,x,y)\}
														$$
														is a continuous function, whose derivatives, of any order, a.e. continuously exist  and for every $l', l$,
														$$
														|\partial_x^{l'}\partial_y^{l}  \tilde u^h_{N-2}(\cdot ,y)|_{L^2(\R,dx)}\leq C^{h,N-2}_{l',l}\left(1+y^{a^{N-2}_{l',l}+l''}\right)\quad \mbox{a.e.},
														$$

														By iterating, we get that, at time step  $n=N-k$,  the function $v^h_{N-k}$ satisfies 
														$$
														|\partial_x^{l'}\partial_y^{l}  v^h_{N-k}(\cdot ,y)|_{L^2(\R,dx)}\leq C^{(h,N-k)}_{l',l}\left(1+y^{a^{(N-k)}_{l',l}+l''}\right)\quad \mbox{a.e.},
														$$
														where $$
														C^{(h,N-k)}_{l',l}=C_{l'+2kl,l}(1+C_lh)^{k(l+1)}, \qquad a^{(N-k)}_{l',l}=a_{l'+2kl,l}.
														$$ Again
														$$
														\sup_{t\in [nh,(n+1)h)}|\partial_t^{l''}\partial_x^{l'}\partial_y^{l} v^h_{N-k}(t,\cdot, y)|_{L^2(\R,dx)}\leq cl C^{(h,N-k)}_{l'+2,l+2}\left(1+y^{a^{(N-k)}_{l',l}+l''}\right).
														$$
														
														In order to have $v^h_n\in C^{2,6}_{\pol, [nh,(n+1)h]}( \R, \D)$ a.e. and uniformly in $n$ and $h$, we need estimates of the derivatives $\partial_x^{l'}\partial_y^{l}  v^h_n$ for $l+l'\leq 6$ which are uniform in $n$ and $h$. It is clear that for each  $k\leq N$, since $h=T/N$ and $l\leq 6$,
														$$
														(1+C_lh)^{k(l+1)}\leq e^{C_l h N(l+1)}\leq e^{7TC_6}.
														$$
														Moreover,  the assumption  that  there exist $C,a>0$ such that  $C_{l',l}\leq C$  and $a_{l',l}\leq a$ for any $l,l'\in\N$ now comes in. Thanks to this, we can   deduce that $v^h_n\in C^{2,6}_{\pol, [nh,(n+1)h]}( \R, \D)$ a.e. and uniformly in $n$ and $h$, so by Theorem \ref{conv-H}  we get the result.
													\end{proof}

													\begin{remark}
														In Theorem \ref{Bates_conv_am} we require  really strong regularity and boundedness assumptions on the test function $f$.  On the other hand, let us stress that our algorithm is strongly based on numerical analysis techniques. When these procedures are used, as far as we know,  literature is missing in results  on the rate of convergence of numerical schemes for obstacle problems.
														
														Let us mention that, in some particular cases, different approaches could in principle be followed. For example,  let us  consider the scheme introduced in Section \ref{sect-linf}, where the linear operator is given by
														$$
														\Pi^h_{\dx}(y)=(A^h_{\dx})^{-1}B^h_{\dx}(y),
														$$
														$A^h_{\dx}(y)$ and $B^h_{\dx}(y)$ being defined in \eqref{A2} and \eqref{B} respectively. Here, we have proved in Lemma \ref{lemmamatrici} that $\Pi^h_{\dx}(y)$ is a stochastic operator.  From a probabilistic point of view,  this means that the  algorithm can be written 
														through  a Markov chain (see \cite{bcz}). Then, one could apply purely probabilistic methods  to prove the convergence of the procedure, for example by developing techniques similar to the ones introduced in \cite{BP}.  On the other hand, in this case, $\Pi^h_{\dx}(y)$  is  a monotone linear operator, so another possible way to proceed is to use the theory introduced by Barles \cite{barles}, which uses viscosity solutions. In order to do this, we need  a comparison principle for viscosity solutions  of Heston-type degenerate parabolic problems (note that in Section \ref{sect-existenceanduniqueness} we have proved such a result in the case of weak solutions).
														However,  both the mentioned approaches  give in principle just the convergence, that is, no information about the rate of convergence is provided.
													\end{remark}

													\section{Appendix}	
													\subsection{Lattice properties of the CIR approximating tree}\label{app-jumps}
													The aim of this section is to prove  Propostition \ref{propjumps}. 
													For later use, let us first give some (trivial) properties of the lattice. First, by construction, $k_d(n,k)\leq k<k_u(n,k)$, so that
													$
													y^{n+1}_{k_d(n,k)}\leq y^{n+1}_k\leq y^n_k\leq y^{n+1}_{k+1}\leq y^{n+1}_{k_u(n,k)}.
													$
													Moreover for every $n$ and $k$, it is easy to see that
													\begin{equation}\label{tree1}
													\begin{array}{c}
													y_k^n\leq y_{k+1}^n,
													\quad y^{n+1}_{k}\leq y^n_k\leq y^{n+1}_{k+1},\smallskip\\
													\displaystyle
													y^n_k\leq y  ^n_{k-1}+\sigma^2h+2\sigma\sqrt{v^n_{k-1}h},\quad
													y^{n+1}_k\leq y^n_k+\frac{\sigma^2} 4\,h-\sigma\sqrt{y^n_k h}.
													\end{array}
													\end{equation}
													%The properties in the first line are trivial; as for the second line, both inequalities are less obvious but straightforward.
													
													\medskip
													
													\noindent
													\textit{Proof of Proposition \ref{propjumps}.}
													1. The statement is an immediate consequence of the following facts:
													\begin{align}
													&\mbox{if $k_u(n,k)\geq k+2$, then $y^n_k < \theta_*h$,}\label{ass1}\\
													&\mbox{if $k_d(n,k)\leq k-1$, then $y^n_k > \theta^*/h$,}\label{ass2}
													\end{align}
													which we now prove.
													%So, we prove \eqref{ass1} and \eqref{ass2}.
													
													First of all, note that $y^n_k+\mu_Y(y^n_k)h= \kappa\theta h +y^n_k(1-\kappa h)$, so by choosing $\bar h = 1/ \kappa$, 	one has 	$y^n_k+\mu_Y(y^n_k)h >0$. Moreover, as   a direct consequence of \eqref{ku2}--\eqref{kd2} and of \eqref{tree1}, we have that, if $\mu_Y(y^n_k)>0$, then $k_d(n,k)=k$, and if $\mu_Y(y^n_k)<0$, then $k_u(n,k)=k+1$.
													
													Concerning \eqref{ass1}, we obviously assume $y^n_k>0$, so that  $y^{n+1}_{k+1}>0.$ Note that, from \eqref{ku2},
													\begin{align*}
													y^n_k+\mu_Y(y^n_k)h > y^{n+1}_{k_u(n,k)-1}\geq y^{n+1}_{k+1}=y^n_k+\frac{\sigma^2}{4}h +\sigma \sqrt{y^n_k h }.
													\end{align*}
													%	so that
													%		\begin{align*}
													%		\mu_Y(y^n_k)h > \frac{\sigma^2}{4}h +\sigma \sqrt{y^n_k h }.
													%		\end{align*}
													Since $\mu_Y(y^n_k) \leq \kappa\theta$,  we get
													$$
													\kappa\theta h > \frac{\sigma^2}{4}h +\sigma \sqrt{y^n_k h }>\sigma \sqrt{y^n_k h},
													$$
													from which
													$$
													y^n_k < \Big( \frac{\kappa\theta}{\sigma}\Big)^2	h=\theta_*h.
													$$
													We prove now \eqref{ass2}.  First of all observe that, if $y^n_k \leq \theta$, then $\mu_Y(y^n_k) >0$ and so $k_d(n,k)=k$. Then we have $y^n_k >\theta$ and from \eqref{vnk} we can assume   $y^{n+1}_k>0$ up to take  $h< ( 2 \sqrt{\theta}/\sigma)^2$.
													Now, by \eqref{kd2} we get
													\begin{align*}
													y^n_k+\mu_Y(y^n_k)h < y^{n+1}_{k_d(n,k)+1}\leq  y^{n+1}_{k}=y^n_k+\frac{\sigma^2}{4}h -\sigma \sqrt{y^n_k h },
													\end{align*}
													so that
													$$
													\kappa(\theta-y^n_k)h %=	\mu_Y(y^n_k)h 
													< \frac{\sigma^2}{4}h -\sigma \sqrt{y^n_k h }.
													$$
													This gives $\kappa y^{n}_k h> \sigma\sqrt{v^{n}_k h} -\frac{\sigma^{2}}{4}\, h +\kappa\theta h$ and, for $h$ small enough, one gets $y^n_k h>\frac{\sigma^2}{4\kappa^2}$.
													
													\smallskip

													2. If $y^n_k \leq \theta_*h$, \eqref{ass2} gives $k_d(n,k)= k$. As regards the up jump, the case $y^{n+1}_{k_u(n,k)}= 0$ is trivial so we consider $y^{n+1}_{k_u(n,k)}>0$.
													In order to prove \eqref{stimasottosoglia}, we consider two possible cases: $k_u(n,k)=k+1$
													and $k_u(n,k)\geq k+2$. In the first case, we have
													\begin{align*}
													y^{n+1}_{k_u(n,k)}-y^n_k= \frac{\sigma^2}{4}h +\sigma \sqrt{y^n_k h }\leq \Big(   \frac{\sigma^2}{4}+\sigma \sqrt{\theta_*}  \Big)h\leq C_\ast h,
													\end{align*}
													and the statement holds. If instead $k_u(n,k)\geq k+2$, then by \eqref{ku2} we have
													$$
													y^{n+1}_{k_u(n,k)-1} - y^n_k <\mu_Y(y^n_k)h.
													$$
													We apply the third inequality in \eqref{tree1} (with $n$ replaced by $n+1$ and $k=k_u(n,k)$) and we get
													\begin{align*}
													0\leq y^{n+1}_{k_u(n,k)}-y^n_k&\leq y^{n+1}_{k_u(n,k)-1}+2\sigma \sqrt{y^{n+1}_{k_u(n,k)-1}h } +\sigma^2h-y^n_k \\
													& \leq \mu_Y(y^n_k)h+2\sigma \sqrt{(y^n_k +\mu_Y(y^n_k)h)h}+\sigma^2h\\
													&\leq (\kappa\theta + 2\sigma\sqrt{ \theta_*+\kappa\theta}+\sigma^2)h\leq C_\ast h.
													\end{align*}
													
													3. The statement follows from \eqref{ass1}.
													
													\smallskip
													
													4. Formula \eqref{proba}  follows from the fact that the sets $K_u(n,k)$ and $K_d(n,k)$ are nonempty. Indeed, 
													if 	$y^n_k> \theta_*h$ then $k_u=k+1$, so $K_u(n,k)\neq \emptyset$.	And if $y^n_k<\theta_*h$,
													\begin{align*}
													y^{n+1}_{n+1}-y^n_k-\mu_Y(y^n_k)h \geq Y_0 -\theta_*h-\kappa\theta h =Y_0 -(\theta_*+\kappa\theta )h>0
													\end{align*}
													for $h< Y_0/(\theta_*+\kappa\theta)$, which gives $k_u(n,k)<n+1$. Therefore   $K_u(n,k)\neq \emptyset$ for every $(n,k)$.
													
													As regards $K_d(n,k)$, if $y^n_k< \theta^*/h$ then $k_d(n,k)=k$ by Proposition \ref{propjumps}, so that $K_d(n,k)\neq \emptyset$. If instead $y^n_k\geq \theta^*/h$, then
													$$
													y^{n+1}_0-y^n_k-\mu_Y(y^n_k)h\leq Y_0-\frac {\theta^*} h -\kappa\theta h +\kappa y^n_kh\leq Y_0-\frac {\theta^*} h  +\kappa y^n_kh .
													$$
													Recalling that $h=T / N$, we note that there exists $C>0$ such that
													\begin{align*}
													y^n_k h\leq y^N_Nh= \Big(\sqrt {Y_0}+\frac{\sigma} 2N\sqrt{h}\Big)^2h= \Big(\sqrt {Y_0}\sqrt{\frac T N }+\frac{\sigma} 2T  \Big)^2\leq C.
													\end{align*}
													Therefore
													$$
													y^{n+1}_0-y^n_k-\mu_Y(y^n_k)h\leq Y_0-\frac {\theta^*} h+\kappa C <0
													$$
													for $h<\frac{\theta^*  }{Y_0+\kappa C}$. So, $K_d(n,k)\neq \emptyset$.

													Now, by \eqref{kd2} and \eqref{ku2}, since $K_d(n,k)\neq \emptyset$ and $K_u(n,k)\neq \emptyset$,
													\begin{align*}
													& \frac{\mu_Y(y^n_k)h+ y^n_k-y^{n+1}_{k_d(n,k)} }{y^{n+1}_{k_u(n,k)}-y^{n+1}_{k_d(n,k)}}\geq 0,
													&\frac{\mu_Y(y^n_k)h+ y^n_k-y^{n+1}_{k_d(n,k)} }{y^{n+1}_{k_u(n,k)}-y^{n+1}_{k_d(n,k)}}= 1+ \frac{\mu_Y(y^n_k)h+ y^n_k-y^{n+1}_{k_u(n,k)} }{y^{n+1}_{k_u(n,k)}-y^{n+1}_{k_d(n,k)}}\leq 1.
													\end{align*}
													\begin{flushright}
														$\square$
													\end{flushright}

	\subsection{Proof of Lemma \ref{lemma-poisson} }\label{app-poisson}

	\begin{proof}
	For $x\in\R$, let $\lfloor x\rfloor=\sup\{k\in \Z\,:\,k\leq x\}$ denote the integer part. For $N\in\N$, straightforward computations give
		$$
		\sum_{|n|\leq N}\varphi(n)=\frac 12 (\varphi(N)+\varphi(-N))+\int_{-N}^{N}\varphi(x)dx+\int_{-N}^{N}\Big(x-\lfloor x\rfloor-\frac 12\Big)\varphi'(x)dx.
$$	
We recall that $\varphi(\pm N)\to 0$ as $N\to\infty$ (because $\varphi,\varphi'\in L^1(\R,dx)$). Moreover, 
%$x\mapsto x-\lfloor x\rfloor-\frac 12$ %$=-\sum_{n\geq 1}\frac {\sin(2\pi n x)}{\pi n}$ is periodic with period 1 and its 
the Fourier series representation gives
$$
x-\lfloor x\rfloor-\frac 12=\sum_{n\in\Z,n\neq 0}\frac {e^{-2\pi\ii n x}}{2\pi\ii n},\quad x\in\R.
$$
So,
\begin{align*}
\sum_{n\in\Z}\varphi(n)
&=\int_\R \varphi(x)dx+\int_\R\sum_{n\in\Z,n\neq 0}\frac {e^{-2\pi\ii n x}}{2\pi\ii n}\varphi'(x)dx.
\end{align*}
With $\mathfrak{F}[\cdot]$ denoting the Fourier transform, we have $\int_\R e^{-2\pi\ii n x}\varphi'(x)dx=\mathfrak{F}[\varphi'](2\pi n)=2\pi \ii n\mathfrak{F}[\varphi]$ $(2\pi n)$ and $|\mathfrak{F}[\varphi'](2\pi n)|\leq |\frac{\mathfrak{F}[\varphi''](2\pi n)}{2\pi n}|\leq \frac M n$ because $\varphi''\in L^1(\R,dx)$. Thus, we can put the sum outside the integral and  the statement holds.
													\end{proof}

													\clearpage
													\fancyhf{} \fancyfoot[CE,CO]{\thepage}
													\fancyhead[CO]{\textit{Acknowledgments}}
													\fancyhead[CE]{\textit{Acknowledgments}}
													\renewcommand{\headrulewidth}{0.5pt}
													\renewcommand{\footrulewidth}{0.0pt}
													\addtolength{\headheight}{0.5pt}
													\fancypagestyle{plain}{\fancyhead{}\renewcommand{\headrulewidth}{0pt}}
													
													\chapter*{Acknowledgments}
													\addcontentsline{toc}{chapter}{Acknowledgments}
													%\begin{flushright}
													%	\emph{\avirg Le véritable voyage de découverte\\ne consiste pas à chercher de nouveaux paysages\\ mais à avoir de nouveaux yeux."}\\
													%M. Proust
													%\end{flushright}
													This thesis was written at the Dipartimento di Matematica of the Università di Roma Tor Vergata and at the Laboratoire d' Analyse et des Mathématiques Appliquées  of the Université Paris-Est Marne-la-Vallée.
													I had two advisors, two laboratories, two places to call \emph{home}.

													I cannot find the words to say how much last three years have meant to me, both from a professional and a personal point of view.   What is for sure is that  I would never been here without the help and the support of some  people. Therefore, I want to dedicate this last part of my thesis,  maybe the most difficult to write, to say them thanks.
													
													\medskip
													
													My first, heartfelt, thanks go to my advisors, Lucia Caramellino and Damien Lamberton.

													Grazie Lucia, for having supported me in all these years,   for having me taught much more than maths, for your friendship.
													Merci Damien, because if now I know some French I owe it to you, for your  kindness and irony, for always  welcoming me with a smile.
													Thank you both for having me taught more or less everything I can do now,  for your necessary help, for all the time you have dedicated to me with  infinite patience.

													I heartily thank Maya Briani, who is like a third advisor to me: thank you Maya   for everything you taught me, for your kindness and for all the time we have spent together discussing about maths and talking as good friends do. I am also  grateful to Antonino Zanette for having given me the opportunity to work with him.
													
													I would like to express my gratitude  to Prof. A. Neuenkirch and to Prof. A. Pascucci, for accepting to  referee my work,  and to  Prof. A.  Alfonsi and Prof. F. Antonelli,  for accepting to be  part of the committee.
													\medskip
													
													I was lucky enough to find warm and stimulating environments in both my laboratories. 
													Many  thanks to the Probability groups both in Tor Vergata and in UPEM. In particular, thank to Prof. V. Bally and to Prof. P. Baldi, for his  help during my stays in France and for all the discussions we had about Paris and maths. 
													
													I am sincerely grateful to  the coordinator of the Ph.D. School in Rome, Prof. A. Braides, and to
													Simonetta De Nicola, Christiane Lafargue, Audrey Patout and Sylvie Cach: your availability  and patience were precious  during all these years. 
													
													I met great people and I had  fantastic  Ph.D. colleagues in both my universities. In France, I wish to especially thank Huong and Yushun, my first friends in UPEM. As regards Tor Vergata, I have a long list of friends to thank. First of all, thanks to the members of  my beloved 1225, especially to Duccio, Rossana and Gianluca, who have been by my side from the beginning. I also wish  to thank the other Ph.D. students and researchers (mathematicians, of course, but computer scientists too!)  with whom I have shared  fruitful  discussions, lunches, laughs and  beer during  all these years.   Among all, a special thank goes to Michele: I could not have a better colleague, flatmate, friend. Thanks, a lot.
													
													\medskip
													
													%Finally, It's the turn to thank my friends,  who had the difficult task to put up with me .
													It's the turn to thank my friends, who have been standing close to me during all these years. 
													%Somehow, each of you helped me to write this thesis.

													Thanks to all the fantastic  people I met in Paris, with whom I have shared unforgettable moments. In particular, thanks to 
													%Davide, Greta, Graziano, Matteo, Michela, Valeria and, of course,  to 
													Danilo and Vanessa, best Roman presents Paris could have  given me. 
													
													Thanks to  Anna, Anna Paola, Martina and Maurizia, my favourite, amazing, women in maths: your support  helped me  in many situations, thanks! 
													Thanks also to my old university colleagues, who are now dear friends of mine.
													
													Thanks to Federica and to my \textit{Sardinia girls}, who were very close to me in the last year, and to my neighbour and friend Giulia.

													Finally, my heartfelt thanks to Beatrice, Isabella, Liliana, Manlio, Michela,  Simone and  Simone: you have been my certainty in all these years, and I feel so  lucky to have friends like you. A special thank to Beatrice, who always believed in me, even when I did not do it.
													\medskip
													
													L'ultimo e più grande ringraziamento va a tutta la mia famiglia, in particolare ai miei genitori e ai miei fratelli Marco e Simone. Siete e sarete sempre il mio punto di riferimento più grande. 
													\clearpage

													\clearpage
													\fancyhf{} \fancyfoot[CE,CO]{\thepage}
													\fancyhead[CO]{\textit{Bibliography}}
													\fancyhead[CE]{\textit{Bibliography}}
													\renewcommand{\headrulewidth}{0.5pt}
													\renewcommand{\footrulewidth}{0.0pt}
													\addtolength{\headheight}{0.5pt}
													\fancypagestyle{plain}{\fancyhead{}\renewcommand{\headrulewidth}{0pt}}

														%
														%\bibliography{bibbase}
														%\bibliographystyle{amsplain}

													\end{document}